\documentclass{amsart}
\usepackage{amssymb, amsmath,amsthm }
\usepackage{hyperref}
\usepackage{stmaryrd}
\usepackage[all]{xy}
\newtheorem{prop}{Proposition}[section]

\newtheorem{lem}[prop]{Lemma}
\newtheorem{thm}[prop]{Theorem}

\newtheorem{cor}[prop]{Corollary}

\theoremstyle{remark}
\newtheorem{remar}[prop]{Remark}
\theoremstyle{definition}
\newtheorem{defi}[prop]{Definition}

\DeclareMathAlphabet{\mathpzc}{OT1}{pzc}{m}{it}
\DeclareMathOperator{\Aut}{Aut}
\DeclareMathOperator{\End}{End}
\DeclareMathOperator{\wTor}{\widehat{Tor}}
\DeclareMathOperator{\Hom}{Hom}
\DeclareMathOperator{\Mor}{Mor}
\DeclareMathOperator{\Ind}{Ind}
\DeclareMathOperator{\cInd}{c-Ind}
\DeclareMathOperator{\Res}{Res}
\DeclareMathOperator{\Sym}{Sym}

\DeclareMathOperator{\GL}{GL}
\DeclareMathOperator{\SL}{SL}
\DeclareMathOperator{\Ker}{Ker}
\DeclareMathOperator{\Coker}{Coker}
\DeclareMathOperator{\WD}{WD}
\DeclareMathOperator{\Gal}{Gal}
\DeclareMathOperator{\Image}{Im}

\DeclareMathOperator{\soc}{soc}
\DeclareMathOperator{\cosoc}{cosoc}

\DeclareMathOperator{\tr}{tr}

\DeclareMathOperator{\Irr}{Irr}

\DeclareMathOperator{\gr}{gr}
\DeclareMathOperator{\Spec}{Spec}
\DeclareMathOperator{\supp}{Supp}
\DeclareMathOperator{\id}{id}
\DeclareMathOperator{\Mod}{Mod}

\DeclareMathOperator{\val}{val}
\DeclareMathOperator{\Sp}{Sp}
\DeclareMathOperator{\Rep}{Rep}

\DeclareMathOperator{\Ext}{Ext}

\DeclareMathOperator{\dt}{det}
\DeclareMathOperator{\Ban}{Ban}
\DeclareMathOperator{\QBan}{QBan}
\DeclareMathOperator{\ind}{ind}
\DeclareMathOperator{\rad}{rad}
\DeclareMathOperator{\dualcat}{\mathfrak C}
\DeclareMathOperator{\qcat}{\mathfrak Q}

\DeclareMathOperator{\Ord}{Ord}
\DeclareMathOperator{\Def}{Def}
\DeclareMathOperator{\Sets}{Sets}
\DeclareMathOperator{\dInd}{dInd}

\DeclareMathOperator{\MaxSpec}{MaxSpec}

\DeclareMathOperator{\LL}{LL}
\newcommand{\cIndu}[3]{\cInd_{#1}^{#2}{#3}}

\newcommand{\Indu}[3]{\Ind_{#1}^{#2}{#3}}
\newcommand{\dIndu}[3]{\dInd_{#1}^{#2}{#3}}

\newcommand{\pF}{\mathfrak{p}}

\newcommand{\Qp}{\mathbb {Q}_p}
\newcommand{\Zp}{\mathbb{Z}_p}
\newcommand{\Qpbar}{\overline{\mathbb{Q}}_p}

\newcommand{\GG}{\mathcal G}
\newcommand{\HH}{\mathcal H}
\newcommand{\NN}{\mathcal N}
\newcommand{\Eins}{\mathbf 1}
\newcommand{\PP}{\mathcal P}

\newcommand{\FF}{\mathcal F}
\newcommand{\Id}{Id}
\newcommand{\KK}{\mathfrak K}
\newcommand{\ZZ}{\mathbb Z}
\newcommand{\VV}{\mathbf V}

\newcommand{\Aa}{\mathfrak A}
\newcommand{\ev}{\mathrm{ev}}

\newcommand{\Fp}{\mathbb F_p}

\newcommand{\Fbar}{\overline{\mathbb F}_p}
\newcommand{\II}{\mathcal I}
\newcommand{\RR}{\mathbb R}
 
\newcommand{\mm}{\mathfrak m}
\newcommand{\ee}{\mathrm{e}}
\newcommand{\st}{\mathrm{st}}
\newcommand{\Fpbar}{\Fbar}
\newcommand{\pr}{\mathrm{pr}}

\newcommand{\wP}{\widetilde{P}}
\newcommand{\wE}{\widetilde{E}}
\newcommand{\wm}{\widetilde{\mathfrak m}}
\newcommand{\OO}{\mathcal O}
\newcommand{\detr}{\det}
\newcommand{\VVV}{\mathcal V}
\newcommand{\TT}{\mathcal T}

\newcommand{\gal}{\mathcal G_{\Qp}}
\newcommand{\hA}{\widehat{\mathfrak A}}
\newcommand{\ad}[1]{\mathrm {ad}(#1)}
\newcommand{\wM}{\widetilde{M}}
\DeclareMathOperator{\wtimes}{\widehat{\otimes}}
\newcommand{\pL}{\mathfrak p_L}
\newcommand{\cV}{\check{\mathbf{V}}}
\DeclareMathOperator{\Ad}{Ad}

\newcommand{\BB}{\mathfrak B}

\newcommand{\cJ}{\check {\mathcal J}}
\newcommand{\nn}{\mathfrak n}
\newcommand{\wa}{\widetilde{\mathfrak a}}
\newcommand{\md}{\mathrm m}
\newcommand{\rr}{\mathfrak r}
\newcommand{\rG}{\mathrm G}

\newcommand{\rL}{\mathrm L}
\newcommand{\rK}{\mathrm K}
\newcommand{\rel}{\leftrightarrow}
\newcommand{\wZ}{\mathcal{Z}}
\newcommand{\Kr}{\mathcal K}
\newcommand{\Hr}{\mathcal H}
\newcommand{\Nr}{\mathcal N}

\title{The image of Colmez's Montreal functor}
\author{Vytautas Pa\v{s}k\={u}nas}
\date{\today.}
\begin{document} 

\begin{abstract} We prove a conjecture of Colmez concerning the reduction modulo $p$ of 
invariant lattices in irreducible admissible unitary $p$-adic Banach space representations of $\GL_2(\Qp)$ with $p\ge 5$. 
This enables us to restate nicely the $p$-adic local Langlands correspondence for $\GL_2(\Qp)$ 
and deduce a conjecture of Breuil on irreducible admissible unitary completions of locally algebraic representations. 
\end{abstract}
\maketitle
\tableofcontents
\section{Introduction} 
In this paper we study $p$-adic and  mod-$p$ representation theory of $G:=\GL_2(\Qp)$. Our results complement the
work of Berger, Breuil and Colmez on the $p$-adic and mod-$p$ Langlands correspondence for $\GL_2(\Qp)$, see \cite{berbour} for an overview. 
Let $L$ be a finite extension of $\Qp$ with a ring of integers $\OO$, uniformizer $\varpi$  and residue field $k$.  

\begin{thm}\label{I1} Assume $p\ge 5$ and let $\Pi$ be a unitary admissible absolutely irreducible $L$-Banach space representation of $G$ with
a central character and let $\Theta$ be an open bounded $G$-invariant lattice in $\Pi$. Then $\Theta\otimes_{\OO} k$ 
is of finite length. Moreover, one of the following holds: 
\begin{itemize}
\item[(i)] $\Theta\otimes_{\OO} k$ is absolutely irreducible supersingular;
\item[(ii)] $(\Theta\otimes_{\OO} k)^{ss}\subseteq (\Indu{P}{G}{\chi_1\otimes \chi_2 \omega^{-1}})^{ss}\oplus  
(\Indu{P}{G}{\chi_2\otimes \chi_1 \omega^{-1}})^{ss}$ for some smooth characters $\chi_1, \chi_2: \Qp^{\times} \rightarrow k^{\times}$,
where $\omega(x)=x|x|\pmod{p}$.
\end{itemize}
Further, the inclusion in (ii) is not an isomorphism if and only if  $\Pi$ is ordinary.
\end{thm}
We say that $\Pi$ is \textit{ordinary} if it is a subquotient of a parabolic induction of a unitary character, so that
$\Pi$ is either a unitary character $\Pi\cong \eta\circ \det$, 
a twist of the universal unitary completion of the smooth Steinberg representation by a unitary character 
$\Pi\cong \widehat{\Sp}\otimes \eta\circ \det$ or $\Pi$ is a unitary parabolic induction of  a unitary character.
An irreducible 
smooth $k$-representation is \textit{supersingular} if it is not a subquotient of any principal series representation. The theorem answers 
affirmatively for $p\ge 5$ a question of Colmez denoted (Q3) in \cite{colmez}. 

Let $Z$ be the centre of $G$, we fix a continuous character $\zeta: Z\rightarrow \OO^{\times}$. Let $\Mod^{\mathrm{sm}}_G(\OO)$ be the category 
of $\OO$-torsion modules with a continuous $G$-action for the discrete topology on the module, let  $\Mod^{\mathrm{sm}}_{G,\zeta}(\OO)$
be the full subcategory of $\Mod^{\mathrm{sm}}_{G}(\OO)$, consisting of representations on which
the centre $Z$ acts by (the image of) $\zeta$, and let 
$\Mod^{\mathrm{fin}}_{G, \zeta}(\OO)$ be the  full subcategory of $\Mod^{\mathrm{sm}}_{G, \zeta}(\OO)$, consisting of representations, 
which are of finite length as $\OO[G]$-modules.  In his Montreal lecture 
Colmez has defined an exact covariant $\OO$-linear functor $\VV: \Mod^{\mathrm{fin}}_{G, \zeta}(\OO)\rightarrow 
\Mod_{\gal}(\OO)$ to the category of $\OO$-modules with a continuous action of $\gal$, the absolute Galois group of $\Qp$. 
Given $\Pi$ as in Theorem \ref{I1} one may choose an open bounded $G$-invariant lattice $\Theta$ in $\Pi$ and define 
$\VV(\Pi):=L\otimes \underset{\longleftarrow}{\lim}\, \VV(\Theta/\varpi^n \Theta)$. Since all open bounded 
lattices in $\Pi$ are commensurable the definition does not depend on the choice of $\Theta$.

\begin{cor}\label{I2} Let $\Pi$ be a unitary admissible absolutely irreducible $L$-Ba\-na\-ch space representation of $G$ with
a central character then $\dim_L \VV(\Pi)\le 2$. Moreover, $\dim_L  \VV(\Pi)<2$ if and only if $\Pi$ is ordinary. 
\end{cor}

Once one has this, the  results of Berger-Breuil \cite{bergerbreuil}, Colmez \cite{colmez}  and Kisin \cite{kisin} imply:
\begin{thm}\label{I3} Assume $p\ge 5$, the functor $\VV$ induces a bijection between isomorphism classes of 
\begin{itemize}
\item[(i)] absolutely irreducible admissible unitary non-ordinary $L$-Banach space representations of $G$ with the central character\footnote{Dospinescu and Schraen  have show recently in \cite{DS} that every absolutely irreducible unitary admissible  $L$-Banach representation of a $p$-adic Lie group admits a central character.} $\zeta$, and 
\item [(ii)] absolutely irreducible $2$-di\-men\-sio\-nal continuous $L$-representations of $\gal$ with determinant equal to $\zeta \varepsilon$,
\end{itemize} 
where $\varepsilon$ is the cyclotomic character, and we view $\zeta$ as  a character of $\gal$ via  class field theory.\footnote{We normalize it the same way as Colmez in \cite{colmez}, see \S\ref{repsGL2}, so that the uniformizers correspond to geometric Frobenii.}
\end{thm}

%
In \cite{colmez} Colmez has also defined a characteristic $0$ construction, which to every 
$2$-di\-men\-sio\-nal continuous $L$-representation of $\gal$ associates an admissible  unitary $L$-Banach space representation $\Pi(V)$
of $G$, such that $\VV(\Pi(V))\cong V$. Colmez has calculated\footnote{In \cite{colmez} some cases are conditional on the results of Emerton, 
which have now appeared in \cite[\S7.4]{emlg}.} locally algebraic vectors
in $\Pi(V)$ in terms of $p$-adic Hodge theoretic data attached to $V$. 
As a consequence of Theorem \ref{I3} we know that for every $\Pi$ in (i) there 
exists a unique $V$ such that $\Pi\cong \Pi(V)$.  Using this in \S\ref{U} we determine 
admissible absolutely irreducible completions of absolutely irreducible locally  algebraic representations. In particular, 
we show that $\Sp \otimes |\det|^{k/2} \otimes \Sym^k L^2$ admits precisely $\mathbb{P}^1(L)$ non-isomorphic absolutely 
irreducible admissible unitary completions, where $k$ is a positive integer and $\Sp$ is the smooth Steinberg representation 
of $G$ over $L$. This confirms a conjecture of Breuil. However, our main result can be summed up as:

\subsection{The correspondence is an equivalence of categories} \label{corrisequiv}  
Let $\Ban_{G, \zeta}^{\mathrm{adm}}(L)$ be the category of unitary admissible $L$-Banach space representations
of $G$ with central character $\zeta$ and let $\Ban_{G, \zeta}^{\mathrm{adm.fl}}(L)$ be the full subcategory consisting of objects of finite length. 
Let $\Mod^{\mathrm{lfin}}_{G,\zeta}(\OO)$ be the full subcategory of $\Mod^{\mathrm{sm}}_{G,\zeta}(\OO)$ consisting 
of those objects which are locally of finite length, that is $(\tau, V)$ is an object of $\Mod^{\mathrm{lfin}}_{G,\zeta}(\OO)$ 
if and only if for every $v\in V$ the $\OO[G]$-module $\OO[G] v$ is of finite length. We obtain Bernstein-centre-like \footnote{Since 
we work in the category of locally finite representations, our rings are analogous to the completions of the rings in 
Bernstein's theory \cite{bernstein} at maximal ideals.} results for the categories  
$\Mod^{\mathrm{lfin}}_{G,\zeta}(\OO)$ and $\Ban_{G, \zeta}^{\mathrm{adm.fl}}(L)$. 
That is we decompose them into a direct product of subcategories and show that each subcategory  is  naturally equivalent 
to the category of modules over the rings related to deformation theory of $2$-di\-men\-sio\-nal $\gal$-representations.

To fix ideas let $\Pi$ be as in Theorem \ref{I3} (i) so that $V:=\VV(\Pi)$ is an absolutely irreducible 
$2$-di\-men\-sio\-nal continuous $L$-representation of $\gal$ with determinant $\zeta \varepsilon$. 
Let $R^{\zeta\varepsilon}_V$ be the deformation ring representing the deformation problem of $V$ with determinant 
$\zeta \varepsilon$ to  local artinian $L$-algebras, and let $V^{\mathrm{u}}$ be the universal deformation of  $V$ with the determinant $\zeta\varepsilon$.
Let $\Ban_{G, \zeta}^{\mathrm{adm.fl}}(L)_{\Pi}$ be the full subcategory of $\Ban_{G, \zeta}^{\mathrm{adm.fl}}(L)$ consisting 
of the representations with all irreducible subquotients isomorphic to $\Pi$.

\begin{thm}\label{I6}  The category $\Ban_{G, \zeta}^{\mathrm{adm.fl}}(L)_{\Pi}$
is a direct summand of the category $\Ban_{G, \zeta}^{\mathrm{adm.fl}}(L)$ and it is naturally equivalent to the category of $R^{\zeta\varepsilon}_V$-modules
of finite length. 
\end{thm}

The first assertion in Theorem \ref{I6} means that any finite length admissible unitary $L$-Banach space representation $\Pi_1$ of $G$ 
with a central character $\zeta$ can be canonically decomposed $\Pi_1\cong \Pi_2\oplus \Pi_3$, 
such that all the irreducible subquotients of $\Pi_2$ are isomorphic to $\Pi$ and none of the irreducible subquotients 
of $\Pi_3$ are isomorphic to $\Pi$. The equivalence of categories in Theorem \ref{I6} is realized as follows: to each 
$\mathrm{B}$ in $\Ban_{G, \zeta}^{\mathrm{adm.fl}}(L)_{\Pi}$ we let $\md(\mathrm B)=\Hom_{\gal}(V^{\mathrm{u}}, \cV(\mathrm B))$, where $\cV(\mathrm B)=\VV(\mathrm B)^*(\varepsilon \zeta)$, and 
then show in Theorem \ref{holiday} that $\cV(\mathrm B)\cong \md(\mathrm B)\otimes_{R^{\zeta\varepsilon}_V} V^{\mathrm u}$.
So at least in some sense we may describe 
what kind of representations of $\gal$ lie in the image of $\VV$, which explains the title of our paper.

We will discuss now what happens with $\Mod^{\mathrm{lfin}}_{G,\zeta}(\OO)$ and recall that we assume $p\ge 5$. We may define 
an equivalence relation on the set of (isomorphism classes of) irreducible objects of $\Mod^{\mathrm{lfin}}_{G,\zeta}(\OO)$, where $\tau \sim \pi$ if and 
only if there exists a sequence of irreducible representations $\tau=\tau_0, \tau_1, \ldots, \tau_n=\pi$ such that $\tau_i=\tau_{i+1}$, 
$\Ext^1_G(\tau_i, \tau_{i+1})\neq 0$ or $\Ext^1_G(\tau_{i+1}, \tau_i)\neq 0$ for each $i$. An equivalence class is called  a block. 
To a block $\BB$ we associate $\pi_{\BB}:=\bigoplus_{\pi\in \BB}\pi $,
$\pi_{\BB}\hookrightarrow J_{\BB}$ an injective envelope of $\pi_{\BB}$ in $\Mod^{\mathrm{lfin}}_{G,\zeta}(\OO)$ and $\wE_{\BB}:=\End_G(J_{\BB})$.
One may show that $\wE_{\BB}$ is naturally a pseudo-compact ring, see \S \ref{zerosec}. By a general result of Gabriel on locally 
finite categories \cite[\S IV]{gab} we have a decomposition of categories:
$$\Mod^{\mathrm{lfin}}_{G,\zeta}(\OO)\cong \prod_{\BB} \Mod^{\mathrm{lfin}}_{G,\zeta}(\OO)^{\BB},$$
where the product is taken over all the blocks $\BB$ and  $\Mod^{\mathrm{lfin}}_{G,\zeta}(\OO)^{\BB}$ denotes a full subcategory 
of $\Mod^{\mathrm{lfin}}_{G,\zeta}(\OO)$ consisting of those representations, such that all the irreducible subquotients 
lie in $\BB$. Moreover, the functor $\tau \mapsto \Hom_G(\tau, J_{\BB})$ induces an anti-equivalence of categories 
between $\Mod^{\mathrm{lfin}}_{G,\zeta}(\OO)^{\BB}$ and the category of compact $\wE_{\BB}$-modules. In this paper we explicitly 
work out the rings $\wE_{\BB}$. 

We are going to describe the blocks. Since we are working over a coefficient field which is not algebraically closed, not 
every irreducible $k$-representation $\tau$ of $G$ is absolutely irreducible. However, we show that given a block 
$\BB$ there exists a finite extension $l$ of $k$ such that  for all $\tau\in \BB$, $\tau\otimes_k l$ is isomorphic to a finite  direct sum
of absolutely irreducible representations. The blocks containing an absolutely irreducible representation are given by:
\begin{itemize}
\item[(i)] $\BB=\{\pi\}$, supersingular;
\item[(ii)] $\BB=\{ \Indu{P}{G}{\chi_1\otimes \chi_2\omega^{-1}}, \Indu{P}{G}{\chi_2\otimes \chi_1\omega^{-1}}\}$, 
$\chi_1\chi_2^{-1}\neq \Eins, \omega^{\pm 1}$;
\item[(iii)] $\BB=\{\Indu{P}{G}{\chi\otimes \chi\omega^{-1}}\}$;
\item[(iv)] $\BB=\{ \eta\circ \det, \Sp\otimes\eta\circ \det, (\Indu{P}{G}{\omega\otimes \omega^{-1}})\otimes\eta\circ \det\}$.
\end{itemize}
To prove this, one has to compute all the $\Ext^1$ groups between irreducible representations of $G$, which have been classified by 
Barthel-Livne \cite{bl} and Breuil \cite{breuil1}. In many cases these computations 
have been dealt with by Breuil and the author \cite{bp}, Colmez \cite{colmez} and Emerton \cite{ord2} by different methods, and 
the computation was completed in \cite{ext2}. To each $\BB$ we may associate a $2$-di\-men\-sio\-nal semi-simple 
$k$-representation $\rho$ of $\gal$ using the semi-simple mod-$p$ correspondence of Breuil, \cite{breuil1}, \cite{breuil2}: (i) $\rho:=\VV(\pi)$ is absolutely irreducible;
(ii) $\rho=\chi_1\oplus \chi_2$; (iii) $\rho= \chi\oplus \chi$; (iv) $\rho:= \eta\oplus \eta \omega$. Let 
$R^{\mathrm{ps}, \zeta\varepsilon}_{\tr \rho}$ be the universal deformation ring representing the deformation problem 
of $2$-di\-men\-sio\-nal pseudocharacters with determinant $\zeta\varepsilon$ lifting the trace of $\rho$, see \S \ref{pseudocharacters}
for a definition. 

\begin{thm}\label{I7} Let $\BB$ be as above then the centre of $\wE_{\BB}$ and hence the centre of the category 
$\Mod^{\mathrm{lfin}}_{G,\zeta}(\OO)^{\BB}$ is naturally isomorphic to $R^{\mathrm{ps}, \zeta\varepsilon}_{\tr \rho}$.
\end{thm}

Recall that the centre of an abelian category is the ring of endomorphisms of the identity functor. In particular, it acts naturally on every object in the category.
We also show that $\wE_{\BB}$ is finitely generated as a module over its centre and after localizing away from the 
reducible locus it is isomorphic to a matrix algebra. In cases (i), (ii) we have a nice characterization 
of $\wE_{\BB}$ in terms of the Galois side. 
We may extend $\VV$ to the category $\Mod^{\mathrm{lfin}}_{G,\zeta}(\OO)$ since every object is a union 
of subobjects of finite length. If $\BB=\{\pi\}$ with $\pi$  supersingular then $J_{\BB}$ is simply an injective envelope of $\pi$. 
Let $\rho=\VV(\pi)$, $R^{\zeta \varepsilon}_{\rho}$ be the deformation ring representing the deformation problem of $\rho$ with 
determinant equal to $\zeta\varepsilon$ and let $\rho^{\mathrm{un}}$ be the universal deformation with determinant $\zeta\varepsilon$. 

\begin{thm} If $\BB=\{\pi\}$ with $\pi$  supersingular then $\VV(J_{\BB})^{\vee}(\zeta \varepsilon)\cong \rho^{\mathrm{un}}$
and $\wE_{\BB}\cong R^{\zeta\varepsilon}_{\rho}$, where $\vee$ denotes the Pontryagin dual. 
\end{thm}
Thus to every $\tau$ in  $\Mod^{\mathrm{lfin}}_{G,\zeta}(\OO)^{\BB}$ we may associate a compact $\wE_{\BB}$-module 
$\md(\tau):=\Hom_G(\tau, J_{\BB})$ and then $\VV(\tau)^{\vee}(\zeta\varepsilon)\cong \md(\tau)\wtimes_{\wE_{\BB}} \rho^{\mathrm{un}}$.
  
In the generic reducible case (ii), $\Ext^1_{\gal}(\chi_1, \chi_2)$ and $\Ext^1_{\gal}(\chi_2, \chi_1)$ are both $1$-di\-men\-sio\-nal. 
Thus there exists unique up to isomorphism non-split extension $\rho_1$ of $\chi_1$ by $\chi_2$ and $\rho_2$ of $\chi_2$ by $\chi_1$. 
Since $\chi_1\neq \chi_2$ the deformation problems  of $\rho_1$ and $\rho_2$ with determinant equal to $\zeta\varepsilon$ are represented 
by $R^{\zeta\varepsilon}_{\rho_1}$ and $R^{\zeta\varepsilon}_{\rho_2}$ respectively. Let $\rho_1^{\mathrm{un}}$ and $\rho_2^{\mathrm{un}}$ 
be the universal deformation of $\rho_1$ and $\rho_2$ with determinant $\zeta\varepsilon$, respectively.   

\begin{thm} If $\BB$ is as in (ii)  then $\VV(J_{\BB})^{\vee}(\zeta \varepsilon)\cong \rho^{\mathrm{un}}_1\oplus \rho^{\mathrm{un}}_2$
and $\wE_{\BB}\cong \End_{\gal}(\rho^{\mathrm{un}}_1\oplus \rho^{\mathrm{un}}_2)$, where $\vee$ denotes the Pontryagin dual. 
\end{thm} 

Again one may describe the image of $\Mod^{\mathrm{lfin}}_{G,\zeta}(\OO)^{\BB}$ under $\VV$ as follows: to 
 every $\tau$ in  $\Mod^{\mathrm{lfin}}_{G,\zeta}(\OO)^{\BB}$ we may associate a compact $\wE_{\BB}$-module 
$\md(\tau):=\Hom_G(\tau, J_{\BB})$ and then $\VV(\tau)^{\vee}(\zeta\varepsilon)\cong \md(\tau)\wtimes_{\wE_{\BB}} 
(\rho^{\mathrm{un}}_1\oplus \rho_2^{\mathrm{un}})$. For non-generic cases, (iii) and (iv) see  the introductions to 
\S \ref{nongenericcaseI} and \S \ref{nongenericcaseII}.

\begin{prop}\label{blockdecompBI} The category $\Ban^{\mathrm{adm}}_{G, \zeta}(L)$ decomposes into a direct sum of 
categories: 
$$\Ban^{\mathrm{adm}}_{G, \zeta}(L)\cong \bigoplus_{\BB} \Ban^{\mathrm{adm}}_{G, \zeta}(L)^{\BB} ,$$
where objects of  $\Ban^{\mathrm{adm}}_{G, \zeta}(L)^{\BB}$ are those $\Pi$ in $\Ban^{\mathrm{adm}}_{G,\zeta}(L)$ such that 
for every open bounded $G$-invariant lattice $\Theta$ in $\Pi$ the irreducible subquotients of $\Theta\otimes_{\OO} k$ 
lie in $\BB$.
\end{prop}
     
Let $\Ban^{\mathrm{adm. fl}}_{G,\zeta}(L)^{\BB}$ be the full subcategory of $\Ban^{\mathrm{adm}}_{G, \zeta}(L)^{\BB}$ consisting of objects of finite length.

\begin{cor} Suppose that $\BB$ contains an absolutely irreducible representation. We have a natural equivalence of categories 
$$\Ban^{\mathrm{adm. fl}}_{G,\zeta}(L)^{\BB}\cong 
\bigoplus_{\nn\in \MaxSpec R_{\tr \rho}^{\mathrm{ps},\zeta \varepsilon}[1/p]}\Ban^{\mathrm{adm. fl}}_{G,\zeta}(L)^{\BB}_{\nn}.$$
The category $\Ban^{\mathrm{adm. fl}}_{G, \zeta}(L)^{\BB}_{\nn}$ is anti-equivalent to the category 
of modules of finite length over the  $\nn$-adic completion of $\wE_{\BB}[1/p]$.
\end{cor} 
   
To explain the last equivalence let $\Pi$ be an object of $\Ban^{\mathrm{adm. fl}}_{G,\zeta}(L)^{\BB}$ and choose an 
open bounded lattice $\Theta$ in $\Pi$. For each $n\ge 1$, $\Theta/\varpi^n \Theta$ is an object of 
$\Mod^{\mathrm{l fin}}_{G, \zeta}(\OO)$. Since $R_{\tr\rho}^{\mathrm{ps},\zeta \varepsilon}$ is  naturally isomorphic to the centre of the category $\Mod^{\mathrm{l fin}}_{G, \zeta}(\OO)$,
it acts on $\Theta/\varpi^n \Theta$. By passing to the limit and inverting $p$ we obtain an action 
of  $R_{\tr \rho}^{\mathrm{ps},\zeta \varepsilon}[1/p]$ on $\Pi$. By definition $\Pi$ is an object 
of  $\Ban^{\mathrm{adm. fl}}_{G, \zeta}(L)^{\BB}_{\nn}$ if and only if it is killed by a power of the maximal ideal 
$\nn$. The corollary is essentially an application of the Chinese remainder theorem. If $\nn$ corresponds to the trace of an absolutely
irreducible representation, defined over the residue field of $\nn$, then  one may show that the $\nn$-adic completion of $\wE_{\BB}[1/p]$ is 
isomorphic to a matrix algebra over the $\nn$-adic completion of $R_{\tr \rho}^{\mathrm{ps},\zeta \varepsilon}[1/p]$.
We obtain:

\begin{thm} Let $\nn$ be a maximal ideal of $R_{\tr \rho}^{\mathrm{ps},\zeta \varepsilon}[1/p]$ with residue field 
$L$ and suppose that the corresponding pseudocharacter $T_{\nn}$ is the trace of an absolutely irreducible representation $V$, defined over $L$. Then the category 
$\Ban^{\mathrm{adm. fl}}_{G, \zeta}(L)^{\BB}_{\nn}$ is naturally  
equivalent to the category of modules of finite length over the $\nn$-adic completion of $R_{\tr \rho}^{\mathrm{ps},\zeta \varepsilon}[1/p]$.
In particular it contains only one irreducible object $\Pi_{\nn}$. The Banach space representation $\Pi_{\nn}$ 
is non-ordinary, and is the unique irreducible admissible unitary $L$-Banach space representation of $G$ with 
a central character satisfying $\VV(\Pi_{\nn})\cong V$.
\end{thm}

\begin{thm} Let $\nn$ be a maximal ideal of $R_{\tr \rho}^{\mathrm{ps},\zeta \varepsilon}[1/p]$ with residue field 
$L$ and suppose that the corresponding pseudocharacter is equal to $\psi_1+\psi_2$, where $\psi_1, \psi_2: \gal\rightarrow L^{\times}$
are continuous group homomorphisms. Then the irreducible objects of  $\Ban^{\mathrm{adm. fl}}_{G, \zeta}(L)^{\BB}_{\nn}$
are subquotients of $(\Indu{P}{G}{\psi_1\otimes \psi_2\varepsilon^{-1}})_{cont}$ and 
$(\Indu{P}{G}{\psi_2\otimes \psi_1\varepsilon^{-1}})_{cont}$, where we consider $\psi_1, \psi_2$ as characters 
of $\Qp^{\times}$ via the class field theory and $\varepsilon(x):= x |x|$, for all $x\in \Qp^{\times}$.
\end{thm}

\subsection{A sketch  of  proof}\label{sketch}
Let $G$ be any $p$-adic analytic group. Let $\Mod^{\mathrm{sm}}_G(\OO)$ be the category of smooth 
representations of $G$ on $\OO$-torsion modules and let $\Mod^{\mathrm{lfin}}_G(\OO)$ be the full subcategory 
of $\Mod^{\mathrm{sm}}_G(\OO)$ consisting of locally finite representations. Let $\Mod^{?}_{G}(\OO)$ be 
a full subcategory of $\Mod^{\mathrm{lfin}}_G(\OO)$ closed under arbitrary direct sums 
and subquotients in $\Mod^{\mathrm{lfin}}_G(\OO)$.  An example of such category  $\Mod^?_{G}(\OO)$
can be  $\Mod^{\mathrm{lfin}}_G(\OO)$ itself, or $\Mod^{\mathrm{lfin}}_G(k)$ the full subcategory 
consisting of objects killed by $\varpi$, or $\Mod^{\mathrm{lfin}}_{G, \zeta}(\OO)$ the full 
subcategory consisting of objects on which $Z$, the centre of $G$, acts by a fixed character $\zeta$, but
there are lots of such categories. It follows from \cite[\S II]{gab} that every 
object in $\Mod^?_{G}(\OO)$ has an injective envelope.

Instead of working with torsion modules we prefer to work dually with compact modules.
Let $H$ be a compact open subgroup of $G$ and  let 
$\Mod^{\mathrm{pro \, aug}}_G(\OO)$ be the category of profinite $\OO[[H]]$-modules with an action 
of $\OO[G]$ such that the two actions are the same when restricted to $\OO[H]$. This category 
has been introduced by Emerton in \cite{ord1}. Sending $\pi$ to its Pontryagin dual $\pi^{\vee}$, see \S\ref{zerosec}, 
  induces an anti-equivalence of categories between 
$\Mod^{\mathrm{sm}}_G(\OO)$ and $\Mod^{\mathrm{pro \, aug}}_G(\OO)$. Let $\dualcat(\OO)$ be the full 
subcategory of $\Mod^{\mathrm{pro \, aug}}_G(\OO)$ anti-equivalent to $\Mod^?_{G}(\OO)$ via  Pontryagin duality 
and let $\dualcat(k)$ be the full subcategory consisting of objects killed by $\varpi$. Since an anti-equivalence 
reverses the arrows every object in $\dualcat(\OO)$ has a projective envelope. 

Let $\pi$ be an irreducible object of $\Mod^?_{G}(\OO)$ such that $\End_G(\pi)=k$. Let $S:=\pi^{\vee}$ 
and  $\wP\twoheadrightarrow S$ a projective envelope of $S$ in $\dualcat(\OO)$.
We assume the existence  of an object $Q$ in $\dualcat(k)$ of finite length, satisfying the following hypotheses:
\begin{itemize}
\item[(H1)] $\Hom_{\dualcat(k)}(Q, S')=0$,  $\forall S'\in \Irr(\dualcat(k))$,  $S\not\cong S'$;
\item[(H2)] $S$ occurs as a subquotient in $Q$ with multiplicity $1$;
\item[(H3)] $\Ext^1_{\dualcat(k)}(Q, S')=0$, $\forall S'\in \Irr(\dualcat(k))$,  $S\not\cong S'$;
\item[(H4)] $\Ext^1_{\dualcat(k)}(Q, S)$ is finite di\-men\-sio\-nal;
\item[(H5)] $\Ext^2_{\dualcat(k)}(Q, R)=0$, where $R=\rad Q$ is the maximal proper subobject of $Q$;
\item[(H0)] $\Hom_{\dualcat(k)}(\wP[\varpi], R)=0$,
\end{itemize}
where $\Irr(\dualcat(k))$ denotes the set of irreducible objects in $\dualcat(k)$ (equivalently in $\dualcat(\OO)$), 
and $\wP[\varpi]$ denotes the kernel of multiplication by $\varpi$. We encourage the reader for the sake 
of this introduction to assume that $Q=S$ then the only real hypotheses 
are (H3) and (H4). As an example one could take $G$ a pro-$p$ group and $\pi$ the trivial representation, or 
$G=\Qp^{\times}$ and $\pi$ a continuous character from $G$ to $k^{\times}$ and $\Mod^?_{G}(\OO)=\Mod^{\mathrm{lfin}}_G(\OO)$.

The ring  $\wE:=\End_{\dualcat(\OO)}(\wP)$ can be naturally equipped with a topology with respect to which 
it is a pseudo-compact ring. It can be shown, see \S \ref{zerosec}, that $\wE$ is a local (possibly non-commutative) ring with
residue field $\End_{\dualcat(k)}(S)=k$. Since $k$ is assumed to be finite $\wE$ is in fact compact. In Proposition \ref{filtdone}
and its Corollaries we show:

\begin{prop}\label{flatpropI} If the hypotheses are satisfied then the natural topology on $\wE$ coincides with 
the topology defined by the maximal ideal; $\wP$ is a flat $\wE$-module and $k\wtimes_{\wE} \wP\cong Q$.
\end{prop}

\begin{remar}\label{Qject}
Let us comment on the rigidity of the setup. There always exists an object of $\dualcat(\OO)$ satisfying (H1), (H2) and (H3). 
Moreover, it is uniquely determined up to isomorphism and is isomorphic to  $k\wtimes_{\wE} \wP$, which is the maximal 
quotient of $\wP$ containing $S$ with multiplicity one. So once we impose (H1), (H2) and (H3) we have no flexibility about
(H4) and (H5), moreover $k\wtimes_{\wE} \wP$ need not be of finite length in general. If either (H4) or (H5) is not satisfied, 
one might try and replace $\dualcat(\OO)$ by a different category, for example a full subcategory or, as we do in \S\ref{nongenericcaseII},
by a quotient category and hope that the hypotheses hold there.
\end{remar}

Using  Proposition \ref{flatpropI} one can do deformation theory with non-commutative coefficients.  Let 
$\mathfrak A$ be the category of finite local (possibly non-com\-mu\-ta\-tive) artinian augmented  
$\OO$-algebras with residue field $k$. The ring $\wE$ is a pro-object 
in this category. A deformation of $Q$ to $A$ is a 
pair $(M, \alpha)$, where  $M$  is an  object of $\dualcat(\OO)$ together with the map of $\OO$-algebras $A\rightarrow \End_{\dualcat(\OO)}(M)$,
which makes $M$ into a  flat  $A$-module and $\alpha: k\wtimes_A M \cong Q$ is an isomorphism in $\dualcat(k)$.  
Let $\Def_Q:\mathfrak A \rightarrow \mathrm{Sets}$ be the functor associating to $A$ the set of isomorphism 
classes of deformations of $Q$ to $A$. We show in Theorem \ref{repnonC} that:

\begin{thm}\label{repnonCI} If the hypotheses are satisfied then
the  map which sends $\varphi: \wE\rightarrow A$ to 
$A\wtimes_{\wE, \varphi} \wP$ induces a bijection between $A^{\times}$-conjugacy classes of $\Hom_{\hA}(\wE, A)$ and 
$\Def_Q(A)$.
\end{thm} 

If we restrict the functor $\Def_Q$ to $\mathfrak A^{ab}$, the full subcategory of $\mathfrak A$ consisting of commutative 
algebras, then we recover  the usual deformation theory with commutative coefficients. 

\begin{cor}\label{representab} $\Def^{ab}_Q(A)=\Hom_{\hA}(\wE^{ab}, A)$, where $\wE^{ab}$ is 
the maximal commutative quotient of $\wE$.
\end{cor}

Let $\Pi$ be an admissible unitary Banach space representation of $G$ in the sense of Schneider-Teitelbaum \cite{iw}
 and let $\Theta$ be an open bounded $G$-invariant lattice in $\Pi$.  We denote by $\Theta^d$  its Schikhof dual,  
$\Theta^d:=\Hom_{\OO}(\Theta, \OO)$ equipped with the topology of pointwise convergence. We have shown in \cite{comp} that 
there exists a natural topological isomorphism $\Theta^d\cong \underset{\longleftarrow} {\lim}\, (\Theta/\varpi^n \Theta)^{\vee}$. 
Thus $\Theta^d$ is an object of $\Mod^{\mathrm{pro \, aug}}_G(\OO)$. Let $\Ban^{\mathrm{adm}}_{\dualcat(\OO)}$ denote the full 
subcategory of the category of admissible Banach space representations of $G$, such that for some (equivalently every) 
open bounded $G$-invariant lattice $\Theta$, $\Theta^d$ is an object of $\dualcat(\OO)$. One may show that, since 
$\dualcat(\OO)$ is assumed to be closed under subquotients in $\Mod^{\mathrm{pro \, aug}}_G(\OO)$, the 
category $\Ban^{\mathrm{adm}}_{\dualcat(\OO)}$ is abelian. The idea is instead of studying Banach space representations
study $\wE$-modules $\Hom_{\dualcat(\OO)}(\wP, \Theta^d)$ and $\md(\Pi):=\Hom_{\dualcat(\OO)}(\wP, \Theta^d)\otimes_{\OO} L$.
 
\begin{lem}\label{multdimI} The (possibly infinite) dimension of  $\md(\Pi)$ is equal to the multiplicity with which $\pi$ occurs 
in $\Theta\otimes_{\OO} k$.
\end{lem} 

This is the $cde$-triangle of Serre, see \S 15 of \cite{serre}.

\begin{prop}\label{longproofI} Suppose that  the centre $\mathcal Z$ of $\wE$ is noetherian and $\wE$ is a finitely generated
 $\mathcal Z$-module. If $\Pi$ in $\Ban^{\mathrm{adm}}_{\dualcat(\OO)}$ is irreducible then $\md(\Pi)$ is finite di\-men\-sio\-nal.
\end{prop} 

Let $\Ban^{\mathrm{adm. fl}}_{\dualcat(\OO)}$ denote the full subcategory of $\Ban^{\mathrm{adm}}_{\dualcat(\OO)}$ consisting 
of objects of finite length. Let $\Ker \md$ be the full subcategory of  $\Ban^{\mathrm{adm. fl}}_{\dualcat(\OO)}$ consisting 
of those $\Pi$ such that $\md(\Pi)=0$. It follows from Lemma \ref{multdimI} that $\Pi$ is an object of 
$\Ker \md$ if and only if $\pi$ does not appear as a subquotient of the reduction of $\Theta$ modulo $\varpi$.
Since $\wP$ is projective one may show that the functor $\md$ is exact and so $\Ker \md$ is a thick 
subcategory. We denote the quotient category by $\Ban^{\mathrm{adm. fl}}_{\dualcat(\OO)}/\Ker \md$.

\begin{thm}\label{banI} Suppose that the hypotheses (H0)-(H5) hold and $Q$ is a finitely generated $\OO[[H]]$-module for 
an open compact subgroup $H$ of $G$. Assume further that the centre of $\wE$ is noetherian and $\wE$ is a finitely 
generated module over its centre. Then the functor $\md$ induces an anti-equivalence of categories 
between  $\Ban^{\mathrm{adm. fl}}_{\dualcat(\OO)}/\Ker \md$ and the category of finite dimensional $L$-vector spaces
with a right $\wE[1/p]$-action.    
\end{thm}

\begin{cor}\label{redmultI1} Under the assumptions of Theorem \ref{banI} the functor $\md$ induces a bijection between isomorphism classes of: 
\begin{itemize}
\item[(i)] irreducible  right $\wE[1/p]$-modules, finite dimensional over $L$;
\item[(ii)] irreducible $\Pi$ in $\Ban^{\mathrm{adm}}_{\dualcat(\OO)}$ such that $\pi$ occurs as a subquotient 
of $\Theta/\varpi \Theta$ for some open bounded $G$-invariant lattice $\Theta$ in $\Pi$.
\end{itemize}
Moreover, $\Pi$ is absolutely irreducible if and only if $\md(\Pi)$ is absolutely irreducible as $\wE[1/p]$-module.
\end{cor} 

The inverse functor to $\md$ in Theorem \ref{banI} is constructed as follows. Let $\md$ be a finite dimensional 
$\wE[1/p]$-module. Let $\md^0$ be any finitely generated $\wE$-submodule of $\md$, which contains an $L$-basis of $\md$. 
Our assumptions imply that $\wE$ is compact and noetherian, thus $\md^0$ is an open  bounded $\OO$-lattice in $\md$. 
Since $\wP$ is a flat $\wE$-module by Proposition \ref{flatpropI} we deduce that $\md^0\wtimes_{\wE}\wP$ is 
$\OO$-torsion free. Let $\Pi(\md):=\Hom_{\OO}^{cont}(\md^0\wtimes_{\wE} \wP, L)$ with the topology induced by the supremum norm. 
One may show that the natural map $\Pi \rightarrow \Pi(\md(\Pi))$ 
in $\Ban^{\mathrm{adm. fl}}_{\dualcat(\OO)}$ is an isomorphism in the quotient category.  If $\Pi$ is irreducible 
and $\md(\Pi)\neq 0$ we deduce that the natural map $\Pi \rightarrow \Pi(\md(\Pi))$ is an injection. 
Let $m$ be the multiplicity with which $\pi$ occurs as a subquotient of $\Theta/\varpi \Theta$. Lemma \ref{multdimI}
says that $\dim_L \md(\Pi)= m$ and thus $\md^0$ is a free $\OO$-module of rank $m$ and so $\md^0\otimes_{\OO} k$ is an 
$m$-di\-men\-sio\-nal $k$-vector space. It follows from the Proposition \ref{flatpropI} that the semisimplification of  
$(\md^0\wtimes_{\wE} \wP)\otimes_{\OO} k \cong  (\md^0\otimes_{\OO} k)\wtimes_{\wE} \wP$ is isomorphic to the semisimplification 
of $Q^{\oplus m}$. Using this we obtain:

\begin{cor}\label{redmultI2} Suppose the assumptions of Theorem \ref{banI} are satisfied. Let $\Pi$ in $\Ban^{\mathrm{adm}}_{\dualcat(\OO)}$
be irreducible and suppose that $\pi$ occurs as a subquotient of $\Theta/\varpi \Theta$ then 
$$ \overline{\Pi}\subseteq ((Q^{\oplus m})^{\vee})^{ss},$$
where $\overline{\Pi}$ denotes the semi-simplification of $\Theta/\varpi \Theta$ and $m$ the multiplicity 
with which $\pi$ occurs in $\overline{\Pi}$.
\end{cor}  

From the hypotheses one may deduce that $\Ext^1_{\dualcat(k)}(Q,Q)$ is finite dimensional and  
so Corollary \ref{representab} implies that the tangent space of $\wE$ is finite dimensional.
Thus if $\wE$ is commutative then it is noetherian. The irreducible modules of $\wE[1/p]$ correspond 
to the maximal ideals and the absolutely irreducible modules correspond to the maximal ideals of $\wE[1/p]$
with residue field $L$. In particular,  the absolutely irreducible modules are $1$-di\-men\-sio\-nal. Hence, we obtain:   

\begin{cor}\label{commI} Suppose that the hypotheses (H0)-(H5) hold and $Q$ is a finitely generated $\OO[[H]]$-module for 
an open compact subgroup $H$ of $G$ and $\wE$ is commutative. Then for every absolutely irreducible $\Pi$  in $\Ban^{\mathrm{adm}}_{\dualcat(\OO)}$
such that $\pi$ is a subquotient in $\Theta/\varpi \Theta$ we have $ \overline{\Pi}\subseteq (Q^{\vee})^{ss}$.
\end{cor}

In Theorem \ref{crit} we devise a criterion for commutativity of $\wE$. 

\begin{thm}\label{critI} Let $d:=\dim \Ext^1_{\dualcat(k)}(Q, Q)$ and $r=\lfloor \frac{d}{2}\rfloor$. 
Suppose that the hypotheses (H0)-(H5) 
are satisfied and there exists a surjection
$\wE\twoheadrightarrow \OO[[x_1,\ldots, x_d]]$.  Further, suppose that for every exact sequence 
\begin{equation}
0\rightarrow Q^{\oplus r} \rightarrow  T\rightarrow Q\rightarrow 0
\end{equation}
with $\dim \Hom_{\dualcat(k)}(T, S)=1$ we have $\dim \Ext^1_{\dualcat(k)}(T, S)\le \frac{r(r-1)}{2}+ d$ then $\wE\cong \OO[[x_1,\ldots, x_d]]$.
\end{thm}
 
Recall that up to now $G$ was an arbitrary $p$-adic analytic group and the category $\Mod^{?}_G(\OO)$ 
was any full subcategory of $\Mod^{\mathrm{lfin}}_G(\OO)$ closed under arbitrary direct sums 
and subquotients in $\Mod^{\mathrm{lfin}}_G(\OO)$.  Now we apply the
 formalism to $G=\GL_2(\Qp)$ and $\Mod^{?}_G(\OO)=\Mod^{\mathrm{lfin}}_{G, \zeta}(\OO)$, where $\zeta$ is a fixed central character. 
We show in Proposition \ref{injgoinj} that injective objects in $\Mod^{\mathrm{lfin}}_{G, \zeta}(\OO)$ are also injective 
in $\Mod^{\mathrm{sm}}_{G, \zeta}(\OO)$ and this implies that they are $p$-divisible and hence 
projective objects in $\dualcat(\OO)$ are $\OO$-torsion free. Thus $\wP[\varpi]=0$ and so the hypothesis (H0) is satisfied. 
Results of Breuil \cite{breuil1} and Barthel-Livne \cite{bl} imply that any object of finite length in $\Mod^{\mathrm{sm}}_{G, \zeta}(\OO)$
is admissible, dually this means that every object of finite length in $\dualcat(\OO)$ is a finitely generated $\OO[[H]]$-module, 
where $H$ is any open compact subgroup of $G$. Thus to make the formalism work we only need to find $Q$ and be able to compute 
$\Ext$-groups. 

One has to consider four separate cases corresponding to the shape of the block $\BB$ described in \S \ref{corrisequiv}. 
In the generic cases (i) and (ii), $Q$ is the Pontryagin dual of what Colmez calls \textit{atome automorphe}, 
that is in case (i) $Q=S=\pi^{\vee}$, in case (ii) $Q=\kappa^{\vee}$ where $\kappa$ is the unique non-split extension between 
the two distinct principal series representations which lie in the block $\BB$ and $S$ is the cosocle of $Q$.   
In \S\ref{supersingularreps} and \S\ref{genericcase} we verify that the hypotheses (H1)-(H5) are satisfied. Thus by Theorem  
\ref{repnonCI} the endomorphism ring $\wE$ of a projective envelope $\wP$ of $S$ in $\dualcat(\OO)$ represents a 
deformation problem of $Q$ with non-commutative coefficients. Using the results  of Kisin \cite{kisin} we show that the functor 
$\cV:\dualcat(\OO)\rightarrow \Rep_{\gal}(\OO)$, $M\mapsto \VV(M^{\vee})^{\vee}(\zeta \varepsilon)$ induces 
a morphism of deformation functors of $Q$ and $\cV(Q)$ and a surjection $\wE^{ab}\twoheadrightarrow R^{\varepsilon \zeta}_{\cV(Q)}$, 
where  $R^{\varepsilon \zeta}_{\cV(Q)}$ is the ring representing a deformation problem of $\cV(Q)$ with commutative coefficients 
and determinant equal to $\varepsilon \zeta$. This argument uses the density of crystalline points in the deformation space and 
essentially  is the same as in \cite{kisin}, except that Kisin deforms objects in $\Mod^{\mathrm{lfin}}_{G, \zeta}(\OO)$
and we deform objects in the dual category $\dualcat(\OO)$. In the generic cases the ring $R^{\varepsilon \zeta}_{\cV(Q)}$ is 
formally smooth and thus a further $\Ext$ computation enables 
us to deduce from Theorem  \ref{critI} that $\cV$ induces an isomorphism $\wE\cong  R^{\varepsilon \zeta}_{\cV(Q)}$. In particular, 
$\wE$ is commutative and  Corollary \ref{commI} applies.

The non-generic cases are much more involved. Let $\pi=\Indu{B}{G}{\chi \otimes\chi \omega^{-1}}$. 
In case (iii),  $\BB=\{\pi\}$ and we show in 
\S\ref{nongenericcaseI} that the hypotheses (H1)-(H5) are satisfied with $Q=S=\pi^{\vee}$. 
One may further show that the dimension of $\Ext^1_{\dualcat(k)}(Q, Q)$ is $2$ and there exists a surjection $\wE\twoheadrightarrow
\OO[[x,y]]$, but the last condition in Theorem \ref{critI} fails. However, we can still compute $\wE$ using the fact that 
$\cV$ induces a morphism of deformation functors. 
 Let $R^{\mathrm{ps}, \zeta \varepsilon}_{2\chi}$ be the universal 
deformation ring parameterising $2$-di\-men\-sio\-nal pseudocharacters of $\gal$ with determinant $\zeta\varepsilon$ 
lifting $\chi+\chi$ and let $T:\gal\rightarrow  R^{\mathrm{ps}, \zeta \varepsilon}_{2\chi}$ be the universal pseudocharacter. 
We show that $\wE$ is naturally isomorphic to $R^{\mathrm{ps}, \zeta \varepsilon}_{2\chi}[[\gal]]/J$, where 
$J$ is a closed two-sided ideal generated by $g^2-T(g)g+\zeta\varepsilon(g)$ for all $g\in \gal$. This time we use in an essential way 
that we allow the coefficients in our deformation theory to be non-commutative. We then show that the absolutely irreducible modules of $\wE[1/p]$ are 
at most $2$-di\-men\-sio\-nal, thus using Lemma \ref{multdimI} and Corollaries \ref{redmultI1}, \ref{redmultI2} we obtain that 
if $\Pi$ is absolutely irreducible and $\overline{\Pi}$ contains $\pi$ then $\overline{\Pi}\subseteq \pi^{\oplus 2}$.
The idea to look for $\wE$ of this shape was inspired by \cite{boston}.      

The last case when the block contains $3$ distinct irreducible representations is the hardest one. The new feature here 
is that we need to pass to a certain quotient category for the formalism to work. This reflects that 
the deformation ring on the Galois side is not formally smooth. We invite the reader to look at the introduction to \S \ref{nongenericcaseII} 
for more details.

If $\Theta$ is an open bounded $G$-invariant lattice in an admissible unitary $L$-Banach space representation $\Pi$ of $G$
with a central character $\zeta$ then $\Theta/\varpi \Theta$ is an admissible $k$-representation of $G$ and thus contains 
an irreducible subquotient. After replacing $L$ with a finite extension we may assume that the subquotient is absolutely irreducible
and thus lies in one of the blocks considered above.

A large part of this paper is devoted to calculations of $\Ext$ groups between smooth $k$-representations
of $\GL_2(\Qp)$. These calculations enable us to apply a general formalism developed in \S \ref{firstsec} 
and \S \ref{banach}. This is the technical heart of the paper and where the restrictions on the residual characteristic appear. 
We also use in an essential way that the group is $\GL_2(\Qp)$. There are two $E_2$-spectral sequences at our disposal. One
is obtained from the work of Ollivier \cite{o2} and Vign\'eras \cite{vig} on the functor of 
invariants of the pro-$p$ Iwahori subgroup of $G$, see \S\ref{hecke}. The other is 
due to Emerton \cite{ord2} and is induced by his functor of ordinary parts, see \S \ref{ordinaryparts}.

\subsection{Organization}
The paper essentially consists of two parts: in \S \ref{zerosec}, \S\ref{firstsec} and \S \ref{banach} we develop
a theory which works for any $p$-adic analytic group $G$ provided certain conditions are satisfied; 
in the rest of the paper we show that these conditions are satisfied when $G=\GL_2(\Qp)$ and $p\ge 5$. The appendix 
contains some results on deformation theory of $2$-di\-men\-sio\-nal $\gal$-representations. 

We will now review the sections in more detail. In \S \ref{zerosec} we introduce and recall some facts 
about locally finite categories. In \S \ref{firstsec} we set up a formalism with which we do 
deformation theory with non-commutative coefficients in \S \ref{def}. In section \ref{critcomm} 
we devise a criterion with the help of which one may show that the deformation rings we obtain in \S \ref{def}
are in fact commutative. This criterion will be applied in the generic cases when  $G=\GL_2(\Qp)$, that is when the deformation
ring on the Galois side is formally smooth. In \S \ref{banach} we work out a theory of blocks for admissible
unitary Banach space representations of a $p$-adic analytic group $G$. Using the work of Schneider-Teitelbaum \cite{iw} 
(and Lazard \cite{laz}) one can forget all the functional analytic problems and
the theory works essentially the same way as if $G$ was a finite group. 
This section up to \S \ref{reldef} is independent of \S \ref{firstsec} and the results are somewhat more general 
than outlined in \S \ref{sketch}. In \S\ref{reldef}
we establish a relationship between Banach space representations and the generic fibre of a (possibly non-commutative) ring
$\wE$ representing a deformation problem of \S \ref{def}. In the applications the ring $\wE$ turns out to be a finitely generated 
module over its centre and the centre is a noetherian ring. We show in \S \ref{banach} that when these conditions are satisfied 
we obtain nice finiteness conditions on Banach space representations. Starting from \S \ref{repsGL2},  $G=\GL_2(\Qp)$ and $p\ge 5$.
The sections \ref{supersingularreps}, \ref{nonsupersingularreps}, \ref{nongenericcaseI}, \ref{nongenericcaseII} correspond 
to $\BB$ being as in  the cases (i), (ii), (iii) and (iv) of \S \ref{corrisequiv}. The argument in the generic cases 
is outlined in \S\ref{strat}.

\subsection{A speculation}
It is  known, see for example \cite{bp}, \cite{hu}, \cite{comp},  that if $G\neq \GL_2(\Qp)$ then there are too many representations
of $G$ to have a correspondence with Galois representations. One possible purely speculative scenario to remedy this, would be that 
a global setting, for example a Shimura curve, cuts out a full subcategory  $\Mod^{?}_{G}(\OO)$ of $\Mod^{\mathrm{lfin}}_G(\OO)$, 
closed under direct sums and subquotients and for this subcategory results similar to those described in \S \ref{corrisequiv} 
hold. Moreover, different global settings with the same group $G$ at $p$ would give rise to different subcategories  
$\Mod^{?}_{G}(\OO)$. For this reason 
we have taken great care in \S\ref{firstsec} and \S\ref{banach} to work with an arbitrary $p$-adic analytic group $G$ and 
arbitrary full subcategory $\Mod^{?}_{G}(\OO)$ of $\Mod^{\mathrm{lfin}}_G(\OO)$, closed under direct sums and subquotients.

\textit{Acknowledgements.} I thank Eike Lau, Michael Spie\ss\ and Thomas Zink for numerous discussions on various aspects of this paper.
I thank Ehud de Shalit and Peter Schneider  for inviting me to present some preliminary results at Minerva School on $p$-adic 
methods in Arithmetic Algebraic Geometry, in April 2009 in Jerusalem.  I have benefited from the correspondence with 
Ga\"etan Chenevier on deformation theory of Galois representations and pseudocharacters, and I thank him especially for pointing 
out \cite{joel}. I would like to thank Pierre Colmez for encouraging me to work on this problem. Parts of the paper were written 
at the Isaac Newton Institute during the programme ``Non-Abelian Fundamental Groups in Arithmetic Geometry'' and at the IHP
during the Galois semester. I am indebted to Kevin Buzzard, Toby Gee, Stefano Morra, Peter Schneider and the anonymous referees for pointing out several blunders and 
for their suggestions regarding the exposition.

\section{Notation and Preliminaries}\label{zerosec}
Let $L$ be a finite extension of $\Qp$, with the ring of integers $\OO$, uniformizer $\varpi$, and $k=\OO/\varpi\OO$. Let $G$ be a topological group which is locally pro-$p$. Later on we will 
assume that $G$ is  $p$-adic analytic and the main application will be to $G=\GL_2(\Qp)$ with $p\ge 5$.

Let $(A, \mm)$ be a complete local noetherian $\OO$-algebra  with residue field $k$. We denote 
by $\Mod_G(A)$ the category of $A[G]$-modules, $\Mod^{\mathrm{sm}}_G(A)$ the full subcategory with 
objects $V$ such that 
$$ V=\bigcup_{H, n} V^H[\mm^n],$$
where the union is taken over all open subgroups of $G$ and integers $n\ge 1$ and $V[\mm^n]$ denotes elements of $V$ killed by all elements of $\mm^n$. We will call such representations \textit{smooth}.
Let $\Mod^{\mathrm{l\, fin}}_G(A)$ be a  full subcategory of $\Mod^{\mathrm{sm}}_G(A)$
with objects smooth $G$-representation which are \textit{locally of finite length}, this means 
for every $v\in \pi$ the smallest $A[G]$-submodule of  $\pi$ containing  $v$ is of finite length.
These categories are abelian and  are closed under direct sums, direct limits  and subquotients in $\Mod_G(A)$, that is
if we have an exact sequence $0\rightarrow \pi_1\rightarrow \pi_2\rightarrow \pi_3\rightarrow 0$
in $\Mod_G(A)$ with $\pi_2$ an object of  $\Mod^{\mathrm{l\, fin}}_G(A)$ then $\pi_1$ and $\pi_3$ 
are objects of $\Mod^{\mathrm{l\, fin}}_G(A)$. It is useful to observe:

\begin{lem}\label{useful} Let $\tau$ be an object of $\Mod^{\mathrm{l\, fin}}_G(A)$ and 
$\Hom_{A[G]}(\pi, \tau)=0$ for all irreducible $\pi$ in  $\Mod^{\mathrm{l\, fin}}_G(A)$ then $\tau$ is zero.  
\end{lem}
We note that the lemma fails in $\Mod^{\mathrm{sm}}_G(k)$, for example $\cIndu{\GL_2(\Zp)}{\GL_2(\Qp)}{\Eins}$ does not contain 
any irreducible subrepresentations. In practice, we will work with a variant of the above 
categories by fixing a central character. Let $Z$ be the centre of $G$ and $\zeta: Z\rightarrow A^{\times}$ a continuous character. We will denote by 
$\Mod^{?}_{G,\zeta}(A)$ the full subcategory of $\Mod^{?}_{G}(A)$ consisting of those objects
on which $Z$ acts by a character $\zeta$. If we have a subgroup $H$ of $G$ then the subscript $\zeta$ in 
$\Mod^{?}_{H, \zeta} (A)$ will indicate that  $\Mod^{?}_{H, \zeta} (A)$ is a full 
subcategory of $\Mod^{?}_{H} (A)$  with objects precisely those $\pi$ such that  
$zv=\zeta(z)v$ for all $z\in Z\cap H$ and all $v\in \pi$.

We recall some standard facts about injective and projective envelopes, see \cite[\S II.5]{gab}.
 Let $\mathcal A$ be an abelian category. A monomorphism $\iota:N\hookrightarrow M$ is essential if for every non-zero subobject $M'$ of $M$
we have $\iota(N)\cap M'$ is non-zero. An injective envelope of  an object $M$ in $\mathcal A$ is 
an essential monomorphism $\iota: M\hookrightarrow I$ with $I$ an injective object of $\mathcal A$.  An epimorphism
$q: M\twoheadrightarrow N$ in $\mathcal A$ is essential if for every morphism $s:P\rightarrow M$ 
in $\mathcal A$ the assertion "$qs$ is an epimorphism" implies that $s$ is an epimorphism.   
A projective envelope of an object $N$ of $\mathcal A$ is an essential epimorphism $q: P\twoheadrightarrow N$
with $P$ a projective object in $\mathcal A$. 
If an injective or projective envelope exists then it is unique up to (non-unique) isomorphism. So by abuse 
of language we will forget the morphism and just say $I$ is an injective envelope of $M$ or $P$ is 
a projective envelope of $M$. 

\begin{lem}\label{genind} The categories $\Mod^{\mathrm{sm}}_G(A)$, $\Mod^{\mathrm{sm}}_{G, \zeta}(A)$, 
$\Mod^{\mathrm{lfin}}_{G}(A)$, $\Mod^{\mathrm{lfin}}_{G, \zeta}(A)$ have generators and 
exact inductive limits. 
\end{lem}
\begin{proof} Let 
$X:=\bigoplus_{\mathcal P, n} \cIndu{\mathcal P}{G}{A/\mm^n}$, 
where the sum is taken over all open pro-$p$ groups of $G$ and positive integers $n$ then
for $V$ in $\Mod^{\mathrm{sm}}_G(A)$ we have 
$$\Hom_{A[G]}(\cIndu{\mathcal P}{G}{A/\mm^n}, V)\cong V^{\mathcal P}[\mm^n]$$ 
Hence, 
$\Hom_{A[G]}(X, V)\cong \prod_{\mathcal{P}, n} V^{\mathcal P}[\mm^n].$ Since $V$ is a smooth representation the above isomorphism implies that $X$ 
is a generator for  $\Mod^{\mathrm{sm}}_G(A)$. 

Let $\zeta: Z\rightarrow A^{\times}$ be a continuous character and let 
$\zeta_n: Z\rightarrow (A/\varpi^n A)^{\times}$ be the reduction of $\zeta$ modulo $\mm^n$. 
Since $\zeta$ is continuous given an open pro-$p$ group $\mathcal P$ of $G$ we may find 
an open subgroup $\mathcal P'$ of $\mathcal P$ such that $\zeta_n$ is trivial on $\mathcal P'\cap Z$. 
In this case it makes sense to consider $\zeta_n$ as a character of $Z \mathcal P'$. Let 
$X_{\zeta}:=\bigoplus_{\mathcal P, n} \cIndu{Z \mathcal P}{G}{\zeta_n}$ where the sum is
taken over all $n\ge 1$ and all open pro-$p$ groups $\mathcal P$ of $G$ such that $\zeta_n$ is trivial on
$\mathcal P \cap Z$. Then the same argument as above gives that $X_{\zeta}$  is a generator in 
$\Mod^{\mathrm{sm}}_{G, \zeta}(A)$.

Let $\FF$ (resp. $\FF_{\zeta}$) be the set of quotients of $X$  (resp. $X_{\zeta}$) of finite length. 
Then  $\FF$ (resp. $\FF_{\zeta}$) is a set of generators in $\Mod^{\mathrm{l\, fin}}_{G}(A)$, 
(resp. $\Mod^{\mathrm{l\, fin}}_{G, \zeta}(A)$). 

It is clear that all the categories have inductive limits. 
The exactness of inductive limits follows from \cite{gab} Proposition I.6 (b). 
\end{proof}

\begin{cor}\label{enoughinj}The categories in Lemma \ref{genind} have injective envelopes.
\end{cor}
\begin{proof} Every object in a category with generators and exact inductive limits 
has an injective envelope, see Theorem 2 in \cite[\S II.6]{gab}.
\end{proof}

\begin{lem} The categories  $\Mod^{\mathrm{lfin}}_G(A)$ and $\Mod^{\mathrm{lfin}}_{G, \zeta}(A)$
are  locally finite. 
\end{lem}
\begin{proof} Both categories have a set of generators which are of finite length, namely 
$\FF$ and $\FF_{\zeta}$ constructed in the proof of Lemma \ref{genind}. Hence they are locally finite,
see \S II.4 in \cite{gab} for details. 
\end{proof} 

An object $V$ of $\Mod^{\mathrm{sm}}_G(A)$ is called \textit{admissible} if $V^H[\mm^i]$ is a finitely generated  
$A$-module for every open subgroup $H$ of $G$ and every $i\ge 1$; $V$ is called \textit{locally admissible} if 
for every $v\in V$ the smallest $A[G]$-submodule of $V$ containing $v$ is admissible. Let $\Mod^{\mathrm{l\, adm}}_G(A)$ be a full subcategory of $\Mod^{\mathrm{sm}}_G(A)$ 
consisting of locally admissible representations. 
Emerton in \cite{ord1} shows that if $G$ is $p$-adic analytic then 
$\Mod^{\mathrm{l\, adm}}_G(A)$ is abelian. Moreover, it follows from \cite[Thm 2.3.8]{ord1} 
that if $G=\GL_2(\Qp)$ or $G$ is a torus then $\Mod^{\mathrm{lfin}}_{G,\zeta}(A)=\Mod^{\mathrm{l\, adm}}_{G, \zeta}(A)$.
If the conjecture \cite[ 2.3.7]{ord1} holds then we would obtain this result in general.

Let $H$ be a compact open subgroup of $G$ and $A[[H]]$ the completed group algebra of $H$. Let 
$\Mod^{\mathrm{pro \, aug}}_G(A)$ be  the category of profinite linearly topological 
$A[[H]]$-modules with an action 
of $A[G]$ such that the two actions are the same when restricted to $A[H]$ with morphisms  
$G$-equivariant continuous homomorphisms of topological $A[[H]]$-modules. Since any two compact open 
subgroups of $G$ are commensurable the definition does not depend on the choice of $H$. Taking  Pontryagin duals induces an anti-equivalence of categories between $\Mod^{\mathrm{sm}}_G(A)$
and $\Mod^{\mathrm{pro \, aug}}_G(A)$, see Lemma 2.2.7 in \cite{ord1}. By Pontryagin dual we mean 
$$M^{\vee}:=\Hom^{cont}_{\OO}(M, L/\OO),$$
where $L/\OO$ carries discrete topology and $M^{\vee}$ is equipped with compact open topology. 
We have a canonical isomorphism $M^{\vee \vee}\cong M$. 
  
We note that the duality reverses the arrows, and so if $\Mod^?_{G}(A)$ is a full abelian subcategory 
of $\Mod^{\mathrm{sm}}_G(A)$  then we may define a full subcategory $\dualcat(A)$ of 
$\Mod^{\mathrm{pro \, aug}}_G(A)$ by taking the objects to be all $M$ isomorphic 
to $\pi^{\vee}$ for some object $\pi$ of $\Mod^?_{G}(A)$. The category $\dualcat(A)$ is abelian and 
if $\Mod^?_{G}(A)$ has exact inductive limits and injective envelopes then 
 $\dualcat(A)$ has exact projective limits and projective envelopes.  

Let $\Mod^?_{G}(A)$ be a full subcategory of $\Mod^{\mathrm{l\, fin}}_G(A)$ closed under arbitrary direct sums and 
subquotients in $\Mod^{\mathrm{l\, fin}}_G(A)$. Since $\Mod^{\mathrm{l\, fin}}_G(A)$ has exact inductive limits so does 
$\Mod^?_{G}(A)$. Moreover, $\Mod^?_{G}(A)$ has a set of generators of finite length, one may just take a subset of 
$\FF$ constructed in the proof of Lemma \ref{genind} consisting of objects that lie in $\Mod^?_{G}(A)$. Hence, 
$\Mod^?_{G}(A)$ is locally finite and has injective envelopes. We may define a full subcategory $\dualcat(A)$ of 
$\Mod^{\mathrm{pro \, aug}}_G(A)$ by taking the objects to be all $M$ isomorphic 
to $\pi^{\vee}$ for some object $\pi$ of $\Mod^?_{G}(A)$. The category $\dualcat(A)$ is anti-equivalent 
to $\Mod^?_{G}(A)$. In particular, it is abelian, has exact projective limits and projective envelopes. 

Let $\pi_1, \ldots, \pi_n$ be  irreducible, pairwise non-isomorphic objects in $\Mod^?_{G}(A)$ and let $\iota: \pi_i\hookrightarrow J_i$ 
be an injective envelope
of $\pi_i$ in $\Mod^?_{G}(A)$. Let $S_i:=\pi_i^{\vee}$, $P_i:=J_i^{\vee}$ and $\kappa:= \iota^{\vee}$ then $\kappa: P_i\twoheadrightarrow S_i$ 
is a projective envelope of $S_i$ in $\dualcat(A)$. We put $\pi:=\oplus_{i=1}^n \pi_i$ then $J:=\oplus_{i=1}^n J_i$ is an injective envelope
of $\pi$ and $P:=J^{\vee}$ is a projective envelope of $S:=\pi^{\vee}\cong \oplus_{i=1}^n S_i$ in $\dualcat(A)$.
Let 
$$ E:=\End_{\dualcat(A)}(P).$$
Each quotient $q:P\twoheadrightarrow M$  defines a right ideal of $E$:
\begin{equation}\label{defrM}
\mathfrak r(M):=\{\phi\in E: q\circ \phi=0\}.
\end{equation}  
We define the natural topology on $E$ by taking $\mathfrak r (M)$ with $M$ of finite length to be a basis of open neighbourhoods 
of $0$ in $E$.  With respect to the natural topology $E$ is a pseudo-compact ring, see Proposition 13 in \cite[\S IV.4]{gab}.
Moreover, $\mm:= \mathfrak r(S)$ is the Jacobson radical of $E$ and 
\begin{equation}\label{goodpoint}
E/\mm\cong \End_{\dualcat(A)}(S)\cong \prod_{i=1}^n \End_{\dualcat(A)}(S_i),
\end{equation}
see Proposition 12 in \cite[\S IV.4]{gab} for the first isomorphism; the second holds since $\pi_i$ are irreducible and distinct. 
Since $\pi_i$ is  irreducible  $\End_{\dualcat(A)}(S_i)$
is a skew field over $k$. We assume for simplicity that $\End_G(\pi_i)$ is finite dimensional for $1\le i \le n$.
This holds if $\pi_i$ are admissible. Since $k$ is a finite field, $k_i:=\End_{\dualcat(A)}(S_i)$ is a finite field extension of $k$. 
Hence, $E/\mm$ is a finite dimensional $k$-vector space  and, since 
$k$ is assumed to be finite, $E$ is a compact ring. Thus all the pseudo-compact modules of $E$ will be in fact compact.

\begin{cor}\label{unitsE} If $S$ is irreducible then every $\alpha\in E$,  $\alpha \not\in \mm$ is a unit in $E$.
\end{cor}
\begin{proof} Since $S$ is irreducible, it follows from \eqref{goodpoint} that $\mm$ is maximal. 
On the other hand, $\mm$  is also the Jacobson radical of $E$ by Proposition 12 in \cite[\S IV.4]{gab}.
Hence, $E$ is a local ring.
\end{proof} 

\begin{cor}\label{Zloc} If $S$ is irreducible then the centre $\mathcal Z$ of $E$ is a local ring with residue field a finite extension of $k$.
\end{cor}
\begin{proof} Let $\mm$ be the maximal ideal of $E$. Let 
$a\in \mathcal Z$ such that  $a\not\in \mm$ . It follows from Lemma \ref{unitsE} that $a$ is a unit in $E$. However, 
this implies that $a$ is a unit in $\mathcal Z$ as for any $c\in E$ we have 
$$ a^{-1}c-ca^{-1}= (a^{-1}c -c a^{-1}) a a^{-1}= ( c - c ) a^{-1}= 0,$$
as $a$ lies in the centre of $E$. Hence $(\mathcal Z, \mathcal Z \cap \mm)$ is a local ring. 
The last assertion follows since $\End_{\dualcat(A)}(S)$ is a finite extension of $k$
and we have injections $\OO/\varpi \OO\hookrightarrow \mathcal Z/(\mathcal Z \cap \mm)\hookrightarrow E/\mm$.
\end{proof}

\begin{lem} $P$ is a left pseudo-compact $E$-module. 
\end{lem}
\begin{proof} We will show that there exists a basis of open neighbourhoods of $0$ in $P$ consisting of 
left $E$-submodules, such that the quotient is an $E$-module of finite length. Now $P$ is a pseudo-compact $A$-module, since it is 
a Pontryagin dual of a discrete $A$-module, thus it is enough to show that every open $A$-submodule $M$ of $P$
contains an open left $E$-submodule.  Since $M$ is open, the quotient  $P/M$ is an $A$-module of finite
length, and hence 
$$\tau:= A[G]\centerdot (P/M)^{\vee}\subset P^{\vee}$$ 
is a smooth representation of $G$ of finite length. Dualizing back, we obtain a factorisation 
$P\twoheadrightarrow \tau^{\vee}\twoheadrightarrow P/M$. Then $\mathfrak r(\tau^{\vee})$ is an open right ideal in $E$ 
for the natural topology. 

Since $E$ with  the natural topology is a pseudo-compact ring, $E/\mathfrak r(\tau^{\vee})$ is a right $E$-module 
of finite length. Since $E$ modulo its Jacobson radical is a finite dimensional $k$-vector space by assumption, we may choose
$\phi_1, \ldots, \phi_m\in E$ such that $\phi_1+ \mathfrak r(\tau^{\vee}), \ldots,  \phi_m+ \mathfrak r(\tau^{\vee})$ 
generate $E/\mathfrak r(\tau^{\vee})$ as an $A$-module. We may assume that $\phi_1$ is the identity map. 
We claim that 
\begin{equation}\label{showopen}
\{v\in P: \phi(v)\in M, \forall \phi\in E\}=\bigcap_{i=1}^m \phi_i^{-1}(M).
\end{equation}
The left hand side of \eqref{showopen} is equal to $\bigcap_{\phi\in E} \phi^{-1}(M)$ and so is contained in the 
right hand side. Since $P\rightarrow P/M$ factors through $q: P\rightarrow \tau^{\vee}$ the kernel of $q$ is 
contained in $M$. Hence for all $\psi\in \mathfrak r(\tau^{\vee})$ and all $v\in P$ we have $\psi(v)\in M$. 
Since $M$ is an $A$-module and every $\phi\in E$ may be written as $\phi=\sum_{i=1}^m \lambda_i \phi_i + \psi$, where
$\lambda_i\in A$ and $\psi\in \mathfrak r(\tau^{\vee})$, we get the opposite inclusion. 

The right hand side of \eqref{showopen} is an open $A$-submodule of $M$ and the left hand side is a left $E$-module.
Hence, $P$ is a pseudo-compact $E$-module.
\end{proof}    

Let $\md$ be a right pseudo-compact $E$-module, for definition and properties see \S IV.3 of \cite{gab}.
Let $\{\md_i\}_{i\in I}$ be a basis of open neighbourhoods of $0$ in $\md$ consisting 
of right $E$-modules and let $\{P_j\}_{j\in J}$ be a basis of open neighbourhoods of $0$ in $P$ consisting 
of left $E$-modules. We define the completed tensor product
\begin{equation}\label{mwt}
 \md\wtimes_E P:= \underset{\longleftarrow}{\lim}\, (\md/\md_i)\otimes_E (P/P_j),
\end{equation}
where the limit is taken over $I\times J$. Since $\md/\md_i$ and $P/P_j$ are $E$-modules of finite length and 
$E$ modulo its Jacobson radical is a finite dimensional $k$-vector space, $\md/\md_i$ and $P/P_j$ are $A$-modules of finite length and
 hence the limit exists in the category of pseudo-compact $A$-modules. By the universality of tensor product we have a natural 
map $\md\otimes_E P\rightarrow \md\wtimes_E P$, and we denote the image of $m\otimes v$ by $m\wtimes v$.

\begin{lem} $\md\wtimes_E P$ is an object of $\dualcat(A)$. 
\end{lem}
\begin{proof} It follows directly from \eqref{mwt} that $\md\wtimes_E P$ is  a pseudo-compact $A[[H]]$-module. Since $G$ acts on $P$ by continuous $E$-linear homomorphisms, it 
follows from the universal property of the completed tensor product, see \cite[\S2]{bru}, that for
each pseudo-compact right $E$-module $\md$ we obtain an action of $G$ on $\md\wtimes_E P$ by continuous, $A$-linear  homomorphisms. Moreover, the action of 
$A[G]$ and $A[[H]]$ induce the same action of $A[H]$. Hence, $\md\wtimes_E P$ is an object of $\Mod^{\mathrm{pro\,aug}}_G(A)$.

If $\md=E$ then $E\wtimes_E P\cong P$ and so $\md\wtimes_E P$ is an object of $\dualcat(A)$.
The functor $\wtimes_E P$ is right exact and commutes with direct products, see  
\cite[Exp $VII_B$]{SGA3}, \cite[Lem.\ A.4]{bru}. Hence, if $\md\cong \prod_{i\in I} E$ for some set $I$ then 
$\md\wtimes_E P\cong \prod_{i\in I} P$. Since direct products exist in $\dualcat(A)$ we deduce 
that $\md\wtimes_E P$ is an object of $\dualcat(A)$. In general, we have an exact sequence of $E$-modules  
$\prod_{i\in I} E \rightarrow \prod_{j\in J} E\rightarrow \md\rightarrow 0$
for some sets $I$ and $J$. Since $\wtimes_E P$ is right exact we deduce that $\md\wtimes_E P$ is the cokernel 
of $\prod_{i\in I} P\rightarrow \prod_{j\in J} P$ and hence is an object of $\dualcat(A)$.   
\end{proof}

Since $P\twoheadrightarrow S$ is essential we have $\Hom_{\dualcat(A)}(P, S')=0$ for all irreducible objects
of $\dualcat(A)$ not isomorphic to $S_i$ for $1\le i\le n$, and 
$$\Hom_{\dualcat(A)}(P, S)\cong \End_{\dualcat(A)}(S)\cong \prod_{i=1}^n k_i.$$
Thus if $M$ is an object of $\dualcat(A)$ of finite length then $\Hom_{\dualcat(A)}(P, M)$ 
is a right $E$-module of finite length. If $M$ is any object of $\dualcat(A)$ then 
we may write $M= \underset{\longleftarrow}{\lim}\, M_i$ with $M_i$ of finite length, and 
$\Hom_{\dualcat(A)}(P, M)\cong \underset{\longleftarrow}{\lim} \,\Hom_{\dualcat(A)}(P, M_i)$ is a pseudo-compact  $E$-module.
Let us also note that, since $E\otimes_E P\cong P \cong E\wtimes_E P$, for any finitely presented right pseudo-compact 
$E$-module $\md$ we have an isomorphism $\md\otimes_E P\cong \md\wtimes_E P$.

\begin{lem}\label{headS0} If $\md$ is a pseudo-compact right $E$-module then 
$$\Hom_{\dualcat(A)}(P, \md \wtimes_E P)\cong \md.$$
\end{lem}
\begin{proof} If $\md\cong \prod_{i\in I} E$ for some set $I$ then $\md \wtimes_E P\cong \prod_{i\in I} P$ 
and hence 
$$\Hom_{\dualcat(A)}(P, \md \wtimes_E P)\cong \Hom_{\dualcat(A)}(P, \prod_{i\in I} P)\cong \prod_{i\in I} E\cong \md.$$
In general, we have an exact sequence of $E$-modules  $\prod_{i\in I} E \rightarrow \prod_{j\in J} E\rightarrow \md\rightarrow 0$
for some sets $I$ and $J$. Applying $\wtimes_E P$ and then $\Hom_{\dualcat(A)}(P, \ast)$ to it we get the assertion.
\end{proof}

\begin{lem}\label{headS} If $M$ is in 
$\dualcat(A)$ such that $\Hom_{\dualcat(A)}(M, S')=0$ for all irreducible $S'$ not isomorphic to $S_i$ for $1\le i \le n$, 
then the natural map 
\begin{equation} 
\Hom_{\dualcat(A)}(P, M)\widehat{\otimes}_E P\rightarrow M
\end{equation}
is surjective.
\end{lem}
\begin{proof} Let $C$ be the cokernel. Lemma \ref{headS0} and the projectivity of $P$ implies that $\Hom_{\dualcat(A)}(P, C)=0$. 
The exactness of $\Hom_{\dualcat(A)}(P, \ast)$  implies that $\Hom_{\dualcat(A)}(C, S)=0$. Since $C$ is a quotient of $M$
this implies that $\Hom_{\dualcat(A)}(C, S')=0$ for all irreducible objects of $\dualcat(A)$. This implies  
$C=0$ by Lemma \ref{useful}.
\end{proof}

\begin{lem}\label{reduceprojenv} Let $\dualcat(\OO/\varpi^n \OO)$ be the full subcategory of $\dualcat(\OO)$ consisting of objects 
killed by $\varpi^n$, let $M$ be an object of $\dualcat(\OO/\varpi^n \OO)$ and let $q:\wP\twoheadrightarrow M$ 
be a projective envelope of $M$ in $\dualcat(\OO)$, then $\wP/\varpi^n\wP\twoheadrightarrow M$ is 
a projective envelope of $M$ in $\dualcat(\OO/\varpi^n \OO)$.
\end{lem} 
\begin{remar} We note that Pontryagin duality identifies $\dualcat(\OO/\varpi^n\OO)$ with the full subcategory of $\Mod^?_G(\OO)$ consisting of objects 
killed by $\varpi^n$.
\end{remar}
\begin{proof} Let $q_n:P\twoheadrightarrow M$ be a projective envelope of $M$ in $\dualcat(\OO/\varpi^n\OO)$.
Since $\wP$ is projective and $q_n$ is essential there exists $\phi: \wP\rightarrow P$ such that $q=q_n \circ \phi$.
Since $\varpi^n$ kills $P$, $\phi$ factors through $\wP/\varpi^n \wP$, which lies in $\dualcat(\OO/\varpi^n \OO)$.
Since $P$ is projective in this category, the surjection splits and we have $\wP/\varpi^n \wP\cong P\oplus N$. Let 
$\psi: \wP\rightarrow N$ be the natural map, then the composition 
$\Ker \psi\hookrightarrow  \wP\twoheadrightarrow M$ is surjective. As $q$ is essential we get $\Ker \psi=\wP$ and 
hence $N=0$, which gives the claim.
\end{proof}     

\section{The formalism}\label{firstsec}
Let $\dualcat$ be a full abelian subcategory of $\Mod^{\mathrm{pro\, aug}}_G(\OO)$  closed under 
direct products and subquotients in $\Mod^{\mathrm{pro\, aug}}_G(\OO)$. Note that this implies that 
$\dualcat$ is closed under projective limits. We further assume that every object $M$ of 
$\dualcat$ can be written as $M\cong \underset{\longleftarrow}{\lim} \, M_i$, where the 
limit is taken over all the quotients of finite length. In the sequel $\dualcat$ will be either 
the category $\dualcat(\OO)$ or its full subcategory $\dualcat(k)$, introduced in \S \ref{zerosec}.

Dually this means that $M^{\vee}$ 
is an object of $\Mod^{\mathrm{l\, fin}}_G(\OO)$. We denote by $\Irr(\dualcat)$ the set of equivalence classes of irreducible 
objects in $\dualcat$ and note that the last assumption implies that if $M$ is an object of $\dualcat$ 
and $\Hom_{\dualcat}(M, S')=0$ for all $S'\in\Irr(\dualcat)$ then $M$ is zero. We denote by $\rad M$ the intersection of all maximal proper subobjects of $M$.

Let $S$ be an irreducible object of $\dualcat$ with $\End_{\dualcat}(S)=k$.  We assume the existence  of an object $Q$ in $\dualcat$ of finite length, satisfying the following hypotheses:
\begin{itemize}
\item[(H1)] $\Hom_{\dualcat}(Q, S')=0$,  $\forall S'\in \Irr(\dualcat)$,  $S\not\cong S'$;
\item[(H2)] $S$ occurs as a subquotient in $Q$ with multiplicity $1$;
\item[(H3)] $\Ext^1_{\dualcat}(Q, S')=0$, $\forall S'\in \Irr(\dualcat)$,  $S\not\cong S'$;
\item[(H4)] $\Ext^1_{\dualcat}(Q, S)$ is finite dimensional;
\item[(H5)] $\Ext^2_{\dualcat}(Q, R)=0$, where $R=\rad Q$ is the maximal proper subobject of $Q$.
\end{itemize}

We note that we will introduce an additional hypothesis (H0) in Proposition \ref{H00} below.  
Hypotheses (H1) and (H2) imply that $\Hom_{\dualcat}(Q, S)$ is one dimensional and that 
 $Q$ has a unique maximal subobject and if we choose 
a non-zero $\varphi: Q\rightarrow S$ then we obtain an exact sequence:
\begin{equation}\label{R}
0\rightarrow R \rightarrow Q\overset{\varphi}{\rightarrow} S\rightarrow 0.
\end{equation}  
We note that $\varphi$ is essential. Since if we have $\psi: A\rightarrow Q$ such that 
$\varphi\circ \psi$ is surjective, then $\Hom_{\dualcat}(\Coker \psi, S')=0$ for all 
$S'\in \Irr(\dualcat)$ and so $\Coker \psi =0$. 

\begin{lem}\label{first} Equation \eqref{R} induces an isomorphism
\begin{equation}\label{iso1}
\Ext^1_{\dualcat}(Q, Q)\cong \Ext^1_{\dualcat}(Q, S)
\end{equation}
and an injection
\begin{equation}\label{iso2}
\Ext^2_{\dualcat}(Q, Q)\hookrightarrow \Ext^2_{\dualcat}(Q, S).
\end{equation}
\end{lem}
\begin{proof} We apply $\Hom_{\dualcat}(Q, \ast)$ to \eqref{R}. The injectivity of \eqref{iso2} and
surjectivity of \eqref{iso1} follows from (H5). To show the injectivity of \eqref{iso1} it is enough 
to show that $\Ext^1_{\dualcat}(Q, R)=0$. However, a more general statement follows from (H3). Namely, 
if $R'$ is of finite length and $S$ is not a subquotient of $R'$ then $\Ext^1_{\dualcat}(Q, R')=0$. 
One argues by induction on the number of irreducible subquotients of $R'$.
\end{proof}

\begin{lem}\label{T} Let $T\in \dualcat$ be of finite length and suppose that $T$ has a filtration by subobjects $T^i$, 
such that $T^0=T$ and $T^i/T^{i+1}\cong Q^{\oplus n_i}$, for $i\ge 0$. Then \eqref{R} induces an isomorphism:
\begin{equation}\label{iso3}
\Ext^1_{\dualcat}(T, Q)\cong \Ext^1_{\dualcat}(T, S).
\end{equation}
Moreover, $\Ext^1_{\dualcat}(T, S')=0$ for all $S'\in \Irr(\dualcat)$, $S'\not\cong S$.
\end{lem} 
\begin{proof}  By devissage and (H3) we have $\Ext^1_{\dualcat}(T, S')=0$ for all $S'\in \Irr(\dualcat)$, $S'\not\cong S$.
Since (H2) implies that $S$ is not a subquotient of $R$, we deduce by devissage that $\Ext^1_{\dualcat}(T, R)=0$. 
Further, devissage and (H5) imply that $\Ext^2_{\dualcat}(T, R)=0$. Thus applying $\Hom_{\dualcat}(T, \ast)$ to 
to \eqref{R} we obtain the isomorphism \eqref{iso3}.
\end{proof}
 
Let $P\overset{\kappa}{\twoheadrightarrow} S$ be a projective envelope of $S$ in $\dualcat$. Note that
since $\kappa$ is essential we have $\Hom_{\dualcat}(P, S')=0$ for all $S'\in \Irr(\dualcat)$, $S'\not\cong S$, and
$\dim \Hom_{\dualcat}(P, S)=1$. Since $P$ is projective the functor $\Hom_{\dualcat}(P, \ast)$ is exact, and thus we get:

\begin{lem}\label{mult=dim} Let $T\in \dualcat$ be of finite length. Then the length of $\Hom_{\dualcat}(P, T)$ as an 
$\OO$-module is equal to the multiplicity, with which $S$ occurs as a subquotient of $T$.
\end{lem}

We note that since $Q/\rad Q\cong S$ is irreducible and $S$ occurs in $Q$ with multiplicity $1$, 
every non-zero  $\phi\in \Hom_{\dualcat}(Q, Q)$ is an isomorphism. This implies that $Q$ is killed by $\varpi$.  
It follows from Lemma \ref{mult=dim} that (H2) could be reformulated as 
\begin{itemize} 
\item[(H2')] $\dim \Hom_{\dualcat}(P, Q)=1$.
\end{itemize} 

Since 
$\varphi: Q\twoheadrightarrow S$ is surjective and $P$ is projective, 
there exists $\theta: P\rightarrow Q$ such that $\varphi\circ \theta=\kappa$. Moreover, since $\varphi$ is essential, 
$\theta$ is surjective.

\begin{lem}\label{filtrationP} There exists a unique decreasing filtration of $P$ by subobjects $P^i$ such that $P=P^0$, 
$P^i/P^{i+1}\cong Q^{\oplus n_i}$, where $n_i\ge 1$, for all $i\ge 0$, and every $\phi: P^i\rightarrow Q$ 
factors through $P^i/P^{i+1}$.
\end{lem}
\begin{proof} If such filtration exists then it is unique as $P^0=P$ and $P^{i+1}=\bigcap_{\phi}\Ker \phi$, 
where the intersection is taken over all $\phi\in \Hom_{\dualcat}(P^i, Q)$. Since $\Hom_{\dualcat}(P, Q)$ 
is $1$-di\-men\-sio\-nal we get that  $P^1:=\Ker \theta$ satisfies the conditions. Suppose that we have defined $P^i$, 
for $0\le i\le n$. Consider the exact sequence:
\begin{equation}\label{defpn}
0\rightarrow P^n \rightarrow P\rightarrow P/P^n\rightarrow 0
\end{equation} 
Let $S'\in \Irr(\dualcat)$, we apply $\Hom_{\dualcat}(\ast, S')$ to \eqref{defpn} to get an isomorphism
$\Hom_{\dualcat}(P^n, S')\cong \Ext^1_{\dualcat}(P/P^n, S')$. Since $P$ is projective 
$\Ext^1_{\dualcat}(P, \ast)=0$. We may apply Lemma \ref{T} to $T=P/P^n$ and $T^i= P^i/P^n$. We get 
\begin{equation}\label{headpn} 
\Hom_{\dualcat}(P^n, S')=0, \quad \forall S'\in \Irr(\dualcat),\quad  S'\not\cong S.
\end{equation} 
Moreover, (H4) implies that $d:=\dim \Hom_{\dualcat}(P^n, S)=\dim \Ext^1_{\dualcat}(P/P^n, S)$ is finite. Hence, 
\begin{equation}\label{realhead}
P^n/\rad P^n \cong S^{\oplus d}.
\end{equation}
 We define  
$\phi_i: P^n\rightarrow P^n/\rad P^n \rightarrow S$, where the last map is projection to the $i$-th  component.
So $\phi_i$ form a basis of $\Hom_{\dualcat}(P^n, S)$. We apply 
$\Hom_{\dualcat}(\ast, Q)$ and $\Hom_{\dualcat}(\ast, S)$ to \eqref{defpn} to  get a commutative diagram
\begin{displaymath}
\xymatrix@1{ \Hom_{\dualcat}(P^n, Q)\ar[d]\ar[r]^-{\cong} & \Ext^1_{\dualcat}(P/P^n, Q)\ar[d]^{\cong}_{\eqref{iso3}}\\
             \Hom_{\dualcat}(P^n, S) \ar[r]^-{\cong} &  \Ext^1_{\dualcat}(P/P^n, S).}
\end{displaymath}
The second vertical arrow is an isomorphism by Lemma \ref{T}. Hence the first vertical arrow is an isomorphism. 
Hence there exists $\psi_i\in \Hom_{\dualcat}(P^n, Q)$, such that $\varphi\circ \psi_i= \phi_i$. Then $\psi_i$ form 
a basis of $\Hom_{\dualcat}(P^n, Q)$. Let $\theta_n: P^n\rightarrow Q^{\oplus d}$ be the map $v\mapsto (\psi_1(v), \ldots, \psi_d(v))$. 
We have a commutative diagram:
   
\begin{displaymath}
\xymatrix@1{   & Q^{\oplus d}\ar@{->>}[d]^-{\varphi^{\oplus d}}\\
              P^n  \ar[ur]^-{\theta_n}\ar@{->>}[r] &  P^n /\rad P^n.}
\end{displaymath}
Since the vertical arrow is essential, we get that $\theta_n$ is surjective, and define $P^{n+1}:= \Ker \theta_n$. Then 
\begin{equation}\label{pn1}
P^n/P^{n+1}\cong  Q^{\oplus d}
\end{equation}
where $d=\dim \Ext^1_{\dualcat}(P/P^n, S)$. Moreover, we have 
$$ P^{n+1}=\bigcap_{i=1}^d \Ker \psi_i= \bigcap_{\psi\in \Hom_{\dualcat}(P^{n}, Q)} \Ker \psi, $$
since $\psi_i$ form a basis of $\Hom_{\dualcat}(P^{n}, Q)$.
\end{proof}  

\begin{lem}\label{Plimit} The natural map $P\rightarrow \underset{\longleftarrow}{\lim}\, P/P^n$ is an isomorphism.
\end{lem} 
\begin{proof} Let $\FF$ be a set of quotients of $P$ in $\dualcat$ of finite length. 
Since $P$ is an object of $\dualcat$, we have $P\cong \underset{\longleftarrow}{\lim}\, N$, where the limit is 
taken over all $N\in \FF$.
Since $P/P^n$ are of finite length, it is enough 
to show that for every quotient $q:P\twoheadrightarrow N$ of finite length there exists $n$ such that $P^n$ is contained in the 
kernel of $q$. Let $N$ be a counterexample of minimal length $m$. If $N$ is irreducible then,  as 
$\kappa: P\twoheadrightarrow S$ is essential, we get that $N\cong S$ and $q=\lambda \kappa$ for some $\lambda\in k$. But then 
$P^1$ is contained in the kernel of $q$. Hence, $m>1$ and we may consider an exact sequence 
$0\rightarrow S' \rightarrow N\rightarrow N'\rightarrow 0$, with $S'$ irreducible and $N'$ non-zero. The minimality 
of $m$ implies that there exists  $n$ such that  $P^n$ is contained in the kernel of $q':P\rightarrow N\rightarrow N'$. 
Since by assumption $P^n$ is not contained in the kernel of $q$, we obtain a non-zero map $q: P^n \rightarrow S'$. Since $S'$ is 
irreducible, $\rad P^n$ is contained in the kernel. As by construction $P^{n+1}$ is contained in $\rad P^n$ we obtain 
a contradiction.
\end{proof}

\begin{defi}\label{Em} Let $P\overset{\kappa}{\twoheadrightarrow} S$ be a projective envelope of $S$ in $\dualcat$.
We let 
$$E:= \End_{\dualcat}(P), \quad  \mm:= \{\phi\in E: \kappa\circ \phi=0\}.$$
\end{defi}
A priori $\mm$ is only a right ideal of $E$.  
Since $P$ is projective we get  a surjection $\Hom_{\dualcat}(P, P)\twoheadrightarrow \Hom_{\dualcat}(P, S)$. Now 
$\dim \Hom_{\dualcat}(P, S)=1$, and hence any $\phi\in E$ may be written as $\phi=\lambda +\psi$, where 
$\lambda\in \OO$ and $\psi\in \mm$. Since the image of $\OO$ lies in the centre of $E$, this implies that 
$\mm$ is a two-sided ideal and $E/\mm \cong \End_{\dualcat}(S)\cong k$.

\begin{lem}\label{inclusion} We have $\mm^n P \subseteq P^n$ for $n\ge 0$ and  $\mm P=P^1$, so that $P/\mm P\cong Q$.
\end{lem} 
\begin{proof} Recall that $\kappa:P\twoheadrightarrow S$ factors through $\theta: P\twoheadrightarrow Q$. If $\phi\in \mm$ then 
$\theta\circ \phi$ maps $P$ to $R=\rad Q$.  Since $\Hom_{\dualcat}(P, R)=0$, we obtain
$\theta\circ \phi=0$. Thus 
\begin{equation}\label{mtheta}
\mm=\{\phi\in E: \theta\circ \phi=0\}.
\end{equation}
Hence, $\mm P\subseteq P^1=\ker \theta$. We fix $n\ge 1$ and we claim that if  $i\le n$ then for all 
 $\psi\in \Hom_{\dualcat}(P, P/P^i)$ and $\phi\in \mm^n$ we have $\psi\circ \phi=0$. The claim 
applied to the natural map $P\rightarrow P/P^n$ implies $\mm^n P\subseteq P^n$.
We argue by induction on $i$. If $i=1$ then $P/P^1\cong Q$, $\Hom_{\dualcat}(P, Q)$ is $1$-di\-men\-sio\-nal, spanned by $\theta$, and 
so the claim follows from \eqref{mtheta}. In general we have an exact sequence:
\begin{equation}\label{sequenceyeah}
0\rightarrow \Hom_{\dualcat}(P, P^{i-1}/P^{i})\rightarrow \Hom_{\dualcat}(P, P/P^{i})\rightarrow \Hom_{\dualcat}(P, P/P^{i-1})\rightarrow 0.
\end{equation}
Let $\psi\in \Hom_{\dualcat}(P, P/P^{i})$, $\phi_1\in\mm^{n-1}$ and  $\phi_2\in \mm$. The image of $\psi\circ \phi_1$ in 
$\Hom_{\dualcat}(P, P/P^{i-1})$ is zero 
by the induction hypothesis. Hence, $\psi\circ \phi_1$ induces a map from $P$ to $P^{i-1}/P^i\cong Q^{\oplus d}$. Hence, 
$\psi\circ\phi_1\circ \phi_2=0$, as $\phi_2\in \mm$. Now any $\phi\in \mm^n$ can be written as a linear combination of $\phi_1\circ \phi_2$ 
with $\phi_1\in\mm^{n-1}$ and $\phi_2\in \mm$. Hence, $\psi\circ\phi=0$. 

We know, see \eqref{realhead}, that $P^1/\rad P^1 \cong S^{\oplus d}$. Hence, there exists a surjection 
$P^{\oplus d}\twoheadrightarrow P^1$. For $1\le i\le d$ let $X_i: P\rightarrow P^{\oplus d}\rightarrow P^1\hookrightarrow P$ 
denote the composition, where the first map is inclusion to the $i$-th component. Then $X_i\in E$ and $\kappa\circ X_i=0$. 
So $X_i\in \mm$ and $P^1=\sum_{i=1}^d X_i P\subseteq \mm P$.  
\end{proof}

\begin{prop}\label{filtdone} For $n\ge 0$ we have:
\begin{itemize}
\item[(i)] $\mm^n P=P^n$;
\item[(ii)] the natural map  $\mm^n \rightarrow \Hom_{\dualcat}(P, \mm^nP)$ is an isomorphism;
\item[(iii)] $\dim \mm^{n+1}/\mm^{n+2}=\dim \Hom_{\dualcat}(P^{n+1}, S)=\dim \Ext^1_{\dualcat}(P/P^{n+1}, S)$;
\item[(iv)] the natural map $\mm^n/\mm^{n+1}\wtimes_E P\rightarrow \mm^n P/\mm^{n+1}P$ is an isomorphism. 
\end{itemize}
 \end{prop}
\begin{proof} We prove (i) and (ii) by induction on $n$, and obtain (iii) and (iv) as by-products
of the proof. We note (i) and (ii) hold trivially for $n=0$. Suppose that (i) and (ii) hold for $n$.
Let $d:= \dim \Hom_{\dualcat}(P^n, S)$ then $P^n/\rad P^n \cong S^{\oplus d}$, see \eqref{realhead}. Since 
$\mm^n P=P^n$ we get a surjection $P^{\oplus d} \twoheadrightarrow \mm^n P$. For $1\le i\le d$ let 
$$X_i: P\rightarrow P^{\oplus d}\rightarrow \mm^nP\hookrightarrow P$$ 
denote the composition, where the first map is the inclusion to the $i$-th component. Then $X_i\in E$ and 
(ii) implies that $X_i\in\mm^n$. Suppose that $\phi\in \mm$ then $X_i\circ\phi \in \mm^{n+1}$ so the surjection 
$P^{\oplus d}\twoheadrightarrow \mm^nP\twoheadrightarrow \mm^n P/ \mm^{n+1}P$ factors through 
\begin{equation}\label{1surj}
 Q^{\oplus d}\cong (P/\mm P)^{\oplus d}\twoheadrightarrow \mm^n P/\mm^{n+1}P
\end{equation}
where the first isomorphism follows from Lemma \ref{inclusion}. On the other hand Lemma \ref{inclusion} 
gives $\mm^{n+1}P\subseteq P^{n+1}$ and since $\mm^n P=P^n$ we have a surjection
\begin{equation}\label{2surj}
 \mm^n P/\mm^{n+1}P\twoheadrightarrow P^n/P^{n+1}\cong Q^{\oplus d}
\end{equation}
where the last isomorphism is \eqref{pn1}. Since $Q$ is of finite length the composition of \eqref{1surj}
and \eqref{2surj} is an isomorphism. Thus $\mm^n P/\mm^{n+1}P\cong P^n/P^{n+1}$ and since 
$\mm^nP= P^n$ we get $\mm^{n+1}P=P^{n+1}$.

It remains to show that the map $\mm^{n+1}\rightarrow \Hom_{\dualcat}(P, \mm^{n+1} P)$ is an isomorphism. We 
apply $\Hom_{\dualcat}(P,\ast)$ to the surjection $P^{\oplus d}\twoheadrightarrow \mm^n P$ to obtain a surjection:
\begin{equation}\label{mnfingen}
E^{\oplus d}\twoheadrightarrow \Hom_{\dualcat}(P, \mm^n P)\cong \mm^n,
\end{equation}
where $(\phi_1,\ldots, \phi_d)\mapsto \sum_{i=1}^d X_i\circ \phi_i$. So every $\psi\in \mm^n$ may be written as
$\psi=\sum_{i=1}^d X_i\circ \phi_i$, with $\phi_i\in E$. Let $\lambda_i$ be the image of $\phi_i$ in $E/\mm\cong k$, 
then $\phi_i-\lambda_i\in \mm$ and so $\psi\in \sum_{i=1}^d \lambda_i X_i +\mm^{n+1}$. Hence, $\dim \mm^n/\mm^{n+1}\le d$.
We apply  $\Hom_{\dualcat}(P,\ast)$ to the surjection $\mm^n P\twoheadrightarrow \mm^n P/\mm^{n+1} P$ to obtain a surjection:
\begin{equation}\label{3surj}
\mm^n\twoheadrightarrow \Hom_{\dualcat}(P,\mm^nP)\twoheadrightarrow \Hom_{\dualcat}(P, \mm^n P/\mm^{n+1}P).
\end{equation}
Now $\mm^nP /\mm^{n+1}P\cong Q^{\oplus d}$ and so $\dim  \Hom_{\dualcat}(P, \mm^n P/\mm^{n+1}P)=d$. The composition 
in \eqref{3surj} factors through $\mm^n/\mm^{n+1}\twoheadrightarrow \Hom_{\dualcat}(P, \mm^n P/\mm^{n+1}P)$. So 
$\dim \mm^n/\mm^{n+1}\ge d$. Hence, $\dim  \mm^n/\mm^{n+1}= d$ and the surjection is an isomorphism. The commutative diagram 
with exact rows:
 \begin{displaymath}
\xymatrix@1{ 0\ar[r] & \mm^{n+1}\ar[d]\ar[r] & \mm^n \ar[r]\ar[d]^{\cong} & \mm^n/\mm^{n+1}\ar[r]\ar[d]^{\cong} & 0\\
             0 \ar[r]& \Hom(P, \mm^{n+1} P)\ar[r]&\Hom(P, \mm^n P)\ar[r] &  \Hom(P,\mm^nP/\mm^{n+1}P)\ar[r] & 0.}
\end{displaymath}
implies that $\mm^{n+1}\rightarrow \Hom_{\dualcat}(P, \mm^{n+1} P)$ is an isomorphism. We have shown that the image 
of $\{X_1, \ldots, X_d\}$ in $\mm^n/\mm^{n+1}$ is a basis of $\mm^n/\mm^{n+1}$ as a $k$-vector space. Thus \eqref{1surj}
may be interpreted as an isomorphism $\mm^n/\mm^{n+1}\wtimes_E P\overset{\cong}{\rightarrow} \mm^n P/\mm^{n+1} P$, which 
proves (iv).
\end{proof}

\begin{cor}\label{mnfingen1} The ideals $\mm^n$ are finitely generated  right $E$-modules.
\end{cor}
\begin{proof} This follows from \eqref{mnfingen}.
\end{proof}

\begin{cor}\label{Emnetc} We have an isomorphism of $\OO$-modules:
\begin{equation}\label{isomo1}
\Hom_{\dualcat}(P/\mm^n P, P/\mm^n P)\cong \Hom_{\dualcat}(P, P/\mm^n P)
\end{equation}
and an isomorphism of rings:
\begin{equation}\label{isomo2}
E/\mm^n\cong \End_{\dualcat}(P/\mm^n P).
\end{equation}
\end{cor}

\begin{proof} Application of $\Hom_{\dualcat}(P, \ast)$ and $\Hom_{\dualcat}(\ast, P/\mm^n P)$ to 
$0\rightarrow \mm^n P\rightarrow P\rightarrow P/\mm^nP\rightarrow 0$ gives exact sequences
\begin{equation}\label{stub1}  
0\rightarrow \Hom_{\dualcat}(P, \mm^n P)\rightarrow \Hom_{\dualcat}(P, P)\rightarrow \Hom_{\dualcat}(P, P/\mm^n P)\rightarrow 0
\end{equation}
\begin{equation}\label{stub2}
0\rightarrow \Hom_{\dualcat}(P/\mm^n P, P/\mm^n P)\rightarrow \Hom_{\dualcat}(P, P/\mm^n P)\rightarrow 
\Hom_{\dualcat}(\mm^n P,  P/\mm^n P)
\end{equation} 
Let $\phi\in \Hom_{\dualcat}(P, P/\mm^n P)$. We may lift it to $\tilde{\phi}\in E$ using \eqref{stub1}. Since $\mm^n$ is a 
two-sided ideal of $E$, we have $\tilde{\phi}(\mm^n P) \subseteq \mm^n P$. Hence,  $\phi$ maps to zero in \eqref{stub2}, 
which implies \eqref{isomo1}. The last assertion follows from Proposition \ref{filtdone} (ii). 
\end{proof}

\begin{cor}\label{Einvlim} We have $E\cong \underset{\longleftarrow}{\lim}\, E/\mm^n$. The $\mm$-adic topology on $E$ 
coincides with the natural one, defined in \S\ref{zerosec}.
\end{cor}
\begin{proof} Lemma \ref{Plimit} and Proposition \ref{filtdone} (i) imply that $P\cong \underset{\longleftarrow} \lim\, P/\mm^n P$.
Thus
\begin{equation}
 E\cong \Hom_{\dualcat}(P, \underset{\longleftarrow}{\lim}\, P/\mm^n P)\cong  \underset{\longleftarrow}{\lim}\, 
\Hom_{\dualcat}(P,P/\mm^n P)\cong \underset{\longleftarrow}{\lim}\, E/\mm^n
\end{equation}
where the last isomorphism follows from Corollary \ref{Emnetc}. It follows from Proposition \ref{filtdone} (i) and Lemma \ref{filtrationP}
that $P/\mm^n P$ is of finite length, hence the ideal $\mathfrak r(P/\mm^n P)$, defined in \eqref{defrM}, is an open ideal of $E$. 
It follows from Proposition \ref{filtdone} (iii), that $\mathfrak r(P/\mm^n P)=\mm^n$. Conversely, if $\mathfrak r$ is an open ideal 
of $E$ then, $E/\mathfrak r$ is an $E$-module of finite length, and so will be annihilated by some power of $\mm$, which implies that 
$\rr$ is open in the $\mm$-adic topology. Hence the two topologies coincide. 
\end{proof}

\begin{cor}\label{PisEflat} The functor $\wtimes_E P$ is exact. 
\end{cor}
\begin{proof} We will show that if  $0\rightarrow \md_1\rightarrow \md_2\rightarrow \md_3\rightarrow 0$ is an exact sequence
of right pseudo-compact $E$-modules then $0\rightarrow \md_1\wtimes_E P\rightarrow \md_2\wtimes_E P\rightarrow \md_3\wtimes_E P\rightarrow 0$
is an exact sequence in $\dualcat$. Since projective limits commute with the completed tensor product and are exact in $\dualcat$, 
we may assume that $\md_1$, $\md_2$ and $\md_3$ are of finite length. The functor $\wtimes_E P$ is right exact, let $\wTor^i_E(\ast, P)$ 
be the $i$-th left derived functor of $\wtimes_E P$. It is enough to show that $\wTor^1_E(k, P)=0$, since by devissage this implies that 
$\wTor^1_E(\md, P)=0$ for all pseudo-compact $E$-modules $\md$, which are of finite length. We apply $\wtimes_E P$ to 
the exact sequence $0\rightarrow \mm \rightarrow E \rightarrow k\rightarrow 0$ to  obtain an exact sequence:
\begin{equation}\label{Tor}
0\rightarrow \wTor^1_E(k, P)\rightarrow \mm\wtimes_E P\rightarrow P\rightarrow P/\mm P\rightarrow 0.
\end{equation}
It is enough to show that the natural map $\mm\wtimes_E P\rightarrow P$ is injective. 
Proposition \ref{filtdone} (iv) says that the natural map $\mm^n/\mm^{n+1}\wtimes_E P\rightarrow \mm^n P/\mm^{n+1} P$ is an isomorphism 
for all $n\ge 0$. By devissage, we obtain that the natural map $\mm/\mm^n \wtimes_E P\rightarrow \mm P/\mm^n P$ is an isomorphism.
Passing to the limit we obtain that the natural map $\mm\wtimes_E P\rightarrow \mm P$ is an isomorphism. 
\end{proof}

We will refer to the result of the Corollary \ref{PisEflat} as ``$P$ is $E$-flat''.  

\begin{cor}\label{flatA} Let $\varphi: E\rightarrow A$ be a map of pseudo-compact rings,
which makes $A$ into a pseudo-compact $E$-module then $A\wtimes_{E, \varphi} P$ is 
$A$-flat.
\end{cor}
\begin{proof} Since $A$ is a pseudo-compact $E$-module, every pseudo-compact $A$-module 
$\md$ is also a pseudo-compact $E$-module via $\varphi$. The assertion follows from the 
isomorphism $\md \wtimes_A (A \wtimes_E P)\cong \md\wtimes_E P$ and Corollary \ref{PisEflat}.
\end{proof}

\begin{cor}\label{flatmodule} Let $\md$ be an $\OO$-torsion free, pseudo-compact $E$-module. Then $\md\wtimes_{E} P$ is $\OO$-torsion 
free.
\end{cor} 
\begin{proof} Since $\md$ is $\OO$-torsion free multiplication by $\varpi$ is injective. Since $\wtimes_E P$ is exact 
it remains injective.
\end{proof}

\begin{remar}
Let us point out a special case, where our results are particularly easy to prove, and which was the motivation for the 
formalism.  If $\Ext^1_{\dualcat}(S, S')=0$ for all irreducible $S'$ non-isomorphic to $S$, and 
$\Ext^1_{\dualcat}(S, S)$ is finite dimensional, then the hypotheses (H1)-(H5) are satsified with $Q=S$. 
The filtration in Lemma \ref{filtrationP} is simply the radical filtration, which 
is exhaustive, as by assumption $P$ can be written as projective limit taken over all the quotients of $P$ of finite length.
Hence, all the irreducible subquotients of $P$ are isomorphic to $S$. If $\md$ is a pseudo-compact $E$-module, then 
$\md\wtimes_E P$ is a quotient of $\prod_I P$ for some set $I$, thus all the irreducible subquotients 
of $\md\wtimes_E P$ are isomorphic to $S$. Let $\md_1\hookrightarrow \md_2$ be an injection of pseudo-compact $E$-modules, and 
let $K$ be the kernel of the induced map $\md_1\wtimes_E P\rightarrow \md_2\wtimes_E P$. All the irreducible subquotients of 
$K$ are isomorphic to $S$, but Lemma \ref{headS0} implies that $\Hom_{\dualcat}(P, K)=0$. Hence, $K$ is zero and $P$ is $E$-flat. 
\end{remar}

\subsection{Deformations}\label{def}

Let $\dualcat(\OO)$ be a full abelian subcategory of $\Mod^{\mathrm{pro\, aug}}_G(\OO)$ closed under 
direct products and subquotients in $\Mod^{\mathrm{pro\, aug}}_G(\OO)$. We further assume that 
every object $M$ of $\dualcat(\OO)$ can be written as $M\cong \underset{\longleftarrow}{\lim} \, M_i$, where the limit is taken over all the quotients of finite length. Let $\dualcat(k)$ be  a full subcategory 
of $\dualcat(\OO)$ consisting of the objects which are killed by $\varpi$. 

Let $S$ and $Q$ be as in the previous section with $\dualcat=\dualcat(k)$. We assume that hypotheses (H1)-(H5) are satisfied in 
$\dualcat=\dualcat(k)$. Let $P\twoheadrightarrow S$ be a projective envelope of $S$ in $\dualcat(k)$, $E=\End_{\dualcat(k)}(P)$ 
and $\mm$ the maximal ideal of $E$ defined by  \ref{Em}.

Let $\wP\twoheadrightarrow S$ be a projective envelope of $S$ in $\dualcat(\OO)$, 
$\wE:=\End_{\dualcat(\OO)}(\wP)$ and $\wm$ two sided ideal of $\wE$ defined by \ref{Em}.  
For every $M$ in $\dualcat(k)$ we have 
$\Hom_{\dualcat(\OO)}(\wP, M)\cong \Hom_{\dualcat(k)}(\wP/\varpi \wP, M)$ thus 
$\wP/\varpi \wP$ is projective in $\dualcat(k)$, and the map $\wP/\varpi \wP\rightarrow S$ 
is essential. Since projective envelopes are unique up to isomorphism, we obtain $P\cong \wP/\varpi \wP$. Thus we have an
exact sequence: 
\begin{equation}\label{tildeseq}
0\rightarrow \wP[\varpi]\rightarrow \wP\overset{\varpi}{\rightarrow}\wP\rightarrow P\rightarrow 0.
\end{equation}
Since $\wP$ is projective applying $\Hom_{\dualcat(\OO)}(\wP,\ast)$ to \eqref{tildeseq} we obtain an exact sequence 
\begin{equation}\label{tildeseq1}
0\rightarrow \Hom_{\dualcat(\OO)}(\wP, \wP[\varpi])\rightarrow \wE\overset{\varpi}{\rightarrow}\wE\rightarrow E\rightarrow 0.
\end{equation}
\begin{lem}\label{lemok0} Let $A$ and $B$ be objects of $\dualcat(k)$  then there exists an exact sequence
\begin{equation}
0\rightarrow \Ext^1_{\dualcat(k)}(A, B) \rightarrow \Ext^1_{\dualcat(\OO)}(A,B)\rightarrow \Hom_{\dualcat(k)}(A,B)
\end{equation}
\end{lem}
\begin{proof} Let   
 $0\rightarrow B \rightarrow C \rightarrow A\rightarrow 0$ be an extension in $\dualcat(\OO)$. Multiplication 
by $\varpi$ induces an exact sequence 
\begin{equation}
0\rightarrow B[\varpi]\rightarrow C[\varpi] \rightarrow A[\varpi]\overset{\partial}{\rightarrow} B/\varpi B.
\end{equation} 
Since $A$ and $B$ are in $\dualcat(k)$ we have $B=B[\varpi]=B/\varpi B$ and $A=A[\varpi]$ 
 so  $\partial$ defines an element of $\Hom_{\dualcat(k)}(A, B)$, which  
depends only on the  class of the extension in $\Ext^1_{\dualcat(\OO)}(A, B)$. Now $\partial=0$ if and only if $C=C[\varpi]$, which 
means if and only if the extension lies in $\dualcat(k)$. 
\end{proof}

We note that since $\dualcat(k)$ is a full subcategory of $\dualcat(\OO)$, (H1) and (H2) for $\dualcat(k)$ trivially 
imply  (H1) and (H2) for $\dualcat(\OO)$. It follows from the  Nakayama's lemma that $\dualcat(k)$ and $\dualcat(\OO)$ have the same irreducible objects. Further, it follows  from Lemma \ref{lemok0} and (H1) for $\dualcat(k)$ that 
(H3) and (H4) for $\dualcat(k)$ imply (H3) and (H4) for $\dualcat(\OO)$.

\begin{prop}\label{H00} Suppose that the following hypothesis holds: 
\begin{itemize} 
\item[(H0)] $\Hom_{\dualcat(\OO)}(\wP[\varpi], \rad Q)=0$.
\end{itemize}  
then (H5) for $\dualcat(k)$ implies (H5) for $\dualcat(\OO)$. 
\end{prop} 
\begin{proof} It follows from (H1) and (H2) that $S$ is not a subquotient of $R=\rad Q$. Thus  
$\Hom_{\dualcat(\OO)}(\wP, R)=\Hom_{\dualcat(k)}(P, R)=0$, by Lemma \ref{mult=dim}.
Since $\wP$ is projective  using \eqref{tildeseq} we get $\Ext^1_{\dualcat(\OO)}(P, R)=0$ and 
\begin{equation} 
\Hom_{\dualcat(\OO)}(\wP[\varpi], R)\cong \Ext^1_{\dualcat(\OO)}(\varpi \wP, R)\cong \Ext^2_{\dualcat(\OO)}(P, R).
\end{equation} 
Thus (H0) is equivalent to $\Ext^2_{\dualcat(\OO)}(P, R)=0$. 
We apply $\Hom_{\dualcat(\OO)}(\ast, R)$ to 
$0\rightarrow \mm P \rightarrow P\rightarrow Q\rightarrow 0$ to get an isomorphism 
\begin{equation}\label{reducetoext1} 
\Ext^1_{\dualcat(\OO)}(\mm P, R)\cong \Ext^2_{\dualcat(\OO)}(Q, R).
\end{equation}   
Since $P$ is projective in $\dualcat(k)$ we also have $\Ext^1_{\dualcat(k)}(\mm P, R)\cong \Ext^2_{\dualcat(k)}(Q, R)=0$ by 
(H5) for $\dualcat(k)$. Since $\mm P$ is a quotient of $P^{\oplus d}$ and $\Hom_{\dualcat(k)}(P, R)=0$, we 
get $\Hom_{\dualcat(k)}(\mm P, R)=0$. Lemma \ref{lemok0} implies  $\Ext^1_{\dualcat(\OO)}(\mm P, R)=0$ and the assertion follows from
\eqref{reducetoext1}.
\end{proof}
For the rest of the section we assume (H1)-(H5) for $\dualcat(k)$ and (H0). It follows from the Proposition 
that (H1)-(H5) also hold for $\dualcat(\OO)$. Hence, the results of \S\ref{first} apply to $\wP$, $\wE$ and $\wm$. 

\begin{remar} In the application to $G=\GL_2(\Qp)$ we will show that $\wP$ is in fact $\OO$-torsion free,
so (H0) will be satisfied. 
\end{remar}

\begin{defi}\label{deficatA} Let $\mathfrak A$ be the category of finite local (possibly non-com\-mu\-ta\-tive) artinian 
$\OO$-algebras $(A, \mm_A)$ such that 
the image of $\OO$ under the structure morphism $\sigma: \OO\rightarrow A$ lies in the centre of $A$, and 
$\sigma$ induces an isomorphism $\OO/\varpi \OO\cong A/\mm_A$.  We denote by $\mathfrak A^{ab}$ the full subcategory of 
$\mathfrak A$ consisting of commutative algebras.
\end{defi}

\begin{remar} The category $\mathfrak A$ contains genuinely non-commutative rings: for example, every group algebra of a  finite $p$-group 
over $\OO/(\varpi^n)$ is in $\mathfrak A$.
\end{remar}

We refer the reader to \cite[\S19]{lam} for basic facts about non-commutative local rings.
Let $\hA$ denote the category of local $\OO$-algebras $(R, \mm_R)$ such that for every $n\ge 1$, 
$R/\mm_R^n$ is an object of $\Aa$ and $R\cong \underset{\longleftarrow}{\lim}\, R/\mm^n_R$ and 
morphisms are given by $\Hom_{\hA}(R, S)=\underset{\longleftarrow}{\lim} \, \Hom_{\hA}(R, S/\mm^m_S)= \underset{\longleftarrow}{\lim} \, \Hom_{\Aa}(R/\mm_R^m, S/\mm^m_S)$, 
where the limit is taken over all $m\ge 1$.

\begin{defi}\label{defiDef} Let $(A, \mm_A)$ be an object of $\mathfrak A$. A deformation of $Q$ to $A$ is a 
pair $(M, \alpha)$, where  $M$   is an object of $\dualcat(\OO)$ together with the map of $\OO$-algebras $A\rightarrow \End_{\dualcat(\OO)}(M)$,
which makes $M$ into a flat $A$-module and $\alpha: k\wtimes_A M \cong Q$ is an isomorphism in $\dualcat(k)$.
\end{defi}
  
Let $(A, \mm_A)\in \mathfrak A$, let $n$ be the largest integer such that $\mm^n_A\neq 0$  and $(M, \alpha)$ a 
deformation of $Q$ to  $A$. We note that $A$ is finite (as a set). In particular, every finitely generated $A$-module 
$N$ is also finitely presented, and for such $N$  we have 
\begin{equation}\label{wtimeso}
N\wtimes_A M\cong N\otimes_A M.
\end{equation}

\begin{lem}\label{grading} For $0\le i\le n$ we have  
\begin{equation}\label{artin}
\mm^i_A/\mm^{i+1}_A \wtimes_A M\cong \mm^i_A M/\mm^{i+1}_A M \cong Q^{\oplus d_i},
\end{equation}
where $d_i= \dim  \mm^i_A/\mm^{i+1}_A$.
\end{lem}
\begin{proof} We argue by induction on $i$. The statement is true if $i=0$. In general, by 
applying $\wtimes_A M$ to  $0\rightarrow \mm^i_A/\mm^{i+1}_A\rightarrow A/\mm^{i+1}_A\rightarrow A/\mm^i_A\rightarrow 0$,
and using flatness of $M$ and \eqref{wtimeso} we get an isomorphism $\mm^i_A/\mm^{i+1}_A \wtimes_A M\cong \mm^i_A M/\mm^{i+1}_A M$. Now, 
$\mm^i_A/\mm^{i+1}_A\cong k^{\oplus d_i}$ as an $A$-module, since $k\wtimes_A M\cong Q$ we obtain the last assertion.
\end{proof}

Given an $\OO$-module of finite length, we denote by $\ell_{\OO}$ its length.

\begin{lem}\label{length} We have $\ell_{\OO} ( \Hom_{\dualcat(\OO)}(\wP, M))=\ell_{\OO}(A)$.
\end{lem}
\begin{proof} Since $\wP$ is projective, $\Hom_{\dualcat(\OO)}(\wP, \ast)$ is exact and 
$\dim \Hom_{\dualcat(\OO)}(\wP, Q)= \dim \Hom_{\dualcat(k)}(P, Q)=1$. Hence, 
$\ell_{\OO} ( \Hom_{\dualcat(\OO)}(\wP, M))=\sum_{i=0}^n d_i= \ell_{\OO}(A).$ 
\end{proof}

\begin{lem}\label{ohyeah} The natural map 
\begin{equation}\label{naturalmap} 
\Hom_{\dualcat(\OO)}(\wP, M)\wtimes_{\wE} \wP\rightarrow M
\end{equation}
is an isomorphism of (left) $A$-modules. 
\end{lem}
\begin{proof} Since $\wP$ is projective and $\wE$-flat by Corollary \ref{PisEflat}, the functor $F: \dualcat(\OO)\rightarrow \dualcat(\OO)$, 
$F(N):=\Hom_{\dualcat(\OO)}(\wP, N)\wtimes_{\wE}{ \wP}$ is exact. Moreover, if $N$ is of finite length 
in $\dualcat(\OO)$ then  $\Hom_{\dualcat(\OO)}(\wP, N)$ is an $\OO$-module of finite length, and 
so the completed and the usual tensor products coincide. Further, we have $F(Q)\cong k \otimes_{\wE} \wP\cong Q$
and \eqref{artin} gives $F(\mm^i M/\mm^{i+1}M)\cong \mm^i M/\mm^{i+1}M$, $0\le i\le n$. The exactness of $F$ implies  
$F(M)\cong M$. Since the map $F(N)\rightarrow N$ is functorial, we obtain that the isomorphism 
in $\dualcat(\OO)$ is also  an isomorphism of $A$-modules.
\end{proof} 

Recall that the map $\wP\twoheadrightarrow S$ factors through $\theta: \wP\rightarrow Q$, 
which induces an isomorphism $\alpha^{univ}: k\wtimes_{\wE} \wP \cong Q$, $\lambda\wtimes v \mapsto \lambda \theta(v)$.
We will think of $(\wP, \alpha^{univ})$ as the universal deformation of $Q$.

\begin{lem}\label{surjectDef} Let $(M, \alpha)$ be a deformation of $Q$ to $A$. There exists $\varphi \in \Hom_{\hA}(\wE, A)$ and  
 $\iota: M\cong A\wtimes_{\varphi, \wE} \wP$
such that $\alpha= \alpha^{univ}\circ (k\wtimes_A\iota)$.  
\end{lem}
\begin{proof} Since $\wP$ is projective, there exists $\psi: \wP\rightarrow M$ making the diagram:
\begin{displaymath}
\xymatrix@1{  M\ar@{->>}[r] & k \wtimes_A  M \ar[d]^-{\cong}_{\alpha}\\
              \wP  \ar[u]^-{\psi}\ar@{->>}[r] &  Q}
\end{displaymath}
commute. We claim that  the map $A\rightarrow \Hom_{\dualcat(\OO)}(\wP, M)$, $a\mapsto a \circ \psi$ induces 
an isomorphism of $A$-modules. Lemma \ref{length} says that it is enough to prove that the map is injective. 
Choose $v\in \wP$, such that the image of $v$ in $Q$ is non-zero. 
Suppose $a\in \mm^{i}_A$ and $a\not\in \mm^{i+1}_A$ then \eqref{artin} gives an isomorphism:
$$ \mm^i_A/\mm^{i+1}_A \wtimes_k M/\mm_A M \cong \mm^{i}_{A} M/\mm^{i+1}_A M.$$
Since $(a +\mm^{i+1})\wtimes (\psi(v)+ \mm_A M)$ is non-zero, we also obtain $a(\psi(v))$ is non-zero. Hence $a\circ \psi=0$ if and 
only if $a=0$ and so the map is injective. This means that for every $b\in \wE$ there exists a unique $\varphi(b)\in A$ 
such that $\varphi(b) \circ \psi = \psi \circ b$. Uniqueness implies that $\varphi$ is a homomorphism 
of algebras. The assertion follows from Lemma \ref{ohyeah}.
\end{proof}

Let $\Def_Q:\mathfrak A \rightarrow \mathrm{Sets}$ be the functor associating to $A$ the set of isomorphism 
classes of deformations of $Q$ to $A$. We denote by $\Def_Q^{ab}$ the restriction of $\Def_Q$ to $\mathfrak A^{ab}$.
Let $(A, \mm_A)$ be in $\mathfrak A$, then  to $\varphi\in \Hom_{\hA}(\wE, A)$ we may associate an isomorphism 
class of $(A\wtimes_{\wE, \varphi} \wP, \alpha_{\varphi})$, where $\alpha_{\varphi}$ is the composition of 
$A\wtimes_{\wE, \varphi} \wP\rightarrow k\wtimes_{\wE, \varphi} \wP$ with $\alpha^{univ}$. By Corollary \ref{flatA} this 
gives us a point in $\Def_Q(A)$. 

\begin{thm}\label{repnonC} The above map induces a bijection between $\Def_Q(A)$ and $A^{\times}$-conjugacy classes of $\Hom_{\hA}(\wE, A)$.
\end{thm}
\begin{proof} Lemma \ref{surjectDef} says that the map $\Hom_{\hA}(\wE, A)\rightarrow \Def_Q(A)$ is surjective. Suppose 
we have $\varphi_1, \varphi_2\in \Hom_{\hA}(\wE, A)$ and an isomorphism $\beta: A\wtimes_{\wE, \varphi_1} \wP\cong
 A\wtimes_{\wE, \varphi_2} \wP$ in $\dualcat(A)$ such that the diagram 
\begin{displaymath}
\xymatrix@1{  A\wtimes_{\wE, \varphi_1} \wP \ar@{->>}[r] & Q \\
             A\wtimes_{\wE, \varphi_2} \wP    \ar[u]^-{\beta}_{\cong} \ar@{->>}[ur] &}
\end{displaymath}
commutes. For $i\in\{1,2\}$ define  $\psi_i:\wP\rightarrow   A\wtimes_{\wE, \varphi_i} \wP$ by  $\psi_i(v):=1\wtimes v$. 
It follows from the proof of Lemma \ref{surjectDef} that $\Hom_{\dualcat(\OO)}(\wP, A\wtimes_{\wE, \varphi_i} \wP)$ is a free $A$-module 
of rank $1$, and $\psi_i$ is a generator. Since $\beta$ is an isomorphism, $\beta_{*}:= \Hom_{\dualcat(\OO)}(\wP, \beta)$  is also an isomorphism. Hence 
there exists $u\in A^{\times}$ such that $u \psi_1=\beta_{*}(\psi_2)$. Since $\beta_{*}$ is $A$-linear, we obtain
\begin{equation}\label{wishover} 
\beta(a\wtimes v)= \beta( a (1\wtimes v))= [a \beta_{*}(\psi_2)](v)= a u \psi_1(v)= au\wtimes v
\end{equation}
So for all $b\in \wE$, \eqref{wishover} gives 
\begin{equation}\label{wishover1}
\beta(1 \wtimes b v)= \beta( \varphi_2(b) \wtimes v) = \varphi_2(b) u \wtimes v
\end{equation}  
\begin{equation}\label{wishover2}
\beta(1 \wtimes b v)= u \wtimes b v = u \varphi _1(b) \wtimes v
\end{equation} 
It follows from \eqref{wishover1} and \eqref{wishover2} that $(u \varphi_1(b) - \varphi_2(b) u) \psi_1= 0$. Hence, 
$\varphi_2(b)= u \varphi_1(b) u^{-1}$ for all $b\in \wE$. 

Conversely, suppose that $\varphi_1$ and $\varphi_2$ lie in the same $A^{\times}$-conjugacy class.
Since $\OO\rightarrow A/\mm_A$ is  surjective and the image of $\OO$ in $A$ is contained in the centre, there 
exists $u\in 1+\mm_A$ such that $\varphi_2 = u \varphi_1 u^{-1}$. An easy check shows that 
$\beta: A\wtimes_{\wE, \varphi_2} \wP \rightarrow A \wtimes_{\wE, \varphi_1} \wP$, $a\wtimes v\mapsto au \wtimes v$ is the 
required isomorphism of deformations.    
\end{proof}

\begin{cor}\label{defab} $\Def^{ab}_Q(A)=\Hom_{\hA}(\wE^{ab}, A)$, where $\wE^{ab}$ is the maximal commutative quotient of $\wE$. 
\end{cor}
\begin{proof} Since $A$ is commutative, every $A^{\times}$-conjugacy class consists of one element. Thus 
$\Def^{ab}_Q(A)= \Def_Q(A)= \Hom_{\hA}(\wE, A)= \Hom_{\hA}(\wE^{ab}, A)$. The last equality follows from the universal property 
of $\wE^{ab}$.
\end{proof} 

\begin{remar} If $R$ is  an arbitrary  non-commutative topological ring  then $R^{ab}$ might be the zero ring. This is 
not the case here, since $\wE/\wm\cong k$ is commutative.
\end{remar}

\begin{lem}\label{tangentspace} Let $k[\varepsilon]$ be the ring of dual numbers so that $\varepsilon^2=0$. Then we have natural isomorphisms
\begin{equation} 
\Ext^1_{\dualcat(k)}(Q, Q)\cong \Hom_{\hA}(\wE^{ab}, k[\varepsilon])\cong \Hom_{\hA}(\wE, k[\varepsilon])\cong (\mm /\mm^2)^*,
\end{equation}
where $*$ denotes $k$-linear dual. 
\end{lem}
\begin{proof} The first isomorphism is classical. The second follows from the fact that $k[\varepsilon]$ is commutative. The 
third is again classical.
\end{proof}

Let $(A, \mm_A)$ be in $\Aa$ and let $F:\Aa\rightarrow \Sets$ be a covariant functor.  For each $u\in A^{\times}$, 
$\ad{u}: A\rightarrow A$, $a \mapsto u a u^{-1}$ is a morphism in $\Aa$, and hence induces a morphism of sets 
$F(\ad{u}): F(A)\rightarrow F(A)$. We say that the functor $F$ is \textit{stable under conjugation} if 
$F(\ad{u})= \id_{F(A)}$ for all objects $A$ of $\Aa$ and all $u\in A^{\times}$. For $R$ in $\hA$ we denote 
$h_R:\Aa\rightarrow \Sets$ and $F_R: \Aa\rightarrow \Sets$ the  functors   $h_R(A):= \Hom_{\hA}(R, A)$ and 
$F_R(A)$ the set of $A^{\times}$-conjugacy classes in $h_R(A)$. We have a variant of Yoneda's lemma.

\begin{lem}\label{yoneda} Let $F:\Aa\rightarrow \Sets$ be a covariant functor stable under conjugation then 
the map $\eta \mapsto  \eta_R(\{\id_R\})$ induces a bijection between the set of natural transformations 
$\Mor(F_R, F)$ and $F(R):=\underset{\longleftarrow}{\lim}\,  F(R/\mm_R^n)$.
\end{lem}
\begin{proof} Mapping a homomorphism to its conjugacy class gives rise to a natural transformation 
of functors $\kappa: h_R\rightarrow F_R$ and hence a map  $\Mor(F_R, F)\rightarrow \Mor(h_R, F)$, 
$\eta\mapsto \eta\circ \kappa$, which is clearly injective. We claim that it is also surjective. 
Let $\xi: h_R\rightarrow F$ be a natural transformation, $A$ an object of $\Aa$ and $u\in A^{\times}$. Then we have 
$$ \xi_A \circ h_R(\ad{u})= F(\ad{u}) \circ \xi_A= \id_{F(A)} \circ \xi_A =\xi_A.$$
Thus $\xi$ factors through $\kappa$ and hence the map is surjective. The assertion follows from the usual Yoneda's lemma.
\end{proof}

\begin{lem}\label{isoFisoR} Let $R$ and $S$ be in $\hA$ and suppose that $\eta: F_R\rightarrow F_S$ is an isomorphism of functors then 
the rings $R$ and $S$ are isomorphic. Moreover, $\eta$ determines the isomorphism up to  conjugation.
\end{lem}
\begin{proof} It follows from Lemma \ref{yoneda} that $\Mor(F_R, F_S)\cong \Hom_{\hA}(S, R)/R^{\times}$. Thus we may find 
$\varphi: S\rightarrow R$ such that for each $A$ in $\Aa$, $\eta_A: F_R(A)\rightarrow F_S(A)$ sends the conjugacy class
of $\psi$ to the conjugacy class of $ \psi\circ \varphi$. Since, $\eta$ is  a bijection for all $A$, we may find 
$c\in F_R(S)$ such that $\eta_R(c)=\{\id_S\}$. Choose any $\psi\in c$ then the last equality reads $\psi\circ\varphi=\id_S$,
which implies that $\varphi$ is an isomorphism. The last assertion follows from Lemma \ref{yoneda}.
\end{proof}

\subsection{Examples}
We give some examples of deformations with possibly non-com\-mu\-ta\-ti\-ve coefficients. Our coefficients are objects of the 
category $\Aa$ defined in \ref{deficatA}.

\begin{lem}\label{1ex} Let $\GG$ be a finitely generated pro-finite group and $Q=\Eins$ the trivial representation. 
Then $\Def_{\Eins}(A)= \Hom_{\hA}(\OO[[\GG(p)]]^{op}, A)/\sim$, where $\GG(p)$ is the maximal pro-$p$ quotient 
of $\GG$, and $\sim$ denotes the equivalence under conjugation by $A^{\times}$. Moreover, 
$\Def^{ab}_{\Eins}(A)= \Hom_{\hA}(\OO[[\GG(p)^{ab}]], A)$. 
\end{lem}
\begin{proof} Let $(M, \alpha)$ be an $A$-deformation. Since $M$ is $A$-flat and $k \wtimes_A M \cong k$ we get 
that $M$ is a free $A$-module of rank $1$. Choose $v\in M$, such that $\alpha( 1\otimes v)=1$. Then $v$ is a basis vector 
of $M$ and for every $g\in \GG$ we obtain a unique $a_g\in A$ such that $g v= a_g v$. Now 
$$ a_{gh} v= (gh) v= g (h v) = g a_h v= a_h g v= a_h a_g v.$$
Hence, we get a group homomorphism $\GG^{op}\rightarrow 1+\mm_A$, $g \mapsto a_g$. Since $1+\mm_A$ is a finite group 
of $p$-power order, the map factors through $\GG(p)^{op}$. By extending $\OO$-linearly we obtain a homomorphism 
$\OO[[\GG(p)]]^{op}\rightarrow A$. A different choice of $v$ would conjugate the homomorphism by $u\in 1+\mm_A$. 

Conversely, $\OO[[\GG(p)]]$ is a free right $\OO[[\GG(p)]]^{op}=\End_{\dualcat(\OO)}(\OO[[\GG(p)]])$ module, with the action $b\centerdot a:= ab$. Thus 
every $\varphi\in \Hom_{\hA}(\OO[[\GG(p)]]^{op}, A)$ defines a deformation $A\wtimes_{\OO[[\GG(p)]]^{op}, \varphi} \OO[[\GG(p)]]$.

If $A$ is commutative then the map $\GG\rightarrow \GG(p) \rightarrow 1+\mm_A$ must further factor through $\GG(p)^{ab}$, 
and the same argument gives the claim.
\end{proof}  

\begin{lem}\label{2ex} Let $G=\Qp^{\times}$ and $\chi: \Qp^{\times}\rightarrow k^{\times}$ a continuous character. If $p\neq 2$ then 
$\Def_{\chi}(A)= \Hom_{\hA}(\OO[[x,y]], A)/A^{\times}$, where $\OO[[x,y]]$ denotes the ring of formal (commutative) power series.
\end{lem}
\begin{proof} We may choose a character  $\tilde{\chi}: \Qp^{\times}\rightarrow \OO^{\times}$ lifting $\chi$. 
After  twisting with $\tilde{\chi}$ we may assume that $\chi$ is the trivial character.  It follows from the 
proof of Lemma \ref{1ex} that the deformation problem does not change if we replace $G$ with its pro-$p$ completion
$\widehat{G}$. Since $p\neq 2$ we have  $G\cong \ZZ\oplus \ZZ/(p-1) \oplus \Zp$ and hence $\widehat{G}\cong \Zp^2$. Thus
$\OO[[\widehat{G}]]\cong \OO[[x, y]]$ and the assertion follows from the Lemma \ref{1ex}.
\end{proof}

\begin{prop}\label{projTistfree} 
Let $G=\Qp^{\times}$ and $\chi: \Qp^{\times}\rightarrow k^{\times}$ a continuous character and let $S:=\chi^{\vee}$, 
then if $p\neq 2$ then $\wE\cong \OO[[x,y]]$ and $E\cong k[[x,y]]$. Moreover, $\wP$ is a free $\wE$-module of rank $1$ and 
in particular it is $\OO$-torsion free.
\end{prop}
\begin{proof} We claim that the hypotheses (H1)-(H5) are satisfied for $Q=S=\chi^{\vee}$ and note that since in this case $R=0$
the hypothesis (H0) is  satisfied. Since $\Ext^1_{G}(\chi, \chi)\cong \Hom^{cont}(G, k)$ is $2$-di\-men\-sio\-nal,  (H4) holds. 
Consider a non-split extension $0\rightarrow\chi\rightarrow \epsilon \rightarrow \tau\rightarrow 0$ in $\Mod^{\mathrm{sm}}_G(k)$
with $\tau$ irreducible. Since $G$ is commutative for each $g\in G$ the map $\phi_g: \epsilon\rightarrow \epsilon$, 
$v\mapsto gv -\chi(g) v$ is $G$-equivariant. If $\phi_g$ is non-zero for some $g$ then it induces an isomorphism between 
$\tau$ and $\chi$, if $\phi_g$  is zero for all $g$ then any $k$-vector space splitting of the sequence is $G$-equivariant. 
Hence, (H3) is satisfied and all the other 
hypotheses hold trivially, since $R=0$. It follows from Lemma \ref{2ex} and Lemma \ref{isoFisoR} that $\wE\cong \OO[[x,y]]$ and 
hence $E\cong \wE\otimes_{\OO} k\cong k[[x,y]]$. Since $\wP$ is flat over $\wE\cong\OO[[x,y]]$ by Corollary \ref{PisEflat} 
and $k\wtimes_{\wE} \wP\cong \chi^{\vee}$ is one dimensional, $\wP$ is a free $\wE$-module of rank $1$ and in particular 
it is also $\OO$-torsion free.   
\end{proof} 

\begin{cor}\label{extToruschar}$\dim \Ext^1_{\dualcat(k)}(\chi^{\vee}, \chi^{\vee})=2$, $\dim \Ext^2_{\dualcat(k)}(\chi^{\vee}, \chi^{\vee})=1$. Moreover,  
 $\Ext^i_{\dualcat(k)}(\chi^{\vee}, \chi^{\vee})=0$ 
for $i\ge 3$ and $\Ext^i_{\dualcat(k)}(\chi^{\vee}, S')=0$ for all $i\ge 0$ and all ireducible $S'\in \dualcat(k)$ not isomorphic to 
$\chi^{\vee}$.
\end{cor}
\begin{proof} Since $E\cong k[[x,y]]$ we apply  $\wtimes_{E} P$ to the exact sequence
$$0\rightarrow k[[x,y]]\rightarrow k[[x,y]]\oplus k[[x,y]]\rightarrow k[[x,y]]\rightarrow k\rightarrow 0$$
where the first arrow is $f \mapsto (xf, yf)$, the second is $(f,g)\mapsto yf-xg$ to get a projective resolution 
of $\chi^{\vee}\cong P/\mm P$: 
$$ 0\rightarrow P\rightarrow P^{\oplus 2} \rightarrow P\rightarrow \chi^{\vee}\rightarrow 0.$$
The assertions follow from a calculation with this projective resolution. 
\end{proof}

\subsection{Criterion for commutativity}\label{critcomm}

In this section we devise a criterion, see Theorem \ref{crit}, for the ring $\wE$ to be commutative. 
When $G=\GL_2(\Qp)$ we will show that this criterion is satisfied in the generic cases, see \S \ref{strat}, 
and it will enable us to apply Corollary \ref{commutativeOK}. We use the notation of \S \ref{def}, we assume 
the hypotheses (H1)-(H5) for $\dualcat(\OO)$ or equivalently (H0) and (H1)-(H5) for $\dualcat(k)$.

\begin{lem}\label{graded} If there exists a surjection
$\wE\twoheadrightarrow \OO[[x_1,\ldots, x_d]]$, with $d=\dim \mm/\mm^2$, 
and the graded ring $\gr^{\bullet}_{\mm}(E)$ is commutative then 
$\wE\cong \OO[[x_1,\ldots, x_d]]$.
\end{lem}
\begin{proof} Let $R:=k[[x_1, \ldots, x_d]]$ and $\mm_1=(x_1,\ldots x_d)$ be the maximal ideal of $R$.
Applying $\otimes_{\OO} k$ we obtain a surjection $E\twoheadrightarrow R$, thus a surjection of 
graded rings
\begin{equation}\label{graded1}
 \gr^{\bullet}_{\mm}(E)\twoheadrightarrow \gr^{\bullet}_{\mm_1}(R)\cong k[x_1, \ldots x_d].
\end{equation} 
Since $\gr^{\bullet}_{\mm}(E)$ is commutative and $\dim \mm/\mm^2=d$, there exists a surjection
\begin{equation}\label{graded2}
  k[x_1, \ldots x_d]\twoheadrightarrow \gr^{\bullet}_{\mm}(E).
\end{equation}
It follows from \eqref{graded1} and \eqref{graded2} that $\gr^{\bullet}_{\mm}(E)\cong  \gr^{\bullet}_{\mm_1}(R)$.
Hence $\mm^n/\mm^{n+1}\cong \mm^n_1/\mm^{n+1}_1$ for all $n\ge 1$. By induction we get that $E/\mm^n \cong R/\mm_1^n$ 
for all $n\ge 1$. Since both rings are complete we get $E\cong R$.  
Let $K$ be the kernel of $\wE\twoheadrightarrow \OO[[x_1,\ldots, x_d]]$. Since $\OO[[x_1,\ldots, x_d]]$
is $\OO$-flat, we have $K\otimes_{\OO} k=0$ and hence $K=0$, by Nakayama's lemma for compact $\OO$-modules, \cite{SGA3} Exp. $VII_B$ (0.3.3).
\end{proof}

Let $d$ be the dimension of $\mm/\mm^2$  as a $k$-vector space and let $W$ be a $(d-r)$-dimensional $k$-subspace of 
$\mm/\mm^2$ then $W+\mm^2$ is a $2$-sided ideal of $E$ and 
the exact sequence of $E$-modules $0\rightarrow \mm/(W+\mm^2)\rightarrow E/(W+\mm^2)\rightarrow k\rightarrow 0$ leads
by tensoring with $P$  to an exact sequence of $G$-representations
\begin{equation}\label{defAr}
0\rightarrow Q^{\oplus r} \rightarrow T\rightarrow Q\rightarrow 0
\end{equation}
with $T\cong P/(W+\mm^2)P$. Conversely, any $T$ in \eqref{defAr}, such that  $ \Hom_{\dualcat(k)}(T, S)$ is one dimensional, 
is a quotient $\psi: P\twoheadrightarrow T$, as the cosocle of $T$ is isomorphic to $S$,  and defines a $(d-r)$-dimensional subspace 
\begin{equation}\label{defW}
 W:=  \{a\in \mm : \psi\circ a=0\}/\mm^2\subseteq \mm/\mm^2.
\end{equation}

\begin{lem}\label{AW} Let $T$ and $W$ be as above then 
\begin{equation}\label{dimextAW}
\dim \Ext^1_{\dualcat(k)}(T, S)= \dim \frac{W+\mm^2}{W\mm +\mm^3}=\dim W+\dim \frac{\mm^2}{W \mm +\mm^3}.
\end{equation}
\end{lem}
\begin{proof}We have an exact sequence:
\begin{equation}\label{dimextAW1}
0\rightarrow (W+\mm^2)\wtimes_E P\rightarrow P \rightarrow T\rightarrow 0.
\end{equation}
Since $\dim \Hom_{\dualcat(k)}(T, S)=\dim \Hom_{\dualcat(k)}(P, S)=1$ and $P$ is projective, by 
applying $\Hom_{\dualcat(k)}(\ast, S)$ to \eqref{dimextAW1} we obtain an isomorphism 
\begin{equation}\label{dimextAW2}
\Hom_{\dualcat(k)}((W+\mm^2)\wtimes_E P, S)\cong \Ext^1_{\dualcat(k)}(T,S).
\end{equation}
Let $n$ be the dimension of $ (W+\mm^2)/ (W\mm +\mm^3)$ then the exact sequence of right $E$-modules 
$0\rightarrow W\mm +\mm^3\rightarrow W+\mm^2\rightarrow k^{\oplus n}\rightarrow 0$ leads to an exact 
sequence of $G$-representations: 
\begin{equation}\label{dimextAW3}
0\rightarrow (W\mm+\mm^3)\wtimes_E P\rightarrow (W +\mm^2)\wtimes_E P \rightarrow Q^{\oplus n}\rightarrow 0.
\end{equation}
So for the first equality it is enough to show that any $\psi: (W +\mm^2)\wtimes_E P\rightarrow Q$ is zero 
on $(W\mm+ \mm^3)\wtimes_E P$. Suppose that $\psi(a\wtimes v)\neq 0$ for some 
$a\in W+\mm^2$ and $v\in P$ then the composition  $\varphi:P\rightarrow (W +\mm^2)\wtimes_E P\rightarrow Q$, 
$v\mapsto \psi(a\wtimes v)$ is non-zero. Since $\Hom_{\dualcat(k)}(P, Q)$ 
is one dimensional, $\varphi$ is is trivial on $\mm P$ and so for all $b\in \mm$ we have 
$$0=\psi(a\wtimes b v)= \psi(ab \wtimes v).$$  
Hence, $\psi$ is trivial on $(W+\mm^2)\mm \wtimes_E P= (W\mm +\mm^3)\wtimes_E P$. The last equality 
follows from the exact sequence $0\rightarrow \frac{\mm^2}{W\mm+\mm^3}\rightarrow \frac{W+\mm^2}{W \mm +\mm^3}\rightarrow W\rightarrow 0$.
\end{proof}

\begin{lem}\label{commut} Let $(R, \mm)$ be a local $k$-algebra with $R/\mm\cong k$ and $\mm^3=0$. Suppose  there exists a surjection 
\begin{equation}\label{surjection}
\varphi: R\twoheadrightarrow k[[x_1, \ldots, x_d]]/(x_1,\ldots, x_d)^3,
\end{equation}
where $d=\dim \mm/\mm^2$. Let $r=\lfloor \frac{d}{2}\rfloor$ and further suppose that for every  
$d-r$ dimensional $k$-subspace $W$ of $\mm/\mm^2$ we have 
\begin{equation}\label{ineqf}
\dim \frac{\mm^2}{W\mm}\le \frac{r(r+1)}{2}
\end{equation}
then \eqref{surjection} is an isomorphism. In particular, $R$ is commutative.  
\end{lem}
\begin{proof} Any commutative local $k$-algebra $(A, \mm_{A})$ with  $A/\mm_A= k$, $\mm^{3}_A=0$   and $\dim \mm_A/\mm_A^2\le d$ 
 is a quotient of  $k[[x_1, \ldots, x_d]]/(x_1,\ldots, x_d)^3$. 
Hence, 
$$R^{ab}\cong k[[x_1, \ldots, x_d]]/(x_1,\ldots, x_d)^3,$$
 where $R^{ab}$ is the maximal 
commutative quotient of $R$. Let $\mathfrak a $ be the kernel of $\varphi$.  
Since 
$\dim \mm/\mm^2=\dim \varphi(\mm)/\varphi(\mm)^2=d$, we get that $\mathfrak a$ is contained in $\mm^2$.  
Since $\mm^3=0$, any $k$-subspace $V$ of $\mathfrak a$ is also a two-sided 
ideal of $R$. Suppose that $\mathfrak a\neq 0$ and let $V\subset \mathfrak a $ be any $k$-subspace 
such that the quotient $\mathfrak a/ V$ is one dimensional. The surjection $\mm^2_R\twoheadrightarrow \mm^2_{R/V}$ induces a surjection
$\mm^2_R/ W\mm_{R}\twoheadrightarrow \mm^2_{R/V} /W\mm_{R/V}$. Hence, by  replacing  $R$ with $R/V$ we may assume that $\mathfrak a$ is a one dimensional $k$-vector space. We let $t$ be a basis 
vector of $\mathfrak a$.  
 
 If $a, b\in \mm$ then  the image of $ab-ba$ in $R^{ab}$ is zero. Thus there exists $\kappa(a, b)\in k$ such that
$ab-ba= \kappa(a,b) t$. If $a\in \mm^2$ or $b\in \mm^2$ then $\kappa(a, b)=0$, as $\mm^3=0$. Hence, $\kappa$ defines an alternating bilinear form on 
$\mm/\mm^2$.  

We may choose a basis
 $\mathcal B=\{x_1,\ldots, x_d\}$ of $\mm/\mm^2$  
such that for any two $a, b\in \mathcal B$ we have $\kappa(a, b)=0$, except $\kappa(x_i, x_{d-i+1})=- \kappa(x_{d-i+1}, x_i)=1$,
 $1\le i\le s$, where $d-2s$ is the dimension of $\{a\in \mm/\mm^2: \kappa(a,b)=0, \forall b\in \mm/\mm^2\}$.  Let $W$ be the linear span of $\mathcal S=\{x_1, \ldots, x_{d-r}\}$. The
$k$-subspace of $\mm^2$ spanned by the set 
$\mathcal S\centerdot\mathcal B:=\{ a b : a\in \mathcal S,  b\in \mathcal B\}$ is equal to  $W\mm$. 
The set $\mathcal S\centerdot\mathcal B$ consists of monomials $x^2_i$, $1\le i\le d-r$ and $x_i x_j$ with 
$1\le i\le d-r$, $1\le j\le d$ and $i<j$, since by construction $ab=ba$ for all $a,b \in \mathcal S$.  We note that $d-r\le s$, as $d-2s \ge 0$. 
It follows from \eqref{surjection}
that $\varphi$ induces a bijection between the sets $\mathcal S\centerdot\mathcal B$ and $\varphi(\mathcal S)\centerdot\varphi(\mathcal B)$. 
Since distinct monomials are linearly independent in $R^{ab}$ the set
$\varphi(\mathcal S)\centerdot\varphi(\mathcal B)$ is a basis of $\varphi(W \mm)=\varphi(W)\varphi(\mm)$. Hence, 
the dimension of $W\mm$  is equal to the dimension of $\varphi (W\mm)$,  which is equal to the cardinality of 
of the set $\varphi(\mathcal S)\centerdot\varphi(\mathcal B)$. The latter can be calculated as 
$|\mathcal B| + (|\mathcal B|-1)+\ldots+ (|\mathcal B|- |\mathcal S|+1)$. Since the dimension of $\varphi(\mm)^2$ is equal to 
$|\mathcal B|+(|\mathcal B|-1)+\ldots+1$, we deduce that the dimension of $\varphi(\mm)^2/\varphi(W \mm)$ is equal to 
$1+2+\ldots+ (|\mathcal B|-|\mathcal S|)= \frac{r(r+1)}{2}$.  Since we have assumed $\mathfrak a\neq 0$ we have $\dim \mm^2 > \dim \varphi(\mm)^2$ and hence $\dim \mm^2/ W\mm > \dim \varphi(\mm)^2/\varphi(W\mm)$. This contradicts
\eqref{ineqf}.
\end{proof}

\begin{thm}\label{crit} Let $d:=\dim \mm/\mm^2$ and $r=\lfloor \frac{d}{2}\rfloor$ and  suppose that there exists a surjection
$\wE\twoheadrightarrow \OO[[x_1,\ldots, x_d]]$.  Further, suppose that for every exact sequence 
\begin{equation}\label{crit1}
0\rightarrow Q^{\oplus r} \rightarrow  T\rightarrow Q\rightarrow 0
\end{equation}
with $\dim \Hom_{\dualcat(k)}(T, S)=1$ we have $\dim \Ext^1_{\dualcat(k)}(T, S)\le \frac{r(r-1)}{2}+ d$ then $\wE\cong \OO[[x_1,\ldots, x_d]]$.
\end{thm}
\begin{proof} The bound on $\dim \Ext^1_{\dualcat(k)}(T, S)$ and Lemmas \ref{AW}, \ref{commut} imply that $E/\mm^3$ is commutative. Hence, the commutator 
of any two elements in $\gr^1_{\mm} E$ is zero in $\gr^{\bullet}_{\mm} E$. Thus the graded ring $\gr^{\bullet}_{\mm}(E)$ is commutative,
 as it is generated as a ring by $\gr^1_{\mm} E$ over $\gr^0_{\mm} E\cong k$, and the result follows from  Lemma \ref{graded}.
\end{proof} 

In the applications to $G=\GL_2(\Qp)$, $r$ will turn out to be equal to $1$. We finish 
the section with lemmas of technical nature tailored for this situation. 

Let $a, b \in \Ext^1_{\dualcat(k)}(Q, Q)$ be equivalence classes of extensions of 
$0\rightarrow Q\rightarrow A\overset{\alpha}{\rightarrow}Q\rightarrow  0$ and 
$0\rightarrow Q\overset{\beta}{\rightarrow} B\rightarrow Q \rightarrow 0$, respectively. 
We denote by $a\circ b\in \Ext^2_{\dualcat(k)}(Q, Q)$ the equivalence class of 
$0\rightarrow Q\rightarrow A \overset{\beta\circ\alpha}{\rightarrow} B\rightarrow Q\rightarrow  0$. 
Applying $\Hom_{\dualcat(k)}(Q, \ast)$ we get an exact sequence:
\begin{equation}\label{delta1}
 \Ext^1_{\dualcat(k)}(Q, A)\rightarrow \Ext^1_{\dualcat(k)}(Q, Q)\overset{\partial_1}\rightarrow \Ext^2_{\dualcat(k)}(Q, Q).
\end{equation}
Applying $\Hom_{\dualcat(k)}(\ast, Q)$ we get an exact sequence
\begin{equation}\label{delta2}
 \Ext^1_{\dualcat(k)}(A, Q)\rightarrow \Ext^1_{\dualcat(k)}(Q, Q)\overset{\partial_2}\rightarrow \Ext^2_{\dualcat(k)}(Q, Q).
\end{equation}
Then $\partial_1(b)= a\circ b$ and $\partial_2(b)= b\circ a$, \cite[\S7.6 Prop 5]{bour}.

\begin{lem}\label{alter} The following are equivalent:
\begin{itemize}
\item[(i)]  $\Hom(E, k[x]/(x^3))\rightarrow \Hom(E, k[x]/(x^2))$ is surjective;
\item[(ii)] $a\circ a=0$ for all $a\in \Ext^1_{\dualcat(k)}(Q, Q)$.
\end{itemize}
\end{lem}
\begin{proof} By $\Hom$ in (i) we mean  homomorphisms of local $k$-algebras. We will show that (i) implies (ii).
An extension $a$ may be considered as a deformation of $Q$ to $k[x]/(x^2)$ and hence 
as $\varphi\in \Hom(E, k[x]/(x^2))$ by Theorem \ref{repnonC}. More precisely, $a$ is the equivalence class of 
\begin{equation}\label{class1} 
0\rightarrow k \wtimes_{E}P\rightarrow k[x]/(x^2)\wtimes_{E, \varphi} P\rightarrow  k \wtimes_{E}P\rightarrow 0
\end{equation}
By assumption there exists $\psi\in \Hom(E, k[x]/(x^3))$ lifting $\varphi$. This gives an extension 
 \begin{equation}\label{class2} 
0\rightarrow k \wtimes_{E}P\rightarrow k[x]/(x^3)\wtimes_{E, \psi} P\rightarrow  k[x]/(x^2) \wtimes_{E, \varphi}P\rightarrow 0.
\end{equation}
The image of \eqref{class2} in $\Ext^1_{\dualcat(k)}(Q,Q)$ via \eqref{delta2} is the extension class of 
\begin{equation}\label{class3} 
0\rightarrow k \wtimes_{E}P\rightarrow (x)/(x^3)\wtimes_{E, \psi} P\rightarrow k \wtimes_{E}P\rightarrow 0
\end{equation}
and is equal to $a$. Hence, $a$ lies in the kernel of $\partial_2$ and so $a\circ a=0$. 

Conversely suppose that $a\circ a=0$ then since $a\circ a=\partial_2(a)$ there exists a commutative diagram:

\begin{displaymath}\label{class4}
\xymatrix@1{ 0 \ar[r] &Q\ar[d]^{=}\ar[r] & A \ar@{^(->}[d]\ar[r]& Q \ar[r]\ar@{^(->}[d] & 0  \\ 
0 \ar[r] & Q\ar[r] & B \ar[r]\ar@{->>}[d]& A \ar[r]\ar@{->>}[d] & 0\\
& & Q\ar[r]^=& Q}
\end{displaymath}
Since $P$ is projective and $a$ is non-split there exists a surjection $\psi: P\twoheadrightarrow B$ lifting 
$\varphi: P\twoheadrightarrow A$. It is enough to show that  $\mathfrak a:=\{b\in \mm: \psi\circ b=0\}$ is a two-sided ideal of 
$E$. Since the composition $P\overset{\psi}{\rightarrow} B\twoheadrightarrow Q$ is trivial on $\mm P$,  
the image of $\psi\circ b : P\rightarrow B$ is contained in $A\cong \Ker(B\twoheadrightarrow Q)$ for all $b\in \mm$. Now $ \Hom_{\dualcat(k)}(P, A)$ 
is $2$-di\-men\-sio\-nal with basis $\varphi$, $\varphi_1: P\twoheadrightarrow Q\hookrightarrow A$.
For a fixed $b\in \mm$ we may write $\psi\circ b = \lambda \varphi +\mu \varphi_1$. For all $c\in E$ we have 
$\varphi\circ  c \equiv \psi\circ c \pmod {Q}$ and hence $\varphi\circ c =0$ if $c\in \mathfrak a$. Thus we obtain 
$\psi\circ b\circ c  = \lambda \varphi\circ c +\mu \varphi_1\circ c=0$ for all $c\in \mathfrak a$. Hence, 
$\mathfrak a$ is a two sided ideal.  
 \end{proof} 

\begin{lem}\label{alter1} Let $(R, \mm)$ be a commutative local artinian $k$-algebra  with $\mm^3=0$
and $R/\mm=k$.  Let $d$ be the dimension of $\mm/\mm^2$ as a $k$-vector space. Then the following 
are equivalent:
\begin{itemize}
\item[(i)] $\Hom(R, k[x]/(x^3))\rightarrow \Hom(R, k[x]/(x^2))$ is surjective;
\item[(ii)] $R\cong k[[x_1,\ldots, x_d]]/(x_1, \ldots, x_d)^3$.
\end{itemize}
\end{lem}
\begin{proof} Let $S:=k[[x_1,\ldots, x_d]]/(x_1, \ldots, x_d)^3$ and $\mm_S$ be the maximal ideal of $S$. Since 
$R$ is commutative and $\dim \mm/\mm^2=\dim \mm_S/\mm_S^2$ there exists a surjection 
$\varphi: S\twoheadrightarrow R$, inducing an isomorphism $S/\mm^2_S\cong R/\mm^2$. For $1\le i\le j\le d$ define
$\varphi_{ij}: R\rightarrow R/\mm^2\cong S/\mm_S^2\rightarrow k[x]/(x^2)$, where the last arrow is given by 
sending $x_i\mapsto x$, $x_j\mapsto x$ and $x_k\mapsto 0$, if $k\neq i$ and $k\neq j$. By assumption there exists 
$\psi_{ij}: R\rightarrow k[x]/(x^3)$ lifting $\varphi_{ij}$. Let $\kappa$ be the composition
$$ S\overset{\varphi}{\twoheadrightarrow} R\overset{\prod \psi_{ij}}{\longrightarrow} \prod_{1\le i\le j\le d} k[x]/(x^3).$$ 
The kernel of $\kappa$ is contained in $\mm^2_S$.
Any element $y\in\mm_S^2$ maybe written as $y=\sum_{1\le i\le j\le d} a_{ij} x_i x_j$, and 
$\psi_{ii}(\varphi(y))= a_{ii} x^2$ and $\psi_{ij}(\varphi(y))= (a_{ii}+a_{ij} +a_{jj})x^2$, if $i<j$. Hence, $\kappa$ is 
injective and so $\varphi$ is injective. The other implication is trivial. 
     
\end{proof}

\begin{lem}\label{equivalentcond} Assume that $\Hom(E, k[x]/(x^3))\rightarrow \Hom(E, k[x]/(x^2))$ is surjective and 
let $a$ be a non-zero  extension class of $0\rightarrow Q\rightarrow T\overset{\alpha}{\rightarrow}Q\rightarrow  0$.
Then the following are equivalent: 
\begin{itemize} 
\item[(i)] the kernel of $\Ext^1_{\dualcat(k)}(Q, Q)\rightarrow \Ext^2_{\dualcat(k)}(Q, Q)$, $b\mapsto b\circ a$ is at most
$1$-di\-men\-sio\-nal;
\item[(ii)] $\dim \Ext^1_{\dualcat(k)}(T, Q)=\dim \Ext^1_{\dualcat(k)}(T, S)\le\dim \mm/\mm^2$;
\item[(iii)] the kernel of $\Ext^1_{\dualcat(k)}(Q, Q)\rightarrow \Ext^2_{\dualcat(k)}(Q, Q)$, $b\mapsto a\circ b$ is  at most $1$-di\-men\-sio\-nal;
\item[(iv)] $\dim \Ext^1_{\dualcat(k)}( Q, T)\le \dim \mm/\mm^2$.
\end{itemize}
If the conditions hold then  all the inequalities above are in fact equalities.
\end{lem}
\begin{proof} Since $\circ$ is bilinear, Lemma \ref{alter} gives $a\circ b= - b\circ a$. Thus (i) is equivalent to (iii). 
We show the equivalence of (i) and (ii).  Let 
$\Upsilon$ be the kernel  of  $\Ext^1_{\dualcat(k)}(Q, Q)\rightarrow \Ext^2_{\dualcat(k)}(Q, Q)$. Since $Q/\rad Q\cong S$ 
is irreducible and occurs with multiplicity $1$ we have $\dim \Hom_{\dualcat(k)}(Q, Q)=1$. Since $a$ is non-split,
we also have $\dim \Hom_{\dualcat(k)}(Q, T)=\dim \Hom_{\dualcat(k)}(T, Q)=1$. Since $\dim \Ext^1_{\dualcat(k)}(Q, Q)=\dim \mm/\mm^2$ 
the exact sequence
$$\Hom_{\dualcat(k)}(Q,Q)\hookrightarrow \Ext^1_{\dualcat(k)}(Q, Q)\rightarrow \Ext^1_{\dualcat(k)}(T, Q)\twoheadrightarrow \Upsilon$$
gives $\dim \Ext^1_{\dualcat(k)}(T, Q)= \dim \mm/\mm^2 + \dim \Upsilon -1$. Lemma \ref{T} implies  
$\Ext^1_{\dualcat(k)}(T, Q)$ and $ \Ext^1_{\dualcat(k)}(T, S)$ have the same dimension, so (i) is equivalent to (ii).  It follows from Lemma \ref{alter} that $a\circ a=0$ and so 
$a\in \Upsilon$, which  implies that $\dim \Upsilon\ge 1$ and so  $\dim \Ext^1_{\dualcat(k)}(T, Q)=\dim \Ext^1_{\dualcat(k)}(T, S)\ge\dim \mm/\mm^2$.
This implies that if (i) or (ii) hold then the inequalities are in fact equalities. The same proof  shows that (iii) is equivalent to (iv).
\end{proof}

\begin{lem}\label{enoughisenough} Assume that $\Hom(E, k[x]/(x^3))\rightarrow \Hom(E, k[x]/(x^2))$ is surjective and
 that there exists a $(d-1)$-di\-men\-sio\-nal subspace $V$ of $\Ext^1_{\dualcat(k)}(Q, Q)$ such that 
the equivalent conditions of Lemma \ref{equivalentcond} hold for every non-zero $a\in V$. Then they hold
for every non-zero $a\in \Ext^1_{\dualcat(k)}(Q, Q)$. 
\end{lem}
\begin{proof} Let $\varphi_a: \Ext^1_{\dualcat(k)}(Q, Q)\rightarrow \Ext^2_{\dualcat(k)}(Q, Q)$ be the map 
$b\mapsto b\circ a$.  Lemma \ref{alter}  implies that $a$ lies in $\Ker \varphi_a$. Thus (i) in Lemma \ref{equivalentcond}
holds if and only if $a$ spans $\Ker \varphi_a$. If $a\in V$ then the conditions hold by assumption and so $\Ker \varphi_a= \langle a\rangle$. If $a\not\in V$ then  
using $\varphi_a(b)=-\varphi_b(a)$ we deduce that the restriction of $\varphi_a$ to $V$ is injective. Thus the image of $\varphi_a$ is at least $d-1$ 
dimensional, and so the kernel is at most $1$-dimensional. Hence, the conditions of Lemma \ref{equivalentcond} 
hold for $a$. 
\end{proof}

\section{Banach space representations}\label{banach}
From now on we assume that $G$ is a $p$-adic analytic group. The following fact is essential: for every  
compact  open subgroup $H$ of $G$ the completed group ring $\OO[[H]]$ is noetherian. 
An \textit{$L$-Banach space representation} $\Pi$ of $G$ is an $L$-Banach space $\Pi$ together with a $G$-action 
by continuous linear automorphisms such that the map $G\times \Pi\rightarrow \Pi$ describing the action is continuous.
A Banach space representation $\Pi$ is called \textit{unitary}, if there exists a 
$G$-invariant norm defining the topology on $\Pi$. The existence of such norm is equivalent to the existence of 
an open  bounded $G$-invariant $\OO$-lattice $\Theta$ in $\Pi$. A unitary $L$-Banach space representation 
is \textit{admissible} if $\Theta\otimes_{\OO} k$ is an admissible (smooth) representation of $G$, this means 
that the space of invariants  $(\Theta\otimes_{\OO} k)^H$ is finite dimensional for every open subgroup $H$ of $G$.
We note that it is enough to check this for a single open pro-$p$ subgroup of $G$, see for example \cite[6.3.2]{coeff}.
Our definition of admissibility does not depend on the choice of $\Theta$.  Moreover, it is equivalent to that of \cite{iw}, see \cite[6.5.7]{em1}, which requires $\Theta^d:=\Hom_{\OO}(\Theta, \OO)$ 
to be a finitely generated module over  $\OO[[H]]$.
We say that an $L$-Banach space representation $\Pi$ is \textit{irreducible}, if it does not contain 
a proper closed $G$-invariant subspace. We say that $\Pi$ is \textit{absolutely irreducible} if $\Pi\otimes_L L'$ is 
irreducible for every finite extension $L'$ of $L$. 

\begin{lem}\label{endoirrban} Let $\Pi$ be an absolutely irreducible and admissible unitary $L$-Banach space representation of $G$ and let $\phi\in \End_{L[G]}^{cont}(\Pi)$. 
If the algebra $L[\phi]$ is finite dimensional over $L$ then 
$\phi\in L$. 
\end{lem}
\begin{proof} Let $f\in L[X]$ be the minimal polynomial of $\phi$ over $L$, and let $L'$ be the splitting field of $f$. If $M$ is a finitely generated $L[[H]]:= L\otimes \OO[[H]]$ module, then $M_{L'}$ is a finitely generated 
$L'[[H]]$-module. Thus, it follows from \cite[Thm 3.5]{iw} that $\Pi_{L'}$ is an admissible unitary $L'$-Banach space representation 
of $G$. Since by assumption $\Pi_{L'}$ is irreducible, it follows from the proof of \cite[Cor. 3.7]{iw} 
that any non-zero continuous linear $G$-equivariant map $\psi: \Pi_{L'}\rightarrow \Pi_{L'}$ is an isomorphism. 
Since $f(\phi)=0$ using this we may find  $\lambda\in L'$ such that $f(\lambda)=0$ and $\phi\otimes \id -\lambda$ kills 
$\Pi_{L'}$.  Now $\Gal(L'/L)$ acts on $\Pi_L'$ via   $\sigma(v\otimes\mu)= v\otimes \sigma(\mu)$ for all $\mu\in L'$.  
Choose a non-zero $v\in \Pi$, then $\phi(v)\in \Pi$, and hence $\sigma(\lambda) v = \lambda v$ for all $\sigma\in \Gal(L'/L)$. 
This implies $\lambda\in L$, and hence $\phi=\lambda$. 
\end{proof}

\begin{lem}\label{eirrb} Let $\Pi$ be an irreducible admissible unitary $L$-Banach space representation of $G$. 
If $\End^{cont}_{L[G]}(\Pi)=L$ then $\Pi$ is absolutely irreducible.
\end{lem}
\begin{proof} Suppose that $\Pi$ is not absolutely irreducible. Then there exists a finite Galois extension $L'$ of $L$ such that 
$\Pi_{L'}$ contains a closed proper $G$-invariant subspace $\Sigma$.  Since $L'$ is a finite extension of $L$ we have isomorphisms:
$$ \End_{L'[G]}^{cont}(\Pi_{L'})\cong \Hom^{cont}_{L[G]}(\Pi, \Pi_{L'})\cong \End_{L[G]}^{cont}(\Pi)_{L'}\cong L'.$$
Hence, it is enough to show that $\End_{L'[G]}^{cont}(\Pi_{L'})$ contains a non-trivial idempotent.

As observed in the proof of Lemma \ref{endoirrban},  $\Pi_{L'}$
is admissible. This implies that any descending chain of closed $G$-invariant subspaces must become constant. 
Hence we may assume that $\Sigma$ is irreducible (and admissible).  The group $\Gamma:=\Gal(L'/L)$ 
acts on $\Pi_{L'}$ by $G$-equivariant, $L$-linear isometries 
$$r_{\gamma}: \Pi_{L'}\rightarrow \Pi_{L'},\quad v\otimes \lambda\mapsto v\otimes \gamma(\lambda),\quad  \forall \gamma\in \Gamma.$$ 
In particular,
$r_{\gamma}$ is continuous and $r_{\gamma}(\Sigma)$ is a closed $G$-invariant $L'$-subspace of $\Pi_{L'}$. 
Since $\Sigma$ is an irreducible  admissible unitary $L'$-Banach space representation of $G$, 
so are $r_{\gamma}(\Sigma)$ for all $\gamma\in \Gamma$. Let $\Upsilon$ be the image of 
the natural map
\begin{equation}\label{defUpsi}
\bigoplus_{\gamma\in \Gamma} r_{\gamma}(\Sigma) \rightarrow \Pi_{L'}.
\end{equation}
Since both representations are admissible $\Upsilon$ is a closed $G$-invariant subspace of $\Pi_{L'}$.
 Now $\Upsilon$ is $\Gamma$-invariant and $\Upsilon^{\Gamma}= \Upsilon \cap \bigcap_{\gamma\in \Gamma} \Ker(r_{\gamma}-1)$ is a closed
$G$-invariant $L$-subspace  of $\Pi_{L'}^{\Gamma}=\Pi$. Linear independence of characters implies that 
if $v\in \Upsilon$ is non-zero then there exists 
$\lambda\in L'$ such that $\sum_{\gamma\in \Gamma} r_{\gamma}(\lambda v)\neq 0$. Hence, $\Upsilon^{\Gamma}$ is non-zero.
Since $\Pi$ is irreducible we deduce that $\Upsilon^{\Gamma}=\Pi$ and hence \eqref{defUpsi} is surjective.

 Now any non-zero continuous $G$-equivariant $L$-linear map between two admissible irreducible unitary $L$-Banach space
representations of $G$ is an isomorphism. Using this fact and arguing by induction on $n$ one may show that 
any quotient of $\bigoplus_{i=1}^n \Pi_i$, where $\Pi_i$ are admissible and irreducible, is semi-simple. Hence, $\Pi_{L'}$
is semi-simple. As we have assumed that $\Pi_{L'}$ is not irreducible $\End_{L'[G]}^{cont}(\Pi_{L'})$ contains a 
non-trivial idempotent.
\end{proof}

\begin{lem}\label{commen} Let $\Pi$ be a unitary $L$-Banach space representation of $G$,  let $\Theta$ and $\Xi$ be open bounded $G$-invariant lattices
in $\Pi$, and let $\pi$ be an irreducible smooth $k$-representation of $G$. Then $\pi$ is a subquotient of $\Theta\otimes_{\OO} k$ 
if and only if it is a subquotient of $\Xi\otimes_{\OO} k$. Moreover, if $\Theta\otimes_{\OO} k$ is a $G$-representation of finite length then 
so is $\Xi\otimes_{\OO} k$, and their semi-simplifications are isomorphic.
\end{lem}
\begin{proof} Let $\pi\hookrightarrow J$ be an injective envelope of $\pi$ in $\Mod^{\mathrm{sm}}_G(k)$ the category of smooth $k$-representations 
of $G$. Since $J$ is injective, $\Hom_G(\ast, J)$ is exact, thus  if $\pi$ occurs as a subquotient of some smooth $k$-representation 
$\kappa$, then $\Hom_{G}(\kappa, J)\neq 0$. Conversely, if there exists  some non-zero $\varphi: \kappa \rightarrow J$,  then the image of $\varphi$
must contain $\pi$, as $\pi\hookrightarrow J$ is essential. Further, if $\kappa$ is of finite length the same argument shows 
that $\pi$ occurs in $\kappa$ with multiplicity $\dim \Hom_G(\kappa, J)$. Since $\Theta\otimes_{\OO} k$ and $\Xi \otimes_{\OO} k$ 
are smooth representations of $G$, the assertion of the lemma is equivalent to $\Hom_G(\Theta\otimes_{\OO} k, J)\neq 0$ 
if and only if $\Hom_G(\Xi\otimes_{\OO} k , J)\neq 0$;  $\Theta\otimes_{\OO} k$ is of finite length 
if and only if $\Xi\otimes_{\OO} k$ is of finite length, in which case 
$$\dim \Hom_G(\Xi\otimes_{\OO} k , J)=\dim \Hom_G(\Theta\otimes_{\OO} k , J).$$
Since any two open bounded lattices in $\Pi$ are commensurable, 
one can show this by adapting the proof of analogous statement for finite groups, see the 
proof of Theorem 32 in \S15.1 of \cite{serre} 
and use the exactness of $\Hom_G(\ast, J)$. 
\end{proof} 

Let $\Pi$ be a unitary $L$-Banach space representation of $G$ and $\Theta$ an open bounded $G$-invariant lattice in $\Pi$. We denote 
by $\Theta^d$ its Schikhof dual
$$ \Theta^d:=\Hom_{\OO}(\Theta, \OO)$$
equipped with the topology of pointwise convergence. If $\Theta\otimes_{\OO} k$ is a $G$-rep\-re\-sen\-ta\-tion of finite length, then we denote by 
$\overline{\Pi}$ its semi-simplification
$$\overline{\Pi}:=(\Theta\otimes_{\OO} k)^{ss}.$$
Lemma \ref{commen} shows that $\overline{\Pi}$ does not depend on the choice of $\Theta$. 
 
\begin{lem} $\Theta^d$ is an object of $\Mod_G^{\mathrm{pro\, aug}}(\OO)$. 
\end{lem} 
\begin{proof} For every $n\ge 1$, $\Theta/\varpi^n \Theta$ is a smooth representation of $G$ 
on an $\OO$-torsion module, thus $(\Theta/\varpi^n \Theta)^{\vee}$ is an object of $\Mod_G^{\mathrm{pro\, aug}}(\OO)$. 
It follows from the proof of \cite[Lem. 5.4]{comp} that we have a topological 
isomorphism: 
\begin{equation}\label{Thetadnew}
\Theta^d \otimes_{\OO} \OO/\varpi^n \OO\cong (\Theta/\varpi^n \Theta)^{\vee}
\end{equation} 
Thus $\Theta^d\cong \underset{\longleftarrow}{\lim}\, \Theta^d/\varpi^n \Theta^d\cong 
\underset{\longleftarrow}{\lim}\, (\Theta/\varpi^n \Theta)^{\vee}$ is an object of $\Mod_G^{\mathrm{pro\, aug}}(\OO)$.
\end{proof}

\begin{lem}\label{imageofphi} Suppose that $\Pi$ is irreducible and admissible and let $\phi: M\rightarrow \Theta^d$ be a non-zero morphism 
in $\Mod_G^{\mathrm{pro\, aug}}(\OO)$, then there exists an open bounded $G$-invariant lattice $\Xi$ in $\Pi$ such that 
$\Xi^d=\phi(M)$.
\end{lem}
\begin{proof} Let $H$ be an open  $p$-adic analytic pro-$p$ subgroup of $G$. 
The completed group algebra $\OO[[H]]$ is noetherian. The admissibility 
of $\Pi$ is equivalent to $\Theta^d$ being a finitely generated $\OO[[H]]$-module. Hence,  
$\phi(M)$ is a finitely generated $\OO[[H]]$-submodule of $\Theta^d$ and is $\OO$-torsion free. 
Hence, there
exist a unique Hausdorff topology on $\phi(M)$ such that $\OO[[H]]\times \phi(M)\rightarrow \phi(M)$ is continuous, 
\cite[Prop 3.1 (i)]{iw}, and $\phi(M)$ is a closed submodule of $\Theta^d$ with respect to this topology, \cite[Prop 3.1 (ii)]{iw}.
The uniqueness of the topology on $\phi(M)$ implies that the submodule topology coincides with the quotient topology.
Since $\phi(M)$ is $G$-invariant and non-zero and $\Pi$ is irreducible it follows from \cite[Thm 3.5]{iw}, 
that $\Pi$  is naturally isomorphic to the Banach space representation 
$\Hom^{cont}_{\OO}(\phi(M), L)$ with the topology induced by the supremum norm.  
If we let $\Xi:= \Hom^{cont}_{\OO}(\phi(M), \OO)$ then $\Xi$ will 
be an open bounded $G$-invariant lattice in $\Pi$ and it follows from the proof of \cite[Thm 1.2]{iw} 
that $\Xi^d=\phi(M)$. 
\end{proof}

Let $\Mod^?_{G}(\OO)$ be  a full subcategory of $\Mod^{\mathrm{l\, fin}}_G(\OO)$ closed under subquotients and 
arbitrary direct sums in $\Mod^{\mathrm{l\, fin}}_G(\OO)$.  Let $\dualcat(\OO)$ be a full subcategory of $\Mod_G^{\mathrm{pro\, aug}}(\OO)$ 
anti-equivalent to $\Mod^?_{G}(\OO)$ via Pontryagin duality. We note that $\Mod^?_G(\OO)$ has injective envelopes and 
so $\dualcat(\OO)$ has projective envelopes, see \S \ref{zerosec}. 

\begin{lem}\label{obCO} For an  admissible unitary $L$-Banach space representation $\Pi$ of $G$ the following are equivalent:
\begin{itemize}
\item[(i)] there exists an open bounded $G$-invariant lattice $\Theta$ in $\Pi$ such that 
$\Theta^d$ is an object of $\dualcat(\OO)$;
\item[(ii)] $\Theta^d$ is an object of $\dualcat(\OO)$ for every open bounded $G$-invariant lattice $\Theta$ in $\Pi$.
\end{itemize}
\end{lem}
\begin{proof} Clearly (ii) implies (i). The converse holds because any two open bounded lattices 
are commensurable and $\dualcat(\OO)$ is closed under subquotients.
\end{proof}

\begin{defi} Let $\Ban_G^{\mathrm{adm}}(L)$ be  the category of admissible 
unitary $L$-Banach space representations of $G$ with morphisms continuous $G$-equivariant $L$-linear homomorphisms. 
Let $\Ban^{\mathrm{adm}}_{\dualcat(\OO)}$ be the full subcategory of $\Ban_G^{\mathrm{adm}}(L)$  with objects admissible 
unitary $L$-Banach space representations of $G$ satisfying the conditions of Lemma \ref{obCO}.
\end{defi}

\begin{lem} $\Ban^{\mathrm{adm}}_{\dualcat(\OO)}$ is closed under subquotients in $\Ban^{\mathrm{adm}}_{G}(L)$. In particular, 
it is abelian. 
\end{lem}
\begin{proof} We note that it follows from \cite{iw} and \cite[6.2.16]{em1} that  $\Ban^{\mathrm{adm}}_{G}(L)$ is an abelian category.
Let $\Pi$ be an object of $\Ban^{\mathrm{adm}}_{\dualcat(\OO)}$ and let $\Theta$ be an open bounded $G$-invariant lattice in $\Pi$. 
Then $\Theta^d$ is an object of $\dualcat(\OO)$ and  any subquotient of $\Theta^d$ in $\Mod_G^{\mathrm{pro\, aug}}(\OO)$ lies in $\dualcat(\OO)$,
since  $\dualcat(\OO)$ is a full subcategory of $\Mod_G^{\mathrm{pro\, aug}}(\OO)$ closed under subquotients. Dually this implies that any 
subquotient of $\Pi$ in $\Ban^{\mathrm{adm}}_{G}(L)$ lies in  $\Ban^{\mathrm{adm}}_{\dualcat(\OO)}$. Hence $\Ban^{\mathrm{adm}}_{\dualcat(\OO)}$
is abelian.
\end{proof}

\begin{lem}\label{PitomPI} Let $\wP$ be a projective object in $\dualcat(\OO)$ and let $\wE:=\End_{\dualcat(\OO)}(\wP)$. Let $\Pi$ be in 
$\Ban^{\mathrm{adm}}_{\dualcat(\OO)}$, choose an open bounded $G$-invariant lattice $\Theta$ in $\Pi$ and 
put $\md(\Pi):=\Hom_{\dualcat(\OO)}(\wP, \Theta^d) \otimes_{\OO} L$. Then $\Pi \mapsto \md(\Pi)$ defines 
an exact functor from $\Ban^{\mathrm{adm}}_{\dualcat(\OO)}$ to the category of right $\wE[1/p]$-modules.  
\end{lem} 
\begin{proof} We note that since any two open bounded lattices in $\Pi$ are commensurable the definition 
of $\md(\Pi)$ does not depend on the choice of $\Theta$. Let $0\rightarrow \Pi_1\rightarrow \Pi_2\rightarrow \Pi_3\rightarrow 0$ 
be an exact sequence in $\Ban^{\mathrm{adm}}_{\dualcat(\OO)}$. Let $\Theta$ be an open bounded $G$-invariant lattice in $\Pi_2$. 
Since all the Banach space representations are admissible, $\Pi_1\cap \Theta$ is an open bounded $G$-invariant 
lattice in $\Pi_1$ and the image of $\Theta$ in $\Pi_3$ is an open bounded $G$-invariant lattice in $\Pi_3$. So we have an 
exact sequence $0\rightarrow \Theta_1 \rightarrow \Theta_2\rightarrow \Theta_3\rightarrow 0$ with $\Theta_i$ an open bounded 
$G$-invariant lattice in $\Pi_i$. Dually this gives an exact sequence  
$0\rightarrow \Theta_3^d \rightarrow \Theta_2^d\rightarrow \Theta_1^d\rightarrow 0$ 
in $\dualcat(\OO)$. Since $\wP$ is projective in $\dualcat(\OO)$ we obtain an exact sequence of right  $\wE$-modules:
$$0\rightarrow \Hom_{\dualcat(\OO)}(\wP, \Theta_3^d)\rightarrow \Hom_{\dualcat(\OO)}(\wP, \Theta_2^d)\rightarrow\Hom_{\dualcat(\OO)}(\wP, \Theta_1^d)
\rightarrow 0.$$ 
 The sequence remains exact after tensoring with $L$.
\end{proof} 

\begin{cor}\label{smallest} Let $\wP$ be a projective object in $\dualcat(\OO)$ and let $\Pi$ be in $\Ban^{\mathrm{adm}}_{\dualcat(\OO)}$ then there
exists a smallest  closed $G$-invariant subspace of $\Pi_1$ of $\Pi$ such that $\md(\Pi/\Pi_1)$ is zero.
\end{cor}
\begin{proof} Since $\Pi$ is admissible any descending chain of closed $G$-invariant subspaces must become stationary, \cite[Lemma 5.8]{comp}.
The assertion follows from the exactness of $\md$. 
\end{proof}

In the application we will be in the following situation.

\begin{lem}\label{contextGL2} Let $G=\GL_2(\Qp)$, $\zeta: Z\rightarrow \OO^\times$ be a continuous character of 
the centre of $G$ and let $\dualcat(\OO)$ be the full subcategory of $\Mod^{\mathrm{pro\, aug}}_G(\OO)$ anti-equivalent 
to $\Mod^{\mathrm{lfin}}_{G,\zeta}(\OO)$ by Pontryagin duality, see \S\ref{zerosec}.
 Let $\Pi$ be an admissible $L$-Banach space representation of $G$ with a central 
character $\zeta$  and let $\Theta$ be an open bounded $G$-invariant lattice in $\Pi$. Then 
$\Theta^d$ is an object of $\dualcat(\OO)$. In particular, $\Ban^{\mathrm{adm}}_{\dualcat(\OO)}=\Ban^{\mathrm{adm}}_{G, \zeta}(L)$ the 
category of admissible unitary $L$-Banach space representations of $G$ on which $Z$ acts by the character $\zeta$.
\end{lem}
\begin{proof} Recall that an object $M$ of $\Mod^{\mathrm{pro\, aug}}_G(\OO)$ is an object of $\dualcat(\OO)$ 
if and only if $M\cong \underset{\longleftarrow}{\lim}\, M_i$ where the limit is taken over all the quotients
in $\Mod^{\mathrm{pro\, aug}}_G(\OO)$ of finite length and $Z$ acts on $M$ via $\zeta^{-1}$. 

Since $\Pi$ is admissible $\Theta/\varpi^n \Theta$ is an admissible smooth 
representation of $G$ for all $n\ge 1$. Since $Z$ acts on $\Theta/\varpi^n \Theta$ by a character $\zeta$
\cite[Thm 2.3.8]{ord1} says that any finitely generated subrepresentation of $\Theta/\varpi^n \Theta$ 
is of finite length. Hence by definition $(\Theta/\varpi^n \Theta)^{\vee}$ is an object of $\dualcat(\OO)$.
The assertion follows from \eqref{Thetadnew}.
\end{proof} 
\begin{remar} If the Conjecture formulated by Emerton  in \cite[2.3.7]{ord1} holds then the proof of Lemma \ref{contextGL2} goes through unchanged for a $p$-adic reductive 
group $G$.
\end{remar}

\begin{lem}\label{pisub} Let $\wP$ be a a projective envelope of an irreducible object $S$ in $\dualcat(\OO)$, 
$\pi:=S^{\vee}$ a smooth irreducible $k$-representation of $G$, $\Pi$ an object of $\Ban^{\mathrm{adm}}_{\dualcat(\OO)}$ and $\Theta$ 
an open bounded $G$-invariant lattice in $\Pi$. Then the following are equivalent:
\begin{itemize}
\item[(i)] $\pi$ is a subquotient of $\Theta\otimes_{\OO} k$;
\item[(ii)] $S$ is a subquotient of $\Theta^d\otimes_{\OO} k$;
\item[(iii)] $\Hom_{\dualcat(\OO)}(\wP, \Theta^d\otimes_{\OO} k)\neq 0$;
\item[(iv)] $\Hom_{\dualcat(\OO)}(\wP, \Theta^d)\neq 0$.
\end{itemize} 
\end{lem}
\begin{proof} It follows from 
\eqref{Thetadnew} that (i) is equivalent to (ii). Since $\wP\twoheadrightarrow S$ is essential (iii) implies (ii).
 Since $\dualcat(\OO)$ is closed under subquotients and $\Hom_{\dualcat(\OO)}(\wP, \ast)$
is exact (ii) implies (iii). We have isomorphisms:
\begin{equation} 
\Hom_{\dualcat(\OO)}(\wP, \Theta^d)\cong \Hom_{\dualcat(\OO)}(\wP, \underset{\longleftarrow}{\lim} \, \Theta^d/\varpi^n \Theta^d)\cong 
\underset{\longleftarrow}{\lim} \, \Hom_{\dualcat(\OO)}(\wP, \Theta^d/\varpi^n \Theta^d).
\end{equation}
The transition maps are surjective since $\wP$ is projective. Hence (iii) implies (iv).  
Since $\Theta^d$ is $\OO$-torsion free multiplication by $\varpi^n$ induces isomorphism $\Theta^d/\varpi \Theta^d\cong \varpi^n 
\Theta^d/\varpi^{n+1} \Theta^d$. If $\Hom_{\dualcat(\OO)}(\wP, \Theta^d/\varpi \Theta^d)=0$ then by considering short exact sequences we obtain 
$\Hom_{\dualcat(\OO)}(\wP,  \Theta^d/\varpi^n \Theta^d)=0$ for all $n\ge 1$ and so (iv) implies (iii).
\end{proof}  

\begin{lem}\label{whyfinite?} Let $\wP$, $S$, $\pi$ and $\Theta$ be as in Lemma \ref{pisub}. If $\Hom_{\dualcat(\OO)}(\wP, \Theta^d)\neq 0$ then 
$\pi$ is an admissible representation of $G$. In particular, $\End_{\dualcat(\OO)}(S)\cong \End_G(\pi)$ is a finite field extension of $k$.
\end{lem}
\begin{proof} Let $H$ be an open $p$-adic analytic pro-$p$ subgroup of $G$. Since $\Pi$ is admissible 
$\Theta^d$ is finitely generated over $\OO[[H]]$. Since $\OO[[H]]$ is noetherian it follows from Lemma \ref{pisub}
that $S$ is a finitely generated $\OO[[H]]$-module. Since $H$ is pro-$p$, dually this implies that  $\pi^H$ 
is finite dimensional. Since $\pi$ is irreducible, $\End_G(\pi)$ is a skew field over $k$ contained in 
$\End_k(\pi^H)$. Since $\pi^H$ is finite dimensional $\End_G(\pi)$ is finite dimensional. Since $k$ is a 
finite field $\End_G(\pi)$ is a finite field extension of $k$.
\end{proof}

\begin{lem}\label{mult=rank0} Let $\wP$ be a projective envelope of an irreducible object $S$ in $\dualcat(\OO)$
with $d:=\dim_k \End_{\dualcat(\OO)}(S)$ finite.  
Let $M$ be in $\dualcat(\OO)$, $\OO$-torsion free and such that $M_k:=M\otimes_{\OO} k$ is of finite length in 
$\dualcat(\OO)$. Then $\Hom_{\dualcat(\OO)}(\wP, M)$ is a free $\OO$-module of rank equal to the multiplicity with which $S$ occurs as 
a subquotient of $M_k$ multiplied by $d$. 
\end{lem} 
\begin{proof} Since $M$ is $\OO$-torsion free so is $\Hom_{\dualcat(\OO)}(\wP, M)$. 
Let $m$ be the multiplicity with which $S$ occurs as a subquotient of $M_k$. It follows from Lemma 
\ref{mult=dim} that 
$\Hom_{\dualcat(\OO)}(\wP, M)_k \cong\Hom_{\dualcat(\OO)}(\wP, M_k)$
is an $md$-dimensional $k$-vector space. Since $\Hom_{\dualcat(\OO)}(\wP, M)$ is a compact $\wE$, and hence $\OO$-module, the assertion follows from Nakayama's lemma.
\end{proof}

From now on we assume (unless it is stated otherwise) the following setup. Let $S_1, \ldots, S_n$ be irreducible 
pairwise non-isomorphic objects of $\dualcat(\OO)$ such that $\End_{\dualcat(\OO)}(S_i)$ is finite dimensional over $k$ 
for $1\le i \le n$. Let $\wP$ be a projective envelope of $S:=\oplus_{i=1}^n S_i$ and let $\wE=\End_{\dualcat(\OO)}(\wP)$. 
Recall from \S\ref{zerosec}, that $\wE$ is a compact ring and $\wE/\rad \wE \cong \prod_{i=1}^n \End_{\dualcat(\OO)}(S_i)$, 
where $\rad \wE$ is the Jacobson radical of $\wE$. Moreover, uniqueness of projective envelopes implies that 
$\wP\cong \oplus_{i=1}^n \wP_i$, where $\wP_i$ is a projective envelope of $S_i$ in $\dualcat(\OO)$.  
For $1\le i \le n$ let $\pi_i:=S^{\vee}_i$, so that $\pi_i$ is a smooth 
irreducible $k$-representation of $G$ and $\pi:=\oplus_{i=1}^n \pi_i \cong S^{\vee}$.

\begin{remar} The assumption on the finite dimensionality of $\End_{\dualcat(\OO)}(S_i)$  holds if $\pi_i$ is a subquotient of the reduction modulo $\varpi$ of admissible Banach space representations, see Lemma \ref{whyfinite?}.
\end{remar}   
 
\begin{prop}\label{admfg} Let $\Pi$ be in $\Ban^{\mathrm{adm}}_{\dualcat(\OO)}$ and let $\Theta$ be an open bounded $G$-invariant lattice in $\Pi$. Then 
$\Hom_{\dualcat(\OO)}(\wP, \Theta^d)$ is a finitely generated module over $\wE$.
\end{prop}
\begin{proof}
Let $\md= \Hom_{\dualcat(\OO)}(\wP, \Theta^d)$ and let $M\subseteq \Theta^d$ be the image of the natural map
$\md\wtimes_{\wE} \wP\rightarrow \Theta^d$. We may assume that $M\neq 0$, since otherwise $\md=0$ is finitely generated. 
We apply $\Hom_{\dualcat(\OO)}(\wP, \ast)$ to $\md\wtimes_{\wE} \wP\twoheadrightarrow M\hookrightarrow \Theta^d$ and 
use Lemma \ref{headS0} to obtain $\Hom_{\dualcat(\OO)}(\wP, M)\cong \md$.  Since $\Pi$ is admissible, $\Theta^d$ is a finitely generated $\OO[[H]]$-module, 
which implies that $((\Theta)^d)^{\vee}$ is admissible-smooth. Since the quotients of admissible representations are admissible, we deduce that 
$M^{\vee}$ is admissible-smooth. The  $G$-socle  of $M^{\vee}$ is a finite direct sum of irreducible representations, because every summand 
contributes to invariants by a pro-$p$ subgroup of $G$.
 Hence, $\Hom_G(\pi_i, M^{\vee})$ is a finite dimensional $k$-vector space, for $1\le i\le n$. 
Dually, we obtain that $\Hom_{\dualcat(\OO)}( M, S_i)$ is a finite dimensional $k$-vector space of dimension $d_i$ (say), for all $1\le i \le n$. 
Since $M$ is a quotient of $\md\wtimes_{\wE} \wP$, all the irreducible summands appearing in its cosocle are isomorphic to $S_i$ for some $1\le i\le n$. 
Hence, $\cosoc M\cong \oplus_{i=1}^n S_i^{\oplus n_i}$, with $n_i$ equal to $d_i$ divided by the dimension of $\End_{\dualcat(\OO)}(S_i)$.  We may choose 
a surjection $a: \wP^{\oplus m}\twoheadrightarrow \cosoc M$ for some integer $m$. Since $\wP$ is projective, $a$ factors through $b: \wP^{\oplus m} \rightarrow M$. 
Since $M\twoheadrightarrow \cosoc M$ is an essential epimorphism, $b$ is surjective. We apply $\Hom_{\dualcat(\OO)}(\wP, \ast)$ to $b: \wP^{\oplus m} \twoheadrightarrow M$ 
to obtain a surjection $\wE^{\oplus m} \twoheadrightarrow \md$. 
\end{proof}

\begin{prop}\label{modulequotientnew} Let $\Pi$ be in $\Ban^{\mathrm{adm}}_{\dualcat(\OO)}$ and let $\Theta$ be an open bounded
$G$-invariant lattice in $\Pi$. Suppose that $\Pi$ is irreducible and $\Theta\otimes_{\OO} k$ contains $\pi_i$ as a subquotient 
for some $i$. Let $\phi\in \Hom_{\dualcat(\OO)}(\wP, \Theta^d)$ be non-zero 
and  let
$\mathfrak a:=\{a\in \wE: \phi\circ a=0\}$. There exists an open bounded $G$-invariant 
$\OO$-lattice $\Xi$ in $\Pi$ such that $\phi(\wP)= \Xi^d$.
Moreover, 
\begin{itemize}
\item[(i)] $\Hom_{\dualcat(\OO)}(\wP, \Xi^d)\cong\wE/\mathfrak a$ as a right $\wE$-module;
\item[(ii)] $\Hom_{\dualcat(\OO)}(\wP, \Xi^d)_L$ is an irreducible right $\wE_L$-module;
\item[(iii)] the natural map $\Hom_{\dualcat(\OO)}(\wP, \Xi^d)\wtimes_{\wE} \wP\rightarrow \Xi^d$ is surjective.
\end{itemize}
\end{prop}
\begin{proof} Since by assumption $\Theta\otimes_{\OO} k$ contains $\pi_i$ as a subquotient, 
Lemma \ref{pisub} implies that $\Hom_{\dualcat(\OO)}(\wP, \Theta^d)$ is non-zero.
Let $\phi,\psi\in \Hom_{\dualcat(\OO)}(\wP, \Theta^d)$ be non-zero then 
by Lemma \ref{imageofphi} there exists an open bounded $G$-invariant 
lattice $\Xi$ in $\Pi$ such that $\phi(\wP)=\Xi^d$. 
By applying $\Hom_{\dualcat(\OO)}(\wP, \ast)$ to the exact sequence 
$0\rightarrow \Ker \phi\rightarrow \wP \rightarrow \phi(\wP)\rightarrow 0$
we obtain that 
\begin{equation}\label{determineHom}
\Hom_{\dualcat(\OO)}(\wP, \phi(\wP))= \phi \wE\cong \wE/\mathfrak a.
\end{equation} 
Lemma \ref{imageofphi} implies that  $\psi(\wP)$ is 
commensurable with $\phi(\wP)$. Thus for some $n\ge 0$,  
$\varpi^n \psi(\wP)\subseteq \phi(\wP)$, and hence 
$\varpi^n \psi \in \Hom_{\dualcat(\OO)}(\wP, \phi(\wP))$. It follows from \eqref{determineHom} 
that $\varpi^n \psi= \phi \circ a$ for some $a\in \wE$. Hence,
$\Hom_{\dualcat(\OO)}(\wP, \Xi^d)\otimes_{\OO} L$ is an irreducible $\wE\otimes_{\OO} L$-module. 
The image of the natural map $\ev: \Hom_{\dualcat(\OO)}(\wP, \Xi^d)\wtimes_{\wE}\wP \rightarrow \Xi^d$ will
 contain $\phi(\wP)$, and hence $\ev$ is surjective.
\end{proof} 

\begin{prop}\label{ringsareiso} Let $\Xi$ be as in Proposition \ref{modulequotientnew}  then we have natural isomorphisms of rings:
$$ \End_{\dualcat(\OO)}(\Xi^d)\cong \End_{\wE}(\md)\cong \End_{\dualcat(\OO)}(\md\wtimes_{\wE} \wP),$$
where $\md:=\Hom_{\dualcat(\OO)}(\wP, \Xi^d)$.
\end{prop}
\begin{proof} We note that $\md$ is a compact right $\wE$-module and $\End_{\wE}(\md)$ denotes continuous $\wE$-linear 
endomorphisms of $\md$. We have natural maps 
$$ \End_{\dualcat(\OO)}(\Xi^d)\rightarrow \End_{\wE}(\md)\rightarrow\End_{\dualcat(\OO)}(\md\wtimes_{\wE} \wP),$$
where the first one sends $\phi$ to $\psi\mapsto \phi\circ \psi$, the second one sends 
$\phi$ to $\psi\wtimes v \mapsto \phi(\psi)\wtimes v$. The natural 
map $\ev: \md\wtimes_{\wE}\wP \rightarrow \Xi^d$ is surjective by Proposition \ref{modulequotientnew} (iii), let $K$ be its kernel.  By applying $\Hom_{\dualcat(\OO)}(\wP, \ast)$
to the exact sequence $0\rightarrow K\rightarrow   \md\wtimes_{\wE}\wP \rightarrow \Xi^d\rightarrow 0$ and 
using  Lemma \ref{headS0} we deduce that $\Hom_{\dualcat(\OO)}(\wP, K)=0$.

We claim that $\Hom_{\dualcat(\OO)}(K, \Xi^d)=0$. Suppose we have a non-zero morphism 
$\phi: K\rightarrow \Xi^d$ in $\dualcat(\OO)$. It follows from Lemma \ref{imageofphi} that 
$\phi(K)$ contains $\varpi^n \Xi^d$ for some $n\ge 1$. This implies that 
$\Hom_{\dualcat(\OO)}(\wP, \phi(K))\neq 0$. Since $\wP$ is projective we get  
$\Hom_{\dualcat(\OO)}(\wP, K)\neq 0$, which is a contradiction. The claim implies that 
every $\phi\in \End_{\dualcat(\OO)}(\md\wtimes_{\wE} \wP)$ maps $K$ to itself. Hence we
obtain a well defined map $\End_{\dualcat(\OO)}(\md\wtimes_{\wE} \wP)\rightarrow \End_{\dualcat(\OO)}(\Xi^d)$, which sends $\phi$ to $ \psi\wtimes v+ K \mapsto \phi(\psi\wtimes v)+ K$. The composition of any three
consecutive arrows is the identity map, hence all the maps are isomorphisms.
\end{proof}  

\begin{prop}\label{longproof} Let $\Pi \in \Ban^{\mathrm{adm}}_{\dualcat(\OO)}$ be irreducible and let $\Theta$ be an open bounded
$G$-invariant lattice in $\Pi$. Suppose that $\Theta\otimes_{\OO} k$ contains $\pi_i$ as a subquotient 
for some $i$.  If the centre $\mathcal Z$ of $\wE$ is noetherian and $\wE$ is a finitely generated
 $\mathcal Z$-module then $\Hom_{\dualcat(\OO)}(\wP, \Theta^d)_L$ is finite dimensional over $L$.
\end{prop} 
\begin{proof} Let $\Xi$ be an open bounded $G$-invariant lattice constructed in Proposition \ref{modulequotientnew}. Since $\Theta$ and $\Xi$ are commensurable, 
$\Theta\otimes_{\OO} k$ and $\Xi\otimes_\OO k$ have the same irreducible subquotients by Lemma \ref{commen} and $\Hom_{\dualcat(\OO)}(\wP, \Theta^d)_L\cong \Hom_{\dualcat(\OO)}(\wP, \Xi^d)_L$ as $\wE[1/p]$-modules.

Since $\Pi$ is admissible and irreducible it follows from \cite[Thm.3.5]{iw} that
 the ring $D:=\End^{cont}_{L[G]}(\Pi)$ is a skew field. Since $\Xi$ is an open bounded $G$-invariant
lattice  in $\Pi$, \cite[Prop.3.1]{nfa} implies that $\End_{\OO[G]}(\Xi)$ is an $\OO$-order in $D$. It 
follows from the anti-equivalence of categories established in \cite[Thm 3.5]{iw} that sending $f$ to its Schikhof dual $f^d$ induces an isomorphism  
$B:=\End_{\dualcat(\OO)}(\Xi^d)\cong \End_{\OO[G]}(\Xi)^{op}$ and $B[1/p]\cong D^{op}$. Hence $B[1/p]$ is a skew field and 
since $\Xi$ is $\OO$-torsion free so is 
$B$ and we have an injection $B\hookrightarrow B[1/p]$. 

Let $R$ be the centre of $B$. Since $R$  is contained in a skew field $B[1/p]$ it is an integral domain
and $B[1/p]$ contains the quotient field $K$ of $R$. Let $s \in K\cap B$ be non-zero, then we may 
find non-zero $a, b\in R$ such that $ as =b$. For all $t\in B$ we have 
$ (st-ts) a = b t-tb =0$ as $a$ and $b$ are central. Since $B$ is contained in a skew field we deduce that 
$ st=ts$ for all $t$ and hence $B\cap K= R$. Since $K$ is contained in $B[1/p]$ we deduce that for 
every $x\in K$ there exists $n\ge 0$ such that $p^n x\in R$ and so $K=R[1/p]$. 

Without loss of generality we may assume that $n=1$, so that $S$ is irreducible. This may be seen as follows: since $\wP\cong \oplus_{i=1}^n \wP_i$ and 
thus $\Hom_{\dualcat(\OO)}(\wP_i, \Xi^d)\neq 0$ for some $i$, and if $e\in \wE$ denotes the idempotent such that
$e \wP= \wP_i$, then $\End_{\dualcat(\OO)}(\wP_i)= e \wE e$ is a finitely generated $e \mathcal Z e$-module and $e \mathcal Z e$ is 
contained in the centre of $e\wE e$.  Moreover, since $\mathcal Z$ is noetherian so is $e \mathcal Z e$ and this implies that the centre of 
$e \wE e$ is noetherian.

Let $\md:=\Hom_{\dualcat(\OO)}(\wP, \Xi^d)$. We have a natural map $\mathcal Z\rightarrow 
\End_{\wE}(\md)$, which sends  $z$ to $\psi\mapsto \psi\circ z$. Let $\phi$ and $\mathfrak a$ 
be as in Proposition \ref{modulequotientnew}. It follows from Proposition \ref{modulequotientnew} (i) that 
for every $\alpha\in \End_{\wE}(\md)$ there exists $\beta\in \wE$ such that 
$\alpha(\phi)=\phi\circ \beta$ and the map $\alpha\mapsto \beta+\mathfrak a$ is 
an injection of $\mathcal Z$-modules $\End_{\wE}(\md)\hookrightarrow \wE/\mathfrak a$.
Since by assumption $\wE$ is finitely generated over $\mathcal Z$ and $\mathcal Z$ is noetherian we deduce 
that $\End_{\wE}(\md)$ is finitely generated over $\mathcal Z$. The image of $\mathcal Z$ in $\End_{\wE}(\md)$ is contained in the centre. 
We identify $\End_{\wE}(\md)$ with $B$ using 
Proposition \ref{ringsareiso}. Then the image of $\mathcal Z$ in $B$ is contained in $R$, hence $R$ is a 
$\mathcal Z$-submodule of a finitely generated $\mathcal Z$-module $B$. We deduce that  $R$
is a noetherian ring.

Since $R[1/p]$ is a field and $R$ is a noetherian integral domain, Theorem 146 in \cite{kapa} implies that $R/pR$ is artinian. 
Hence, $R/pR\cong \prod_{i=1}^n (A_i, \nn_i)$, where $(A_i, \nn_i)$ are artinian local rings.   
Let $\bar{\mathcal Z}_i$ be the image of $\mathcal Z$ in $A_i/\nn_i$ via $R/pR\rightarrow A_i\rightarrow A_i/\nn_i$. Since $R$ is a 
finitely generated  
$\mathcal Z$-module, $A_i/\nn_i$ is a finitely generated $\bar{\mathcal Z}_i$-module. Since $A/\nn_i$ is a field we deduce that $\bar{\mathcal Z}_i$ 
is a field. Since $\mathcal Z$ is a local ring with residue field a finite extension of $k$,  Corollary \ref{Zloc}, we deduce that 
$\bar{\mathcal Z}_i$ and hence $A_i/\nn_i$ is a finite extension of $k$. Since $A_i$ is an artinian local ring, $A_i$ is an $A_i$-module 
of finite length with irreducible subquotients isomorphic to $A_i/\nn_i$. Hence $A_i$ is a finitely generated $\OO$-module and so $R/pR$ 
is a finitely generated $\OO$-module. As $\Xi^d$ is $p$-adically 
complete, so is $B$ and hence so is $R$. Thus $R$ is a finitely generated $\OO$-module. Since by assumption
$\wE$ is a finitely generated $\mathcal Z$-module we deduce from Proposition \ref{modulequotientnew} that 
$\md$ is a finitely generated $\mathcal Z$-module and hence $\md $ is a finitely generated $R$-module and 
so a finitely generated $\OO$-module. Thus $\md\otimes_{\OO} L$ is finite dimensional over $L$. 
\end{proof}

\begin{cor}\label{BCBan} Let $\Pi$ be an irreducible admissible unitary $L$-Banach space representation of $G$, let $\Theta$  be an open bounded $G$-invariant lattice in $\Pi$.
If the conditions  of Proposition \ref{longproof} are satisfied then there exists a finite extension $L'$ of $L$ such that 
$\Pi_{L'}$ is isomorphic to a finite direct sum of absolutely irreducible unitary $L'$-representations. 
\end{cor}
\begin{proof} It follows from Propositions \ref{modulequotientnew}, \ref{ringsareiso} and \ref{longproof} that 
$\End^{cont}_G(\Pi)$ is a skew field, finite dimensional over $L$. Let $L'$ be a finite Galois extension of $L$ splitting 
$\End^{cont}_G(\Pi)$. Then $\End^{cont}_G(\Pi_{L'})\cong \End^{cont}_{G}(\Pi)_{L'}$ is a matrix algebra over $L'$. 
Let $\{e_i\}_{1\le i\le n}$ be the complete set of orthogonal idempotents and let $\Pi_i:= e_i(\Pi_{L'})$. Then 
$\Pi_{L'}\cong \oplus_{i=1}^n \Pi_i$ and $\End^{cont}_G(\Pi_i)= e_i \End^{cont}_G(\Pi_{L'}) e_i = L'$. 
It follows from the proof of Lemma \ref{eirrb} that, after possibly enlarging $L'$, we may assume that $\Pi_i$  is semi-simple.
Since $\End^{cont}_G(\Pi_i)\cong L'$, we deduce that $\Pi_i$ is irreducible and hence absolutely irreducible by Lemma \ref{eirrb}.
\end{proof}

We equip every  finitely generated $\OO$-module (resp. every finite dimensional $L$-vector space) with the $p$-adic topology.

\begin{lem}\label{toponmd2} If $\wE$ is right noetherian then any $\OO$-linear right action of $\wE$ on a finite dimensional 
$L$-vector space is continuous. 
\end{lem}
\begin{proof} Let $\md_L$ be a finite dimensional $L$-vector space with an $\OO$-linear right $\wE$-action. 
Choose a basis $\{v_1, \ldots, v_n\}$ of $\md_L$ and let $\md:= v_1 \wE+\ldots+v_n \wE$. Since $\wE$ is right noetherian 
the kernel of $\wE^{\oplus n}\twoheadrightarrow \md$ is finitely generated as a right $\wE$-module and, since $\wE$ is compact, the
kernel is a closed submodule of $\wE^{\oplus n}$. Thus the quotient topology on $\md$ is Hausdorff and it has a system of open neighborhoods of $0$ 
consisting of $\wE$-modules, and in particular of $\OO$-modules. The action of $\OO$ on $\md$ via 
$\OO\rightarrow \wE$ on $\md$ is continuous for the quotient topology. Any compact linear-topological $\OO$-torsion free $\OO$-module is isomorphic to
$\prod_{i\in  I} \OO$ for some 
set $I$, see Remark 1.1 in \cite{iw}. Since $\md$ is contained in a finite dimensional $L$-vector space  we deduce that 
$\md$ is an $\OO$-module of finite rank. Thus, for each $n\ge 1$ the quotient topology on $\md/\varpi^n \md$ is discrete, as it is Hausdorff and the underlying set 
is finite. In particular, the sets $\varpi^n \md$ are open in the quotient topology on $\md$, for all $n\ge 0$. Since $\md$ is $\varpi$-adically complete, we deduce that 
the sets $\varpi^n \md$ for $n\ge 0$ build a system of open neighborhoods of $0$ in the quotient topology on $\md$. In particular, the quotient topology and the $p$-adic topology on $\md$ coincide. 
Let $n\ge 0$, and let $v\in \md_L$. The same argument as above shows that $(v\wE  +\varpi^n \md)/\varpi^n  \md$ with the discrete topology is a topological $\wE$-module. This implies that the set 
$\mathfrak a(v, n):= \{a\in \wE: va\in \varpi^n \md\}$ is open in $\wE$. Let $U$ be the preimage in $\wE\times \md_L$ of $v+ \varpi^n \md$. If  $(a, w)\in U$, then 
$(a + \mathfrak a(w, n), w +\varpi^n \md)$ is  open in $\wE\times \md_L$, and is a subset of $U$ containing $(a, w)$. In particular, $U$ is open and hence the action of $\wE$ on $\md_L$ is continuous. 
\end{proof}

\begin{prop}\label{moregen} 
Let $\md$ be a compact right $\wE$-module, free of finite rank  over $\OO$. Assume that $(\wE/\rad \wE) \wtimes_{\wE} \wP$ is of finite length
in $\dualcat(\OO)$ and is a finitely generated $\OO[[H]]$-module, where $\rad \wE$ is the Jacobson radical of $\wE$. Then $\md\wtimes_{\wE}\wP$ is finitely generated over $\OO[[H]]$ and 
 $(\md \wtimes_{\wE} \wP)\otimes_\OO k$ is of finite length in $\dualcat(\OO)$. 
\end{prop} 
\begin{proof} 
Let $\mathrm n$ be a finite dimensional $k$-vector space with a continuous $\wE$-action. If $\mathrm n$ is an irreducible 
$\wE$-module then it is killed by $\rad \wE$, and hence 
$$ \mathrm n \wtimes_{\wE} \wP \cong \mathrm n \wtimes_{\wE/\rad \wE} ((\wE/\rad \wE) \wtimes_{\wE} \wP).$$ 
Thus it follows from our assumptions that $ \mathrm n \wtimes_{\wE} \wP$ is of finite length in $\dualcat(\OO)$ and 
is a finitely generated $\OO[[H]]$-module. In general, arguing inductively on the dimension of $\mathrm n$ we deduce 
that $\mathrm n \wtimes_{\wE} \wP$ is of finite length in $\dualcat(\OO)$ and 
is a finitely generated $\OO[[H]]$-module. Applying $\wtimes_{\wE} \wP$ to the exact sequence  
$\md\overset{\varpi}{\rightarrow} \md \rightarrow \md_k \rightarrow 0$ we get 
$$(\md\wtimes_{\wE} \wP)\otimes_{\OO} k \cong \md_k \wtimes_{\wE} \wP.$$ Nakayama's lemma for compact $\OO[[H]]$-modules 
implies that $\md\wtimes_{\wE} \wP$ is a finitely generated $\OO[[H]]$-module, see \cite[Cor 1.5]{bru}.
\end{proof}

\begin{remar} If $S$ is an irreducible object in $\dualcat(\OO)$ and $\kappa: \wP\twoheadrightarrow S$ is its projective envelope, then 
$\rad \wE=\{ \phi\in \wE: \kappa\circ \phi=0\}$, which is the ideal $\mm$ defined in Definition \ref{Em}, and $(\wE/\rad \wE )\wtimes_{\wE} \wP$ is the 
object $Q$ considered in \S \ref{firstsec}, see Remark \ref{Qject} and Lemma \ref{inclusion}.
\end{remar}

\begin{lem}\label{getback} Under the hypotheses of Proposition \ref{moregen} the maximal 
$\OO$-torsion free quotient $(\md\wtimes_{\wE}\wP)_{\mathrm{tf}}$
of $\md\wtimes_{\wE}\wP$ is an object of $\dualcat(\OO)$. Moreover, 
$$\Hom_{\dualcat(\OO)}(\wP, (\md\wtimes_{\wE}\wP)_{\mathrm{tf}})\cong \md.$$
\end{lem}
\begin{proof} Since $\OO[[H]]$ is noetherian and  $\md\wtimes_{\wE}\wP$ is finitely generated, the torsion submodule 
$(\md\wtimes_{\wE}\wP)_{\mathrm{tors}}$ is finitely generated, and hence is equal to the kernel of multiplication by $\varpi^n$ for $n$ large enough.  
So $(\md\wtimes_{\wE}\wP)_{\mathrm{tors}}$ and $(\md\wtimes_{\wE}\wP)_{\mathrm{tf}}$ are both objects of $\dualcat(\OO)$.
Now $\Hom_{\dualcat(\OO)}(\wP, \md\wtimes_{\wE} \wP)\cong \md$, see Lemma \ref{headS0}, is $\OO$-torsion free. 
Hence, 
$\Hom_{\dualcat(\OO)}(\wP, (\md\wtimes_{\wE} \wP)_{\mathrm{tors}})=0$.
Since $\wP$ is projective we obtain an isomorphism :
$$\md\cong \Hom_{\dualcat(\OO)}(\wP, \md\wtimes_{\wE} \wP)\cong \Hom_{\dualcat(\OO)}(\wP, (\md\wtimes_{\wE} \wP)_{\mathrm{tf}}).$$
\end{proof}

\begin{defi}\label{defPimd} Under the hypotheses of Proposition \ref{moregen} to a right $\wE$-module $\md$ free of 
finite rank over $\OO$ we associate an admissible unitary $L$-Banach space representation of $G$:
$$ \Pi(\md):= \Hom^{cont}_{\OO}((\md\wtimes_{\wE}\wP)_{\mathrm{tf}}, L)$$
with the topology induced by the supremum norm.
\end{defi}

\begin{remar}\label{killTorsion}We define $\Pi(\md)$ in terms of the maximal torsion free quotient of $\md\wtimes_{\wE}\wP$, as this allows us to appeal to the results
of \cite{iw}. Since any $\OO$-linear homomorphism to $L$ kills off the $\OO$-torsion, we have $\Pi(\md)\cong \Hom^{cont}_{\OO}(\md\wtimes_{\wE}\wP, L)$.
\end{remar}

A  continuous homomorphism of compact $\wE$-modules $\md_1 \rightarrow \md_2$ induces a morphism
$\md_1\wtimes_{\wE}\wP\rightarrow  \md_2\wtimes_{\wE}\wP$  in $\dualcat(\OO)$   and hence $\md \mapsto \Pi(\md)$ defines a contravariant functor from the category of compact right $\wE$-modules, free of finite rank over $\OO$ 
to $\Ban^{\mathrm{adm}}_{\dualcat(\OO)}$. Since $\varpi$ is invertible in $\Ban^{\mathrm{adm}}_{\dualcat(\OO)}$
the functor factors through the category of finite dimensional $L$-vector spaces with continuous $\wE$-action (note that $\wE$
is compact). 

\begin{lem}\label{getback2} Let $\md_L$ be a finite dimensional $L$-vector space with continuous $\wE$-action. Assume that $(\wE/\rad \wE) \wtimes_{\wE} \wP$ is of finite length
in $\dualcat(\OO)$ and is a finitely generated $\OO[[H]]$-module. Then 
$$\md(\Pi(\md_L))\cong \md_L$$
where $\md$ is the functor defined in Lemma \ref{PitomPI}.
\end{lem}
\begin{proof} Since $\wE$ is compact and the action is continuous there exists an open bounded $\OO$-lattice 
$\md$ in $\md_L$ which is $\wE$-stable. Then 
$$\Pi(\md_L)= \Hom^{cont}_{\OO} ( (\md\wtimes_{\wE} \wP)_{\mathrm{tf}}, L)$$ 
and let $\Pi(\md_L)^0$ be the unit 
ball in $\Pi(\md_L)$ with respect to the supremum norm, so that $\Pi(\md_L)^0=\Hom_{\OO}^{cont}((\md\wtimes_{\wE} \wP)_{\mathrm{tf}}, \OO)$. Then $(\Pi(\md_L)^0)^d\cong (\md\wtimes_{\wE} \wP)_{\mathrm{tf}}$ 
and $\Hom_{\dualcat(\OO)}(\wP, (\Pi(\md_L)^0)^d)\cong \md$ by Lemma \ref{getback}. Since $\md$ is an open 
$\OO$-lattice in $\md_L$ we get $\md(\Pi(\md_L))\cong \md_L$.
\end{proof}

\begin{lem}\label{getback3} Let $\Pi \in \Ban^{\mathrm{adm}}_{\dualcat(\OO)}$ be irreducible, and let $\md_L:=\md(\Pi)$, where $\md$ is the functor defined in Lemma \ref{PitomPI}. 
If $\md_L$ is a non-zero finite dimensional $L$-vector space then $\Pi$ is isomorphic to a closed subspace of $\Pi(\md_L)$.
\end{lem}
\begin{proof} Let $\Theta$ be an open bounded $G$-invariant lattice in $\Pi$. The evaluation map $\Hom_{\dualcat(\OO)}(\wP, \Theta^d)\wtimes_{\wE} \wP\rightarrow \Theta^d$ induces 
a non-zero continuous, $G$-equivariant map $\Pi\rightarrow \Pi(\md_L)$. Since $\Pi$ is irreducible, and both representations are admissible, the map induces an isomorphism between 
$\Pi$ and a closed subspace of $\Pi(\md_L)$.
\end{proof} 

\begin{lem}\label{obviousexact} Assume that $(\wE/\rad \wE) \wtimes_{\wE} \wP$ is of finite length
in $\dualcat(\OO)$ and is a finitely generated $\OO[[H]]$-module. The functor $\md_L\mapsto \Pi(\md_L)$ is left exact.
\end{lem}
\begin{proof}  This follows from the right exactness of $\wtimes_{\wE}\wP$, left exactness of $\Hom^{cont}_{\OO}(\ast, L)$ and Remark \ref{killTorsion}.
\end{proof}

\begin{lem}\label{redfinlen} Assume that $(\wE/\rad \wE) \wtimes_{\wE} \wP$ is of finite length in $\dualcat(\OO)$ and is a finitely generated $\OO[[H]]$-module. 
Let $\md_L$ be a finite dimensional $L$-vector space with a continuous $\wE$-action. Then $\overline{\Pi(\md_L)}$ is an admissible smooth, finite length 
representation of $G$.
\end{lem}
\begin{proof} The assertion follows from Proposition \ref{moregen} together with  \eqref{Thetadnew}.
\end{proof}

\begin{prop}\label{closedred} Assume that $(\wE/\rad \wE) \wtimes_{\wE} \wP$ is of finite length
in $\dualcat(\OO)$ and is a finitely generated $\OO[[H]]$-module.
Let $\md_L$ be a finite dimensional $L$-vector space with continuous $\wE$-action and 
let $\Pi$ be a  closed non-zero $G$-invariant subspace of $\Pi(\md_L)$.
Suppose that $\md_L$ is an irreducible right $\wE_L$-module, then $\md( \Pi(\md_L)/\Pi)=0$
and $\md(\Pi)\cong \md(\Pi(\md_L))\cong \md_L$.  
In particular, each $\pi_i$ occurs in $\overline{\Pi}$ with the same (finite) multiplicity as in $\overline{\Pi(\md_L)}$. 
Further, if $S$ is irreducible and $\End_{\dualcat(\OO)}(S)=k$ then $\pi=S^{\vee}$ occurs in $\overline{\Pi}$ 
with multiplicity $\dim_L \md_L$.
\end{prop}
\begin{proof} Let $\Pi(\md_L)^0$ be the unit ball in $\Pi(\md_L)$ with respect to the supremum norm and let 
$\Theta:= \Pi\cap \Pi(\md_L)^0$. Then $\Theta$ is an open bounded $G$-invariant lattice in $\Pi$. By \cite[Prop 1.3.iii]{iw} we have a surjection 
$\sigma: (\md\wtimes_{\wE} \wP)_{\mathrm{tf}}\cong (\Pi(\md_L)^0)^d\twoheadrightarrow \Theta^d$.  As $\sigma$ is non-zero,
there exists $n\in \md$ and $v\in \wP$ such that $\sigma(n\wtimes v)$ is non-zero. Then $\psi: \wP\rightarrow \Theta^d$, 
$v\mapsto \sigma(n\wtimes v)$ is a non-zero element of $\Hom_{\dualcat(\OO)}(\wP, \Theta^d)$. Thus $\md(\Pi)\neq 0$.
Since the functor $\md$ is exact and contravariant and $\md_L$ is an irreducible $\wE$-module we deduce that 
$\md(\Pi(\md_L))\cong \md(\Pi)$ and $\md(\Pi(\md_L)/\Pi)=0$, which is a contradiction. The rest follows from Lemma \ref{mult=rank0} and \eqref{Thetadnew}.
\end{proof}

\begin{cor}\label{Piexistsuniq} Assume the setup of Proposition \ref{closedred} then $\Pi(\md_L)$ contains 
 a unique irreducible non-zero closed $G$-in\-va\-riant subspace $\Pi$. Moreover, for 
any $\phi: \Pi(\md_L) \rightarrow \Pi(\md_L)$ continuous and $G$-equivariant we have $\phi(\Pi)\subseteq \Pi$.
\end{cor} 
\begin{proof} This is immediate from Corollary \ref{smallest} and Proposition \ref{closedred}.
\end{proof}  

Let $\Ban^{\mathrm{adm. fl}}_{\dualcat(\OO)}$ be the full subcategory of $\Ban^{\mathrm{adm}}_{\dualcat(\OO)}$ consisting 
of objects of finite length. Let $\Ker \md$ be the full subcategory of $\Ban^{\mathrm{adm. fl}}_{\dualcat(\OO)}$
consisting of those $\Pi$ such that $\md(\Pi)=0$. Since $\md$ is an exact functor, $\Ker \md$ is a thick subcategory of 
$\Ban^{\mathrm{adm. fl}}_{\dualcat(\OO)}$ and hence we may build a quotient category $\Ban^{\mathrm{adm. fl}}_{\dualcat(\OO)}/ \Ker \md$,
see \cite[\S III.1]{gab}. 

\begin{thm}\label{antiequiv} Let $\wP$ and $\wE$ be as in the setup described before Proposition \ref{modulequotientnew}. Assume that 
\begin{itemize} 
\item[(i)] $(\wE/\rad \wE)\wtimes_{\wE} \wP$ is a finitely generated $\OO[[H]]$-module and is of finite length in $\dualcat(\OO)$;
\item[(ii)] For every irreducible $\Pi$  in $\Ban^{\mathrm{adm}}_{\dualcat(\OO)}$, $\md(\Pi)$ is finite dimensional.
\end{itemize} 
Then the functors $\md_L \mapsto \Pi(\md_L)$ and $\Pi\mapsto \md(\Pi)$ induce an anti-equivalence of categories between 
$\Ban^{\mathrm{adm. fl}}_{\dualcat(\OO)}/ \Ker \md$ and the category of finite dimensional $L$-vector spaces with continuous right $\wE$-action.
\end{thm} 
\begin{proof} Let $\TT: \Ban^{\mathrm{adm. fl}}_{\dualcat(\OO)}\rightarrow \Ban^{\mathrm{adm. fl}}_{\dualcat(\OO)}/ \Ker \md$ be the natural 
functor. Recall that a morphism $\phi: \Pi_1\rightarrow \Pi_2$ in $\Ban^{\mathrm{adm. fl}}_{\dualcat(\OO)}$ induces an 
isomorphism $\TT(\phi)$ in the quotient category if and only $\Ker \phi$ and $\Coker \phi$ lie 
in $\Ker \md$ (that is $\md(\Ker \phi)=0$ and $\md(\Coker \phi)=0$), see Lemme 4 in \cite[\S III.1]{gab}.

Since $\md$ is exact  assumption (ii) implies that $\md(\Pi)$ is finite dimensional for all 
$\Pi$ in $\Ban^{\mathrm{adm. fl}}_{\dualcat(\OO)}$. Let $\Pi$ be in $\Ban^{\mathrm{adm. fl}}_{\dualcat(\OO)}$ and 
$\Theta$ be an open bounded $G$-invariant lattice in $\Pi$ and let $\md:=\Hom_{\dualcat(\OO)}(\wP, \Theta^d)$. 
Evaluation induces a morphism $\md \wtimes_{\wE} \wP \rightarrow \Theta^d$ in $\dualcat(\OO)$ and dually 
we obtain a morphism $L$-Banach spaces $\Pi\rightarrow \Pi(\md(\Pi))$. We claim that the map 
$\TT(\Pi)\rightarrow \TT(\Pi(\md(\Pi))$ is an isomorphism. It is enough to prove the claim for irreducible $\Pi$, since then 
we get the rest by induction on the length of $\Pi$. The diagram:
\begin{displaymath}
\xymatrix@1{ 0\ar[r] & \TT(\Pi_1)\ar[d]^{\cong}\ar[r] & \TT(\Pi_2) \ar[d] \ar[r]& \TT(\Pi_3)\ar[d]^{\cong} \ar[r] & 0\\
            0\ar[r] & \TT(\Pi(\md_1))\ar[r] & \TT(\Pi(\md_2))  \ar[r]& \TT(\Pi(\md_3))}
\end{displaymath}
where $\md_i:=\md(\Pi_i)$ gives the induction step. We note that $\TT$ is exact by Proposition 1 in \cite[\S III.1]{gab} and 
hence the rows are exact.      

Suppose that $\Pi$ in $\Ban^{\mathrm{adm}}_{\dualcat(\OO)}$ is irreducible. 
If $\md(\Pi)=0$ then $\Pi\cong 0$ in the quotient category  $\Ban^{\mathrm{adm. fl}}_{\dualcat(\OO)}/ \Ker \md$ and hence 
$\TT(\Pi) \cong \TT(\Pi(\md(\Pi)))$. Suppose that $\md(\Pi)\neq 0$ then $\md(\Pi)$ is an irreducible right 
$\wE$-module by Proposition \ref{modulequotientnew} (ii). By dualizing Proposition \ref{modulequotientnew} (iii)
we obtain an injection $\iota: \Pi\hookrightarrow \Pi(\md(\Pi))$ and it follows from Proposition \ref{closedred} that 
$\md( \Pi(\md(\Pi))/\Pi)=0$. Hence, $\TT(\iota)$ is an isomorphism  between $\TT(\Pi)$ and $\TT(\Pi(\md(\Pi)))$. 
For the other composition we observe that $\md$ factors through the quotient category, see \cite[\S III.1 Cor.2]{gab}, 
so $\md(\TT (\Pi(\md_L)))\cong \md(\Pi(\md_L))\cong \md_L$, where the last assertion is given by  
Lemma \ref{getback2}. 
\end{proof}

\begin{remar} We note that since we assume that  $(\wE/\rad \wE)\wtimes_{\wE} \wP$ is of finite length in $\dualcat(\OO)$, 
Lemma \ref{redfinlen} implies that $\overline{\Pi(\md_L)}$ is of finite length. 
The statement of Theorem \ref{antiequiv} holds if  
instead of making the assumption (ii)  we  replace $\Ban^{\mathrm{adm. fl}}_{\dualcat(\OO)}$  by a smaller category. Namely a full subcategory 
of $\Ban^{\mathrm{adm}}_{\dualcat(\OO)}$ with objects $\Pi$ such that $\Theta\otimes_{\OO} k$ is of finite length where 
$\Theta$ is an open bounded $G$-invariant lattice in $\Pi$. Such $\Pi$ are of finite length and 
it follows from Lemma \ref{mult=rank0} that  $\md(\Pi)$ is 
finite dimensional. However, in the application to $\GL_2(\Qp)$-representations we will verify that the assumption (ii) 
is satisfied using Proposition \ref{longproof}.
\end{remar} 

\begin{thm}\label{furtherDBan} Let $\wP$ and $\wE$ be as in the setup described before Proposition \ref{modulequotientnew}. Assume that 
\begin{itemize} 
\item[(i)] $(\wE/\rad \wE)\wtimes_{\wE} \wP$ is a finitely generated $\OO[[H]]$-module and is of finite length in $\dualcat(\OO)$;
\item[(ii)] the centre $\wZ$ of $\wE$ is noetherian and $\wE$ is a finitely generated $\wZ$-module.
\end{itemize} 
Then 
$$\Ban^{\mathrm{adm. fl}}_{\dualcat(\OO)}/\Ker \md  \cong \bigoplus_{\nn\in \MaxSpec \wZ[1/p]} (\Ban^{\mathrm{adm. fl}}_{\dualcat(\OO)}/\Ker \md)_{\nn},$$
where the direct sum is taken over all the maximal ideals of $\wZ[1/p]$, and  for  a maximal ideal $\nn$ of $\wZ[1/p]$, $(\Ban^{\mathrm{adm. fl}}_{\dualcat(\OO)}/\Ker \md)_{\nn}$  
is the full subcategory of $\Ban^{\mathrm{adm. fl}}_{\dualcat(\OO)}/\Ker \md$, consisting of all Banach spaces which are killed by a power of $\nn$.

Further, the functor $\md \mapsto \Pi(\md)$ induces an anti-equivalence of categories between the category of 
modules of finite length of the $\nn$-adic completion of $\wE[1/p]$ and 
$(\Ban^{\mathrm{adm. fl}}_{\dualcat(\OO)}/\Ker \md)_{\nn}$.
\end{thm}
\begin{proof} We claim that  $\wZ[1/p]/\nn$ is a finite extension of $L$ for every 
maximal ideal $\nn$. For $1\le i \le n$ let $e_i\in \wE$ be orthogonal idempotents such that $e_i \wP= \wP_i$ and let $\wZ_i$ 
be the centre of $\End_{\dualcat(\OO)}(\wP_i)$. Since $\wZ\subset \prod_{i=1}^n e_i \wZ e_i \subset \prod_{i=1}^n \wZ_i$, $\wZ[1/p]/\nn$
will be a subfield of  $\wZ_i[1/p]/\nn_i$ for some $1\le i\le n$ and some maximal ideal $\nn_i$ of $\wZ_i[1/p]$. It follows from the 
proof of Proposition \ref{longproof} that $\wZ_i[1/p]/\nn_i$ is a finite extension of $L$. Since $\wE$ is a finitely generated $\wZ$-module 
the claim implies that every irreducible $\wE[1/p]$-module is finite dimensional over $L$, see the proof of Proposition \ref{longproof}. 
 
Proposition \ref{longproof} says that the assumption (ii) in Theorem \ref{antiequiv} is satisfied. Moreover, 
since $\wZ$ is noetherian and $\wE$ is a finitely generated $\wZ$-module we deduce that $\wE$ is left and right noetherian 
and hence any $\OO$-linear action of $\wE$ on a finite dimensional $L$-vector space is automatically continuous by Lemma \ref{toponmd2}.
Thus it follows from Theorem \ref{antiequiv} and the claim that the functor $\md \mapsto \Pi(\md)$ induces an anti-equivalence of categories between 
the category of  $\wE[1/p]$-modules of finite length and $\Ban^{\mathrm{adm. fl}}_{\dualcat(\OO)}/\Ker \md$. 

Let $\nn$ be a maximal ideal of $\wZ[1/p]$ and let $\md$ be an $\wE$-module of finite length. It follows from the anti-equivalence that $\Pi(\md)$ is an object of  $(\Ban^{\mathrm{adm. fl}}_{\dualcat(\OO)}/\Ker \md)_{\nn}$ if and only if $\md$ is annihilated by a power of $\nn$, and, since $\md$ is of finite length and $\nn$ is maximal, this is equivalent to $\md= \md_{\nn}$, the localization of $\md$ at $\nn$. As already observed, $\md$ is a finite dimensional 
$L$-vector space. Hence, the image of $\wZ[1/p]$ in $\End_L(\md)$ is a finite dimensional $L$-algebra, which implies via  the Chinese remainder theorem, that $\md\cong \oplus_{\nn} \md_{\nn}$, where the sum is taken over all the maximal 
ideals of $\wZ[1/p]$, and $\md_{\nn}=0$ for almost all $\nn$. Applying the functor $\Pi$ we deduce the last assertion.
\end{proof}  

\begin{prop} We assume the hypotheses of Theorem \ref{furtherDBan} and let $\nn$ be a maximal ideal of $\wZ[1/p]$ and 
$\nn_0:= \varphi^{-1}(\nn)$, where $\varphi: \wZ\rightarrow \wZ[1/p]$. The irreducible objects of 
$(\Ban^{\mathrm{adm. fl}}_{\dualcat(\OO)}/\Ker \md)_{\nn}$ are precisely the irreducible Banach subrepresentations  
of $\Hom_{\OO}^{cont}((\wP/\nn_0 \wP)_{\mathrm{tf}}, L)$.
\end{prop}
\begin{proof} Since $\wZ$ is noetherian $\nn_0$ is finitely generated and hence $\nn_0 \wP$ is closed in $\wP$. Thus 
$\wE/\nn_0 \wE \wtimes_{\wE} \wP \cong \wP/\nn_0 \wP$. Since $\wE$ is a finitely generated $\wZ$-module 
$\wE/\nn_0 \wE$ is a finitely generated $\wZ/\nn_0$-module and so $(\wE/\nn_0 \wE)_{\mathrm{tf}}$ is a finitely generated 
$(\wZ/\nn_0)_{\mathrm{tf}}$-module. Now, $(\wZ/\nn_0)_{\mathrm{tf}}$ is equal to the image of $\wZ$ in $\wZ[1/p]/\nn$ and 
hence is a finitely generated $\OO$-module. We deduce that $(\wE/\nn_0 \wE)_{\mathrm{tf}}$ is a free $\OO$-module of 
finite rank. It follows from Lemma \ref{getback} that $((\wE/\nn_0 \wE)_{\mathrm{tf}}\wtimes_{\wE}\wP)_{\mathrm{tf}}$ 
is an $\OO$-torsion free object of $\dualcat(\OO)$ and from Proposition \ref{moregen} that it is finitely generated 
over $\OO[[H]]$. It is immediate that any $\OO$-linear homomorphism from 
$\wE/\nn_0 \wE \wtimes_{\wE} \wP$ to a torsion free $\OO$-module must factor through 
$(\wE/\nn_0 \wE)_{\mathrm{tf}}\wtimes_{\wE}\wP$ and then through $((\wE/\nn_0 \wE)_{\mathrm{tf}}\wtimes_{\wE}\wP)_{\mathrm{tf}}$. 
We deduce that $(\wP/\nn_0 \wP)_{\mathrm{tf}}\cong((\wE/\nn_0 \wE)_{\mathrm{tf}}\wtimes_{\wE}\wP)_{\mathrm{tf}}$ 
is a finitely generated $\OO[[H]]$-module and is $\OO$-torsion free, and so 
the Banach space representation $\Pi:=\Hom_{\OO}^{cont}((\wP/\nn_0 \wP)_{\mathrm{tf}}, L)$ is admissible.

Let $\Pi_1$ be a closed non-zero subspace of $\Pi$, irreducible as a Banach space representation of $G$. Let $\Pi^0$ be the unit ball  
in $\Pi$ with respect to the supremum norm and let $\Pi_1^0:=\Pi_1\cap \Pi^0$. Dually we obtain a surjection 
$\psi: (\wP/\nn_0 \wP)_{\mathrm{tf}}\cong (\Pi^0)^d \twoheadrightarrow (\Pi_1^0)^d$. Composing $\psi$ with the natural 
map $\wP\twoheadrightarrow  (\wP/\nn_0 \wP)_{\mathrm{tf}}$ we deduce that $\md(\Pi_1)\neq 0$ and hence $\Pi_1$ is non-zero in the quotient 
category.  Since $\wP$ is projective we get a surjection of $\wE$-modules:
$$(\wE/\nn_0 \wE)_{\mathrm{tf}}\cong \Hom_{\dualcat(\OO)}(\wP,(\wP/\nn_0 \wP)_{\mathrm{tf}})\twoheadrightarrow 
\Hom_{\dualcat(\OO)}(\wP, (\Pi_1^0)^d),$$ 
where the first isomorphism follows from Lemma \ref{getback}. Hence, $\nn$ kills $\md(\Pi_1)$ and so $\Pi_1$ is an object 
of $(\Ban^{\mathrm{adm. fl}}_{\dualcat(\OO)}/\Ker \md)_{\nn}$.

Conversely, let $\Pi_1\in\Ban^{\mathrm{adm}}_{\dualcat(\OO)}$ be  irreducible  with  $\md(\Pi_1)\neq 0$.
Then $\Pi_1$ is non-zero in the quotient category. Suppose  $\Pi_1$ is an object of 
$(\Ban^{\mathrm{adm. fl}}_{\dualcat(\OO)}/\Ker \md)_{\nn}$. Since $\md(\Pi_1)$ is an irreducible $\wE[1/p]$-module by 
Proposition \ref{modulequotientnew}, $\nn$ kills $\md(\Pi_1)$. Let $\Theta$ be an open bounded $G$-invariant lattice in 
$\Pi_1$, every $\psi\in \nn_0$ induces a map $\psi^*: \Hom_{\dualcat(\OO)}(\wP, \Theta^d)\rightarrow \Hom_{\dualcat(\OO)}(\wP, \Theta^d)$, 
which is zero after inverting $p$. Since $\Theta^d$ is  $\OO$-torsion free, so is  
$\Hom_{\dualcat(\OO)}(\wP, \Theta^d)$ and hence $\psi^*$ is zero. We deduce that 
$$\Hom_{\dualcat(\OO)}((\wP/\nn_0\wP)_{\mathrm{tf}}, \Theta^d)\cong \Hom_{\dualcat(\OO)}(\wP/\nn_0\wP, \Theta^d)\cong 
\Hom_{\dualcat(\OO)}(\wP, \Theta^d).$$
Since $\md(\Pi_1)\neq 0$, $\Hom_{\dualcat(\OO)}((\wP/\nn_0\wP)_{\mathrm{tf}}, \Theta^d)\neq 0$ and dually 
$\Hom^{cont}_G(\Pi_1, \Pi)\neq 0$. As both spaces are admissible and $\Pi_1$ is irreducible any such 
non-zero homomorphism induces an isomorphism between $\Pi_1$ and a closed subspace of $\Pi$.
\end{proof} 

\begin{remar} If  we assume that  $\wZ$ is noetherian, $\wE$ is $\OO$-torsion free 
and is a free module of finite rank over $\wZ$ then $\wZ$ is 
$\OO$-torsion free, thus $\nn_0=\wZ\cap \nn$ and so $\wZ/\nn_0$ is a free $\OO$-module of finite 
rank, which implies $\wE/\nn_0 \wE$ is a free $\OO$-module of finite rank. If additionally we 
assume that $\wP$ is flat over $\wE$ then Corollary \ref{flatmodule} implies that 
$\wP/\nn_0\wP$ is $\OO$-torsion free. This situation will arise in the applications to $\GL_2(\Qp)$.
\end{remar}

\subsection{Relation to the deformation theory}\label{reldef}

In this subsection we assume a more restrictive setup which will be used in the applications. Let 
$\wP$ be a projective envelope of an irreducible object $S$ in 
$\dualcat(\OO)$ such that $\End_{\dualcat(\OO)}(S)=k$. Let $\wE:=\End_{\dualcat(\OO)}(\wP)$ and $\pi:=S^{\vee}$. Assume that there exists $Q$ 
in $\dualcat(k)$ of finite length in $\dualcat(k)$, a finitely generated $\OO[[H]]$-module satisfying 
hypotheses (H1)-(H4) made in \S\ref{firstsec}, (we do not assume (H5)). Then it follows from Lemma \ref{inclusion} that 
$(\wE/\rad \wE) \wtimes_{\wE} \wP\cong   \wP/(\rad \wE) \wP\cong Q$ and hence the hypothesis (i) in Theorem \ref{antiequiv} is satisfied. 

\begin{remar} If $G=\GL_2(\Qp)$ then it follows from the classification in \cite{bl} and \cite{breuil1} that every smooth  irreducible $k$-representation of $G$ 
with a central character is admissible and hence any smooth 
finite length $k$-representation of $G$ with a central character is admissible. So the assumption 
that $Q$ is finitely generated over $\OO[[H]]$ will be automatically satisfied.
\end{remar}

\begin{thm}\label{bijBM} There exists a natural bijection between isomorphism classes of 
\begin{itemize}
\item[(i)] irreducible topological right $\wE_L$-modules, finite dimensional over $L$, and 
\item[(ii)] irreducible admissible  unitary $L$-Banach space representations $\Pi$ of $G$ containing 
an open bounded $G$-invariant lattice $\Theta$ such that 
\begin{itemize}
\item[(a)]  $\Theta\otimes_{\OO} k $ is of finite length;
\item[(b)] $\Theta\otimes_{\OO} k$   contains $\pi$ as a subquotient ;
\item[(c)]$\Theta^d$ is an object of $\dualcat(\OO)$.
\end{itemize} 
\end{itemize} 
\end{thm}
\begin{proof} We recall if the conditions are satisfied for one open bounded $G$-invariant lattice
$\Theta$ then by Lemmas  \ref{commen} and \ref{obCO} they are satisfied for all such lattices 
inside $\Pi$. 

Suppose we are given $\Pi$, containing such $\Theta$, then 
$\md(\Pi):=\Hom_{\dualcat(\OO)}(\wP, \Theta^d)_L$
does not depend on the choice of $\Theta$ and it follows from 
Proposition \ref{modulequotientnew} that it is  an irreducible $\wE_L$-module and from Lemma \ref{mult=rank0} that it is finite dimensional. 

Given an irreducible 
$\wE_L$-module, finite dimensional over $L$, we may choose an $\wE$-invariant $\OO$-lattice $\md$ inside it as $\wE$ is compact. 
Let $\Pi(\md)$ be the admissible unitary $L$-Banach space representation of $G$ defined in \ref{defPimd}.
By Corollary \ref{Piexistsuniq} $\Pi(\md)$ contains a unique closed irreducible $G$-invariant subspace of $\Pi$. Lemma \ref{redfinlen}
implies that $\overline{\Pi}$ is a $G$-representation of finite length.

It is shown at the end of the proof of Theorem \ref{antiequiv} that we have a natural injection 
$\Pi\hookrightarrow \Pi(\md(\Pi))$. This fact together with Propositions \ref{closedred} implies that the two maps are mutually inverse.
\end{proof} 

\begin{cor}\label{fgZfl} Let $\Pi$ be an irreducible admissible unitary $L$-Banach space representation of $G$
containing an open bounded $G$-invariant lattice $\Theta$ such that $\Theta^d$ is an object 
of $\dualcat(\OO)$ and $\pi$ is a subquotient of $\Theta\otimes_{\OO} k$. If the centre 
$\mathcal Z$ of $\wE$ is noetherian and $\wE$ is a finitely generated $\mathcal Z$-module then 
$\Theta\otimes_{\OO} k$ is of finite length as a $G$-representation. 
\end{cor} 
\begin{proof} Proposition \ref{longproof} implies that $\Hom_{\dualcat(\OO)}(\wP, \Theta^d)_L$ is
finite dimensional over $L$ and the assertion follows from Theorem \ref{bijBM}.
\end{proof}

\begin{cor}\label{absirre} Let $\Pi$ be as in Theorem \ref{bijBM} and  $\md:=\Hom_{\dualcat(\OO)}(\wP, \Theta^d)$, 
where $\Theta$ is an open bounded $G$-invariant lattice in $\Pi$,
then the following are equivalent:
\begin{itemize}
\item[(i)] $\End^{cont}_{L[G]}(\Pi)= L$;
\item[(ii)] $\End_{\wE_L}(\md_L)=L$;
\item[(iii)] $\md_L$ is an absolutely irreducible right $\wE_L$-module;
\item[(iv)] $\Pi$ is an absolutely irreducible $L$-Banach space representation of $G$.
\end{itemize}
\end{cor} 

\begin{proof} It follows from Proposition \ref{ringsareiso} that $\End^{cont}_{L[G]}(\Pi)\cong \End_{\wE_L}(\md_L)^{op}$. Hence 
(i) is equivalent to (ii). The assumptions on $\Pi$ made in Theorem \ref{bijBM} imply that $\md_L$ is finite dimensional. Hence
(ii) is equivalent to (iii), see \cite[Cor. 12.4]{bourbakialg8}. Moreover, we deduce that $\End^{cont}_{L[G]}(\Pi)$ is finite dimensional over $L$ and so we deduce
from Lemma \ref{endoirrban} that (iv) implies (i). Finally Lemma \ref{eirrb} says that (i) implies (iv).
\end{proof}

\begin{cor}\label{theimageofcenter}  Let $\Pi$ and $\Theta$ be as in Theorem \ref{bijBM}. If $\Pi$ is absolutely irreducible then 
the image of the centre of $\wE$ in $\End_{\OO}(\Hom_{\dualcat(\OO)}(\wP, \Theta^d))$ is equal to $\OO$. 
\end{cor} 
\begin{proof} The image of $\mathcal Z$ contains $\OO$ and is contained in $\End_{\wE}( \Hom_{\dualcat(\OO)}(\wP, \Theta^d))$, 
which is isomorphic to $\OO$ by Corollary \ref{absirre}. 
\end{proof}

\begin{cor}\label{commutativeOK} Let $\Pi$ be an absolutely irreducible admissible $L$-Banach space 
representation of $G$  containing an open bounded $G$-invariant lattice $\Theta$ such that $\Theta^d$ 
is an object of $\dualcat(\OO)$ and $\pi$ is a subquotient of $\Theta\otimes_{\OO} k$.  
If $\wE$ is commutative then $\overline{\Pi}\subseteq (Q^{\vee})^{ss}$.
\end{cor}
\begin{proof} As a consequence of the hypotheses (H1)-(H4) we know that the maximal ideal of $\wE$ is generated by 
at most $1+\dim_k \Ext^1_{\dualcat(k)}(Q, S)$ elements, see Lemma \ref{inclusion} and Proposition \ref{filtdone} (iii), which implies that $\wE$ is noetherian. We note that the 
proof of Lemma \ref{inclusion} and Proposition \ref{filtdone} (iii) for $n=1$ does not use (H5).  Let $\Xi$ and $\mathfrak a$ be as in 
Proposition \ref{modulequotientnew} then $\Hom_{\dualcat(\OO)}(\wP, \Xi^d)_L$ is finite dimensional by 
Proposition \ref{longproof}. Since $\wE$ is 
commutative and $\Pi$ is absolutely irreducible it follows from Corollary \ref{theimageofcenter} 
that $\wE/\mathfrak a\cong \OO$. Tensoring the surjection 
$\OO\wtimes_{\wE}\wP\twoheadrightarrow \Xi^d$ with $k$ we obtain a surjection 
$Q\cong k \wtimes_{E}P \twoheadrightarrow \Xi^d\otimes_{\OO} k$. Thus 
$\overline{\Pi}\cong (\Xi\otimes_{\OO} k)^{ss}\cong ((\Xi^d \otimes_{\OO} k)^{\vee})^{ss} \subseteq (Q^{\vee})^{ss}.$
\end{proof}

\subsection{Extensions of Banach space representations}\label{ext_of_ban}

Let $\Mod^?_{G}(\OO)$ be  a full subcategory of $\Mod^{\mathrm{l\, fin}}_G(\OO)$ closed under subquotients and 
arbitrary direct sums in $\Mod^{\mathrm{l\, fin}}_G(\OO)$.  Let $\dualcat(\OO)$ be a full subcategory of $\Mod_G^{\mathrm{pro\, aug}}(\OO)$ 
anti-equivalent to $\Mod^?_{G}(\OO)$ via Pontryagin duality. Assume that $\Mod^?_{G}(\OO)$ has only finitely many irreducible objects $\pi_1, \ldots, \pi_n$, which are admissible. 
Let $\wP$ be a projective envelope of $\pi_1^{\vee}\oplus\ldots\oplus \pi_n^{\vee}$ in $\dualcat(\OO)$, and let $\wE=\End_{\dualcat(\OO)}(\wP)$.
It follows from \cite[\S IV.4, Cor. 1]{gab} that the functor $M\mapsto \Hom_{\dualcat(\OO)}(\wP, M)$ induces an equivalence of categories between 
$\dualcat(\OO)$ and the category of compact right $\wE$-modules, with the inverse functor given by $\md\mapsto \md\wtimes_{\wE} \wP$. This implies that 
$\wE/\rad \wE \wtimes_{\wE} \wP \cong \pi_1^{\vee}\oplus\ldots\oplus \pi_n^{\vee}$, which is a finitely generated $\OO[[H]]$-module, as $\pi_i$ are assumed to be admissible.
We further assume that  the centre $\wZ$ of $\wE$ is noetherian, and $\wE$ is a finitely generated module over $\wZ$. Let $\Mod^{\mathrm{fg}}_{\wE[1/p]}$ be the category 
of finitely generated right $\wE[1/p]$-modules. 

\begin{lem}\label{fully_faithful} The functor $\md: \Ban^{\mathrm{adm}}_{\dualcat(\OO)}\rightarrow \Mod^{\mathrm{fg}}_{\wE[1/p]}$ is fully faithful.
\end{lem} 
\begin{proof} Lemma \ref{PitomPI} and Proposition \ref{admfg} show that $\md$ is well defined. It remains to show that it is fully faithfull.
Let $\Pi_1$, $\Pi_2$ be in $\Ban^{\mathrm{adm}}_{\dualcat(\OO)}$ and let $\Theta_1$ and $\Theta_2$ be open bounded $G$-invariant lattices
in $\Pi_1$ and $\Pi_2$, respectively. Then $\Theta_1^d$ and $\Theta_2^d$ are objects of $\dualcat(\OO)$ by Lemma \ref{obCO}. For $i=1$ and $i=2$ let 
$\md_i:=\Hom_{\dualcat(\OO)}(\wP, \Theta^d_i)$, then, because of equivalence of categories explained above, we have $\Hom_{\dualcat(\OO)}(\Theta^d_2, \Theta^d_1)\cong 
\Hom_{\wE}(\md_2, \md_1)$.  Since $\Hom_G(\Pi_1, \Pi_2)\cong \Hom_{\dualcat(\OO)}(\Theta^d_2, \Theta^d_1)\otimes_{\OO} L$ by \cite{iw}, 
and $\md(\Pi_i)=\md_i\otimes_{\OO} L$, we deduce the result.
\end{proof}

\begin{prop}\label{Yoneda} Let $R$ be a ring, $\mathcal A$ the category of finitely generated (right) modules of $R$, and let $\mathcal B$ be a full subcategory of $\mathcal A$ containing all 
the modules of finite length and cÁlosed under extensions and subquotients in $\mathcal A$.  Let $Z$ be the centre of $R$.  If $Z$ is noetherian and 
$R$ is a finitely generated $Z$-module, then for every $A, B\in \mathcal B$ with $B$ a module of finite length,  the natural map between the Yoneda-$Ext$ groups : 
$$ \varphi^n: \Ext^n_{\mathcal B} (A, B)\rightarrow \Ext^n_{\mathcal A}(A, B)$$
is an isomorphism, for all $n\ge 0$.
\end{prop}
\begin{proof} Since $Z$ is noetherian, and $R$ is a finitely generated $Z$-module, $R$ is left and right noetherian. Hence $\mathcal A$ is an abelian category.
Since $\mathcal B$ is a full subcategory, closed under subquotients in $\mathcal A$, $\mathcal B$ is also an abelian category. 

If $\varphi^n$ is bijective for a given $n$ and all $A, B\in \mathcal B$ then $\varphi^{n+1}$ is injective for all $A, B\in \mathcal B$, see 
\cite[Prop. 3.3]{oort}. Moreover, $\varphi^0$ and $\varphi^1$ are bijective by assumption. So it is enough to show that $\varphi^n$ is surjective for $n\ge 2$. 
Let $0\rightarrow B\rightarrow X_{1}\rightarrow \ldots \rightarrow X_n \rightarrow A\rightarrow 0$ be an extension representing $\xi\in \Ext^n_{\mathcal A}(A, B)$. 
Let $I$ be the $Z$-annihilator of $B$, then by Artin-Rees lemma, there exists a positive integer $c$, such that $B\cap I^c X_1=0$. Since $B$ is of finite length, 
$Z/I^c$ is a $Z$-module of finite length, and hence $X_1/I^c X_1$ is an $R$-module of finite length. We thus may represent 
$\xi$ with the extension $0\rightarrow B\rightarrow X_1/I^c X_1\rightarrow X_2/I^c X_1\rightarrow\ldots \rightarrow A\rightarrow 0$. Arguing inductively, 
we deduce that $\xi$ can be represented by an extension in $\mathcal B$, and so $\varphi^n$ is surjective for $n\ge 2$.
\end{proof}
\begin{remar} The upshot of Proposition \ref{Yoneda} is that $\mathcal A$ has enough projectives and the Yoneda $\Ext$-groups can be calculated using 
projective resolutions.
\end{remar}

\begin{cor}\label{extensionsbanach}Let $\Pi_1$ and $\Pi_2\in \Ban^{\mathrm{adm}}_{\dualcat(\OO)}$, with $\Pi_1$ of finite length. The functor $\md$ induces an isomorphism
  $$\Ext^i_G(\Pi_1, \Pi_2)\cong \Ext^i_{\wE[1/p]}(\md(\Pi_2), \md(\Pi_1))$$
 between the Yoneda $\Ext$-groups computed in $\Ban^{\mathrm{adm}}_{\dualcat(\OO)}$ and in $\Mod^{\mathrm{fg}}_{\wE[1/p]}$, respectively.
 \end{cor}
 \begin{proof} We apply Proposition \ref{Yoneda} with $R=\wE[1/p]$, $\mathcal A=\Mod^{\mathrm{fg}}_{\wE[1/p]}$, and $\mathcal B$ the full subcategory 
 with objects all the finitely generated $\wE[1/p]$-modules, which are isomorphic to $\md(\Pi)$, with $\Pi\in \Ban^{\mathrm{adm}}_{\dualcat(\OO)}$. Theorem 
 \ref{furtherDBan} implies that $\mathcal B$ contains all the modules of finite length. Let  $\md$ be a finitely generated $\wE[1/p]$-module, and let $\md^0$ be a finitely generated $\wE$-submodule, which is 
 an $\OO$-lattice in $\md$. If $\md^0\wtimes_{\wE} \wP$ is finitely generated over $\OO[[H]]$, then $\Pi(\md):=\Hom^{cont}_{\OO}(\md^0\wtimes_{\wE} \wP, L)$ is an admissible Banach space representation of $G$, 
 and $\md(\Pi(\md))\cong \md$. Since $\OO[[H]]$ is noetherian, this implies that $\mathcal B$ is closed under extensions and subquotients in $\mathcal A$.
 Lemma \ref{fully_faithful} implies that $\md$ induces an equivalence of categories 
 between  $\Ban^{\mathrm{adm}}_{\dualcat(\OO)}$ and $\mathcal B$.
\end{proof}
 
 \begin{remar} The assumptions made in this subsection are satisfied if $G=\GL_2(\Qp)$, $p\ge 5$  and 
$\Mod^?_{G}(\OO)$ is a block in the category of smooth locally finite representations of $G$ with a fixed central character, see \S \ref{blocks}. 
Further, for each block we will compute the ring $\wE$ and show that it satisfies the assumptions made in this subsection. Since  the decomposition into blocks is functorial, there are no extensions between Banach space representations 
lying in different blocks, so  Corollary \ref{extensionsbanach} will enable us to compute the $\Ext$-groups in the category of admissible unitary Banach space representations of $\GL_2(\Qp)$ with a fixed central character.
\end{remar}

\section{\texorpdfstring{Representations of $\mathrm{GL}_2(\mathbb{Q}_p)$}{Representations of $\mathrm{GL}_2(\mathbb{Q}_p)$}}\label{repsGL2}
\subsection{Notation}

Let $G:=\GL_2(\Qp)$, let $P$ be the subgroup of upper-triangular matrices, $T$ the subgroup of diagonal matrices, $U$ be the 
unipotent upper triangular matrices and $K:=\GL_2(\Zp)$. Let $\pF:=p\Zp$ and  
$$ I:=\begin{pmatrix} \Zp^{\times} & \Zp \\ \pF & \Zp^{\times} \end{pmatrix},\quad I_1:= \begin{pmatrix} 1+\pF & \Zp \\ \pF & 1+\pF \end{pmatrix}, 
\quad K_1:=\begin{pmatrix} 1+\pF & \pF \\ \pF & 1+\pF \end{pmatrix}.$$
For $\lambda\in \Fp$ we denote the Teichm\"uller lift of $\lambda$ to $\Zp$ by $[\lambda]$.  Set
$$H:=\biggl \{\begin{pmatrix} [\lambda] & 0\\ 0 & [\mu]\end{pmatrix}: \lambda, \mu\in \Fp^{\times}\biggr \}.$$  
Let $\varepsilon:\Qp\rightarrow L$, $x\mapsto x|x|$,  $\omega: \Qp\rightarrow k$, $x\mapsto x |x| \pmod{\pL}$, where $|\centerdot|$ is a norm on 
$\Qp$ with $|p|=\frac{1}{p}$, and  
$\alpha:T\rightarrow k^{\times}$ be the character
$$\alpha(\begin{pmatrix} \lambda & 0\\ 0 & \mu\end{pmatrix}):=\omega(\lambda\mu^{-1}).$$
Further, define 
$$\Pi:=\begin{pmatrix} 0 & 1\\ p & 0\end{pmatrix}, \quad s:=\begin{pmatrix} 0 & 1 \\ 1 & 0\end{pmatrix}, \quad 
t:= \begin{pmatrix} p & 0 \\ 0 & 1 \end{pmatrix}. $$
For $\lambda\in k^{\times}$ we define an unramified character
$\mu_{\lambda}:\Qp^{\times} \rightarrow k^{\times}$, by  $x\mapsto \lambda^{\val(x)}$. Given two characters 
$\chi_1, \chi_2:\Qp^{\times}\rightarrow k^{\times}$ we consider $\chi_1\otimes\chi_2$ as a character of $P$, 
which sends $\left(\begin{smallmatrix} a & b \\ 0 & d \end{smallmatrix}\right)$ to $\chi_1(a)\chi_2(d)$.

Let $Z$ be the centre of $G$, and set $Z_1:=Z\cap I_1$. Let $G^0:=\{g\in G: \detr g\in \Zp^{\times}\}$ and set $G^+:=ZG^0$.

Let  $\mathcal G$ be a topological group. We denote by  $\Hom(\mathcal G, k)$ the continuous group homomorphism from $\mathcal G$ to 
$(k, +)$. If $\VVV$ is a representation of $\mathcal G$ and $S$ is a subset of $\VVV$ we denote by 
$\langle \mathcal G \centerdot S\rangle$ the smallest subspace of $\VVV$ containing $S$ and stable under the action of $\mathcal G$.
The socle $\soc_{\mathcal G} \VVV$ is the maximal semi-simple $\mathcal G$-subrepresentation of $\VVV$. 
The socle filtration $\soc^i_{\mathcal G} \VVV\subseteq \VVV$ is defined by an exact sequence 
$0\rightarrow \soc_{\mathcal G}^i \VVV\rightarrow \soc_{\mathcal G}^{i+1}\VVV \rightarrow \soc_{\mathcal G}(\VVV/\soc^i_{\mathcal G} \VVV)\rightarrow 0$,
for $i\ge 0$ and $\soc^0_{\mathcal G} \VVV:=0$.

We make the same conventions as in \cite{colmez} regarding local class field theory: if $\Lambda$ is a topological ring let 
$\widehat{\mathcal T}(\Lambda)$ be the set of continuous characters $\delta:\Qp^{\times}\rightarrow \Lambda^{\times}$. 
Local class field theory gives us an isomorphism of topological groups between the abelianisation $W_{\Qp}^{ab}$ of 
the Weil group $W_{\Qp}$ of $\Qp$ and $\Qp^{\times}$. This enables us to consider an element $\delta\in \widehat{\mathcal T}(\Lambda)$ 
as a continuous character of $W_{\Qp}$ by the formula: 
\begin{equation}\label{local_class}
\delta(g)= \delta(p)^{-deg(g)} \delta(\varepsilon(g)), \quad \forall g\in W_{\Qp},
\end{equation}
where $\deg(g)$ is an integer defined by $g(x)= x^{p^{deg(g)}}$, for all $x\in \Fpbar$, and $\varepsilon$ is the cyclotomic character. Since $\gal$ is isomorphic to 
the profinite completion of $W_{\Qp}$, the character $\delta$ defined by \eqref{local_class} extends to a continuous character $\delta: \gal\rightarrow \Lambda^{\times}$ 
if and only if  $n\mapsto \delta(p^n)$ extends continuously to $\widehat{\ZZ}$. This is the case if $\Lambda=k$ or $\Lambda=L$ and $\delta$ is unitary.
The formula \eqref{local_class} identifies the cyclotomic character with the character $\Qp^{\times}\rightarrow \Zp^{\times}$, $x\mapsto x|x|$, which is also denoted by $\varepsilon$ above.

\subsection{Rationality}

\begin{lem}\label{abstractrat} Let $\rG$ be a group, $\rK$  a field and $\rL$ a field  extension of
$\rK$. Let $V$ and $W$ be $\rK[\rG]$-modules. If either $V$ is finitely generated over 
$\rK[\rG]$ or $\rL$ is finite over $\rK$ then the natural injection 
\begin{equation}\label{basechange}
\Hom_{\rK[\rG]}(V, W)\otimes_{\rK} \rL \hookrightarrow \Hom_{\rL[\rG]}(V\otimes_{\rK} \rL, W\otimes_{\rK} \rL).
\end{equation}
is an isomorphism. In particular, $W^{\rG}\otimes_{\rK} \rL\cong (W\otimes_{\rK} \rL)^{\rG}$.
\end{lem}
\begin{proof} If $V$ is finitely generated over $\rK[\rG]$ we have an exact sequence
of $\rK[\rG]$-modules 
$0\rightarrow U\rightarrow \rK[\rG]^{\oplus n}\rightarrow V\rightarrow 0$. We obtain a commutative diagram
\begin{displaymath}
\xymatrix@1{ 0 \ar[r] &\Hom_{\rK[\rG]}(V, W)_{\rL}\ar[r]\ar[d] & (W^{\oplus n})_{\rL}\ar[r]\ar[d]^{\cong} &
\Hom_{\rK[\rG]}(U, W)_{\rL}\ar@{^(->}[d]\\  
0 \ar[r] & \Hom_{\rL[\rG]}(V_{\rL}, W_{\rL})\ar[r] & (W_\rL)^{\oplus n}\ar[r]&
\Hom_{\rL[\rG]}(U_{\rL}, W_{\rL})}
\end{displaymath}
Since the third vertical arrow is injective  and the second is an isomorphism we deduce that the 
first is also an isomorphism. Since $W^{\rG}\cong \Hom_{\rK[\rG]}(\Eins, W)$ we deduce 
$W^{\rG}\otimes_{\rK} \rL\cong (W\otimes_{\rK} \rL)^{\rG}$. 

Let $V$ be arbitrary. The group $\rG$ acts naturally on $\Hom_{\rK}(V, W)$ by conjugation. 
If  $\rL$ is finite over $\rK$ then we have $\Hom_{\rK}(V, W)_\rL\cong \Hom_{\rL}(V_\rL, W_{\rL})$,
see for example Proposition 16(i) in \S II.7.7 of \cite{bouralg}. Since by the previous part taking $\rG$-invariants commutes with the tensor product with $\rL$ we deduce that 
\eqref{basechange} is an isomorphism.  
\end{proof} 

\begin{remar} In the foundational papers \cite{bl}, \cite{breuil1}, \cite{breuil2}, \cite{vig} the authors study representation theory over an algebraically closed field.
Using the Lemma one may show that their results also hold over an extension of $\Fp$, provided the extension is ``large enough", see Lemma \ref{fielddefi}. Lemma \ref{abstractrat} will allow us to deduce various results
on $\Ext$-groups between irreducible representations over $k$ from the corresponding results over algebraically closed fields, which have already appeared in the literature, see Remark \ref{Hecke_algebra_extend_scalars}, Lemma 
\ref{bcH}, Proposition \ref{bcEXT}.
 \end{remar}
 
\subsection{Irreducible representations }
We recall the classification of  the (absolutely) irreducible smooth $k$-rep\-re\-sen\-ta\-tions of $G$ with a central character\footnote{Laurent Berger has shown recently in \cite{cch}, that every smooth irreducible representation of $G$ over an 
algebraically closed field of characteristic $p$ admits a central character.} due to Barthel-Livn\' e \cite{bl} and Breuil \cite{breuil1}. We then show that the category 
$\Mod^{\mathrm{l\, fin}}_{G,\zeta}(k)$ behaves well when we replace $k$ by an extension. We let $k$ be an arbitrary field of characteristic $p$ until Proposition \ref{irrkreps}, from then onwards $k$ is a finite field,
which is the situation we are most interested in. This  assumption is made for the sake of simplicity, one has to work harder if $k$ is not  a perfect field,  see Remark \ref{perfecto}. 
We assume from \S \ref{hecke} onwards that $k$ contains a square root of $\zeta(p)$, where $\zeta$ is the fixed central character.

Let $\sigma$ be an irreducible smooth representation of $K$. Since $K_1$ is a normal pro-$p$ subgroup of 
$K$, $\sigma^{K_1}$ is non-zero and since $\sigma$ is irreducible we deduce that $K_1$ acts trivially.
Hence $\sigma$ is an irreducible representation of $K/K_1\cong \GL_2(\Fp)$ and so $\sigma\cong \Sym^r k^2\otimes \det^a$ 
for uniquely determined  integers $0\le r\le p-1$ and $0\le a \le p-2$. We also note that 
this  implies that $\sigma$ is absolutely irreducible and can be defined over $\Fp$.

Let $\zeta:Z\rightarrow k^{\times}$ be a smooth character extending the central character of $\sigma$.
We extend the action of $K$ on $\sigma$ to the action of   $KZ$ by making $p$ act by a scalar $\zeta(p)$. 
 It is shown in \cite[Prop. 8]{bl} that there exists an isomorphism of algebras:
\begin{equation}\label{endoalgebra}
\End_{G}(\cIndu{KZ}{G}{\sigma})\cong k[T]
\end{equation}
for a certain Hecke operator $T\in  \End_{G}(\cIndu{KZ}{G}{\sigma})$ defined in \cite[\S 3]{bl}.  

\begin{prop}\label{reducetoabs} Let $\pi$ be a smooth irreducible $k$-representation of $G$ with a central character
 $\zeta$. There exists a finite extension $l$ of $k$ such that $\pi\otimes_k l$ is of finite length and all the irreducible subquotients are absolutely irreducible in the sense of Remark \ref{absirreequiv} (iii).
\end{prop}

\begin{proof} Following the proof of Proposition 32 in  \cite{bl} we deduce that $\pi$ is a quotient of 
$\cIndu{KZ}{G}{\sigma}/ P(T) \cIndu{KZ}{G}{\sigma}$, where $P\in k[T]$ is a non-zero 
polynomial,  irreducible over $k$ and $T$ is as in \eqref{endoalgebra}. We know 
that the assertion holds if $P(T)=T-\lambda$, for some $\lambda\in k$, by \cite{bl} if $\lambda\neq 0$  and \cite{breuil1}
if $\lambda=0$. We may take $l$ to be the splitting field of $P$. 
\end{proof}

\begin{remar}\label{perfecto} If $k$ is perfect then the same proof shows that for every smooth irreducible $k$-representation of $G$ with a central character $\zeta$ there  
exist a finite extension $l$ of $k$ such that $\pi\otimes_k l$ is isomorphic to a finite direct sum of of absolutely irreducible representations.
\end{remar}

\begin{cor} If $\pi$ is an object of  $\Mod^{\mathrm{l\, fin}}_{G,\zeta}(k)$ then
$\pi\otimes_k l$ is an object of $\Mod^{\mathrm{l\, fin}}_{G,\zeta}(l)$.
\end{cor}
\begin{proof} Given $v\in \pi\otimes_k l$ we may express $v=\sum_{i=1}^n \lambda_i v_i$ 
with $v_i \in \pi$ and $\lambda\in l$. Hence, we may assume that $\pi$ is of finite length.
Proposition \ref{reducetoabs} implies that $\pi\otimes_k l$ is of finite length. 
\end{proof} 

\begin{cor}\label{finlengthadm} Every smooth finite length $k$-representation of $G$ with a central character 
is admissible.
\end{cor}
\begin{proof} It follows from the classification, see \cite{bl} and \cite{breuil1}, that 
every absolutely irreducible representation is admissible. The assertion follows from 
Proposition \ref{reducetoabs} and Lemma \ref{abstractrat}. 
\end{proof}

\begin{lem} Let $\pi$ and $\tau$ be objects of $\Mod^{\mathrm{l \, fin}}_{G, \zeta}(k)$
(resp. $\Mod^{\mathrm{sm}}_{G, \zeta}(k)$) and let $l$ be a field extension of $k$. If 
$\pi$ is finitely generated over $G$ then the natural map
\begin{equation}\label{bcinjext1}
\Ext^1_{k[G],\zeta}(\pi, \tau)\otimes_k l \rightarrow \Ext^1_{l[G],\zeta}(\pi\otimes_k l, \tau\otimes_k l), 
\end{equation} 
is injective, where $\Ext^1$ are computed in the corresponding categories.
\end{lem}
\begin{proof} In terms of Yoneda $\Ext$ the map is given by sending 
an extension $0\rightarrow \tau\rightarrow \kappa\rightarrow \pi\rightarrow 0$
to $0\rightarrow \tau_l\rightarrow \kappa_l\rightarrow \pi_l\rightarrow 0$. 
Since $\pi$ is assumed to be finitely generated over $k[G]$, it follows from Lemma \ref{abstractrat} 
that $\Hom_G(\pi_l, \kappa_l)=\Hom_G(\pi, \kappa)_l$, hence  any splitting  of $0\rightarrow \tau_l\rightarrow \kappa_l\rightarrow \pi_l\rightarrow 0$
is already defined over $k$. Thus the map is an injective.
\end{proof}

Let $\bar{k}$ be the algebraic closure of $k$. It follows from \cite[Thm. 33]{bl} and \cite[Thm 1.1]{breuil1} that the  irreducible smooth 
$\bar{k}$-representations of $G$ with a central character fall 
into four disjoint classes:
\begin{itemize}
\item[(i)] characters, $\eta\circ \det$;
\item[(ii)] special series, $\Sp\otimes\eta\circ \det$;
\item[(iii)] principal series $\Indu{P}{G}{\chi_1\otimes\chi_2}$, with $\chi_1\neq \chi_2$;
\item[(iv)] supersingular $\cIndu{KZ}{G}{\sigma}/(T)$.
\end{itemize} 
The Steinberg representation $\Sp$ is defined by the exact sequence: 
\begin{equation}\label{defineSp}
0\rightarrow \Eins\rightarrow \Indu{P}{G}{\Eins}\rightarrow \Sp\rightarrow 0.
\end{equation}

\begin{defi}Let $\pi$ be a $\bar{k}$-representation of a group $\rG$ and $l$ a subfield of $\bar{k}$.
 We say that $\pi$ can be defined over $l$ if there exists 
an $l$-representation  $\tau$ of $\rG$ such that $\tau\otimes_l \bar{k} \cong \pi$. We say that $l$ is a field of definition 
of $\pi$ if it is the smallest subfield of $\bar{k}$ over which $\pi$ can be defined.
\end{defi}

\begin{lem}\label{DefineCHI} Let $\chi:\Qp^{\times}\rightarrow \bar{k}^{\times}$ be a smooth character then the field of definition of $\chi$ is $\Fp[\chi(p)]$.
\end{lem}
\begin{proof} Since $\chi$ is smooth it is trivial on $1+p \Zp$ and hence $\chi(\Zp^{\times})\subseteq \Fp^{\times}$, 
the group of $(p-1)$-st roots of unity in $\bar{k}$. 
Since $\Qp^{\times}\cong \Zp^{\times} \times p^{\ZZ}$ the assertion follows.
\end{proof}

\begin{lem}\label{fielddefi} Let $\pi$ be a smooth irreducible $\bar{k}$-representation of $G$ with a central character $\zeta$. Then there exists 
a smallest subfield $l$ of $\bar{k}$ over which $\pi$ can be defined. Moreover, 
\begin{itemize}
\item[(i)] if $\pi\cong \eta\circ \det$ then $l=\Fp[\eta(p)]$; 
\item[(ii)] if $\pi\cong \Sp\otimes\eta\circ \det$ then $l=\Fp[\eta(p)]$;
\item[(iii)] if $\pi\cong \Indu{P}{G}{\chi_1\otimes\chi_2}$ then $l= \Fp[\chi_1(p), \chi_2(p)]$;
\item[(iv)] if $\pi$ is supersingular then $l=\Fp[\zeta(p)]$.
\end{itemize} 
\end{lem}
\begin{proof} Let $l$ be a subfield of $\bar{k}$ and $\tau$ an $l$-representation of $G$ such that $\tau\otimes_l \bar{k}\cong \pi$. 
Since $\pi$ is irreducible $\tau$ is irreducible and hence it follows from Lemma \ref{abstractrat} that $\tau$ is uniquely determined
up to an isomorphism over $l$. As already mentioned $\pi^{I_1}$ is finite dimensional (and non-zero), this implies $\End_G(\pi)\cong \bar{k}$. 
We deduce from Lemma \ref{abstractrat} that $\End_G(\tau)\cong l$. Thus $Z$ acts on $\tau$ by a central character. We deduce 
that $\zeta(p)\in l$. 

If $\pi$ is supersingular then we are done since $\sigma|_K$ can be defined over $\Fp$, 
$\sigma$ can be defined over $\Fp[\zeta(p)]$ by using $KZ\cong K \times p^{\ZZ}$ and the endomorphism $T$ can also be defined over 
$\Fp[\zeta(p)]$, as is immediate from  \cite[\S 3]{bl}.

If $\pi$ is a character or special series then $\pi^{I_1}$ is $1$-di\-men\-sio\-nal. Lemma \ref{abstractrat} implies that $\tau^{I_1}$ 
is $1$-di\-men\-sio\-nal. Since $\bigl (\begin{smallmatrix} 0 & -1\\ p & 0\end{smallmatrix}\bigr)$ acts on $\pi^{I_1}$ by a scalar $\pm\eta(p)$, 
we deduce that $\eta(p)\in l$ and hence $\Fp[\eta(p)]$ is a field of definition of $\pi$. We note that it is immediate from 
\eqref{defineSp} that $\Sp$ can be defined over $\Fp$.

If $\pi$ is principal series then $\pi^{I_1}$ is $2$-di\-men\-sio\-nal with basis $\{\varphi_1, \varphi_2\}$, where $\supp \varphi_1=Ps I_1$, 
$\varphi_1(s)=1$, $\supp \varphi_2= PI_1$, $\varphi_2(1)=1$. The $K$-representation $\sigma:=\langle K\centerdot \varphi_1\rangle \subset \pi$
is  irreducible , $\sigma^{I_1}= \bar{k} \varphi_1$ and $\Hom_K(\sigma, \pi)$ is $1$-di\-men\-sio\-nal.
Now $\sigma$ can be realized over $\Fp$, so in particular over $k$. It follows from Lemma \ref{abstractrat} that 
$\Hom_K(\sigma, \tau)$ is $1$-di\-men\-sio\-nal. Choose a non-zero $\phi\in \Hom_K(\sigma, \tau)$ and let $v\in \phi(\sigma)^{I_1}$ be non-zero. 
The $1$-di\-men\-sio\-nality of the spaces involved implies that any $G$-equivariant isomorphism $\pi\cong \tau\otimes_l \bar{k}$ must map $\varphi_1$ to 
$v\otimes \lambda$ for some $\lambda\in \bar{k}$. A direct calculation shows that 
$\sum_{\lambda\in \Fp} \bigl (\begin{smallmatrix} p & [\lambda] \\ 0 & 1 \end{smallmatrix} \bigr ) \varphi_1= \chi_2(p) \varphi_1$, hence
$\sum_{\lambda\in \Fp} \bigl (\begin{smallmatrix} p & [\lambda] \\ 0 & 1 \end{smallmatrix} \bigr ) v= \chi_2(p) v$ and so $\chi_2(p)\in l$.
Since $\zeta(p)=\chi_1(p) \chi_2(p)$ we deduce that $\chi_1(p)\in l$. Hence, $\Fp[\chi_1(p), \chi_2(p)]$ is the field of definition 
of $\pi$. Further, since both $\chi_1$ and $\chi_2$ maybe defined over  $\Fp[\chi_1(p), \chi_2(p)]$ by Lemma \ref{DefineCHI} we 
deduce that $\tau$ is a principal series representation.
\end{proof}

Let $\chi: T\rightarrow \bar{k}^{\times}$ be a smooth character and let $X$ be the orbit of $\chi$ under the action of $\Gamma:=\Gal(\bar{k}/k)$. 
It follows from Lemma \ref{DefineCHI}  that $\chi$ can be defined over a finite extension of $k$ and so $X$ is finite. Let  
$$V_{\chi}:= (\bigoplus_{\psi\in X} \psi)^{\Gamma},$$
where the action of $\Gamma$ on $\bigoplus_{\psi\in X} \psi$ is given by $\gamma \centerdot (\lambda_{\psi})_{\psi}:= (\gamma(\lambda_{\psi}))_{\gamma(\psi)}$. 
Then $V_{\chi}$ is the unique irreducible $k$-representation of $T$ such that $V_{\chi}\otimes_k \bar{k}$ contains $\chi$. We note that if $\chi$ factors through the determinant, then we may consider both $\chi$ and
$V_{\chi}$ as representations of $G$. 

\begin{prop}\label{irrkreps} Let $\pi$ be an irreducible smooth $k$-representation of $G$ 
with a central character. Then $\pi$ is isomorphic to one of the following: 
\begin{itemize}
\item[(i)] $V_{\eta\circ \det}$, $\eta: \Qp^{\times}\rightarrow \bar{k}^{\times}$; 
\item[(ii)]  $\Sp\otimes V_{\eta\circ \det}$, $\eta: \Qp^{\times}\rightarrow \bar{k}^{\times}$;
\item[(iii)] $\Indu{P}{G}{V_{\chi}}$, $\chi:T\rightarrow \bar{k}^{\times}$ with $\chi\neq \chi^s$;
\item[(iv)] supersingular $\cIndu{KZ}{G}{\sigma}/(T)$.
\end{itemize}
\end{prop}
\begin{proof} Since $k$ is perfect it follows from the proof of Proposition \ref{reducetoabs} that there exists 
a finite Galois extension $l$ of $k$ such that 
$$\pi\otimes_k l  \cong \pi_1\oplus \ldots\oplus \pi_n$$ 
with $\pi_i$ absolutely irreducible and distinct. If $\pi_1$ is supersingular then, since the central character of $\pi_1$ is $k$-rational, $\pi_1$ 
can be realized over $k$ by Lemma \ref{fielddefi} and 
and  since $\pi$ is irreducible Lemma \ref{abstractrat} 
implies that $\pi_l \cong \pi_1$ and so $\pi$ is absolutely irreducible supersingular. 

The proof in the cases $\pi_1$ is a character, special series or principal series is the same. We only treat the principal series case 
so $\pi_1\cong \Indu{P}{G}{\chi}$ with $\chi:T\rightarrow \bar{k}^{\times}$ a smooth character with $\chi\neq \chi^s$. 
Let $\tau:=\Indu{P}{G}{V_{\chi}}$, then $\tau_l\cong \oplus_{\psi\in X} \Indu{P}{G}{\psi}$. Since $\chi\neq \chi^s$ we have $\psi\neq \psi^s$ 
for every $\psi\in X$. Hence, all the principal series are irreducible and distinct. Since the $\Gamma$-action on $\tau\otimes_k l$ permutes 
the irreducible subspaces transitively we deduce that  $\tau_l$ does not contain a proper $G$-invariant subspace, which is stable under the action of $\Gamma$.
Hence $\tau$ is an irreducible $G$-representation.   
Since $\Hom_G(\tau_l, \pi_l)\neq 0$ Lemma \ref{abstractrat} implies that $\Hom_G(\tau, \pi)\neq 0$.  Since both $\pi$ and $\tau$ are irreducible
they must be isomorphic. 
\end{proof}
\begin{remar}\label{absirreequiv} It follows from the Proposition \ref{irrkreps} that for an irreducible $\pi$ with a central character the following are equivalent:
\begin{itemize} 
\item[(i)] $\pi\otimes_k l$ is irreducible for all $l/k$ finite;
\item[(ii)] $\pi\otimes_k l$ is irreducible for all $l/k$;
\item[(iii)] $\pi\otimes_k l$ irreducible for some $l/k$ with $l$ algebraically closed.
\end{itemize}
In this case, we will say that $\pi$ is \textit{absolutely irreducible}.
\end{remar}  
Suppose  $p\in Z$ acts trivially on $\sigma$ and $\sigma|_K\cong \Sym^r k^2$. Let 
$\varphi\in \cIndu{KZ}{G}{\Sym^r k^2}$ be such that 
 $\supp \varphi =ZK$ and $\varphi(1)$ is non-zero and $I_1$-invariant. If we identify $\Sym^r k^2$ 
with the space of homogeneous polynomials in two variables $x$ and $y$ of degree $r$, then we may take 
$\varphi(1)=x^r$.  Since $\varphi$ generates $\cIndu{KZ}{G}{\Sym^r k^2}$ as 
a $G$-representation $T$ is determined by $T\varphi$. 
\begin{lem}\label{TT}
 \begin{itemize}
\item[(i)] If $r=0$  then $ T \varphi = \left(\begin{smallmatrix} 0 & 1 \\ p & 0 \end{smallmatrix} \right ) \varphi +\sum_{\lambda \in \Fp} \left(\begin{smallmatrix} 1 & [\lambda]\\ 0 & 1
\end{smallmatrix}\right) t \varphi .$
\item[(ii)] Otherwise, $ T \varphi = \sum_{\lambda \in \Fp} \left(\begin{smallmatrix} 1 & [\lambda] \\ 0 & 1\end{smallmatrix} \right)t  \varphi.$
\end{itemize}
\end{lem}
\begin{proof} In the notation of \cite{bl} this is a calculation of $T([1, e_{\vec{0}}])$. The claim follows from the formula (19) in the proof 
of \cite{bl} Theorem 19. 
 \end{proof}
Theorem 19 in  \cite{bl} says that $\cIndu{KZ}{G}{\sigma}$ is a free $k[T]$-module. 
Hence, the map $T-\lambda$ is injective, for all $\lambda\in k$.
\begin{defi} Let $\pi(r,\lambda)$ be a representation of $G$ defined by the exact sequence:
\begin{equation}\label{definingpir}
\xymatrix{ 0 \ar[r] & \cIndu{ZK}{G}{\Sym^r k^2}\ar[r]^-{T-\lambda}&  \cIndu{ZK}{G}{\Sym^r k^2}\ar[r]& \pi(r,\lambda)\ar[r]& 0.}
\end{equation}
If $\eta: \Qp^{\times}\rightarrow k^{\times}$ is a smooth character then let $\pi(r,\lambda,\eta):=\pi(r,\lambda)\otimes \eta\circ \det$.
\end{defi}
It follows from \cite[Thm.30]{bl} and \cite[Thm.1.1]{breuil1} that $\pi(r,\lambda)$ is  absolutely irreducible unless $(r,\lambda)=(0,\pm1)$ or $(r,\lambda)=(p-1,\pm1)$.
Moreover, one has non-split exact sequences:
\begin{equation}\label{zero1}
 0 \rightarrow \mu_{\pm 1}\circ \det\rightarrow \pi(p-1,\pm 1)\rightarrow \Sp\otimes\mu_{\pm 1}\circ \det \rightarrow 0, 
\end{equation}
\begin{equation}\label{p-1}
 0 \rightarrow \Sp\otimes\mu_{\pm 1}\circ \det\rightarrow \pi(0,\pm 1)\rightarrow \mu_{\pm 1}\circ \det \rightarrow 0, 
\end{equation}
where  $\mu_{\lambda}:\Qp^{\times} \rightarrow k^{\times}$, $x\mapsto \lambda^{\val(x)}$.
Further, if $\lambda\neq 0$ and $(r,\lambda)\neq (0,\pm 1)$ then \cite[Thm.30]{bl} asserts that 
\begin{equation}\label{induced}
 \pi(r,\lambda)\cong \Indu{P}{G}{\mu_{\lambda^{-1}}\otimes \mu_{\lambda}\omega^r}.
\end{equation}

If $\pi$ is an absolutely irreducible $k$-representation of $G$ with a central character $\zeta$ and $\zeta(p)$ is a square in $k$ then 
$\pi$ is a quotient of $\pi(r, \lambda, \eta)$ for some $\lambda\in k$ and $\eta: \Qp^{\times}\rightarrow k^{\times}$. The supersingular representations 
are isomorphic (over $k[\sqrt{\zeta(p)}]$) to $\pi(r, 0, \eta)$.  
All the isomorphism between supersingular representations cor\-res\-pon\-ding to different $r$ and $\eta$ are given by 
\begin{equation}\label{intertwine}
\pi(r,0,\eta)\cong \pi(r,0,\eta\mu_{-1})\cong \pi(p-1-r,0,\eta\omega^{r})\cong \pi(p-1-r,0,\eta\omega^{r}\mu_{-1})
\end{equation}
see \cite[Thm 1.3]{breuil1}. We refer to the \textit{regular} case if $\pi\cong \pi(r, 0, \eta)$ with $0<r<p-1$, 
and \textit{Iwahori} case if $\pi\cong \pi(0, 0, \eta)\cong \pi(p-1, 0, \eta)$.

\subsection{Hecke algebra and extensions}\label{hecke}

Let $\HH:=\End_{G}(\cIndu{ZI_1}{G}{\zeta})$ and let $\Mod_{\HH}$ be the category of right $\HH$-modules. 
 Let $\II: \Mod^{\mathrm{sm}}_{G, \zeta}(k)\rightarrow \Mod_{\HH}$ be the functor: 
$$\II(\pi):=\pi^{I_1}\cong \Hom_G(\cIndu{ZI_1}{G}{\zeta}, \pi).$$
Let $\TT:\Mod_{\HH}\rightarrow \Mod^{\mathrm{sm}}_{G, \zeta}(k)$ be the functor:
$$\TT(M):=M\otimes_{\HH} \cIndu{ZI_1}{G}{\zeta}.$$
One has $\Hom_{\HH}(M, \II(\pi))\cong \Hom_G(\TT(M),\pi)$. Moreover, Vign\'eras 
in \cite[Thm.5.4]{vig} shows that $\II$ induces a bijection between irreducible objects in $\Mod^{\mathrm{sm}}_{G, \zeta}(k)$ and $\Mod_{\HH}$. 
Let $\Mod^{\mathrm{sm}}_{G, \zeta}(k)^{I_1}$ be the  full subcategory of  $\Mod^{\mathrm{sm}}_{G, \zeta}(k)$ consisting 
of representations generated by their $I_1$-invariants.  Ollivier 
has shown\footnote{In fact, both Vign\'eras and Ollivier work with the full Hecke algebra $\End_G(\cIndu{I_1}{G}{\Eins})$. Our Hecke algebra is the quotient of 
the full Hecke algebra by the ideal generated by all the elements of the form $T_z -\zeta(z)^{-1}$, where $T_z$ is the Hecke operator corresponding to the (double) coset $zI_1$, 
see \cite[\S2]{coeff}, for all $z\in Z$. In particular, if $\pi$ is a smooth representation of $G$, the action of the full Hecke algebra on $\pi^{I_1}$ factors through the action of $\HH$ if and only 
$Z$ acts on $\pi^{I_1}$ by the character $\zeta$, or equivalently the subrepresentation of $\pi$ generated by $\pi^{I_1}$ has a central character equal to $\zeta$.  The results of \cite{o2} imply that \eqref{functorsI&T} induces an equivalence of categories.}
in \cite{o2} that 
\begin{equation}\label{functorsI&T}
\II: \Mod^{\mathrm{sm}}_{G, \zeta}(k)^{I_1}\rightarrow \Mod_{\HH}, \quad \TT: \Mod_{\HH}\rightarrow \Mod^{\mathrm{sm}}_{G, \zeta}(k)^{I_1}
\end{equation} 
are quasi-inverse to each other and so $\Mod_{\HH}$ is equivalent to   $\Mod^{\mathrm{sm}}_{G, \zeta}(k)^{I_1}$. 

\begin{remar}\label{Hecke_algebra_extend_scalars} We note that $(\cIndu{ZI_1}{G}{\zeta})\otimes_k l\cong \cIndu{ZI_1}{G}{(\zeta\otimes_k l)}$, 
and since it is finitely generated we have $\HH\otimes_k l\cong \End_G(\cIndu{ZI_1}{G}{\zeta\otimes_k l})$. Moreover, $\II(\pi)\otimes_k l \cong \II(\pi\otimes_k l)$ by Lemma \ref{abstractrat} and 
$\TT(M)\otimes_k l\cong \TT(M\otimes_k l)$. Hence, if we show that the functors in \eqref{functorsI&T} induce an equivalence of  categories over 
some extension of $k$ then the same  also holds over  $k$.
\end{remar}

In particular, if  $\tau=\langle G \centerdot \tau^{I_1}\rangle$ and  $\pi$ is 
in $\Mod^{\mathrm{sm}}_{G, \zeta}(k)$  then one has:
\begin{equation}\label{homit}
\begin{split}
\Hom_G(\tau, \pi) \cong  \Hom_{\HH}(\II(\tau), \II(\pi))
\end{split}
\end{equation}
and the natural map $\TT \II(\tau)\rightarrow \tau$  is an isomorphism. We have shown in \cite[\S9]{ext2} that 
\eqref{homit} gives an $E_2$-spectral sequence:
\begin{equation}\label{specseq}
\Ext^i_{\HH}(\II(\tau), \RR^j \II(\pi))\Longrightarrow \Ext^{i+j}_{G, \zeta}(\tau, \pi)
\end{equation}
where $\Ext^n_{G, \zeta}(\tau, \ast)$ is the $n$-th right derived functor of $\Hom_G(\tau, \ast)$ on $\Mod^{\mathrm{sm}}_{G, \zeta}(k)$.
The $5$-term sequence associated to \eqref{specseq} gives us:
\begin{equation}\label{5T}
\begin{split}
0\rightarrow &\Ext^1_{\HH}(\II(\tau), \II(\pi))\rightarrow 
\Ext^1_{G, \zeta}(\tau, \pi)\rightarrow \Hom_{\HH}(\II(\tau), \RR^1\II(\pi))\\
&\rightarrow \Ext^2_{\HH}(\II(\tau), \II(\pi))\rightarrow \Ext^2_{G, \zeta}(\tau, \pi)
\end{split}
\end{equation}

Let $\Mod_{G, \zeta}^{\mathrm{l \, adm}}(\OO)$ (resp.\,$\Mod_{G, \zeta}^{\mathrm{l \, adm}}(k)$) be  the full subcategory of $\Mod^{\mathrm{sm}}_{G, \zeta}(\OO)$ (resp. $\Mod^{\mathrm{sm}}_{G, \zeta}(k)$)
consisting of all locally admissible representations, see \S\ref{zerosec}. As already explained in \S \ref{zerosec}, it follows from \cite[Thm.2.3.8]{ord1} that a smooth representation of $G$ with a central character is locally admissible 
if and only if it is locally of finite length, so that  $\Mod_{G, \zeta}^{\mathrm{l \, adm}}(\OO)=\Mod_{G, \zeta}^{\mathrm{lfin}}(\OO)$ and 
$\Mod_{G, \zeta}^{\mathrm{l \, adm}}(k)=\Mod_{G, \zeta}^{\mathrm{lfin}}(k)$. The inclusion 
  $\iota: \Mod_{G, \zeta}^{\mathrm{l \, adm}}(\OO)\rightarrow \Mod^{\mathrm{sm}}_{G, \zeta}(k)$ has a right adjoint functor  $V\mapsto V_{\mathrm{l\,adm}}$, which associates to $V$ the subset of all locally admissible elements. 
 Taking locally admissible elements is a left exact functor, see \cite[2.2.19]{ord1}, which is the identity functor on locally admissible representations. 
 Let $\dualcat(\OO)$ (resp.\,$\dualcat(k)$) be the full subcategory of $\Mod^{\mathrm{pro\, aug}}_{G}(\OO)$ 
 anti-equivalent to  $\Mod_{G, \zeta}^{\mathrm{l \, adm}}(\OO)$ (resp.\,$\Mod_{G, \zeta}^{\mathrm{l \, adm}}(k)$) via the Pontryagin duality. 
 
 \begin{prop}\label{injgoinj} The functor $\iota: \Mod^{\mathrm{l\, adm}}_{G, \zeta}(k)\rightarrow \Mod^{\mathrm{sm}}_{G, \zeta}(k)$ maps injectives to injectives. 
\end{prop} 
\begin{proof} Let $J$ be an injective object in $\Mod^{\mathrm{l\, adm}}_{G, \zeta}(k)$ and let $\iota(J)\hookrightarrow J_1$ be 
an injective envelope of $\iota(J)$ in 
$\Mod^{\mathrm{sm}}_{G, \zeta}(k)$. If $\iota(J)\neq J_1$ then $(J_1/\iota(J))^{I_1}\neq 0$ and thus there exists 
$v\in  (J_1/\iota(J))^{I_1}$ such that $\sigma:=\langle K\centerdot v\rangle$ is an irreducible
representation of $K$. Let $A:= \langle G\centerdot v\rangle \subseteq J_1/\iota(J)$ then by pulling back 
we obtain $B\subseteq J_1$ and an exact sequence:
\begin{equation}\label{AB}
0\rightarrow \iota(J)\rightarrow B\rightarrow A\rightarrow 0.
\end{equation}
Since $\iota(J)\hookrightarrow J_1$ is essential, the class of the sequence \eqref{AB} is a non-zero element 
in $\Ext^1_{G, \zeta}(A, \iota(J))$. We will show that $\Ext^1_{G, \zeta}(A, \iota(J))=0$ and thus obtain
a contradiction to $\iota(J)\neq J_1$.

Since $J_1$ has a central character $\zeta$, $Z$ acts on $\sigma$ by $\zeta$. 
 Let $\varphi\in \cIndu{KZ}{G}{\sigma}$ be such that 
 $\supp \varphi =ZK$ and $\varphi(1)$ spans $\sigma^{I_1}$. By Frobenius reciprocity 
we obtain a map $\psi: \cIndu{KZ}{G}{\sigma}\rightarrow A$, which sends $\varphi$ to $v$. 
Since $v$ generates $A$ as a $G$-representation, $\psi$ is surjective. It is shown in \cite[Cor. 3.8]{eff} that 
the restriction functor $\Mod_{G, \zeta}^{\mathrm{l \, adm}}(k)\rightarrow \Mod^{\mathrm{sm}}_{K, \zeta}(k)$, 
$\pi \mapsto \pi|_{K}$ sends injectives to injectives. Hence, 
\begin{equation}
\Ext^1_{G, \zeta}( \cIndu{KZ}{G}{\sigma}, \iota(J))\cong \Ext^1_{K, \zeta}(\sigma, \iota(J))=0
\end{equation} 
and so $\psi$ cannot be injective. Thus $\Ker \psi$ is non-zero, and \cite[Prop 18]{bl} asserts that
$(\Ker \psi)^{I_1}$ is of finite codimension in $(\cIndu{KZ}{G}{\sigma})^{I_1}$. In particular, 
the set $\{\psi( T^n \varphi): n\ge 0\}$, where $T$ is the Hecke operator defined in \eqref{endoalgebra}, is linearly dependent and so  there exists a non-zero 
polynomial $P$ such that $\psi( P(T) \varphi)=0$. Hence, $\psi$ factors through
\begin{equation}\label{bald2}
 \cIndu{KZ}{G}{\sigma}\twoheadrightarrow \cIndu{KZ}{G}{\sigma}/(P(T)) \twoheadrightarrow A. 
\end{equation}
Since $\cIndu{KZ}{G}{\sigma}/(T-\lambda)$ is of finite  length, for all $\lambda$, by base changing
to the splitting field of $P(T)$, we see that $\cIndu{KZ}{G}{\sigma}/(P(T))$ is of finite length 
and hence is  admissible and thus $A$ is admissible. 

We  claim  that there exists a finite length subrepresentation $\kappa$ of $B$ such the $B=\kappa+\iota(J)$. The claim  implies that $B$ is locally finite (or equivalently locally admissible). Since $J$ is injective in $ \Mod^{\mathrm{l\, adm}}_{G, \zeta}(k)$, the claim implies that \eqref{AB} is split. To prove the claim, we proceed as follows. Choose $w\in B$, which maps to $v$ in $A$. Let $\tau$ be the $KZ$-subrepresentation 
of $B$ generated by $v$. Since $Z$ acts by the central  character and the action of $K$ is smooth, $\tau$ is finite dimensional. By Frobenius reciprocity we obtain a map 
$\theta: \cIndu{KZ}{G}{\tau} \rightarrow B$, such that the composition $\cIndu{KZ}{G}{\tau} \overset{\theta}{\rightarrow} B\rightarrow A$ is surjective.
Let $\tau'$ be the kernel of the surjection $\tau\twoheadrightarrow \sigma$. Since compact induction is an exact functor we obtain an exact sequence 
\begin{equation}\label{bald}
0\rightarrow \cIndu{KZ}{G}{\tau'}\rightarrow \cIndu{KZ}{G}{\tau}\rightarrow \cIndu{KZ}{G}{\sigma}\rightarrow 0.
\end{equation}
Let $\Upsilon$ be the subrepresentation of $ \cIndu{KZ}{G}{\tau}$ fitting into the exact sequence
\begin{equation}\label{bald1}
0\rightarrow \cIndu{KZ}{G}{\tau'}\rightarrow \Upsilon\rightarrow P(T) (\cIndu{KZ}{G}{\sigma})\rightarrow 0,
\end{equation}
where $P(T)$ is the endomorphism of $ \cIndu{KZ}{G}{\sigma}$ constructed in \eqref{bald2}. It follows from \eqref{bald2} that  $\cIndu{KZ}{G}{\tau}/\Upsilon$ is of finite length, and the composition 
$\Upsilon\overset{\theta}{\rightarrow} B \rightarrow A$ is zero. Hence, $\theta(\Upsilon)$ is contained in $\iota(J)$.  It follows from \eqref{bald1} that  $\Upsilon$ is a finitely generated $G$-representation, and since $J$ is locally finite, we deduce that $\theta(\Upsilon)$ is of finite length. Thus $\kappa:=\theta(\cIndu{KZ}{G}{\tau})$ is of finite length, and the composition 
$\kappa\hookrightarrow B\twoheadrightarrow A$ is surjective, which implies the claim.

\end{proof}  

If $\tau$ is in $\Mod_{G, \zeta}^{\mathrm{l \, adm}}(k)$ then  we denote by $\Ext_{G, \zeta}^{ \mathrm{l\,adm}, n}(\tau, \ast)$ (resp. $\Ext_{G, \zeta}^{ n}(\tau, \ast)$)
the $n$-th right derived functor of $\Hom_G(\tau, \ast)$ in $\Mod_{G, \zeta}^{\mathrm{l \, adm}}(k)$ (resp. $\Mod_{G, \zeta}^{\mathrm{sm}}(k)$).
It follows from Corollary \ref{enoughinj} or \cite[2.1.1]{ord2} that $\Mod_{G, \zeta}^{\mathrm{l \, adm}}(k)$ and $\Mod_{G, \zeta}^{\mathrm{l \, adm}}(\OO)$ have enough injectives.

\begin{cor}\label{thesame} Let $\tau$ and $\pi$ be in $\Mod_{G, \zeta}^{\mathrm{l \, adm}}(k)$ then $\iota$ induces an isomorphism 
\begin{equation}
\Ext^n_{G, \zeta}(\iota(\tau), \iota(\pi))\cong \Ext^{\mathrm{l\, adm},n}_{G, \zeta}(\tau,\pi)\cong \Ext^n_{\dualcat(k)}(\pi^{\vee}, \tau^{\vee}),
\quad  \forall n\ge 0.
\end{equation}
\end{cor}

\begin{cor}\label{injgoinjO} The functor 
$\iota: \Mod_{G, \zeta}^{\mathrm{l \, adm}}(\OO)\rightarrow \Mod_{G, \zeta}^{\mathrm{sm}}(\OO)$ maps 
injectives to injectives.
\end{cor} 
\begin{proof} Let $J$ be an injective object in $\Mod_{G, \zeta}^{\mathrm{l \, adm}}(\OO)$ and $\iota(J)\hookrightarrow J_1$ be 
an injective envelope of $\iota(J)$ in $\Mod_{G, \zeta}^{\mathrm{sm}}(\OO)$. Now $J[\varpi]$ is injective in 
$\Mod_{G, \zeta}^{\mathrm{l \, adm}}(k)$ and $\iota(J[\varpi])\hookrightarrow J_1[\varpi]$ is an injective envelope 
of $\iota(J[\varpi])$ in  $\Mod_{G, \zeta}^{\mathrm{sm}}(k)$. It follows from Proposition \ref{injgoinj} that 
$J_1[\varpi]\cong \iota(J[\varpi])=\iota(J)[\varpi]$. Hence, we obtain an injection 
$(J_1/\iota(J))[\varpi]\hookrightarrow \iota(J)/\varpi \iota(J)$. This implies that $(J_1/\iota(J))[\varpi]$ is an object of 
$\Mod_{G, \zeta}^{\mathrm{l \, adm}}(\OO)$ and so the extension
$$0\rightarrow \iota(J)\rightarrow A\rightarrow (J_1/\iota(J))[\varpi]\rightarrow 0$$ 
splits, where $A\subset J_1$. Since $\iota(J)\hookrightarrow J_1$ is essential, we get that $(J_1/\iota(J))[\varpi]=0$, 
which implies that $J_1=\iota(J)$.  
\end{proof} 

\begin{cor}\label{projaretfree} Projective objects in $\dualcat(\OO)$ are $\OO$-torsion free. In particular, the hypothesis (H0) of 
\S\ref{def} is satisfied.
\end{cor}
\begin{proof} Let $P$ be a projective object in $\dualcat(\OO)$ then $P^{\vee}$ is an injective object in 
$\Mod_{G, \zeta}^{\mathrm{l \, adm}}(\OO)$  and also in $\Mod_{G, \zeta}^{\mathrm{sm}}(\OO)$ by Corollary \ref{injgoinjO}
and it is enough to show that $P^{\vee}$ is $\varpi$-divisible. 
We claim that any $V$ in $\Mod_{G, \zeta}^{\mathrm{sm}}(\OO)$ may be embedded into an object which is $\varpi$-divisible. 
The claim gives the result, since injectivity of  $P^{\vee}$ implies that the embedding must split. Since direct summands
of $\varpi$-divisible modules are $\varpi$-divisible, we are done.  We may embed $j:V\hookrightarrow W$  into a $\varpi$-divisible 
$\OO$-torsion module, since the category of $\OO$-torsion modules has enough injectives 
and these are $\varpi$-divisible. The embedding $V\hookrightarrow C^u(G, W)$, $v\mapsto [g\mapsto j(gv)]$ where the target is the space of uniformly continuous 
functions with discrete topology on $W$ solves the problem.
\end{proof} 
 
\begin{cor}\label{projaretfree2} Let $M$ be an $\OO$-torsion free object of $\dualcat(\OO)$ then $M$ is projective 
in $\dualcat(\OO)$ if and only if $M\otimes_{\OO} k$ is projective in $\dualcat(k)$.
\end{cor}
\begin{proof} Since every $A$ in $\dualcat(k)$ is killed by $\varpi$ we have 
$$\Hom_{\dualcat(k)}(M\otimes_{\OO} k, A)\cong \Hom_{\dualcat(\OO)}(M, A).$$
Hence, if $M$ is projective in $\dualcat(\OO)$ then the functor $\Hom_{\dualcat(k)}(M\otimes_{\OO} k, \ast)$ is exact and 
so $M\otimes_{\OO} k$ is projective in $\dualcat(k)$.

Let $\phi:P\twoheadrightarrow M\otimes_{\OO} k$ be a projective envelope of $M\otimes_{\OO} k$ in 
$\dualcat(\OO)$. Since $M\rightarrow M\otimes_{\OO} k$ is essential and $P$ is projective there exists a surjection 
$\psi: P\twoheadrightarrow M$ such that the diagram 
\begin{displaymath}
\xymatrix@1{P\ar@{->>}[r]^-{\psi}\ar@{->>}[dr]_-{\phi} & M\ar@{->>}[d] \\ & M\otimes_{\OO} k }
\end{displaymath}
commutes. Since $M\otimes_{\OO} k$ is projective it is its own projective envelope in $\dualcat(k)$. Thus 
it follows from Lemma \ref{reduceprojenv} that $\phi$ induces an isomorphism $P\otimes_{\OO} k\cong M \otimes_{\OO} k$. 
Since $M$ is $\OO$-torsion free we get  $(\Ker \psi)\otimes_{\OO} k=0$. 
Nakayama's lemma implies that $\Ker \psi=0$ and hence $M\cong P$ is projective.
\end{proof}   

\begin{cor}\label{projaretfree3} Let $P_{\bullet}\twoheadrightarrow M$ be a projective resolution of $M$
in $\dualcat(k)$. Let $\wM$ be an $\OO$-torsion free object of $\dualcat(\OO)$ such that $\wM\otimes_{\OO} k\cong M$.
Then there exists a projective resolution $\wP_{\bullet}\twoheadrightarrow \wM$ of $\wM$ in $\dualcat(\OO)$
lifting the resolution of $M$.
\end{cor}
\begin{proof} Let $\phi:P\twoheadrightarrow M$ be  an epimorphism in $\dualcat(k)$ with $P$ projective and 
let $\wP$ be a projective  envelope of $P$ in $\dualcat(\OO)$. Lemma \ref{reduceprojenv} says that 
$P\cong \wP\otimes_{\OO} k$.  Since $\wM \otimes_{\OO} k \cong \wM/\varpi \wM\cong M$, the epimorphism 
$\wM\twoheadrightarrow M$ is essential by an application of Nakayama's lemma and since $\wP$ is projective there exists $\tilde{\phi}: \wP\twoheadrightarrow \wM$ 
such that the diagram 
\begin{displaymath}
\xymatrix@1{\wP\ar@{->>}[r]^-{\tilde{\phi}}\ar@{->>}[d] & \wM\ar@{->>}[d] \\ P \ar@{->>}[r]^-{\phi}& M}
\end{displaymath}
commutes. Since $\wM$ is $\OO$-torsion free it is $\OO$-flat and hence $(\Ker \tilde{\phi})\otimes_{\OO} k\cong \Ker \phi$. 
Moreover, $\wP$ is $\OO$-torsion free by Corollary \ref{projaretfree}, and hence $\Ker \tilde{\phi}$ is $\OO$-torsion free.
We may then continue to lift the whole resolution.
\end{proof}

\begin{lem}\label{RisH} 
Let $\pi$ be a smooth $k$-representation of $G$ with a central character $\zeta$. Forgetting the $\HH$-action  induces 
an isomorphism $\RR^i\II(\pi)\cong H^i(I_1/Z_1, \pi)$ for all $i\ge 0$.
\end{lem} 
\begin{proof} Let $\Res_{I_1}: \Mod^{\mathrm{sm}}_{G, \zeta}(k)\rightarrow \Mod^{\mathrm{sm}}_{I_1, \zeta}(k)$ 
be the restriction to $I_1$. Since $\Res_{I_1}$ is right adjoint to an exact functor
$\cInd_{ZI_1}^{G}$, $\Res_{I_1}$ maps injective objects to injective objects.
Since $\zeta$ is smooth and $I_1$ is pro-$p$, $\zeta$ is trivial 
on $Z_1:=I_1\cap Z$, hence we may identify 
$\Mod^{\mathrm{sm}}_{I_1, \zeta}(k)$ with $\Mod^{\mathrm{sm}}_{I_1/Z_1}(k)$. 
Choose an injective resolution $\pi\hookrightarrow J^{\bullet}$ 
of $\pi$ in $\Mod^{\mathrm{sm}}_{G, \zeta}(k)$. Then $\pi|_{I_1}\hookrightarrow (J|_{I_1})^{\bullet}$
is an injective resolution of $\pi|_{I_1}$ in  $\Mod^{\mathrm{sm}}_{I_1/Z_1}(k)$. Hence, for all 
$i\ge 0$ we get an isomorphism of $k$-vector spaces 
$\RR^i \II(\pi)\cong H^i(I_1/Z_1, \pi)$.
\end{proof}

The results proved in the rest of the subsection will only be used in \S\ref{nongenericcaseII}.

\begin{lem}\label{vanish3} If $p\ge 5$ then $\RR^i \II=0$ for $i\ge 4$.
\end{lem}
\begin{proof} We have an isomorphism 
$$ I_1/Z_1\cong (I_1\cap U^s)\times (I_1\cap T)/Z_1 \times (I_1\cap U) \cong \Zp \times \Zp \times \Zp.$$ 
Hence, $I_1/Z_1$ is a compact $p$-adic analytic group of dimension $3$. Since we assume $p\ge 5$ it is $p$-saturable 
\cite[III.3.2.7.5]{laz}, and hence 
torsion free. Thus $I_1/Z_1$ is a Poincar\'e group of dimension $3$, \cite[V.2.5.8]{laz} and \cite{serreprop}. Hence
$H^i(I_1/Z_1, \ast)=0$ for all $i>3$ and the assertion follows from Lemma \ref{RisH}. 
\end{proof}

\begin{lem}\label{vanish2} Let $\tau$ be a smooth  irreducible representation 
of $G$ with a central character $\zeta$, such that $\tau\not\cong \pi(r, 0, \eta)$ with  $0<r<p-1$.
Then  $\Ext^i_{\HH}(\II(\tau), \ast)=0$ for $i\ge 2$. 
\end{lem}
\begin{proof} Using Lemma \ref{bcH} we may reduce to the case  where  $\tau$ is absolutely irreducible, which we now assume.
It is enough to produce an exact sequence of $\HH$-modules:
\begin{equation}\label{resprojmodH}
0\rightarrow P_1\rightarrow P_0\rightarrow \II(\tau)\oplus M\rightarrow 0
\end{equation}
with $P_0$ and $P_1$ projective and $M$ arbitrary. We observe that if $A$ is a direct summand 
of $\cIndu{I_1Z}{G}{\zeta}$, then $\II(A)$ is a direct summand of  $\II(\cIndu{I_1Z}{G}{\zeta})\cong \HH$
and hence $\II(A)$ is projective. If $\tau\cong \pi(r, \lambda, \eta)$, with 
$\lambda\neq 0$ and $0<r<p-1$, then such sequence is constructed in \cite[Cor.6.6, Eq.(12)]{bp}. 
If $\tau\cong \pi(r, \lambda, \eta)$ with $r=0$ or $r=p-1$ then one may obtain \eqref{resprojmodH} by applying 
$\II$ to \eqref{definingpir}. The sequence remains exact by \cite[2.9, 2.8]{bl} in the non-, and by
\cite[3.2.4, 3.2.5]{breuil1} in the supersingular case. If $\tau=\eta\circ \det$ or $\tau=\Sp\otimes\eta\circ \det$, 
then $\tau$ may be realized as an $H_0$ of the diagram $\tau^{I_1}\hookrightarrow \tau^{K_1}$, see \cite[Thm.10.1]{bp}. 
This means an exact sequence:
\begin{equation}
0\rightarrow \cIndu{\KK}{G}{\tau^{I_1}\otimes \delta}\rightarrow \cIndu{KZ}{G}{\tau^{K_1}}\rightarrow \tau\rightarrow 0.
\end{equation}
where $\delta(g)=(-1)^{\val(\det g)}$, where $\KK$ is the $G$-normalizer of $I$.  Again applying $\II$ we get \eqref{resprojmodH}.
\end{proof}

\begin{cor}\label{vanish4} Let $M$ be a finite dimensional $\HH$-module, such that the 
irreducible subquotients are isomorphic to $\II(\tau)$, where $\tau$ is as above, then $\Ext^i_{\HH}(M, \ast)=0$, for $i\ge 2$.
\end{cor}

\begin{prop}\label{comphext} 
Let $\pi$ and $\tau$ be in $\Mod^{\mathrm{sm}}_{G, \zeta}(k)$. Suppose that $\tau$ is admissible, generated by $\tau^{I_1}$,
and the irreducible subquotients of $\II(\tau)$ are not isomorphic to $M(r,0, \eta)$ with $0<r<p-1$.  Then for $i\ge 1$ there exists an exact sequence: 
\begin{equation}\label{higherexts}
\Ext^1_{\HH}(\II(\tau), \RR^{i-1}\II(\pi)) \hookrightarrow \Ext^i_{G,\zeta}(\tau, \pi)\twoheadrightarrow 
\Hom_{\HH}(\II(\tau), \RR^i\II(\pi)).
\end{equation}
If $p\ge 5$ then $\Ext^4_{G,\zeta}(\tau, \pi)\cong \Ext^1_{\HH}(\II(\tau), \RR^3\II(\pi))$ and $\Ext^i_{G,\zeta}(\tau, \pi)=0$ for $i\ge 5$. 
\end{prop}
\begin{proof} This follows from  a calculation with the spectral sequence \eqref{specseq}.
Let $E_2^{pq}= \Ext^p_{\HH}(\II(\tau), \RR^q\II(\pi))$. Then $E_2^{pq}=0$ for $p>1$, by Corollary \ref{vanish4}. Thus 
$E_{\infty}^{pq}= E_2^{pq}$ and for all $n\ge 0$ we obtain an exact sequence
\begin{equation}
0\rightarrow E_2^{0,n}\rightarrow E^n \rightarrow E_2^{1,n-1}\rightarrow 0.
\end{equation}   
If $p\ge 5$ then $E_2^{pq}=0$ for $q>3$ by Lemma \ref{vanish3}, which implies  the assertion.
\end{proof}

\begin{lem}\label{comphext1} Let $M$, $N$ be absolutely irreducible $\HH$-modules 
and let $d$ be the dimension of $\Ext^1_{\HH}(M, N)$. 
If $p>2$  and $d\neq 0$ then one of the following holds:
\begin{itemize}
\item[(i)] $M\cong N \cong \II(\pi(r,0, \eta))$ with $0<r<p-1$ and $d=2$;
\item[(ii)] $M\cong N$ and  $M\not\cong \II(\Sp\otimes \eta)$, $M\not\cong \II(\eta)$, $M\not \cong \II(\pi(r,0, \eta))$ 
with $0<r<p-1$ and $d=1$;
\item[(iii)] either ($M\cong \II(\eta)$ and $N\cong \II(\Sp\otimes \eta)$) or 
($N\cong \II(\eta)$ and $M\cong \II(\Sp\otimes \eta)$) and $d=1$.
\end{itemize}
where  $\eta:G\rightarrow k^{\times}$ is a smooth character .
\end{lem}
\begin{proof} If $N\cong \II(\Sp\otimes \eta)$ or $N\cong \II(\eta)$ the assertion follows from \cite[11.3]{ext2}.
Otherwise, $N\cong \II(\pi(r, \lambda, \eta))$ and the assertion follows from Corollaries 6.5, 6.6 and 6.7 \cite{bp}.
We note that when the module denoted by $M$ in \cite[Cor.6.7]{bp} is irreducible, which is the case of interest here, it is isomorphic 
to the module denoted by $M'$  in \cite[Cor.6.7]{bp}, as they are both isomorphic to  $\II(\Indu{B}{G}{\mu_{\lambda}\otimes\mu_{\lambda^{-1}}}).$ 
\end{proof}
\begin{remar} Let $T_p$ be the Hecke operator in the full Hecke algebra $\End_G(\cIndu{I_1}{G}{\Eins})$ corresponding to the (double) coset
$\left( \begin{smallmatrix} p & 0\\ 0 & p\end{smallmatrix}\right) I_1$, and let $\lambda=\zeta(\left( \begin{smallmatrix} p & 0\\ 0 & p\end{smallmatrix}\right))$.
In \cite{bp} we work with the algebra $\HH_{p=\lambda}$, which is the quotient of the full Hecke algebra by the ideal generated by $T_p -\lambda^{-1}$.
Let $e_{\zeta}:=|Z\cap K/Z\cap K_1|^{-1}\sum_{z\in Z\cap K/Z\cap K_1}  \zeta(z) T_z$, where $T_z$ is the Hecke operator corresponding to the (double) coset $z I_1$, see 
 \cite[\S 2]{coeff}. Then $e_{\zeta}$ is a central idempotent in $\HH_{p=\lambda}$ and 
$\HH= e_{\zeta} \HH_{p=\lambda} e_{\zeta}$. Since $e_{\zeta}$ is a central idempotent we may calculate the $\Ext$-groups of $\HH$-modules in the category 
of $\HH_{p=\lambda}$-modules, which allows us to use the results of \cite{bp}.
\end{remar}

\subsection{Blocks}\label{blocks}
We show that the category $\Mod^{\mathrm{lfin}}_{G,\zeta}(\OO)$ naturally decomposes into a direct 
product of subcategories. 

\begin{lem}\label{bccompact} Let $J$ be an injective object of $\Mod^{\mathrm{sm}}_{\GG}(k)$, 
where $\GG$ is a profinite group. Let 
$l$ be a field extension of $k$ then $J\otimes_k l$ is an injective object of
$\Mod^{\mathrm{sm}}_{\GG}(l)$.
\end{lem}
\begin{proof} Let $V$ be a $k$-vector space and let $C(\GG, V)$ be the 
space of continuous functions $f:\GG\rightarrow V$. For every smooth $k$-representation $\pi$ of $\GG$ the map $\phi\mapsto [v\mapsto \phi(v) (1)]$ 
induces an isomorphism
$$\Hom_{\GG}(\pi, C(\GG, V))\cong \Hom_k(\pi, V).$$
The inverse is given by $\ell \mapsto [ v\mapsto [g\mapsto \ell(gv)]]$. The functor 
$\Hom_k(\ast, V)$ is exact and so $C(\GG, V)$ is an injective object 
of  $\Mod^{\mathrm{sm}}_{\GG}(k)$. The natural injection 
$C(\GG, V)\otimes_k l \hookrightarrow C(\GG, V\otimes_k {l})$ is 
also a surjection, since for every open subgroup $\mathcal P$ of $\mathcal G$ we have 
$$ (C(\GG, V)\otimes_k l)^{\PP}\cong  (k[\GG/\PP]\otimes_k V)\otimes_k l \cong l[\GG/\PP]\otimes_l V_l \cong C(\GG, V_l)^{\PP},$$
as $\PP$ is of finite index in $\GG$. This gives us the lemma for $J= C(\GG, V)$. 
In general, one can embed $J$ 
into $C(\GG, V)$ by taking $V$ to be the underlying vector space of $J$. 
Since $J$ is injective the embedding splits. Thus $J\otimes_k l$ is 
a direct summand of an injective object of $\Mod^{\mathrm{sm}}_{\GG}(l)$
and hence it is itself injective.
\end{proof}

\begin{cor}\label{bccoh} Let $\GG$, $\pi$ and $l$ be as above then 
$H^i(\GG, \pi)\otimes_k l\cong H^i(\GG, \pi\otimes_k l)$ for all $i\ge 0$. 
\end{cor} 
\begin{proof} Choose an injective resolution $\pi\hookrightarrow J^{\bullet}$ 
of $\pi$ in $\Mod^{\mathrm{sm}}_{\GG}(k)$. Lemma \ref{bccompact} says that 
$\pi_l \hookrightarrow J^{\bullet}_l$ is an injective resolution of $\pi_l$ 
in $\Mod^{\mathrm{sm}}_{\GG}(l)$. Since taking $\GG$-invariants commutes with $\otimes_k l$
by Lemma \ref{abstractrat} we get the assertion.
\end{proof}

\begin{cor}\label{[R,T]=0}
 Let $\pi$ be a smooth representation of $G$ with a central character $\zeta$  then 
\begin{equation}\label{RcommutesT}
\RR^i \II(\pi)\otimes_{k} l\cong \RR^i \II(\pi\otimes_k l)
\end{equation} 
for all field extensions $l$ of  $k$ and  all $i\ge 0$.
\end{cor}
\begin{proof} Lemma \ref{RisH}, Corollary \ref{bccoh}.
\end{proof}

\begin{lem}\label{bcH} Let $M$ and $N$ be $\HH$-modules. If $M$ is finitely generated over $\HH$ then
\begin{equation}\label{EcommutesT}
\Ext^i_{\HH}(M, N)\otimes_k l\cong  \Ext^i_{\HH_l}(M\otimes_k l, N\otimes_k l)
\end{equation}
for all field extensions $l$ of $k$ and all $i\ge 0$.
\end{lem}
\begin{proof} It follows from the explicit description of $\HH$ given by Vign\'eras in 
\cite{vig} that the centre of $\HH$ is noetherian and $\HH$ is a finitely generated 
module over its centre, see \cite[\S1.2, 2.1.1, Cor.2.3]{vig}. Hence $\HH$ is noetherian and since $M$ is finitely generated 
we may find a resolution $P^{\bullet}\twoheadrightarrow  M$ by free $\HH$-modules of 
finite rank. Since $\Hom_{\HH}(\HH^{\oplus n}, N)_l\cong N^{\oplus}_l\cong 
\Hom_{\HH_l}(\HH^{\oplus n}_l, N_l)$ we get the assertion.
\end{proof}

\begin{prop}\label{bcEXT} Let $\tau$ and $\pi$ be in $\Mod^{\mathrm{sm}}_{G,\zeta}(k)$ and suppose that 
$\tau$ is of finite length. Then
\begin{equation}\label{bcextseq}
\Ext^i_{k[G], \zeta}(\tau, \pi)\otimes_k l \cong \Ext^i_{l[G], \zeta}(\tau \otimes_k l, \pi \otimes_k l)
\end{equation}   
for all  field extensions $l$ of $k$ and all $i\ge 0$.
\end{prop} 
\begin{proof} We will first prove the result when $\tau$ is irreducible. Then $\tau^{I_1}$ is finite dimensional and $\tau$ is generated by as a $G$-representation by the $I_1$-invariants.
By Lemma \ref{abstractrat},   
$\HH\otimes_k l\cong \End_{l[G]}( \cIndu{ZI_1}{G}{\zeta\otimes_k l })$. Since $\tau^{I_1}$ is finite dimensional, it is a finitely generated  $\HH$-module. Combining \eqref{RcommutesT} and 
\eqref{EcommutesT} we get an isomorphism of spectral sequences:
$$\Ext^i_{\HH}(\II(\tau), \RR^j \II(\pi))\otimes_k l\cong \Ext^i_{\HH_l}(\II(\tau_l), \RR^j \II(\pi_l)).$$
Since $k$ is a field, $l$ is $k$-flat  and so it follows from \eqref{specseq} that 
$\Ext^i_{\HH}(\II(\tau), \RR^j \II(\pi))_l$ converges to $\Ext^{i+j}_{k[G], \zeta}(\tau, \pi)_l$. 
We use \eqref{specseq} again to deduce the assertion. 

We will finish the proof by induction on the length of $\tau$. We have already proved the result when $\tau$ is irreducible and \eqref{bcextseq} 
is an isomorphism for $i=0$ by Lemma \ref{abstractrat}. If $\tau$ is not irreducible, then we may consider a short exact sequence 
$0\rightarrow \tau_1\rightarrow \tau\rightarrow \tau_2\rightarrow 0$, with both $\tau_1$ and $\tau_2$ of length strictly less than the length of $\tau$.
The induction step is given by comparing the two long exact sequences induced by the short exact 
sequence  and  the $5$-Lemma.
\end{proof}

Let $\Irr_{G, \zeta}(k)$ be the set of equivalence classes of smooth irreducible $k$-rep\-re\-sen\-ta\-tions
of $G$ with central character $\zeta$. We write $\pi\rel \tau$ if  $\pi\cong \tau$ or 
$\Ext^1_{G,\zeta}(\pi, \tau)\neq 0$ or $\Ext^1_{G,\zeta}(\tau, \pi)\neq 0$. We 
write $\pi\sim \tau$ if there exists $\pi_1, \ldots, \pi_n\in \Irr_{G,\zeta}(k)$, such that $\pi\cong\pi_1$,
$\tau\cong\pi_n$ and $\pi_i\rel\pi_{i+1}$ for $1\le i\le n-1$. The relation  $\sim$ is an equivalence relation on $\Irr_{G,\zeta}(k)$.
A block is  an equivalence class of $\sim$.

\begin{prop}\label{blockdecopm} The category $\Mod^{\mathrm{l\, fin}}_{G,\zeta}(\OO)$ decomposes into
a direct product of subcategories   
 \begin{equation}\label{eq531}
 \Mod^{\mathrm{l\, fin}}_{G,\zeta}(\OO)\cong \prod_{\BB} \Mod^{\mathrm{l\, fin}}_{G,\zeta}(\OO)^{\BB}
 \end{equation}
where the product is taken over all the blocks $\BB$ and 
the objects of $\Mod^{\mathrm{l\, fin}}_{G,\zeta}(\OO)^{\BB}$ are representations with all the irreducible subquotients 
lying in $\BB$. The equivalence in \eqref{eq531} is induced by sending $(\pi^{\BB})_{\BB}$, where  each $\pi^{\BB}$ is an object of  $\Mod^{\mathrm{l\, fin}}_{G,\zeta}(\OO)^{\BB}$,
 to the direct sum $\oplus_{\BB} \pi^{\BB}$.
\end{prop}
\begin{proof} This is standard,  see \cite[\S IV.2]{gab}, especially  the Corollary after Theorem 2. Let us note that 
every irreducible object in $\Mod^{\mathrm{l\, fin}}_{G,\zeta}(\OO)$ is killed by $\varpi$ and so 
$\Irr_{G,\zeta}(k)=\Irr_{G,\zeta}(\OO)$. Moreover, if $\tau$ and $\pi$ are irreducible then $\Ext^1_{\OO[G], \zeta}(\tau, \pi)\neq 0$
implies that either $\pi\cong \tau$ or $\Ext^1_{k[G], \zeta}(\tau, \pi)\neq 0$, see the proof of Lemma \ref{lemok0}. So 
we could have defined $\sim$ by considering the extensions in  $\Mod^{\mathrm{l\, fin}}_{G,\zeta}(\OO)$. Let $J_{\pi}$ and
$J_{\tau}$ be injective envelopes  of $\pi$ and $\tau$ in $\Mod^{\mathrm{l\, fin}}_{G,\zeta}(\OO)$. Then the following are equivalent:
1) $\Hom_{G}(J_{\pi}, J_{\tau})\neq 0$;  2) $\tau$ is a subquotient of $J_{\pi}$; 3) there exists a representation $\kappa$ 
of finite length which contains $\tau$ as a subquotient and $\soc_G \kappa\cong \pi$. Using this one can 
show that our definition of a subcategory cut out by a block coincides with the one used in \cite{gab}.
\end{proof}

Dually we obtain: 
\begin{cor}\label{blockdecompD} The category $\dualcat(\OO)$ decomposes into a direct product of subcategories 
\begin{equation}\label{eq532}
\dualcat(\OO)\cong \prod_{\BB}\dualcat(\OO)^{\BB},
\end{equation}
where  the product is taken over all the blocks $\BB$ and the objects of $\dualcat(\OO)^{\BB}$ are those $M$ in $\dualcat(\OO)$ such that for every irreducible subquotient $S$ of $M$, 
$S^{\vee}$ lies in $\BB$. The equivalence in \eqref{eq532} is induced by sending $(M^{\BB})_{\BB}$, where  each $M^{\BB}$ is an object of  $\dualcat(\OO)^{\BB}$,
 to the direct product $\prod_{\BB} M^{\BB}$.
\end{cor}

Let $\Ban^{\mathrm{adm}}_{G, \zeta}(L)$ be the category of admissible unitary $L$-Banach space representations of $G$ with a central character $\zeta$. 
We note that it follows from \cite{iw} and \cite[6.2.16]{em1} that  $\Ban^{\mathrm{adm}}_{G, \zeta}(L)$ is an abelian category. 

\begin{prop}\label{blockdecompB} The category $\Ban^{\mathrm{adm}}_{G, \zeta}(L)$ decomposes into a direct sum of 
categories: 
$$\Ban^{\mathrm{adm}}_{G, \zeta}(L)\cong \bigoplus_{\BB} \Ban^{\mathrm{adm}}_{G, \zeta}(L)^{\BB} ,$$
where the objects of  $\Ban^{\mathrm{adm}}_{G, \zeta}(L)^{\BB}$ are those $\Pi$ in $\Ban^{\mathrm{adm}}_{G,\zeta}(L)$ such that 
for every open bounded $G$-invariant lattice $\Theta$ in $\Pi$ the irreducible subquotients of $\Theta\otimes_{\OO} k$ 
lie in $\BB$.
\end{prop}
\begin{proof} Recall that we have showed in Lemma \ref{commen} that the reductions mod $p$ of any two open bounded $G$-invariant lattices 
in $\Pi$ have the same irreducible subquotients. Let $\Theta$ be an open bounded $G$-invariant lattice in $\Pi$, $\pi$ an irreducible 
subquotient of $\Theta\otimes_{\OO} k$ and $\BB$ the block of $\pi$. By Lemma \ref{contextGL2}  
the  Schikhof dual $\Theta^d$ is an object of $\dualcat(\OO)$. Let $\dualcat(\OO)^{\BB}$ be the full subcategory of $\dualcat(\OO)$ as in Corollary \ref{blockdecompD} and let  
$\dualcat(\OO)_{\BB}$ be the full subcategory of $\dualcat(\OO)$, such that $M$ is an object if and only if for all irreducible subquotients $S$ of $M$, the Pontryagin dual 
$S^{\vee}$ does not lie in $\BB$. It follows from  Corollary \ref{blockdecompD} that  we may canonically decompose 
$\Theta^d\cong (\Theta^d)_{\BB} \oplus (\Theta^d)^{\BB}$, where  $(\Theta^d)_{\BB}$ is an object of $\dualcat(\OO)_{\BB}$ 
and $(\Theta^d)^{\BB}$ is an object of $\dualcat(\OO)^{\BB}$. 

Let $\Pi_{\BB}:=\Hom^{cont}_{\OO}((\Theta^d)_{\BB}, L)$ and 
$\Pi^{\BB}:=\Hom^{cont}_{\OO}((\Theta^d)^{\BB}, L)$ with the supremum norm. Then it follows from the anti-equivalence of categories established in 
\cite{iw} that $\Pi\cong \Pi_{\BB} \oplus \Pi^{\BB}$. Further, since the decomposition in Corollary \ref{blockdecompD} is a decomposition of categories
we have no non-zero morphisms in $\dualcat(\OO)$ between $(\Theta^d)_{\BB}$ and $(\Theta^d)^{\BB}$. Dually this implies that there
are no non-zero morphisms between $\Pi_{\BB}$ and $\Pi^{\BB}$ in $\Ban^{\mathrm{adm}}_{G, \zeta}(L)$. Using  $(\Theta\otimes_{\OO} k)^{\vee}\cong
\Theta^d\otimes_{\OO} k$, see \cite[Lem. 5.4]{comp}, 
we deduce that $\Pi^{\BB}$ is a non-zero  object of $\Ban^{\mathrm{adm}}_{G, \zeta}(L)^{\BB}$ and 
none of the irreducible representations in $\BB$ appear as subquotients of the reduction modulo $p$ of any open bounded lattice in 
$\Pi_{\BB}$. Inductively we obtain a sequence of closed $G$-invariant subspaces $\Pi_i$ of $\Pi$ such that  
$\Pi_i\cong \Pi^{\BB_i} \oplus \Pi_{i+1}$ for some block $\BB_i$ with $\Pi^{\BB_i}\neq 0$ if $\Pi_i\neq 0$. Since 
$\Pi$ is admissible such sequence must become stationary, see \cite[Lem.5.8]{comp}. Hence, there exist finitely many  blocks $\BB_1, \ldots, \BB_m$ 
such that $\Pi^{\BB_i}\neq 0$ and so $\Pi\cong \oplus_{i=1}^m \Pi^{\BB_i}$.
\end{proof}

\begin{cor} Let $\Pi$ be an irreducible admissible $L$-Banach space representation of $G$ with a central character and let $\Theta$ be 
an open
bounded $G$-invariant lattice in $\Pi$. Then $\Theta\otimes_{\OO} k$ contains an irreducible subquotient $\pi$ and all other irreducible subquotients 
lie in the block of $\pi$. 
\end{cor}

\begin{prop}\label{defoverk} Let $l$ be a field extension of $k$. 
Let $\pi$  in $\Mod^{\mathrm{sm}}_{G,\zeta}(k)$ be absolutely irreducible and 
let $\tau$ in $\Mod^{\mathrm{sm}}_{G,\zeta}(l)$ be irreducible.  If 
$\pi\otimes_k l \rel \tau$ then there exists an absolutely irreducible
$\sigma$ in $\Mod^{\mathrm{sm}}_{G,\zeta}(k)$ such that $\tau\cong \sigma\otimes_k l$. 
Moreover, $\sigma$ is unique up to isomorphism.
\end{prop} 
\begin{proof} It follows from Lemma \ref{abstractrat} that if such $\sigma$ exists then it 
is unique. It follows from the proof of Proposition \ref{reducetoabs} that 
$\tau\otimes_l \bar{l}\cong \bigoplus_{i=1}^n \tau_i,$
where each $\tau_i$ is of finite length and $\tau_i^{ss}\cong \kappa_i^{\oplus m_i}$ with 
$\kappa_i$ absolutely irreducible. (Note that we do not require $l$ to be a perfect field, 
and hence $\tau\otimes_l \bar{l}$ need not be semisimple.) 
Thus using Proposition \ref{bcEXT} we may reduce the problem to the case  
when $l$ is algebraically closed. 
In this case in \cite{ext2}  we have determined all possible $\tau$ such that $\tau \rel \pi_l$.
It follows from the explicit description (recalled in Proposition \ref{blocksoverk} below) and Lemma \ref{fielddefi} 
that every such $\tau$ can be defined over the field of definition of $\pi$.
\end{proof}

\begin{cor}\label{bcinjenv} Let $\pi\in \Irr_{G,\zeta}(k)$ be absolutely irreducible and let $\pi\hookrightarrow J$
be an injective envelope of $\pi$ in $\Mod^{\mathrm{l\, fin}}_{G,\zeta}(k)$. Then 
$\pi\otimes_k l\hookrightarrow J\otimes_k l$ is an injective envelope of $\pi\otimes_k l$ 
in $\Mod^{\mathrm{l\, fin}}_{G,\zeta}(l)$.
\end{cor}
\begin{proof} Let $\iota: J_l\hookrightarrow J'$ be an injective envelope of $J_l$ in $\Mod^{\mathrm{l\, fin}}_{G,\zeta}(l)$. 
We claim that the quotient is zero. Otherwise, there exists 
$\tau\in \Irr_{G,\zeta}(l)$ such that $\Hom_G(\tau, J'/J_l)\neq 0$. Since $\iota$ is 
essential we have $\Hom_G(\tau, J_l)\cong \Hom_{G}(\tau, J')$ and so $\Ext^1_{G,\zeta}(\tau, J_l)\neq 0$. 
Since we are working in the category of locally finite representations, every representation is equal to the 
union of its subrepresentations of finite length. Since $\tau$ is irreducible, and hence finitely generated as a $G$-representation, 
we deduce that $\Ext^1_{G,\zeta}(\tau, J_l)\cong \underset{\rightarrow}{\lim}\Ext^1_{G, \zeta}(\tau, \kappa_l)$, where the limit is taken over all 
the finite length  subrepresentations $\kappa$ of $J$.  This implies the existence of 
a subobject $\kappa$ of $J$ of finite length, such that $\Ext^1_{G,\zeta}(\tau, \kappa_l)\neq 0$. 
Further passing to short exact sequences we may assume that $\kappa$ is irreducible 
and lies in the block of $\pi$. Proposition \ref{defoverk} implies that $\kappa$ is absolutely 
irreducible, and applying it again we deduce that there exists $\sigma\in \Irr_{G,\zeta}(k)$ 
such that $\tau\cong \sigma\otimes_k l$. As $J$ is injective in $\Mod^{\mathrm{l\, fin}}_{G,\zeta}(k)$
we have $\Ext^1_{G,\zeta}(\sigma, J)=0$ and so Proposition \ref{bcEXT} implies that 
$\Ext^1_{G,\zeta}(\tau, J_l)=0$. This is a contradiction and so $J_l\cong J'$ is injective.

Since all the irreducible subquotients $\sigma$ of $J$ are absolutely irreducible, all the irreducible
subquotients of $J_l$ can be defined over $k$. Since  
$\Hom_G(\sigma_l, J_l)\cong \Hom_G(\sigma, J)_l$ by Lemma \ref{abstractrat}, we deduce that $\pi_l \hookrightarrow J_l$ is essential.  
\end{proof}

Let $L'$ be a finite extension of $L$ with the ring of integers $\OO'$ and residue field 
$k'$. Let $\pi$ be an absolutely irreducible $k$-representation of $G$ with a central character $\zeta$ and 
let $\wP$ be a projective envelope of $S:=\pi^{\vee}$ in $\dualcat(\OO)$. 

\begin{cor}\label{bcisok1} $\wP\otimes_{\OO} \OO'$ is a projective envelope of $S\otimes_k k'$ in 
$\dualcat(\OO')$. Moreover, $\End_{\dualcat(\OO')}(\wP\otimes_{\OO} \OO')\cong \End_{\dualcat(\OO)}(\wP)\otimes_{\OO} \OO'$.
\end{cor}
\begin{proof} It is enough to show that $\wP \otimes_{\OO} k'$ is a projective envelope of $S\otimes_k k'$ in 
$\dualcat(k')$, see Corollary \ref{projaretfree2},  as $\wP$ is $\OO$-torsion free by Corollary \ref{projaretfree}. 
Now, 
$$ \wP \otimes_{\OO} k'\cong J_{\pi}^{\vee} \otimes_k k' \cong (J_{\pi}\otimes_k k')^{\vee},$$
where $J_{\pi}$ is an injective envelope of $\pi$ in $\Mod^{\mathrm{l\, fin}}_{G, \zeta}(k)$ and the last isomorphism follows from 
Lemma \ref{abstractrat}. Since $J_{\pi}\otimes_k k'$ is an injective envelope of $\pi\otimes_k k'$ in 
$\Mod^{\mathrm{l\, fin}}_{G, \zeta}(k')$ by Corollary \ref{bcinjenv} we get the first assertion. The second assertion follows since $\OO'$ 
is free of finite rank over $\OO$.
\end{proof}

\begin{cor}\label{bcisok2} Let $S$ be as above and suppose there exists $Q$ in $\dualcat(k)$ satisfying the hypotheses 
(H1)-(H5) of  \S\ref{firstsec}. Then $Q\otimes_k k'$ and $S\otimes_k k'$ satisfy the hypotheses (H1)-(H5) in 
$\dualcat(k')$.
\end{cor}  
\begin{proof} Propositions \ref{defoverk}, \ref{bcEXT}.
\end{proof}

\begin{prop}\label{blocksoverk} Let $\pi\in \Irr_{G,\zeta}(k)$ be absolutely irreducible and let $\BB$ be the equivalence class
of $\pi$ in $\Irr_{G,\zeta}(k)$ under $\sim$. If $p\ge 5$ then one of the following holds:
\begin{itemize} 
\item[(i)] if $\pi$ is supersingular then $\BB=\{ \pi\}$;
\item[(ii)] if $\pi\cong \Indu{P}{G}{\chi_1\otimes \chi_2 \omega^{-1}}$ with $\chi_1\chi_2^{-1}\neq \Eins,
\omega^{\pm 1}$ then  
$$\BB=\{ \Indu{P}{G}{\chi_1\otimes \chi_2 \omega^{-1}}, \Indu{P}{G}{\chi_2\otimes \chi_1 \omega^{-1}}\};$$
\item[(iii)] if $\pi\cong \Indu{P}{G}{\chi\otimes\chi\omega^{-1}}$ then $\BB=\{\pi\}$;
\item[(iv)] otherwise, $\BB=\{ \eta, \Sp \otimes \eta, (\Indu{P}{G}{\alpha})\otimes\eta\}$;
\end{itemize}  
where $\eta:G\rightarrow k^{\times}$ is a smooth character.
\end{prop}
\begin{proof} If $k$ is algebraically closed this is the main result of \cite{ext2}. It follows
from Lemma \ref{fielddefi}, Propositions \ref{defoverk} and \ref{bcEXT} that the same statement is true over $k$.
\end{proof}
\begin{remar} In fact, \cite{ext2} also computes the blocks for $p=3$, the only  difference is that the block in case (iv) 
contains $4$ distinct irreducible representations, because $\Indu{P}{G}{\alpha}$ is reducible if $p=3$, and its  semi-simplification 
is isomorphic to $\omega\circ \det\oplus \Sp\otimes\omega\circ\det$. In \cite{ext_new} 
we have found a new method to compute the blocks, which also works for $p=2$. If $p=2$ then the cases (iii) and (iv) collapse to one: a block with two irreducible 
representations $\BB=\{\eta, \Sp\otimes\eta\}$.
\end{remar}

\begin{cor}\label{irrsubQabs} Let $\Pi$ be an absolutely irreducible admissible unitary $L$-Banach space representation of $G$ with a central character
$\zeta$,  let $\Theta$ be an open bounded $G$-invariant lattice in $\Pi$ and let $\pi$ be an irreducible subquotient of $\Theta\otimes_{\OO} k$. 
Then either $\pi$ is absolutely irreducible or there exists a smooth character 
$\chi:\Qp^{\times}\rightarrow \bar{k}^{\times}$ such that $l:=k[\chi(p)]$ is a quadratic extension of $k$ and 
$$\pi\otimes_k l \cong \Indu{P}{G}{\chi\otimes \chi^{\sigma}\omega^{-1}}\oplus  \Indu{P}{G}{\chi^{\sigma}\otimes \chi\omega^{-1}},$$
where $\chi^{\sigma}$ is a conjugate of $\chi$ by the non-trivial element in $\Gal(l/k)$.   
\end{cor}
\begin{proof} We observe that for every  finite extension $l$ of $k$ the irreducible subquotients of $\pi_l$ are contained in 
the same  block. Since otherwise Proposition \ref{blockdecompB} implies that  $\Pi_{L'}$ is not irreducible, where  $L'$ 
is a finite extension of $L$ with residue field $l$. The assertion follows from the description of irreducible 
$k$-representations of $G$ in Proposition \ref{irrkreps} and Proposition \ref{blocksoverk}.
\end{proof}  

It follows from Propositions \ref{blocksoverk}, \ref{irrkreps} and \ref{bcEXT} that a block  $\BB$ contains only finitely 
many isomorphism classes of irreducible $k$-representations of $G$. Fix a representative $\pi_i$ for each isomorphism 
class in $\BB$ and let $\wP_{\BB}$ be a projective envelope of $(\oplus_{i=1}^n \pi_i)^{\vee}$ then $\wE_{\BB}:=\End_{\dualcat(\OO)}(\wP_{\BB})$
is a compact ring, see \S \ref{zerosec}.

\begin{prop}\label{gabriel} The category $\Mod^{\mathrm{l\, fin}}_{G,\zeta}(\OO)^{\BB}$ is anti-equivalent to the category 
of compact right $\wE_{\BB}$-modules. The centre of $\Mod^{\mathrm{l\, fin}}_{G,\zeta}(\OO)^{\BB}$ is isomorphic to the centre 
of $\wE_{\BB}$.
\end{prop}
\begin{proof}  See \cite[\S IV.4 Thm 4, Cor 1, Cor 5]{gab}.
\end{proof}

\subsection{Representations of the torus}
Let $T$ be the subgroup of diagonal matrices in $G$, $T_0:=T\cap I$, $T_1:= T\cap I_1$.
\begin{prop}\label{Tirrfin} Every smooth irreducible $k$-representation of $T$ is finite dimensional and hence  admissible.
The absolutely irreducible representations are $1$-di\-men\-sio\-nal.
\end{prop}
\begin{proof} Let $\tau$ be an irreducible smooth $k$-representation of $T$. Since $T_1$ is a pro-$p$ group we have $\tau^{T_1}\neq 0$, 
and since $T_1$ is normal in $T$ and $\tau$ is irreducible we obtain $\tau^{T_1}=\tau$. Since $T_0/T_1$ is a finite group of 
prime to $p$ order and with all its absolutely irreducible representations defined over $\Fp$, we may find  a smooth character 
$\chi: T_0\rightarrow k^{\times}$, such that $\chi$ is a direct summand of $\tau|_{T_0}$. Since $T$ is commutative and $\tau$ is irreducible
we deduce that $\tau|_{T_0}$ is isomorphic to a direct sum of $\chi$'s, in particular any $k$-subspace of $\tau$ is $T_0$-invariant.
Choose $t_1, t_2\in T$ such that their images generate $T/T_0$ as 
a group. Let $R=k[t_1^{\pm 1}, t_2^{\pm 1}]\subset k[T]$ then any $R$-invariant subspace of $\tau$ is $T_0$-invariant and hence $T$-invariant. 
Thus $\tau$ is an irreducible $R$-module and hence is isomorphic to $R/\mm$, where $\mm$ is a maximal ideal of $R$. Since 
$R$ is just the ring of Laurent polynomials in $2$ variables, $R/\mm$ is a finite extension of $k$. Thus, $\tau$ is finite dimensional
and $R/\mm$ is an absolutely irreducible $R$-module if and only if $R/\mm\cong k$. 
\end{proof} 

\begin{cor}\label{tau12sub} 
Let $\tau_1, \tau_2$  be  smooth irreducible $k$-representations of $T$. If  $\Indu{P}{G}{\tau_1}$ and $\Indu{P}{G}{\tau_2}$ 
have an irreducible subquotient in common then  $\tau_1\cong \tau_2$. 
\end{cor}
\begin{proof} If $\tau_1$ and $\tau_2$ are absolutely irreducible then they are characters and the assertion follows from 
\cite[\S7]{bl}. In general, since $\tau_1$ and $\tau_2$ are finite dimensional, we may find a finite extension $l$ of $k$ such that 
$$
\tau_1\otimes_k l\cong \bigoplus_{\sigma\in \Gal(l/k)} \chi_1^{\sigma}, \quad \tau_2\otimes_k l\cong  \bigoplus_{\sigma\in \Gal(l/k)} \chi_2^{\sigma},$$
where $\chi_1$ and $\chi_2$ are smooth characters $T\rightarrow l^{\times}$. From the absolutely irreducible case, we deduce that
$\chi_1$ is Galois conjugate to $\chi_2$ and thus $\tau_1\otimes_k l\cong \tau_2 \otimes_k l$. Lemma \ref{abstractrat} implies 
$\tau_1\cong \tau_2$.
\end{proof}

\begin{lem}\label{strongchar} Let $\psi: T\rightarrow k^{\times}$ be a smooth character such that $\psi|_Z=\zeta$ and let 
$0\rightarrow \psi\rightarrow \epsilon \rightarrow \psi\rightarrow 0$ be a non-split extension in $\Mod^{\mathrm{sm}}_{T, \zeta}(k)$. If $p>2$ 
then $\dim \Ext^1_{T, \zeta}(\psi, \epsilon)=2$.
\end{lem}
\begin{proof} After twisting we may assume that $\psi=\zeta=\Eins$. As we have seen in Proposition \ref{projTistfree} 
the hypotheses (H0)-(H5) are satisfied for $T/Z\cong \Qp^{\times}$ and $S=Q=\Eins^{\vee}$. Moreover, the
endomorphism ring $\wE$ of the projective envelope of $\Eins^{\vee}$ is isomorphic to $\OO[[x,y]]$ and 
$E=\wE\otimes_{\OO} k\cong k[[x,y]]$ with the maximal ideal $\mm=(x,y)$. The non-split extension $\epsilon^{\vee}$, 
defines a $1$-di\-men\-sio\-nal subspace $W$ of $\mm/\mm^2$, see \S\ref{critcomm}. Without loss of generality we may assume that 
the image of $x$ is a basis of $W$. Then the image of $\{x,  y^2\}$ in $(W+\mm^2)/(W\mm+\mm^3)$ is a basis and 
and the assertion follows from Lemma \ref{AW}.
\end{proof}
\subsection{Colmez's Montreal functor}\label{CMF}

Let $\Mod_{G,Z}^{\mathrm{fin}}(\OO)$ be the full subcategory of $\Mod_{G}^{\mathrm{sm}}(\OO)$ consisting of representations of finite 
length with a central character.  Let $\Mod^{\mathrm{fin}}_{\gal}(\OO)$ be the category of continuous $\gal$-representations on $\OO$-modules of finite length with the discrete topology, 
where  $\gal$ is the absolute Galois group of $\Qp$. Colmez in \cite{colmez} has defined an exact covariant 
functor $\VV: \Mod_{G,Z}^{\mathrm{fin}}(\OO)\rightarrow \Mod^{\mathrm{fin}}_{\gal}(\OO)$. 
If $\psi: \Qp^{\times}\rightarrow \OO^{\times}$ is a continuous character, then we may also consider it as a continuous character 
$\psi: \gal\rightarrow \OO^{\times}$ via \eqref{local_class} and for all  $\pi$ in $\Mod_{G,Z}^{\mathrm{fin}}(\OO)$ we have $\VV(\pi\otimes \psi\circ \det)\cong \VV(\pi)\otimes \psi$. Moreover, 
$ \VV(\Eins)=0$, $\VV(\Sp)=\omega$, $\VV(\Indu{P}{G}{(\chi_1\otimes \chi_2\omega^{-1})})\cong \chi_2$, 
$\VV(\pi(r,0))\cong \ind_{\GG_{\mathbb Q_{p^2}}}^{\gal} \omega_2^{r+1}$, where $\omega$ is the reduction of the cyclotomic character 
modulo $p$, $\omega_2:\GG_{\mathbb Q_{p^2}}\rightarrow k^{\times}$ is a character of the absolute Galois group of an unramified quadratic extension of $\Qp$, which via local class field theory corresponds 
to the character $\mathbb Q_{p^2}^{\times}\rightarrow k^{\times}$, $x\mapsto x|x| \mod{\varpi}$, see \cite[\S VII.4]{colmez}.
In particular, the representation $\VV(\pi(r,0))$ is absolutely irreducible. 

Let $\zeta: Z\rightarrow \OO^{\times}$ be a continuous character and let $\dualcat(\OO)$ be the category dual to 
$\Mod_{G, \zeta}^{\mathrm{lfin}}(\OO)$.  Let $\Rep_{\gal}(\OO)$ be the category of continuous $\gal$-representations on compact $\OO$-modules. We define a functor $\cV: \dualcat(\OO)\rightarrow \Rep_{\gal}(\OO)$ as follows. 
Let  $M$ be in $\dualcat(\OO)$, if it is of finite length then $\cV(M):=\VV(M^{\vee})^{\vee}(\varepsilon \zeta)$, 
where $\vee$ denotes the Pontryagin dual, $\varepsilon$ the cyclotomic character and we consider $\zeta$ as a character
of $\gal$ via the class field theory.
 In general, we may write 
$M\cong \underset{\longleftarrow}{\lim}\, M_i$, where the limit is taken over all quotients in $\dualcat(\OO)$ of finite length, we define
$\cV(M):=\underset{\longleftarrow}{\lim}\, \cV(M_i)$. Since we have dualized twice, the functor  $\cV$ is covariant. 
Moreover, it preserves exactness of short exact sequences of objects of finite length. Since all the maps $M_i\rightarrow M_j$ 
in the projective system are surjective with $M_i$ and $M_j$ of finite length, we deduce that the maps $\cV(M_i)\rightarrow \cV(M_j)$
are surjective. The exactness of projective limits in $\Rep_{\gal}(\OO)$ implies that the functor $\cV$ is exact. Let 
us note that with our normalization of $\cV$ we have:
$$ \cV(\pi^{\vee})\cong \VV(\pi), \quad \cV((\Indu{P}{G}{\chi_1\otimes \chi_2 \omega^{-1}})^{\vee})= \chi_1, \quad 
\cV((\Sp\otimes \eta\circ \det)^{\vee})=\eta,$$
where $\pi$ is a supersingular representation. 

Let $\Pi$ be an admissible unitary $L$-Banach space representation of $G$ with central character $\zeta$, and let $\Theta$ be an open bounded $G$-invariant lattice in $\Pi$. Then $\Theta/\varpi^n \Theta$ is admissible-smooth representation of $G$ for all $n\ge 1$, 
and hence locally finite. It follows from \S\ref{banach} that $\Theta^d$ is an object of $\dualcat(\OO)$. Since $\Theta^d$ is $\OO$-torsion free and $\cV$ is covariant and exact, $\cV(\Theta^d)$ is $\OO$-torsion free. We let $\cV(\Pi):=\cV(\Theta^d)\otimes_{\OO} L$. Since different 
open lattices in $\Pi$ are commensurable, $\cV(\Pi)$ does not depend on the choice of $\Theta$.

\begin{lem} Let $\Pi$ and $\Theta$ be as above. If $\Theta/\varpi \Theta$ is of finite length as a $G$-representation, then let $\VV(\Theta):=\underset{\longleftarrow}{\lim}\, \VV(\Theta/\varpi^n \Theta)$, and $\VV(\Pi):=\VV(\Theta)\otimes_{\OO} L$. 
Then $\cV(\Pi)\cong \VV(\Pi)^* (\varepsilon \zeta)$, where $*$ denotes $L$-linear dual. If $\cV(\Pi)$ is $2$-dimensional and $\det \cV(\Pi)=\varepsilon \zeta$, then $\cV(\Pi)\cong \VV(\Pi)$.
\end{lem}
\begin{proof} Since $\VV$ sends irreducible representations of $G$ to finite dimensional Galois representations, and $\Theta/\varpi \Theta$ is of finite length by assumption, 
we deduce that $\VV(\Theta/\varpi \Theta)$ is a finite dimensional $k$-vector space. It follows from \cite[2.2.2]{kisin} that $\VV(\Theta/\varpi^n \Theta)$ is a flat $\OO/(\varpi^n)$-module for all $n\ge 1$ and 
hence $\VV(\Theta)$ is a free $\OO$-module of finite rank, such that $\VV(\Theta)/\varpi^n \VV(\Theta)\cong \VV(\Theta/\varpi^n \Theta)$ for all $n\ge 1$. For every $n\ge 1$ we have 
$ \Theta^d/\varpi^n \Theta^d\cong (\Theta/\varpi^n \Theta)^{\vee}$, and hence $\cV(\Theta^d)/\varpi^n \cV(\Theta^d)\cong (\VV(\Theta)/\varpi^n \VV(\Theta))^{\vee} (\varepsilon \zeta)\cong \Hom_{\OO}(\VV(\Theta), \OO/(\varpi^n))(\varepsilon \zeta)$, for all $n\ge 1$, 
which implies that $\cV(\Theta^d)\cong \Hom_{\OO}(\VV(\Theta), \OO)(\varepsilon \zeta)$, and $\cV(\Pi)\cong \VV(\Pi)^*(\varepsilon\zeta)$. The last assertion follows from the fact that 
$\bigl (\begin{smallmatrix} 0 & 1 \\-1 & 0\end{smallmatrix} \bigr)$ conjugates $A=\bigl (\begin{smallmatrix} a & b \\ c & d\end{smallmatrix} \bigr)$ to $(\det A) A^{t-1}$, where $t$ denotes the transpose.
\end{proof}

We are going to adapt an argument of Kisin \cite[\S2]{kisin}, which uses Colmez's functor to relate the  deformation theory on the $\GL_2(\Qp)$-side to the deformation theory on the Galois side.

\begin{lem}\label{VT} Let $M$ be in $\dualcat(\OO)$, let $A$ be a noetherian $\OO$-subalgebra of 
$\End_{\dualcat(\OO)}(M)$ and let $\md$ be a finitely generated $A$-module, then there exists a natural isomorphism 
$\cV(\md\otimes_A M)\cong \md \otimes_A \cV(M)$.
\end{lem}
\begin{proof} This is identical to \cite[2.2.2]{kisin}, via \cite[1.2.7]{kisin2}. We recall the argument for the sake of completeness. 
 Since $A\subset \End_{\dualcat(\OO)}(M)$ and $\cV$ is a covariant functor, $\cV(M)$ is naturally an $A$-module.
 Since $\VV$ is exact and additive, we have an isomorphism   
 $\cV(A^n\otimes_A M)\cong \cV(M^{\oplus n})\cong \cV(M)^{\oplus n}\cong A^n \otimes_A \cV(M)$.  The isomorphism, exactness of $\cV$ and the snake lemma imply that for any finitely presented $A$-module $\md$ we have an
isomorphism $\md\otimes_A \cV(M)\cong \cV(\md\otimes_A M)$. We leave it as an exercise to the reader to check that 
the isomorphism does not depend on the presentation of $\md$ and is functorial. Since $A$ is noetherian any 
finitely generated $A$-module $\md$ is also finitely presented. 
\end{proof}

\begin{lem}\label{flattofree} Let $Q$ be an object of finite length in $\dualcat(k)$. Let $A\rightarrow A'$ be a morphism in $\mathfrak A$,  let $Q_A$ be 
a deformation of $Q$ to $A$ in the sense of Definition \ref{defiDef}. Then $\cV(Q_A)$ is a flat $A$-module  and
$$A'\otimes_A \cV(Q_A)\cong \cV(A'\otimes_A Q_A).$$ 
In particular, $k\otimes_A \cV(Q_A)\cong \cV(Q)$ and $\cV(Q_A)$ is a finite free $A$-module of rank $\dim_k \cV(Q)$.
\end{lem}
\begin{proof} By definition of a deformation, $Q_A$ is $A$-flat. Hence the functor $\md\mapsto \md \otimes_A Q_A$ is exact. Since 
$\cV$ is exact, using Lemma \ref{VT} we deduce that the functor $\md \mapsto \md\otimes_A \cV(Q_A)$ is exact, so that $\cV(Q_A)$ is 
$A$-flat. The $A$-linear map $Q_A \rightarrow A'\otimes_A Q_A$, $v\mapsto 1\otimes v$ induces an $A$-linear map
$\cV(Q_A)\rightarrow \cV(A'\otimes_A Q_A)$. Since $\cV$ is  a functor, $A'$ acts on  $\cV(A'\otimes_A Q_A)$ and by the universality of the tensor product we obtain an 
$A'$-linear  map $A'\otimes_A \cV(Q_A)\rightarrow \cV(A'\otimes_A Q_A)$. This map is an isomorphism, since it follows from Lemma \ref{VT} that it is an isomorphism of $A$-modules.
Since $Q$ is of finite length and irreducible subquotients are mapped to finite dimensional $k$-vector spaces we deduce that 
$\cV(Q)$ is finite dimensional, as $k\otimes_A \cV(Q_A)\cong \cV(k\otimes_A Q_A)\cong \cV(Q)$, by Lemma \ref{VT},  we deduce the last assertion  from Nakayama's lemma
for $A$.
\end{proof} 

\begin{cor}\label{ftof1} Let $Q$ be an object of finite length in $\dualcat(k)$. The functor $\cV$ induces 
natural transformations between the deformation functors $\Def_Q\rightarrow \Def_{\cV(Q)}$, $\Def^{ab}_Q\rightarrow \Def^{ab}_{\cV(Q)}$, 
$Q_A \mapsto \cV(Q_A)$.
\end{cor}

\begin{lem}\label{VT2} Let $\wP\twoheadrightarrow S$ be a projective envelope of an absolutely irreducible object $S$  in $\dualcat(\OO)$ and let $\wE=\End_{\dualcat(\OO)}(\wP)$.
For every compact right $\wE$-module $\md$ there exists a natural isomorphism $\cV(\md\wtimes_{\wE} \wP)\cong \md\wtimes_{\wE} \cV(\wP)$.
\end{lem}
\begin{proof} Let $\{\md_i\}_{i\in I}$ be a basis of open neighbourhoods  of $0$ in $\md$, consisting of right $\wE$-modules, and let $\mathfrak a_i$ be the $\wE$-annihilator of of $\md/\md_i$. 
Since $\md_i$ is open in $\md$, the quotient is an $ \wE$-module of finite length, and the quotient topology on $\md/\md_i$ is discrete. In particular, $\mathfrak a_i$ is open. Moreover, since $\wE$ is a local ring with residue field $k$, we deduce 
that $\wE/\mathfrak a_i$ is a finite $\OO$-module, which implies that it is noetherian. Let $\wP_i$ be the closure of $\mathfrak a_i \wP$ in $\wP$, so that $\wP/\wP_i\cong \wE/\mathfrak a_i \wtimes_{\wE} \wP$. Then 
$\cV(\md/\md_i \wtimes_{\wE} \wP)\cong \cV(\md/\md_i \otimes_{\wE/\mathfrak a_i} \wP/\wP_i)\cong \md/\md_i \otimes_{\wE/\mathfrak a_i} \cV(\wP/\wP_i)\cong \md/\md_i \wtimes_{\wE}\cV(\wP)$, 
where the second isomorphism is given by Lemma \ref{VT}. Since $\cV$ and $\wtimes$ commute with projective limits and $\md\cong \underset{\longleftarrow}{\lim} \,{\md/\md_i}$, by passing to the limit we obtain
$\cV(\md\wtimes_{\wE} \wP)\cong \md\wtimes_{\wE} \cV(\wP)$.
\end{proof}

\begin{cor}\label{ftof3} Let $\wP\twoheadrightarrow S$ be a projective envelope of an absolutely irreducible object $S$  in $\dualcat(\OO)$ and let $\wE=\End_{\dualcat(\OO)}(\wP)$. 
 Let  $\Pi$ be an irreducible admissible unitary $L$-Banach space representation of $G$ with a central character 
$\zeta$ and the reduction $\overline{\Pi}$ of finite length. Suppose that $S^{\vee}$ is a subquotient of $\overline{\Pi}$ 
and let $\Xi$ be an open bounded $G$-invariant lattice in $\Pi$ such that the natural map 
$\Hom_{\dualcat(\OO)}(\wP, \Xi^d)\wtimes_{\wE} \wP\rightarrow \Xi^d$ is surjective, see  Proposition \ref{modulequotientnew} (iii). 
Then we have a surjection 
$$\Hom_{\dualcat(\OO)}(\wP, \Xi^d)\wtimes_{\wE} \cV(\wP)\cong \cV(\Hom_{\dualcat(\OO)}(\wP, \Xi^d)\wtimes_{\wE} \wP)\twoheadrightarrow\cV(\Xi^d).$$
\end{cor}
\begin{proof} We note that since $\overline{\Pi}$ is of finite length, $\Hom_{\dualcat(\OO)}(\wP, \Xi^d)$ is a finitely generated $\OO$-module
by Lemma \ref{mult=rank0}. The isomorphism follows from Lemma \ref{VT2}, the surjection from the exactness of $\cV$.
\end{proof}

\begin{cor}\label{ftof2} Let $\wP\twoheadrightarrow S$ be a projective envelope of an absolutely irreducible object $S$  in $\dualcat(\OO)$, let $\wE=\End_{\dualcat(\OO)}(\wP)$ and let 
$\wTor^i_{\wE}(\ast, \wP)$ be the $i$-th derived functor of $\wtimes_{\wE} \wP$ in the category of compact right $\wE$-modules. If 
$\SL_2(\Qp)$ acts trivially on $\wTor^1_{\wE}(k, \wP)$ then the functor $\md\mapsto \md\wtimes_{\wE} \cV(\wP)$ is exact. Moreover, if 
$k\wtimes_{\wE}\wP$ is of finite length then  $\cV(\wP)$ is a free $\wE$-module of rank equal to $\dim_k \cV(k\wtimes_{\wE}\wP)$. 
\end{cor} 
\begin{proof} Since every compact $\wE$-module can be written as a projective limit of $\wE$-modules of finite length, and $\wtimes$ commutes with projective limits, which 
are exact in the category of continuous $\gal$-representations on compact $\OO$-modules, it is enough to show that the functor $\md\mapsto \md\wtimes_{\wE} \cV(\wP)$
maps short exact sequences of continuous $\wE$-modules of finite length to short exact sequences.

If $\md$ is a continuous  $\wE$-module of finite length then by devissage we deduce that $\SL_2(\Qp)$ acts trivially on $\wTor^1_{\wE}(\md, \wP)$. 
Since $\cV$ kills every irreducible on which $\SL_2(\Qp)$ acts trivially, we obtain $\cV(\wTor^1_{\wE}(\md, \wP))=0$ for all finite length modules $\md$.  Using Lemma 
\ref{VT} we deduce that the functor $\md\mapsto \md\wtimes_{\wE} \cV(\wP)$ maps short exact sequence of continuous $\wE$-modules of finite length to 
short exact sequences. 

If $k\wtimes_{\wE} \wP$ is of finite length then $\cV(k\wtimes_{\wE} \wP)$ is a finite dimensional $k$-vector space. Since 
$k\wtimes_{\wE} \cV(\wP)\cong \cV(k\wtimes_{\wE} \wP)$ by Lemma \ref{VT}, we deduce from Nakayama's lemma that 
$\cV(\wP)$ is a free $\wE$-module of rank equal to $\dim_k \cV(k\wtimes_{\wE}\wP)$.
\end{proof}

 Let $\wP\twoheadrightarrow S$ be a projective envelope of an absolutely irreducible object $S$  in $\dualcat(\OO)$, let $\wE=\End_{\dualcat(\OO)}(\wP)$. From now on we assume 
 the existence of $Q$ in $\dualcat(k)$ of finite length, satisfying the hypothesis (H1)-(H5) in $\dualcat(k)$. Since (H0) holds in $\dualcat(\OO)$ by Corollary \ref{projaretfree}, the hypotheses (H1)-(H5) hold in $\dualcat(\OO)$ by Proposition \ref{H00}, and so $\wP$ is a flat $\wE$-module by 
 Corollary \ref{PisEflat}.

 Since $Q$ is of finite length, $\cV(Q)$ is a continuous representation of 
 $\gal$ on a finite dimensional $k$-vector space. Let $\Def^{ab, \psi}_{\cV(Q)}$ be a subfunctor of $\Def^{ab}_{\cV(Q)}$ parameterising deformations with determinant equal to 
$\psi:=\varepsilon\zeta$, 
where $\varepsilon$ is the cyclotomic character. If $\Def^{ab}_{\cV(Q)}$ is pro-representable  then so is  $\Def^{ab, \psi}_{\cV(Q)}$, 
see \cite[\S24]{mazur}, and 
we denote the 
corresponding ring by $R^{\psi}$, and the universal deformation of $\cV(Q)$ with determinant equal to $\psi$ by $\rho^{un, \psi}$. Let $\mm$ be a maximal ideal of  $R^{\psi}[1/p]$, then the residue field $\kappa(\mm)$ is a finite totally ramified extension of $L$, and the image of $R^{\psi}$ in $\kappa(\mm)$ is equal to the 
ring of integers $\OO_{\kappa(\mm)}$. 
Let $\mathfrak a$ be the intersection of 
those maximal ideals $\mm$ of $R^{\psi}[1/p]$ for which $\kappa(\mm)\otimes_{R^{\psi}} \rho^{un, \psi}$ is an absolutely irreducible $\kappa(\mm)[\gal]$-module and let $R'$ be the image of $R^{\psi}$ in 
$R^{\psi}[1/p]/\mathfrak a$.

\begin{prop}\label{ftof4} Assume that $S^{\vee}$ is not a character, so that $\cV(S)\neq 0$.
Suppose that the following hold:
\begin{itemize}
\item[(i)] $\End_{\gal}(\cV(Q))=k$; 
\item[(ii)] $\VV$ induces an injection, 
$$\Ext^1_{G,\zeta}(Q^{\vee}, Q^{\vee})\hookrightarrow \Ext^1_{k[\gal]}(\VV(Q^{\vee}), \VV(Q^{\vee}));$$
\item[(iii)] for every irreducible representation $\rho$ of $\gal$ defined over some finite totally ramified extension $L'$ of $L$
and satisfying $\det \rho=\psi$ and  $\overline{\rho}\cong \cV(Q)^{ss}$ 
there exists an open bounded $G$-invariant lattice $\Xi$ in a unitary admissible $L'$-Banach space representation $\Pi$ of $G$ such that the following hold:
\begin{itemize}
\item[a)] $\zeta$ is the central character of $\Pi$;
\item[b)]  $\overline{\Pi}$ contains $S^{\vee}$ with multiplicity $1$;
\item[c)] $\rho\cong \cV(\Xi^d)\otimes_{\OO} L$. 
\end{itemize}
\end{itemize}
Then there exists a natural surjection $\wE^{ab}\twoheadrightarrow R'$.
\end{prop}
 \begin{proof} We note that in this Proposition we allow only commutative coefficients for our deformations. In particular, 
all the rings representing different functors are commutative. 
Corollary \ref{ftof1} gives a natural transformation of functors $\Def_Q^{ab}\rightarrow \Def_{\cV(Q)}^{ab}$.
Since both functors are pro-representable we obtain a map $\varphi: R\rightarrow \wE^{ab}$, where $R$ is the ring pro-representing
$\Def_{\cV(Q)}^{ab}$. Now (ii) is equivalent to 
the assertion that $\cV$ induces an injection 
$$ \Def^{ab}_{Q}(k[\epsilon])\cong \Ext^1_{\dualcat(k)}(Q, Q)\hookrightarrow \Ext^1_{k[\gal]}(\cV(Q),\cV(Q))\cong \Def^{ab}_{\cV(Q)}(k[\epsilon]),$$
which is equivalent to the assertion that $\varphi$ induces a surjection 
$$\mm_R/(\mm_R^2+\varpi R)\twoheadrightarrow 
\wm_{ab}/(\wm_{ab}^2+\varpi \wE^{ab}).$$ 
Since both rings are complete we deduce that $\varphi: R\twoheadrightarrow \wE^{ab}$ is surjective.

Let $\mm$ be a maximal ideal of $R^{\psi}[1/p]$, such that $\kappa(\mm)\otimes_{R^{\psi}} \rho^{un, \psi}$ is absolutely irreducible. We claim that there 
exists a map of $\OO$-algebras $x: \wE\rightarrow \OO_{\kappa(\mm)}$, such that $\kappa(\mm)\otimes_{\wE} \cV(\wP)$ is isomorphic to 
$\kappa(\mm)\otimes_{R^{\psi}} \rho^{un, \psi}$ as a $\kappa(\mm)$-representation of $\gal$. Since $\cV(Q)$ has only scalar endomorphisms by (i), there exists a unique $\gal$-invariant 
$\OO_{\kappa(\mm)}$-lattice in $\kappa(\mm)\otimes_{\wE} \cV(\wP)$, which reduces to $\cV(Q)$ modulo the maximal ideal of $\OO_{\kappa(\mm)}$. 
Hence, the claim implies that $\OO_{\kappa(\mm)}\otimes_{\wE} \cV(\wP)$ and $\OO_{\kappa(\mm)}\otimes_{R^{\psi}} \rho^{un, \psi}$ define the same deformation of 
$\cV(Q)$ with determinant $\psi$. Thus the natural map $R^{\psi}\rightarrow \kappa(\mm)$ factors through $x\circ \varphi$, which implies that  the surjection $R^{\psi}\twoheadrightarrow R'$ 
factors through $\varphi: R^{\psi}\twoheadrightarrow \wE^{ab}$.

We will deduce the claim from (iii).  Let $\Xi$ and  $\Pi$ be as in (iii) with $L'=\kappa(\mm)$ and $\rho=\kappa(\mm)\otimes_{R^{\psi}} \rho^{un, \psi}$, so that 
$\rho\cong \cV(\Xi^d)\otimes_{\OO} L$.  As $\cV$ is exact, and $\Xi^d$ is $\OO$-torsion free, we deduce that $\cV(\Xi^d)$ is $\OO$-torsion free and
it follows from (iii) c) that $\cV(\Xi^d)$ is an $\OO_{\kappa(\mm)}$-lattice in $\rho$. Part (iii) b) implies that $S$ occurs as a subquotient of $k\otimes_{\OO_{\kappa(\mm)}} \Xi^d$  with multiplicity one. 
It follows from 
Lemma \ref{mult=rank0} and Corollary \ref{bcisok1}  that $\Hom_{\dualcat(\OO)}(\wP, \Xi^d)$ is  a free $\OO_{\kappa(\mm)}$-module of rank $1$.
  The action of $\wE$ gives us an $\OO$-linear map $x: \wE\rightarrow \End_{\OO_{\kappa(\mm)}}( \Hom_{\dualcat(\OO)}(\wP, \Xi^d))\cong \OO_{\kappa(\mm)}$. 
Let $C$ be cokernel of the natural map $\OO_{\kappa(\mm)}\wtimes_{\wE, x} \wP\rightarrow \Xi^d$. It follows from Lemma \ref{headS0} that 
$\Hom_{\dualcat(\OO)}(\wP, C)=0$ thus $S$ is not a subquotient of $C$ by Lemma \ref{pisub}. Since $\cV(S)\neq 0$ by assumption, and $\cV$  maps distinct irreducibles to distinct irreducibles, 
we deduce that $\cV(S)$ is not a subquotient of $\cV(C)$. Hence, the map $\cV(\OO_{\kappa(\mm)}\wtimes_{\wE, x} \wP)\rightarrow \cV(\Xi^d)$ is non-zero. Lemma \ref{VT2} 
and the irreducibility of $\rho$ implies that the induced map $\kappa(\mm)\otimes_{\wE, x} \cV(\wP)\rightarrow \rho$ is surjective. The map is an isomorphism as both $\kappa(\mm)$-vector spaces  have dimension 
equal to $\dim_k \cV(Q)$.
\end{proof}   

\subsection{The strategy in the generic case}\label{strat}
We are now in a position to explain how in the generic case the proof of the main theorem reduces to a computation 
of dimensions of some $\Ext$ groups in the category of smooth $k$-representations of $G$ with  a central character, when $p\ge 5$. 
By the generic case we mean that $Q^{\vee}$ is an \textit{atome atomorphe} in the sense of Colmez, which is either irreducible 
supersingular, so that $S=Q$, or $Q^{\vee}$ is a non-split extension of $\Indu{P}{G}{\chi_1\otimes \chi_2\omega^{-1}}$
by $\Indu{P}{G}{\chi_2\otimes \chi_1\omega^{-1}}$ with $\chi_1\chi_2^{-1}\neq \omega^{\pm 1}, \Eins$ and 
$S^{\vee}$ is a principal series representation. 

We know that the hypothesis (H0) is satisfied by Proposition \ref{projaretfree} and to verify (H1)-(H5) we only need 
to compute the dimensions of some $\Ext$ groups. Suppose that we can do this and (H1)-(H5) hold. Now $\cV(Q)$ is 
$2$-di\-men\-sio\-nal and is either irreducible or a non-split extension  of two characters $\chi_2$ by $\chi_1$.
Since $p\ge 5$ and and $\chi_1\chi_2^{-1}\neq \omega^{\pm 1}, \Eins$ the universal deformation ring $\Def^{ab}_{\cV(Q)}$ 
is representable by $R\cong \OO[[x_1, \ldots, x_5]]$ and the deformation ring with the determinant condition $R^{\psi}$
is isomorphic to $\OO[[x_1, x_2, x_3]]$. Moreover, one may show that the irreducible locus is dense, hence the ring $R'$ introduced
before Proposition \ref{ftof4} is isomorphic to $R^{\psi}$. The condition (ii) in Proposition \ref{ftof4} in 
this case is a result of Colmez \cite[VII.5.2]{colmez}, and the condition (iii) is a result of Kisin \cite[2.3.8]{kisin}. Hence, 
Proposition \ref{ftof4} gives us a surjection $\wE^{ab}\twoheadrightarrow R^{\psi}\cong\OO[[x_1, x_2, x_3]]$. One 
may calculate that $\dim \Ext^1_{G, \zeta}(Q^{\vee}, Q^{\vee})=3$  and hence $\dim \wm/ (\wm^2+\varpi \wE)=3$. If we can show that 
for every non-split extension $0\rightarrow Q^{\vee}\rightarrow \tau \rightarrow Q^{\vee}\rightarrow 0$ in $\Mod^{\mathrm{sm}}_{G, \zeta}(k)$ the dimension of $\Ext^1_{G, \zeta}(S^{\vee}, \tau)$ is at most $3$ 
then using Theorem \ref{crit} we may deduce that $\wE\cong R^{\psi}$.
In particular, $\wE$ is commutative and hence Corollary  \ref{commutativeOK} says that every absolutely irreducible admissible 
unitary $L$-Banach space representation $\Pi$ of $G$ with the central character $\zeta$ and such that $\overline{\Pi}$ 
contains $S^{\vee}$ satisfies $\overline{\Pi}\subseteq (Q^{\vee})^{ss}$.

\section{Supersingular representations}\label{supersingularreps}

In this section we carry out the strategy described in \S \ref{strat} in the supersingular case. 
The main result is Theorem \ref{mainsuper} and its Corollaries. In \S\ref{iwahoricase} we carry out 
some $\Ext$ calculations, we suggest to skip them at first reading. We assume throughout this section that $p\ge 5$. 
Let $\pi\cong \pi(r, 0, \eta)$ be a supersingular representation with a central character congruent to $\zeta$. 

\begin{prop}\label{H15} The hypotheses (H1)-(H5) of \S\ref{first} hold with $Q=S= \pi^{\vee}$. Moreover, 
$\dim \Ext^1_{\dualcat(k)}(S, S)=\dim \Ext^1_{G, \zeta}(\pi, \pi)=3.$
\end{prop} 
\begin{proof} Let $\tau$ be an irreducible representation in $\Mod^{\mathrm{sm}}_{G, \zeta}(k)$ not isomorphic to 
$\pi$, then $\Ext^1_{G, \zeta}(\tau,\pi)=0$, \cite[10.7]{ext2}. Moreover, 
$\dim \Ext^1_{G,\zeta}(\pi, \pi)=3$, \cite[10.13]{ext2}.   This implies (H3) and (H4) via Corollary \ref{thesame}, 
the rest holds trivially.
\end{proof} 

Since (H0) holds vacuously in the supersingular case, we may apply the results of \S\ref{def} and \S\ref{banach}. Let 
$\wP\twoheadrightarrow S$ be a projective envelope of $S$ in $\dualcat(\OO)$, let $\wE=\End_{\dualcat(\OO)}(\wP)$, 
$\wm$ the maximal ideal of $\wE$ and 
let $\mm$ be the maximal ideal of $\wE\otimes_{\OO} k$. We note that Proposition \ref{H15} and 
Lemma \ref{tangentspace} imply that $d:=\dim \mm/\mm^2=3$. 
Let $\rho:=\VV(\pi)$ then $\rho\cong \ind \omega_2^{r+1} \otimes \eta$
is absolutely irreducible.
We note that $\det \rho$ is congruent to $\zeta \varepsilon$, 
where $\varepsilon$ is the cyclotomic character. Let $R_{\rho}$ be the universal deformation ring of $\rho$ and 
$R_{\rho}^{\zeta\varepsilon}$ be the deformation ring of $\rho$ pro-representing a deformation problem 
with a fixed determinant equal to $\zeta\varepsilon$. 

\begin{prop}\label{kis} The functor $\cV$ induces  a surjection 
$$\wE\twoheadrightarrow R^{\zeta\varepsilon}_{\rho}\cong \OO[[x_1, x_2,x_3]].$$
\end{prop} 
\begin{proof} This is a consequence of Proposition \ref{ftof4}. We note that $\cV(S)\cong \VV(\pi)=\rho$. 
Since $p\ge 5$ using local Tate duality and Euler characteristic, we may calculate that $H^2(\gal, \Ad \rho)=0$ and 
$H^1(\gal, \Ad \rho)$ is $5$-di\-men\-sio\-nal. This implies, see \cite[\S 1.6]{mazur2}, \cite[\S 24]{mazur}, that the 
universal deformation problem $\Def^{ab}_{\cV(S)}$ is represented by $R_{\rho}\cong \OO[[x_1, \ldots, x_5]]$ and the deformation 
problem with the fixed determinant is represented by $R_{\rho}^{\zeta\varepsilon}\cong \OO[[x_1, x_2, x_3]]$. 
Since the residual representation is irreducible, the ring $R'$ in Proposition \ref{ftof4} is isomorphic to $R_{\rho}^{\zeta\varepsilon}$. Part (ii) 
of Proposition \ref{ftof4} is satisfied by \cite[VII.5.2]{colmez}, and (iii) is satisfied by \cite[2.3.8]{kisin}.
\end{proof}

\begin{prop}\label{superdone} The functor $\cV$ induces an isomorphism  $\wE\cong R^{\zeta\varepsilon}_{\rho}$. 
In particular, $\wE$ is commutative and  $\cV(\wP)$ is the universal deformation of $\rho$ with 
determinant $\zeta\varepsilon$.
\end{prop}
\begin{proof} Since $\dim \mm/\mm^2=3$ we deduce from the map in Proposition \ref{kis} induces an isomorphism 
$\wE^{ab}\cong  R^{\zeta\varepsilon}_{\rho}\cong \OO[[x_1, x_2,x_3]].$ It follows then from Lemma \ref{alter1} that the natural map
$\Hom(E, k[x]/(x^3))\rightarrow \Hom(E, k[x]/(x^2))$ is surjective. In view of Theorem \ref{crit} and Lemma \ref{enoughisenough}, it is enough to find a $2$-di\-men\-sio\-nal subspace
$V$ of $\Ext^1_{G, \zeta}(\pi, \pi)$ such that for every non-zero $\xi\in V$, representing 
an extension $0\rightarrow \pi\rightarrow E_{\xi}\rightarrow \pi\rightarrow 0$ we have 
$\dim \Ext^1_{G, \zeta}(\pi, E_{\xi})\le 3$ or equivalently $\dim \Ext^1_{G, \zeta}( E_{\xi}, \pi)\le 3$.

We have shown in  \cite[10.14]{ext2} that for any non-zero $\xi$ lying in the image of $\Ext^1_{\HH}(\II(\pi), \II(\pi))$ in 
$\Ext^1_{G, \zeta}(\pi, \pi)$ via \eqref{5T}, we have $\dim \Ext^1_{G, \zeta}(E_{\xi}, \pi)\le 3$.
In the regular case, we have  $\dim \Ext^1_{\HH}(\II(\pi), \II(\pi))=2$, \cite[Cor 6.6]{bp}, and 
so we are done. In the Iwahori case,  $\dim \Ext^1_{\HH}(\II(\pi), \II(\pi))=1$, but in  Proposition \ref{Iwahoridone} below, 
we find a two dimensional subspace $V$ in $\Ext^1_{G, \zeta}(\pi, \pi)$ such that     
for any non-zero $\xi\in V$ we have $\dim  \Ext^1_{G, \zeta}(\pi, E_{\xi})\le 3.$ Hence, $\cV$ induces an isomorphism 
of deformation functors, Corollary \ref{ftof1},  and so $\cV(\wP)$ is the universal deformation of $\rho$ with determinant 
$\zeta\varepsilon$.
\end{proof}

\begin{thm}\label{mainsuper} Let $\Pi$ be a unitary absolutely irreducible $L$-Banach space representation  
with a central character $\zeta$. Suppose that the reduction of some open $G$-invariant lattice in $\Pi$  
contains $\pi$ as a subquotient then $\overline{\Pi}\cong \pi$. 
\end{thm}
\begin{proof} Since $\wE$ is commutative the assertion follows from Corollary \ref{commutativeOK}.
\end{proof} 

Recall   that the block $\BB$ of $\pi$ consists of only one isomorphism class, Proposition \ref{blocksoverk}. Then 
$\Mod^{\mathrm{l\, fin}}_{G,\zeta}(\OO)^{\BB}$ is the full subcategory of $\Mod^{\mathrm{l\, fin}}_{G,\zeta}(\OO)$ consisting 
of representations with every irreducible subquotient isomorphic to $\pi$.

\begin{cor}\label{Csuper} The category $\Mod^{\mathrm{l\, fin}}_{G,\zeta}(\OO)^{\BB}$ is anti-equivalent to the category of compact
$R_{\rho}^{\zeta\varepsilon}$-modules. The centre of $\Mod^{\mathrm{l\, fin}}_{G,\zeta}(\OO)^{\BB}$ is naturally isomorphic to $R_{\rho}^{\zeta\varepsilon}$.
\end{cor} 
\begin{proof} The assertion follows from Proposition \ref{gabriel} and Proposition \ref{superdone}.
\end{proof}

\begin{remar}\label{tracesuper} Since $\rho$ is absolutely irreducible and $p>2=\dim \rho$, sending a deformation to its trace 
induces an isomorphism between $R_{\rho}^{\zeta \varepsilon}$ and $R^{\mathrm{ps},\zeta \varepsilon}_{\tr \rho}$, the 
deformation ring parameterizing $2$-di\-men\-sio\-nal pseudocharacters with determinant $\zeta \varepsilon$ lifting $\tr \rho$,
see \cite{Nyssen}.
\end{remar} 

\begin{cor}\label{superkill} Let $T:\gal\rightarrow  R^{\mathrm{ps},\zeta \varepsilon}_{\tr \rho}$ be the universal $2$-dimensional pseudocharacter 
with determinant $\zeta\varepsilon$ lifting $\tr \rho$. For every $N$ in $\dualcat(\OO)^{\BB}$, $\cV(N)$ is killed by 
$g^2-T(g)g + \zeta \varepsilon(g)$, for all $g\in \gal$.
\end{cor}
\begin{proof} Proposition \ref{superdone} and Remark \ref{tracesuper} imply that the assertion is true if $N=\wP$. In general,
$\cV(N)\cong \cV(\Hom_{\dualcat(\OO)}(\wP, N)\wtimes_{\wE} \wP)\cong \Hom_{\dualcat(\OO)}(\wP, N)\wtimes_{\wE}\cV(\wP)$, by Lemma \ref{VT2}.
\end{proof}

Let $\Ban^{\mathrm{adm}}_{G,\zeta}(L)^{\BB}$ be as in Proposition \ref{blockdecompB} and let $\Ban^{\mathrm{adm. fl}}_{G,\zeta}(L)^{\BB}$
be the full subcategory consisting of objects of finite length.

\begin{cor}\label{superBanfin} We have an equivalence of categories 
$$\Ban^{\mathrm{adm. fl}}_{G,\zeta}(L)^{\BB}\cong 
\bigoplus_{\nn\in \MaxSpec R_{\rho}^{\zeta \varepsilon}[1/p]}\Ban^{\mathrm{adm. fl}}_{G,\zeta}(L)^{\BB}_{\nn}.$$
The category $\Ban^{\mathrm{adm. fl}}_{G, \zeta}(L)^{\BB}_{\nn}$ is anti-equivalent to the category 
of modules of finite length of the  
$\nn$-adic completion of $R_{\rho}^{\zeta \varepsilon}[1/p]$. In particular,  $\Ban^{\mathrm{adm. fl}}_{G, \zeta}(L)^{\BB}_{\nn}$
contains only one irreducible object.
\end{cor} 
\begin{proof} Apply Theorem \ref{furtherDBan} with $\dualcat(\OO)=\dualcat(\OO)^{\BB}$.
\end{proof}

\subsection{Iwahori case}\label{iwahoricase}

Let $\pi\cong \pi(0,0,\eta)\cong \pi(p-1,0, \eta)$. In this section we identify a two dimensional subspace 
$V$ of $\Ext^1_{G, \zeta}(\pi, \pi)$ such that for any non-zero $\xi\in V$, the equivalence class of 
an extension $0\rightarrow \pi \rightarrow E_{\xi}\rightarrow \pi \rightarrow 0$, we have either 
$\dim \Ext^1_{G, \zeta}(E_{\xi}, \pi)\le 3$ or $\dim \Ext^1_{G, \zeta}(\pi, E_{\xi})\le 3$, thus 
completing the proof of Proposition \ref{superdone}. The proof involves tracking down the dimension of various  
$\Ext$ groups. Essentially the same argument should also work for $p=3$, but we have not checked the details.  


After twisting we may assume that $\eta$ is the trivial character, and so $Z$ acts trivially on $\pi$. We will 
write $\Mod^{\mathrm{sm}}_{G/Z}(k)$ instead of $\Mod^{\mathrm{sm}}_{G, \zeta}(k)$ and $\Ext^1_{G/Z}$ instead of $\Ext^1_{G,\zeta}$.  
To ease the  notation in this section we will also write $\Rep$ to mean smooth representations on $k$-vector spaces.
It follows from \cite[3.2.4, 4.1.4]{breuil1} that $\pi^{I_1}$ is $2$-di\-men\-sio\-nal. Moreover, \cite[4.1.5]{breuil1} 
implies that there exists a basis $\{v_{\Eins}, v_{\st}\}$ of $\pi^{I_1}$, such that $\Pi v_{\Eins}=v_{\st}$, 
$\Pi v_{\st}=v_{\Eins}$ and there exists an isomorphism of $K$-representations:
\begin{equation}
\langle K\centerdot v_{\Eins}\rangle \cong \Eins, \quad \langle K\centerdot v_{\st}\rangle \cong \st,
\end{equation} 
where $\st$ is the inflation of the Steinberg representation of $\GL_2(\Fp)$. In particular, $H$ acts trivially 
on $v_{\Eins}$ and $v_{\st}$. We recall the results of \cite[\S4]{ext2}. Let
\begin{equation}\label{M1st}
\begin{split}
&M_{\Eins}:= \bigl \langle \begin{pmatrix} p^{2n} & b\\ 0 & 1\end{pmatrix}  v_{\Eins}: n\ge 0, b\in \Zp
\bigr \rangle,\\ 
 &M_{\st}:=  \bigl \langle \begin{pmatrix} p^{2n} & b\\ 0 & 1\end{pmatrix}  v_{\st}: n\ge 0, b\in \Zp
\bigr \rangle.
\end{split}
\end{equation}
Then  $M_{\Eins}$, $M_{\st}$ are stable under the action of $I$, \cite[4.6]{ext2}, $M_{\Eins}^{I_1}= k v_{\Eins}$
and $M_{\st}^{I_1}=k v_{\st}$.  We set 
\begin{equation}\label{pi1st} 
\pi_{\Eins}:=M_{\Eins}+\Pi\centerdot M_{\st}, \quad \pi_{\st}:=M_{\st}+\Pi\centerdot M_{\Eins}.
\end{equation}
The subspaces $\pi_{\Eins}$ and $\pi_{\st}$ are stable under the action of $G^+$, \cite[4.12]{ext2}. Moreover,  we have
\begin{equation}\label{restpi+}
\pi|_{G^+}\cong \pi_{\Eins}\oplus \pi_{\st}. 
\end{equation}
This implies
\begin{equation}\label{indupi} 
\pi\cong \Indu{G^+}{G}{\pi_{\Eins}}\cong \Indu{G^+}{G}{\pi_{\st}}. 
\end{equation}
Further, \cite[6.4]{ext2} says that 
\begin{equation}\label{resultext2} 
\pi_{\Eins}^{I_1}=M_{\Eins} \cap \Pi M_{\st}= k v_{\Eins},\quad \pi_{\st}^{I_1}=M_{\st} \cap \Pi M_{\Eins}= k v_{\st}.
\end{equation}
The key observation that goes into the proof of this result is that the restrictions of $M_{\Eins}$ and 
$M_{\st}$ to $H(I\cap U)$ are injective envelopes of the trivial representation in $\Rep_{H(I\cap U)}$.

\begin{lem}\label{invariants} Let $N$ be a representation of $I/Z_1$ such that $N|_{I_1\cap U}$ is an injective envelope 
of the trivial representation in $\Rep_{I_1\cap U}$. Let $v\in N$ such that $H$ acts on $v$ by a character 
$\chi$ and let $\kappa:= \langle I\centerdot v\rangle$, then 
\begin{itemize}
\item[(i)] $\dim (N/\kappa)^{I_1}=1$;
\item[(ii)] $H$ acts on $(N/\kappa)^{I_1}$ by a character $\chi\alpha^{-1}$;
\item[(iii)] $\chi\alpha^{-1}\hookrightarrow (N/\kappa)|_{H(I_1\cap U)}$ is an injective envelope 
of $\chi\alpha^{-1}$ in $\Rep_{H(I_1\cap U)}$. 
\end{itemize}
\end{lem}
\begin{proof} Since $N$ is smooth and $\kappa$ is finitely generated, $\kappa$ is of finite length. We argue by 
induction on the length $\ell$ of $\kappa$. Suppose the length of $\kappa$ is $1$, then $\kappa= N^{I_1}=N^{I_1\cap U}$ and the assertion 
follows from \cite[Prop. 5.9]{ext2}. In general, let $N_1:= N/ N^{I_1}$ and let $\kappa_1$ denote the image 
of $\kappa$ in $N_1$. Now \cite[Prop. 5.9]{ext2} says that $N_1|_{I_1\cap U}$ is an injective envelope of the trivial representation
in $\Rep_{I_1\cap U}$. Since $N^{I_1}=N^{I_1\cap U}$ is $1$-di\-men\-sio\-nal and $\kappa \cap N^{I_1}=\kappa^{I_1}\neq 0$ we have 
$\ell(\kappa_1)=\ell(\kappa)-1$, and hence we get the assertion by induction.
\end{proof}

\begin{prop}\label{cutdown} Let 
\begin{equation}\label{w}
w_{\Eins}:= \sum_{\lambda, \mu\in \Fp} \lambda \begin{pmatrix} 1 & [\mu] +p [\lambda] \\ 0 & 1\end{pmatrix} t^2 v_{\Eins}, \quad 
w_{\st}:= \sum_{\lambda\in \Fp} \lambda \begin{pmatrix} 1 & [\lambda] \\ 0 & 1 \end{pmatrix} t v_{\Eins},
\end{equation}
and set $\tau:=\langle K\centerdot w_{\Eins} \rangle + \langle K\centerdot (\Pi w_{\st}) \rangle \subset \pi_{\Eins}$. 
There exist an exact non-split sequence of $K$-representations 
\begin{equation}\label{defntau}
0\rightarrow \Eins \rightarrow \tau\rightarrow \Indu{I}{K}{\alpha}\rightarrow 0.
\end{equation}
Moreover, $\Ext^1_{I/Z_1}(\Eins, \pi_{\Eins}/\tau)=0$ and $(\pi_{\Eins}/\tau)^{I_1}\cong \alpha^{-2} \oplus \alpha^2$.
\end{prop} 
\begin{proof}From \eqref{w} we get 
\begin{equation}\label{relationw}
\sum_{\mu\in \Fp} \begin{pmatrix} 1 & [\mu] \\ 0 & 1\end{pmatrix} s (\Pi w_{\st})= w_{\Eins}.  
\end{equation}
Let $\bar{\tau}$, $u_1$ and $u_2$ be the images of $\tau$, $w_{\Eins}$ and $\Pi w_{\st}$ in $\pi_{\Eins}/\Eins$, respectively. 
It follows from \cite[Lem. 6.1]{ext2} that $u_1$ and $u_2$ are $I_1$-invariant. Moreover, they are linearly independent, since
$H$ acts on $u_1$ by $\alpha^{-1}$ and on $u_2$ by $\alpha$, and these characters are distinct, as $p\ge 5$. 
Now \eqref{resultext2} implies
$\pi_{\Eins}/k v_{\Eins}\cong M_{\Eins}/k v_{\Eins} \oplus \Pi (M_{\st}/k v_{\st})$. Moreover, since the restrictions of $M_{\Eins}$ and 
$M_{\st}$ to $H(I\cap U)$ are injective envelopes of the trivial representation in $\Rep_{H(I\cap U)}$, Lemma \ref{invariants} implies that 
the space of $I_1$-invariants of  $M_{\Eins}/k v_{\Eins} \oplus \Pi (M_{\st}/k v_{\st})$ is two dimensional.
Hence, $\{u_1,u_2\}$ 
is a basis for $(\pi_{\Eins}/\Eins)^{I_1}$ and $\Pi w_{\st}\in \Pi M_{\st}$. Moreover, \eqref{relationw} implies that the natural surjection
$\Indu{I}{K}{\alpha} \twoheadrightarrow \langle K\centerdot u_2\rangle$ is injective, since it induces an injection 
on $(\Indu{I}{K}{\alpha})^{I_1}$. Thus $\bar{\tau} \cong \Indu{I}{K}{\alpha}$ and the extension 
$0\rightarrow \Eins \rightarrow \tau\rightarrow \bar{\tau}\rightarrow 0$ is non-split, since 
$\soc_{K}\tau\subseteq \soc_K \pi_{\Eins}\cong \Eins$.  Now  $s u_2$ is the image of
$$ s (\Pi w_{\st})= t w_{\st}=\sum_{\lambda\in \Fp}\lambda \begin{pmatrix} 1 & p[\lambda]\\ 0 & 1\end{pmatrix} t^2 v_{\Eins},$$
which lies in $M_{\Eins}$. Since  $\bar{\tau}= k u_2\oplus \langle I_1 \centerdot(s u_2)\rangle$ we obtain 
$$\pi_{\Eins}/\tau \cong M_{\Eins}/\langle I \centerdot(t w_{\st})\rangle \oplus \Pi  ( M_{\st}/\langle I \centerdot w_{\st}\rangle).$$  
Let $N_1:=M_{\Eins}/\langle I \centerdot(t w_{\st})\rangle$. Lemma \ref{invariants} gives that $N_1^{I_1}$ is $1$-di\-men\-sio\-nal, $H$ acts 
on $N_1^{I_1}$ by the character $\alpha^{-2}$ and $\alpha^{-2}\hookrightarrow N_1|_{H(I_1\cap U)}$ is an injective envelope 
of $\alpha^{-2}$ in $\Rep_{H(I_1\cap U)}$. Let $\psi: I\rightarrow k^{\times}$ be a smooth character, 
\cite[Prop.7.2, Cor.7.4]{ext2} say that $\Ext^1_{I/Z_1}(\psi, N_1)\neq 0$ if and only if $\psi=\alpha^{-1}$ or 
$\psi=\alpha^{-2}$. Since $p\ge 5$ we get $\Ext^1_{I/Z_1}(\Eins, N_1)= 0$. Similarly, one gets 
$\Ext^1_{I/Z_1}(\psi, \Pi (M_{\st}/\langle I \centerdot w_{\st}\rangle ))\neq 0$ if and only if $\psi=(\alpha^{-1})^{\Pi}= \alpha$ or 
$\psi=(\alpha^{-2})^{\Pi}=\alpha^2$. Again we obtain, $\Ext^1_{I/Z_1}(\Eins, \Pi (M_{\st}/\langle I \centerdot w_{\st}\rangle ))= 0$
and hence $\Ext^1_{I/Z_1}(\Eins, \pi_{\Eins}/\tau)=0$.
\end{proof}

\begin{lem}\label{dimsttriv} We have
\begin{itemize}
\item[(i)] $\Ext^1_{K/Z_1}(\st, \Eins)=\Ext^1_{K/Z_1}(\Eins, \st)=0$;
\item[(ii)] $\dim \Ext^1_{K/Z_1}(\st, \st)= 1$;
\item[(iii)] $\dim H^1(I/Z_1, \st)= 1$.
\end{itemize}
\end{lem}
\begin{proof} Since $\Eins$ and $\st$ are self dual, sending an extension to its dual induces an isomorphism 
$\Ext^1_{K/Z_1}(\st, \Eins)\cong\Ext^1_{K/Z_1}(\Eins, \st)$. So for (i) it is enough to prove 
$\Ext^1_{K/Z_1}(\st, \Eins)=0$. Let $\kappa$ be a smooth representation of $K/Z_1$, then 
\begin{equation}\label{comp2fun}
\Hom_{K/Z_1}(\st, \kappa)\cong \Hom_{K/K_1}(\st, \kappa^{K_1}),
\end{equation}
 since $K_1$ acts trivially on $\st$. Now $\st$ is a projective object in $\Rep_{K/K_1}$, \cite[\S 16.4]{serre}. Thus,
$\Hom_{K/K_1}(\st, \ast)$ is exact and we get: 
\begin{equation}\label{comp2fun1}
\Ext^1_{K/Z_1}(\st, \kappa)\cong \Hom_{K/K_1}(\st, H^1(K_1/Z_1,\kappa)).
\end{equation}
If $K_1$ acts trivially on $\kappa$ we have an isomorphism of $K$-rep\-re\-sen\-ta\-tions:
\begin{equation}\label{comp2fun2}
H^1(K_1/Z_1,\kappa)\cong \Hom(K_1/Z_1, k)\otimes \kappa\cong (\Sym^2 k^2 \otimes \dt^{-1})\otimes \kappa,
\end{equation}
see \cite[Prop 5.1]{bp}. Now $\dim \Hom_K(\st,\st \otimes \Sym^2 k^2 \otimes \det^{-1})=1$, by \cite[Prop 5.4 (ii)]{bp}
and $\Hom_K(\st,\Sym^2 k^2 \otimes \dt^{-1})=0$ as $p\ge 5$.
So we get the assertions (i) and (ii). For (iii) we observe that 
$$H^1(I/Z_1, \st)\cong \Ext^1_{K/Z_1}(\Indu{I}{K}{\Eins}, \st)\cong \Ext^1_{K/Z_1}(\Eins\oplus\st, \st).$$
\end{proof}

\begin{lem}\label{1dimst} We have $\dim \Ext^1_{I/Z_1}(\st, \alpha)=1$. The natural map 
\begin{equation}\label{shoot}
\Ext^1_{I/Z_1}( \st, \alpha)\rightarrow \Ext^1_{(I\cap P)/Z_1}(\st^{I}, \alpha)
\end{equation}
is an isomorphism. 
\end{lem}
\begin{proof} Since $\Indu{I}{K}{\Eins}\cong \Eins \oplus \st$, we have an isomorphism
$\st|_I \cong \Indu{HK_1}{I}{\Eins}$ and hence 
\begin{equation}\label{isoextst}
\Ext^1_{I/Z_1}( \st, \alpha)\cong \Ext^1_{HK_1/Z_1}(\Eins, \alpha)\cong H^1(K_1/Z_1,\alpha)^H,
\end{equation}
which is one dimensional, see the proof of \cite[Prop. 5.4]{ext2}. We identify $H^1(K_1/Z_1,\alpha)$ 
with $\Hom(K_1/Z_1, k)$, then $\Ext^1_{I/Z_1}(\st, \alpha)$ is 
identified with the subspace  generated by $\kappa$, where
$$ \kappa(g)= (bp^{-1}) \pmod{p}, \quad \forall g=\bigl( \begin{smallmatrix} a & b \\ c & d \end{smallmatrix}\bigl )\in K_1.$$
Let $0\rightarrow \alpha \rightarrow E \rightarrow \st\rightarrow 0$ be the unique non-split extension and let $v$ be the basis
vector of $\alpha$. For each coset $c \in I/HK_1$ fix a coset representative $\overline{c}$ of the form 
$\bigl( \begin{smallmatrix} 1 & [\lambda] \\ 0 & 1 \end{smallmatrix}\bigl )$. We note that given $g\in I_1$ 
the element $\overline{gc}^{-1}g \overline{c}$ lies in $K_1$.
The isomorphism $\st|_I\cong \Indu{HK_1}{I}{\Eins}$ and \eqref{isoextst} imply
that there exists $w\in E$ such that the image of $\{\overline{c} w : c\in I/HK_1\}$ is a basis of 
$\st$ and for all $g\in K_1$ we have $g w = w + \kappa(g)v $.
 Let $w_1=\sum_c \overline{c} w$, then the image of $w_1$ in $\st$ spans $\st^{I}$. We have 
\begin{equation}
\begin{split}
g w_1=g \sum_c \overline{c} w= \sum_c  \overline{gc} (\overline{gc}^{-1}  g  \overline{c})w &=w_1 +\sum_c 
\kappa(\overline{gc}^{-1}  g  \overline{c})v\\
&= w_1 + \kappa(\prod_c (\overline{gc}^{-1}  g  \overline{c})) v.
\end{split}
\end{equation}
If $g=\bigl( \begin{smallmatrix} 1 & 1 \\ 0 & 1 \end{smallmatrix}\bigl )$ then $g$ commutes with 
$\overline{c}$, 
and so $\prod_c (\overline{gc}^{-1}  g  \overline{c})= g^p$, thus $g w_1= w_1 + v$. This implies that the map \eqref{shoot}
is non-zero. Moreover, the target is $1$-di\-men\-sio\-nal by \cite[Prop. 5.4]{ext2}, hence \eqref{shoot} is an isomorphism.   
\end{proof}

\begin{cor}\label{pitau} $\dim \Ext^1_{K/Z_1}(\sigma, \pi_{\Eins})=1$ 
and $\Ext^1_{K/Z_1}(\sigma, \pi_{\Eins}/\tau)=0$, where $\sigma=\Eins$ or $\sigma=\st$.
\end{cor}
\begin{proof} It follows from \ref{cutdown} that 
$$\Ext^1_{K/Z_1}(\Indu{I}{K}{\Eins}, \pi_{\Eins}/\tau)\cong \Ext^1_{I/Z_1}(\Eins, \pi_{\Eins}/\tau)=0, $$
$$\Hom_{K/Z_1}(\Indu{I}{K}{\Eins}, \pi_{\Eins}/\tau)\cong \Hom_{I/Z_1}(\Eins, \pi_{\Eins}/\tau)=0.$$
Since $\Indu{I}{K}{\Eins}\cong \Eins\oplus \st$, we get $\Ext^1_{K/Z_1}(\sigma, \pi_{\Eins}/\tau)=0$ and 
$\Hom_K(\sigma, \pi_{\Eins}/\tau)=0$. Thus, applying $\Hom_K(\sigma,\ast)$ to the exact sequence
$$0\rightarrow \Indu{I}{K}{\alpha} \rightarrow \pi_{\Eins}/\Eins \rightarrow \pi_{\Eins}/\tau\rightarrow 0$$  
we obtain
$$\Ext^1_{K/Z_1}(\sigma, \pi_{\Eins}/\Eins) \cong \Ext^1_{K/Z_1}(\sigma, \Indu{I}{K}{\alpha})\cong 
\Ext^1_{I/Z_1}(\sigma, \alpha).$$
Now $\Ext^1_{I/Z_1}(\st, \alpha)$ is $1$-di\-men\-sio\-nal by Lemma \ref{1dimst} and $\Ext^1_{I/Z_1}(\Eins, \alpha)$ is
$1$-di\-men\-sio\-nal by \cite[Prop. 5.4]{ext2}. Hence, 
$$\Ext^1_{I/Z_1}(\Eins, \pi_{\Eins}/\Eins)\cong \Ext^1_{K/Z_1}(\Eins\oplus \st, \pi_{\Eins}/\Eins)$$
is $2$-di\-men\-sio\-nal. We know that $\Ext^1_{I/Z_1}(\Eins, \Eins)=0$, \cite[Prop. 5.4]{ext2}, so 
\begin{equation}\label{map}
\Ext^1_{I/Z_1}(\Eins, \pi_{\Eins})\rightarrow\Ext^1_{I/Z_1}(\Eins, \pi_{\Eins}/\Eins)
\end{equation}
is an injection. The source has dimension $2$ by \cite[Thm. 7.9]{ext2}. Hence, 
\eqref{map} is an isomorphism. Using $\Indu{I}{K}{\Eins}\cong \Eins\oplus \st$ again we get 
that $\Ext^1_{K/Z_1}(\st, \pi_{\Eins})$ and $\Ext^1_{K/Z_1}(\Eins, \pi_{\Eins})$ are both $1$-di\-men\-sio\-nal.
\end{proof}

\begin{lem} We have exact sequences of $G^+$-representations:
\begin{equation}\label{tree1}
0\rightarrow \cIndu{K^{\Pi}Z}{G^+}{\Eins} \rightarrow \cIndu{KZ}{G^+}{\Eins} \rightarrow \pi_{\Eins}\rightarrow 0
\end{equation}
\begin{equation}\label{tree2}
0\rightarrow \cIndu{K^{\Pi}Z}{G^+}{\st^{\Pi}} \rightarrow \cIndu{KZ}{G^+}{\st} \rightarrow \pi_{\st}\rightarrow 0.
\end{equation}
\end{lem}

\begin{proof} Below we let ($\sigma=\st$ and $\check{\sigma} =\Eins$) or ($\check{\sigma}=\st$ and $\sigma =\Eins$).
Let 
$$\FF^+:=\{f\in \cIndu{KZ}{G}{\sigma}: \supp f\subseteq G^+\}\cong \cIndu{KZ}{G^+}{\sigma},$$
$$\FF^-:=\{f\in \cIndu{KZ}{G}{\sigma}: \supp f\subseteq \Pi G^+\}\cong \cIndu{K^{\Pi}Z}{G^+}{\sigma^{\Pi}}.$$
We have  $\cIndu{KZ}{G}{\sigma}|_{G^+}\cong \FF^+\oplus \FF^-$. Let $\varphi\in\cIndu{KZ}{G}{\sigma}$, such 
that $\supp \varphi=KZ$ and $\varphi(1)$ spans $\sigma^{I_1}$. Then $\FF^+=\langle G^+\centerdot \varphi\rangle$ and 
$\FF^-= \langle G^+\centerdot \Pi\varphi\rangle$. It follows from Lemma \ref{TT} that $T\varphi \in \FF^{-}$ and 
$T \Pi \varphi =\Pi T\varphi \in \FF^{+}$. Hence, $T(\FF^+)\subset \FF^-$ and $T(\FF^-) \subset \FF^+$. Hence,
\begin{equation}\label{piplusminus}
\pi|_{G^+}\cong \FF^+/T(\FF^-)\oplus \FF^-/T(\FF^+).
\end{equation}
Now \eqref{resultext2} implies that $\Hom_K(\check{\sigma}, \pi_{\sigma})=0$, and thus \eqref{restpi+} and \eqref{piplusminus}
give $\pi_{\sigma}\cong \FF^+/T(\FF^-)$, $\pi_{\check{\sigma}}\cong \FF^-/T(\FF^+)$. Since $T$ is an injection we obtain the result. 
\end{proof}

\begin{prop}\label{extst1st1} We have 
\begin{itemize} 
\item[(i)] $\dim \Ext^1_{G^+/Z}(\pi_{\st}, \pi_{\Eins})=\dim \Ext^1_{G^+/Z}(\pi_{\Eins}, \pi_{\st})=2$;
\item[(ii)]$\dim \Ext^1_{G^+/Z}(\pi_{\st}, \pi_{\st})=\dim \Ext^1_{G^+/Z}(\pi_{\Eins}, \pi_{\Eins})=1$;
\end{itemize} 
\end{prop}

\begin{proof} By applying $\Hom_{G^+/Z}(\ast, \pi_{\Eins})$ to \eqref{tree1} we get an exact sequence 
$$\Hom_{K^{\Pi}}(\Eins, \pi_{\Eins})\rightarrow \Ext^1_{G^{+}/Z}(\pi_{\Eins}, \pi_{\Eins})\rightarrow  
\Ext^1_{K/Z_1}(\Eins, \pi_{\Eins}).$$
It follows from \eqref{M1st} that  $\pi_{\Eins}^{\Pi}\cong \pi_{\st}$, hence 
$$\Hom_{K^{\Pi}}(\Eins, \pi_{\Eins})\cong \Hom_{K}(\Eins, \pi_{\st})=0$$ and 
so $\dim \Ext^1_{G^{+}/Z}(\pi_{\Eins}, \pi_{\Eins})\le 1$. Similarly, by applying $\Hom_{G^+/Z}(\ast, \pi_{\Eins})$ to 
\eqref{tree2}
 we obtain $\dim \Ext^1_{G^+/Z}(\pi_{\st}, \pi_{\Eins})\le 2$. On the other hand $\pi\cong \Indu{G^+}{G}{\pi_{\Eins}}$ and 
 we know that
$$\Ext^1_{G/Z}(\pi, \pi)\cong \Ext^1_{G^+/Z}(\pi_{\Eins}\oplus \pi_{\st}, \pi_{\Eins})$$
is $3$-di\-men\-sio\-nal, \cite[10.13]{ext2}. Hence, both inequalities are in fact equalities. We obtain the rest 
by using $\pi_{\st}=\pi_{\Eins}^{\Pi}$ and $\pi_{\Eins}=\pi_{\st}^{\Pi}$.
\end{proof}  

\begin{cor}\label{extstpist} We have $\dim \Ext^1_{K/Z_1}(\st, \pi_{\st})=\dim \Ext^1_{K/Z_1}(\Eins, \pi_{\st})=1$.
\end{cor}
\begin{proof} Applying $\Hom_{G/Z}(\ast, \pi_{\st})$ to \eqref{tree1} we get an exact sequence:
$$\Hom_{K^{\Pi}}(\Eins, \pi_{\st})\hookrightarrow \Ext^1_{G/Z}(\pi_{\Eins}, \pi_{\st})\rightarrow \Ext^1_{K/Z_1}(\Eins, \pi_{\st}).$$
Proposition \ref{extst1st1} implies that $\dim \Ext^1_{K/Z_1}(\Eins, \pi_{\st})\ge 1$. We apply 
$\Hom_{K/Z_1}(\ast, \pi_{\st})$ to \eqref{tree2} we get an injection
 $$\Ext^1_{K/Z_1}(\pi_{\st}, \pi_{\st})\hookrightarrow \Ext^1_{K/Z_1}(\st, \pi_{\st}).$$
Proposition \ref{extst1st1} implies that $\dim \Ext^1_{K/Z_1}(\st, \pi_{\st})\ge 1$. We know \cite[7.9]{ext2} that 
$$\Ext^1_{I/Z_1}(\Eins, \pi_{\st})\cong \Ext^1_{K/Z_1}(\Indu{I}{K}{\Eins}, \pi_{\st})\cong \Ext^1_{K/Z_1}(\Eins\oplus \st, \pi_{\st})$$
is $2$-di\-men\-sio\-nal. This implies that both inequalities must be equalities.
\end{proof}

\begin{lem} Let $E$ be the unique non-split extension $0\rightarrow \Eins\rightarrow E\rightarrow \Eins\rightarrow 0$
of $(I\cap P)/Z_1$-representations. Then 
the natural map 
\begin{equation}\label{zero}
\Ext^1_{(I\cap P)/Z_1}(E, \alpha)\rightarrow \Ext^1_{(I\cap P)/Z_1}(\Eins, \alpha)
\end{equation}
 is zero.
\end{lem}

\begin{proof} We know that $\Ext^1_{(I\cap P)/Z_1}(\Eins, \Eins)\cong \Hom((I\cap P)/Z_1, k)$ is one dimensional 
and we may choose a basis 
$\{w_1, w_2\}$ of $E$ such that $w_2$ is fixed by $I\cap P$ and $d w_1= w_2+ w_1$, $uw_1=w_1$, where
$d=\bigl ( \begin{smallmatrix} 1+p & 0 \\ 0 & 1\end{smallmatrix}\bigr)$ and 
$u=\bigl (\begin{smallmatrix} 1 & 1 \\ 0 & 1\end{smallmatrix}\bigr )$, \cite[5.7]{ext2}. Suppose that the map is non-zero, 
then we have an extension $0\rightarrow \alpha \rightarrow E'\rightarrow E\rightarrow 0$. 
Since $\Ext^1_{(I\cap P)/Z_1}(\Eins, \alpha)$ is one dimensional, \cite[5.7]{ext2}, we may choose a basis 
$\{v_1,v_2, v_3\}$ of $E'$ such that $I$ acts by $\alpha$ on $v_3$, $v_2$ maps to $w_2$, $dv_2=  v_2$, 
$u v_2= v_2 + v_3$, $v_1$ maps to $w_1$ and $H$ act trivially on $v_1$ and $v_2$. Now $H$ act trivially 
on $(d-1) v_1$, hence $(d-1)v_1 = \lambda v_1 + \mu v_2$. By considering the image in $E$ we get 
$ dv_1= v_1 + v_2$. The image of $(u-1) v_1$ is zero in $E$. Hence $(u-1)v_1= \lambda v_3$, for some $\lambda$, 
$(u^p-1)v_1= (u-1)^p v_1= 0$ and so $ u^{p+1}v_1= u v_1= v_1 + \lambda v_3$. Now 
$$ du v_1= d ( v_1 + \lambda v_2)= v_1 + v_2 +\lambda v_3,$$ 
$$ u^{p+1} d v_1= u^{p+1}(v_1 + v_2)= v_1 +\lambda v_3 + v_2 + v_3.$$
Since $d u = u^{p+1} d$ in $I\cap P$ we deduce that $E'$ cannot exist.
\end{proof}  

\begin{lem}\label{cocycle} Let $\ee$ be the unique non-split extension 
$0\rightarrow \st\rightarrow \ee\rightarrow \st\rightarrow 0$   
of $K/Z_1$-representations. Then $\ee^{I\cap U}$ is the unique non-split 
extension $0\rightarrow \Eins\rightarrow \ee^{I\cap U}\rightarrow 
\Eins\rightarrow 0$ of $(I\cap P)/Z_1$-representations. 
\end{lem}  
\begin{proof} By taking $I$-invariants we obtain an exact sequence:
$$ \st^{I}\hookrightarrow \ee^I\rightarrow \st^I \overset{\partial}{\rightarrow} H^1(I/Z_1, \st).$$
Now $\dim \ee^I=1$, since $\Hom_I(\Eins, \ee)\cong \Hom_K (\st, \ee)$ by Frobenius reciprocity.
Hence, $\partial$ is an injection. Since by Lemma \ref{dimsttriv} $\dim H^1(I/Z_1, \st)=1$ we get that 
$\partial$ is an isomorphism. Fix a non-zero $v\in \st^I=\ee^I$. To prove the assertion it is enough to give a $1$-cocycle 
$f: I/Z_1\rightarrow \st$, such that for some scalar $\lambda\neq 0$ we have
\begin{equation}\label{condition}
f(\bigl(\begin{smallmatrix} a & b \\ 0 & d \end{smallmatrix}\bigr ))= \frac{a-d}{p}\lambda v, \quad 
\forall \bigl(\begin{smallmatrix} a & b \\ 0 & d \end{smallmatrix}\bigr )\in I_1\cap P.
\end{equation}
Since then there exists $w\in \ee$ such that $(g-1) w= f(g)$, for all $g\in I$, as $\partial$ is an isomorphism. Then  
\eqref{condition} implies that $w\in \ee^{I\cap U}$ and $w\not\in \ee^{I\cap P}$. Thus $v$ and $w$ are linearly independent, 
and so $\dim \ee^{I\cap U}\ge 2$. Since 
$\st^I=\st^{I\cap U}$ is $1$-di\-men\-sio\-nal we get that  $ \ee^{I\cap U}= \langle v , w\rangle$. 
thus we have an exact sequence $0\rightarrow \st^I\rightarrow \ee^{I\cap U}\rightarrow \st^I\rightarrow 0$ and \eqref{condition}
implies that this sequence is non-split.  

We will construct  a cocycle $f$ satisfying \eqref{condition}. We have $\Zp^{\times}\cong \mu_{p-1} \times (1+p\Zp)$, let 
$\pr: \Zp^{\times}\rightarrow 1+p\Zp$ denote the projection  and let 
$\delta: K\rightarrow 1+p\Zp$ be the character $\delta(g)=\pr(\det(g))$. Let 
$M:= \Sym^{p-1} \Zp^2 \otimes \delta^{\frac{1-p}{2}}$, then $K$ acts on $M$, $Z_1$ acts trivially 
and $M/pM\cong \st$. We have an exact sequence of $\Zp/p^2 \Zp[K]$-modules 
$$ 0\rightarrow M/pM\overset{p}{\rightarrow} M/p^2 M \rightarrow M/pM \rightarrow 0.$$
Let $w:= x^{p-1} +p^2M \in M/p^2 M$, then the image of $w$ in $M/pM$ is $I$-invariant. Thus
$f(g):= (g-1)w$ takes values in $pM/p^2M\cong \st$ for all $g\in I$. Moreover, it is 
immediate that $f$ satisfies \eqref{condition} with $v= x^{p-1}+ pM$ and $\lambda=\frac{p-1}{2}$. 
  
\end{proof}

\begin{lem}\label{die} Let $\ee$ be the unique non-split extension $0\rightarrow \st\rightarrow \ee\rightarrow \st\rightarrow 0$   
of $K/Z_1$-representations. Then the natural map
\begin{equation}\label{zeroagain}
\Ext^1_{I/Z_1} (\ee, \alpha)\rightarrow \Ext^1_{I/Z_1}(\st, \alpha)
\end{equation}
is zero.
\end{lem}  
\begin{proof} Lemma \ref{cocycle} gives a commutative diagram
\begin{displaymath} 
\xymatrix@1{ 0\ar[r] & \st^{I}\ar[r]\ar@{^(->}[d]& \ee^{I\cap U}\ar[r]\ar@{^(->}[d] & \st^{I}\ar[r]\ar@{^(->}[d] & 0\\
           0\ar[r] & \st\ar[r]& \ee \ar[r] & \st\ar[r] & 0}
\end{displaymath}
with exact rows, and the top row a non-split extension of $I\cap P$-representations, with the middle vertical arrow 
$I\cap P$-equivariant. Applying $\Hom_{I/Z_1}(\ast, \alpha)$ to the bottom row, and $\Hom_{(I\cap P)/Z_1}(\ast, \alpha)$
to the top row we obtain a commutative diagram: 
\begin{displaymath}
\xymatrix@1{ \Ext^1_{I/Z_1}(\ee, \alpha)\ar[d]\ar[r]^{\beta} & \Ext^1_{I/Z_1}(\st, \alpha)\ar[d]^{\cong}_{\eqref{shoot}}\\
             \Ext^1_{(I\cap P)/Z_1}(\ee^{I\cap U}, \alpha) \ar[r]^{0}_{\eqref{zero}} &  \Ext^1_{(I\cap P)/Z_1}(\st^I, \alpha).}
\end{displaymath}
It follows from the diagram that $\beta$ is the zero map.
\end{proof}

\begin{lem}\label{peekoutst} Let $\ee$ be the unique non-split extension $0\rightarrow \st\rightarrow \ee\rightarrow \st\rightarrow 0$   
of $K/Z_1$-representations, then $\dim \Ext^1_{K/Z_1}(\ee, \pi_{\Eins})=1$. Moreover, let 
$$0\rightarrow \pi_{\Eins}\rightarrow E'_1\rightarrow \ee\rightarrow 0$$
be an exact sequence of $K/Z_1$-representations, then $\Hom_K(\st, E'_1)\neq 0$.
\end{lem}
\begin{proof} Applying $\Hom_{K/Z_1}(\ast, \pi_{\Eins})$ to $0\rightarrow \st\rightarrow \ee\rightarrow \st\rightarrow 0$ 
we get an exact sequence
$$\Ext^1_{K/Z_1}(\st, \pi_{\Eins})\hookrightarrow \Ext^1_{K/Z_1}(\ee, \pi_{\Eins})\overset{\beta}{\rightarrow} 
\Ext^1_{K/Z_1}(\st, \pi_{\Eins}).$$ 
We claim that $\beta$ is zero. The claim and Lemma \ref{1dimst} gives $\dim \Ext^1_{K/Z_1}(\ee, \pi_{\Eins})=1$. 
The Yoneda interpretation of the claim  gives the second assertion. Let $\tau\subset \pi_{\Eins}$ be the representation 
considered  in Proposition \ref{cutdown}, then $\Hom_K(\st, \pi_{\Eins}/\tau)=0$, hence
$\Hom_{K}(\ee, \pi_{\Eins}/\tau)=0$. Moreover, Corollary \ref{pitau} says that $\Ext^1_{K/Z_1}(\st, \pi_{\Eins}/\tau)=0$, 
this implies
$\Ext^1_{K/Z_1}(\ee, \pi_{\Eins}/\tau)=0$. Hence, we have a commutative diagram:
\begin{displaymath}\label{squareone}
\xymatrix@1{ \Ext^1_{K/Z_1}(\ee, \tau)\ar[d]^{\cong}\ar[r]^{\gamma} & \Ext^1_{K/Z_1}(\st, \tau)\ar[d]^{\cong}\\
             \Ext^1_{K/Z_1}(\ee, \pi_{\Eins}) \ar[r]^{\beta} &  \Ext^1_{K/Z_1}(\st, \pi_{\Eins}).}
\end{displaymath}
Recall that \eqref{defntau} is an exact sequence $0\rightarrow \Eins\rightarrow \tau\rightarrow \Indu{I}{K}{\alpha}\rightarrow 0$ 
of $K/Z_1$-rep\-re\-sen\-ta\-tions. Lemma \ref{dimsttriv} says that $\Ext^1_{K/Z_1}(\st, \Eins)=0$. This implies  
$\Ext^1_{K/Z_1}(\ee, \Eins)=0$. 
Thus applying $\Hom_{K/Z_1}(\ee, \ast)$ and   $\Hom_{K/Z_1}(\st, \ast)$ to \eqref{defntau} we get a commutative diagram: 
\begin{displaymath}
\xymatrix@1{ \Ext^1_{K/Z_1}(\ee, \tau)\ar@{^(->}[d]\ar[r]^{\gamma} & \Ext^1_{K/Z_1}(\st, \tau)\ar@{^(->}[d]\\
             \Ext^1_{K/Z_1}(\ee, \Indu{I}{K}{\alpha}) \ar[r]^{\delta} &  \Ext^1_{K/Z_1}(\st, \Indu{I}{K}{\alpha}).}
\end{displaymath}     
Now $\delta$ is zero by Shapiro's lemma and Lemma \ref{die}, hence $\gamma=\beta=0$. 
\end{proof}

\begin{prop}\label{st0} Let 
\begin{equation}\label{defne1}
0\rightarrow \pi_{\Eins}\rightarrow E_1 \rightarrow \pi_{\st}\rightarrow 0
\end{equation}
 be a non-split 
extension of $G^{+}/Z$ representations. Suppose that $E_1^{I_1}$ is $1$-di\-men\-sio\-nal, then $\Ext^1_{K/Z_1}(\st, E_1)=0$.
\end{prop}
\begin{proof} We note that the assumption $\dim E_1^{I_1}=1$, implies that $E_1^{I_1}=\pi_{\Eins}^{I_1}$ and hence 
$\Hom_{K}(\st, E_1)=0$. Let $0\rightarrow \st\rightarrow \ee\rightarrow \st\rightarrow 0$ be the unique non-split extension
of $K/Z_1$-representations. Now, $\ee$ cannot be a subrepresentation of $\pi_{\st}$, since in that case by pulling back we 
would obtain 
a subrepresentation $E'_1\subset E_1$ such that we have an exact sequence
$0\rightarrow \pi_{\Eins}\rightarrow E'_1\rightarrow \ee\rightarrow 0$
of $K/Z_1$-representations with $\Hom_K(\st, E'_1)=\Hom_K(\st, E_1)=0$, which would contradict Lemma \ref{peekoutst}. 
Hence, $\Hom_K(\st, \pi_{\st}/\st)=0$ and so we obtain an injection 
\begin{equation}\label{stpist}
\Ext^1_{K/Z_1}(\st, \st)\hookrightarrow \Ext^1_{K/Z_1}(\st, \pi_{\st})
\end{equation}
 Corollary \ref{extstpist} asserts $\dim  \Ext^1_{K/Z_1}(\st, \pi_{\st})=1$, 
so the map of \eqref{stpist} is an isomorphism. The Yoneda interpretation of 
this says if we  let 
\begin{equation}
0\rightarrow \pi_{\st}\rightarrow E_2\rightarrow \st\rightarrow 0
\end{equation}
be the unique non-split extension of $K/Z_1$-representations, then $E_2$ contains $\ee$ as a subrepresentation. 
Corollary \ref{pitau} says that $\dim \Ext^1_{K/Z_1}(\st, \pi_{\Eins})=1$ so  applying 
$\Hom_{K/Z_1}(\st, \ast)$ to \eqref{defne1} gives us an injection
\begin{equation}\label{giveitome}
\Ext^1_{K/Z_1}(\st, E_1)\hookrightarrow \Ext^1_{K/Z_1}(\st, \pi_{\st}).
\end{equation}
Suppose that $\Ext^1_{K/Z_1}(\st, E_1)\neq 0$ then \eqref{giveitome} would give  a non-split extension of $K/Z_1$-representations
\begin{equation}\label{giveitome2}
0\rightarrow E_1 \rightarrow E_3\rightarrow \st\rightarrow 0
\end{equation}
such that $E_3/\pi_{\Eins}\cong E_2$. Now, $\Hom_K(\st, E_1)=0$ and so  $\Hom_K(\st, E_3)=0$, as otherwise we would 
obtain a splitting of \eqref{giveitome2}. As $\ee$ is a subrepresentation of $E_2$, by pulling back we obtain a subrepresentation 
$E'_3\subset E_3$, which sits in an exact sequence:
$$0\rightarrow \pi_{\Eins}\rightarrow E'_3\rightarrow \ee\rightarrow 0.$$
Since $\Hom_K(\st, E_3)=0$, we have $\Hom_K(\st, E'_3)=0$. Hence, we obtain a contradiction
to Lemma \ref{peekoutst}.
\end{proof}

\begin{cor}\label{conseq} Let $E_1$ be as above, then $\dim H^1(I_1/Z_1, E_1)\le 2$.
\end{cor}
\begin{proof} Taking $I_1/Z_1$-invariants of \eqref{defne1} gives us an exact sequence:
\begin{equation}\label{miau}
 \pi_{\st}^{I_1}\hookrightarrow H^1(I_1/Z_1, \pi_{\Eins})\rightarrow H^1(I_1/Z_1, E_1)\rightarrow H^1(I_1/Z_1, \pi_{\st})
\end{equation}
By \cite[7.9]{ext2} $H$ acts trivially on $H^1(I_1/Z_1, \pi_{\Eins})$ and $H^1(I_1/Z_1, \pi_{\st})$. Hence,
\begin{equation}\label{longiso}
\begin{split}
H^1&(I_1/Z_1, E_1) \cong H^1(I_1/Z_1, E_1)^H \cong \Ext^1_{I/Z_1}(\Eins, E_1)\\
 &\cong \Ext^1_{K/Z_1}(\Indu{I}{K}{\Eins}, E_1)\cong \Ext^1_{K/Z_1}(\Eins, E_1)\oplus \Ext^1_{K/Z_1}(\st, E_1)
\end{split}
\end{equation}
 Application of $\Hom_{K/Z_1}(\Eins, \ast)$ to 
\eqref{defne1} gives an exact sequence:
\begin{equation}\label{miau2}
\Ext^1_{K/Z_1}(\Eins, \pi_{\Eins})\hookrightarrow  \Ext^1_{K/Z_1}(\Eins, E_1)\rightarrow \Ext^1_{K/Z_1}(\Eins, \pi_{\st}).
\end{equation}
It follows from \eqref{miau2} and  Corollaries \ref{pitau}, \ref{extstpist} that $\dim   \Ext^1_{K/Z_1}(\Eins, E_1)\le 2$. 
Proposition \ref{st0} says $\Ext^1_{K/Z_1}(\st, E_1)=0$, hence \eqref{longiso} gives us the assertion.
\end{proof}

Recall that $\pi\cong \Indu{G^+}{G}{\pi_{\Eins}}\cong \Indu{G^+}{G}{\pi_{\st}}$. Thus we have an injection
\begin{equation}\label{niceinjection}
\Ext^1_{G^+/Z}(\pi_{\st}, \pi_{\Eins})\hookrightarrow \Ext^1_{G^+/Z}(\pi_{\st}\oplus \pi_{\Eins}, \pi_{\Eins})\cong \Ext^1_{G/Z}(\pi, \pi).
\end{equation}

\begin{prop}\label{Iwahoridone} Let $\xi$ lie in the image of \eqref{niceinjection}. Suppose that $\xi\neq 0$ 
then either $\dim \Ext^1_{G/Z}(\pi, E_{\xi})\le 3$ or $\dim \Ext^1_{G/Z}(E_{\xi}, \pi)\le 3$, 
where $E_{\xi}$ is the corresponding extension of $\pi$ by $\pi$.
\end{prop} 
\begin{proof} Since $\xi$ lies in the image of \eqref{niceinjection}
we have $E_{\xi}\cong \Indu{G^+}{G}{E_1}$, where $E_1$ is an extension of $G^+/Z$-representations:
$0\rightarrow \pi_{\Eins}\rightarrow E_1\rightarrow \pi_{\st}\rightarrow 0.$
Moreover, $\xi\neq 0$ implies that $E_1$ is non-split. If $\dim E_1^{I_1}=2$ then $\dim E_{\xi}^{I_1}=4$ and 
hence $\xi$ lies in the image 
of $\Ext^1_{\HH}(\pi^{I_1}, \pi^{I_1})\hookrightarrow \Ext^1_{G/Z}(\pi, \pi)$ via \eqref{5T} and 
we know that the assertion is true for such $\xi$ by \cite[10.14]{ext2}. 
Suppose that $\dim E_1^{I_1}=1$ then, since $E_{\xi}|_{I_1}\cong E_1\oplus E_1^{\Pi}$, we get 
$\dim E_{\xi}^{I_1}=2$, hence $E_{\xi}^{I_1}=\pi^{I_1}$ and \cite[6.7]{bp}, or Lemma \ref{comphext1} gives 
$\dim \Ext^1_{\HH}(\pi^{I_1}, E_{\xi}^{I_1})=1$. Moreover, Corollary \ref{conseq} implies that
$\dim  H^1(I_1/Z_1, E_{\xi})=2 \dim H^1(I_1/Z_1, E_1)\le 4$. Since $\pi^{I_1}$ is an irreducible $\HH$-module and 
its underlying vector space is  $2$-di\-men\-sio\-nal, we deduce from Lemma \ref{RisH} that
$\dim \Hom_{\HH}(\pi^{I_1}, \RR^1 \II(E_{\xi}))\le 2$.  Now \eqref{5T} implies 
$\dim \Ext^1_{G/Z}(\pi, E_{\xi})\le 3$.
\end{proof}

\section{Non-supersingular representations}\label{nonsupersingularreps}
We recall the properties of Emerton's functor of ordinary parts. This functor is an extremely useful tool 
for calculating $\Ext$ groups, when some principal series are involved. In \S \ref{parabolic} we discuss 
Banach space representations of $G$ obtained by parabolic induction of admissible unitary Banach 
space representations of the torus $T$. We assume throughout this section $p\ge 3$.

\subsection{Ordinary parts}\label{ordinaryparts}
Let $A$ be a complete local noetherian commutative $\OO$-algebra with a finite residue field. 
Emerton in \cite{ord1} has defined a functor 
$\Ord_P: \Mod^{\mathrm{l adm}}_{G, \zeta}(A)\rightarrow  \Mod^{\mathrm{ladm}}_{T, \zeta}(A)$, satisfying 
\begin{equation}\label{adjord}
\Hom_{A[G]}( \Indu{\overline{P}}{G}{U}, V)\cong \Hom_{A[M]}(U, \Ord_P(V)),
\end{equation}
where $\overline{P}$ is the parabolic subgroup of $G$ opposite to $P$ with respect to $T$, see Theorem 4.4.6 in \cite{ord1} if $U$ is admissible, the general statement follows from the fact that
$\Ord_{P}$ and $\Indu{\overline{P}}{G}{}$ commute with inductive limits, see Lemmas 3.2.2 and 4.1.4 in \cite{ord1}.
Since induction is an exact functor, \cite[4.1.5]{ord1}, the functor $\Ord_P$ is left exact.
It follows from  \cite[Prop. 4.3.4]{ord1} that for every $U$ in $\Mod^{\mathrm{ladm}}_{T, \zeta}(A)$ we have:
\begin{equation}\label{ordandind}
\Ord_P (\Indu{\overline{P}}{G}{U})\cong U.
\end{equation}
From now on we suppose that $A$ is artinian. It is shown in \cite[\S 3.7]{ord2} that \eqref{adjord} induces  an $E_2$-spectral sequence:
\begin{equation}\label{specseqord}
\Ext^i_{T,\zeta}(U, \RR^j \Ord_P V) \Longrightarrow \Ext^{i+j}_{G, \zeta}(\Indu{\overline{P}}{G}{U}, V).
\end{equation}
The $\Ext$ groups in \cite{ord2} are computed in the category of locally admissible representations.  This category coincides with the category 
of locally finite representations  by Proposition \ref{Tirrfin}, Corollary \ref{finlengthadm} and \cite[2.3.8]{ord1}.
However, we have shown in Corollary \ref{thesame} that for $G=\GL_2(\Qp)$ and $A=k$ these groups coincide with the $\Ext$ groups 
computed in $\Mod^{\mathrm{sm}}_{G, \zeta}(k)$. This answers a question raised in \cite[3.7.8]{ord2}.
It follows from \cite{eff} that $\RR^j\Ord_P =0$ for $j\ge 2$. Moreover, it follows from Proposition 
\ref{projTistfree} that each block of the category $\Mod^{\mathrm{ladm}}_{T, \zeta}(k)$ is anti-equivalent to the category 
of compact $k[[x,y]]$-modules. Hence,  $\Ext^i_{T, \zeta}=0$ for $i\ge 3$. Thus 
\eqref{specseqord} yields an exact sequence:
\begin{equation}\label{ordseq}
\begin{split}
&\Ext^1_{T, \zeta}(U, \Ord_P V)\hookrightarrow \Ext^1_{G, \zeta}(\Indu{\overline{P}}{G}{U}, V)\rightarrow 
\Hom_T( U, \RR^1\Ord_P V) \\ &\rightarrow \Ext^2_{T, \zeta}(U, \Ord_P V)\rightarrow \Ext^2_{G, \zeta}(\Indu{\overline{P}}{G}{U}, V)
\twoheadrightarrow \Ext^1_{T, \zeta}(U, \RR^1\Ord_P V)
\end{split}
\end{equation}
and an isomorphism 
\begin{equation}\label{ordext3}
\Ext^3_{G, \zeta}(\Indu{\overline{P}}{G}{U}, V)\cong \Ext^2_{T, \zeta}(U, \RR^1\Ord_P V).
\end{equation}
Moreover, we have $\Ext^i_{G, \zeta}(\Indu{\overline{P}}{G}{U}, V)=0$ for $i\ge 4$. Since, we prefer working with $P$ instead of 
$\overline{P}$ we note that the map $f\mapsto [g\mapsto f(s g)]$ induces an isomorphism:
\begin{equation}
\Indu{P}{G}{U}\cong \Indu{\overline{P}}{G}{U^s}
\end{equation}
It follows from  \cite[4.2.10]{ord2} that 
\begin{equation}\label{ordinduced}
\Ord_P(\Indu{P}{G}{U})\cong U^s, \quad \RR^1\Ord_P (\Indu{P}{G}{U})\cong U \otimes \alpha^{-1}
\end{equation}

\begin{prop}\label{projectiveandord} Let 
$\chi: T\rightarrow k^{\times}$ be a smooth character such that $\chi|_Z=\zeta$. Let 
$\iota: \Indu{\overline{P}}{G}{\chi}\hookrightarrow J$ 
be an injective envelope of $\Indu{\overline{P}}{G}{\chi}$ in  $\Mod_{G, \zeta}^{\mathrm{l \, adm}}(A)$. Then the following hold
\begin{itemize} 
\item[(i)] $ \Ord_P( \Indu{\overline{P}}{G}{\chi}) \hookrightarrow \Ord_P J$ is an injective envelope of 
$\chi$ in $\Mod_{T, \zeta}^{\mathrm{l \, adm}}(A)$;
\item[(ii)] the adjoint map $\Indu{\overline{P}}{G}{\Ord_P J}\rightarrow J$ is injective;
\item[(iii)] There exists a natural surjective  ring homomorphism 
$$\End_{A[G]}(J)\twoheadrightarrow \End_{A[G]}( \Indu{\overline{P}}{G}{\Ord_P J})\cong \End_{A[T]}( \Ord_P J).$$
\end{itemize}
\end{prop}
\begin{proof} Since $\Ord_P$ is right adjoint to the exact functor $\Indu{\overline{P}}{G}{}$, $\Ord_P J$ is
injective  in  $\Mod_{T, \zeta}^{\mathrm{l \, adm}}(A)$ and we obtain an injection $\Ord_P \iota: \chi\hookrightarrow \Ord_P J$. 
For every $\tau$ in $\Mod_{T, \zeta}^{\mathrm{l \, adm}}(A)$ we have a commutative diagram:
 \begin{equation}\label{diagramofhoms}
\xymatrix@1{\; \Hom_{A[T]}(\tau, \chi)\;\ar[d]^{\cong} \ar@{^(->}[r] & \Hom_{A[T]}(\tau, \Ord_P J)\ar[d]^{\cong}\\
            \; \Hom_{A[G]}(\Indu{\overline{P}}{G}{\tau}, \Indu{\overline{P}}{G}{\chi})\; \ar@{^{(}->}[r] & 
\Hom_{A[G]}(\Indu{\overline{P}}{G}{\tau}, J).}
\end{equation}
We claim that if $\tau$ is irreducible then the bottom  arrow is an isomorphism. Suppose that 
$\Hom_{A[G]}(\Indu{\overline{P}}{G}{\tau}, J)$ is non-zero, then as $\iota:\Indu{\overline{P}}{G}{\chi}\hookrightarrow J$
is essential, the representations $\Indu{\overline{P}}{G}{\tau}$ and $\Indu{\overline{P}}{G}{\chi}$ have an irreducible 
subquotient in common. In this case it follows from Corollary \ref{tau12sub} that $\tau\cong \chi$. If $\chi\neq \chi^s$ then 
$\Indu{\overline{P}}{G}{\chi}$ is irreducible and the claim follows from the essentiality of $\iota$. If $\chi=\chi^s$ then 
$\chi$ factors through the determinant and thus extends to a character $\chi: G\rightarrow k^{\times}$ and we have an exact non-split 
sequence $0\rightarrow \chi\rightarrow \Indu{\overline{P}}{G}{\chi} \rightarrow \Sp \otimes \chi \rightarrow 0$. Since the 
sequence is non-split, $J$ is also an injective envelope of $\chi$ and any non-zero map $A[G]$-equivariant map 
$\Indu{\overline{P}}{G}{\chi}\rightarrow J$ is an injection. Thus the claim in this case is equivalent to the  assertion that  
$\Ext^1_{G/Z}(\Sp, \Eins)$ is one dimensional. This is shown in \cite[Thm 11.4]{ext2}. 
The claim implies that the top arrow in \eqref{diagramofhoms} is an isomorphism and hence $\Ord_P \iota$ is essential, which proves 
(i). 

We claim that the map $\Indu{\overline{P}}{G} \Ord_P \iota: \Indu{\overline{P}}{G}{\chi}\rightarrow \Indu{\overline{P}}{G}{\Ord_P J}$
is essential. It is enough to show that the natural map
\begin{equation}
 \Hom_G(\pi,  \Indu{\overline{P}}{G}{\chi})\rightarrow  \Hom_G(\pi, \Indu{\overline{P}}{G}{\Ord_P J})
 \end{equation}
is an isomorphism for all irreducible representations $\pi$. By adjointness this is equivalent to showing that  the natural map
\begin{equation}\label{neediso}
\Hom_T(\pi_{\overline{U}},  \chi)\rightarrow  \Hom_T( \pi_{\overline{U}}, \Ord_P J)
\end{equation}
is an isomorphism, where $\pi_{\overline{U}}$ denotes the coinvariants  by the subgroup of lower-triangular  unipotent matrices. Since $\pi$ is an irreducible
representation the coinvariants are either zero  or an irreducible representation of $T$. Since $\chi\hookrightarrow \Ord_P J$ is 
essential by Part (i), in both cases we obtain that \eqref{neediso} is an isomorphism.


Applying $\Hom_{A[G]}(\Indu{\overline{P}}{G}{\Ord_P J}, \ast)$ to the injection $\Indu{\overline{P}}{G}{\Ord_P J}\hookrightarrow J$ 
we get a commutative diagram 
 \begin{displaymath}
\xymatrix@1{ \Hom_{A[T]}(\Ord_P J, \Ord_P J)\ar[d]^{\cong} \ar[r]^{\id} & \Hom_{A[T]}(\Ord_P J,  \Ord_P J)\ar[d]^{\cong}\\
             \Hom_{A[G]}(\Indu{\overline{P}}{G}{\Ord_P J}, \Indu{\overline{P}}{G}{\Ord_P J}) \ar[r] & 
\Hom_{A[G]}(\Indu{\overline{P}}{G}{\Ord_P J}, J).}
\end{displaymath}
Hence, the bottom arrow is an isomorphism. Applying $\Hom_{A[G]}(\ast, J)$ and using injectivity of $J$ we get a surjection
$$\Hom_{A[G]}(J,J)\twoheadrightarrow \Hom_{A[G]}(\Indu{\overline{P}}{G}{\Ord_P J}, J).$$
This implies that every endomorphism of $J$ maps $\Indu{\overline{P}}{G}{\Ord_P J}$ to itself and every endomorphism 
of $\Indu{\overline{P}}{G}{\Ord_P J}$ extends to an endomorphism of $J$, which implies (iii).
\end{proof}

\begin{cor}\label{endoPOrd} Let $\chi:T\rightarrow k^{\times}$ be a smooth character such that $\chi|_Z=\zeta$. Let 
$\wP$ be a projective envelope of $(\Indu{\overline{P}}{G}{\chi})^{\vee}$ in $\dualcat_{G,\zeta}(\OO)$,  let
$\wP_{\chi^{\vee}}$ be a projective envelope of $\chi^{\vee}$ in $\dualcat_{T, \zeta}(\OO)$ and let 
$\wM=(\Indu{\overline{P}}{G}{(\wP_{\chi^{\vee}})^{\vee}})^{\vee}$ then there exists a continuous surjection of rings:
\begin{equation}\label{surjrings}
\End_{\dualcat(\OO)}(\wP)\twoheadrightarrow \End_{\dualcat(\OO)}(\wM)\cong \OO[[x,y]].
\end{equation}
\end{cor}
\begin{proof} It follows from Proposition \ref{projectiveandord} that $\Ord_P( \wP^{\vee})$ is an injective envelope
of $\chi$ in $\Mod_{T, \zeta}^{\mathrm{l\, adm}}(\OO)$. Since injective envelopes are unique up to an isomorphism 
we deduce that $\Ord_P( \wP^{\vee})\cong (\wP_{\chi^{\vee}})^{\vee}$. Duality induces an isomorphism between 
the endomorphism ring of an object and the opposite of the endomorphism ring of its dual. Thus 
it follows from Proposition \ref{projectiveandord} (iii) 
that we have  natural maps:
$$\End_{\dualcat_G(\OO)}(\wP)\twoheadrightarrow \End_{\dualcat_G(\OO)}(\wM)\cong \End_{\dualcat_T(\OO)}(\wP_{\chi^{\vee}}).$$
The last ring is isomorphic to $\OO[[x,y]]$ by Proposition \ref{projTistfree}.
\end{proof} 

\begin{cor}\label{headM}
 We keep notations of Corollary \ref{endoPOrd}. Let $R:=\End_{\dualcat(\OO)}(\wM)$ and
let $\md$ be a compact $R$-module then there exists a natural isomorphism:
$$\md\wtimes_R \wM \cong (\Indu{P}{G}{(\md\wtimes_R \wP_{\chi^{\vee}})^{\vee}})^{\vee}.$$
In particular, $\md \mapsto \md\wtimes_R \wM$ defines an exact functor from the category of 
compact $R$-modules to $\dualcat(\OO)$. Moreover, $\Hom_{\dualcat(\OO)}(\wP, \md \wtimes_R \wM)\cong \md$.
\end{cor}
\begin{proof} The assertion is true by definition if $\md=R$. If $\md=\prod_{i\in I} R$ for some set 
$I$ then 
$$(\md\wtimes_R \wM)^{\vee}\cong \bigoplus_{i\in I} \wM^{\vee}\cong \Indu{P}{G}{(\bigoplus_{i\in I} \wP_{\chi}^{\vee})}\cong 
\Indu{P}{G}{(\md\wtimes\wP_{\chi^{\vee}})^{\vee}}.$$
In general, we may present $\md$ as $\prod_{j\in J} R\rightarrow \prod_{i\in I} R\rightarrow \md\rightarrow 0$ and argue as in 
Lemma \ref{headS0}. Since $\wP_{\chi^{\vee}}$ is a free $R$-module of rank $1$ by Proposition \ref{projTistfree}, 
the functor $\md\mapsto \md\wtimes_R \wP_{\chi^{\vee}}$ is exact and since induction and Pontryagin dual 
are exact functors $\md\mapsto \md\wtimes_R \wM$ is exact.

The last assertion is proved similarly. It follows from Proposition \ref{projectiveandord} that we have an isomorphism 
$\Hom_{\dualcat(\OO)}(\wP, \wM)\cong \Hom_{\OO[T]}(\Ord_{P} J, \Ord_P, J)\cong \End_{\dualcat(\OO)}(\wM)$.
Hence the assertion is true when $\md=R$ and thus it is also true 
when $\md\cong \prod_{i\in I} R$ for some set $I$. In general, 
we may present $\md$ as $\prod_{j\in J} R\rightarrow \prod_{i\in I} R\rightarrow \md\rightarrow 0$ and argue as in Lemma \ref{headS0}.
\end{proof}

\begin{lem}\label{UindJ} Let $U$ and $J$ be in $\Mod_{T, \zeta}^{\mathrm{l \, adm}}(k)$ and suppose that
$J$ is injective. Then 
\begin{equation}
\Ext^1_{G, \zeta}(\Indu{P}{G}{U}, \Indu{P}{G}{J})\cong \Hom_T(U^s, J\otimes \alpha^{-1}) 
\end{equation}
and $\Ext^i_{G, \zeta}(\Indu{P}{G}{U}, \Indu{P}{G}{J})=0$, for $i\ge 2$. 
\end{lem}
\begin{proof} It follows from \eqref{ordinduced} that $\Ord_P ( \Indu{P}{G}{J}) $ and 
$\RR^1\Ord_P( \Indu{P}{G}{J})$ are both injective objects. Thus the terms $\Ext^i_{T,\zeta}$
in \eqref{ordseq} and \eqref{ordext3} vanish and we get the assertion.
\end{proof}

\begin{prop}\label{projresprincser} 
Let $\chi: T\rightarrow k^{\times}$ be a smooth character, such that $\chi\neq \chi^s$ and $\chi\neq \chi^s \alpha^2$. 
Then there exists an exact sequence in $\dualcat_{G, \zeta}(k)$: 
\begin{equation}\label{resindordinj}
0\rightarrow P_{S'}\rightarrow P_S\rightarrow M_{\chi^{\vee}}\rightarrow 0
\end{equation}
where $S=(\Indu{P}{G}{\chi})^{\vee}$, $S'=(\Indu{P}{G}{\chi^s\alpha})^{\vee}$, $P_S$ a  projective envelope 
of $S$ in $\dualcat_{G, \zeta}(k)$ and 
$$M_{\chi^{\vee}}:=(\Indu{P}{G}{(P_{\chi^{\vee}})^{\vee}})^{\vee},$$
where $P_{\chi^{\vee}}$ is a projective envelope of $\chi^{\vee}$ in $\dualcat_{T, \zeta}(k)$. 
\end{prop} 
\begin{remar} If we write $\chi=\chi_1\otimes \chi_2 \omega^{-1}$ then $\chi^s \alpha= \chi_2\otimes \chi_1\omega^{-1}$ and we exclude
the case $\chi_1\chi_2^{-1}=\omega^{\pm 1}$. In particular, both principal series representations are irreducible. For analogous sequences,
when $\chi_1\chi_2^{-1}=\omega^{\pm 1}$ see Proposition \ref{resindJNG} and \eqref{1B}, \eqref{2B}.
\end{remar}
\begin{proof} We show the existence of the dual sequence in $\Mod_{G, \zeta}(k)$. Let $J_{\chi}$ be an injective envelope of $\chi$ in 
$\Mod_{T, \zeta}^{\mathrm{l \, adm}}(k)$ and let $J_{\pi_{\chi}}$ be an injective envelope of 
$\pi_{\chi}:=\Indu{P}{G}{\chi}$ in $\Mod_{G, \zeta}^{\mathrm{l \, adm}}(k)$. 
Then Proposition \ref{projectiveandord} gives an injection $\Indu{P}{G}{J_{\chi}}\hookrightarrow J_{\pi_{\chi}}$ and we 
denote the quotient  by $\kappa_1$. Let $\pi$ be an irreducible smooth representation of $G$, then by applying $\Hom_G(\pi, \ast)$ we get an isomorphism 
\begin{equation}\label{hompiq1}
\Hom_G(\pi, \kappa_1)\cong \Ext^1_{G, \zeta}(\pi, \Indu{P}{G}{J_{\chi}}).
\end{equation}
If $\Hom_G(\pi, \kappa_1)\neq 0$ then $\pi$ is a subquotient of $J_{\pi_{\chi}}$ and hence lies in the block of $\Indu{P}{G}{\chi}$, see the proof of Proposition \ref{blockdecopm}. Hence, $\pi\cong \Indu{P}{G}{\chi}$ or 
$\pi\cong \Indu{P}{G}{\chi^s \alpha}$, see Proposition \ref{blocksoverk}. It follows from Lemma \ref{UindJ} and \eqref{hompiq1} that $\pi\cong \Indu{P}{G}{\chi^s\alpha}$
and $\Hom_G(\pi, \kappa_1)$ is one dimensional. Thus we may embed $\kappa_1\hookrightarrow J_{\pi}$, where $J_{\pi}$ is 
an injective envelope of $\pi$. Let 
$\kappa_2$ be the quotient. Then for every irreducible $\tau$ we have isomorphisms
\begin{equation}\label{hompiq2}
\Hom_G(\tau, \kappa_2)\cong \Ext^1_{G, \zeta}(\tau, \kappa_1)\cong \Ext^2_{G, \zeta}(\tau, \Indu{P}{G}{J_{\chi}}).
\end{equation} 
If $\Hom_G(\tau, \kappa_2)\neq 0$ then $\tau$ lies in the  block of $\Indu{P}{G}{\chi}$. Hence, $\tau\cong \Indu{P}{G}{\chi}$ or 
$\tau\cong \Indu{P}{G}{\chi^s \alpha}$. It follows from Lemma \ref{UindJ} and \eqref{hompiq1} that $\Ext^2$ term vanishes. 
Since every non-zero object of  $\Mod_{G, \zeta}^{\mathrm{l \, adm}}(k)=\Mod_{G, \zeta}^{\mathrm{lfin}}(k)$ has a non-zero socle, 
we deduce that $\kappa_2=0$. 
\end{proof}    

\begin{cor}\label{projresprincser2} 
Let $\chi: T\rightarrow k^{\times}$ be a smooth character, such that $\chi\neq \chi^s$ and $\chi\neq \chi^s \alpha^2$. 
Then there exists an exact sequence in $\dualcat_{G, \zeta}(\OO)$: 
\begin{equation}\label{resindordinj2}
0\rightarrow \wP_{S'}\rightarrow \wP_S\rightarrow \wM_{\chi^{\vee}}\rightarrow 0
\end{equation}
where $S=(\Indu{P}{G}{\chi})^{\vee}$, $S'=(\Indu{P}{G}{\chi^s\alpha})^{\vee}$, $\wP_S$ a  projective envelope 
of $S$ in $\dualcat_{G, \zeta}(\OO)$ and 
$$\wM_{\chi^{\vee}}:=(\Indu{P}{G}{(\wP_{\chi^{\vee}})^{\vee}})^{\vee},$$
where $\wP_{\chi^{\vee}}$ is a projective envelope of $\chi^{\vee}$ in $\dualcat_{T, \zeta}(\OO)$. 
\end{cor} 

\begin{proof} Recall that  if $A$ is an object in $\dualcat(k)$ and $\wP\twoheadrightarrow A$ is  a
projective envelope of $A$ in $\dualcat(\OO)$ then $\wP/\varpi \wP \rightarrow A$ is a projective envelope
of $A$ in $\dualcat(k)$, see Lemma \ref{reduceprojenv}. From this and Corollary \ref{headM} we deduce that  
$\wM_{\chi^{\vee}}\otimes_{\OO} k  \cong M_{\chi^{\vee}}$. 
Proposition \ref{projTistfree} says that $\wP_{\chi^{\vee}}$ is $\OO$-torsion free, hence its dual is $\varpi$-divisible, 
hence $\Indu{P}{G}{\wP_{\chi^{\vee}}^{\vee}}$ is $\varpi$-divisible and so $\wM_{\chi^{\vee}}$ is $\OO$-torsion free. 
The assertion follows from Corollary \ref{projaretfree3} and Proposition \ref{projresprincser}.
\end{proof}

\subsection{Parabolic induction of unitary characters}\label{parabolic}

Let $\Mod^{\mathrm{lfin}}_{T, \zeta}(\OO)$ be the full subcategory of $\Mod^{\mathrm{sm}}_{T}(\OO)$ with objects
locally finite representations on which $Z$ acts by $\zeta$.  The irreducible objects correspond to 
the maximal ideals of $k[T]/(z-\zeta(z) : z\in Z)$, or alternatively $\Gal(\bar{k}/k)$-orbits of 
smooth characters $\chi: T\rightarrow \bar{k}^{\times}$, such that the restriction of $\chi$ to $Z$ is congruent to $\zeta$, and are of the form 
$V_{\chi}$, see Proposition \ref{irrkreps}. It follows from the proof of Proposition \ref{projTistfree} that there are no 
extensions between distinct irreducible representations. Hence, each block  consists of only one irreducible representation
and so if $\BB=\{\tau\}$ then $\Mod^{\mathrm{lfin}}_{T, \zeta}(\OO)^{\BB}$ is a full subcategory of 
$\Mod^{\mathrm{lfin}}_{T, \zeta}(\OO)$ with objects locally finite representations with all the irreducible subquotients isomorphic to $\tau$.
It follows from \cite[\S IV.2]{gab} that we have a decomposition of categories:
\begin{equation}\label{dT}
\Mod^{\mathrm{lfin}}_{T, \zeta}(\OO)\cong \prod_{\BB} \Mod^{\mathrm{lfin}}_{T, \zeta}(\OO)^{\BB},
\end{equation}
where the product is taken over all the blocks $\BB$. Using \eqref{dT} and arguing as in Proposition \ref{blockdecompB},
we obtain a decomposition of the category of admissible unitary $L$-Banach space representations of $T$ on which $Z$ acts by $\zeta$ into a direct sum of 
subcategories:
\begin{equation}\label{dTBan}
\Ban^{\mathrm{adm}}_{T,\zeta}(L)\cong \bigoplus_{\BB} \Ban^{\mathrm{adm}}_{T,\zeta}(L)^{\BB},
\end{equation}
where $\Pi$ is an object of $\Ban^{\mathrm{adm}}_{T,\zeta}(L)^{\BB}$ if and only if all the irreducible subquotients of 
$\Theta/\varpi \Theta$ lie in $\BB$, where $\Theta$ is an open bounded $T$-invariant lattice in $\Pi$.
By Proposition \ref{Tirrfin}, an irreducible $\tau$ is absolutely irreducible if and only if it is a character. 

Let $\chi: T\rightarrow k^{\times}$ be a smooth character with $\chi|_Z\equiv\zeta$, let $\wP\twoheadrightarrow \chi^{\vee}$ be a 
projective envelope of $\chi^{\vee}$ in $\dualcat_{T, \zeta}(\OO)$, the category anti-equivalent to $\Mod^{\mathrm{lfin}}_{T, \zeta}(\OO)$ 
by Pontryagin duality, and let $\wE$ be the endomorphism ring of $\wP$ in $\dualcat_{T, \zeta}(\OO)$. We have showed in Proposition \ref{projTistfree}
that $\wE\cong \OO[[x,y]]$. In particular, $\wE$ is commutative and noetherian. Let $\BB$ be the block of $\chi$ and let 
$\Ban^{\mathrm{adm.fl}}_{T,\zeta}(L)^{\BB}$ be the full subcategory of $\Ban^{\mathrm{adm}}_{T,\zeta}(L)^{\BB}$ consisting of 
all objects of finite length.

\begin{lem}\label{BanfinT} We have an equivalence of categories 
$$\Ban^{\mathrm{adm. fl}}_{T,\zeta}(L)^{\BB}\cong 
\bigoplus_{\nn\in \MaxSpec \wE[1/p]}\Ban^{\mathrm{adm. fl}}_{T,\zeta}(L)^{\BB}_{\nn}.$$
The category $\Ban^{\mathrm{adm. fl}}_{T, \zeta}(L)^{\BB}_{\nn}$ is anti-equivalent to the category 
of modules of finite length of the  
$\nn$-adic completion of $\wE[1/p]$. In particular,  $\Ban^{\mathrm{adm. fl}}_{T, \zeta}(L)^{\BB}_{\nn}$
contains only one irreducible object $\Pi_{\nn}$. \end{lem} 
\begin{proof} Apply Theorem \ref{furtherDBan} with $\dualcat(\OO)=\dualcat_{T,\zeta}(\OO)^{\BB}$.  
\end{proof}

\begin{lem}\label{7point9} Let $\nn$ be a maximal ideal of $\wE[1/p]$ and let $\Pi_{\nn}$ be as above then $\Pi_{\nn}$ is 
finite dimensional over $L$, with $\dim_L \Pi_{\nn}=[\wE[1/p]/\nn: L]$.
\end{lem}
\begin{proof} Let $\Theta$ be an open bounded $T$-invariant lattice in $\Pi_{\nn}$. It follows from Theorem 
\ref{furtherDBan} that $\md(\Pi_{\nn}):=\Hom_{\dualcat_{T,\zeta}(\OO)}(\wP, \Theta^d)_L$ is an irreducible $\wE[1/p]$-mo\-du\-le 
killed by $\nn$. Since $\wE[1/p]/\nn$ is a field we have  $\md(\Pi_{\nn})\cong \wE[1/p]/\nn$. Corollary \ref{fgZfl} implies that $\Theta\otimes_{\OO} k$ is of finite length and the 
irreducible subquotients are isomorphic to $\chi$. Lemma \ref{mult=rank0} says that 
$\chi$ occurs in $\Theta/\varpi \Theta$ with multiplicity $[\wE[1/p]/\nn:L]$.
 Hence, $ [\wE[1/p]/\nn:L]=\dim_k \Theta\otimes_{\OO} k=\dim_L \Pi_{\nn}$.
\end{proof}

Let $\Pi$ be in $\Ban^{\mathrm{adm}}_{T,\zeta}(L)$ and let $|\centerdot|$ be a $T$-invariant norm defining the topology on $\Pi$. 
We may consider $\Pi$ as a representation of $P$ by letting $U$ act trivially.
We let $(\Indu{P}{G}{\Pi})_{cont}$ be the space of continuous functions $f: G\rightarrow \Pi$ such that $f(bg)=b f(g)$ for all $b\in P$ and $g\in G$.
The function $g\mapsto |f(g)|$ is continuous and constant on the cosets $Pg$ since the norm on $\Pi$ is $T$-invariant. Since $P\backslash G$ is compact, 
the function $f\mapsto \| f\|:=\sup_{g\in G} |f(g)|$
defines a $G$-invariant norm on $(\Indu{P}{G}{\Pi})_{cont}$ with respect to which  it is complete. If $\Theta$ is an open bounded $T$-invariant lattice in $\Pi$, 
then $(\Indu{P}{G}{\Theta})_{cont}$ is an open bounded $G$-invariant lattice in $(\Indu{P}{G}{\Pi})_{cont}$ and we have 
\begin{equation}\label{redlatind}
(\Indu{P}{G}{\Theta})_{cont}\otimes_{\OO} \OO/(\varpi^n)\cong \Indu{P}{G}{(\Theta/\varpi^n \Theta)}, \quad \forall n\ge 1.
\end{equation}
Using \eqref{redlatind} one may show that the admissibility of $\Pi$ implies the admissibility of  $(\Indu{P}{G}{\Pi})_{cont}$. 
Let $\wM:= (\Indu{P}{G}{\wP^{\vee}})^{\vee}$ and recall that $\End_{\dualcat_{G, \zeta}(\OO)}(\wM)$ is naturally isomorphic to $\wE=\End_{\dualcat_{T, \zeta}(\OO)}(\wP)$
by Corollary \ref{endoPOrd}.

\begin{lem}\label{thruOrd} Let $\Pi$ be in $\Ban^{\mathrm{adm}}_{T,\zeta}(L)^{\BB}$, let $\Theta$ be an open bounded $T$-in\-va\-riant lattice in $\Pi$ and let 
$\md:=\Hom_{\dualcat_{T,\zeta}(\OO)}(\wP, \Theta^d)$. Then
$$ (\Indu{P}{G}{\Theta})_{cont}^d\cong \md\wtimes_{\wE} \wM, \quad (\Indu{P}{G}{\Pi})_{cont}\cong \Hom_{\OO}^{cont}(\md\wtimes_{\wE} \wM, L).$$
\end{lem}
\begin{proof} Since $\chi$ is the only irreducible object  of $\Mod^{\mathrm{lfin}}_{T, \zeta}(\OO)^{\BB}$, for every object $N$ of  $\dualcat_{T, \zeta}(\OO)^{\BB}$,
$\Hom_{\dualcat_{T, \zeta}(\OO)}(\wP, N)=0$ is equivalent to $N=0$. Thus, it follows from Lemma \ref{headS0} that the map
$\Hom_{\dualcat_{T, \zeta}(\OO)}(\wP, N)\wtimes_{\wE} \wP\rightarrow N$ is an isomorphism. In particular, $\md\wtimes_{\wE}\wP\cong \Theta^d$, 
$(\md/\varpi^n \md)\wtimes_{\wE}\wP\cong \Theta^d/\varpi^n \Theta^d$, for all $n\ge 1$. It follows from \eqref{Thetadnew} that
$\Theta/\varpi^n \Theta \cong ((\md/\varpi^n \md)\wtimes_{\wE} \wP)^{\vee}$, for all $n\ge 1$. Hence, 
$$(\Indu{P}{G}{\Theta})_{cont}^d\otimes_{\OO} \OO/(\varpi^n) \cong (\Indu{P}{G}{(\Theta/\varpi^n \Theta)})^{\vee}\cong 
(\md\wtimes_{\wE}\wM)\otimes_{\OO} \OO/(\varpi^n),$$
where the first isomorphism is \eqref{Thetadnew} and \eqref{redlatind}, the second is given by  Corollary \ref{headM}. We get the first assertion by passing to the
limit. The second assertion follows from \cite{iw}.
\end{proof}

\begin{prop}\label{QofM} Let $\Pi$ be an absolutely irreducible admissible  unitary $L$-Banach space representation of $G$
with a central character $\zeta$.
If $\Hom_{\dualcat_{G, \zeta}(\OO)}(\wM, \Xi^d)\neq 0$ for some open $G$-invariant lattice $\Xi$ in $\Pi$, 
then either $\Pi\cong \eta\circ \det$ or 
$\Pi\cong (\Indu{P}{G}{\psi})_{cont}$ for some continuous unitary character $\psi: T\rightarrow L^{\times}$ lifting $\chi$
with $\psi\neq \psi^s$.
\end{prop} 
\begin{proof} Let $S:=(\Indu{P}{G}{\chi})^{\vee}$ and let $\wP_S\twoheadrightarrow S$ be a projective 
envelope of $S$ in $\dualcat_{G,\zeta}(\OO)$. We note that if $\Indu{P}{G}{\chi}$ is reducible then it is 
a non-split extension of two irreducible representations, hence $\wP_S$ is a projective envelope of 
an irreducible object in $\dualcat_{G,\zeta}(\OO)$, namely the cosocle of $S$. 
Let $\wE_S:=\End_{\dualcat_{G,\zeta}(\OO)}(\wP_S)$ and let $\wE=\End_{\dualcat_{G,\zeta}(\OO)}(\wM)$ as above. Recall that  in 
the Corollary \ref{endoPOrd} we have shown that the natural map $\wP_S\twoheadrightarrow \wM$ induces a  
 surjection of rings $\varphi:\wE_S\twoheadrightarrow \wE$.

Lemma \ref{imageofphi} allows us to assume that there exists a surjection $\phi:\wM\twoheadrightarrow \Xi^d$ in
$\dualcat_{G, \zeta}(\OO)$. Let $\gamma$ be the composition $\wP_S\rightarrow \wM\overset{\phi}{\rightarrow} \Xi^d$
and let $\md:=\Hom_{\dualcat_{G,\zeta}(\OO)}(\wP_S, \Xi^d)$. It follows from Proposition \ref{modulequotientnew} 
that $\md= \gamma\circ \wE_S$. Since $\gamma$ factors through $\wM$ it 
will be killed by any $\phi_1\in \Ker \varphi$. Hence, $\md\cong \phi\circ \wE\cong \Hom_{\dualcat_{G,\zeta}(\OO)}(\wM, \Xi^d)$ and 
the action of $\wE_{S}$ on $\md$ factors through the action of $\wE$. By Proposition \ref{modulequotientnew}, $\md_L$ is an irreducible 
$\wE[1/p]$-module and, since $\wE\cong \OO[[x,y]]$, $\md_L$ is a finite dimensional $L$-vector space. 
Moreover, 
$\End_{\wE_S}(\md)\cong \End_{\wE}(\md)$. Since $\md_L$ is finite dimensional, $\wE$ is commutative and $\Pi$ absolutely irreducible we deduce from Proposition \ref{ringsareiso} and Lemma \ref{endoirrban} that $\md_L$ is one dimensional and so $\md$ is a free $\OO$-module of rank $1$.
Since $\phi\in \md$ the  map $\ev: \md\wtimes_{\wE} \wM\rightarrow \Xi^d$ is surjective. Dually this means that we have an injection 
$\Pi\hookrightarrow \Hom_{\OO}^{cont}(\OO\wtimes_{\wE} \wM, L)\cong (\Indu{P}{G}{\psi})_{cont}$ and the character $\psi$ comes from Lemma \ref{7point9}, where the last isomorphism is given by
Lemma \ref{thruOrd}, which identifies $\Pi$ with a closed $G$-invariant subspace of  
$(\Indu{P}{G}{\psi})_{cont}$. (We note that both Banach space representations are admissible.)
If $\psi\neq \psi^s$ then $(\Indu{P}{G}{\psi})_{cont}$ is topologically irreducible and if $\psi=\psi^s$ then it has a 
unique closed $G$-invariant subspace isomorphic to a character, \cite[5.3.4]{emcoates}. This implies the assertion.
\end{proof}

\section{Generic residually reducible case}\label{genericcase}
In this section we deal with the case where in Colmez's terminology the \textit{atome automorphe} consists
of two distinct irreducible representations.  More precisely, let $\chi_1, \chi_2: \Qp^{\times}\rightarrow k^{\times}$ 
be smooth characters and assume that $\chi_1\chi_2^{-1}\neq \Eins, \omega^{\pm 1}$. We assume throughout this section that $p\ge 3$.
Let 
$\chi: T\rightarrow k^{\times}$ be the character $\chi=\chi_1\otimes \chi_2\omega^{-1}$ then $\chi^s\alpha=\chi_2\otimes\chi_1\omega^{-1}$. 
Let 
$$\pi_1:=\Indu{P}{G}{\chi}, \quad \pi_2:=\Indu{P}{G}{\chi^s\alpha}.$$
We note that the assumption on $\chi_1$ and $\chi_2$ implies that both representations are irreducible and distinct.
Let $\pi$ be an irreducible smooth representation of $G$ with a central character. It is well known, see for example \cite[11.5]{ext2}, that
if $\Ext^1_{G, \zeta}(\pi, \pi_1)\neq 0$ then $\pi\cong \pi_1$ or $\pi\cong \pi_2$. Moreover, 
$$\dim \Ext^1_{G,\zeta}(\pi_1, \pi_1)=2, \quad \dim \Ext^1_{G, \zeta}(\pi_2, \pi_1)=1.$$ 
Let 
\begin{equation}\label{defkappa}
0\rightarrow \pi_1\rightarrow \kappa \rightarrow \pi_2\rightarrow 0
\end{equation} 
be the unique non-split extension.

\begin{lem}\label{ordkappa} $\Ord_P \kappa \cong \Ord_P\pi_1\cong \chi^s$,  $\RR^1\Ord_P \kappa\cong \RR^1\Ord_P \pi_2\cong \chi^s.$
\end{lem} 
\begin{proof} Since $\RR^i\Ord_P=0$ for $i\ge 2$, we apply $\Ord_P$ to \eqref{defkappa} to get an exact sequence: 
\begin{equation}
\begin{split} 
0\rightarrow \Ord_P  \pi_1&\rightarrow \Ord_P \kappa \rightarrow \Ord_P \pi_2 \\
&\rightarrow \RR^1\Ord_P \pi_1\rightarrow  \RR^1\Ord_P \kappa \rightarrow \RR^1\Ord_P \pi_2\rightarrow 0.
\end{split}
\end{equation}
It follows from \eqref{ordinduced} that $\Ord_P \pi_1\cong \chi^s$ and $\Ord_P \pi_2\cong \chi\alpha^{-1}$. 
Hence, if  the map $\Ord_P\kappa\rightarrow \Ord_P \pi_2$ is non-zero then it must be surjective. Hence, 
we have an exact sequence of $T$-representations 
$0\rightarrow \chi^s\rightarrow \Ord_P \kappa \rightarrow \chi\alpha^{-1}\rightarrow 0$. Since $\chi^s\neq \chi\alpha^{-1}$
this sequence must split, see Corollary \ref{extToruschar}. But then using adjointness \eqref{adjord} we would obtain a splitting 
of \eqref{defkappa}. Hence, the map $\Ord_P \pi_2\rightarrow \RR^1\Ord_P \pi_1$ is non-zero, and since 
$\RR^1\Ord_P \pi_1\cong \chi\alpha^{-1}$ the map is an isomorphism. Thus we obtain the claim.
\end{proof} 

\begin{lem}\label{extpi1kappa} Let $\pi$ be irreducible  and suppose that $\Ext^1_{G, \zeta}(\pi, \kappa)\neq 0$ 
then $\pi\cong \pi_1$. Moreover, $\dim \Ext^1_{G, \zeta}(\pi_1, \kappa)\le 3$ and $\Ext^i_{G, \zeta}(\pi_2, \kappa)=0$ for all 
$i\ge 0$.
\end{lem}
\begin{proof} If $\Ext^1_G(\pi, \kappa)\neq 0$ then $\Ext^1_G(\pi, \pi_1)\neq 0$ or $\Ext^1_G(\pi, \pi_2)\neq 0$ and 
hence $\pi\cong \pi_1$ or $\pi\cong \pi_2$. The assertion follows from the degeneration of spectral sequence \eqref{ordseq}, 
Lemma \ref{ordkappa} and the fact that for distinct characters  $\chi, \psi: T\rightarrow k^{\times}$ 
we have  $\Ext^i_{T, \zeta}(\chi, \psi)=0$ for all $i\ge 0$, see Corollary \ref{extToruschar}. 
\end{proof} 

\begin{prop}\label{genHok} Let $S:=\pi_1^{\vee}$ and $Q:=\kappa^{\vee}$ then the hypotheses (H1)-(H5) of \S\ref{first} are satisfied. 
\end{prop} 
\begin{proof} (H1) holds because \eqref{defkappa} is non-split, (H2) holds as $\pi_1\not\cong \pi_2$, (H3), (H4) and (H5) follow
from Lemma \ref{extpi1kappa}.
\end{proof} 
Since (H0) holds for $G$ by Corollary \ref{projaretfree}, we may apply the results of \S\ref{def} and \S\ref{banach}. Let 
$\wP\twoheadrightarrow S$ be a projective envelope of $S$ in $\dualcat(\OO)$, let $\wE=\End_{\dualcat(\OO)}(\wP)$ and 
let $\mm$ be the maximal ideal of $\wE\otimes_{\OO} k$.  Let $\rho:=\cV(Q)$ then since $\cV$ is exact  we get an exact sequence 
of Galois representations 
$$0\rightarrow \chi_2\rightarrow \rho\rightarrow \chi_1\rightarrow 0.$$ 
This sequence is non-split by \cite[VII.4.13]{colmez}. We note that $\det \rho$ is congruent to $\varepsilon\zeta$, where 
$\varepsilon$ is the cyclotomic character.

\begin{prop}\label{surjformsmooth} The functor $\cV$ induces a surjection 
$$\wE\twoheadrightarrow R^{\varepsilon\zeta}_{\rho}\cong \OO[[x, y,z]],$$
where $R^{\varepsilon \zeta}_{\rho}$ pro-represents the deformation functor of $\rho$
with determinant $\varepsilon \zeta$.
\end{prop}
\begin{proof} 
Since $\chi_1\neq \chi_2$ and the sequence is non-split, we get that $\End_{k[\gal]}(\rho)=k$ and 
hence the universal deformation functor $\Def^{ab}_{\rho}$ is representable. Since $\chi_1\chi_2^{-1}\neq \omega^{\pm 1}$
a standard calculation with local Tate duality and Euler characteristic gives  $H^2(\gal, \Ad \rho)=0$ and 
$H^1(\gal, \Ad \rho)$ is $5$-di\-men\-sio\-nal. This implies, see \cite[\S 1.6]{mazur2}, \cite[\S 24]{mazur}, that 
$\Def^{ab}_{\rho}$ is represented by $R\cong \OO[[x_1, \ldots, x_5]]$ and the deformation 
problem with the fixed determinant is represented by $R^{\varepsilon\zeta}\cong\OO[[x_1, x_2, x_3]]$. It follows 
from \cite[2.3.4]{kisin} that $\Spec \wE^{ab}$ is a closed subset of $\Spec R$ and contains
$\Spec R^{\varepsilon\zeta}$, which is stronger than (iii) in Proposition \ref{ftof4}. Since $R^{\varepsilon\zeta}$ is reduced  
we obtain a surjection $\wE^{ab}\twoheadrightarrow R^{\varepsilon\zeta}\cong \OO[[x_1, x_2, x_3]].$
\end{proof}

\begin{cor}\label{dimgentang} We have 
$$\dim \Ext^1_{G, \zeta}(\pi_1, \kappa)=\dim \Ext^1_{\dualcat(k)}(Q, S)=\dim \Ext^1_{\dualcat(k)}(Q, Q)=3.$$
\end{cor}
\begin{proof} We note that all three $\Ext^1$ groups are isomorphic, the first two by anti-equivalence of categories, the last
two by Lemma \ref{first}. Now $\Ext^1_{\dualcat(k)}(Q, Q)$ is isomorphic to $(\mm/\mm^2)^*$ by Lemma \ref{tangentspace} and the surjection of 
Proposition \ref{surjformsmooth} implies that $\dim_k \mm/\mm^2\ge 3$. 
Since $\dim \Ext^1_{G, \zeta}(\pi_1, \kappa)\le 3$ by   Lemma \ref{extpi1kappa} we are done.
\end{proof}

\begin{prop}\label{critforgenOK} $\dim \Ext^1_{G, \zeta}(\pi_1, \tau)\le 3$ for all non-split 
extensions  $0\rightarrow \kappa\rightarrow \tau\rightarrow \kappa\rightarrow 0$ in $\Mod^{\mathrm{sm}}_{G, \zeta}(k)$.   
\end{prop}

\begin{proof} Proposition \ref{surjformsmooth}, Corollary \ref{dimgentang} and Lemma \ref{alter1} imply that the equivalent conditions of 
Lemma \ref{alter} are satisfied and thus by Lemma \ref{enoughisenough} it is enough to check the statement for every non-split extension 
in some $2$-di\-men\-sio\-nal subspace of $\Ext^1_{G, \zeta}(\kappa, \kappa)$. Let $\Upsilon$ be the image of:
$$ \Ext^1_{T, \zeta}(\chi^s, \chi^s)\cong \Ext^1_{G,\zeta}(\pi_1, \pi_1)\hookrightarrow 
\Ext^1_{G,\zeta}(\pi_1, \kappa)\cong \Ext^1_{G,\zeta}(\kappa, \kappa).$$
The extension class of $0\rightarrow \kappa\rightarrow \tau \rightarrow \kappa\rightarrow 0$ lies in $\Upsilon$ if and only if 
there exists an extension $0\rightarrow \chi \rightarrow \epsilon \rightarrow \chi\rightarrow 0$ in $\Mod^{\mathrm{sm}}_{T, \zeta}(k)$ 
and an injection $\Indu{P}{G}{\epsilon}\hookrightarrow \tau$. We denote the quotient by $\kappa_1$.
Since the semi-simplification $\tau^{ss}\cong \pi_1^{\oplus 2}\oplus \pi_2^{\oplus 2}$ we have $\kappa_1^{ss}\cong \pi_2^{\oplus 2}$. 
As $\chi\neq \chi^s \alpha$, the $5$-term sequence \eqref{ordseq} implies that
$\Ext^1_{G, \zeta}(\pi_2, \Indu{P}{G}{\epsilon})$ is $1$-di\-men\-sio\-nal. Since $\Hom_G(\pi_2, \tau)=0$ we deduce that 
$\kappa_1$ cannot be semisimple. We use \eqref{ordseq} again to obtain
$\Ext^1_{G, \zeta}(\pi_2, \pi_2)\cong \Ext^1_{T, \zeta}(\chi^s \alpha, \chi^s\alpha)$. Hence, 
$\kappa_1\cong \Indu{P}{G}{\delta}$, where $0\rightarrow \chi^s\alpha\rightarrow \delta\rightarrow \chi^s\alpha\rightarrow 0$ is an
extension in $\Mod^{\mathrm{sm}}_{T, \zeta}(k)$. Applying $\Ord_P$ to $0\rightarrow \Indu{P}{G}{\epsilon} \rightarrow \tau\rightarrow \Indu{P}{G}{\delta}
\rightarrow 0$ gives an exact sequence:
$$0\rightarrow \epsilon^s \rightarrow \Ord_P \tau\rightarrow \delta^s\overset{\partial}{\rightarrow} \epsilon \alpha^{-1}\rightarrow 
\RR^1\Ord_P \tau \rightarrow \delta \alpha^{-1}\rightarrow 0.$$  
 Since $\Hom_G(\pi_2, \kappa)=0$ we have
$\Hom_T(\chi\alpha^{-1}, \Ord_P \tau)\cong \Hom_G(\pi_2, \tau)=0$. Since $\chi^s\neq \chi\alpha^{-1}$ we have 
$\Ext^1_{T, \zeta}(\chi\alpha^{-1}, \chi^s)=0$ and hence $\partial$ is injective. Since the source and the target are $2$-di\-men\-sio\-nal, $\partial$
is an isomorphism and hence $\Ord_P \tau\cong \RR^1\Ord_P \tau\cong \epsilon^s$ and we have an exact sequence 
$$0\rightarrow \Ext^1_{T, \zeta}(\chi^s, \epsilon^s)\rightarrow \Ext^1_{G,\zeta}(\pi_1, \tau)\rightarrow \Hom_T(\chi^s, \epsilon^s)$$
Since the first term is $2$-di\-men\-sio\-nal by Lemma \ref{strongchar} and the last term is $1$-di\-men\-sio\-nal as $\epsilon$ is non-split, 
we deduce that $\dim \Ext^1_{G, \zeta}(\pi_1, \tau)\le 3$.
\end{proof}

\begin{cor}\label{wEcommGen} The functor $\cV$ induces an isomorphism $\wE\cong R^{\varepsilon \zeta}_{\rho}$. In particular,
$\cV(\wP)$ is the universal deformation of $\rho$ with 
determinant $\zeta\varepsilon$.
\end{cor} 
\begin{proof} The first assertion follows from Theorem \ref{crit}. We then deduce that $\cV$ induces an isomorphism of deformation 
functors, Corollary \ref{ftof1}, and thus $\cV(\wP)$ is the universal deformation of $\rho$ with determinant $\zeta\varepsilon$.
\end{proof}

\begin{thm}\label{mainGen} Let $\Pi$ be an admissible unitary absolutely irreducible $L$-Banach space representation of $G$   
with a central character $\zeta$. Suppose that the reduction of some open bounded $G$-invariant lattice in $\Pi$ 
 contains $\pi_1$ as a subquotient then $\overline{\Pi}\subseteq \pi_1\oplus \pi_2$. 
\end{thm}
\begin{proof} The Schikhof dual of an open bounded $G$-invariant lattice in $\Pi$ is an object of $\dualcat(\OO)$ 
by Lemma \ref{contextGL2}.
Since $\wE$ is commutative the assertion follows from Corollary \ref{commutativeOK}.
\end{proof}

\begin{cor}\label{cormainGen} Let $\Pi$ be as in Theorem \ref{mainGen} and suppose that $\overline{\Pi}$ does not contain $\pi_2$ then 
$\Pi\cong (\Indu{P}{G}{\psi})_{cont}$ for some continuous unitary character $\psi: T\rightarrow L^{\times}$ lifting $\chi$ and satisfying 
$\psi|_Z=\zeta$.
\end{cor}
\begin{proof} Let $\wP_2$ be a projective envelope of $\pi_2^{\vee}$ in $\dualcat_{G, \zeta}(\OO)$
and let $\Theta$ be an open bounded $G$-invariant lattice in $\Pi$. 
Theorem \ref{mainGen} implies that $\overline{\Pi}\cong \pi_1$. Hence  Lemma \ref{mult=rank0}
 says that $\Hom_{\dualcat(\OO)}(\wP_2, \Theta^d)=0$ and $\Hom_{\dualcat(\OO)}(\wP, \Theta^d)\neq 0$. 
We deduce from Corollary \ref{projresprincser2} that  $\Hom_{\dualcat(\OO)}(\wM, \Theta^d)\neq 0$, 
where $\wM=(\Indu{P}{G}{\wP_{\chi^{\vee}}^{\vee}})^{\vee}$ and $\wP_{\chi^{\vee}}$ is a projective envelope of $\chi^{\vee}$ 
in $\dualcat_{T, \zeta}(\OO)$. The assertion follows from Proposition \ref{QofM}.
\end{proof}

\subsection{The centre}
Recall   that the block $\BB$ of $\pi_1$ contains only two irreducible representations $\pi_1$ and $\pi_2$, Proposition \ref{blocksoverk}, and so
$\Mod^{\mathrm{l\, fin}}_{G,\zeta}(\OO)^{\BB}$ is the full subcategory of $\Mod^{\mathrm{l\, fin}}_{G,\zeta}(\OO)$ consisting 
of representations with every irreducible subquotient isomorphic to either $\pi_1$ or $\pi_2$. Let $\dualcat(\OO)^{\BB}$ be the full 
subcategory of $\dualcat(\OO)$ anti-equivalent to $\Mod^{\mathrm{l\, fin}}_{G,\zeta}(\OO)^{\BB}$, as in Proposition \ref{blockdecompD}.
Let $\wP_1$ and $\wP_2$ be projective envelopes of $S_1:=\pi_1^{\vee}$ and $S_2:=\pi_2^{\vee}$ in $\dualcat(\OO)$, respectively. Let 
$\wP_{\BB}:=\wP_1\oplus \wP_2$ and $\wE_{\BB}:=\End_{\dualcat(\OO)}(\wP_{\BB})$. The aim of this subsection is to compute the ring 
$\wE_{\BB}$ and determine its centre. 
 
\begin{lem}\label{samehoms} Let $M$ and $N$ be objects of $\dualcat(\OO)^{\BB}$ then 
$\cV$ induces an isomorphism $\Hom_{\dualcat(\OO)}(M, N)\cong \Hom_{\gal}(\cV(M), \cV(N))$.
\end{lem} 
\begin{proof} Since $\cV$ commutes with projective limits it is enough to show the statement for objects of finite 
length. Now $\dualcat(\OO)^{\BB}$ has only two irreducible objects $S_1$,  $S_2$. For 
$A$ and $B$ isomorphic to $S_1$ or $S_2$ we have $\Hom_{\dualcat(\OO)}(A, B)\cong \Hom_{\gal}(\cV(A), \cV(B))$, since both sides are equal either to 
$k$ or to $0$ and an injection $\Ext^1_{\dualcat(\OO)}(A, B)\hookrightarrow \Ext^1_{\gal}(\cV(A), \cV(B))$ by \cite[\S VII.5]{colmez}.
We then may argue by induction on $\ell(M)+\ell(N)$, where $\ell$ denotes the length, see the proof of Lemma A.1 in 
\cite{ext2}. 
\end{proof}   

Let $\rho_1$ and $\rho_2$ be $2$-di\-men\-sio\-nal $k$-representations of $\gal$ such that we have exact non-split sequences of Galois representations:
$$ 0\rightarrow \chi_2\rightarrow \rho_1\rightarrow \chi_1\rightarrow 0, \quad 0\rightarrow \chi_1\rightarrow \rho_2\rightarrow \chi_2\rightarrow 0.$$
Since $\Ext^1_{\gal}(\chi_1, \chi_2)$ and $\Ext^1_{\gal}(\chi_2, \chi_1)$ are one dimensional such representations exist and are uniquely 
determined up to isomorphism. We note that $\det \rho_1=\det \rho_2$ is congruent to $\zeta \varepsilon$. Let $\rho^{un}_1$ and $\rho^{un}_2$ 
be the universal deformations of $\rho_1$ and $\rho_2$ respectively with determinant $\zeta \varepsilon$.
Let $\chi:=\tr \rho_1=\tr \rho_2$ and 
let $R^{\mathrm{ps}, \varepsilon\zeta}_{\chi}$ be the universal deformation ring parameterizing $2$-di\-men\-sio\-nal pseudocharacters with determinant 
$\zeta\varepsilon$ lifting $\chi$.

\begin{cor}\label{Cgen} The category $\Mod^{\mathrm{l\, fin}}_{G,\zeta}(\OO)^{\BB}$ is anti-equivalent to the category of compact
$\End_{\gal}(\rho_1^{un}\oplus \rho_2^{un})$-modules. The centre of $\Mod^{\mathrm{l\, fin}}_{G,\zeta}(\OO)^{\BB}$ is naturally isomorphic to 
$R_{\chi}^{\mathrm{ps}, \zeta\varepsilon}$.
\end{cor} 
\begin{proof} Corollary \ref{wEcommGen} and Lemma \ref{samehoms} imply  that 
$$\wE_{\BB}\cong \End_{\gal}(\cV(\wP_1)\oplus \cV(\wP_2))\cong \End_{\gal}(\rho_1^{un}\oplus \rho_2^{un}).$$
In Proposition \ref{genrc} we have showed that $\End_{\gal}(\rho_1^{un}\oplus \rho_2^{un})$ is a free $R_{\chi}^{\mathrm{ps}, \zeta\varepsilon}$-module 
of rank $4$ and its centre is isomorphic to $R_{\chi}^{\mathrm{ps}, \zeta\varepsilon}$. 
The assertion follows from Proposition \ref{gabriel}.
\end{proof}

\begin{cor}\label{genkill} Let $T:\gal\rightarrow  R^{\mathrm{ps},\zeta \varepsilon}_{\chi}$ be the universal $2$-dimensional pseudocharacter 
with determinant $\zeta\varepsilon$ lifting $\chi$. For every $N$ in $\dualcat(\OO)^{\BB}$, $\cV(N)$ is killed by 
$g^2-T(g)g + \zeta \varepsilon(g)$, for all $g\in \gal$.
\end{cor}
\begin{proof} Corollary  \ref{wEcommGen} and Proposition \ref{rpsr} imply that the assertion is true if $N=\wP_1$ or $N=\wP_2$. 
Hence, the assertion holds for $N=\wP_{\BB}$. The general case follows from the isomorphism:
$$\cV(N)\cong \cV(\Hom_{\dualcat(\OO)}(\wP_{\BB}, N)\wtimes_{\wE_{\BB}} \wP_{\BB})\cong \Hom_{\dualcat(\OO)}(\wP_{\BB}, N)\wtimes_{\wE_{\BB}}
\cV(\wP_{\BB}),$$
which is proved in the same way as Lemma \ref{VT2}.
\end{proof}

Let $\Ban^{\mathrm{adm}}_{G,\zeta}(L)^{\BB}$ be as in Proposition \ref{blockdecompB} and let $\Ban^{\mathrm{adm. fl}}_{G,\zeta}(L)^{\BB}$
be the full subcategory consisting of objects of finite length.

\begin{cor}\label{genBanach} We have an equivalence of categories 
$$\Ban^{\mathrm{adm. fl}}_{G,\zeta}(L)^{\BB}\cong 
\bigoplus_{\nn\in \MaxSpec R_{\chi}^{\mathrm{ps},\zeta \varepsilon}[1/p]}\Ban^{\mathrm{adm. fl}}_{G,\zeta}(L)^{\BB}_{\nn}.$$
The category $\Ban^{\mathrm{adm. fl}}_{G, \zeta}(L)^{\BB}_{\nn}$ is anti-equivalent to the category 
of modules of finite length of the  
$\nn$-adic completion of $\End_{\gal}(\rho_1^{un}\oplus \rho_2^{un})[1/p]$.
\end{cor} 
\begin{proof} Apply Theorem \ref{furtherDBan} with $\dualcat(\OO)=\dualcat(\OO)^{\BB}$.
\end{proof}

\begin{cor}\label{genirr} Suppose that the pseudo-character corresponding to a maximal ideal $\nn$ of 
$R_{\chi}^{\mathrm{ps}, \zeta\varepsilon}[1/p]$ is irreducible over the residue field of $\nn$ then 
the category $\Ban^{\mathrm{adm. fl}}_{G, \zeta}(L)^{\BB}_{\nn}$ is anti-equivalent to the category 
of modules of finite length of the  
$\nn$-adic completion of $R_{\chi}^{\mathrm{ps}, \zeta\varepsilon}[1/p]$. In particular, it contains 
only one irreducible object.
\end{cor} 
\begin{proof} Since the pseudocharacter corresponding to $\nn$ is irreducible, $\nn$ cannot contain the reducibility ideal of  $R_{\chi}^{\mathrm{ps}, \zeta\varepsilon}[1/p]$, see \S\ref{easier}. It follows from Corollary \ref{redloc2} that for such $\nn$ the $\nn$-adic completion 
of $\End_{\gal}(\rho_1^{un}\oplus \rho_2^{un})[1/p]$ is isomorphic to the ring of two by two matrices 
over the $\nn$-adic completion of $R_{\chi}^{\mathrm{ps}, \zeta\varepsilon}[1/p]$.
\end{proof} 

Let $\nn$ be a maximal ideal of $R_{\chi}^{\mathrm{ps},\zeta \varepsilon}[1/p]$ with residue field $L$, let $T_{\nn}: \gal\rightarrow L$ be 
the pseudocharacter corresponding to $\nn$ and let 
$\Irr(\nn)$ denote the set (of equivalence classes of) irreducible objects in  
$\Ban^{\mathrm{adm. fl}}_{G,\zeta}(L)^{\BB}_{\nn}$. 

\begin{cor}\label{genred} If $T_{\nn}=\psi_1+\psi_2$ with $\psi_1, \psi_2: \gal\rightarrow L^{\times}$ continuous homomorphisms then 
$$\Irr(\nn)=\{ (\Indu{P}{G}{\psi_1\otimes \psi_2 \varepsilon^{-1}})_{cont}, (\Indu{P}{G}{\psi_2\otimes \psi_1 \varepsilon^{-1}})_{cont}\}.$$
\end{cor}
\begin{proof} Corollary \ref{genkill} implies that, since 
$$\VV((\Indu{P}{G}{\psi_1\otimes \psi_2 \varepsilon^{-1}})_{cont})=\psi_2, \quad \VV((\Indu{P}{G}{\psi_2\otimes \psi_1 \varepsilon^{-1}})_{cont})=\psi_1,$$
both Banach space representations lie in $\Irr(\nn)$. Since $\chi_1\chi_2^{-1}\neq \omega^{\pm 1}, \Eins$ we also have $\psi_1\psi_2^{-1}\neq 
\varepsilon^{\pm 1}, \Eins$. Thus 
the Banach space representations are irreducible and distinct. It follows from the explicit description of 
 $\End_{\gal}(\rho_1^{un}\oplus \rho_2^{un})$ in Proposition \ref{genrc} that the ring
$\End_{\gal}(\rho_1^{un}\oplus \rho_2^{un})[1/p]/\nn$ has two non-isomorphic irreducible modules.
\end{proof}

\section{Non-generic case I}\label{nongenericcaseI}

In this section we deal with the case where in Colmez's terminology the \textit{atome automorphe} consists
of two isomorphic irreducible representations. We assume throughout this section that $p\ge 3$. Let $\pi:=\Indu{P}{G}{\chi}$, where 
$\chi: T\rightarrow k^{\times}$ is the character $\chi=\chi_1\otimes \chi_1\omega^{-1}$, for some 
smooth character $\chi_1: \Qp^{\times}\rightarrow k^{\times}$.
 We note that $\chi^s\alpha=\chi$. 
 We show that the formalism of \S \ref{firstsec} applies with $Q=S=\pi^{\vee}$. Hence, the projective envelope 
 $\wP$ of $S$ is the universal deformation of $S$, and its endomorphism ring $\wE$ is the universal deformation ring in the sense of 
Theorem \ref{repnonC}. The new feature in this case is that the ring $\wE$ is non-commutative.  Indeed, if $\wE$ were commutative, then by arguing as in the proof 
of Theorem \ref{mainsuper}, we would deduce that if $\pi$ is a subquotient of a reduction modulo $\varpi$ of an open bounded $G$-invariant lattice in 
an absolutely irreducible $L$-Banach space representation $\Pi$ with central character $\zeta$, then the reduction is isomorphic to $\pi$. However, 
the Banach space representations corresponding to $2$-dimensional crystalline Galois representations of small weight provide a counterexample to this, see 
\cite[5.3.3.1]{breuil2} with $a_p=2p$. By applying the functor $\cV$ we deduce that $\cV(\wP)$ is a deformation to $\wE$ of one dimensional Galois representation 
$\cV(S)=\VV(\pi)=\chi_1$.   Since  we allow the coefficients in our deformation theory be non-com\-mu\-ta\-ti\-ve, Lemma \ref{1ex} implies that the ring $\OO[[\gal(p)]]^{op}$ solves the universal 
deformation problem of $\chi_1$, where $\gal(p)$ is the maximal pro-$p$ quotient of $\gal$. Hence we obtain a map $\varphi_{\cV}: \OO[[\gal(p)]]^{op}\rightarrow \wE$ uniquely determined up to $\wE$-conjugation.
We show that $\varphi_{\cV}$ is surjective by looking at the tangent spaces.

 Let $R^{\mathrm{ps}, \zeta \varepsilon}_{2\chi_1}$ be the universal deformation ring parameterising $2$-di\-men\-sio\-nal pseudocharacters of $\gal$ with determinant $\zeta\varepsilon$ 
lifting $2 \VV(\pi)=2\chi_1$ and let $T:\gal\rightarrow  R^{\mathrm{ps}, \zeta \varepsilon}_{2\chi_1}$ be the universal pseudocharacter. Kisin has shown that every two dimensional Galois representation, 
with reduction modulo $\varpi$ equal to $\chi_1\oplus \chi_1$, lies in the image of $\VV$. This result combined with a  "non-commutative Zariski closure" argument, see Corollary \ref{quotientNGI}, shows that 
$\varphi_{\cV}$ induces a surjection $\wE\twoheadrightarrow (R^{\mathrm{ps}, \zeta \varepsilon}_{2\chi_1}[[\gal]]/J)^{op}$, where 
$J$ is a closed two-sided ideal generated by $g^2-T(g)g+\zeta\varepsilon(g)$ for all $g\in \gal$. We show that this map is an isomorphism, Corollary \ref{weiterso2},  by 
proving structure theorems about both  rings, see Lemma \ref{condNGI} and Proposition \ref{weiterso}.  We also show that $R^{\mathrm{ps}, \zeta \varepsilon}_{2\chi_1}[[\gal]]/J$
is a free module of rank $4$ over its center, which is isomorphic to $R^{\mathrm{ps}, \zeta \varepsilon}_{2\chi_1}$. We record the consequences for Banach space representations in \S \ref{CandBsp}.

The idea to try and  show that $\wE$ is isomorphic to a Cayley-Hamilton quotient was inspired by  \cite{boston}.  


\subsection{Deformation theory.}\label{defNgI}
\begin{prop}\label{hypngI} Let $S=Q=\pi^{\vee}$  then the hypotheses (H1)-(H5) of \S\ref{first} are satisfied. Moreover, 
$d:=\dim \Ext^1_{\dualcat(k)}(S, S)=2$. 
\end{prop} 
\begin{proof} Let $\tau$ be irreducible in  $\Mod^{\mathrm{sm}}_{G, \zeta}(k)$. It is well known, see for example \cite[Thm 11.5]{ext2}, that
if $\Ext^1_{G, \zeta}(\tau, \pi)\neq 0$ then $\pi\cong \tau$ and $\dim \Ext^1_{G, \zeta}(\pi, \pi)=2$. Dually this implies (H3) and 
(H4) and all the other hypotheses hold trivially.
\end{proof} 
Since (H0) holds for $G$ by Corollary \ref{projaretfree}, we may apply the results of \S\ref{def} and \S\ref{banach}. Let 
$\wP\twoheadrightarrow S$ be a projective envelope of $S$ in $\dualcat(\OO)$, let $\wE=\End_{\dualcat(\OO)}(\wP)$, 
$\wm$ the maximal ideal of $\wE$ and 
let $\mm$ be the maximal ideal of $\wE\otimes_{\OO} k$. We note that the last part of Proposition \ref{hypngI} and 
Lemma \ref{tangentspace} gives $\dim \mm/\mm^2=2$.

Let $\wP_{\chi^{\vee}}$ be a projective envelope of $\chi^{\vee}$ in $\dualcat_{T, \zeta}(\OO)$ and let 
$\wM:=(\Indu{P}{G}{\wP_{\chi^{\vee}}^{\vee}})^{\vee}$. Corollary \ref{endoPOrd} gives us a 
surjection 
\begin{equation}\label{qNGI}
\wE\twoheadrightarrow \End_{\dualcat_{G, \zeta}(\OO)}(\wM)\cong \End_{\dualcat_{T, \zeta}(\OO)}(\wP_{\chi^{\vee}})\cong \OO[[x,y]].
\end{equation}
Let $\mathfrak a$ be the kernel of \eqref{qNGI}. Since 
$\dim \mm/\mm^2=2$ we deduce from \eqref{qNGI} that $\wE/\mathfrak a \cong \wE^{ab}$.

\begin{lem} There exists $t\in \wm^2$ such that $\mathfrak a = \wE t$ and  $\phi t\neq 0$ for all non-zero $\phi\in \wE$.
\end{lem}
\begin{proof} Since $\chi=\chi^s\alpha$ Corollary \ref{projresprincser2} gives us an exact sequence
\begin{equation}\label{eNGI}
0\rightarrow \wP\overset{t}{\rightarrow} \wP\rightarrow \wM\rightarrow 0.
\end{equation}
Applying the exact functor $\Hom_{\dualcat(\OO)}(\wP, \ast)$ to \eqref{eNGI} we get an exact sequence 
\begin{equation}
0\rightarrow \wE\overset{t_*}{\rightarrow} \wE\rightarrow \Hom_{\dualcat(\OO)}(\wP,\wM)\rightarrow 0.
\end{equation}
The last term is isomorphic to $\End_{\dualcat(\OO)}(\wM)$ by Proposition \ref{projectiveandord} (iii), Corollary \ref{endoPOrd}. Hence, 
$\mathfrak a=t_{*}(\wE)= \wE t$ and since $t_*$ is injective we get that $\phi t=0$ implies $\phi=0$. As the image of $t$ in $\wE^{ab}$ is zero, 
the image of $t$ in the commutative ring $\wE/\wm^2$ will also be zero. Hence, $t\in \wm^2$.
\end{proof}

\begin{lem}\label{condNGI} Let $\varphi: \wE\twoheadrightarrow R$ be a quotient such that $R^{ab}\cong \OO[[x,y]]$ and there
exists an element $t'\in R$ such that $\Ker(R\twoheadrightarrow R^{ab}) =R t'$ and  $a t'=0$ implies that $a=0$ for 
all $a\in R$. Then $\varphi$ is an isomorphism.
\end{lem}
\begin{proof} The composition $\wE\twoheadrightarrow R\twoheadrightarrow R^{ab}$ factors through $\wE^{ab}$ and since both rings 
are formally smooth of the same dimension we deduce that $\varphi^{ab}: \wE^{ab}\rightarrow R^{ab}$ is an isomorphism.
Thus $\Ker(R\rightarrow R^{ab})= R\varphi(t)$. Hence, we may write $\varphi(t)= a  t'$ and $t'=b \varphi(t)$ for some 
$a, b\in R$. Hence, $(1-ba) t'=0$ and so $ba=1$ and this implies that $b$ and $a$ are units in $R$. (Note that any element 
of $1+\mm_R$ is a unit and hence if the image of $a$ in $R^{ab}$ is a unit then $a$ is a unit in $R$.) 
So we may assume that $t'=\varphi(t)$. 

Since $\mathfrak a$ is  a two-sided ideal and $\mathfrak a=\wE t$, for every $b\in \wE$ there exists $a\in \wE$ such that $tb=at$. This implies 
that for $n\ge 1$ we have $\mathfrak a^n= \wE t^n$. Moreover, since the right multiplication by $t$ is injective, 
multiplication by $t^n$ induces an isomorphism $\wE/\mathfrak a\cong \mathfrak a^n/\mathfrak a^{n+1}$. Since the multiplication 
by $\varphi(t)$ is injective in $R$, multiplication by $\varphi(t)^n$ induces an isomorphism 
$R/\varphi(\mathfrak a)\cong \varphi(\mathfrak a)^n/\varphi(\mathfrak a)^{n+1}$. Hence, $\varphi$ induces an isomorphism
$\mathfrak a^n/\mathfrak a^{n+1}\cong \varphi(\mathfrak a)^n/\varphi(\mathfrak a)^{n+1}$, for all $n\ge 1$. Thus an isomorphism 
$\wE/\mathfrak a^n\cong R/\varphi(\mathfrak a)^n$ for all $n$. Passing to the limit  we get $\wE\cong R$.
\end{proof}    

Now $\cV(S)$ is a $1$-di\-men\-sio\-nal $k$-representation of $\gal$, the absolute Galois group of $\Qp$.
Let $\mathfrak A$ be the category of local finite artinian augmented (possibly non-com\-mu\-ta\-ti\-ve) $\OO$-algebras
defined in Definition \ref{deficatA} and let $\Def_{\cV(S)}:\mathfrak A\rightarrow \Sets$ be the functor, 
such that $\Def_{\cV(S)}(A)$ is the set of isomorphism classes of deformations of $\cV(S)$ to $A$, see 
Definition \ref{defiDef}.  Lemma \ref{1ex} says that the functor 
$\Def_{\cV(S)}:\mathfrak A\rightarrow \Sets$ is 
pro-represented (in the sense of Theorem \ref{repnonC}) by the ring $\OO[[\GG]]^{op}\cong \End_{\OO[[\GG]]}(\OO[[\GG]])$, 
where $\GG:=\gal(p)$ is the maximal pro-$p$ quotient of $\gal$ and $\OO[[\GG]]$ is the universal deformation.

It follows from Corollary  \ref{ftof1} that  $\cV$ induces a natural transformation 
$\cV: \Def_S\rightarrow \Def_{\cV(S)}$. Since $\Def_S$ is pro-represented by $\wE$ by  Theorem \ref{repnonC} we deduce from 
Yoneda's Lemma in this non-commutative context, see Lemma \ref{yoneda}, that the natural transformation 
of functors is induced by 
$$\varphi_{\cV}: \OO[[\GG]]^{op}\rightarrow \wE,$$
where the morphism $\varphi_{\cV}$ is 
uniquely determined up to conjugation by $\wE^{\times}$. Since by  a result of Colmez, \cite[VII.4.15]{colmez}, 
we know that $\cV$ induces an injection 
$$\Ext^1_{\dualcat_{G, \zeta}(k)}(S, S)\hookrightarrow \Ext^1_{\gal}(\cV(S), \cV(S)),$$
we deduce via Lemma \ref{tangentspace} and the proof of Proposition \ref{ftof4} that $\varphi_{\cV}$ is surjective. 

\begin{remar} A note on actions: our groups always act on the left, $(g, v)\mapsto gv$, hence a representation $(\rho, V)$ of $\GG$ 
gives rise to a left $\OO[[\GG]]$-module, which we may write down as a homomorphism 
$\rho: \OO[[\GG]]\rightarrow \End_{\OO}(V)$. In our context, it is also natural to consider $\rho$ as a right  $\OO[[\GG]]^{op}$-module, 
via the isomorphisms $\Hom_{\OO[[\GG]]}(\OO[[\GG]], \rho)\cong \rho$, $\phi\mapsto \phi(1)$, and $\OO[[\GG]]^{op}\cong 
\End_{\OO[[\GG]]}(\OO[[\GG]])$. Having made this point we will not distinguish between left $\OO[[\GG]]$-modules and 
right $\OO[[\GG]]^{op}$-modules.
\end{remar} 

 \begin{prop} Let $M$ be a finite extension of $L$ and let $\rho: \OO[[\GG]]\rightarrow \End_M(W)$ be
a continuous absolutely irreducible representation of $\GG$ with $\dim_M W\le 2$. Then $\Ker \varphi_{\cV} \subset \Ker \rho$.
\end{prop}
\begin{proof} If $\dim W=1$ then $\rho$ factors through $\OO[[\GG]]^{ab}$. Since $p>2$, 
$\GG$ is a free pro-$p$ group on $2$-generators, \cite[7.5.8]{nsw}. 
It follows from \eqref{qNGI} that 
$\varphi$ induces an isomorphism $\OO[[\GG]]^{ab}\cong \wE^{ab}$ and we are done. 

Suppose that $\dim W=2$ by base change we may assume that $M=L$.
 It follows from \cite[2.3.8]{kisin} that there exists an open bounded $G$-invariant lattice 
$\Xi$ in a unitary admissible $L$-Banach space representation $\Pi$ of $G$ such that $L\otimes_{\OO}\cV(\Xi^d)\cong \rho$
and $\overline{\Pi}\cong \pi^{\oplus 2}$.
Since all open bounded lattices are commensurable,  $L\otimes_{\OO}\cV(\Xi^d)$ does not depend on the choice of $\Xi$. Thus we may choose $\Xi$
so that we have a surjection
$$ \Hom_{\dualcat(\OO)}(\wP, \Xi^d)\wtimes_{\wE} \wP\twoheadrightarrow \Xi^d,$$
see Proposition \ref{modulequotientnew}. Corollary \ref{ftof2} says that $\cV(\wP)$ is a free $\wE$-module of rank $1$ and Corollary \ref{ftof3} 
gives us a surjection
$$ \Hom_{\dualcat(\OO)}(\wP, \Xi^d)\wtimes_{\wE}\cV(\wP)\twoheadrightarrow \cV(\Xi^d).$$
Choose a basis element of $\cV(\wP)$ over $\wE$, then this gives us an isomorphism of $\wE$-modules, $\wP\cong \wE$ and hence
a map $\OO[[\GG]]\twoheadrightarrow \wP$ compatible with $\varphi_{\cV}$. (We note that all such choices 
differ by a unit of $\wE$, and in the non-commutative setting $\varphi_{\cV}$ is uniquely determined up to conjugation by 
$\wE^{\times}$.)  And thus we have a surjection of $\GG$-representations 
\begin{displaymath}
\begin{split}
 \Hom_{\dualcat(\OO)}(\wP, \Xi^d)\cong &\Hom_{\dualcat(\OO)}(\wP, \Xi^d)\wtimes_{\OO[[\GG]]^{op}}\OO[[\GG]]\\ 
&\twoheadrightarrow \Hom_{\dualcat(\OO)}(\wP, \Xi^d)\wtimes_{\wE}\cV(\wP)\twoheadrightarrow \cV(\Xi^d),
\end{split}
\end{displaymath}
where the first isomorphism is given by $\phi\mapsto \phi\wtimes 1$. The $\GG$-action on $\Hom_{\dualcat(\OO)}(\wP, \Xi^d)$
is given by 
\begin{equation}\label{good_action}
 g\centerdot \phi= g \centerdot (\phi\wtimes 1)= \phi \wtimes g= \phi \wtimes (1\centerdot g)=
(\phi \circ \varphi_{\cV}(g) )\wtimes 1=
\phi \circ \varphi_{\cV}(g).
\end{equation}
 Since $\overline{\Pi}\cong \pi^{\oplus 2}$ Lemma \ref{mult=rank0} says that $\Hom_{\dualcat(\OO)}(\wP, \Xi^d)$ is a
free $\OO$-module of rank $2$, and hence $\Hom_{\dualcat(\OO)}(\wP, \Xi^d)\cong \cV(\Xi^d)$. Since $\cV(\Xi^d)$ is a lattice 
in $\rho$ we deduce that $\Ker \varphi_{\cV} \subset \Ker \rho$.
\end{proof}

\begin{cor}\label{quotientNGI} Let $\varphi: \OO[[\GG]]\twoheadrightarrow R$ be a quotient such that 
$\bigcap_{\rho} \Ker \rho =0$, where the intersection is taken over all 
continuous representations $\rho: R\rightarrow \End_M(W)$, where $M$ is a finite extension of $L$, 
$\dim_M W\le 2$ and $(\rho, W)$ is absolutely irreducible. Then $\Ker \varphi_{\cV}\subseteq \Ker \varphi$.
\end{cor}

\subsection{Cayley-Hamilton quotient.}
We will construct a quotient $\OO[[\GG]]\twoheadrightarrow R$, satisfying the conditions of Corollary \ref{quotientNGI}, 
and such that $R^{op}$ satisfies the conditions of Lemma \ref{condNGI}. This will imply that $\wE\cong R^{op}$. 
After twisting we may assume that $\chi_1$ is trivial and $\zeta=\varepsilon^{-1}$. The ring $R$ will turn out 
to be isomorphic to $R^{\mathrm{ps}, \Eins}[[\GG]]/J$, where $R^{\mathrm{ps},\Eins}$
is a (commutative) deformation ring parameterizing all $2$-di\-men\-sio\-nal pseudocharacters lifting the trace of the trivial 
$2$-di\-men\-sio\-nal $k$-rep\-re\-sen\-ta\-tion of $\GG$ with determinant equal to $1$, see the conditions 
(o)-(iii) in Proposition \ref{pseudoprop} below,
 and $J$ is a closed two-sided ideal generated by $g^2-T(g) g+1$, for all $g\in \GG$, where $T: \GG\rightarrow R^{\mathrm{ps}, \Eins}$ is the 
universal pseudocharacter with determinant $1$. Using this we will show that for a finite extension 
$M$ of $L$ an absolutely irreducible 
$M$-representation  of $\wE_M$ can be at most $2$-di\-men\-sio\-nal. 

Recall that the maximal pro-$p$ quotient  $\GG$ of $\gal$ is a free pro-$p$ group generated by $2$ elements, which we denote by 
 $\gamma$ and $\delta$.  We let 
\begin{equation}
R:= \frac{\OO[[t_1, t_2, t_3]]\widehat{\otimes}_{\OO}\OO [[ \GG]]}{J},
\end{equation}
where $J$ is  a  closed two-sided ideal generated by  
\begin{equation}
\gamma^2-2(1+t_1) \gamma +1, \quad \delta^2-2(1+t_2) \delta +1,
\end{equation} 
\begin{equation}
(\gamma \delta)^2-2(1+t_3) \gamma \delta +1,\quad  (\delta\gamma)^2-2(1+t_3) \delta\gamma+1.
\end{equation}
Sending $x\mapsto \gamma-1$, $y\mapsto \delta-1$ induces an isomorphism between 
$\OO[[\GG]]$ and $\OO[[x, y]]^{\mathrm{nc}}$, the ring of non-commutative formal power series. 
We denote the images of $t_1, t_2, t_3, x, y$ in $R$ by the same letters. We note that the elements $t_1$, $t_2$  and $t_3$ 
are central in $R$. 

Substituting $\gamma=1+x$ and $\delta=1+y$ in the relations defining the ideal $J$ we get  
 that the following relations hold in $R$:
\begin{equation}\label{xyt} 
 x^2=2t_1(1+x), \quad y^2= 2t_2(1+y),
\end{equation}
\begin{equation}\label{xy+}
(x+y+xy)^2= 2t_3(1+x+y+xy), \quad (x+y+yx)^2= 2t_3(1+x+y+yx)
\end{equation}

Since $2$ is invertible in $R$, and every $a\in 1+\mm_R$ is a unit, we get that $t_1, t_2, t_3\in \mm^2_R$. Thus
the natural map $\OO[[\GG]] \rightarrow R$ is surjective on tangent spaces and,  since both rings are complete,
\begin{equation}\label{surjectOGR}
\OO[[\GG]] \twoheadrightarrow R
\end{equation} 
is surjective and $\dim \mm_R/(\mm^2_R +\varpi_L R)\le 2$. Let $J^{ab}$ be the ideal generated by the  relations 
\eqref{xyt},\eqref{xy+} in the commutative ring $\OO[[t_1, t_2, t_3, x, y]]$. Then we have a natural surjection 
\begin{equation}\label{abel}
R\twoheadrightarrow \frac{\OO[[t_1, t_2, t_3, x, y]]}{J^{ab}}\cong \OO[[x, y]].       
\end{equation}
Since the target is commutative, \eqref{abel} factors through $R^{ab} \twoheadrightarrow \OO[[x,y]]$. 
Since $\dim \mm_R/(\mm^2_R +\varpi_L R)\le 2$ we obtain $R^{ab}\cong \OO[[x,y]]$ and  
$\dim \mm_R/(\mm^2_R +\varpi_L R)=2$.

\begin{defi}\label{bigC} Let $C$ be the commutative ring 
$$\frac{\OO[[t_1, t_2, t_3]] [a_1, a_2, b_1, b_2]} {(a_1+a_2-2t_1, a_1 a_2- 2t_1,  b_1+b_2-2t_2, b_1 b_2-2t_2)}$$
let $\mm_C$ be the maximal ideal of $C$ and let  
$\mathfrak A:=\left(\begin{smallmatrix} C & C \\ \mm_C & C\end{smallmatrix}\right)$,  $\mathfrak P= \left(\begin{smallmatrix} \mm_C & C \\ \mm_C & \mm_C\end{smallmatrix}\right).$
\end{defi}

\begin{prop}\label{realizeT} There exists a continuous representation of $\GG$ on a free rank $2$ module over $C$, which induces a homomorphism 
of $\OO[[t_1, t_2, t_3]]$-algebras $\rho: R\rightarrow \mathfrak A$. In particular, $\tr \rho(\gamma)= 2(1+t_1)$, $\tr \rho(\delta)= 2(1+t_2)$, 
$\tr \rho(\gamma \delta)= 2(1+t_3)$ and $\det(\rho(g))=1$, for all $g\in \GG$.
\end{prop}
\begin{proof} We note that $\mathfrak P$ is a two sided ideal of $\mathfrak A$ and $\mathfrak A$ is $\mathfrak P$-adically complete.
Let 
$$\alpha=\begin{pmatrix} a_1 & 1 \\ 0 & a_2 \end{pmatrix}, \quad \beta= \begin{pmatrix} b_1 & 0 \\ b & b_2 \end{pmatrix}$$
with $b= 2+2t_3-(1+a_1)(1+b_1) - (1+a_2)(1+b_2)\in \mm_C$.  
Sending $g\mapsto g-1$ induces an isomorphism
 $ (1+\mathfrak P^i)/(1+\mathfrak P^{i+1})\cong \mathfrak P^i/\mathfrak P^{i+1}$, where the right hand side is a group with respect to addition.
Since $\mathfrak P^i/\mathfrak P^{i+1}$ is a finite dimensional  $\OO/\varpi_L \OO$-vector space, we deduce that
$1+\mathfrak P$ is a pro-$p$ group. Hence, $\gamma\mapsto 1+\alpha$ and $\delta\mapsto 1+\beta$ induces a group homomorphism 
$\GG\rightarrow 1+\mathfrak P$ and hence an algebra homomorphism 
\begin{equation}\label{algebrahom}
\OO[[t_1, t_2, t_3]]\bigl[\bigl[\GG\bigr]\bigr]\rightarrow \mathfrak A.
\end{equation} 
By construction of $C$ we have 
$\det(1+\alpha)=\det(1+\beta)=1$, $\tr(1+\alpha)=2(1+t_1)$ and $\tr(1+\beta)=2(1+t_2)$.  Hence, $\det( (1+\alpha)(1+\beta))=1$ and 
a direct calculation shows that $\tr( (1+\alpha)(1+\beta))= 2(1+t_3)$. Hence, \eqref{algebrahom} factors through $\rho:R\rightarrow \mathfrak A$. 
\end{proof}

\begin{cor}\label{injectcentre} The natural map $\OO[[t_1, t_2, t_3]]\rightarrow R$ is injective.
\end{cor} 
\begin{proof}
Since the composition $\OO[[t_1, t_2, t_3]] \rightarrow R\overset{\rho}{\rightarrow} \mathfrak A$ is injective, where $\rho$ is the representation constructed in the Proposiition \ref{realizeT}, we obtain the claim.
\end{proof} 

Let $H$ be the subgroup of $\GG$ generated as an abstract group by $\gamma$ and $\delta$. There is a natural length function 
$\ell: H\rightarrow \ZZ_{\ge 0}$, $\ell(h)= \min (\sum_{i\ge 1}  |m_i|)$, where the minimum is taken over all 
finite expressions $h=\gamma^{m_1}\delta^{m_2}\ldots$, with $m_i\in \ZZ$. We let 
$$\Gamma:=\{ 1, \gamma, \delta, \gamma\delta, \delta\gamma\}$$
and given an integer $m\ge 0$ 
we define
$$  S_m:= \{ g \in H: \ell(g)\le m\},$$
and for a subset $S$ of $H$ we define
$$\Sigma(S):=\{ g_1g_2:  g_1, g_2,  g_1^{-1}g_2\in S \}\cup \{ g_1^{-1}g_2:  g_1, g_2,  g_1g_2\in S \}.$$ 
We note that if $1\in S$ then by taking $g_1=1$  we obtain that $\Sigma(S)$ contains $S$ and by taking $g_2=1$ we get that 
$\Sigma(S)$ contains $S^{-1}$, the set of  inverses of the elements of $S$.  

\begin{lem}\label{permutealot} 
\begin{itemize} 
\item[(i)] 
$S_2\subset \Sigma(\Sigma(\Gamma))$; 
\item[(ii)] $S_m \subset \Sigma(\Sigma( \Sigma(S_{m-1})))$, for  $m\ge 3$. 
\end{itemize} 
\end{lem}
\begin{proof} Since $1\in \Gamma$, $\Sigma(\Gamma)$ will contain $\Gamma\cup \Gamma^{-1}$ and 
also $\gamma^2$, $\delta^2$, $\gamma^{-1}\delta$, $\delta^{-1}\gamma$. Thus all the elements of $S_2$, except for $\gamma\delta^{-1}$ and $\delta \gamma^{-1}$, are contained in $\Sigma(\Gamma)\cup \Sigma(\Gamma)^{-1}$, 
which is a subset of $\Sigma(\Sigma(\Gamma))$.
To finish the proof of (i), we observe that since $\gamma, \delta^{-1}, \gamma^{-1}\delta^{-1} \in \Sigma(\Gamma)$,  $\gamma\delta^{-1}\in \Sigma(\Sigma(\Gamma))$ and 
since $\delta, \gamma^{-1}, \delta^{-1}\gamma^{-1} \in \Sigma(\Gamma)$,  $\delta\gamma^{-1}\in \Sigma(\Sigma(\Gamma))$.

Let $g=\gamma^{m_1}\delta^{m_2}\ldots$ be an expression of $g$  such that $\ell(g)=\sum_{i\ge 1} |m_i|\ge 3$. 
Without loss of generality we may assume $m_1\neq 0$. If $|m_1|>1$ then $g\in \Sigma(S_{m-1})$ as we may take 
$g_1= \gamma^{\varepsilon}$,  $g_2=\gamma^{-\varepsilon} g$ 
where $\varepsilon=m_1/|m_1|$, so that $g=g_1g_2$ and $g_1^{-1}g_2\in S_{m-1}$. Hence, if $|m_j|>1$ for some $j$, then 
$g\in \Sigma(\Sigma(S_{m-1}))$, for we may take (for odd $j$ ) 
$g_1= \gamma^{m_1}\ldots \delta^{m_{j-1}}\gamma^{m_j}$ and $g_2= \delta^{m_{j+1}}\ldots$ and so $g=g_1 g_2$ and $ g_1^{-1} g_2\in \Sigma(S_{m-1})$
by the previous calculation. 
Thus we may assume $|m_i|=1$ for all $i$ and since $\ell(g)\ge 3$ this implies $m_3\neq 0$.
Let $g_1=\gamma^{m_1}\delta^{m_2}$, $g_2= \gamma^{m_3}\ldots$, if $m_1$ and $m_3$ have the same sign then $g_1^{-1}g_2\in S_{m-1}$, if not 
then $g_1^{-1}g_2= \delta^{-m_2} \gamma^{m_3-m_1} \ldots \in \Sigma(\Sigma(S_{m-1}))$.
\end{proof}

\begin{cor}\label{relationdetermines} Let $B$ be a topological ring and let $f: \GG\rightarrow B$ be a continuous function such 
that 
\begin{equation}\label{relationdetermines1}
f(g^{-1}h)-f(g)f(h)+f(gh)=0, \quad  \forall g, h\in \GG.
\end{equation}
Then $f$ is uniquely determined by its values at the elements of $\Gamma$. Moreover, 
the image of $f$ is contained in the closure of the subring  of $B$ generated by 
$f(g)$, $g\in\Gamma$.
\end{cor}
\begin{proof} Since $f$ is continuous and $H$ is dense in $\GG$, $f$ is uniquely determined by its restriction 
to $H$. Using \eqref{relationdetermines1} and Lemma \ref{permutealot} we deduce that $f|_H$ is uniquely 
determined by $f(g)$, $g\in \Gamma$ and $f(H)$ is contained in the subring of $B$ generated by $f(g)$, $g\in \Gamma$. 
Since $H$ is dense in $G$, $f(G)$ is contained in the closure of this ring. 
\end{proof} 

\begin{prop}\label{pseudoprop} Let $T(g):= g+ g^{-1}\in R$ then  
$T(g)\in \OO[[t_1, t_2, t_3]]$ for each $g\in \GG$. Moreover, $T$ is the unique 
continuous function $T: \GG \rightarrow \OO[[t_1, t_2, t_3]]$ such that 
\begin{itemize} 
\item[(o)] $T(1)=2$;
\item[(i)] $\frac{ T(g)^2-T(g^2)}{2}=1$;
\item[(ii)] $T(gh)=T(hg)$;
\item[(iii)] $T(g^{-1}h)-T(g) T(h) +T(gh)=0$;
\item[(iv)] $T(\gamma)=2(1+t_1)$,  $T(\delta)=2(1+t_1)$, $T(\gamma \delta)=T(\delta \gamma)=2(1+t_3)$.
\end{itemize}
\end{prop}
\begin{proof} We note that $T: \GG\rightarrow R$ is continuous and satisfies (o), (i) and (iv). 
Now
\begin{equation}
T(g^{-1}h)-T(g) T(h) +T(gh)= h^{-1} T(g) - T(g) h^{-1}.   
\end{equation}
So (iii) holds for all $g, h \in G$ such that  $T(g)$ is central in $R$. Since $T(g)$ is central in $R$ for every $g\in \Gamma$, 
using Lemma \ref{permutealot} we deduce that (iii) holds for every $g, h\in H$ and by continuity of $T$ and density of $H$, we get that 
(iii) holds. It follows from Lemma \ref{injectcentre} and (iv) that the closure of the subring of $R$ generated by $T(g)$, $g\in \Gamma$ 
is $\OO[[t_1, t_2, t_3]]$.  It follows from Corollary \ref{relationdetermines} that $T(g)\in \OO[[t_1, t_2, t_3]]$ and 
is uniquely determined. It remains to show that $T$ satisfies (ii). Let $\rho: R\rightarrow \mathfrak A$ be the homomorphism constructed 
in Proposition \ref{realizeT}. Recall that $\mathfrak A$ is a subring of the ring of  $2\times 2$ matrices over  a commutative ring $C$.
Hence  for every $a\in \mathfrak A$ we have 
\begin{equation}\label{CHam}
a^2 -\tr(a) a +\det a = 0,
\end{equation}
where $\tr: \mathfrak A\rightarrow C$ and $\det:\mathfrak A\rightarrow C$ are the usual trace and determinant. Since by construction 
$\det \rho(\gamma)=\det\rho(\delta)=1$ we get that $\det\rho(g)=1$ for all $g\in \GG$, and so  we may rewrite \eqref{CHam} to get 
\begin{equation}\label{Tistrace}
\rho(T(g))= \rho(g) +\rho(g)^{-1}= \tr(\rho(g)) 
\end{equation}    
and hence $\rho(T(gh)-T(hg))=0$. The restriction of $\rho$ to $\OO[[t_1, t_2, t_3]]$ is injective and this implies (ii).  
\end{proof} 

\begin{cor}\label{thesameps} Let $\eta: \gal \rightarrow k^{\times}$ and $\psi:\gal\rightarrow \OO^{\times}$ be  continuous characters 
such that $\psi \equiv \eta^2\pmod{\pL}$. Then the universal deformation ring $R^{\mathrm{ps}, \psi}$ 
parameterizing $2$-di\-men\-sio\-nal pseudo-characters of $\gal$ lifting $2\eta$ with determinant $\psi$ is isomorphic to 
$\OO[[x_1, x_2, x_3]]$ and the universal pseudocharacter is equal to the trace of the representation constructed in 
Proposition \ref{realizeT} twisted with $\sqrt{\psi}$.
\end{cor} 
\begin{proof} Since $\psi$ modulo $\pL$ is a square and $p\neq 2$ there exists a continuous 
character $\psi_1: \gal\rightarrow \OO^{\times}$ such that $\psi_1^2=\psi$. Corollary \ref{pst3} implies that it is enough 
to show that the assertion after  replacing $\gal$ with its maximal pro-$p$ quotient $\mathcal G$ and this follows 
Proposition \ref{pseudoprop}.
\end{proof}

Following \cite{boston} we introduce an involution $\ast$ on $R$, by letting $g^{\ast}:=g^{-1}$, extending it linearly on 
$\OO[[t_1, t_2, t_3]]\bigl[\bigl[ \GG \bigr] \bigr]$ and observing that $J^{\ast}=J$. 

\begin{lem}\label{fixedpointsinv} $\OO[[t_1, t_2, t_3]]= \{a\in R: a=a^{\ast}\}$.
\end{lem}
\begin{proof} Every $a\in \OO[[t_1, t_2, t_3]]$ is fixed by $\ast$ by construction. The map 
$R\rightarrow R$, $a\mapsto \frac{a+a^*}{2}$ is continuous and maps the subring $\OO[[t_1, t_2, t_3]] \bigl[\GG\bigr] +J$ 
into $\OO[[t_1, t_2, t_3]]$ by Proposition \ref{pseudoprop}. Since the subring is dense in $R$, we conclude that 
the fixed points of $\ast$ are contained in $\OO[[t_1, t_2, t_3]]$.
\end{proof}

\begin{cor}\label{involute1} Let $a\in R$ then $a+a^{\ast}$ and $a^{\ast}a$ are in $\OO[[t_1, t_2, t_3]]$. 
\end{cor}
\begin{proof} This follows from $(ab)^{\ast}=b^{\ast} a^{\ast}$ and $(a^{\ast})^{\ast}=a$ and Lemma \ref{fixedpointsinv}. 
\end{proof}
  
\begin{cor}\label{involute2} Let $\rho: R\rightarrow \mathfrak A$ be the representation constructed in Proposition \ref{realizeT}. Then 
\begin{equation}\label{theendisinevitable}
\rho(a+a^{\ast})= \tr(\rho(a)), \quad \rho(a^{\ast}a)= \det(\rho(a)), \quad \forall a\in R.
\end{equation}
\end{cor}
\begin{proof} The function $R\rightarrow \mathfrak A$, $a\mapsto \rho(a+a^{\ast})- \tr(\rho(a))$ is $\OO[[t_1, t_2, t_3]]$-linear, continuous, 
 and zero on $\GG$ by \eqref{Tistrace}. Hence, it is zero on $\OO[[t_1, t_2, t_3]] \bigl [ \GG\bigr]+J$
and since it is  dense  the function is zero on $R$. Now $a^2-( a+a^{\ast}) a +a^{\ast} a=0$ in $R$. 
Hence, 
$$0=\rho(a)^2- \rho(a+a^{\ast}) \rho(a) +\rho(a^{\ast}a)= \rho(a)^2-\tr(\rho(a))\rho(a) +\rho(a^{\ast} a).$$
 Since,  $\rho(a)^2-\tr(\rho(a)) \rho(a) +\det(\rho(a))=0$ in $\mathfrak A$ we get $\rho(a^{\ast}a)=\det \rho(a)$.
\end{proof} 

\begin{cor}\label{involute3} Let $g\in \mathcal G$ and $T$ as in Proposition \ref{pseudoprop} then 
$g^2-T(g) g +1=0$ in $R$.
\end{cor} 
\begin{proof} We have $T(g)=g+g^{-1}=g +g^{\ast}$ and the assertion follows from \eqref{theendisinevitable} and 
the identity $g^2-(g+g^{\ast}) g + g^{\ast}g=0$.
\end{proof}

 To ease the calculations we set 
\begin{equation}\label{defuv}
u:= x-t_1=\gamma-1-t_1, \quad v:= y-t_2=\delta-1-t_2.
\end{equation}
We note that the images of $u$ and $v$ form a $k$-basis of $\mm_R/(\mm_R^2+\varpi_L R)$ and hence it follows from \eqref{surjectOGR} 
that $u$ and $v$ generate $R$ topologically over $\OO$. Then \eqref{xyt} reads 
\begin{equation}\label{uvsquared} 
u^2= 2t_1-t_1^2, \quad v^2=2t_2-t_2^2.
\end{equation}
In particular, $u^2$ and $v^2$ are central in $R$. Hence, 
\begin{equation}\label{commuteanti}
u(uv-vu)= -(uv-vu)u, \quad v(uv-vu)= -(uv-vu)v.
\end{equation}
We also note that substituting $t_1= \frac{\gamma+\gamma^{-1}}{2}-1$ and $t_2= \frac{\delta+\delta^{-1}}{2}-1$ in \eqref{defuv} gives 
$u= \frac{\gamma-\gamma^{-1}}{2}$, $v= \frac{\delta-\delta^{-1}}{2}$. Hence, 
\begin{equation}\label{uvast}
u^{\ast}=-u , \quad v^{\ast}=-v, \quad (uv-vu)^{\ast}=-(uv-vu), \quad (uv+vu)^{\ast}= uv+vu
\end{equation}

\begin{lem}\label{expressa1} Every element $a\in R$ maybe written as
\begin{equation}\label{expressa}
a= \lambda_1 + \lambda_2 u + \lambda_3 v + \lambda_4 (uv-vu)
\end{equation}
with $\lambda_i\in \OO[[t_1, t_2, t_3]]$.
\end{lem}
\begin{proof} It follows form \eqref{surjectOGR} that $a$ may be written as a formal power series with coefficients 
in $\OO$ in (non-commuting) variables $u$, $v$. It follows from \eqref{uvsquared} that $u^2, v^2\in \OO[[t_1, t_2, t_3]]$ 
so we only need to deal with monomials  of the form $(uv)^n$, $(uv)^n u$, $(vu)^n$, $(vu)^nv$. 
Lemma \ref{fixedpointsinv} and \eqref{uvast} give that $(uv-vu)^2, uv+vu\in \OO[[t_1, t_2, t_3]]$ and since $2$ is invertible 
in $R$,  we may substitute $uv= \frac{(uv-vu)+ (uv+vu)}{2}$ and $vu= \frac{(uv+vu)- (uv-vu)}{2}$. Thus 
$(uv)^2= \lambda uv- \lambda vu + \mu$  and $(vu)^2=\lambda vu- \lambda uv + \mu$ with $\lambda, \mu\in \OO[[t_1,t_2,t_3]]$,
which leaves us to deal with  
$uvu$ and $vuv$. Since $uvu=(uv+vu)u - u^2v$ and $vuv=(uv+vu)v - v^2 u$ we are done.
\end{proof}

\begin{cor}\label{dimle2} Let $L'$ be a finite extension of $L$ and $W$ a finite dimensional $L'$-vector space with continuous 
$R$-action $\tau: R\otimes_{\OO} L' \rightarrow \End_{L'}(W)$. Suppose that the representation $W$ is absolutely irreducible, then $\dim_{L'} W\le 2$. 
\end{cor}
\begin{proof} Since $W$ is absolutely irreducible and finite dimensional over $L'$ we have $\End_{\GG}(W)=L'$ and thus $\tau$ induces a continuous
homomorphism of $\OO$-algebras $\OO[[t_1, t_2, t_3]]\rightarrow L'$. Since $W$ is absolutely irreducible $\tau$ is surjective. 
It follows from Lemma \ref{expressa1} that  $(\dim_{L'}  W)^2 = \dim_{L'} \End_{L'} W =\dim_{L'} (\tau(R)\otimes_{\OO} L') \le 4$. 
\end{proof}

\begin{lem}\label{uvvunon} $(uv-vu)^{\ast}(uv-vu)\neq 0$ in $R$.
\end{lem}
\begin{proof} It follows from \eqref{defuv} that $uv-vu= \gamma \delta-\delta\gamma$. Using Corollary \ref{involute2} it is enough 
to show that $\det(\rho(\gamma\delta-\delta \gamma))\neq 0$. If we specialize $t_1=t_2=0$ this means it is enough to show that 
the determinant of 
$$ \begin{pmatrix} 1 & 1 \\ 0 & 1\end{pmatrix} \begin{pmatrix}1 & 0 \\ 2t_3 & 1\end{pmatrix}- \begin{pmatrix}1 & 0 \\ 2t_3 & 1\end{pmatrix} 
\begin{pmatrix} 1 & 1 \\ 0 & 1\end{pmatrix}= \begin{pmatrix} 2t_3 & 0 \\ 0 & -2t_3 \end{pmatrix}$$ 
is non-zero in $\OO[[t_3]]$, which is clear.
\end{proof}

\begin{lem}\label{preweiterso} Let $\mm$ be a maximal ideal of $C[1/p]$ and let $\rho_{\mm}: R\rightarrow \mathfrak A\otimes_C \kappa(\mm)$
 be the specialization at $\mm$ of the representation $\rho$ constructed in Proposition \ref{realizeT}. Then the following are equivalent: 
 \begin{itemize}
 \item[(i)] $\rho_{\mm}$ is absolutely irreducible;
 \item[(ii)] $\rho_{\mm}(uv -vu)$ is invertible;
 \item[(iii)] $(uv - vu)(uv -vu)^{\ast}\not\in \mm$.
 \end{itemize}
 \end{lem}
 \begin{proof} (i) implies (ii). The kernel of $\rho_{\mm}(uv -vu)$ is stable under $u$ and $v$, see \eqref{commuteanti}. If $\rho_{\mm}(uv -vu)=0$ then 
 the action of $\GG$ factors throughout its abelian quotient, as $\gamma \delta -\delta \gamma= uv -vu$. Since $\rho_{\mm}$ is absolutely 
 irreducible, this would force the dimension of $\rho_{\mm}$ over $\kappa(\mm)$ to be $1$. Since the dimension is $2$, we deduce that the kernel of 
 $\rho_{\mm}(uv-vu)$ is zero, and hence it is invertible.

 (ii) implies (i). Suppose that $\rho_{\mm}$ is not absolutely irreducible. Then after replacing $\kappa(\mm)$ be a finite extension, we may choose a basis 
 such that the matrices of $\rho_{\mm}(\gamma)$ and $\rho_{\mm}(\delta)$ are both upper-triangular. Since, $uv-vu =\gamma \delta- \delta \gamma$ we deduce 
 that $\rho_{\mm}(uv -vu)$ is nilpotent.

(ii) is equivalent to (iii).  It follows from  \eqref{theendisinevitable} that the image of $(uv - vu)(uv -vu)^{\ast}$ in $\kappa(\mm)$ is equal to the determinant of $\rho_{\mm}(uv-vu)$.
 \end{proof}

\begin{lem}\label{fnotvanish} Let $f\in \OO[[x_1, \ldots, x_n]]$ be non-zero then there exists $a_i\in \pL$, $1\le i\le n$ such that
$f(a_1,\ldots, a_n)\neq 0$.
\end{lem}
\begin{proof} Since $\OO[[x_1,\ldots, x_n]]$ is a unique factorisation domain, \cite[20.3]{matsumura}, $f$ is divisible by only finitely 
many prime elements. Hence, we may find $a_n\in \pL$ such that $x_n-a_n$ does not divide $f$. Let $f_1$ be the image 
of $f$ in $\OO[[x_1, \ldots, x_n]]/(x_n-a_n)\cong \OO[[x_1,\ldots, x_{n-1}]]$. By construction $f_1$ is non-zero and we proceed as before.
\end{proof}

\begin{prop}\label{weiterso} Let $a=\lambda_1 + \lambda_2u + \lambda_3 v + \lambda_4(uv-vu)\in R$ with $\lambda_i\in \OO[[t_1, t_2, t_3]]$ and 
not all $\lambda_i$ equal to zero. Then there exists a finite extension $L'$ of $L$ and a $2$-di\-men\-sio\-nal $L'$-vector space 
$W$, with a continuous action $\tau: R\otimes_{\OO} L'\rightarrow \End_{L'} W$ such that $\tau$ is absolutely irreducible and $\tau(a)\neq 0$.
\end{prop}
\begin{proof} Let $\lambda=(uv-vu)^{\ast}(uv-vu) \in R$, we note that $\lambda$ is non-zero in $R$ by Lemma \ref{uvvunon}. 
Let $f\in \OO[[t_1, t_2, t_3]]$ be the product of $\lambda$ and non-zero $\lambda_i$'s. By Lemma \ref{fnotvanish} we may find 
$a_1, a_2, a_3\in \pL$ such that $f(a_1, a_2, a_3)\neq 0$. 
Let $C$ be the ring defined in Definition \ref{bigC} and let $\mm$ be any maximal 
ideal of $C[1/p]$ containing $(t_1-a_1, t_2-a_2, t_3-a_3)$. Then the residue field $\kappa(\mm)$ of $\mm$  is a finite extension of $L$. Moreover, the image 
of $f$ in $\kappa(\mm)$ is equal to $f(a_1, a_2, a_3)$ and hence is non-zero. So the image of $\lambda$ in $\kappa(\mm)$ is non-zero, 
and not all $\lambda_i$ map to $0$ in $\kappa(\mm)$.  Let $\tau=\rho_{\mm}: R\otimes_{\OO}\kappa(\mm)\rightarrow \mathfrak A\otimes_C \kappa(\mm)$ as in Lemma \ref{preweiterso}.
Since the image of $\lambda$ in $\kappa(\mm)$ is non-zero by construction, Lemma \ref{preweiterso} implies that $\tau$ is absolutely irreducible. Thus, $\tau$ is surjective. 
Since $\dim_{\kappa(\mm)}  \mathfrak A\otimes_C \kappa(\mm)=4$, 
we deduce from Lemma \ref{expressa1} that $1$, $\tau(u)$, $\tau(v)$ and $\tau(uv-vu)$ are linearly independent. Hence, $\tau(a)\neq 0$, 
as $\kappa(\mm)$ was constructed so that the images of non-zero $\lambda_i$ are non-zero.
\end{proof}

\begin{cor}\label{weiterso3} The centre of $R$ is equal to $\OO[[t_1, t_2, t_3]]$.
\end{cor}
\begin{proof} Suppose there exists  a non-zero 
element $z$ in the centre of $R$ such that $z^{\ast}=-z$. Let $(\tau, W)$ and $L'$ be as in Proposition \ref{weiterso} with $\tau(z)\neq 0$ then 
$\tau(z)$ is a scalar matrix in  $\End_{L'}(W)$. It follows from Corollary \ref{involute2} that
$\tr \tau(z)= \tau(z^{\ast}+z)=0$ and thus $\tau(z)=0$. We obtain a contradiction.  Since $2$ is invertible in $R$,  
Lemma \ref{fixedpointsinv} implies that the centre is contained in $\OO[[t_1, t_2, t_3]]$. The other inclusion holds by construction.
\end{proof}

\begin{cor}\label{weiterso1} $R$ is a free $\OO[[t_1, t_2, t_3]]$-module of rank $4$.  
\end{cor}
\begin{proof} If $0=\lambda_1 + \lambda_2 u + \lambda_3 v + \lambda_4(uv-vu)$ then it follows from Proposition 
\ref{weiterso} that all $\lambda_i=0$. The result then follows from Lemma \ref{expressa1}.
\end{proof}

\begin{cor}\label{uv-vunonteiler} Let $a\in R$ and suppose that $a(uv-vu)=0$ or $(uv-vu)a=0$ then $a=0$. 
\end{cor}
\begin{proof} Since $(uv-vu)^2=-(uv-vu)(uv-vu)^*$ is in $\OO[[t_1,t_2, t_3]]$ and is non-zero by Lemma \ref{uvvunon}, the assertion 
follows from Corollary \ref{weiterso1}.
\end{proof}     

\begin{cor}\label{weiterso2} $\wE\cong R^{op}$. In particular, the functor $\cV$ induces an equivalence of categories between 
$\dualcat(\OO)^{\BB}$ and the category of compact $R^{op}$-modules.
\end{cor}
\begin{proof} Proposition \ref{weiterso} says that $\varphi: \OO[[\GG]]\twoheadrightarrow R$ 
satisfies the conditions of Corollary \ref{quotientNGI} and thus we have $\Ker \varphi_{\cV} \subseteq \Ker \varphi$ and 
hence a surjection $\wE\twoheadrightarrow R^{op}$. Corollary \ref{uv-vunonteiler} implies that $R^{op}$ satisfies the
conditions of Lemma \ref{condNGI} with  $t'= uv-vu$, hence the surjection is an isomorphism. The last assertion follows from 
Proposition \ref{gabriel}.
\end{proof} 

\begin{cor}\label{weiterso4} Let $\mathcal Z$ be the centre of $R$, let $\nn$ be a maximal ideal of $\mathcal Z[1/p]$ with residue field $\kappa(\nn)$
and let $T_{\nn}: \GG \rightarrow \kappa(\nn)$ be the specialization at $\nn$ of the universal pseudocharacter $T$, see Proposition \ref{pseudoprop}.
If the image of $(uv-vu)(uv-vu)^{\ast}$ in $\kappa(\nn)$ is non-zero then 
$R\otimes_{\mathcal Z} \kappa(\nn)$ is a central simple $\kappa(\nn)$-algebra of dimension $4$. Moreover, $R\otimes_{\mathcal Z} \kappa(\nn)$
is a matrix algebra over $\kappa(\nn)$ if and only if $T_{\nn}$ is the trace of a $2$-dimensional representation of $\GG$ defined over $\kappa(\nn)$. 
\end{cor}  
\begin{proof} It follows from Corollary \ref{weiterso1} that $R\otimes_{\mathcal Z} \kappa(\nn)$ is  a $4$-dimensional $\kappa(\nn)$ algebra. Let 
$\mm$ be any maximal ideal of $C[1/p]$ containing $\nn$ and let $\kappa(\mm)$ be its residue field. The representation $\rho_{\mm}$ is absolutely irreducible, 
as part (iii) of Lemma \ref{preweiterso} is satisfied. Hence, $R\otimes_{\mathcal Z} \kappa(\mm)\cong \mathrm M_2(\kappa(\mm))$ the algebra of $2\times 2$ matrices over $\kappa(\nn)$. 
Thus the centre of $R\otimes_{\mathcal Z} \kappa(\nn)$ is a one dimensional $\kappa(\nn)$-vector space, which implies that $R\otimes_{\mathcal Z} \kappa(\nn)$ is a central simple $\kappa(\nn)$-algebra.
If $R\otimes_{\mathcal Z} \kappa(\nn)\cong \mathrm M_2(\kappa(\nn))$ then letting $\tau$ be the standard module, we obtain that $\tau\otimes_{\kappa(\nn)} \kappa(\mm)\cong \rho_{\mm}$ and 
hence $\tr \tau= \tr \rho_{\mm}=T_{\nn}$ by Proposition \ref{realizeT}. Conversely, if there exists a representation $\tau:\GG\rightarrow \GL_2(\kappa(\nn))$ such that $\tr \tau =T_{\nn}$ then 
$\tr \tau= \tr \rho_{\mm}$ and so $\tau$ is absolutely irreducible and the 
surjection $R\twoheadrightarrow \End_{\kappa(\nn)}(\tau)$ factors through $R\otimes_{\mathcal Z} \kappa(\nn)$ and is then an isomorphism, since both the source and the target are $4$-dimensional.
\end{proof}

\begin{cor}\label{weiterso4a}  Let $\mathcal Z$ be the centre of $R$, let $\nn$ be a maximal ideal of $\mathcal Z[1/p]$ with residue field $\kappa(\nn)$
and let $T_{\nn}: \GG \rightarrow \kappa(\nn)$ be the specialization at $\nn$ of the universal pseudocharacter $T$, see Proposition \ref{pseudoprop}.
If the image of $(uv-vu)(uv-vu)^{\ast}$ in $\kappa(\nn)$ is non-zero then the $\nn$-adic completion of $R[1/p]$ is an Azumaya algebra of rank $4$
over the $\nn$-adic completion of $\mathcal Z[1/p]$. Moreover, it 
is a matrix algebra over the $\nn$-adic completion of $\mathcal Z$ if and only if $T_{\nn}$ is the trace of a $2$-dimensional, absolutely irreducible representation of $\GG$ defined over $\kappa(\nn)$. 
\end{cor}  
\begin{proof} It follows from Corollaries \ref{weiterso3}, \ref{weiterso1} that the $\nn$-adic completion of $R$ is a free, rank $4$ module over the $\nn$-adic completion of $\mathcal Z$. The assertion follows from Corollary \ref{weiterso4} and the idempotent lifting Lemma.
\end{proof}

\begin{cor}\label{weiterso5} Let $\mathcal Z$ be the centre of $R$ and let $\nn$ be a maximal ideal of $\mathcal Z[1/p]$. 
If the image of $(uv-vu)(uv-vu)^{\ast}$ in $\mathcal Z[1/p]/\nn$ is zero then 
$R[1/p]/ \nn R[1/p]$ has at most $2$ non-isomorphic irreducible modules.
\end{cor}
\begin{proof} Let $R_1:= R[1/p]/\nn R[1/p]$, $L':=\mathcal Z[1/p]/\nn$ and let $\theta$ be the image of $uv-vu$ in $R_1$.
Let $V$ be an irreducible right $R_1$-module. It follows from \eqref{commuteanti} that $V\theta$ is an $R_1$-submodule of $V$.
Since the image of $(uv-vu)^2=-(uv-vu)(uv-vu)^*$ in $L'$ is zero, we deduce that $\theta^2=0$ and since $V$ is irreducible we get 
 $V\theta=0$. Thus $\Hom_{R_1}(R_1/R_1\theta, V)\neq 0$ and it is enough to show that $\dim_{L'} R_1/R_1\theta \le 2$.
It follows from Corollary \ref{weiterso1} that $R_1$ is a $4$-di\-men\-sio\-nal $L'$ vector space. Let 
$m_{\theta}: R_1\rightarrow R_1$, $a\mapsto a\theta$, then $\dim \Ker m_{\theta}+ \dim \Image m_{\theta}=4$ and since $\theta^2=0$
we have $\dim \Image m_{\theta}\le \dim \Ker m_{\theta}$. Thus $\dim_{L'} R_1/R_1\theta\le 2$.
\end{proof}  
 
\subsection{The centre and Banach space representations.}\label{CandBsp}
\begin{thm}\label{mainNGI} Let $\Pi$ be a unitary absolutely irreducible admissible $L$-Ba\-nach space representation  
with the central character $\zeta$. Suppose that the reduction of some open bounded $G$-invariant lattice in $\Pi$ 
contains $\pi$ as a subquotient then $\overline{\Pi}\subseteq \pi\oplus \pi$. 
\end{thm}
\begin{proof} By Proposition \ref{modulequotientnew} we may choose an open bounded $G$-invariant lattice $\Xi$ in $\Pi$ such 
that the natural map $\Hom_{\dualcat(\OO)}(\wP, \Xi^d)\wtimes_{\wE} \wP\rightarrow \Xi^d$ is surjective.          
 It follows from Corollaries \ref{weiterso1} and \ref{weiterso2} that the centre of $\wE$
is noetherian and $\wE$ is a finite module over its centre. Hence $\Xi\otimes_{\OO} k$ is of finite length by 
Corollary \ref{fgZfl} and  $\Hom_{\dualcat(\OO)}(\wP, \Xi^d)_L$ is finite dimensional over $L$.
Since the block of $\pi$ consists only of $\pi$ itself
we deduce that 
$$\overline{\Pi}\cong ((\Xi^d\otimes_{\OO} k)^{ss})^{\vee}\cong \pi^{\oplus m},$$ 
where $m$ is equal to the dimension of $\Hom_{\dualcat(\OO)}(\wP, \Xi^d)_L$ by Lemma \ref{mult=rank0}.  Since 
$\Pi$ is absolutely irreducible, $\Hom_{\dualcat(\OO)}(\wP, \Xi^d)_L$ is an absolutely irreducible right $\wE_L$-module by 
Proposition \ref{absirre}. Since $\wE\cong R^{op}$ we deduce from Corollary  \ref{dimle2} that the  dimension 
of $\Hom_{\dualcat(\OO)}(\wP, \Xi^d)_L$ is at most $2$. 
\end{proof}

\begin{cor}\label{cormainNGI} Let $\Pi$ be as in Theorem \ref{mainNGI} and suppose that $\overline{\Pi}\cong\pi$ then 
$\Pi\cong (\Indu{P}{G}{\psi})_{cont}$ for some continuous unitary character $\psi: T\rightarrow L^{\times}$ lifting $\chi$ and satisfying 
$\psi|_Z=\zeta$.
\end{cor}
\begin{proof} Let $\Xi$ be as in the proof of Theorem \ref{mainNGI}. Since $\overline{\Pi}\cong \pi$ 
we deduce from Lemma \ref{mult=rank0}
that $\Hom_{\dualcat(\OO)}(\wP, \Xi^d)$ is a free $\OO$-module of rank $1$. 
Hence, the action of $\wE$ on it factors through the action 
of $\wE^{ab}$. In particular, the element $t\in \wE$ defined  in \eqref{eNGI} kills $\Hom_{\dualcat(\OO)}(\wP, \Xi^d)$, and 
hence it follows from \eqref{eNGI} that we have an isomorphism $\Hom_{\dualcat(\OO)}(\wM, \Xi^d)\cong \Hom_{\dualcat(\OO)}(\wP, \Xi^d)$.
The assertion follows from Proposition \ref{QofM}.
\end{proof}

Let $\chi_1:\Qp^{\times}\rightarrow k^{\times}$ be a continuous character. Recall  that the block $\BB$ of 
$\pi:=\Indu{P}{G}{\chi_1\otimes\chi_1\omega^{-1}}$ consists of only one isomorphism class, Proposition \ref{blocksoverk}. So
$\Mod^{\mathrm{l\, fin}}_{G,\zeta}(\OO)^{\BB}$ is the full subcategory of $\Mod^{\mathrm{l\, fin}}_{G,\zeta}(\OO)$ consisting 
of representations with every irreducible subquotient isomorphic to $\pi$. Let $R^{\mathrm{ps}, \varepsilon \zeta}_{\chi}$ be 
the universal deformation ring parameterizing $2$-di\-men\-sio\-nal pseudocharacters of $\gal$ with determinant 
$\zeta\varepsilon$ lifting $\chi:=2\chi_1$ and let $T:\gal\rightarrow R^{\mathrm{ps}, \varepsilon \zeta}_{\chi}$ be the universal deformation of 
$\chi$.

\begin{cor}\label{NgIkill} The category $\Mod^{\mathrm{l\, fin}}_{G,\zeta}(\OO)^{\BB}$ is anti-equivalent to the category of right compact
$R_{\chi}^{\mathrm{ps}, \zeta\varepsilon}[[\gal]]/J$-modules, where $J$ is a closed two-sided ideal generated by 
$g^2-T(g)g + \varepsilon\zeta(g)$ for all $g\in \gal(p)$.
\end{cor} 
\begin{proof} By twisting we may assume that $\chi_1$ is trivial and $\zeta=\varepsilon^{-1}$, see the proof of 
Corollary \ref{thesameps}. We have shown in Corollary \ref{thesameps} that $T$ factors through $\GG$, the maximal pro-$p$ quotient of $\gal$. 
Corollary \ref{pst4} says that $R_{\chi}^{\mathrm{ps}, \zeta\varepsilon}[[\gal]]/J\cong R_{\chi}^{\mathrm{ps}, \zeta\varepsilon}[[\GG]]/J'$, 
where the ideal $J'$ is a closed two-sided ideal of $R_{\chi}^{\mathrm{ps}, \zeta\varepsilon}[[\GG]]$ defined by the same relations.
It follows from Proposition \ref{pseudoprop} and Corollary \ref{involute3} that 
$R_{\chi}^{\mathrm{ps}, \Eins}[[\GG]]/J'$  is the ring $R$ considered above. 
The assertion follows from Corollaries 
\ref{weiterso3}, \ref{weiterso2} and Proposition \ref{gabriel}. We also note that the involution 
$\ast$ induces an isomorphism between $R$ and $R^{op}$, so the category of right compact $R$
is equivalent to the category of left compact $R$-modules.   
\end{proof}

\begin{cor}\label{CngI} The centre of the category $\Mod^{\mathrm{ladm}}_{G, \zeta}(\OO)^{\BB}$ is naturally isomorphic to $R_{\chi}^{\mathrm{ps},\zeta\varepsilon}$.
\end{cor}
\begin{proof} Corollary \ref{weiterso4}, Corollary \ref{thesameps}, Proposition \ref{gabriel}.
\end{proof}

Let $\Ban^{\mathrm{adm}}_{G,\zeta}(L)^{\BB}$ be as in Proposition \ref{blockdecompB} and let $\Ban^{\mathrm{adm. fl}}_{G,\zeta}(L)^{\BB}$
be the full subcategory consisting of objects of finite length. Let $\Pi$ be in $\Ban^{\mathrm{adm. fl}}_{G,\zeta}(L)^{\BB}$, 
 and let $\md(\Pi):=\Hom_{\dualcat(\OO)}(\wP, \Theta^d)\otimes_{\OO} L$, where $\Theta$ is an open  bounded $G$-invariant lattice in $\Pi$. 
It follows from Proposition \ref{longproof} that $\md(\Pi)$ is a finite dimensional $L$-vector space with continuous $\wE$-action. 
Let $\nn$ be a maximal ideal in $R^{\mathrm{ps}, \varepsilon\zeta}_{\chi}[1/p]$, recall that 
$\Ban^{\mathrm{adm. fl}}_{G, \zeta}(L)^{\BB}_{\nn}$ is the full subcategory of $\Ban^{\mathrm{adm. fl}}_{G,\zeta}(L)^{\BB}$
consisting of those $\Pi$ such that $\md(\Pi)$ is killed by a power of $\nn$.

\begin{cor}\label{pups} We have an equivalence of categories 
$$\Ban^{\mathrm{adm. fl}}_{G,\zeta}(L)^{\BB}\cong 
\bigoplus_{\nn\in \MaxSpec R_{\chi}^{\mathrm{ps},\zeta \varepsilon}[1/p]}\Ban^{\mathrm{adm. fl}}_{G,\zeta}(L)^{\BB}_{\nn}.$$
The category $\Ban^{\mathrm{adm. fl}}_{G, \zeta}(L)^{\BB}_{\nn}$ is anti-equivalent to the category 
of modules of finite length of the  
$\nn$-adic completion of $(R_{\chi}^{\mathrm{ps}, \zeta\varepsilon}[[\gal]]/J)[1/p]$.
\end{cor} 
\begin{proof} Apply Theorem \ref{furtherDBan} with $\dualcat(\OO)=\dualcat(\OO)^{\BB}$.
\end{proof}

\begin{cor}\label{NgIirr} Suppose that the pseudo-character corresponding to a maximal ideal $\nn$ of 
$R_{\chi}^{\mathrm{ps}, \zeta\varepsilon}[1/p]$ is the trace of an absolutely irreducible representation of $\gal$ defined over the residue field of $\nn$ then 
the category $\Ban^{\mathrm{adm. fl}}_{G, \zeta}(L)^{\BB}_{\nn}$ is anti-equivalent to the category 
of modules of finite length of the  
$\nn$-adic completion of $R_{\chi}^{\mathrm{ps}, \zeta\varepsilon}[1/p]$. In particular, it contains only one irreducible object.
\end{cor} 
\begin{proof} Corollaries \ref{weiterso4a} and \ref{pups}. The last assertion follows from the fact that the only 
irreducible module is $R_{\chi}^{\mathrm{ps}, \zeta\varepsilon}[1/p]/\nn$.
\end{proof} 

Let $\nn$ be a maximal ideal of $R_{\chi}^{\mathrm{ps},\zeta \varepsilon}[1/p]$ with residue field $L$, let $T_{\nn}: \gal\rightarrow L$ be 
the pseudocharacter corresponding to $\nn$ and let 
$\Irr(\nn)$ denote the set (of equivalence classes of) irreducible objects in  
$\Ban^{\mathrm{adm. fl}}_{G,\zeta}(L)^{\BB}_{\nn}$. 

\begin{cor}\label{NgIred} If $T_{\nn}=\psi_1+\psi_2$ with $\psi_1, \psi_2: \gal\rightarrow L^{\times}$ continuous homomorphisms then 
$$\Irr(\nn)=\{ (\Indu{P}{G}{\psi_1\otimes \psi_2 \varepsilon^{-1}})_{cont}, (\Indu{P}{G}{\psi_2\otimes \psi_1 \varepsilon^{-1}})_{cont}\}.$$
\end{cor}
\begin{proof} Let $\mathcal Z$ be the centre of $\wE$. We may identify $\wE$ with $R$ and  $\mathcal Z$ with $R_{\chi}^{\mathrm{ps}, \zeta\varepsilon}$.
Corollary \ref{thesameps} says that 
the universal pseudocharacter is equal to the trace of the representation constructed in the proof 
of Lemma \ref{injectcentre}. In particular, if the image of $(uv-vu)(uv-vu)^*$ in $\mathcal Z[1/p]/\nn$
is non-zero, then $T_{\nn}$ is the trace of an absolutely irreducible $2$-di\-men\-sio\-nal representation, see the proof 
of Proposition \ref{weiterso}. Since $T_{\nn}=\psi_1+\psi_2$ we deduce that the image of $(uv-vu)(uv-vu)^*$
in $\mathcal Z[1/p]/\nn$ is zero. Corollary \ref{NgIkill} implies that for every $N$ in $\dualcat(\OO)$, $\cV(N)$ is 
killed by $g^2-T(g)g+\varepsilon \zeta(g)$, for all $g\in \gal$. Since
$$\VV((\Indu{P}{G}{\psi_1\otimes \psi_2 \varepsilon^{-1}})_{cont})=\psi_2, \quad  \VV((\Indu{P}{G}{\psi_2\otimes \psi_1 \varepsilon^{-1}})_{cont})=\psi_1,$$
both Banach space representations lie in $\Irr(\nn)$. If $\psi_1\neq \psi_2$ then the representations are non-isomorphic and we are done,  
since Corollary \ref{weiterso5} says that $\Ban^{\mathrm{adm. fl}}_{G,\zeta}(L)^{\BB}_{\nn}$
has at most $2$ irreducible objects. Suppose that $\psi_1=\psi_2$ and  $\Irr(\nn)$ contains an 
irreducible object $\Pi\not\cong (\Indu{P}{G}{\psi_1\otimes\psi_1\varepsilon^{-1}})_{cont}$. Then it follows from 
the proof of Corollary \ref{weiterso5} that $\md(\Pi)$ is one dimensional. By Corollary \ref{cormainNGI}, 
$\Pi$ is isomorphic to the parabolic induction of a unitary character, and thus must be contained in one 
of the components that we have handled already. Hence, if $\psi_1=\psi_2$ then $|\Irr(\nn)|=1$.
\end{proof}

\section{Non-generic case II}\label{nongenericcaseII}
In this section we deal with the case, where in Colmez's terminology the \textit{atome automorphe} consists
of three distinct irreducible representations. We assume throughout this section that $p\ge 5$. After twisting we may assume that our fixed central character is trivial 
and the block $\BB$ consists of $\Eins$, $\Sp$ and $\pi_{\alpha}:=\Indu{P}{G}{\alpha}$.  
The formalism developed in \S \ref{firstsec} does not work in the category $\dualcat(\OO)^{\BB}$. 
However, Colmez's functor kills off all the representations on which $\SL_2(\Qp)$ acts trivially
and so it is natural to work in the quotient category. We show in \S \ref{qcat} that the category of compact $\OO$-modules 
with the trivial $G$-action is a thick subcategory of $\dualcat(\OO)^{\BB}$ and the formalism 
of \S \ref{firstsec} applies in the quotient category $\qcat(\OO)^{\BB}$ to a projective envelope $\wP_{\pi_{\alpha}^{\vee}}$ 
of $\pi_{\alpha}^{\vee}$. Using Proposition \ref{ftof4} we show that $\cV$ induces 
a surjection $\varphi: \wE:=\End_{\dualcat(\OO)}(\wP_{\pi_{\alpha}^{\vee}})\twoheadrightarrow R^{\psi}_{\rho}$, where 
$\rho$ is the non-split extension $0\rightarrow \Eins\rightarrow \rho\rightarrow \omega \rightarrow 0$ and 
$R^{\psi}_{\rho}$ is the universal deformation ring of $\rho$ with a fixed determinant. The proof requires all kinds
of $\Ext$ calculations, which are carried out in \S \ref{hextII}, \S\ref{prepa}. (We suggest to skip them on first reading.)

The second difficulty is that $R^{\psi}_{\rho}$ is not formally smooth and hence we cannot use the same argument as in the 
generic case. The functor $\Indu{P}{G}{\Ord_{\overline{P}}}$ 
is left exact and we have a natural transformation to the identity functor. This induces a functorial filtration 
on every object of $\Mod^{\mathrm{ladm}}_{G, \zeta}(\OO)$ and dually on every object of $\dualcat(\OO)^{\BB}$ and by 
functoriality on  $\wE$. In \S \ref{filtration} we compare this filtration 
to the filtration on $R^{\psi}_{\rho}$ induced by powers of the ideal defined by the intersection of $R_{\rho}^{\psi}$
and the reducible locus in $R^{\psi}_{\rho}[1/p]$. We show in Theorem \ref{varphisoNGII} that $\varphi$ is an isomorphism 
and $\cV(\wP_{\pi_{\alpha}^{\vee}})$ is the universal deformation of $\rho$ with the fixed determinant.
In order to do this we need a good knowledge of the ring $R_{\rho}^{\psi}$. This is provided  by the appendix \S \ref{someDef} 
using results of B\"ockle \cite{bockle}. 

In \S \ref{Thecentre} we compute the endomorphism ring of $\wP_{\Eins^{\vee}}\oplus \wP_{\Sp^{\vee}}\oplus \wP_{\pi_{\alpha}^{\vee}}$
and show that its centre is naturally isomorphic to  $R_{\rho}^{\psi}$ and it is a finitely generated module over its 
centre. As a consequence we may describe $\dualcat(\OO)^{\BB}$ as a module category over an explicit ring. 
        
In \S \ref{Banachagain} we apply the theory of \S \ref{banach} to describe the category of 
admissible unitary $L$-Banach space representations of $G$ of finite length whose reduction mod $\varpi$ 
lies in $\Mod^{\mathrm{ladm}}_{G, \zeta}(k)^{\BB}$.

If $\pi$ and $\tau$ are smooth $k$-representations of $G$ on which $Z$ acts trivially, in order to simplify the notation we will 
write:
$$ e^i_{G/Z}(\pi,\tau):=\dim_k \Ext^i_{G/Z}(\pi, \tau).$$
If it is clear from the context that we are working with $G$-representations, then we will drop the index $G/Z$ and write $e^i(\pi, \tau)$ instead.
Simirlarly,  if $\pi$ and $\tau$ are representations of $T$ on which $Z$ acts trivially, we will let $e^i_{T/Z}(\pi, \tau):= \dim_k \Ext^i_{T/Z}(\pi, \tau)$. 

We assume all the way till \S \ref{Banachagain} that our fixed central character $\zeta$ is trivial. This is harmless 
since we may always twist to achieve this, see Lemma \ref{liftetatwist}. We recall that the representation $\pi(0,1)$, defined
in \eqref{p-1}, is the unique non-split extension of $\Sp$ by $\Eins$ with $2$-dimensional $I_1$-invariants. 

\subsection{\texorpdfstring{Higher $\mathrm{Ext}$-groups}{Higher $\mathrm{Ext}$-groups}}\label{hextII}
The dimensions of $\Ext^1_{G/Z}$ groups between irreducible representations in the block of the trivial representation, are given by: 
\begin{equation}
e^1(\Eins, \Eins)=0, \quad e^1(\Sp, \Eins)=1, \quad e^1(\Indu{P}{G}{\alpha}, \Eins)=1,
\end{equation}
\begin{equation}
e^1(\Eins, \Sp)=2, \quad e^1(\Sp, \Sp)=0, \quad e^1(\Indu{P}{G}{\alpha}, \Sp)=0,
\end{equation}
\begin{equation}
e^1(\Eins, \Indu{P}{G}{\alpha})=0, \quad e^1(\Sp, \Indu{P}{G}{\alpha})=1, \quad e^1(\Indu{P}{G}{\alpha}, \Indu{P}{G}{\alpha})=2,
\end{equation}
see Theorems 11.4 and 11.5 (ii) in \cite{ext2}. We are going to determine the dimensions of higher $\Ext$-groups.
It is shown in \cite[4.1.3]{ord2} that 
\begin{equation}\label{ordtriv}
\Ord_P \Eins =0, \quad \RR^1\Ord_P \Eins = \alpha^{-1},
\end{equation}
\begin{equation}\label{ordsp}
\Ord_P \Sp =\Eins, \quad \RR^1\Ord_P \Sp = 0.
\end{equation}
It follows directly from \eqref{ordseq}, \eqref{ordext3}  and \eqref{ordtriv} that
\begin{equation}\label{extindtriv}
e^i_{G/Z}(\Indu{P}{G}{\Eins}, \Eins)=0, \quad i\ge 0
\end{equation} 
and from \eqref{ordseq}, \eqref{ordext3}, \eqref{ordsp} and Corollary \ref{extToruschar} that
\begin{equation}
e^1_{G/Z}(\Indu{P}{G}{\Eins}, \Sp)=e^1_{T/Z}(\Eins, \Eins)=2, \quad e^2_{G/Z}(\Indu{P}{G}{\Eins}, \Sp)=e^2_{T/Z}(\Eins, \Eins)=1
\end{equation} 
and $e^i(\Indu{P}{G}{\Eins}, \Sp)=0$ for $i\ge 3$. 

\begin{prop}\label{RIi1}  $\RR^1\II(\Eins)\cong \II(\Indu{P}{G}{\alpha})$, $\RR^2\II(\Eins)\cong \II(\Indu{P}{G}{\alpha})$,  $\RR^3\II(\Eins)\cong \II(\Eins)$
and $\RR^i\II(\Eins)=0$ for $i\ge 4$.
\end{prop}
\begin{proof} The first assertion is given by \cite[11.2]{ext2}.
Since $I_1/Z_1$ is a Poincar\'e group of dimension $3$, see the proof of Corollary  \ref{vanish3}, 
$ H^3(I_1/Z_1, \Eins)$ is one dimensional  and $H^i(I_1/Z_1, \Eins)=0$ for $i\ge 4$. 
We deduce that $\RR^3\II(\Eins)\cong \II(\pi\otimes \mu)$, 
where $\pi=\Eins$ or $\pi=\Sp$ and $\mu: G\rightarrow k^{\times}$ is a smooth character, since all the 
$1$-di\-men\-sio\-nal modules of the  Hecke algebra $\HH$ are of this form. 
It follows from Proposition \ref{comphext} that $\Ext^3_{G/Z}(\pi\otimes \mu, \Eins)\neq 0$. Hence, 
$\pi\otimes\mu$ is in the block of $\Eins$ and so $\mu$ is trivial. If $\pi\cong \Sp$ then the 
same argument implies $\Ext^3_{G/Z}(\Indu{P}{G}{\Eins}, \Eins)\neq 0$, thus contradicting \eqref{extindtriv}.
It follows from \eqref{ordseq}, \eqref{ordext3} and \eqref{ordtriv} that $e^i(\Indu{P}{G}{\alpha}, \Eins)=0$ for $i\ge 4$ and  
\begin{equation}\label{extaltriv}  
e^3(\Indu{P}{G}{\alpha}, \Eins)=1, \quad e^2(\Indu{P}{G}{\alpha}, \Eins)=2.
\end{equation}
Since $\RR^1 \II(\Eins)\cong \II(\Indu{P}{G}{\alpha})$, Lemma \ref{comphext1} (i) implies that $\Ext^1_{\HH}(\II(\Indu{P}{G}{\alpha}), \RR^1\II(\Eins))$ is one dimensional.  Proposition \ref{comphext} and  \eqref{extaltriv} 
imply that  $\Hom_{\HH}(\II(\Indu{P}{G}{\alpha}), \RR^2\II(\Eins))$ is non-zero. Since 
$I_1/Z_1$ is a Poincar\'e group of dimension $3$ we have  
$$\dim H^2(I_1/Z_1, \Eins)=\dim H^1(I_1/Z_1, \Eins)=2.$$ 
As $\II(\Indu{P}{G}{\alpha})$ is irreducible and $2$-di\-men\-sio\-nal we obtain $\RR^2\II(\Eins)\cong \II(\Indu{P}{G}{\alpha})$.
\end{proof} 

\begin{cor}\label{extirest1}For $i\ge 2$, $e^i(\Eins, \Eins)=0$ and $e^i(\Sp, \Eins)=0$, except $e^3(\Eins, \Eins)=1$ and $e^4(\Sp, \Eins)=1$.
\end{cor} 
\begin{proof} The only non-zero $\Ext^i_{\HH}$ groups for $i\ge 1$ between $\II(\Eins)$ and  $\II(\Sp)$ 
are $\Ext^1_{\HH}(\II(\Eins), \II(\Sp))$ and  $\Ext^1_{\HH}(\II(\Sp), \II(\Eins))$, see Lemmas \ref{vanish2} and \ref{comphext1}, 
which are $1$-di\-men\-sio\-nal.
The assertion follows from Proposition \ref{RIi1} and Proposition \ref{comphext}. 
\end{proof}

\begin{lem}\label{indHi} Let $\chi:T\rightarrow k^{\times}$ be a smooth character, then $H^2(I_1/Z_1, \Indu{P}{G}{\chi})$ is $2$-di\-men\-sio\-nal  
and  $H^i(I_1/Z_1, \Indu{P}{G}{\chi})=0$, for $i\ge 3$. 
\end{lem}
\begin{proof} By restricting to $I_1$ we obtain 
 \begin{equation}
(\Indu{P}{G}{\chi})|_{I_1}\cong \Indu{I_1\cap P}{I_1}{\Eins} \oplus \Indu{I_1\cap P^s}{I_1}{\Eins}.
\end{equation}
Shapiro's lemma gives
\begin{equation}
H^i(I_1/Z_1, \Indu{P}{G}{\chi})\cong H^i((I_1\cap P)/Z_1, \Eins)\oplus H^i((I_1\cap P^s)/Z_1, \Eins).
\end{equation}
Since $(I_1\cap P)/Z_1\cong (I_1\cap P^s)/Z_1\cong \Zp \rtimes \Zp$ is a compact torsion-free $p$-adic analytic group 
of dimension $2$, the assertion follows from \cite[V.2.5.8]{laz} and \cite{serreprop}. 
 \end{proof}

\begin{cor}\label{comment7} Let Let $\chi:T\rightarrow k^{\times}$ be a smooth character, then $\RR^2\II(\Indu{P}{G}{\chi})$ is a $2$-di\-men\-sio\-nal  
$k$-vector space and  $\RR^i\II(\Indu{P}{G}{\chi})=0$, for $i\ge 3$. 
\end{cor}
\begin{proof} Lemma \ref{RisH} provides a natural isomorphism of $k$-vector spaces between $H^i(I_1/Z_1, \Indu{P}{G}{\chi})$ and $\RR^i\II( \Indu{P}{G}{\chi})$ and  the assertion follows from  Lemma \ref{indHi}.
\end{proof}

\begin{prop}\label{RIiSp} $\RR^1\II(\Sp)\cong \II(\Indu{P}{G}{\Eins})$, $\RR^2\II(\Sp)\cong \II(\Eins)$ and $\RR^i\II(\Sp)=0$ for $i\ge 3$.
\end{prop} 
\begin{proof} Proposition 11.2 of \cite{ext2} says that the natural maps induce  isomorphisms 
$\RR^1\II(\Indu{P}{G}{\Eins}) \cong \RR^1\II(\Eins)\oplus \RR^1\II(\Sp)$
and $\RR^1\II(\Sp)\cong \II(\Indu{P}{G}{\Eins})$. Hence, applying $\II$ to the exact sequence 
$0\rightarrow \Eins\rightarrow \Indu{P}{G}{\Eins}\rightarrow \Sp\rightarrow 0$, and observing that $\RR^3(\Indu{P}{G}{\Eins})$ vanishes by Corollary \ref{comment7}, we get 
\begin{equation}\label{longyeah} 
\RR^2\II(\Eins)\hookrightarrow \RR^2\II(\Indu{P}{G}{\Eins})\rightarrow \RR^2\II(\Sp)\twoheadrightarrow \RR^3\II(\Eins)
\end{equation} 
The first arrow in \eqref{longyeah} is a surjection, since both the source and the  target are $2$-di\-men\-sio\-nal, see Proposition \ref{RIi1}
and Corollary \ref{comment7} respectively. This implies 
the last arrow is an isomorphism.  Further, we deduce from Corollary \ref{comment7} and Proposition \ref{RIi1} that 
 $\RR^i\II(\Sp)\cong \RR^{i+1}\II(\Eins)=0$ for $i\ge 3$.
\end{proof}

\begin{prop}\label{resP} Let $U$ be in  $\Mod_{T/Z}^{\mathrm{l \, adm}}(k)$ then for all $i\ge 0$ we have an exact sequence
\begin{equation}\label{observe}
\Ext^i_{G/Z}(\Sp, \Indu{P}{G}{U})\hookrightarrow \Ext^i_{G/Z}(\Indu{P}{G}{\Eins}, \Indu{P}{G}{U})\twoheadrightarrow
\Ext^i_{G/Z}(\Eins, \Indu{P}{G}{U}).
\end{equation}
\end{prop}
\begin{proof} Recall that by Corollary \ref{thesame} it does not matter whether we compute the $\Ext$ groups in 
$\Mod^{\mathrm{sm}}_{G/Z}(k)$ or in $\Mod_{G/Z}^{\mathrm{l \, adm}}(k)$. If $V$ is in $\Mod^{\mathrm{sm}}_{G/Z}(k)$ then 
\begin{equation}
\Ext^i_{G/Z}(V, \Indu{P}{G}{U})\cong \Ext^i_{P/Z}(V, U),
\end{equation}
see \cite[4.2.1]{ord2}. Since the sequence $0\rightarrow \Eins\rightarrow \Indu{P}{G}{\Eins}\rightarrow \Sp\rightarrow 0$
splits, when restricted to $P$, we obtain the result. 
\end{proof}

\begin{cor}\label{injtriv} Let $\kappa$ be in $\Mod^{\mathrm{ladm}}_{T/Z}(k)$ such that
$\Hom_T(\chi, \kappa)=0$ for all $\chi\in \Irr_{T/Z}(k)$, $\chi\neq \Eins_T$. Then $e^i(\Sp, \Indu{P}{G}{\kappa})=0$ 
and $\Ext^i_{G/Z}(\Eins, \Indu{P}{G}{\kappa})\cong 
\Ext^i_{T/Z}(\Eins, \kappa)$ for all $i\ge 0$. 
\end{cor} 
\begin{proof} Suppose that $\kappa=J$ is injective in $\Mod^{\mathrm{ladm}}_{T/Z}(k)$. Then it follows from 
\eqref{ordseq} and \eqref{ordext3} that $\Ext^i_{G/Z}(\Indu{P}{G}{\Eins}, \Indu{P}{G}{J})=0$ 
for all $i\ge 1$. Proposition \ref{resP} implies that $\Indu{P}{G}{J}$ is acyclic for 
$\Hom_{G/Z}(\Eins, \ast)$ and $\Hom_{G/Z}(\Sp, \ast)$. Moreover, since $\RR^1\Ord_P \Sp=0$, the $U$-coinvariants 
$\Sp_U$ are zero  by \cite[3.6.2]{ord2}. Hence, $\Hom_{G}(\Sp, \Indu{P}{G}{J})=0$ and 
$\Hom_G(\Eins_G, \Indu{P}{G}{J})\cong \Hom_T(\Eins_T, J)$.

In general, let $\kappa\hookrightarrow J^{\bullet}$ be an injective resolution of $\kappa$ in 
$\Mod^{\mathrm{ladm}}_{T/Z}(k)$. Since the block of $\Eins_T$ contains only $\Eins_T$ itself, see Corollary \ref{extToruschar}, 
we may assume that for each $i\ge 1$ all $\Hom_T(\chi, J^i)=0$ for all $\chi\in \Irr_{T/Z}(k)$, $\chi\neq \Eins_T$. By inducing 
we obtain a resolution $\Indu{P}{G}{\kappa}\hookrightarrow \Indu{P}{G}{J^{\bullet}}$ by acyclic objects 
for functors $\Hom_G(\Eins, \ast)$, $\Hom_G(\Sp, \ast)$. Since $\Hom_{G}(\Sp, \Indu{P}{G}{J^i})=0$ and 
$\Hom_G(\Eins, \Indu{P}{G}{J^i})\cong \Hom_T(\Eins, J^i)$ we obtain the assertion.
\end{proof}

\begin{cor}\label{acyclic1Jal} Let $\kappa$ be in $\Mod^{\mathrm{ladm}}_{T/Z}(k)$ such that $\Hom_T(\chi, \kappa)=0$ for all 
$\chi\in \Irr_{T/Z}(k)$, $\chi\neq \alpha$, then $e^i(\Eins, \Indu{P}{G}{\kappa})=0$ for all $i\ge 0$. 
\end{cor}
\begin{proof} Frobenius reciprocity and the assumption on $\kappa$ imply the assertion for $i=0$. 
It is enough to show the assertion when $\kappa$ is injective in $\Mod^{\mathrm{ladm}}_{T/Z}(k)$, 
since then we may deduce the general case as in the proof of Corollary \ref{injtriv}. Suppose that $\kappa=J$ is injective,  
it follows from Proposition \ref{resP} and Lemma \ref{UindJ} that $\Ext^i$ vanishes for $i\ge 2$. 
It is enough to show the statement for $i=1$. We know that $\Ext^1_{G/Z}(\Eins,\Indu{P}{G}{\alpha})=0$, \cite[11.5]{ext2}.
Hence, if $U$ is any representation of finite length with irreducible subquotients isomorphic
to $\Indu{P}{G}{\alpha}$ then $\Ext^1_{G/Z}(\Eins, U)=0$. Since $ \Indu{P}{G}{J}$ is a union of 
subobjects of finite length with the irreducible subquotients isomorphic to  $\Indu{P}{G}{\alpha}$ we deduce the assertion.
\end{proof}

\begin{cor}\label{ext1al} We have $\Ext^i_{G/Z}(\Eins, \Indu{P}{G}{\alpha})=0$ for all $i\ge 0$. Moreover, 
\begin{equation}
e^2_{G/Z}(\Sp, \Indu{P}{G}{\alpha})=2, \quad e^3_{G/Z}(\Sp, \Indu{P}{G}{\alpha})=1,
\end{equation}
and $e^i_{G/Z}(\Sp, \Indu{P}{G}{\alpha})=0$ for $i\ge 4$.
\end{cor} 
\begin{proof} The first assertion follows  from Corollary \ref{acyclic1Jal}.  It follows from Proposition \ref{resP} that 
$\Ext^i_{G/Z}(\Sp, \Indu{P}{G}{\alpha})\cong \Ext^i_{G/Z}(\Indu{P}{G}{\Eins},  \Indu{P}{G}{\alpha})$, for all $i\ge 0$.
The last assertion follows from \eqref{specseqord}.
\end{proof}

\begin{cor}\label{forgotten}$\RR^1\II(\pi_{\alpha})\cong \II(\pi_{\alpha})\oplus \II(\pi(0,1))$, 
$\RR^2\II(\pi_{\alpha})\cong \II(\pi(0,1))$, $\RR^i\II(\pi_{\alpha})=0$ for $i\ge 3$.
\end{cor}
\begin{proof} The first assertion is  \cite[Thm. 7.16]{bp}. Since $\Ext^3_{G/Z}(\Sp, \Indu{P}{G}{\alpha})\neq 0$ by Corollary \ref{ext1al} and 
$\RR^3\II(\Indu{P}{G}{\alpha})=0$ by Corollary \ref{comment7}, Proposition \ref{comphext} implies that 
$\Ext^1_{\HH}(\II(\Sp),\RR^2\II(\Indu{P}{G}{\alpha}))\neq 0$. If $M$ is irreducible 
then
$\Ext^1_{\HH}(\II(\Sp), M)\neq 0$ implies that $M\cong \II(\Eins)$, see \cite[11.3]{ext2}.
Thus $\II(\Eins)$ is an irreducible subquotient of $\RR^2\II(\Indu{P}{G}{\alpha})$.  Since 
$\Ext^2_{G/Z}(\Eins, \Indu{P}{G}{\alpha})=0$, it follows from Proposition \ref{comphext} that 
$\II(\Eins)$ cannot be a submodule of  $\RR^2\II(\Indu{P}{G}{\alpha})$. If $M$ is irreducible 
then  $\Ext^1_{\HH}(\II(\Eins), M)\neq 0$ implies that $M\cong \II(\Sp)$, \cite[11.3]{ext2}.
Since by Corollary \ref{comment7} the underlying vector space of $\RR^2\II(\Indu{P}{G}{\alpha})$ is $2$-di\-men\-sio\-nal, 
we deduce that there exists a non-split sequence:
\begin{equation}
0\rightarrow \II(\Sp)\rightarrow \RR^2\II(\Indu{P}{G}{\alpha})\rightarrow \II(\Eins)\rightarrow 0
\end{equation}   
Now $\Ext^1_{\HH}(\II(\Eins), \II(\Sp))$ is one dimensional, \cite[11.3]{ext2}, and the only non-split 
extension is obtained by applying $\II$ to \eqref{p-1}.
\end{proof}

We  record below the dimensions of $\Ext^i_{G/Z}(\pi, \tau)$, where $\tau$, $\pi$ are  $\Eins$, $\Sp$ or $\Indu{P}{G}{\alpha}$. 
All the other $\Ext$-groups vanish. 

\begin{displaymath}
\begin{array}{ccccccccccccc}
\tau=\Eins&   &   &  &   & \tau=\Sp & & & & \tau=\Indu{P}{G}{\alpha}& \\
    i & 1 & 2 & 3 & 4& i&1 & 2 & 3& i & 1 & 2 & 3\\
\hline
\Eins & 0 & 0 & 1 & 0& \Eins& 2 & 2 & 0& \Eins & 0 & 0 & 0\\
\hline
\Sp &  1 & 0 & 0 & 1& \Sp& 0 & 0 & 1 & \Sp & 1 & 2 & 1\\
\hline 
\Indu{P}{G}{\alpha}& 1& 2 & 1 & 0&\Indu{P}{G}{\alpha} &0& 0 & 0& \Indu{P}{G}{\alpha} & 2 & 1 & 0\\
\end{array}
\end{displaymath}
Using the table one can construct minimal injective resolutions of $\Eins$, $\Sp$ and $\Indu{P}{G}{\alpha}$:

\begin{remar}\label{comment11}
Let $\kappa$ be an object of $\Mod^{\mathrm{l adm}}_{G/Z}(k)$ and $\iota: \soc_G \kappa\hookrightarrow J$ an 
injective envelope of $\soc_G \kappa$ in  $\Mod^{\mathrm{l adm}}_{G/Z}(k)$. Since $J$ is injective there exists 
$\phi: \kappa\rightarrow J$ such that the composition $\soc_G \kappa\rightarrow \kappa\rightarrow J$ is equal 
to $\iota$. Since $\iota$ is an injection,  we deduce that $\soc_G \Ker \phi =0$ and since $\Ker \phi$ is
an object of  $\Mod^{\mathrm{l adm}}_{G/Z}(k)$, we deduce that $\phi$ is injective. Since $\iota$ is essential, 
so is $\phi$ and hence for every irreducible object $\pi$ of $\Mod^{\mathrm{l adm}}_{G/Z}(k)$ we have 
$\Hom_{G/Z}(\pi, \kappa)\cong \Hom_{G/Z}(\pi, J)$ and thus $\Hom_{G/Z}(\pi, J/\kappa)\cong \Ext^1_{G/Z}(\pi, \kappa)$.
Hence, if we know the dimensions of $\Ext^1_{G/Z}(\pi, \kappa)$ for all irreducible $\pi$ then we may determine 
$\soc_G (J/\kappa)$ and thus construct the next step in the injective resolution. This way we obtain an injective resolution 
$\kappa\hookrightarrow J^{\bullet}$ such that for all irreducible $\pi$ in $\Mod^{\mathrm{l adm}}_{G/Z}(k)$ the complex 
$\Hom_G(\pi, J^{\bullet})$ has zero differentials.  In particular, $\Ext^i_{G/Z}(\pi, \kappa)\cong \Hom_{G}(\pi, J^i)$ for all $i\ge 0$.
Since in a locally finite category every injective object is determined by its socle up to isomorphism, the knowledge of 
$\Ext^i_{G/Z}(\pi, \kappa)$ for all irreducible $\pi$ determines $J^i$ up to isomorphism. It should be pointed out that these kind of arguments are standard in commutative algebra, see
for example \cite[Thm.18.5]{matsumura}.
\end{remar}

Using the table  and Remark \ref{comment11} we get:
\begin{equation}\label{res1NG}
0\rightarrow \Eins \rightarrow J_{\Eins} \rightarrow  J_{\Sp}\oplus J_{\pi_{\alpha}} \rightarrow J_{\pi_{\alpha}}^{\oplus 2}\rightarrow 
J_{\pi_{\alpha}}\oplus J_{\Eins}\rightarrow J_{\Sp}\rightarrow 0
\end{equation}
\begin{equation}\label{resSpNG}
0\rightarrow \Sp\rightarrow J_{\Sp} \rightarrow J_{\Eins}^{\oplus 2}\rightarrow J_{\Eins}^{\oplus 2} \rightarrow J_{\Sp} \rightarrow 0
\end{equation}
\begin{equation}\label{resalphaNG}
0\rightarrow \pi_{\alpha}\rightarrow J_{\pi_{\alpha}} \rightarrow J_{\Sp}\oplus J_{\pi_{\alpha}}^{\oplus 2}\rightarrow J_{\Sp}^{\oplus 2} \oplus 
J_{\pi_{\alpha}}\rightarrow J_{\Sp}\rightarrow 0
\end{equation}
where $\pi_{\alpha}=\Indu{P}{G}{\alpha}$ and $J_{\pi}$ denotes an injective envelope of $\pi$.

\subsection{Preparation}\label{prepa}
Since $e^1(\Eins,\Sp)=2$ there exists a unique smooth $k$-rep\-re\-sen\-ta\-tion $\tau_1$ of $G/Z$  such that 
$\Hom_G(\Eins, \tau_1)=0$ and we have an exact sequence:
\begin{equation}\label{defipi1}
0\rightarrow \Sp\rightarrow \tau_1\rightarrow \Eins\oplus \Eins \rightarrow 0.
\end{equation}
Applying $\Ord_P$ to \eqref{defipi1} and using \eqref{ordtriv}, \eqref{ordsp} we get 
\begin{equation}\label{ordpi1}
\Ord_P \tau_1\cong \Ord_P \Sp \cong \Eins, \quad 
\RR^1\Ord_P \tau_1\cong (\RR^1\Ord_P \Eins)^{\oplus 2}\cong (\alpha^{-1})^{\oplus 2}.
\end{equation}

\begin{lem}\label{ext1tau1} $e^1(\Eins, \tau_1)=0$, $e^1(\Sp, \tau_1)=2$, $e^1(\Indu{P}{G}{\alpha}, \tau_1)=2$. 
\end{lem}
\begin{proof} Since $e^1(\Eins,\Eins)=0$, we get the first claim by applying $\Hom_{G/Z}(\Eins, \ast)$ 
to \eqref{defipi1}. From \eqref{ordpi1} and the $5$-term sequence for $\Ord_P$, see \eqref{ordseq}, we get that
\begin{equation}
\Ext^1_{G/Z}(\Indu{P}{G}{\Eins}, \tau_1)\cong \Ext^1_{T/Z}(\Eins, \Eins)
\end{equation}
\begin{equation}
\Ext^1_{G/Z}(\Indu{P}{G}{\alpha}, \tau_1)\cong \Hom_{T/Z}(\alpha^{-1}, (\alpha^{-1})^{\oplus 2})
\end{equation}
are both  $2$-di\-men\-sio\-nal. Since $e^1(\Eins, \tau_1)=0$, by applying $\Hom_{G/Z}(\ast, \tau_1)$ to 
the exact sequence $0\rightarrow \Eins \rightarrow \Indu{P}{G}{\Eins}\rightarrow \Sp\rightarrow 0$ we deduce that
$\Ext^1_{G/Z}(\Sp, \tau_1)\cong \Ext^1_{G/Z}(\Indu{P}{G}{\Eins}, \tau_1)$.  
\end{proof} 

\begin{prop}\label{RItau1}
$\II(\tau_1)\cong \II(\pi(0,1))$, $\RR^1\II(\tau_1)\cong \II(\Sp)\oplus \II(\Indu{P}{G}{\alpha})^{\oplus 2}$, $\RR^2\II(\tau_1)\cong 
\II(\Eins)\oplus \II(\Indu{P}{G}{\alpha})^{\oplus 2}$.
\end{prop}
\begin{proof} We apply $\II$ to \eqref{defipi1}. Suppose that the connecting homomorphism 
$\partial: \II(\Eins)^{\oplus 2}\rightarrow \RR^1\II(\Sp)$ is zero. Then we would have an exact 
sequence of $\HH$-modules $0\rightarrow \II(\Sp)\rightarrow \II(\tau_1)\rightarrow \II(\Eins)^{\oplus 2}\rightarrow 0$.
Since $\Ext^1_{\HH}(\II(\Eins), \II(\Sp))=1$ by \cite[11.3]{ext2}, we would obtain 
$\Hom_G(\Eins, \tau_1)\cong \Hom_{\HH}(\II(\Eins), \II(\tau_1))\neq 0$ contradicting the construction of $\tau_1$.
Hence, $\partial$ is non-zero. Since $\RR^1\II(\Sp)\cong \II(\Indu{P}{G}{\Eins})$ by Proposition \ref{RIiSp}, the image of $\partial$
is $1$-di\-men\-sio\-nal. Hence, we obtain a non-split extension 
$0\rightarrow \II(\Sp)\rightarrow \II(\tau_1)\rightarrow \II(\Eins)\rightarrow 0$. Since the only non-split 
extension between $\II(\Sp)$ and $\II(\Eins)$ is realized by applying $\II$ to \eqref{p-1} we deduce that 
$\II(\tau_1)\cong \II(\pi(0,1))$. The cokernel of $\partial$ is isomorphic to $\II(\Sp)$. Hence, we obtain an exact sequence 
\begin{equation}\label{RItau12}
0\rightarrow \II(\Sp)\rightarrow \RR^1\II(\tau_1)\rightarrow \RR^1\II(\Eins)^{\oplus 2}.
\end{equation}
As $e^1(\Indu{P}{G}{\alpha}, \tau_1)=2$ by Lemma \ref{ext1tau1} and $\Ext^1_{\HH}(\II(\Indu{P}{G}{\alpha}), \II(\tau_1))=0$
by Lemma \ref{comphext1}, Proposition \ref{comphext} implies that $\dim \Hom_{\HH}(\II(\Indu{P}{G}{\alpha}), \RR^1\II(\tau_1))=2$.
Since $\RR^1\II(\Eins)\cong \II(\Indu{P}{G}{\alpha})$ by Proposition \ref{RIi1}, we deduce that the last arrow in 
\eqref{RItau12} is surjective. Since $\Ext^1_{\HH}(\II(\Indu{P}{G}{\alpha}), \II(\Sp))=0$ by Lemma \ref{comphext1},
we get $\RR^1\II(\tau_1)\cong \II(\Sp)\oplus \II(\Indu{P}{G}{\alpha})^{\oplus 2}$.  As $\RR^3\II(\Sp)=0$ by Proposition \ref{RIiSp},
we have an exact sequence: 
\begin{equation}\label{RItau13}
0\rightarrow \RR^2\II(\Sp)\rightarrow \RR^2\II(\tau_1)\rightarrow \RR^2\II(\Eins)^{\oplus 2}\rightarrow 0.
\end{equation}
Propositions \ref{RIiSp} and \ref{RIi1} give  $\RR^2\II(\Sp)\cong \II(\Eins)$ and $\RR^2\II(\Eins)\cong \II(\Indu{P}{G}{\alpha})$.
Lemma \ref{comphext1} implies that the sequence \eqref{RItau13} is split. This gives the last assertion.
\end{proof}

Since $e^i(\Eins, \Indu{P}{G}{\alpha})=0$ for $i\ge 0$, Corollary \ref{ext1al}, by applying $\Hom_G(\ast, \Indu{P}{G}{\alpha})$
to \eqref{defipi1}, we deduce that $\Ext^i_{G/Z}(\tau_1, \Indu{P}{G}{\alpha})\cong \Ext^i_{G/Z}(\Sp, \Indu{P}{G}{\alpha})$ for 
all $i\ge 0$. In particular, $e^1(\tau_1, \Indu{P}{G}{\alpha})=1$ and hence there exists a unique smooth $k$-representation 
$\tau_2$ of $G/Z$ such that $\Hom_G(\Sp, \tau_2)=0$ and there exists an exact sequence:
\begin{equation}\label{defipi2}
0\rightarrow \Indu{P}{G}{\alpha}\rightarrow \tau_2\rightarrow \tau_1\rightarrow 0.
\end{equation} 

\begin{lem}\label{ordpi2} $\Ord_P \tau_2\cong \alpha^{-1}$, $\RR^1\Ord_P \tau_2\cong \alpha^{-1}\oplus \alpha^{-1}$.
\end{lem}
\begin{proof} We apply $\Ord_P$ to \eqref{defipi2}. Since $\Hom_G(\Sp, \tau_2)=\Hom_G(\Eins, \tau_2)=0$ we  have   
$\Hom_T(\Eins, \Ord_P \tau_2)\cong \Hom_G(\Indu{P}{G}{\Eins}, \tau_2)=0$. Since $\Ord_P \tau_1=\Eins$ and 
there are no extensions between $\alpha^{-1}$ and $\Eins$, we deduce that the connecting homomorphism 
$\partial: \Ord_P \tau_1\rightarrow \RR^1\Ord_P (\Indu{P}{G}{\alpha})$ is injective. Since both the source and the target
are $1$-di\-men\-sio\-nal we deduce that $\partial$ is an isomorphism. Hence, $\Ord_P \tau_2\cong \Ord_P (\Indu{P}{G}{\alpha})$
and $\RR^1\Ord_P \tau_2\cong \RR^1\Ord_P \tau_1$. 
\end{proof} 

\begin{cor}\label{Indu1pi2} $e^i(\Indu{P}{G}{\Eins}, \tau_2)=0$, for $i\ge 0$.
\end{cor}
\begin{proof} Lemma \ref{ordpi2}, \eqref{ordseq}.
\end{proof}

\begin{lem}\label{smallextcompute} 
$e^1(\Eins, \tau_2)=0$, $e^1(\Sp, \tau_2)=e^2(\Sp, \tau_2)=0$, $e^1(\Indu{P}{G}{\alpha}, \tau_2)\le 4$.
\end{lem}
\begin{proof} Since $e^1(\Eins, \Indu{P}{G}{\alpha})=e^1(\Eins, \tau_1)=0$ we deduce that $e^1(\Eins, \tau_2)=0$. 
By applying $\Hom_{G/Z}(\ast, \tau_2)$ to the exact sequence 
$0\rightarrow \Eins\rightarrow \Indu{P}{G}{\Eins}\rightarrow \Sp\rightarrow 0$  and using Corollary \ref{Indu1pi2} 
we obtain $\Ext^i_{G/Z}(\Eins, \tau_2)\cong \Ext^{i+1}_{G/Z}(\Sp, \tau_2)$ for $i\ge 0$. Hence, 
$e^1(\Sp, \tau_2)=e^2(\Sp, \tau_2)=0$. Since $e^1(\Indu{P}{G}{\alpha}, \Indu{P}{G}{\alpha})=e^1(\Indu{P}{G}{\alpha}, \tau_1)=2$
the last assertion follows after  applying $\Hom_{G/Z}(\Indu{P}{G}{\alpha}, \ast)$ to \eqref{defipi2}.
\end{proof}

For a smooth character $\chi: T/Z\rightarrow k^{\times}$ we denote by $J_{\chi}$ its injective envelope in $\Mod^{\mathrm{l adm}}_{T/Z}(k)$.
We note that uniqueness of injective envelopes implies that $(J_{\chi})^s\cong J_{\chi^s}$ and $J_{\chi}\cong J_{\Eins} \otimes \chi$. 
 Let $J_{\Eins_G}$, $J_{\Sp}$ and $J_{\pi_{\alpha}}$ be injective envelopes of the trivial representation, $\Sp$ and $\pi_{\alpha}:=\Indu{P}{G}{\alpha}$ in $\Mod^{\mathrm{l adm}}_{G/Z}(k)$, respectively.

\begin{prop}\label{resindJNG} There exist exact sequences:
\begin{equation}\label{1JNG}
0\rightarrow \Indu{P}{G}{ J_{\Eins_T}}\rightarrow J_{\Eins_G}\rightarrow J_{\pi_{\alpha}}\rightarrow 0
\end{equation}
\begin{equation}\label{alphaJNG}
  0\rightarrow \Indu{P}{G}{ J_{\alpha}}\rightarrow J_{\pi_{\alpha}}\rightarrow J_{\Sp}\rightarrow 0
\end{equation} 
\begin{equation}\label{SpNG}
0\rightarrow (\Indu{P}{G}{J_{\Eins_T}})/\Eins_G\rightarrow J_{\Sp} \rightarrow J_{\pi_{\alpha}}^{\oplus 2}
\end{equation}
\end{prop}
\begin{proof} The injections in \eqref{1JNG} and \eqref{alphaJNG} follow from Proposition \ref{projectiveandord}. Lemma \ref{UindJ} 
gives $e^i(\Indu{P}{G}{\Eins}, \Indu{P}{G}{J_{\Eins}})=0$ for $i\ge 1$. Proposition \ref{resP} gives
$e^i(\Eins, \Indu{P}{G}{J_{\Eins}})=0$ and $e^i(\Sp,\Indu{P}{G}{J_{\Eins}})=0$ for $i\ge 1$. Lemma  \ref{UindJ} 
implies that $e^1(\Indu{P}{G}{\alpha}, \Indu{P}{G}{J_{\Eins}})=1$, and $e^i(\Indu{P}{G}{\alpha}, \Indu{P}{G}{J_{\Eins}})=0$
for all $i\ge 2$. This gives \eqref{1JNG}, see Remark \ref{comment11}. Similarly  we obtain \eqref{alphaJNG}, 
noting that $e^i(\Eins, \Indu{P}{G}{J_{\alpha}})=0$ for all $i\ge 0$, see Corollary \ref{acyclic1Jal}.

Applying $\Hom_{G/Z}(\Eins, \ast)$ to the exact sequence:
\begin{equation}\label{definequotientNG} 
0\rightarrow \Eins_G \rightarrow \Indu{P}{G}{J_{\Eins}}\rightarrow \kappa\rightarrow 0
\end{equation}
we get $e^i(\Eins, \kappa)=0$ for $i=0$ and $i=1$. (Here we are using  the fact that  $e^1(\Eins, \Eins)=e^2(\Eins, \Eins)=e^1(\Eins,  \Indu{P}{G}{J_{\Eins}})=0$.)
As $e^1(\Sp, \Eins)=1$ and $e^1(\Sp, \Indu{P}{G}{J_{\Eins}})=e^2(\Sp, \Eins)=0$ we get $\dim \Hom_{G/Z}(\Sp, \kappa)=1$ 
and $\Ext^1_{G/Z}(\Sp, \kappa)=0$. Since all the irreducible subquotients of $\Indu{P}{G}{J_{\Eins}}$ are either $\Eins_G$ or $\Sp$
we have $\Hom_{G/Z}(\Indu{P}{G}{\alpha}, \kappa)=0$. Moreover, $e^1(\Indu{P}{G}{\alpha}, \Eins)= e^1(\Indu{P}{G}{\alpha}, \Indu{P}{G}{J_{\Eins}})=1$, 
$e^2(\Indu{P}{G}{\alpha}, \Indu{P}{G}{J_{\Eins}})=0$ and thus $\Ext^1_{G/Z}(\Indu{P}{G}{\alpha}, \kappa)\cong \Ext^2_{G/Z}(\Indu{P}{G}{\alpha}, \Eins)$
is $2$-di\-men\-sio\-nal. Hence, we deduce the existence of \eqref{SpNG}.
\end{proof}

Let 
\begin{equation}\label{defikappa}
0\rightarrow \Eins\rightarrow \kappa\rightarrow \pi_{\alpha}\rightarrow 0
\end{equation}
be a non-split extension. Since $\Ext^1_{G/Z}(\pi_{\alpha}, \Eins)$ is one di\-men\-sio\-nal $\kappa$ is uniquely determined up to isomorphism. Applying   
$\Ord_P$ to \eqref{defikappa} we obtain:
\begin{equation}\label{kappa1}
\Ord_P \kappa=0, \quad \RR^1\Ord_P \kappa\cong \RR^1 \Ord_P \pi_{\alpha} \cong \Eins.
\end{equation}
It follows from \eqref{ordseq} that 
\begin{equation}\label{kappa2}
e^1(\pi_{\alpha}, \kappa)=0, \quad e^1(\Indu{P}{G}{\Eins}, \kappa)=1.
\end{equation}

\begin{lem}\label{helpkis1} $e^1(\kappa, \kappa)=0$, $e^1(\Sp, \kappa)=2$, $e^1(\kappa, \Sp)=2$.
\end{lem}
\begin{proof} The first assertion follows since $e^1(\Eins, \Eins)=0$, $e^1(\Eins, \pi_{\alpha})=0$ thus $e^1(\Eins, \kappa)=0$  
and $e^1(\pi_{\alpha},\kappa)=0$ by \eqref{kappa2}.
For the second apply $\Hom_{G/Z}(\ast, \kappa)$ to $0\rightarrow \Eins\rightarrow \Indu{P}{G}{\Eins}\rightarrow \Sp\rightarrow 0$ and 
use \eqref{kappa2}. Since  $e^1(\pi_{\alpha}, \Sp)=e^2(\pi_{\alpha}, \Sp)=0$ we have  $e^1(\kappa, \Sp)=e^1(\Eins, \Sp)=2$.
\end{proof}
 
\begin{lem}\label{helpkis} Let $\beta$ in $\Mod^{\mathrm{sm}}_{G/Z}(k)$ be such that $\soc_G \beta \cong \Sp$ and the semisimplification 
is isomorphic to $\Sp\oplus \Eins \oplus \pi_{\alpha}$ then $e^1(\beta,\beta)\le 3$.  
\end{lem} 
\begin{proof} Since $e^1(\pi_{\alpha}, \Sp)=e^2(\pi_{\alpha}, \Sp)=0$ there exists an exact sequence 
$0\rightarrow \Sp \rightarrow \beta\rightarrow \kappa\rightarrow 0$. Since $e^1(\Sp, \Sp)=e^2(\Sp, \Sp)=0$ we get 
$e^1(\Sp, \beta)=e^1(\Sp, \kappa)=2$. Since $e^1(\kappa, \kappa)=0$ we get $e^1(\kappa, \beta)=e^1(\kappa, \Sp)-e^0(\kappa, \kappa)=1$. 
Thus $e^1(\beta, \beta)\le e^1(\kappa,\beta)+e^1(\Sp, \beta)=3$.   
\end{proof}

\begin{remar} Using the bound of Lemma \ref{helpkis} and the results of Kisin \cite{kisin} one may show that $\VV$ induces an isomorphism 
between the deformation functors  of $\beta$ with a fixed central character and  $\VV(\beta)$ with a fixed 
determinant.
\end{remar}

\begin{lem}\label{wellknown} Let $\rG$ be a compact torsion-free $p$-adic analytic pro-$p$ group of dimension $d$ and let $\tau$ be in 
$\Mod^{\mathrm{sm}}_{\rG}(k)$ then there exists a natural isomorphism between $\Ext^d_{\rG}(\Eins,\tau)$ and the $\rG$-coinvariants $\tau_{\rG}$.
\end{lem}
\begin{proof} Since $H^0(\rG, \ast)\cong \Hom_{\rG}(\Eins, \ast)$ and $H^i(\rG, \ast)$ is the $i$-th derived functor 
of $H^0(\rG, \ast)$, \cite[\S 2.2]{serregal}, for all $i\ge 0$ we have a natural isomorphism of functors 
$\Ext^i_{\rG}(\Eins, \ast)\cong H^i(\rG, \ast)$. Since $\rG$ is compact torsion-free and $p$-adic analytic, 
it is a Poincar\'e group of dimension $d$, \cite[2.5.8]{laz}, 
\cite{serreprop}. Since $\rG$ is pro-$p$, it acts trivially on the dualizing module. If $\tau$ is finite then 
Poincar\'e duality induces an isomorphism $H^d(\rG, \tau)\cong H^0(\rG, \tau^*)^*\cong \tau_{\rG}$, \cite[I.4.5]{serregal}, 
where $\ast$ denotes $k$-linear dual. In general, we may write $\tau$ as a union of finite subrepresentations 
$\tau=\underset{\longrightarrow}{\lim}\, \tau_i$. We have 
$$ H^d(\rG, \tau)\cong \underset{\longrightarrow}{\lim}\, H^d(\rG, \tau_i)\cong \underset{\longrightarrow}{\lim}\,(\tau_i)_{\rG}\cong \tau_{\rG},$$
where the first isomorphism is given by \cite[I.2.2 Cor.2]{serregal}.
\end{proof} 

In Lemmas below $\kappa$ is the representation defined in \eqref{defikappa}.

\begin{lem}\label{H30} $\II(\kappa)\cong \II(\Eins)$, $\RR^3\II(\kappa)=0$.
\end{lem}
\begin{proof} Since $\Ext^1_{\HH}(\II(\pi_{\alpha}), \II(\Eins))=0$, Lemma \ref{comphext1}, we have $\II(\kappa)\cong \II(\Eins)$. 
Lemma  \ref{indHi} and Lemma  \ref{wellknown} imply that the $I_1/Z_1$-coinvariants of $\pi_{\alpha}$ are zero. Hence, $I_1/Z_1$-coinvariants
of $\kappa$ are also zero, since otherwise we would obtain a $I_1$-equivariant  splitting of \eqref{defikappa}, which would 
contradict $\II(\kappa)\cong \II(\Eins)$. Lemma \ref{wellknown} implies that $H^3(I_1/Z_1, \kappa)=0$ and it follows from Lemma \ref{RisH} that 
$\RR^3 \II(\kappa)=0$.
\end{proof} 
 
\begin{lem}\label{RIikappa} $\RR^1\II(\kappa)\cong \II(\pi(0,1))$, $\RR^2\II(\kappa)\cong \II(\Sp)$, $\RR^i\II(\kappa)=0$ for $i\ge 3$.
\end{lem}
\begin{proof} Lemmas \ref{H30} implies that $\RR^3\II(\kappa)=0$. It follows from Lemma \ref{vanish3} that $\RR^i\II(\kappa)=0$ for $i\ge 4$. 
Applying $\II$ to \eqref{defikappa} and using Proposition \ref{RIi1}, Lemma \ref{H30} and Corollary \ref{forgotten} we obtain an exact sequence: 
$$\RR^1\II(\kappa)\hookrightarrow \II(\pi_{\alpha})\oplus \II(\pi(0,1))\rightarrow \II(\pi_{\alpha})\rightarrow 
\RR^2\II(\kappa)\rightarrow \II(\pi(0,1))\twoheadrightarrow \II(\Eins).$$
It follows from Proposition \ref{comphext} and \eqref{kappa2} that $\Hom_{\HH}(\II(\pi_{\alpha}), \RR^1\II(\kappa))=0$, which implies the assertion.  
\end{proof}

\begin{lem}\label{p01k} $e^1(\pi(0,1), \kappa)=2$, $e^2(\pi(0,1), \kappa)=1$.
\end{lem}
\begin{proof} Using Lemmas \ref{vanish2}, \ref{comphext1} one obtains $\Ext^i_{\HH}(\II(\pi(0,1)), \II(\pi(0,1))$ is $1$-di\-men\-sio\-nal 
and $\Ext^i_{\HH}(\II(\pi(0,1)), \II(\Sp))=0$ for $i=0$, $i=1$. The assertion follows from Proposition \ref{comphext} and Lemma \ref{RIikappa}.
\end{proof}

\subsection{Quotient category}\label{qcat}
\begin{lem}\label{ext1110} Let $0\rightarrow \pi_1\rightarrow \pi_2\rightarrow \pi_3\rightarrow 0$ be an extension in $\Mod^{\mathrm{sm}}_{G/Z}(\OO)$
then  $G$ acts trivially on $\pi_1$ and $\pi_3$ if and only if it acts trivially on $\pi_2$.
\end{lem}
\begin{proof} Choose  $v$ in $\pi_2$ then the map $g\mapsto (g-1) v$ defines a group homomorphism $\psi:G\rightarrow (\pi_1, +)$.
Since $Z$ acts trivially on $\pi_2$, $\psi$ will factor through $G/Z \SL_2(\Qp)$. The order of $G/Z \SL_2(\Qp)$ is prime to $p$,
as $p>2$. Since every element of $\pi_2$ is killed by a power of $p$, we deduce that $\psi$ is zero. Hence, $G$ acts trivially on $\pi_2$.
The other implication is trivial.   
\end{proof}

Let $\mathfrak T(\OO)$ be the category of compact $\OO$-modules with the trivial $G$-action. It follows from Lemma \ref{ext1110} that 
$\mathfrak T(\OO)$ is a thick subcategory of $\dualcat(\OO)$ and hence we may build a quotient category $\qcat(\OO):=\dualcat(\OO)/\mathfrak T(\OO)$.
Recall, \cite[\S III.1]{gab}, that the objects of $\qcat(\OO)$ are the same as the objects of $\dualcat(\OO)$, the morphisms are given 
by 
\begin{equation}\label{defHomQ}
\Hom_{\qcat(\OO)}(M, N):=\underset{\longrightarrow}{\lim}\, \Hom_{\dualcat(\OO)}( M', N/N'),
\end{equation}
where the limit is taken over all subobjects $M'$ of $M$ and $N'$ of $N$ such that $G$ acts trivially on $M/M'$ and 
$N'$. Let $\TT:\dualcat(\OO)\rightarrow \qcat(\OO)$ be the functor $\TT M=M$ for every object of $\dualcat(\OO)$ and 
$\TT f: \TT M\rightarrow \TT N$ is the image of $f: M\rightarrow N$ in $\underset{\longrightarrow}{\lim}\, \Hom_{\dualcat(\OO)}( M', N/N')$
 under the natural map. The category $\qcat(\OO)$ is abelian and $\TT$ is an exact functor, \cite[Prop 1, \S III.1]{gab}.
In our situation it is easy to describe the homomorphisms in the quotient category explicitly. For  an object $M$ of $\dualcat(\OO)$, we denote 
by $I_G(M):= (M^{\vee}/ (M^{\vee})^{G})^{\vee}\subseteq M$. 

\begin{lem}\label{HomQexp} Let $M$ and $N$ be objects of $\dualcat(\OO)$, then $\Hom_{\dualcat(\OO)}(I_G(M), \Eins)=0$ and 
$(N/N^{G})^G=0$. In particular, 
\begin{equation}
\Hom_{\qcat(\OO)}(\TT M, \TT N)\cong \Hom_{\dualcat(\OO)}(I_G(M), N/N^G).
\end{equation} 
\end{lem}
\begin{proof} The first two assertions follow from Lemma \ref{ext1110}. 
Hence, it follows from the definition that 
$\Hom_{\qcat(\OO)}(\TT(I_G(M)), \TT(N/N^G))\cong \Hom_{\dualcat(\OO)}(I_G(M), N/N^G).$
Moreover, Lemme 4 in \cite[\S III.1]{gab} implies that the natural maps induce isomorphisms 
$\TT I_{G}(M)\cong \TT(M)$, $\TT N\cong \TT(N/N^G)$.
\end{proof} 

\begin{lem}\label{projQproj} If $P$ is a projective object of $\dualcat(\OO)$ with $\Hom_{\dualcat(\OO)}(P, \Eins)=0$ then 
$\TT P$ is a projective object of $\qcat(\OO)$ and 
$$\Hom_{\dualcat(\OO)}(P, N)\cong \Hom_{\qcat(\OO)}(\TT P, \TT N)$$ 
for all $N$.
\end{lem}
\begin{proof} Since $\Hom_{\dualcat(\OO)}(P, \Eins)=0$ we get $\Hom_{\dualcat(\OO)}(P, N^G)=0$. Since $P$ is projective we deduce 
$\Hom_{\dualcat(\OO)}(P, N)\cong \Hom_{\dualcat(\OO)}(P,N/N^G)$. The second assertion follows from Lemma \ref{HomQexp}.
The exactness of $\Hom_{\qcat(\OO)}(\TT P, \ast)$ follows from \cite[Cor 1, \S III.1]{gab}, which 
says that every exact sequence of $\qcat(\OO)$ is isomorphic to an exact sequence of the form $0\rightarrow \TT M_1\rightarrow 
\TT M_2 \rightarrow \TT M_3 \rightarrow 0$, where $0\rightarrow M_1\rightarrow M_2\rightarrow M_3\rightarrow 0$ is an exact sequence 
in $\dualcat(\OO)$. 
\end{proof}   

\begin{lem} The category $\qcat(\OO)$ has enough projectives. 
\end{lem}
\begin{proof} Let $M$ be in $\dualcat(\OO)$ and let $P\twoheadrightarrow I_G(M)$ be a projective envelope of $I_G(M)$ in $\dualcat(\OO)$. 
Since $\Hom_{\qcat(\OO)}(I_G(M), \Eins)=0$ by Lemma \ref{HomQexp} we also have $\Hom_{\qcat(\OO)}(P, \Eins)=0$. Thus $\TT P$ is projective 
in $\qcat(\OO)$ by Lemma \ref{projQproj} and since $\TT$ is exact we have $\TT P\twoheadrightarrow \TT I_G(M)\cong \TT M$.
\end{proof}

\begin{lem}\label{essQess} If $\Hom_{\dualcat(\OO)}(N, \Eins)=0$ then for every esssential epimorphism  
$q: M\twoheadrightarrow N$, $\TT q: \TT M\twoheadrightarrow \TT N$ is an essential epimorphism in $\qcat(\OO)$.
\end{lem}
\begin{proof} Let $a: T\rightarrow \TT M$ be a morphism in $\qcat(\OO)$ such that the composition $\TT q \circ a: T\rightarrow \TT N$ is 
an epimorphism. We claim that $a$ is an epimorphism. After replacing $T$ with the image of $a$ we may assume that 
$a$ is a monomorphism. It follows from \cite[Prop 1, \S III.1]{gab} that there exists a monomorphism 
$u: M'\rightarrow M$ in $\dualcat(\OO)$ such that $a: T\rightarrow \TT M$ is isomorphic to $\TT u : \TT M'\rightarrow \TT M$. 
Now  $\TT q \circ \TT u = \TT(q\circ u): \TT M'\rightarrow \TT N'$ is an epimorphism, and hence $G$ acts trivially on the cokernel 
of $q\circ u$ in $\dualcat(\OO)$, see Lemme 3 in \cite[\S III.1]{gab}. As $\Hom_{\dualcat(\OO)}(N, \Eins)=0$, we get that 
$q\circ u$ is an epimorphism, and since $q$ is essential, $u: M'\rightarrow M$ is an epimorphism, which implies that $\TT u$ (and hence 
$a$) is an epimorphism. 
\end{proof}

We note that the category $\qcat(\OO)$ is $\OO$-linear. Since $\TT$ is exact we have $(\TT M)[\varpi]\cong \TT(M[\varpi])$ and
$\TT M/\varpi \TT M\cong \TT( M/\varpi M)$. The composition $\dualcat(k)\rightarrow \dualcat(\OO)
\overset{\TT}{\rightarrow} \qcat(\OO)$ factors through the quotient category $\qcat(k):=\dualcat(k)/\mathfrak T(k)$  and induces an equivalence of 
categories between $\qcat(k)$ and the full subcategory of $\qcat(\OO)$ consisting of the objects killed by $\varpi$.
We denote by $T_{\Eins}$ and $T_{\alpha}$ the following objects of $\qcat(k)$:
\begin{equation}\label{defiT1Ta}
T_{\Eins}:= \TT ( \Indu{P}{G}{\Eins})^{\vee}, \quad T_{\alpha}:= \TT(\Indu{P}{G}{\alpha})^{\vee}.
\end{equation}
We note that since $\TT (\Eins) \cong 0$ in $\qcat(k)$ and since $\TT$ is exact  we have 
\begin{equation}
T_{\Eins}\cong \TT  \Sp^{\vee}\cong \TT \tau_1^{\vee},
\end{equation}
where $\tau_1$ is the representation defined by  \eqref{defipi1}. 
\begin{lem}\label{homT1NG}  
$\Hom_{\qcat(k)}(\TT M , T_{\Eins})\cong \Hom_G(\Indu{P}{G}{\Eins}, M^{\vee}),$ for all $M$  in $\dualcat(k)$.
\end{lem}
\begin{proof}
Since $\Ext^i_{G/Z}(\Indu{P}{G}{\Eins}, \Eins)=0$ for $i\ge 0$ by \eqref{extindtriv}, we have 
$$\Hom_G(\Indu{P}{G}{\Eins}, M^{\vee})\cong \Hom_G(\Indu{P}{G}{\Eins}, M^{\vee}/(M^{\vee})^G)\cong \Hom_{\qcat(k)}(\TT M , T_{\Eins}).$$ 
The last isomorphism follows from Lemma \ref{HomQexp}.
\end{proof} 

\begin{prop}\label{Qhyp} The hypotheses (H1)-(H5) hold in $\qcat(k)$ with $S=T_{\alpha}$ and $Q=\TT \tau_2^{\vee}$, where $\tau_2$ is the representation 
defined by \eqref{defipi2}.
\end{prop}
\begin{remar} The hypotheses (H1)-(H5) are stated in \S \ref{firstsec} assuming that $\dualcat$ is a full subcategory of $\Mod^{\mathrm{pro\, aug}}_G(\OO)$, where 
$G$ is a locally pro-$p$ group, but the statements make sense in any $k$-linear abelian category, such as $\qcat(k)$.
\end{remar}

\begin{proof}[Proof of Proposition \ref{Qhyp}] If $\pi_1$, $\pi_2$ are irreducible non-trivial in $\Mod^{\mathrm{l adm}}_{G/Z}(k)$ then it follows from Lemma \ref{HomQexp} that 
$\TT \pi_1^{\vee}$ and $\TT \pi_2^{\vee}$ are irreducible in $\qcat(k)$ and $\TT\pi_1^{\vee}\cong \TT\pi_2^{\vee}$ implies $\pi_1\cong \pi_2$. 
In particular, $T_{\Eins}$ and $T_{\alpha}$ are irreducible, non-zero  and non-isomorphic in $\qcat(k)$. Conversely, it follows from Lemma \ref{HomQexp}
that every irreducible non-zero object of $\qcat(k)$ is isomorphic to $\TT \pi^{\vee}$, where $\pi$ is an irreducible non-trivial 
representation in $\Mod^{\mathrm{l adm}}_{G/Z}(k)$. Let $J_{\pi_{\alpha}}$ be an injective envelope of $\Indu{P}{G}{\alpha}$ in 
$\Mod^{\mathrm{l adm}}_{G/Z}(k)$, then $P:=J^{\vee}_{\pi_{\alpha}}$ is a projective envelope of $(\Indu{P}{G}{\alpha})^{\vee}$ in $\dualcat(k)$ and 
it follows from  Lemmas \ref{projQproj} and \ref{essQess} that $\TT P$ is a 
projective envelope of $T_{\alpha}$ in $\qcat(k)$. It follows from Lemma \ref{smallextcompute} and Remark \ref{comment11}
that we have an exact sequence in $\Mod^{\mathrm{l adm}}_{G/Z}$: 
\begin{equation}\label{restau2}
 0\rightarrow \tau_2\rightarrow J_{\pi_{\alpha}}\rightarrow J^{\oplus d}_{\pi_{\alpha}}\rightarrow \kappa\rightarrow 0,
\end{equation}   
where $d=\dim \Ext^1_{G/Z}(\Indu{P}{G}{\alpha}, \tau_2)\le 4$ and $\Hom_G(\Indu{P}{G}{\Eins}, \kappa)=0$ by Corollary \ref{Indu1pi2}. By dualizing 
\eqref{restau2} and applying $\TT$ we get an exact sequence:
\begin{equation}\label{restau2VT}
 0\rightarrow \TT\kappa^{\vee}\rightarrow \TT P^{\oplus d}\rightarrow \TT P\rightarrow \TT \tau_2^{\vee}\rightarrow 0,
\end{equation}
Let $\pi$ be an irreducible  representation of $\Mod^{\mathrm{l adm}}_{G/Z}(k)$ with $\pi\not\cong \Eins$ and  $\pi\not\cong \Indu{P}{G}{\alpha}$. Since 
$\TT \pi^{\vee}$ is irreducible in $\qcat(k)$, is not isomorphic to $T_{\alpha}$  and $\TT P$ is a projective envelope of 
$T_{\alpha}$ in $\qcat(k)$, we deduce that $\Hom_{\qcat(k)}(\TT P, \TT \pi^{\vee})=0$. Applying $\Hom_{\qcat(k)}( \ast, \TT \pi^{\vee})$ 
to \eqref{restau2VT} we get that 
\begin{equation}\label{H13holds}
\Hom_{\qcat(k)}(\TT \tau_2^{\vee}, \TT \pi^{\vee})=0, \quad \Ext^1_{\qcat(k)}(\TT \tau_2^{\vee}, \TT \pi^{\vee})=0
\end{equation} 
\begin{equation}\label{H5holds}
\Ext^2_{\qcat(k)}(\TT \tau_2^{\vee}, \TT \pi^{\vee})\cong \Hom_{\qcat(k)}( \TT\kappa^{\vee}, \TT\pi^{\vee})
\end{equation}
It follows from \eqref{H13holds} that (H1) and (H3) hold. Dualizing \eqref{defipi2} 
and applying $\TT$ we get an exact sequence $0\rightarrow T_{\Eins}\rightarrow \TT \tau_2^{\vee} \rightarrow T_{\alpha}\rightarrow 0$. Since 
$T_\alpha\not\cong T_{\Eins}$, (H2) holds. Further, applying $\Hom_{\qcat(k)}(\ast, T_{\alpha})$ to \eqref{restau2VT}  we deduce that 
\begin{equation}\label{lateruse}
\dim \Ext^1_{\qcat(k)}(\TT \tau^{\vee}_2, T_{\alpha})\le d\le 4,
\end{equation}
hence (H4) holds.  Since $\Hom_G(\Indu{P}{G}{\Eins}, \kappa)=0$ we deduce from Lemma \ref{homT1NG} that $\Hom_{\qcat(k)}(\TT \kappa^{\vee}, T_{\Eins})=0$. Since $T_{\Eins}$ is  the maximal proper subobject of $\TT \tau_2^{\vee}$, 
 it follows from \eqref{H5holds} that (H5) is satisfied.
\end{proof} 

\begin{remar}\label{comment27} It follows from \eqref{restau2} that  the hypotheses  (H1)-(H4) hold in $\dualcat(k)$ with $S=\pi_{\alpha}^{\vee}$ and $Q=\tau_2^{\vee}$. The problem 
is that (H5) does not hold in $\dualcat(k)$: one may calculate using the results of \S \ref{hextII} that $\Ext^2_{G/Z}(\Eins, \tau_2)\cong\Ext^2_{G/Z}(\Eins, \tau_1)\cong \Ext^2_{G/Z}(\Eins, \Sp)\neq 0$. 
This implies that $\Ext^2_{G/Z}(\tau_1, \tau_2)\neq 0$ since $e^1(\Sp, \tau_2)=0$ by Lemma \ref{smallextcompute}. Dually we obtain that $\Ext^2_{\dualcat(k)}(\tau_2^{\vee}, \tau_1^{\vee})\neq 0$.
\end{remar}

\begin{lem}\label{dimextgrQ} $\Ext^1_{\qcat(k)}(T_{\Eins}, T_{\Eins})$, $\Ext^1_{\qcat(k)}(T_{\Eins}, T_{\alpha})$, $\Ext^1_{\qcat(k)}(T_{\alpha} ,T_{\alpha})$
are $2$-di\-men\-sio\-nal and $\Ext^1_{\qcat(k)}(T_{\alpha}, T_{\Eins})$ is $1$-di\-men\-sio\-nal.
\end{lem}
\begin{proof} Let $J_{\Sp}$, $J_{\pi_{\alpha}}$ be injective envelopes of $\Sp$ and $\pi_{\alpha}:= \Indu{P}{G}{\alpha}$ in 
$\Mod^{\mathrm{l adm}}_{G/Z}(k)$. It follows from Lemma \ref{ext1tau1} that we have an exact sequence:
\begin{equation}\label{resolvetau1}
0\rightarrow \tau_1\rightarrow J_{\Sp}\rightarrow J_{\Sp}^{\oplus 2} \oplus J_{\pi_{\alpha}}^{\oplus 2}
\end{equation}
Moreover, if we let $\kappa$ be the cokernel of the second arrow then the monomorphism 
$\kappa\hookrightarrow J_{\Sp}^{\oplus 2} \oplus J_{\pi_{\alpha}}^{\oplus 2}$ induced by the third arrow is  essential. Let $\pi$ be $\Sp$ or 
$\pi_{\alpha}$ then we know from Lemmas \ref{projQproj} and \ref{essQess} that $\TT J_{\pi}^{\vee}$ is a projective envelope of $\TT \pi^{\vee}$
in $\qcat(k)$. By dualizing \eqref{resolvetau1}, applying $\TT$ and then $\Hom_{\qcat(k)}(\ast, \TT \pi^{\vee})$ we obtain 
$$\Ext^1_{\qcat(k)}(T_{\Eins}, \TT \pi^{\vee})\cong \Hom_{\qcat(k)}(\TT \kappa^{\vee}, \TT \pi^{\vee})\cong 
\Hom_{\qcat(k)}( \TT J^{\vee}, \TT \pi^{\vee}),$$ 
where $J=J_{\Sp}^{\oplus 2 }\oplus  J_{\pi_{\alpha}}^{\oplus 2}$. The last isomorphism follows from the fact that $\TT \pi^{\vee}$ is irreducible, 
and $\TT J^{\vee}\twoheadrightarrow \TT \kappa^{\vee}$ is essential by Lemma \ref{essQess}. Hence $\Ext^1_{\qcat(k)}(T_{\Eins}, T_{\Eins})$ 
and $\Ext^1_{\qcat(k)}(T_{\Eins}, T_{\alpha})$ are $2$-di\-men\-sio\-nal. To calculate dimensions of $\Ext^1_{\qcat(k)}(T_{\alpha}, T_{\alpha})$ and 
$\Ext^1_{\qcat(k)}(T_{\alpha}, T_{\Eins})$ the same argument may be applied to \eqref{resalphaNG}.
\end{proof}
The functor $\cV: \dualcat(\OO)\rightarrow \Rep_{\gal}(\OO)$ kills the trivial representation and hence every object in $\mathfrak T (\OO)$. It 
follows from Corollaire 2 in \cite[\S III.1]{gab} that $\cV$ factors through $\TT: \dualcat(\OO) \rightarrow \qcat(\OO)$. We denote 
$\cV: \qcat(\OO)\rightarrow   \Rep_{\gal}(\OO)$ by the same letter. We have 
\begin{equation}\label{cVT}
 \cV(T_{\Eins})\cong \VV(\Sp)^{\vee}(\varepsilon)\cong \Eins, \quad \cV(T_{\alpha})\cong \VV(\Indu{P}{G}{\alpha})^{\vee}(\varepsilon)\cong \omega.
\end{equation}

\begin{lem}\label{inithyp} The functor $\cV$ induces an injection 
$$ \cV: \Ext^1_{\qcat(\OO)}(S_1, S_2)\hookrightarrow \Ext^1_{\OO[\gal]}(\cV(S_1), \cV(S_2)),$$
for $S_1,S_2\in \{ T_{\Eins}, T_{\alpha}\}$.
\end{lem}
\begin{proof} We interpret $\Ext^1$ as Yoneda $\Ext$ and the extension  $0\rightarrow S_2\rightarrow E\rightarrow S_1\rightarrow 0$
is mapped to $0\rightarrow \cV(S_2)\rightarrow \cV(E)\rightarrow \cV(S_1)\rightarrow 0$. If this extension splits, then 
$\cV(E)\cong \cV(S_1)\oplus \cV(S_2)$ is killed by $\varpi$. Since $\cV(S_1)$ and $\cV(S_2)$ are non-zero by \eqref{cVT} the exactness of $\cV$ implies 
that $E$ is killed by $\varpi$. Thus it is enough to show that $\cV$ induces an injection
$$\Ext^1_{\qcat(k)}(S_1, S_2)\hookrightarrow \Ext^1_{k[\gal]}(\cV(S_1), \cV(S_2)).$$
This assertion follows from the work of Colmez. We first treat the case $S_1\cong S_2$. Let $\chi:T/Z\rightarrow k^{\times}$ be 
a smooth character.
Since  $T/Z\cong \Qp^{\times}$ the space $\Ext^1_{T/Z}(\Eins, \Eins)\cong \Hom(\Qp^{\times}, k)$ is $2$-di\-men\-sio\-nal. Fix  $\tau\in \Hom(\Qp^{\times}, k)$
and let $Y_{\tau}$ be the corresponding extension of $\Eins$ by itself. Since parabolic induction is exact we have an exact sequence 
$$0\rightarrow \Indu{P}{G}{\chi}\rightarrow \Indu{P}{G}{Y_{\tau}\otimes \chi} \rightarrow \Indu{P}{G}{\chi}\rightarrow 0.$$
We denote $\pi_{\chi}:= \Indu{P}{G}{\chi}$. Since $\chi$ is trivial on $Z$ we may write it as $\chi=\chi_1^{-1}\otimes \chi_1$,  then 
$\VV(\pi_{\chi})\cong \chi_1 \omega$. It is shown in the proof of \cite[VII.4.14]{colmez} that the composition of  
$$\Hom(\Qp^{\times}, k)\rightarrow \Ext^1_{G/Z}(\pi_{\chi}, \pi_{\chi})\overset{\VV}{\rightarrow} \Ext^1_{k[\gal]}(\chi_1\omega, \chi_1\omega)\cong 
\Hom(\Qp^{\times}, k)$$ 
is the identity map. Using the anti-equivalence of categories, we obtain a surjection 
$$\Ext^1_{\dualcat(k)}(\pi_{\chi}^{\vee}, \pi_{\chi}^{\vee})\overset{\cV}{\twoheadrightarrow} 
\Ext^1_{k[\gal]}(\chi_1^{-1}, \chi_1^{-1})\cong  \Hom(\Qp^{\times}, k).$$ 
Since $\cV$ factors through $\TT$, we obtain a surjection 
$$\Ext^1_{\qcat(k)}(\TT \pi_{\chi}^{\vee}, \TT \pi_{\chi}^{\vee})\overset{\cV}{\twoheadrightarrow} 
\Ext^1_{k[\gal]}(\chi_1^{-1}, \chi_1^{-1})\cong\Hom(\Qp^{\times}, k).$$
When $\chi=\Eins$ or $\chi=\alpha$ we know by Lemma \ref{dimextgrQ} that the source is $2$-di\-men\-sio\-nal. Since the target is $2$-di\-men\-sio\-nal, the map
is an isomorphism. We deal with the case $S_1\not\cong S_2$ similarly. 

We claim that the map $\cV: \Ext^1_{\qcat(k)}(T_{\Eins}, T_{\alpha})\rightarrow \Ext^1_{k[\gal]}(\Eins, \omega)$ is surjective. 
For every non-zero smooth homomorphism  $\tau: \Qp^{\times}\rightarrow k$, 
Colmez constructs an extension $0\rightarrow \Sp\rightarrow E_\tau\rightarrow \Eins\rightarrow 0$, see \cite[VII.4.19]{colmez}, and shows 
that $\Ext^1_{G/Z}(\Indu{P}{G}{\alpha}, E_{\tau})$ is $1$-di\-men\-sio\-nal, \cite[VII.4.26]{colmez}, see also Lemma \ref{helpkis1}.
If we let $\epsilon_{\tau}$ 
 be a non-split extension $0\rightarrow E_{\tau}\rightarrow \Pi\rightarrow \Indu{P}{G}{\alpha}\rightarrow 0$, then $\VV(\epsilon_{\tau})$ 
defines an element of $\Ext^1_{k[\gal]}(\Eins, \omega)$.
It follows from \cite[VII.4.25]{colmez}, that the $\VV(\epsilon_{\tau})$ for different $\tau$ span the $2$-di\-men\-sio\-nal space  
$\Ext^1_{k[\gal]}(\Eins, \omega)$. Since $\TT E_{\tau}^{\vee}\cong T_{\Eins}$ we get our claim by applying $\cV$ to 
the extension $0\rightarrow T_{\alpha} \rightarrow \TT \Pi^{\vee}\rightarrow  \TT E_{\tau}^{\vee}\rightarrow 0$. 
Since $\Ext^1_{\qcat(k)}(T_{\Eins}, T_{\alpha})$ is $2$-di\-men\-sio\-nal by Lemma \ref{dimextgrQ} we deduce that $\cV$ induces an isomorphism. 

Finally, since $\Ext^1_{\qcat(k)}(T_{\alpha}, T_{\Eins})$ is $1$-di\-men\-sio\-nal, it is enough to produce an extension 
$0\rightarrow \Indu{P}{G}{\alpha}\rightarrow \Pi\rightarrow \Sp\rightarrow 0$, such that $0\rightarrow \Eins\rightarrow \VV(\Pi)\rightarrow 
\omega\rightarrow 0$ is non-split. We know that $\Ext^1_{G/Z}(\Sp, \Indu{P}{G}{\alpha})$ is $1$-di\-men\-sio\-nal, see \cite[11.5 (ii)]{ext2}. 
Let $0\rightarrow \Indu{P}{G}{\alpha}\rightarrow \Pi\rightarrow \Sp\rightarrow 0$ be a non-split extension. Applying $\Ord_P$ to 
it gives an isomorphism $\RR^1\Ord_P \Pi \cong \RR^1\Ord_P \Sp=0$. It follows from \cite[3.3.1]{ord2} that 
the space of $U$-coinvariants of $\Pi$ is zero. 
Since the space of $U$-coinvariants of $\Indu{P}{G}{\alpha}$ is $1$-di\-men\-sio\-nal, we deduce from \cite[VII.1.8]{colmez} 
that the space of $\Gal(\Qpbar/ \Qp^{ab})$-invariants of $\VV(\Pi)$ is $1$-di\-men\-sio\-nal, where $\Qp^{ab}$ is the maximal abelian extension 
of $\Qp$. Hence, $\VV(\Pi)$ can not be split.    
\end{proof}

Let $\BB=\{\Eins, \Sp^{\vee}, (\Indu{P}{G}{\alpha})^{\vee}\}$ be the block of the trivial representation. Let $\dualcat(\OO)^{\BB}$ be the full 
subcategory  of $\dualcat(\OO)$ consisting of all $M$ whose irreducible subquotients lie in $\BB$. It follows from \ref{blocks} that 
$\dualcat(\OO)^{\BB}$ is abelian and $\dualcat(\OO)\cong  \dualcat(\OO)^{\BB}\oplus \dualcat(\OO)_{\BB}$, where $\dualcat(\OO)_{\BB}$ is the full
subcategory of $\dualcat(\OO)$ consisting of those $M$ which no irreducible subquotient lies in $\BB$. Since $\mathfrak T(\OO)$ is contained 
in $\dualcat(\OO)^{\BB}$ we may build a quotient category $\qcat(\OO)^{\BB}:=\dualcat(\OO)^{\BB}/\mathfrak T(\OO)$ and 
we have an isomorphism of categories $\qcat(\OO)\cong \qcat(\OO)^{\BB}\oplus \dualcat(\OO)_{\BB}$.  Recall that $\Rep_{\gal}(\OO)$ is the category of continuous representations of $\gal$ on compact $\OO$-modules. Let $\Rep_{\gal}^{\BB}(\OO)$ be the full 
subcategory of $\Rep_{\gal}(\OO)$ with objects $\tau$  such that there exists $M$ in $\dualcat(\OO)^{\BB}$, such that $\tau\cong \cV(M)$. 

\begin{prop}\label{equivofcatsII} The functor $\cV$ induces an equivalence of categories between $\qcat(\OO)^{\BB}$ and $\Rep_{\gal}^{\BB}(\OO)$.
\end{prop}
\begin{proof} We note that since $\qcat(\OO)^{\BB}$ is a direct summand of $\qcat(\OO)$, for every object $M$ of
$\qcat(\OO)^{\BB}$ a projective envelope of $M$ in $\qcat(\OO)$ lies in $\qcat(\OO)^{\BB}$. This implies that 
if $M$ and $N$ are objects of $\qcat(\OO)^{\BB}$ then $\Ext^i_{\qcat(\OO)^{\BB}}(M, N)=\Ext^i_{\qcat(\OO)}(M, N)$ for all $i\ge 0$. 
It is enough to show that for $M$ and $N$ objects of $\qcat(\OO)^{\BB}$, $\cV$ induces a bijection
\begin{equation}\label{bijhom}
\Hom_{\qcat(\OO)}(M, N)\overset{\cong}{\rightarrow} \Hom_{\OO[\gal]}(\cV(M), \cV(N)),
\end{equation}   
where $\Hom_{\OO[\gal]}$ means morphisms in the category $\Rep_{\gal}(\OO)$. 
We may write $M\cong \underset{\longleftarrow}{\lim}\, M_i$ and 
$N\cong \underset{\longleftarrow}{\lim}\, N_j$, where the limit is taken over all the quotients of finite length. Then 
$\cV(M)\cong \underset{\longleftarrow}{\lim}\, \cV(M_i)$ and  $\cV(N)\cong \underset{\longleftarrow}{\lim}\, \cV(N_j)$, 
where $\cV(M_i)$ and $\cV(N_j)$ 
are of finite length. Now 
\begin{equation}\label{firstreduct}
\Hom_{\OO[\gal]}(\cV(M), \cV(N))\cong \underset{\longleftarrow}{\lim}\, \Hom_{\OO[\gal]}(\cV(M), \cV(N_j)).
\end{equation}
The kernels of $\cV(M)\rightarrow \cV(M_i)$ form a basis of open neighbourhoods of $0$ in $\cV(M)$. Since $\cV(N_j)$ is 
of finite length it  carries the discrete topology and hence  
every $\phi: \cV(M)\rightarrow \cV(N_j)$ in $\Rep_{\gal}(\OO)$ factors through $\cV(M_i)\rightarrow \cV(N_j)$ for some $i$. We 
obtain: 
\begin{equation}\label{secondreduct}
\Hom_{\OO[\gal]}(\cV(M), \cV(N_j))\cong \underset{\longrightarrow}{\lim}\, \Hom_{\OO[\gal]}(\cV(M_i), \cV(N_j)).
\end{equation}
Since \eqref{firstreduct} and \eqref{secondreduct} also hold for $M$ and $N$ in $\qcat(\OO)$, it is enough to verify \eqref{bijhom}
when $M$ and $N$ are of finite length. 

One may show that \eqref{bijhom} holds, when $M$ and $N$ are of finite length, by proving a stronger statement: 
\eqref{bijhom} holds and $\cV$ induces an injection 
\begin{equation}\label{injext}
\Ext^1_{\qcat(\OO)}(M, N)\hookrightarrow \Ext^1_{\OO[\gal]}(\cV(M), \cV(N)).
\end{equation}
The proof is  by induction on $\ell(M)+\ell(N)$, where $\ell$ denotes the number 
of irreducible subquotients, see the proof of 
Lemma A.1 in \cite{ext2}. Since the only irreducible objects in $\qcat(\OO)^{\BB}$ are $T_{\alpha}$ and $T_{\Eins}$ the initial induction step 
follows from Lemma \ref{inithyp}.
\end{proof} 

\begin{cor}\label{RepBab} The category $\Rep_{\gal}^{\BB}(\OO)$ is abelian.
\end{cor}

\begin{lem}\label{factor-loc-p}  If $M$ is an object of $\qcat(\OO)^{\BB}$ then the action of $\gal$ on $\cV(M)$ factors through 
$\Gal(F(p)|\Qp)$, where $F=\Qp(\mu_p)$ and $F(p)$ denotes the maximal pro-$p$ extension of $F$.
\end{lem}
\begin{proof} If $M$ is irreducible then $M\cong T_{\Eins}$ or $M\cong T_{\alpha}$, and it follows from \eqref{cVT} that  $\Gal(\Qpbar | F)$ acts trivially on $\cV(M)$.
If $M$ is of finite length then the cosocle filtration on $M$ induces a filtration of $\cV(M)$ such that $\Gal(\Qpbar | F)$ acts trivially on the graded pieces. 
This implies that the image of $\Gal(\Qpbar | F)$ in $\Aut_{\OO}(\cV(M))$ is a $p$-group, and hence $\Gal(\Qpbar | F(p))$ acts trivially on $\cV(M)$. 
The general case may be deduced form this by taking projective limits, as in the proof of Proposition \ref{equivofcatsII}.
\end{proof}

\begin{remar}\label{comment29} Lemma \ref{factor-loc-p} allows us to consider $\Rep_{\gal}^{\BB}(\OO)$ as the full subcategory of 
of $\Mod^{\mathrm{pro\, aug}}_{\Gal(F(p)| \Qp)}(\OO)$. Since $\Gal(F(p)| F)$ is an open pro-$p$ subgroup of $\Gal(F(p)| \Qp)$, this enables us to 
to apply the results of \S \ref{firstsec} with $\dualcat= \Rep_{\gal}^{\BB}(\OO)$. 
\end{remar}

Let $\wP\twoheadrightarrow (\Indu{P}{G}{\alpha})^{\vee}$ be a projective envelope of $(\Indu{P}{G}{\alpha})^{\vee}$ in $\dualcat(\OO)$. 
Then $\TT \wP\twoheadrightarrow T_{\alpha}$ is a projective envelope of $T_{\alpha}$ in $\qcat(\OO)$ by Lemma \ref{projQproj} and 
hence $\cV(\wP)\twoheadrightarrow \omega$ is a projective envelope of $\omega$ in $\Rep_{\gal}^{\BB}(\OO)$.  Let 
$$\wE:=\End_{\dualcat(\OO)}(\wP)\overset{\TT}{\cong} \End_{\qcat(\OO)}(\TT\wP)\overset{\cV}{\cong}\End_{\OO[\gal]}^{cont}(\cV(\wP)),$$
where the first isomorphism is given by Lemma \ref{projQproj} and the second by Proposition \ref{equivofcatsII}.

\begin{cor}\label{equivqbgal} The hypotheses (H0)-(H5) hold in $\Rep_{\gal}^{\BB}(\OO)$
with $S=\cV(T_{\alpha})\cong \omega$ and 
$Q=\cV(\TT \tau_2^{\vee})\cong \cV(\tau_2^{\vee})$, which is uniquely determined  up to isomorphism by the  non-split extension 
\begin{equation}\label{karakiri}
 0 \rightarrow \Eins\rightarrow \cV(\tau_2^{\vee})\rightarrow \omega\rightarrow 0.
\end{equation}
 \end{cor}
\begin{proof}  Since $\wP$  is $\OO$-torsion free by Corollary \ref{projaretfree} and $\cV$ is exact and $\OO$-linear, we deduce that the 
sequence $0\rightarrow \cV(\wP)\overset{\varpi}{\rightarrow} \cV(\wP)\rightarrow \cV(\wP/\varpi \wP)\rightarrow 0$ is exact. 
Hence, $\cV(\wP)$ is $\OO$-torsion free and so (H0) holds in $\Rep_{\gal}^{\BB}(\OO)$. 
The equivalence of categories established in Proposition \ref{equivofcatsII} and the $\Ext$- calculations  made 
in Proposition \ref{Qhyp} show that (H1)-(H5) hold in $\Rep_{\gal}^{\BB}(k)$, and hence in $\Rep_{\gal}^{\BB}(\OO)$ by Proposition \ref{H00}.
\end{proof}

\begin{cor}\label{NGIIEflat} The functor $\md\mapsto \md \wtimes_{\wE} \cV(\wP)$ is exact.
\end{cor}
\begin{proof} Since the hypotheses are satisfied by Corollary \ref{equivqbgal}, the assertion follows from the Corollary \ref{PisEflat}.
\end{proof}

\begin{lem}\label{Tcomwtimes} For a compact right $\wE$-module $\md$, we let  $\md \wtimes_{\wE} \TT \wP$ be an object of $\qcat(\OO)^{\BB}$ corresponding to $\md\wtimes_{\wE} \cV(\wP)$ under the equivalence of categories induced by $\cV$, 
see Proposition \ref{equivofcatsII}.
Then $\TT(\md \wtimes_{\wE} \wP)\cong \md \wtimes_{\wE} \TT \wP$.
\end{lem}
\begin{proof} This follows from Lemma \ref{VT2} and the fact that $\cV$ factors through $\TT$.
\end{proof} 

\begin{cor} $\SL_2(\Qp)$ acts trivially on $\wTor^1_{\wE}(k, \wP)$.
\end{cor}
\begin{proof} It follows from Corollary \ref{NGIIEflat} and Lemma \ref{Tcomwtimes} that $\TT(\wTor^1_{\wE}(k, \wP))=0$, which implies the assertion.
\end{proof}

\begin{defi}\label{Rpsi}
Let $R$ be the universal deformation ring of $\cV(\tau_2^{\vee})$ and let $R^{\psi}$ be the deformation ring parameterizing deformations of 
$\cV(\tau_2^{\vee})$ with determinant 
equal to  the cyclotomic character. Here we consider the usual deformations with commutative coefficients. 
\end{defi}
In the appendix \S \ref{someDef} we have recalled a construction  of an explicit presentation of $R$ and $R^{\psi}$ due to  B\"ockle, \cite{bockle}.

\begin{prop}\label{definephi} The functor $\cV$ induces  a surjection $\varphi:\wE\twoheadrightarrow R^{\psi}$.
\end{prop} 
\begin{proof} The intersection of maximal ideals of $R^{\psi}[1/p]$ corresponding to the irreducible representations is zero by Lemma \ref{A7}.
Moreover, it follows from Corollary  \ref{RpsirhoA} that $R^{\psi}$ is $\OO$-torsion free. Hence, the ring denoted by $R'$ in the statement of the 
Proposition \ref{ftof4} is equal to $R^{\psi}$. We will prove the assertion by modifying the proof of Proposition \ref{ftof4}. 

We note that $\cV(\tau_2^{\vee})\cong k\wtimes_{\wE}\cV(\wP)\cong \cV(k\wtimes_{\wE} \wP)$, and has only scalar endomorphisms. 
 Since $\cV(\wP)$ is $\wE$-flat by Corollary \ref{NGIIEflat}, $\wE^{ab}\wtimes_{\wE} \cV(\wP)$
is a deformation of $\cV(\tau_2^{\vee})$ to $\wE^{ab}$. Thus we obtain a natural map $\varphi: R\rightarrow \wE^{ab}$, where $R$ is the universal deformation ring 
of $\cV(\tau_2^{\vee})$. To show the surjectivity of $\varphi$ it is enough to show that it induces a surjection on tangent spaces, which is equivalent to 
showing that the natural map $\Ext^1_{\Rep_{\gal}^{\BB}(k)}(\cV(\tau_2^{\vee}), \cV(\tau_2^{\vee}))\rightarrow \Ext^1_{\gal}((\cV(\tau_2^{\vee}), \cV(\tau_2^{\vee}))$ is injective. 
This assertion follows from \eqref{injext}. Hence, $\varphi: R\twoheadrightarrow \wE^{ab}$ is surjective. We proceed as in the proof of Proposition \ref{ftof4}
to show that every closed point in $\mm\in \Spec R^{\psi}[1/p]$, corresponding to an irreducible representation lies in $\Spec \wE^{ab}$. This implies that 
$R\twoheadrightarrow R^{\psi}$ factors through $\varphi$. As explained in the proof 
of Proposition \ref{ftof4} it is enough to produce a map of $\OO$-algebras $x: \wE\rightarrow \kappa(\mm)$, such that $\kappa(\mm)\otimes_{\wE}\cV(\wP)$ 
is isomorphic to $\kappa(\mm)\otimes_{\wE} \rho^{un, \psi}$, where $\rho^{un, \psi}$ is the universal deformation with determinant $\psi$. 
It follows from \cite[2.3.8]{kisin} that part (iii) of Proposition \ref{ftof4} holds, and then the argument in the proof of Proposition \ref{ftof4} allows us to conclude.
\end{proof}

\begin{cor} Let $\wm$ be the maximal ideal of $\wE$ then $\dim \wm/(\wm^2+\varpi \wE)=4$.
\end{cor}
\begin{proof} It follows from Corollary \ref{RpsirhoA} that the  tangent space of $R^{\psi}$ is $4$-di\-men\-sio\-nal. Hence, Proposition 
\ref{definephi} implies that the tangent space of $\wE$ is at least $4$-di\-men\-sio\-nal.  By Proposition \ref{equivofcatsII} and since (H1) and (H3)  hold in $\qcat(k)$ we have: 
$$\Ext^1_{\Rep_{\gal}^{\BB}(k)}(\cV(\tau_2^{\vee}), \cV(\tau_2^{\vee}))\cong \Ext^1_{\qcat(k)}(\TT\tau_2^{\vee}, \TT\tau_2^{\vee})\cong \Ext^1_{\qcat(k)}(\TT\tau_2^{\vee}, T_{\alpha})$$
and is of dimension at most $4$ by \eqref{lateruse}. Hence it follows from 
Lemma \ref{tangentspace}  that the tangent space of $\wE$ is at most $4$-di\-men\-sio\-nal.
\end{proof} 

Let $\wP_{\alpha^{\vee}}$ be a projective envelope of $\alpha^{\vee}$ in $\dualcat_{T/Z}(\OO)$ and 
$\wM=(\Indu{P}{G}{(\wP_{\alpha^{\vee}})^{\vee}})^{\vee}$. All the irreducible subquotients of $\wM$ are isomorphic to $(\Indu{P}{G}{\alpha})^{\vee}$ 
and hence $\TT$ induces an isomorphism $\End_{\dualcat(\OO)}(\wM)\cong \End_{\qcat(\OO)}(\TT \wM)$ by Lemma \ref{HomQexp}. 
 Thus it follows from Proposition \ref{endoPOrd} that we have a natural surjection
\begin{equation}\label{defiidealaQ}
\wE\cong \End_{\qcat(\OO)}(\TT \wP)\twoheadrightarrow \End_{\qcat(\OO)}(\TT \wM)\cong \OO[[x, y]]
\end{equation}  
We let $\wa$ be the kernel of \eqref{defiidealaQ}, then $\wa$ is also the kernel of $\End_{\dualcat(\OO)}(\wP)\twoheadrightarrow \End_{\dualcat(\OO)}(\wM)$.

\begin{prop}\label{imageofa} The image of $\wa$ in $R^{\psi}$ is equal to $\rr:=R^{\psi}\cap  \bigcap_x \mm_x$ where the intersection is taken over 
all maximal ideals of $R^{\psi}[1/p]$  such that the corresponding representation $\rho_x$ is reducible. Moreover, 
$\wE/\wa \cong R^{\psi}/\varphi(\wa)$. 
\end{prop}
\begin{proof} Since we know that $R^{\psi}/\rr \cong \OO[[x,y]]$ by Corollary \ref{A6}, $\wE/\wa \cong \OO[[x,y]]$ by \eqref{defiidealaQ}, 
and $\varphi$ is surjective, it is enough to show that $\rr$ contains $\varphi(\wa)$. 

Let $x$ be a maximal ideal of $R^{\psi}[1/p]$ with residue 
field $L$ and let $\rho_x$ be the corresponding representation. Suppose that $\rho_x$ is reducible then since $\det \rho_x=\varepsilon$ 
we have an exact sequence $0\rightarrow \delta^{-1}\rightarrow \rho_{x}\rightarrow \delta\varepsilon\rightarrow 0$,
where $\delta: \gal\rightarrow L^{\times}$ is a continuous character, lifting the trivial character $\Eins: \gal\rightarrow k^{\times}$.  

Let $\chi: T\rightarrow L^{\times}$ be the character $\chi:=\delta \varepsilon\otimes \delta^{-1}\varepsilon^{-1}$. Then $\chi$ is trivial on $Z$ and is a deformation 
of $\alpha: T\rightarrow k^{\times}$ and hence defines a maximal ideal $y:\wE\rightarrow \End_{\dualcat(\OO)}(\wM)\rightarrow L$, 
such that $\Hom_{\OO}^{cont}(\OO\otimes_{\wE, y}\wM, L)\cong (\Indu{P}{G}{\chi})_{cont}$.
It follows from the construction of Colmez's functor that $\VV((\Indu{P}{G}{\chi})_{cont})\cong \delta^{-1}$ and hence 
$L\otimes_{\wE,y} \cV(\wM)\cong \cV(\OO\otimes_{\wE,y} \wM)_L\cong \delta\varepsilon$.
Since $\cV(\wP)$ is a free $\wE$-module of rank $2$ by Corollary \ref{ftof2}, $L\otimes_{\wE, y} \cV(\wP)$ is a $2$-di\-men\-sio\-nal $L$-representation 
of $\gal$ lifting $k\otimes_{\wE}\cV(\wP)\cong \cV(\tau_2^{\vee})$. Moreover, we know that $L\otimes_{\wE, y} \cV(\wP)$ admits 
$L\otimes_{\wE, y}\cV(\wM)\cong \delta\varepsilon$ as a quotient. Lemma \ref{moveL} implies that $\rho_x\cong L\otimes_{\wE, y} \cV(\wP)$, which implies that $y=\varphi^{-1}(x)$. 
Since by construction $y$ contains $\wa$, we deduce that $x$ contains $\varphi(\wa)$. Hence,
$\varphi(\wa)$ is contained in $R^{\psi}\cap  \bigcap_x \mm_x$ where the intersection is taken over 
all maximal ideals of $R^{\psi}[1/p]$ with residue field $L$  such that the corresponding representation $\rho_x$ is reducible. 
Remark \ref{enoughL} implies that this ideal is equal to $\rr$.
\end{proof}

\subsection{Filtration by ordinary parts}\label{filtration}
Let $P$ be a projective envelope of $(\Indu{P}{G}{\alpha})^{\vee}$ in $\dualcat(k)$ and let $E=\End_{\dualcat(k)}(P)$. 
Recall that uniqueness of projective envelopes implies the existence of  an isomorphism $P\cong \wP\otimes_{\OO} k$, 
and hence $E\cong \wE\otimes_{\OO} k$. Moreover, $\TT P$ is projective in $\qcat(k)$ and $\End_{\qcat(k)}(\TT P)\cong E$
by Proposition \ref{projQproj}. Since $\wE/\wa \cong \OO[[x,y]]$ is $\OO$-torsion free, we have an injection 
$\wa\otimes_{\OO} k\hookrightarrow E$, and we denote the ideal $\wa \otimes_{\OO} k$  by $\mathfrak a$.
We are going to show that $\varphi$, defined in Proposition \ref{definephi},  induces an isomorphism 
$\mathfrak a^n/\mathfrak a^{n+1}\overset{\cong}{\rightarrow}\varphi(\mathfrak a)^n/ \varphi(\mathfrak a)^{n+1}$, for all  $n\ge 1$.
 Using this we will show in Theorem \ref{varphisoNGII} that $\varphi$ is an isomorphism. 

\begin{lem}\label{vaca1} Let $\kappa$ be an object of $\Mod^{\mathrm{l adm}}_{T/Z}(k)$, let $\theta$ be a subspace of 
$(\Indu{P}{G}{\kappa})^G$ and  let $\tau$ be the quotient: 
\begin{equation}
0\rightarrow \theta \rightarrow \Indu{P}{G}{\kappa}\rightarrow \tau\rightarrow 0
\end{equation}
Then $\Ord_P \tau\cong \kappa^s$, $\RR^1\Ord_P \tau\cong (\kappa/\theta)\otimes \alpha^{-1}$, where we have identified $\theta$ with 
the subspace of $\kappa^T$ by evaluating at $1$.
\end{lem} 
\begin{proof} We note that evaluation at $1$ induces an isomorphism $(\Indu{P}{G}{\kappa})^G\cong \kappa^T$, which 
allows us to identify $\theta$ with a subspace of $\kappa^T$. It follows from 
\cite[4.1.1]{ord2} that $\Ord_P \theta =0$ and from the proof of \cite[4.1.2]{ord2} that we have a commutative diagram
\begin{displaymath}
\xymatrix@1{\RR^1\Ord_P \theta \ar[r]\ar[d]^{\cong}&  \RR^1\Ord_P \Indu{P}{G}{\kappa}\ar[d]^{\cong}\\ 
\;\theta\otimes \alpha^{-1}\;
\ar@{^(->}[r]& \kappa\otimes \alpha^{-1}}. 
\end{displaymath} 
Hence, $\Ord_P \tau\cong \Ord_P \Indu{P}{G}{\kappa}\cong \kappa^s$ and $\RR^1\Ord_P \tau\cong (\kappa/\theta)\otimes \alpha^{-1}$. 
\end{proof}

\begin{lem}\label{vaca2} Let $\tau$ be in  $\Mod^{\mathrm{l adm}}_{G/Z}(k)$ such that $\Hom_G(\pi', \tau)=0$ 
for all irreducible $\pi'$ not isomorphic 
to $\Eins$, $\Sp$ or $\Indu{P}{G}{\alpha}$. Then 
$G$ acts trivially on the kernel of the natural map $\Indu{\overline{P}}{G}{\Ord_P \tau}\rightarrow \tau$.
\end{lem} 
\begin{proof} We denote the  kernel by $K$. By construction we have $\Ord_P K=0$. 
If $\chi$ is irreducible in $\Mod^{\mathrm{l adm}}_{T/Z}(k)$ and $\chi\neq \Eins$, $\chi\neq \alpha^{-1}$, then 
$\Hom_T(\chi, \Ord_P \tau)\cong \Hom_G(\Indu{\overline{P}}{G}{\chi}, \tau)=0$, the first quality holding by the adjointness property  of $\Ord_P$, and the second by our assumption on $\tau$ together with 
Corollary \ref{tau12sub}.
 Since there are no extensions between $\alpha^{-1}$ and $\Eins$ in 
$\Mod^{\mathrm{l adm}}_{T/Z}(k)$, we deduce that $\Ord_{P}\tau\cong \kappa_{\Eins} \oplus \kappa_{\alpha^{-1}}$, where all the 
irreducible subquotients of $\kappa_{\Eins}$ are isomorphic to $\Eins$ and all the irreducible subquotients of   $\kappa_{\alpha^{-1}}$ are isomorphic to 
$\alpha^{-1}$. Hence, there are no non-zero homomorphisms between $\Indu{\overline{P}}{G}{\kappa_{\Eins}}$ and 
$\Indu{\overline{P}}{G}{\kappa_{\alpha^{-1}}}$
and so we may write $K=K_{\Eins} \oplus K_{\alpha^{-1}}$ where all the irreducible subquotients of $K_{\Eins}$ are $\Eins$ or $\Sp$ and all the irreducible 
subquotients of $K_{\alpha^{-1}}$ are $\Indu{\overline{P}}{G}{\alpha^{-1}}$. Since $\Ord_P K=0$ we get that $K_{\alpha^{-1}}=0$. 
Now $(K/K^G)^G=0$ by Lemma  \ref{ext1110}. Hence, if $K\neq K^G$ then we must have $\Hom_G(\Sp, K/K^G)\neq 0$. However, this implies that $K$ contains $\Sp$ or $\Indu{P}{G}{\Eins}$ 
as a subobject, which contradicts $\Ord_P K=0$.
\end{proof}   

\begin{lem}\label{technik} Let $J$ be an injective object in $\Mod^{\mathrm{l adm}}_{T/Z}(k)$ and $\tau$ an object 
of $\Mod^{\mathrm{l adm}}_{G/Z}(k)$. If $\Hom_T(\Ord_P \tau, J\otimes \alpha^{-1})=0$ then 
$\Ext^1_{G/Z}(\tau, \Indu{P}{G}{J})=0$.
\end{lem}
\begin{proof} Since by \eqref{ordinduced} and assumption $\Hom_T(\Ord_P \tau, \RR^1\Ord_P \Indu{P}{G}{J})=0$, by applying 
$\Ord_P$ to the extension $0\rightarrow \Indu{P}{G}{J}\rightarrow \kappa\rightarrow \tau\rightarrow 0$ 
we obtain an injection $J\otimes\alpha^{-1}\hookrightarrow \RR^1\Ord_P \kappa$. Since $J$ is injective 
the injection  splits. As $G=\GL_2(\Qp)$ we have  $\RR^1\Ord_P \kappa\cong \kappa_U \otimes \alpha$, where 
subscript $U$ denotes the coinvariants by the unipotent radical of $P$, \cite[3.6.2]{ord2}.  
Thus $\kappa_U\cong  (\Indu{P}{G}{J})_U \oplus \tau_U\cong J\oplus \tau_U$. Since 
$\Hom_G(\kappa, \Indu{P}{G}{J})\cong\Hom_T(\kappa_U, J)$
we obtain a splitting.
\end{proof}

On every $\tau$ in $\Mod^{\mathrm{l adm}}_{G/Z}(k)$ we define an increasing  filtration $\tau^{\bullet}$ by subobjects 
uniquely determined by  1) $\tau^0=0$ and 
2) $\gr^{i+1} \tau:=\tau^{i+1}/\tau^{i}$  is the image of $\Indu{\overline{P}}{G}{\Ord_P(\tau/\tau^i)}\rightarrow \tau/\tau^i.$
Dually on every $M$ in $\dualcat(k)$ we define a decreasing filtration $M^{\bullet}$ by subobjects 
$M^0=M$ and $M^i$ be the kernel of $M\rightarrow ((M^{\vee})^i)^{\vee}$ and let $\gr^i M:= M^i/M^{i+1}$.

\begin{lem}\label{vaca3} The filtration is functorial: for every $\phi:\tau \rightarrow \kappa$ in $\Mod^{\mathrm{ladm}}_{G/Z}(k)$ 
we have $\phi(\tau^i)\subseteq \kappa^i$ and for every $\psi: M\rightarrow N$ in $\dualcat(k)$ we have 
$\psi(M^i)\subseteq N^i$, for all $i\ge 0$. 
\end{lem}
\begin{proof} Trivially $\phi(\tau^0)\subseteq \kappa^0$. Suppose $\phi(\tau^i)\subseteq \phi(\kappa^i)$ then 
we get a map $\phi: \tau/\tau^i\rightarrow \kappa/\kappa^i$. The natural transformation 
$\Indu{\overline{P}}{G}{\Ord_{P}}\rightarrow \id$ induces a map $\gr^{i+1} \tau \rightarrow \gr^{i+1} \kappa$ and 
hence $\phi(\tau^{i+1})\subseteq \kappa^{i+1}$.
\end{proof}

\begin{lem} Let $J$  be an injective object in $\Mod^{\mathrm{l adm}}_{G/Z}(k)$. 
Then for $i\ge 1$ we have:
\begin{equation}\label{shift}
\Ord_P \gr^{i+1} J\cong \Ord_P J/J^i\cong \RR^1\Ord_P \gr^i J.
\end{equation}
\end{lem} 
\begin{proof}
From $\Indu{\overline{P}}{G}{\Ord_P (J/J^{i-1})}\twoheadrightarrow \gr^{i} J  \hookrightarrow J/J^{i-1}$ and 
left exactness of $\Ord_P$ we deduce that $\Ord_P \gr^i J \overset{\cong}{\rightarrow} \Ord_P J/J^{i-1}$, for all $i\ge 1$. This gives the first isomorphism in \eqref{shift}.
Since $J$ is injective we have $\RR^1\Ord_P J=0$ and since $\RR^2\Ord_P =0$, we get 
$\RR^1\Ord_P J/J^i=0$, for all $i\ge 1$. Thus applying $\Ord_P$ to 
$0\rightarrow \gr^i J\rightarrow J/J^{i-1}\rightarrow J/J^i\rightarrow 0$ we get  the isomorphism 
$ \Ord_P J/J^i\cong \RR^1\Ord_P \gr^i J$, for all $i\ge 1$.   
\end{proof}

\begin{lem}\label{theJ1}  $J^1_{\Eins_G}\cong \Indu{P}{G}{J_{\Eins_T}}$, $J^1_{\pi_{\alpha}}\cong \Indu{P}{G}{J_{\alpha}}$, 
$J^1_{\Sp}\cong (\Indu{P}{G}{J_{\Eins_T}})/\Eins_G$, where $J_{\Eins_T}$ and  $J_{\alpha}$ denote injective envelopes of $\Eins_T$ and $\alpha$ in $\Mod^{\mathrm{ladm}}_{T/Z}(k)$.
\end{lem} 
\begin{proof} The first two isomorphisms follow from the Propositions \ref{resindJNG} and \ref{projectiveandord} (i). Applying $\Ord_P$ to \eqref{alphaJNG}
gives us an isomorphism $\Ord_P(J_{\Sp})\cong \RR^1\Ord_P(\Indu{P}{G}{J_{\alpha}})\cong J_{\alpha}\otimes \alpha^{-1}\cong J_{\Eins_T}$.  The last isomorphism 
follows from \eqref{SpNG}.
\end{proof}

\begin{lem}\label{exactseqJ} For all $i\ge 0$ we have exact sequences $0\rightarrow J^1_{\Eins_G}\rightarrow J^{i+1}_{\Eins_G}\rightarrow 
J^i_{\pi_{\alpha}}\rightarrow 0$, $0\rightarrow J^1_{\pi_{\alpha}}\rightarrow J^{i+1}_{\pi_{\alpha}}\rightarrow 
J^i_{\Sp}\rightarrow 0$,  $0\rightarrow J^2_{\Eins_G}\rightarrow J^{i+2}_{\Eins_G}\rightarrow 
J^i_{\Sp}\rightarrow 0$.
\end{lem}
\begin{proof} By construction of the filtration for each $\tau$ in $\Mod^{\mathrm{ladm}}_{G/Z}(k)$ and $i, j\ge 0$ we have an 
isomorphism $(\tau/\tau^i)^j\cong \tau^{i+j}/\tau^i$. We apply this observation to \eqref{1JNG} and \eqref{alphaJNG}.
\end{proof}  

\begin{lem}\label{extired} $\gr^2 J_{\Sp}\cong \Indu{P}{G}{(J_{\alpha}/\alpha)}$, $e^1(\Eins_G, J^2_{\Sp})=0$.
\end{lem}
\begin{proof} It follows from \eqref{shift} and Lemma \ref{vaca1} that $\Ord_P \gr^2 J_{\Sp}\cong J_{\alpha^{-1}}/\alpha^{-1}$ and hence there is 
a surjection  $\Indu{P}{G}{(J_{\alpha}/\alpha)}\twoheadrightarrow  \gr^2 J_{\Sp}$. The injectivity of this map follows from Lemma \ref{vaca2} and \eqref{SpNG}.
Corollary \ref{acyclic1Jal} implies that 
$e^1(\Eins_G,\gr^2 J_{\Sp})=0$. It follows from \eqref{SpNG} that $e^1(\Eins_G, J^1_{\Sp})=0$. Hence, $e^1(\Eins_G, J^2_{\Sp})=0$.
\end{proof} 

\begin{lem}\label{conf1} Let $0\rightarrow \Eins \rightarrow \kappa\rightarrow \pi_{\alpha}\rightarrow 0$ be a non-split extension. There 
exists an exact sequence $0\rightarrow \kappa\rightarrow J^2_{\Eins_G}\rightarrow J^2_{\Sp}\rightarrow 0$.
\end{lem}
\begin{proof} We recall from \S\ref{prepa} that $\kappa$ is uniquely determined up to isomorphism, since $e^1(\pi_{\alpha}, \Eins)=1$. Since $\soc_G \kappa\cong 
\Eins$ there exists an injection $\iota: \kappa\hookrightarrow J_{\Eins_G}$. Since $\Eins$ occurs  as a subquotient of $\kappa$  
with  multiplicity one, $\Hom_G(\kappa, J_{\Eins_G})$ is $1$-di\-men\-sio\-nal and so the image of $\iota$ does not depend on 
the choice of $\iota$. Since $\Eins\hookrightarrow J_{\Eins_G}$ is essential $J^1_{\Eins_G}\cap \kappa\neq 0$. Since 
$J^1_{\Eins_G}\cong \Indu{P}{G}{J_{\Eins_T}}$ by Lemma \ref{theJ1}, it does not contain $\pi_{\alpha}$ as a subquotient and we deduce that $J^1_{\Eins_G}\cap \kappa=\Eins$ and hence we have an injection 
$\pi_{\alpha}\cong \kappa/\Eins \hookrightarrow J_{\Eins_{G}}/J^1_{\Eins_G}\cong J_{\pi_{\alpha}}$. Hence, 
$\kappa$ is contained in $J^2_{\Eins_G}$ and  since $J^2_{\Eins_G}/J^1_{\Eins_G}\cong J^1_{\pi_{\alpha}}\cong \Indu{P}{G}{J_\alpha}$ by Lemmas \ref{exactseqJ} and \ref{theJ1}, we obtain an exact sequence 
\begin{equation}\label{yetaneq}
 0\rightarrow \Indu{P}{G}{J_{\Eins_T}}/\Eins_G \rightarrow J^2_{\Eins_G}/\kappa\rightarrow \Indu{P}{G}{(J_{\alpha}/\alpha)}\rightarrow 0.
\end{equation}
Since $J_{\Sp}$ is injective there exist a map $J^2_{\Eins_G}/\kappa \rightarrow J_{\Sp}$ extending the injection 
$\Indu{P}{G}{J_{\Eins_T}}/\Eins_G\cong J^1_{\Sp}\hookrightarrow J_{\Sp}$.  
Since $e^1(\pi_{\alpha}, \kappa)=0$, see \eqref{kappa2}, and $\Hom_G(\pi_{\alpha}, J^2_{\Sp})=0$ we obtain  $\Hom_G(\pi_{\alpha}, J^2_{\Eins_G}/\kappa)=0$ and 
it follows from \eqref{yetaneq} that $\soc_G J^2_{\Eins_G}/\kappa\cong \soc_G \Indu{P}{G}{J_{\Eins_T}}/\Eins_G$, which implies that the map $J^2_{\Eins_G}/\kappa \rightarrow J_{\Sp}$, 
constructed above, is an embedding, as it induces an isomorphism on $G$-socles.
By applying $\Ord_P$ to \eqref{yetaneq} and using \eqref{shift} we obtain isomorphisms 
$$ \Ord_P ((J^2_{\Eins_G}/\kappa)/J^1_{\Sp})\cong \RR^1\Ord_P J^1_{\Sp}\cong \Ord_P (J_{\Sp}/J^1_{\Sp}),$$
and so $J^2_{\Eins_G}/\kappa\cong J^2_{\Sp}$. Thus \eqref{yetaneq} coincides with the tautological exact sequence $0\rightarrow J^1_{\Sp}\rightarrow J^2_{\Sp}\rightarrow \gr^2 J_{\Sp}\rightarrow 0$.
\end{proof}

\begin{lem}\label{conf2} $e^1(\pi(0,1), J^2_{\Sp})=1$.
\end{lem} 
\begin{proof} Combining \eqref{1JNG} and \eqref{alphaJNG} we obtain an exact sequence 
$0\rightarrow J^2_{\Eins_G}\rightarrow J_{\Eins_G}\rightarrow J_{\Sp}\rightarrow 0$. Hence, 
$e^1(\pi(0,1), J^2_{\Eins_G})=1$ and $e^2(\pi(0,1), J^2_{\Eins_G})=0$. Since $\Sp$ occurs only once as a subquotient of $\pi(0,1)$
we have $e^0(\pi(0,1), J_{\Sp})=1$. 
It follows from \eqref{SpNG}  and Lemma \ref{theJ1} that any map $\pi(0,1)\rightarrow J_{\Sp}$ has image lying in $J^1_{\Sp}$, and hence in $J^2_{\Sp}$. Thus
$e^0(\pi(0,1), J^2_{\Sp})=1$. We apply $\Hom_G(\pi(0,1), \ast)$ to the exact sequence of Lemma \ref{conf1}.
Since $e^1(\pi(0,1), \kappa)=2$, Lemma \ref{p01k}, we obtain an isomorphism $\Ext^1_{G/Z}(\pi(0,1), J^2_{\Sp})\cong 
\Ext^2_{G/Z}(\pi(0,1), \kappa)$. The assertion follows from Lemma \ref{p01k}.
\end{proof} 

It follows from Lemma \ref{extired} and \eqref{shift} that $\Ord_P \gr^3 J_{\Sp}\cong J_{\Eins_T}/\Eins_T$. Moreover, 
since $e^1(\Eins_G, J^2_{\Sp})=0$ by Lemma \ref{extired}, we have an exact sequence 
\begin{equation}
0\rightarrow (\Indu{P}{G}{(J_{\Eins_T}/\Eins_T)})^G\rightarrow \Indu{P}{G}{(J_{\Eins_T}/\Eins_T)}\rightarrow 
\gr^3 J_{\Sp}\rightarrow 0.
\end{equation}
In particular, $\Hom_G(\Eins, \gr^3 J_{\Sp})=\Hom_G(\pi_{\alpha}, \gr^3 J_{\Sp})=0$ and
$$\Hom_G(\Sp, \gr^3 J_{\Sp})\cong \Hom_G(\Indu{P}{G}{\Eins}, \gr^3 J_{\Sp})\cong \Hom_T(\Eins, J_{\Eins_T}/\Eins_T),$$
where the last isomorphism follows from Lemma \ref{vaca1}. Hence, $\soc_G \gr^3 J_{\Sp}\cong \Sp^{\oplus 2}$ and we have an 
isomorphism
\begin{equation}\label{J3soc}
 \gr^3 J_{\Sp}/\soc_G \gr^3 J_{\Sp} \cong \Indu{P}{G}{(J_{\Eins_T}/\soc^2_T J_{\Eins_T})}.
\end{equation}
Since $e^1(\Eins, \Eins)=e^2(\Eins, \Eins)=0$ we have isomorphisms:
$$\Ext^1_{G/Z}(\Eins, \gr^3 J_{\Sp})\cong \Ext^1_{G/Z}(\Eins, \Indu{P}{G}{(J_{\Eins_T}/\Eins_T)})\overset{\ref{injtriv}}{\cong} 
\Ext^1_{T/Z}(\Eins, J_{\Eins_T}/\Eins_T).$$
In particular, $e^1_{G/Z}(\Eins, \gr^3 J_{\Sp})=e^2_{T/Z}(\Eins, \Eins)=1$. Using \eqref{J3soc} we deduce that 
$$\Hom_G(\Eins, \gr^3 J_{\Sp}/\soc_G \gr^3 J_{\Sp})\cong \Hom_T(\Eins, J_{\Eins_T}/\soc^2_T J_{\Eins_T})$$
 is $3$-di\-men\-sio\-nal. Since $e^1(\Eins, \soc_G \gr^3 J_{\Sp})=e^1(\Eins, \Sp^{\oplus 2})=4$ we deduce that the natural map 
\begin{equation}\label{conf3} 
\Ext^1_{G/Z}(\Eins, \gr^3 J_{\Sp})\rightarrow \Ext^1_{G/Z}(\Eins, \gr^3 J_{\Sp}/\soc_G \gr^3 J_{\Sp})
\end{equation} 
is zero.
\begin{prop}\label{firstzero} $e^1(\Eins, J^3_{\Sp})=0$.
\end{prop}
\begin{proof}
 Let $\tau$ be the subrepresentation of $J^3_{\Sp}$ such that $\tau$ contains  $J^2_{\Sp}$ and 
$\tau/J^2_{\Sp}\cong \soc_G \gr^3 J_{\Sp}$. In particular, $J^3_{\Sp}/\tau = \gr^3 J_{\Sp}/\soc_G \gr^3 J_{\Sp}$. We have a commutative diagram: 
\begin{displaymath}
\xymatrix@1{\Ext^1_{G/Z}(\Eins, J^3_{\Sp})\ar@{^(->}[d]^{\ref{extired}}\ar[r]& 
\Ext^1_{G/Z}(\Eins, J^3_{\Sp}/\tau)  \ar[d]^{=}\\ 
\Ext^1_{G/Z}(\Eins, \gr^3 J_{\Sp})\ar[r]^{0}_{\eqref{conf3}}& 
\Ext^1_{G/Z}(\Eins, J^3_{\Sp}/\tau) .}
\end{displaymath} 
Hence, the top horizontal arrow is zero and we obtain an exact sequence:
\begin{equation}
0\rightarrow \Hom_{G}(\Eins, J^3_{\Sp}/\tau)\rightarrow \Ext^1_{G/Z}(\Eins, \tau)\rightarrow \Ext^1_{G/Z}(\Eins, J^3_{\Sp})\rightarrow 0.
\end{equation}
As $(J^3_{\Sp}/\tau)^G\cong (J_{\Eins_T}/\soc^2_T J_{\Eins_T})^T$ is $3$-dimensional, if $e^1(\Eins, J^3_{\Sp})\neq 0$ then 
$e^1(\Eins,\tau)\ge 4$. As $e^1(\Eins, J^2_{\Sp})=0$ we have an injection $\Ext^1_{G/Z}(\Eins, \tau)\hookrightarrow 
\Ext^1_{G/Z}(\Eins, \tau/J^2_{\Sp})$. Since $\tau/J^2_{\Sp}\cong \Sp^{\oplus 2}$, 
$e^1(\Eins, \tau/J^2_{\Sp})=4$ and the injection must be an isomorphism. This implies the existence of an exact sequence
\begin{equation}\label{contra1}
 0\rightarrow J^2_{\Sp}\rightarrow \tau'\rightarrow \tau_1^{\oplus 2}\rightarrow 0,
\end{equation}
with $\soc_G \tau'\cong \Sp$, where $\tau_1$ is the representation defined by \eqref{defipi1}. Since $\soc_G \tau'\cong \Sp$, we have 
 $e^0(\pi(0,1), J^2_{\Sp})=e^0(\pi(0,1), \tau')=1$. Applying $\Hom_{G/Z}(\pi(0,1), \ast)$ 
to \eqref{contra1} we deduce that $e^1(\pi(0,1), J^2_{\Sp})\ge e^0(\pi(0,1), \tau_1^{\oplus 2})\ge 2$. This contradicts Lemma \ref{conf2}.    
\end{proof} 

\begin{cor}\label{firstzero1} $e^1(\Eins, J^4_{\pi_{\alpha}})=0$. 
\end{cor} 
\begin{proof} Lemma \ref{exactseqJ} gives an exact sequence 
$0\rightarrow J^1_{\pi_{\alpha}}\rightarrow J^4_{\pi_{\alpha}}\rightarrow J^3_{\Sp}\rightarrow 0$.
Corollary \ref{acyclic1Jal} says that $e^i(\Eins, J^1_{\pi_{\alpha}})=0$ for
all $i\ge 0$. Hence, $e^1(\Eins, J^4_{\pi_{\alpha}})=e^1(\Eins, J^3_{\Sp})=0$, where the last equality follows from 
Proposition \ref{firstzero}.  
\end{proof} 

The following technical result will be useful later, in the arguments of  \S \ref{Thecentre}.
\begin{lem}\label{Y1} Let $J$ be an injective envelope  of $\Eins_G$ or $\pi_{\alpha}$ in $\Mod^{\mathrm{ladm}}_{G/Z}(k)$. The exact 
sequence $0\rightarrow J^1\rightarrow J \rightarrow J/J^1\rightarrow 0$ induces  isomorphisms $\End_G(J)\cong \End_G(J/J^1)$,  
$\End_G(\gr^1 J)\cong \End_G(\gr^1(J/J^1))$.
\end{lem}
\begin{remar}\label{Y2} Proposition \ref{resindJNG} and Lemma \ref{theJ1} imply  $J_{\Eins_G}/ J^1_{\Eins_G}\cong J_{\pi_{\alpha}}$ 
and $J_{\pi_{\alpha}}/J^1_{\pi_{\alpha}}\cong J_{\Sp}$.  
\end{remar}
\begin{proof} Let $J_{\chi}$ be an injective envelope of $\chi$ 
in $\Mod^{\mathrm{ladm}}_{G/Z}(k)$, where $\chi$ is either $\Eins_T$ or $\alpha$. Lemma \ref{theJ1} says  that $J^1\cong \Indu{P}{G}{J_{\chi}}$. 
It follows from \eqref{shift} that $\Ord_P J \cong \Ord_P \Indu{P}{G}{J_{\chi}}\cong J_{\chi^s}$. Since $\chi^s\neq \chi \alpha^{-1}$
we get $\Hom_T(\Ord_P J, J_{\chi}\otimes\alpha^{-1})=0$ and thus $\Ext^1_{G/Z}(J, \Indu{P}{G}{J_\chi})=0$ by Lemma \ref{technik}.
Hence we obtain an exact sequence $0\rightarrow \Hom_G(J, J^1)\rightarrow \Hom_G(J,J)\rightarrow \Hom_{G}(J, J/J^1)\rightarrow 0$.
Since $J$ is injective $\RR^1\Ord_P J=0$ thus the $U$-coinvariants $J_U$ are zero, \cite[3.6.2]{ord2}, and so $\Hom_G(J, J^1)=0$. 
As $\Ord_P J/J^1\cong \RR^1\Ord_P J^1\cong J_{\chi}\otimes \alpha^{-1}$ we get $\Hom_G(J^1, J/J^1)=0$ and 
so $\Hom_G(J, J/J^1)\cong \Hom_G(J/J^1, J/J^1)$. 

The second assertion follows by the same argument with $J^2$ instead of $J$. 
Note that $J^1=\gr^1 J\cong \Indu{P}{G}{J_{\chi}}$ and $\gr^2 J=\gr^1(J/J^1)$. Now
 $\Ord_P J^2\cong \Ord_P J$, hence $\Ext^1_{G/Z}(J^2, \Indu{P}{G}{J_\chi})=0$ by Lemma \ref{technik}, and 
$\RR^1\Ord_P J^2\cong \Ord_P J/J^2\cong\RR^1\Ord_P \gr^2 J$ by \eqref{shift}. Hence, 
$J^2_U\cong (\gr^2 J)_U$ and so $$\Hom_G(J^2, \gr^1 J)\cong \Hom_G(\gr^2 J, \gr^1 J)=0.$$ 
The last equality follows from the fact that 
$\gr^1 J$ and $\gr^2 J$ do not have a common irreducible subquotient, as $\gr^1 J\cong \Indu{P}{G}{J_{\chi}}$ and 
$\gr^2 J$ is a quotient of $\Indu{P}{G}{J_{\chi^s\alpha}}$ by \eqref{shift} and $\chi\neq \chi^s\alpha$ since $p\ge 5$. 
We obtain an isomorphism $\Hom_G(J^2, J^2)\cong \Hom_G(J^2, \gr^2 J)\cong \Hom_G(\gr^2 J, \gr^2 J)$. On the other 
hand from the exact sequence $0\rightarrow \Hom_G(J^1, J)\rightarrow \Hom_G(J^2, J)\rightarrow \Hom_G(\gr^2 J, J)\rightarrow 0$
we obtain an exact sequence $0\rightarrow \Hom_G(J^1, J^1)\rightarrow \Hom_G(J^2, J^2)\rightarrow \Hom_G(\gr^2 J, J^1)\rightarrow 0$
from the functoriality of the filtration. Hence, an isomorphism    $\Hom_G(J^1, J^1)\cong \Hom_G(J^2, J^2)$.      
\end{proof}

We let $\dInd: \dualcat_T(\OO)\rightarrow \dualcat_G(\OO)$ be the functor $\dIndu{P}{G}{N}:= (\Indu{P}{G}{N^{\vee}})^{\vee}$. 
With this notation we have $\wM\cong \dIndu{P}{G}{\wP_{\alpha^{\vee}}}$, 
and hence $M:=\wM\otimes_{\OO} k\cong \dIndu{P}{G}{P_\alpha^{\vee}}$, where $\wP_{\alpha^{\vee}}$ (resp. $P_{\alpha^{\vee}}$) is 
a projective envelope of $\alpha^{\vee}$ in $\dualcat_{T}(\OO)$ (resp. $\dualcat_{T}(k)$). Moreover, we have
$$E/\mathfrak a \cong \End_{\dualcat_T(k)}(P_{\alpha^{\vee}})\cong \End_{\dualcat_G(k)}(M)\cong \End_{\qcat(k)}(\TT M),$$
see Corollary \ref{endoPOrd}, \eqref{defiidealaQ}. We let $\cJ:\dualcat_G(k)\rightarrow \dualcat_G(k)$ be 
the functor $\cJ(N)= (\Indu{\overline{P}}{G}{(\Ord_P N^{\vee})})^{\vee}$.

\begin{prop}\label{filtrationNG} Let $P_{\alpha^{\vee}}$ be a projective envelope of $\alpha^{\vee}$ in $\dualcat_{T/Z}(k)$.
The\-re exists a decreasing filtration $P_{\alpha^{\vee}}^{\bullet}$ of $P_{\alpha^{\vee}}$ by subobjects, such that 
\begin{itemize}
\item[(i)] $P^0_{\alpha^{\vee}}=P_{\alpha^{\vee}}$;
\item[(ii)] $\rad P^i_{\alpha^{\vee}}\subseteq P^{i+1}_{\alpha^{\vee}} \subseteq P^i_{\alpha^{\vee}}$, for all $i\ge 0$;
\item[(iii)] for $i\ge 0$ we have 
\begin{equation}\label{ievenNGII}
\TT P^{2i}/ \TT P^{2i+1}\cong \TT\dIndu{P}{G}{ P^i_{\alpha^{\vee}}},
\end{equation} 
\begin{equation}\label{ioddNGII}
\TT P^{2i+1}/ \TT P^{2i+2}\cong 
\TT\dIndu{P}{G}{ ((P^i_{\alpha^{\vee}})^s \otimes \alpha^{\vee})},
\end{equation}
\end{itemize}
where $P^i=(J/J^i)^{\vee}$ and $J^{\bullet}$ is the filtration of $J$ by ordinary parts. Moreover, $\Hom_{\qcat(k)}(\TT P^{2i}, T_{\Eins})=0$.
\end{prop}
\begin{proof} 
Let $J$  be an injective envelope of $\Indu{P}{G}{\alpha}$ in $\Mod^{\mathrm{l adm}}_{G/Z}(k)$. It follows from Lemma \ref{vaca2} that for $i\ge 1$ we have an exact sequence 
\begin{equation}\label{shiftthatthing}
 0\rightarrow \theta_i\rightarrow \Indu{\overline{P}}{G}{\Ord_P (\gr^i J)}\rightarrow \gr^i J\rightarrow 0,
\end{equation}
where $G$ acts trivially on $\theta_i$. By evaluating at the identity, we may identify $\theta_i$ with a subspace of $(\Ord_P \gr^i J)^T$. 
Lemma \ref{vaca1} says that 
\begin{equation}\label{shiftagain}
\RR^1\Ord_P \gr^i J \cong ((\Ord_P \gr^i J)/\theta_i)^s \otimes \alpha^{-1}.
\end{equation}
For $i\ge 1$ let $\kappa_i:= \Ord_P \gr^i J$. We deduce from Proposition \ref{projectiveandord} that we have an injection 
$\Indu{\overline{P}}{G}{J_{\alpha^{-1}}}\hookrightarrow J$. Hence, $\kappa_1\cong J_{\alpha^{-1}}\cong J_{\Eins}\otimes \alpha^{-1}$ and $\theta_1=0$. 
It follows from \eqref{shift} and 
\eqref{shiftagain} that $\kappa_{2i-1}^T=0$ and hence $\theta_{2i-1}=0$, for all $i\ge 1$. We deduce from \eqref{shift} and 
\eqref{shiftagain} that $\kappa_{2i}\cong (\kappa_{2i-1})^s \otimes \alpha^{-1}$ and we have an exact sequence 
\begin{equation}\label{comment50}
0\rightarrow \theta_{2i} \rightarrow \kappa_{2i}\rightarrow \kappa_{2i+2}\rightarrow 0
\end{equation}of $T$-representations, where $T$ acts trivially on
$\theta_{2i}$. In particular, $\kappa_i$ is a successive extension of copies of $\alpha^{-1}$ when $i$ is odd, and a a successive extension of copies of $\Eins_T$ when $i$ is even.

Since $\TT \theta_i^{\vee}=0$ we deduce from \eqref{shiftthatthing} that  
\begin{equation}\label{applyTNG}
 \TT (\gr^i J )^{\vee}\cong \TT (\Indu{\overline{P}}{G}{\kappa_i})^{\vee}\cong  \TT (\Indu{P}{G}{\kappa_i^s})^{\vee}.
\end{equation}
We let $P^i_{\alpha^{\vee}}:= (\kappa_{2i+1}^s)^{\vee}\cong(\kappa_{2i+2}\otimes\alpha)^{\vee}$ then $P^0_{\alpha^{\vee}}\cong (J_{\alpha^{-1}}^s)^{\vee}\cong P_{\alpha^{\vee}}$. 
Moreover, by twisting \eqref{comment50}  by $\alpha$ and dualizing we obtain injections $P^{i+1}_{\alpha^{\vee}}\hookrightarrow P^i_{\alpha^{\vee}}$ with  semi-simple cokernel. Hence, 
$\rad P^i_{\alpha^{\vee}}\subseteq P^{i+1}_{\alpha^{\vee}}$. Part (iii) follows from \eqref{applyTNG}. Moreover, we deduce from 
\eqref{shift} that 
$$\Hom_G(\Indu{P}{G}{\Eins}, J/J^{2i})\cong \Hom_T(\Eins, \kappa_{2i+1})=0.$$
Since $P^{i}\cong (J/J^i)^{\vee}$, Lemma \ref{homT1NG} implies that 
\begin{displaymath}
\begin{split}
\Hom_{\qcat(k)}(\TT P^{2i}, T_{\Eins})&\overset{\ref{homT1NG}}{\cong} \Hom_G(\Indu{P}{G}{\Eins}, J/J^{2i})\cong \Hom_T(\Eins, \Ord_P(J/J^{2i}))\\&\overset{\eqref{shift}}{\cong}
\Hom_T(\Eins, \Ord_P \gr^{2i+1} J)\cong \Hom_T(\Eins, \kappa_{2i+1})= 0.
\end{split}
\end{displaymath}
\end{proof} 

\begin{lem}\label{achtungbaby1} $\Hom_{\dualcat(k)}(P^2, \Eins_G^{\vee})=0$, $\Hom_{\dualcat(k)}(P^4, \Eins_G^{\vee})=0$.
\end{lem}
\begin{proof}The assertion is equivalent to $\Hom_G(\Eins, J_{\pi_{\alpha}}/J^{2i}_{\pi_{\alpha}})=\Ext^1_{G/Z}(\Eins, J^{2i}_{\pi_{\alpha}})=0$ 
for $i=1,2$. If $i=2$ this follows from Corollary \ref{firstzero1}.
Proposition \ref{resindJNG} gives an exact sequence 
$0\rightarrow J^2_{\pi_{\alpha}}\rightarrow J_{\pi_{\alpha}}\rightarrow J_{\pi_{\alpha}}^{\oplus 2}$, 
which proves the assertion for $i=1$.
\end{proof}

\begin{lem}\label{vaca4} The ideal $\mathfrak a$ is a finitely generated right $E$-module and $E/\mathfrak a\wtimes_E P\cong P/P^2$, 
$\mathfrak a P\cong P^2$, $\mathfrak a \wtimes_{E} \TT P\cong \mathfrak a \TT P\cong \TT P^2$.
\end{lem}
\begin{proof}  Since $\wM\otimes_{\OO} k \cong P/P^1$  and $\wM$ is $\OO$-flat, we deduce from the definition of $\wa$ in \eqref{defiidealaQ} that 
$$\mathfrak a=\{\phi\in E: \phi ( P)\subseteq P^1\}=\{\phi\in E: \phi(P)\subseteq  P^2\},$$
where the second equality follows from the fact that $\Hom_{\dualcat(k)}(P, P^1/P^2)=0$ as  
Lemma \ref{vaca31/2} implies that $\pi_{\alpha^{\vee}}$ is not a subquotient of $P^1/P^2$.
Hence,   $\mathfrak a P\subseteq P^2$   and $\mathfrak a \TT P\subseteq \TT P^2$. 
On the other hand using \eqref{alphaJNG} and \eqref{SpNG} we get a 
surjection $P\oplus P\twoheadrightarrow P^2$. For $i=1$ and $i=2$ let $\phi_i\in E$ be the composition 
$$  P\rightarrow P\oplus P\twoheadrightarrow P^2\hookrightarrow  P,$$
where the first arrow is $(\id, 0)$ if $i=1$ and $(0, \id)$ if $i=2$. Then $\phi_1, \phi_2\in \mathfrak a$ and 
$P^2= \phi_1(P) + \phi_2(P)$. Hence, $P^2\subseteq \mathfrak a P$ and so $\mathfrak a P=P^2$ is closed in $P$, which implies 
$E/\mathfrak a \wtimes_E P\cong P/\mathfrak a P\cong P/P^2$. Using Lemma \ref{headS0} and exactness of $\Hom_{\dualcat(k)}(P, \ast)$
we get 
$$ \mathfrak a \cong \Hom_{\dualcat(k)}(P, \mathfrak a P)\cong \Hom_{\dualcat(k)}(P, P^2).$$
Hence, $\mathfrak a= \phi_1 E+ \phi_2 E$ is a finitely generated right $E$-module. In particular, 
$\mathfrak a \TT P=\phi_1(\TT P)+\phi_2(\TT P) =\TT P^2$ is an object of $\qcat(k)$. Since $\TT P$ is $E$-flat by Lemma \ref{NGIIEflat},
we obtain $\mathfrak a\wtimes_E \TT P\cong \mathfrak a \TT P$.
\end{proof}

\begin{lem}\label{vaca31/2} $\Hom_{\qcat(k)}(\TT P, \TT P^{2i-1}/\TT P^{2i})=0$, for all $i\ge 1$.
\end{lem}
\begin{proof} All the irreducible subquotients of $\TT P^{2i-1}/\TT P^{2i}$ are isomorphic 
to $T_{\Eins}$, see \eqref{ioddNGII}. Since $\TT P$ is a projective envelope of $T_{\alpha}$ in $\qcat(k)$, see Lemmas \ref{projQproj} and 
\ref{essQess}, 
there are no non-zero homomorphisms.
\end{proof}

\begin{lem}\label{vaca32/3} We have an isomorphism of $E$-modules:
\begin{equation}\label{holi}
\Hom_{\qcat(k)}(\TT P, \TT P^{2i}/\TT P^{2i+2})\cong \Hom_{\dualcat_T(k)}(P_{\alpha^{\vee}}, P^i_{\alpha^{\vee}}),
\end{equation} 
for all $i\ge 0$.
\end{lem}
\begin{proof} The $E$-module structure on the left hand side is given by the action of $E=\End_{\qcat(k)}(\TT P)$ on $\TT P$ and on the right hand side 
by the action of $E/\mathfrak a\cong \End_{\dualcat_T(k)}(P_{\alpha^{\vee}})$ on $P_{\alpha^{\vee}}$. Since $\TT P$ is 
projective, $\Hom_{\qcat(k)}(\TT P, \ast)$ is exact and so we get:
\begin{displaymath}
\begin{split}
&\Hom_{\qcat(k)}(\TT P, \TT P^{2i}/\TT P^{2i+2})\overset{\ref{vaca31/2}}{\cong}
\Hom_{\qcat(k)}(\TT P, \TT P^{2i}/\TT P^{2i+1})\\
&\overset{\eqref{ievenNGII}}{\cong}\Hom_{\qcat(k)}(\TT P, \TT \dIndu{P}{G}{P^i_{\alpha^{\vee}}})
\overset{\ref{projQproj}}{\cong} \Hom_{\dualcat(k)}( P, \dIndu{P}{G}{P^i_{\alpha^{\vee}}}) \\
&\overset{\ref{vaca3}}{\cong} \Hom_{\dualcat(k)}(P/ P^1, \dIndu{P}{G}{P^i_{\alpha^{\vee}}})\overset{\eqref{ievenNGII}}{\cong} 
\Hom_{\dualcat(k)}(\dIndu{P}{G}{P_{\alpha^{\vee}}},  \dIndu{P}{G}{P^i_{\alpha^{\vee}}})\\
&\cong \Hom_{\dualcat_T(k)}(P_{\alpha^{\vee}}, P^i_{\alpha^{\vee}}).
\end{split}
\end{displaymath}
\end{proof}

\begin{lem}\label{vaca5} Let $\md$ be a compact $E/\mathfrak a$-module. If $\Hom_{\dualcat(k)}(P^{2i}, \Eins_G^{\vee})=0$   for a fixed $i$ then 
\begin{equation}
\Hom_{\qcat(k)}( \TT P^{2i}/\TT P^{2i+2} ,\md \wtimes_E \TT P)\cong
\Hom_{\qcat(k)}( \TT P^{2i},\md \wtimes_E \TT P).
\end{equation}
\end{lem}
\begin{proof} Since $\mathfrak a$ acts trivially on $\md$ we have 
$\md \wtimes_E \TT P\cong \md \wtimes_E \TT P/\mathfrak a \TT P$. It follows from Lemma \ref{vaca3} that the filtration 
on $\TT P$ is $E$-invariant. Lemma \ref{vaca4} gives us an exact sequence:
\begin{equation}
\md \wtimes_E \TT P^1/\TT P^2 \rightarrow \md \wtimes_E \TT P\rightarrow \md \wtimes_E \TT P/ \TT P^1\rightarrow 0.
\end{equation}  
We may find an exact sequence of compact $E$-modules:
\begin{equation} \label{resolvemdNG}
\prod_{i\in I} E/\mathfrak a \rightarrow \prod_{j\in J} E/\mathfrak a \rightarrow \md \rightarrow 0
\end{equation} 
for some index sets $I$ and $J$. Applying $\wtimes_{E} \TT P^1/\TT P^2$ to \eqref{resolvemdNG} we deduce that 
$\md\wtimes_E \TT P^1/\TT P^2$ is a quotient of  $\prod_{j\in J} \TT P^1/\TT P^2$. Hence, it follows from \eqref{ioddNGII} 
that all the irreducible subquotients of $\md\wtimes_E \TT P^1/\TT P^2$ are isomorphic to $T_{\Eins}$. Since 
$\Hom_{\qcat(k)}(\TT P^{2i}, T_{\Eins})=0$ by Proposition \ref{filtrationNG}, we get an injection:
\begin{equation}\label{hominjectNGII}
\Hom_{\qcat(k)}(\TT P^{2i}, \md\wtimes_E \TT P)\hookrightarrow \Hom_{\qcat(k)}(\TT P^{2i},  \md\wtimes_E \TT P/\TT P^1).
\end{equation}
Hence, we obtain a commutative diagram:
\begin{displaymath}
\xymatrix@1{\;\Hom_{\qcat(k)}(\TT P^{2i}, \md\wtimes_E \TT P)\;\ar@{^(->}[r]\ar[d]& 
\Hom_{\qcat(k)}(\TT P^{2i},  \md\wtimes_E \TT P/\TT P^1)\ar[d]\\ 
\;\Hom_{\qcat(k)}(\TT P^{2i+2}, \md\wtimes_E \TT P)\;\ar@{^(->}[r]& 
\Hom_{\qcat(k)}(\TT P^{2i+2},  \md\wtimes_E \TT P/\TT P^1).}
\end{displaymath} 
It is enough to show that the right vertical arrow is zero. As $ P/P^1\cong \dIndu{P}{G}{P_{\alpha^{\vee}}}$, 
Corollary \ref{headM} says  $\md\wtimes_E P/P^1\cong
\dIndu{P}{G}{(\md \wtimes_{E} P_{\alpha^{\vee}})}$. In particular, $\md \wtimes_E P/P^1\cong \cJ(\md \wtimes_E P/P^1)$ and so 
the second step of filtration on $\md \wtimes_E P/ P^1$ is zero. 
Hence, Lemma \ref{vaca3} implies that 
$$\Hom_{\dualcat(k)}(P^{2i}, \md \wtimes_E P/ P^1)\cong \Hom_{\dualcat(k)}(P^{2i}/P^{2i+1}, \md \wtimes_E P/ P^1).$$
Since $\Hom_{\dualcat(k)}(P^{2i}, \Eins_G^{\vee})=0$ by assumption, Lemma \ref{HomQexp} implies that 
\begin{equation}\label{hominjectNGIII}
\Hom_{\qcat(k)}(\TT P^{2i}, \TT(\md \wtimes_E P/ P^1))\cong \Hom_{\qcat(k)}(\TT P^{2i}/\TT P^{2i+1}, \TT(\md \wtimes_E P/ P^1)).
\end{equation}      
Since $\TT(\md \wtimes_E P/ P^1)\cong \md \wtimes_E \TT P/\TT P^1$ by Lemma \ref{Tcomwtimes}, the right vertical arrow in the diagram above is zero.
\end{proof}

Let $\varphi: \wE\twoheadrightarrow R^{\psi}$ be the homomorphism defined in Proposition \ref{definephi}, let $\rr$ be the ideal of $R^{\psi}$ defined in 
Proposition \ref{imageofa} and let $\rr_k=\rr\otimes_{\OO} k$.  

\begin{lem}\label{achtungbaby5} If $\Hom_{\dualcat(k)}(P^{2i}, \Eins_G^{\vee})=0$ and $\mathfrak a^i\wtimes_E \TT P\cong \TT P^{2i}$ for a fixed $i$ then 
the map $\varphi$ induces an isomorphism 
$\mathfrak a^i/\mathfrak a^{i+1}\overset{\cong}{\rightarrow} \rr_k^i/\rr_k^{i+1}\cong \nn^i$, where $\nn$ is 
the maximal ideal of $E/\mathfrak a$. Moreover, $\mathfrak a^{i+1}\wtimes_E \TT P \cong \TT P^{2i+2}$.
\end{lem}
\begin{proof} Recall that $E/\mathfrak a\cong k[[x,y]]$, let $\nn$ be the maximal ideal of $E/\mathfrak a$ and let $\mathrm{K}$ be the quotient field of $E/\mathfrak a$. We have a surjection 
$$\TT P^{2i}\cong \mathfrak a^i\wtimes_E \TT P \twoheadrightarrow \mathfrak a^i/\mathfrak a^{i+1}\wtimes_E \TT P,$$
We note that since $\mathfrak a$ is a finitely generated right $E$-module, $\mathfrak a^j$ is a closed submodule of $\mathfrak a^{j-1}$ for 
all $j\ge 1$ and hence $\mathfrak a^i/\mathfrak a^{i+1}$ 
is a compact $E/\mathfrak a$-module. It follows from Lemma \ref{vaca5} that the surjection factors through 
$\TT P^{2i}/\TT P^{2i+2}\twoheadrightarrow \mathfrak a^i/\mathfrak a^{i+1}\wtimes_E \TT P$. We apply $\Hom_{\qcat(k)}(\TT P, \ast)$ and use 
Lemmas \ref{headS0} and \ref{vaca32/3} to get a  surjection  of (right) $E$-modules: 
$\Hom_{\dualcat_T(k)}(P_{\alpha^{\vee}}, P^i_{\alpha^{\vee}}) \twoheadrightarrow \mathfrak a^i/\mathfrak a^{i+1} $,
where $E$ acts on $\md:=\Hom_{\dualcat_T(k)}(P_{\alpha^{\vee}}, P^i_{\alpha^{\vee}})$ via 
$E/\mathfrak a \cong \End_{\dualcat_T(k)}(P_{\alpha^{\vee}})$.
It follows from Proposition \ref{filtrationNG} (ii) that $\rad^i P_{\alpha^{\vee}}\subseteq P^i_{\alpha^{\vee}}$. Since 
$P_{\alpha^{\vee}}$ is flat over $E/\mathfrak a$, see the proof of Proposition \ref{projTistfree}, and 
$k\wtimes_{E/\mathfrak a} P_{\alpha^{\vee}}\cong \alpha^{\vee}$ is irreducible, we get that 
$\rad^i P_{\alpha^{\vee}}\cong \nn^i\wtimes_E P_{\alpha^{\vee}}
\cong \nn^i P_{\alpha^{\vee}}$. Since $P^i_{\alpha^{\vee}}\subseteq P_{\alpha^{\vee}}$ we have 
$\nn^i\subseteq \md\subseteq  E/\mathfrak a$. Hence, $\dim_{\mathrm K} \md \otimes_{E} \mathrm{K}=1$ and we have an injection 
$\md \hookrightarrow \md \otimes_{E} \mathrm{K}$. Proposition \ref{imageofa} and Corollary \ref{A6} (ii) give a surjection
$ \md\twoheadrightarrow \varphi(\mathfrak a)^i/\varphi(\mathfrak a)^{i+1}\twoheadrightarrow \nn^i.$
Since $\nn^i\otimes_{E} \mathrm{K}$ is $1$-di\-men\-sio\-nal, the map induces an isomorphism 
$\md \otimes_E \mathrm{K}\cong \nn^i\otimes_E \mathrm{K}$. 
Hence, the composition $\md\rightarrow \nn^i$ is injective, and so $\varphi: \mathfrak a^i/\mathfrak a^{i+1}\rightarrow 
\varphi(\mathfrak a)^i/\varphi(\mathfrak a)^{i+1}$ is injective, and thus an isomorphism.  Since $\Hom_{\qcat(k)}(\TT P^{2i}, T_{\Eins})=0$ by 
Proposition \ref{filtrationNG}, Lemma \ref{headS}  implies that the evaluation map $\md\wtimes_E \TT P\rightarrow \TT P^{2i}/\TT P^{2i+2}$
is surjective. Since the composition $\md\wtimes_E \TT P\rightarrow \TT P^{2i}/\TT P^{2i+2}\rightarrow 
\mathfrak a^i/\mathfrak a^{i+1} \wtimes_E \TT P$
is an isomorphism, we deduce that $\TT P^{2i}/\TT P^{2i+2}\cong \mathfrak a^i/\mathfrak a^{i+1} \wtimes_E \TT P$. Since $\TT P$ is $E$-flat 
and  $\mathfrak a^i\wtimes_E \TT P\cong \TT P^{2i}$ by assumption, we deduce that $\TT P^{2i+2}\cong \mathfrak a^{i+1}\wtimes_E \TT P$.
\end{proof} 

\begin{lem}\label{purestress} The map $\varphi$ induces isomorphisms $\mathfrak a/\mathfrak a^2 \cong \rr_k/ \rr_k^2\cong \nn$ and
$\mathfrak a^2/\mathfrak a^3\cong \rr_k^2/\rr_k^3\cong \nn^2$, where $\nn$ is the maximal ideal of $E/\mathfrak a$.
\end{lem}
\begin{proof} Since $\Hom_{\dualcat(k)}(P^2, \Eins_G^{\vee})=0$ by Lemma \ref{achtungbaby1} and $\mathfrak a\wtimes_E \TT P\cong \TT P^2 $ by Lemma 
\ref{vaca4}, Lemma \ref{achtungbaby5} implies that  $\mathfrak a/\mathfrak a^2 \cong \rr_k/ \rr_k^2\cong \nn$ and 
$\mathfrak a^2\wtimes_E \TT P\cong \TT P^4$. Since $\Hom_{\dualcat(k)}(P^4, \Eins_G^{\vee})=0$ by Lemma \ref{achtungbaby1}, Lemma \ref{vaca4} 
implies $\mathfrak a^2/\mathfrak a^3\cong \rr_k^2/\rr_k^3\cong \nn^2$. 
\end{proof}

\begin{prop}\label{filisthesame} The surjection of graded rings $\varphi^{\bullet}: \gr^{\bullet}_{\mathfrak a}(E)\twoheadrightarrow \gr^{\bullet}_{\rr_k}(R^{\psi}_k)$ 
is an isomorphism. 
\end{prop}
\begin{proof} It follows from Lemma \ref{RpsirhoA} that $R^{\psi}_k \cong k[[x,y,z,w]]/(xz-yw)$ and $\rr_k=(z,w)$. Thus
$R^{\psi}_k/\rr_k \cong k[[x,y]]$ and $\gr^{\bullet}_{\rr_k}(R^{\psi}_k)\cong (R^{\psi}_k/\rr_k)[\bar{z}, \bar{w}]/(x\bar{z}-y\bar{w})$.
It follows from Proposition \ref{imageofa} that $\varphi^0$ induces an isomorphism $E/\mathfrak a\cong R^{\psi}_k/\rr_k\cong k[[x,y]]$.
Lemma \ref{purestress} implies that $\varphi^1$ induces an isomorphism $\mathfrak a/\mathfrak a^2\cong \rr_k/\rr_k^2$.  In particular, 
$\gr^{\bullet}_{\mathfrak a}(E)/\gr^{>1}_{\mathfrak a}(E)\cong \gr^{\bullet}_{\rr_k}(R^{\psi}_k)/\gr^{>1}_{\rr_k}(R^{\psi}_k)$  is a 
commutative ring. Hence, we have a surjection
\begin{equation}\label{hardwork}
 \beta:(E/\mathfrak a)[\bar{z}, \bar{w}]^{\mathrm{nc}}\twoheadrightarrow \gr^{\bullet}_{\mathfrak a}(E)
 \end{equation}
such that the image of $x\bar{z}-y \bar{w}$ is zero, where the source is a polynomial ring in two non-commutative variables with coefficients in $\wE/\mathfrak a$. Lemma \ref{purestress} implies that $\varphi^2$ induces an isomorphism 
$\mathfrak a^2/\mathfrak a^3\cong \rr_k^2/\rr_k^3$. Hence,  $\bar{z} \bar{w}-\bar{w}\bar{z}$ maps to zero in $\gr^{\bullet}_{\mathfrak a}(E)$. 
Thus $\gr^{\bullet}_{\mathfrak a}(E)$ is a commutative ring and \eqref{hardwork} factors through
$$ (E/\mathfrak a)[\bar{z}, \bar{w}]/(x\bar{z}-y\bar{w})\twoheadrightarrow  \gr^{\bullet}_{\mathfrak a}(E)\twoheadrightarrow \gr^{\bullet}_{\rr_k}(R^{\psi})\cong 
(E/\mathfrak a)[\bar{z}, \bar{w}]/(x\bar{z}-y\bar{w}).$$
Since any surjection of a noetherian ring onto itself is an isomorphism we deduce the assertion. 
\end{proof}

\begin{thm}\label{varphisoNGII} The map $\varphi$ induces an isomorphism $\wE\cong R^{\psi}$.
\end{thm}
\begin{proof} We deduce from Proposition \ref{filisthesame} that $\varphi$ induces an 
isomorphism $E/\mathfrak a^i \overset{\cong}{\rightarrow} R^{\psi}_k/\varphi(\mathfrak a)^i$, for all 
$i\ge 1$. Passing to the limit we get an isomorphism $E\cong R^{\psi}_k$. Since $R^{\psi}$ is $\OO$-flat 
by Corollary \ref{RpsirhoA}, we get that $(\Ker \varphi)\otimes_{\OO} k=0$. Hence, $\Ker \varphi=0$ by Nakayama's 
lemma.
\end{proof}

\begin{cor}\label{univdefII} $\cV(\wP_{\pi_{\alpha}^{\vee}})$ is the universal deformation 
of $\rho$ with determinant equal to $\zeta\varepsilon$.
\end{cor}
\begin{proof} It follows from Theorem \ref{varphisoNGII} that $\cV$ induces an isomorphism between deformation 
functors and hence an isomorphism between the universal objects.
\end{proof}

\subsection{The centre}\label{Thecentre}
Let $\BB=\{\Eins, \Sp, \pi_{\alpha}\}$ and let  $\wP_{\Eins_G^{\vee}}$, $\wP_{\Sp^{\vee}}$ and $\wP_{\pi_{\alpha}^{\vee}}$ be projective envelopes of 
$\Eins_G^{\vee}$, $\Sp^{\vee}$ and $\pi_{\alpha}^{\vee}$ in $\dualcat(\OO)$. Let 
$\wP_{\BB}:= \wP_{\pi_{\alpha}^{\vee}}\oplus \wP_{\Sp^{\vee}}\oplus\wP_{\Eins_G^{\vee}}$ and $\wE_{\BB}:=\End_{\dualcat(\OO)}(\wP_{\BB})$.
Recall that the functor $N\mapsto \Hom_{\dualcat(\OO)}(\wP_{\BB}, N)$ induces an equivalence of categories between 
$\dualcat(\OO)^{\BB}$ and the category of compact $\wE_{\BB}$-modules, Proposition \ref{gabriel}. In this section we 
compute the ring $\wE_{\BB}$ and show that it is a finitely generated module over its centre, and that the centre is 
naturally  isomorphic to $R^{\psi}$.  

After twisting we may assume that our fixed central character $\zeta$ is trivial, see Lemma \ref{liftetatwist} below. 
For a character  $\chi: T/Z\rightarrow k^{\times}$ we let $\wP_{\chi^{\vee}}$ be a projective envelope of $\chi^{\vee}$
in $\dualcat_{T/Z}(\OO)$ and let $\wM_{\chi^{\vee}}:=(\Indu{P}{G}{(\wP_{\chi^{\vee}})^{\vee}})^{\vee}$. Further
we define $\wM_{\Eins_T^{\vee}, 0}$ by the exact sequence: 
\begin{equation}\label{0B}
0\rightarrow \wM_{\Eins_T^{\vee}, 0}\overset{\xi_{32}}{\rightarrow} \wM_{\Eins_T^{\vee}}\overset{\theta}{\rightarrow} \OO\rightarrow 0,
\end{equation}
where $\OO$ is equipped with the trivial $G$-action. Proposition \ref{resindJNG} and Corollary \ref{projaretfree3} imply the existence of exact sequences:
\begin{equation}\label{1B} 
0\rightarrow \wP_{\pi_{\alpha}^{\vee}}\overset{\varphi_{31}}{\rightarrow} \wP_{\Eins_G^{\vee}}\overset{\psi_3}{\rightarrow} \wM_{\Eins^{\vee}_T}\rightarrow 0
\end{equation}
\begin{equation}\label{2B} 
0\rightarrow \wP_{\Sp^{\vee}}\overset{\varphi_{12}}{\rightarrow} \wP_{\pi_{\alpha}^{\vee}}\overset{\psi_1}{\rightarrow} \wM_{\alpha^{\vee}}\rightarrow 0
\end{equation}
\begin{equation}\label{3B} 
\wP^{\oplus 2}_{\pi_{\alpha}^{\vee}}\rightarrow \wP_{\Sp^{\vee}}\overset{\psi_2}{\rightarrow} \wM_{\Eins^{\vee}_T, 0}\rightarrow 0
\end{equation}

\begin{lem}\label{allvanish} $\Hom_{\dualcat(\OO)}(\wP_{\pi_{\alpha}^{\vee}},  \wM_{\Eins^{\vee}_T})$, $\Hom_{\dualcat(\OO)}(\wP_{\pi_{\alpha}^{\vee}},  \wM_{\Eins^{\vee}_T, 0})$,
$\Hom_{\dualcat(\OO)}(\wP_{\Eins_G^{\vee}}, \wM_{\alpha^{\vee}})$  and $\Hom_{\dualcat(\OO)}(\wP_{\Sp^{\vee}}, \wM_{\alpha^{\vee}})$, all vanish. 
\end{lem}
\begin{proof}  The proof in all the cases is the same, so we prove only the vanishing of $\Hom_{\dualcat(\OO)}(\wP_{\pi_{\alpha}^{\vee}},  \wM_{\Eins^{\vee}_T})$. The irreducible subquotients of 
$\wM_{\Eins^{\vee}_T}$ are isomorphic to $\Eins_G^{\vee}$ and $\Sp^{\vee}$. In particular, $\pi_{\alpha}^{\vee}$ is not a subquotient. 
Since $\wP_{\pi^{\vee}_{\alpha}}$ is  a projective envelope of $\pi_{\alpha}^{\vee}$, we deduce that 
$\Hom_{\dualcat(\OO)}(\wP_{\pi_{\alpha}^{\vee}}, \wM_{\Eins^{\vee}_T})=0$.
\end{proof}

We let $\varphi_{32}:=\varphi_{31}\circ\varphi_{12}: \wP_{\Sp^{\vee}}\hookrightarrow \wP_{\Eins^{\vee}_G}$ and denote: 
\begin{displaymath}
\wE_{11}:= \End_{\dualcat(\OO)}(\wP_{\pi_{\alpha}^{\vee}}), \quad \wE_{22}:= \End_{\dualcat(\OO)}(\wP_{\Sp^{\vee}}), \quad 
\wE_{33}:= \End_{\dualcat(\OO)}(\wP_{\Eins_G^{\vee}}).
\end{displaymath}
For $i=1, 2, 3$ we let $\wa_{ii}:=\{\phi\in \wE_{ii}: \psi_i\circ \phi=0\}$, with $\psi_i$ defined  in \eqref{1B}, \eqref{2B}, \eqref{3B}. Let $e_1$, 
$e_2$ and $e_3$ be idempotents in $\wE_{\BB}$ cutting out  $\wP_{\pi_{\alpha}^{\vee}}$, $\wP_{\Sp^{\vee}}$ and $\wP_{\Eins_G^{\vee}}$ respectively.
\begin{lem}\label{X1} 
\begin{equation}\label{X2}
\wE_{11}\overset{\cong}{\rightarrow}\Hom_{\dualcat(\OO)}(\wP_{\pi_{\alpha}^{\vee}}, \wP_{\Eins_G^{\vee}}),\quad z_{11}\mapsto \varphi_{31}\circ z_{11}
\end{equation} 
\begin{equation}\label{X3}
\wE_{22}\overset{\cong}{\rightarrow}\Hom_{\dualcat(\OO)}(\wP_{\Sp^{\vee}}, \wP_{\pi_{\alpha}^{\vee}}), \quad z_{22}\mapsto \varphi_{12} \circ z_{22}
\end{equation}  
\begin{equation}\label{X4}
\Hom_{\dualcat(\OO)}(\wP_{\Eins_G^{\vee}}, \wP_{\Sp^{\vee}})\overset{\cong}{\rightarrow} \Hom_{\dualcat(\OO)}(\wP_{\Eins_G^{\vee}}, \wP_{\pi_{\alpha}^{\vee}}), 
\quad z_{23}\mapsto \varphi_{12}\circ z_{23}
\end{equation}
\begin{equation}\label{X5}
\Hom_{\dualcat(\OO)}(\wP_{\pi_{\alpha}^{\vee}}, \wP_{\Sp^{\vee}}) \overset{\cong}{\rightarrow} \wa_{11}, \quad z_{21}\mapsto \varphi_{12}\circ z_{21},
\end{equation}
\begin{equation}\label{X6}
\Hom_{\dualcat(\OO)}(\wP_{\Eins_G^{\vee}}, \wP_{\pi_{\alpha}^{\vee}}) \overset{\cong}{\rightarrow} \wa_{33}, \quad z_{13}\mapsto \varphi_{31}\circ z_{13},
\end{equation}
\begin{equation}\label{X7}
\Hom_{\dualcat(\OO)}(\wP_{\Eins_G^{\vee}}, \wP_{\Sp^{\vee}}) \overset{\cong}{\rightarrow} \wa_{33}, \quad z_{23}\mapsto \varphi_{32}\circ z_{23},
\end{equation}
\end{lem} 
\begin{proof} The proof in all the cases is the same, one uses \eqref{1B}, \eqref{2B} and \eqref{3B} together with Lemma \ref{allvanish} and the left exactness of $\Hom$.
 The assertion in \eqref{X7} follows from \eqref{X4} and \eqref{X6}.
\end{proof}

\begin{lem}\label{defibeta} There exists $\beta: \wP_{\Sp^{\vee}}\rightarrow \wP_{\Eins_G^{\vee}}$ such that
$\psi_3\circ \beta= \xi_{32}\circ\psi_2$. Moreover, the following sequence: 
\begin{equation}\label{Z2} 
\wP_{\pi_{\alpha}^{\vee}}\oplus \wP_{\Sp^{\vee}}\overset{\varphi_{13}\oplus\beta}{\longrightarrow} \wP_{\Eins_G^{\vee}}\rightarrow \OO\rightarrow 0
\end{equation} 
is exact.
\end{lem}
\begin{proof} Since $\Hom_{\dualcat(\OO)}(\wP_{\Sp^{\vee}}, \OO)=0$, we deduce from \eqref{0B} that $\xi_{32}$ induces an isomorphism 
$\Hom_{\dualcat(\OO)}(\wP_{\Sp^{\vee}}, \wM_{\Eins_T^{\vee}, 0})\overset{\xi_{32}\circ}{\cong}\Hom_{\dualcat(\OO)}(\wP_{\Sp^{\vee}}, \wM_{\Eins_T^{\vee}})$. 
Since $\wP_{\Sp^{\vee}}$ is projective we deduce from \eqref{1B} that $\psi_3$ induces a surjection $\Hom_{\dualcat(\OO)}(\wP_{\Sp^{\vee}}, \wP_{\Eins_G^{\vee}})\overset{\psi_3\circ}{\twoheadrightarrow }
\Hom_{\dualcat(\OO)}(\wP_{\Sp^{\vee}}, \wM_{\Eins_T^{\vee}})$. Hence, there exists $\beta: \wP_{\Sp^{\vee}}\rightarrow \wP_{\Eins_G^{\vee}}$ such that
$\psi_3\circ \beta= \xi_{32}\circ\psi_2$.  Combining \eqref{0B} with \eqref{1B} we obtain \eqref{Z2}.
\end{proof}

\begin{prop}\label{Risomo} Restriction to  $\wP_{\pi_{\alpha}^{\vee}}$ in \eqref{1B} and to $\wP_{\Sp^{\vee}}$ in \eqref{2B} induces isomorphisms: 
$\wE_{33}\overset{\cong}{\rightarrow} \wE_{11}$, $z_{33}\mapsto z_{33}|_{\wP_{\pi_{\alpha}^{\vee}}}$ and  
$\wE_{11}\overset{\cong}{\rightarrow} \wE_{22}$, $z_{11}\mapsto z_{11}|_{\wP_{\Sp^{\vee}}}$.
\end{prop} 
\begin{proof} We only show the first claim, the second can be proved in an identical manner. Since 
$\Hom_{\dualcat(\OO)}(\wP_{\pi_{\alpha}^{\vee}}, \wM_{\Eins_T^{\vee}})=0$ by Lemma \ref{allvanish}, every endomorphism of $\wP_{\Eins_G^{\vee}}$ maps 
$\wP_{\pi_{\alpha}^{\vee}}$ to itself. Hence, we obtain a well defined map $r:\wE_{33}\rightarrow \wE_{11}$. Now both $\wE_{33}$ 
and $\wE_{11}$ are $\OO$-torsion free, since $\wP_{\Eins_G^{\vee}}$ and $\wP_{\pi_{\alpha}^{\vee}}$ are by Corollary \ref{projaretfree}. 
Nakayama's lemma for compact $\OO$-modules applied to the cokernel and then to the kernel of $r$ implies that it is enough 
to show that $r\otimes_{\OO} k: \wE_{33}\otimes_{\OO} k \rightarrow \wE_{11}\otimes_{\OO} k$ is an isomorphism. Let $J_{\pi}$ be an injective envelope
of an irreducible representation $\pi$ in $\Mod^{\mathrm{ladm}}_{G/Z}(k)$, $P_{\pi^{\vee}}$ projective envelope of $\pi^{\vee}$ in $\dualcat(k)$ and 
$\wP_{\pi^{\vee}}$ projective envelope of $\pi^{\vee}$ in $\dualcat(\OO)$. Then 
$$\End_{\dualcat(\OO)}(\wP_{\pi^{\vee}})\otimes_{\OO} k \cong \End_{\dualcat(k)}(P_{\pi^{\vee}})\cong \End_{G}(J_{\pi})^{op},$$
where the first isomorphism follows from \eqref{tildeseq1} in \S \ref{def}, the second since $J_{\pi}^{\vee}$ is a projective envelope 
of $\pi^{\vee}$ and thus is isomorphic to $P_{\pi^{\vee}}$. Now the assertion of 
the Proposition follows from Lemma \ref{Y1}.
\end{proof}

\begin{cor}\label{injZ} Let $z$ lie in the centre of  $\wE_{\BB}$. If the restriction 
of $z$ to any of $\wP_{\pi_{\alpha}^{\vee}}$, $\wP_{\Sp^{\vee}}$ or $\wP_{\Eins_G^{\vee}}$ is equal to zero 
then $z=0$. 
\end{cor}
\begin{proof} Since $\dualcat(\OO)^{\BB}$ is equivalent to the category  
of compact $\wE_{\BB}$-modules, for every object $M$ of $\dualcat(\OO)^{\BB}$, $z$ defines a functorial 
homomorphism $z_M: M\rightarrow M$. It follows from the functoriality that for every subobject $N$ 
of $M$, $z_N$ is equal to the restriction of $z_M$ to $N$. The assertion follows from 
Proposition \ref{Risomo} and this observation.
\end{proof}

\begin{cor}\label{rings} The rings $\wE_{11}$, $\wE_{22}$ and $\wE_{33}$ are naturally isomorphic to $R^{\psi}$. 
In particular, they are commutative noetherian integral domains. 
\end{cor} 
\begin{proof} The isomorphism $\wE_{11}\cong R^{\psi}$ in Theorem \ref{varphisoNGII} is natural since it is induced by a morphism of deformation 
functors. The sequences \eqref{1B} and  \eqref{2B} are not canonical, but are minimal projective resolutions of $\wM_{\Eins_G^{\vee}}$ and 
$\wM_{\alpha^{\vee}}$ respectively. Since any two minimal projective resolutions of the same object are isomorphic, a different choice 
of an exact sequence in \eqref{1B} would conjugate the homomorphism $\wE_{33}\rightarrow \wE_{11}$ by an element of $\wE_{33}$. 
Since as a consequence  of Proposition \ref{Risomo} all the rings are isomorphic and hence are commutative, we deduce that 
the homomorphism $\wE_{33}\rightarrow \wE_{11}$ does not depend on the choice of \eqref{1B}.
The last assertion follows from the explicit description of $R^{\psi}$ in 
Corollary \ref{RpsirhoA} below.
\end{proof}

\begin{cor}\label{ideals} For $i=1, 2, 3$, $\wa_{ii}$ is the annihilator of $\Hom_{\dualcat(\OO)}(\wP_{\pi_{\alpha}^{\vee}}, \wM_{\alpha^{\vee}})$, 
$\Hom_{\dualcat(\OO)}(\wP_{\Sp^{\vee}}, \wM_{\Eins_T^{\vee}, 0})$ and $\Hom_{\dualcat(\OO)}(\wP_{\Eins_G^{\vee}}, \wM_{\Eins_T^{\vee}})$ respectively.
Moreover, $\wE_{ii}/\wa_{ii}$ is $\OO$-torsion free. 
\end{cor}
\begin{proof} The proof in all cases is the same. We deal with $i=1$. By applying $\Hom_{\dualcat(\OO)}(\wP_{\pi_{\alpha}^{\vee}}, \ast)$ 
to \eqref{1B} we deduce that  $\Hom_{\dualcat(\OO)}(\wP_{\pi_{\alpha}^{\vee}}, \wM_{\alpha^{\vee}})= \psi_1 \circ \wE_{11}\cong \wE_{11}/\wa_{11}$. 
Since $\wE_{11}$ is commutative the annihilator of $\psi_1$ coincides with the annihilator $\psi_1\circ \wE_{11}$. Further, since
$\wM_{\alpha^{\vee}}$ is $\OO$-torsion free so is $\Hom_{\dualcat(\OO)}(\wP_{\pi_{\alpha}^{\vee}}, \wM_{\alpha^{\vee}})$ and hence $\wE_{11}/\wa_{11}$. 
\end{proof}

It follows from Corollary \ref{RpsirhoA} that $R^{\psi}$ is $\OO$-torsion free. Thus we have an injection 
$R^{\psi}\hookrightarrow R^{\psi}[1/p]$. Let $\mathfrak r$ be the intersection of the reducible locus in 
$R^{\psi}[1/p]$ with $R^{\psi}$, see Corollary \ref{A6}.

\begin{lem}\label{imageofr} The image of $\mathfrak r$ in $\wE_{ii}$ via the natural isomorphism of Corollary \ref{rings} is equal to $\wa_{ii}$. 
\end{lem}
\begin{proof} If $i=1$ then the assertion  follows from  Proposition \ref{imageofa} and Theorem \ref{varphisoNGII}. We claim that
the isomorphisms $\wE_{33}\overset{\cong}{\rightarrow} \wE_{11}\overset{\cong}{\rightarrow} \wE_{22}$ of Proposition \ref{Risomo}
identify  $\wa_{33}$ with  $\wa_{11}$ and $\wa_{11}$ with $\wa_{22}$. Since 
$\Hom_{\dualcat(\OO)}(\wP_{\Eins_G^{\vee}}, \wM_{\alpha^{\vee}})$ and 
$\Hom_{\dualcat(\OO)}(\wP_{\pi_{\alpha}^{\vee}}, \wM_{\Eins_T^{\vee},0})$ are zero by Lemma \ref{allvanish}, 
using Corollary  \ref{ideals}, we get that the image of $\wa_{33}$ is contained in $\wa_{11}$ and the image of 
$\wa_{11}$ is contained in $\wa_{22}$. Hence, we obtain surjections $\wE_{33}/\wa_{33}\twoheadrightarrow \wE_{11}/\wa_{11} \twoheadrightarrow 
\wE_{22}/\mathfrak a_{22}$.  Since $\wE_{ii}/\wa_{ii}$ is $\OO$-torsion free it is enough to show that the surjections are isomorphisms 
after tensoring with $k$. This assertion follows from the last assertion in Lemma \ref{Y1}.  
\end{proof} 
 
We embed $R^{\psi}$ into $\wE_{\BB}$ diagonally using the isomorphisms of Corollary \ref{rings}:
\begin{equation}\label{embeddR} 
R^{\psi}\hookrightarrow \wE_{11}\oplus \wE_{22}\oplus \wE_{33}\hookrightarrow \wE_{\BB}, \quad z\mapsto z_{11}\oplus z_{22}\oplus z_{33}.
\end{equation}
\begin{lem}\label{zcommutes1} Let $\psi\in \wE_{\BB}$ such that 
$e_3\circ \psi\circ e_2=0$ then $z\circ \psi=\psi\circ z$ for all $z\in R^{\psi}$. 
\end{lem}
\begin{proof} It follows from the definition of the embedding that $z$ commutes with $\varphi_{31}$ and $\varphi_{12}$ and 
hence with their composition $\varphi_{32}$. Since the rings $\wE_{11}$, $\wE_{22}$ and $\wE_{33}$ are commutative the assertion 
follows from Lemma \ref{X1}.
\end{proof}

\begin{lem}\label{subB} $\Hom_{\dualcat(\OO)}(\OO, \wP_{\Eins_G^{\vee}})=0$.
\end{lem} 
\begin{proof} If not then by composing $\wP_{\Eins_G^{\vee}}\twoheadrightarrow \OO\rightarrow \wP_{\Eins_G^{\vee}}$ we 
would obtain a zero divisor in $\End_{\dualcat(\OO)}(\wP_{\Eins_G^{\vee}})\cong R^{\psi}$.
\end{proof}

\begin{lem}\label{Z3} Let $\wP_{\Eins_{G}^{\vee}, 0}$ be the kernel of $\wP_{\Eins_G^{\vee}}\twoheadrightarrow \OO$. Then restriction 
induces an isomorphism $\End_{\dualcat(\OO)}(\wP_{\Eins_G^{\vee}})\cong \End_{\dualcat(\OO)}(\wP_{\Eins_{G}^{\vee}, 0})$.
\end{lem} 
\begin{proof} It follows from \eqref{Z2} that $\wP_{\Eins_{G}^{\vee}, 0}$ is a quotient of 
$\wP_{\Sp^{\vee}}\oplus \wP_{\pi_{\alpha}^{\vee}}$, which implies that $\Hom_{\dualcat(\OO)}(\wP_{\Eins_{G}^{\vee}, 0}, \OO)=0$. 
Thus every endomorphism of $\wP_{\Eins_G^{\vee}}$ maps $\wP_{\Eins_{G}^{\vee}, 0}$ to itself. Since 
$\wP_{\Eins_{G}^{\vee}, 0}$ contains the image of $\varphi_{13}$ the assertion follows from Proposition \ref{Risomo}.
\end{proof}

\begin{lem}\label{Z1} Let $N$ be an object of $\dualcat(\OO)^{\BB}$. Then $G$ acts trivially on $N$
if and only if  $\Hom_{\dualcat(\OO)}(\wP_{\pi^{\vee}_{\alpha}}\oplus \wP_{\Sp^{\vee}}, N)=0$.
\end{lem}
\begin{proof} Let  $N$ be an object of  $\dualcat(\OO)$ then $\Hom_{\dualcat(\OO)}(\wP_{\pi^{\vee}_{\alpha}}\oplus \wP_{\Sp^{\vee}}, N)=0$
is equivalent to the assertion  that none of the irreducible subquotients of $N$ are isomorphic to $\pi_{\alpha}^{\vee}$ or $\Sp^{\vee}$. 
If $N$ is an object of $\dualcat(\OO)^{\BB}$ then the last  condition is equivalent to the assertion that all the irreducible 
subquotients of $N$ are isomorphic to $\Eins_G^{\vee}$, which is equivalent to $G$ acting trivially on $N$ by Lemma \ref{ext1110}.
\end{proof}

\begin{cor}\label{qcatmod} Let $\wP:=\wP_{\pi^{\vee}_{\alpha}}\oplus \wP_{\Sp^{\vee}}$ and let $\wE:=\End_{\dualcat(\OO)}(\wP)$. The functor 
$\TT N\mapsto \Hom_{\qcat(\OO)}(\TT \wP, \TT N)$ induces an equivalence of categories between $\qcat(\OO)^{\BB}$ and the category of compact 
$\wE$-modules.
\end{cor}
\begin{proof} It follows from Lemma \ref{Z1} that the category $\mathfrak T (\OO)$, defined in \S \ref{qcat}, is precisely the kernel of the functor $N\mapsto \Hom_{\dualcat(\OO)}(\wP, N)$. Given this, the assertion 
follows from \cite[\S IV.4, Thm. 4]{gab}.
\end{proof}

\begin{lem}\label{zcommutes2} Let $z\in \wE_{33}$ and let $z_{11}$ and $z_{22}$ denote the restriction of $z$ to $\wP_{\pi_{\alpha}^{\vee}}$ and 
$\wP_{\Sp^{\vee}}$ respectively via \eqref{1B} and \eqref{2B}. Then $\xi\circ (z_{11}\oplus z_{22})= z\circ \xi$ for all $\xi\in 
\Hom_{\dualcat(\OO)}(\wP_{\pi_{\alpha}^{\vee}}\oplus \wP_{\Sp^{\vee}}, \wP_{\Eins_G^{\vee}})$.
\end{lem} 
\begin{proof} 
Let $\wP:=\wP_{\pi_{\alpha}^{\vee}}\oplus \wP_{\Sp^{\vee}}$, $\wE:=\End_{\dualcat(\OO)}(\wP)$, let $\wP_{\Eins_G^{\vee}, 0}$ be the kernel of 
$\wP_{\Eins_G^{\vee}}\twoheadrightarrow \OO$  and let 
$$\md:=\Hom_{\dualcat(\OO)}(\wP, \wP_{\Eins_G^{\vee}})\cong \Hom_{\dualcat(\OO)}(\wP, \wP_{\Eins_G^{\vee},0}).$$
Since $\wP_{\Eins_G^{\vee}, 0}$ is a quotient of $\wP$ by \eqref{Z2}, Lemma \ref{headS} implies that the natural map $\md\wtimes_{\wE} \wP\rightarrow \wP_{\Eins_G^{\vee},0}$
is surjective. Let $K$ be the kernel, Lemma \ref{headS0} implies that 
$\Hom_{\dualcat(\OO)}(\wP, K)=0$. Thus $G$ acts trivially on $K$ by Lemma \ref{Z1}. 
It follows from Lemma \ref{subB} that every endomorphism of $\md\wtimes_{\wE} \wP$ maps 
$K$ to itself. Thus we obtain well defined sequence of maps
$$ \End_{\wE}(\md)\rightarrow \End_{\dualcat(\OO)}(\md\wtimes_{\wE} \wP)\rightarrow \End_{\dualcat(\OO)}(\wP_{\Eins_G^{\vee},0})\rightarrow \End_{\wE}(\md),$$
in which composition of any three consecutive one is an identity. The arrows are given by 
$\phi\mapsto [\xi\wtimes v\mapsto \phi(\xi)\wtimes v]$; $\phi\mapsto [\xi\wtimes v+ K\mapsto \phi(\xi\wtimes v)+ K]$; 
$\phi\mapsto [\xi\mapsto \phi\circ \xi]$ respectively, see also the proof of Proposition \ref{ringsareiso}. Since $z_{11}\oplus z_{22}$ lies 
in the centre of $\wE$ by Lemma \ref{zcommutes1}, it defines an element of $\End_{\wE}(\md)$ by 
$\xi\mapsto \xi \circ (z_{11}\oplus z_{22})$. Let $z'$ be the image of $z_{11}\oplus z_{22}$ in $\End_{\dualcat(\OO)}(\wP_{\Eins_G^{\vee},0})$ via the
above maps. Tautologically we have $z'\circ \xi=\xi \circ (z_{11}\oplus z_{22})$. From \eqref{Z2} we obtain a 
commutative diagram 
 \begin{displaymath}
\xymatrix@1{ \wP_{\pi_{\alpha}^{\vee}}\oplus \wP_{\Sp^{\vee}} \ar[d]_{z_{11}\oplus z_{22}}\ar[r]^-{\varphi_{13}\oplus \beta} &  \md\wtimes_{\wE} \wP 
\ar[d]_{z_{11}\oplus z_{22}}\ar[r]& \wP_{\Eins_G^{\vee}, 0}\ar[d]^{z'}\\
\wP_{\pi_{\alpha}^{\vee}}\oplus \wP_{\Sp^{\vee}} \ar[r]^-{\varphi_{13}\oplus \beta} &  \md\wtimes_{\wE} \wP 
\ar[r]& \wP_{\Eins_G^{\vee}, 0} .}
\end{displaymath}    
Thus the restriction of $z'$ to $\wP_{\pi_{\alpha}^{\vee}}$ is equal to $z_{11}$, which is  equal to the restriction of $z$ to $\wP_{\pi_{\alpha}^{\vee}}$. 
It follows from Proposition \ref{Risomo} and Lemma \ref{Z3} that $z=z'$.
\end{proof}

\begin{thm}\label{ZNGII} The centre of $\wE_{\BB}$ (and hence the centre of $\dualcat(\OO)^{\BB}$) is naturally isomorphic to $R^{\psi}$, defined in Definition \ref{Rpsi}.
\end{thm}
\begin{proof} It follows from Lemma \ref{zcommutes1} and Lemma \ref{zcommutes2} that the image of $R^{\psi}$ via \eqref{embeddR} 
lies in the centre of $\wE_{\BB}$. Conversely, suppose that $z'$ lies in the centre of $\wE$. Since the restriction map 
$z\mapsto z|_{\wP_{\pi_{\alpha}^{\vee}}}$ induces an isomorphism $R^{\psi}\cong \wE_{11}$ there exists $z\in R^{\psi}$ such that
$(z-z')|_{\wP_{\pi_{\alpha}^{\vee}}}=0$. It follows from Corollary \ref{injZ} that $z=z'$.
\end{proof} 

\begin{remar}\label{pseudocentre} We note that it is shown in Corollary \ref{trbij6} below that sending deformation to its trace induces an isomorphism 
between the ring $R^{\psi}$ and $R^{\mathrm{ps}, \psi}_{\tr \rho}$ the deformation ring parameterizing $2$-di\-men\-sio\-nal pseudocharacters 
lifting $\tr \rho$ with determinant $\psi$.
\end{remar}

\begin{cor}\label{NgIIkill} Let $T:\gal\rightarrow  R^{\mathrm{ps},\zeta \varepsilon}_{\tr \rho}$ be the universal $2$-dimensional pseudocharacter 
with determinant $\zeta\varepsilon$ lifting $\tr \rho$. For every $N$ in $\dualcat(\OO)^{\BB}$, $\cV(N)$ is killed by 
$g^2-T(g)g + \zeta \varepsilon(g)$, for all $g\in \gal$.
\end{cor}
\begin{proof} If $N\cong \wP_{\pi_{\alpha}^{\vee}}$ then the assertion follows from Theorem \ref{varphisoNGII} and Corollary \ref{trbij6}.
Since $\cV$ is exact, $\cV(\wP_{\Sp^{\vee}})$ is a $\gal$-subrepresentation $\cV(\wP_{\pi_{\alpha}^{\vee}})$, see \eqref{2B}, thus 
the assertion also holds for 
$N\cong \wP_{\Sp^{\vee}}$ and hence for $\wP:=\wP_{\pi_{\alpha}^{\vee}}\oplus \wP_{\Sp^{\vee}}$. Let $\wE:=\End_{\dualcat(\OO)}(\wP)$ then 
Proposition \ref{equivofcatsII} and  Lemma \ref{VT2} imply that 
$$\cV(N)\cong \cV(\Hom_{\dualcat(\OO)}(\wP, N)\wtimes_{\wE} \wP)\cong \Hom_{\dualcat(\OO)}(\wP, N)\wtimes_{\wE}\cV(\wP),$$ which implies the claim.
\end{proof}

\begin{lem}\label{fgtfree} $\wE_{\BB}$ is a finitely generated torsion-free $R^{\psi}$-module. 
\end{lem}
\begin{proof} It is enough to prove the statement for $e_i \wE_{\BB} e_j$, $i,j=1,2,3$. If $(i, j)\neq (3,2)$ then the assertion 
follows from Lemma \ref{X1} as $\wE_{ii}\cong R^{\psi}$,  $\wa_{ii} \cong \mathfrak r$ and $R^{\psi}$ is an integral domain. 
Let $\md:=\Hom_{\dualcat(\OO)}(\wP_{\Sp^{\vee}}, \wP_{\Eins_G^{\vee}})$. 
It follows from the proof of  Lemma \ref{zcommutes2} that $\md$ is generated over $R^{\psi}$ by $\varphi_{23}$ and $\beta$. It remains
to show that $\md$ is torsion free. We may dualize \eqref{res1NG} and  using Proposition \ref{projaretfree3}  lift  it to an exact sequence: 
\begin{equation}\label{YMCA}
0\rightarrow \wP_{\Sp^{\vee}}\rightarrow \wP_{\Eins_G^{\vee}}\oplus \wP_{\pi_{\alpha}^{\vee}}\rightarrow \wP_{\pi_{\alpha}^{\vee}}^{\oplus 2}
\rightarrow \wP_{\Sp^{\vee}}\oplus \wP_{\pi_{\alpha}^{\vee}}\rightarrow \wP_{\Eins^{\vee}_G}\rightarrow \OO\rightarrow 0
\end{equation}
We apply $\Hom_{\dualcat(\OO)}(\wP_{\Sp^{\vee}}, \ast)$ to \eqref{YMCA} and use Lemma \ref{X1} to obtain an exact sequence 
$0\rightarrow R^{\psi}\rightarrow \md \oplus R^{\psi} \rightarrow R^{\psi}\oplus R^{\psi}$
of $R^{\psi}$-modules. Since $R^{\psi}$ is an integral domain we deduce that $\md$ is torsion free. 
\end{proof}

For $\delta\in \wE_{\BB}$ we let $\delta_{ij}=e_i \circ \delta\circ e_j$. This notation is consistent with \eqref{embeddR}.

\begin{lem}\label{swap} Let $\delta, \gamma \in \wE_{\BB}$ then the image of $\delta_{ij} \circ \gamma_{ji}$ under 
the isomorphism $\wE_{ii}\cong \wE_{jj}$ of Proposition \ref{Risomo} is equal to $\gamma_{ji} \circ \delta_{ij}$.  
\end{lem}
\begin{proof} There exists $z\in R^{\psi}$ such that  $z_{ii}= \delta_{ij}\circ \gamma_{ji}$. Since 
$(\gamma_{ji}\circ \delta_{ij}- z_{jj})\circ \gamma_{ji}= \gamma_{ji} \circ z_{ii}- z_{jj} \circ \gamma_{ji}=0$
and $\wE_{\BB}$ is a torsion free $R^{\psi}$-module we obtain the claim. 
\end{proof}   

By Corollary \ref{A6} the ideal $\rr$ is generated by two elements $c_0$ and $c_1$. For each pair 
$e_i \wP_{\BB}$ and $e_j\wP_{\BB}$ appearing in \eqref{X5}, \eqref{X6}, \eqref{X7} we may choose 
$\varphi_{ji}^k\in \Hom_{\dualcat(\OO)}(e_i \wP_{\BB}, e_j\wP_{\BB})$ such that $\varphi_{ij}\circ \varphi_{ji}^k= c_k e_i$ for 
$k=0$ and $k=1$. It follows from Lemma \ref{X1} that the elements $\varphi_{ji}^0$ and $\varphi_{ji}^1$ generate 
$e_j \wE_{\BB} e_i$ as an $R^{\psi}$-module, which is isomorphic to $\mathfrak r$. It follows from Lemma \ref{swap} that 
$\varphi_{ji}^k \circ \varphi_{ij}= c_k e_j$ for $k=0, 1$. We record this below: 
\begin{equation}\label{inter1}
\varphi_{12}\circ\varphi_{21}^k= c_k e_{1}, \quad \varphi_{31} \circ \varphi_{13}^k= c_k e_3, \quad \varphi_{32}\circ \varphi_{23}^k=c_k e_3
\end{equation} 
\begin{equation}\label{inter2}
\varphi_{21}^k\circ\varphi_{12}= c_k e_2, \quad \varphi_{13}^k \circ \varphi_{31}= c_k e_1, \quad \varphi_{23}^k\circ \varphi_{32}=c_k e_2.
\end{equation}   
By definition of $\varphi_{32}$ and \eqref{X6} we have:
\begin{equation}\label{inter3}
\varphi_{31}\circ\varphi_{12}= \varphi_{32}, \quad \varphi_{12}\circ \varphi_{23}^k= \varphi_{13}^k.
\end{equation}
\begin{equation}\label{inter4}
\varphi_{21}^k \circ \varphi_{13}^l= \varphi_{21}^k \circ \varphi_{12}\circ \varphi_{23}^l= c_k \varphi_{23}^l, \quad 
\varphi_{32}\circ \varphi_{21}^k=\varphi_{31}\circ \varphi_{12}\circ \varphi_{21}^k= c_k \varphi_{31} 
\end{equation}
Since $\varphi_{12}\circ(\varphi_{23}^k\circ \varphi_{31}- \varphi_{21}^k)= \varphi_{13}^k\circ \varphi_{31}-c_k e_1= 0$
and \eqref{X5} is an isomorphism we obtain:
\begin{equation}\label{inter5}
\varphi_{23}^k\circ \varphi_{31}=\varphi_{21}^k, \quad \varphi_{13}^k\circ \varphi_{32}=\varphi_{13}^k\circ \varphi_{31}\circ \varphi_{12}= c_k \varphi_{12}.
\end{equation}

\begin{lem}\label{last1} Let $\beta\in \Hom_{\dualcat(\OO)}(\wP_{\Sp^{\vee}}, \wP_{\Eins_G^{\vee}})$ be the morphism constructed in Lemma \ref{defibeta}.
 Then there exist unique $d_0, d_1\in R^{\psi}$ such that $c_0 \beta= d_0 \varphi_{23}$ and $c_1 \beta=d_1\varphi_{23}$. 
Moreover, 
\begin{equation}\label{inter6} 
\beta\circ \varphi_{21}^k=d_k \varphi_{31}, \quad \varphi_{13}^k \circ \beta=d_k \varphi_{12}, \quad 
\beta\circ \varphi_{23}^k= d_k e_3, \quad \varphi_{23}^k \circ \beta= d_k e_2
\end{equation}
\end{lem}
\begin{proof} The uniqueness follows from the fact that $\wE_{\BB}$ is $R^{\psi}$-torsion free, see Lemma \ref{fgtfree}. 
It follows from \eqref{X2} that there exists $d_k\in R^{\psi}$ such that $\beta\circ\varphi_{21}^k=d_k \varphi_{31}$. 
It follows from \eqref{inter2} that $c_k \beta= \beta\circ\varphi_{21}^k\circ \varphi_{12}=d_k \varphi_{31}\circ \varphi_{12}=d_k \varphi_{32}$.
There exists $a_k\in R^{\psi}$ such that $\beta\circ \varphi_{23}^k= a_k e_3$. We may multiply by $c_k$ to get 
$c_k a_k e_3= d_k \varphi_{32} \circ \varphi_{23}^k= d_k c_k e_3$. Since $R^{\psi}$ is an integral domain we obtain
$a_k=d_k$. Lemma \ref{swap} implies $\varphi_{23}^k\circ \beta= d_k e_2$. Moreover, $\varphi_{13}^k\circ \beta
=\varphi_{12}\circ\varphi_{23}^k \circ \beta= d_k \varphi_{12}$. 
 \end{proof}

\begin{lem}\label{last2} Sending $x\mapsto c_0$, $y\mapsto c_1$, $z\mapsto d_0$, $w\mapsto d_1$ induces an isomorphism 
of rings $\OO[[x,y,z,w]]/(xw-yz)\cong R^{\psi}$.
\end{lem}
\begin{proof} Since $\wE_{\BB}$ is a torsion free $R^{\psi}$-module, $(c_0d_1-c_1 d_0)\beta= (d_0d_1-d_1d_0)\varphi_{23}=0$ implies 
$c_0d_1=c_1d_0$. Thus the map is well defined. It is enough to show that it is surjective, since we know 
that $R^{\psi}$ can be presented as $\OO[[x,y,z,w]]/(f)$, see Corollary \ref{RpsirhoA}, and $xw-yz$ is a prime element in a factorial ring. 
Let $\mathfrak b:=\{b\in R^{\psi}: \theta\circ b=0\}$, where $\theta: \wM_{\Eins_T^{\vee}}\twoheadrightarrow \OO$ is defined in \eqref{0B}. 
Applying $\Hom_{\dualcat(\OO)}(\wP_{\Eins_G^{\vee}}, \ast)$ to \eqref{3B} we obtain a surjection 
$\Hom_{\dualcat(\OO)}(\wP_{\Eins_G^{\vee}}, \wP_{\Sp^{\vee}})\twoheadrightarrow \Hom_{\dualcat(\OO)}(\wP_{\Eins_G^{\vee}}, \wM_{\Eins_T^{\vee}, 0})$, 
thus $\psi_2\circ \varphi_{23}^0$ and $\psi_2\circ \varphi_{23}^1$ generate 
$\Hom_{\dualcat(\OO)}(\wP_{\Eins_G^{\vee}}, \wM_{\Eins_T^{\vee}, 0})$ as an $R^{\psi}$-module. 
Applying $\Hom_{\dualcat(\OO)}(\wP_{\Eins_G^{\vee}}, \ast)$ to \eqref{0B} we obtain an exact sequence
$$ 0\rightarrow \Hom_{\dualcat(\OO)}(\wP_{\Eins_G^{\vee}}, \wM_{\Eins_T^{\vee}, 0})\rightarrow  
\Hom_{\dualcat(\OO)}(\wP_{\Eins_G^{\vee}}, \wM_{\Eins_T^{\vee}})\rightarrow \OO\rightarrow 0.$$
 As $\Hom_{\dualcat(\OO)}(\wP_{\Eins_G^{\vee}}, \wM_{\Eins_T^{\vee}})\cong \wE_{33}/\wa_{33}\cong R^{\psi}/\rr\cong \OO[[x,y]]$ 
we deduce that $\mathfrak b$ contains $\rr$ and the images of $\xi_{32}\circ \psi_2\circ \varphi_{23}^0$ and 
$\xi_{32}\circ \psi_2\circ \varphi_{23}^1$ generate $\mathfrak b/ \mathfrak r$ as an $R^{\psi}$-module. Since 
by definition, see the proof of Lemma \ref{zcommutes2}, $\psi_3\circ \beta=\xi_{32}\circ \psi_2$ and 
$\beta \circ \varphi_{23}^k= d_k e_3$ by \eqref{inter6}, we deduce that the images of $d_0$ and $d_1$ 
generate $\mathfrak b/\mathfrak r$. Hence $R^{\psi}/(d_0, d_1, c_0, c_1)\cong \OO$ and so the map is surjective. 
\end{proof}

\begin{cor}\label{last3} The $\gal$-representation  corresponding to 
the ideal $(c_0, c_1, d_0, d_1)$ in $R^{\psi}[1/p]$ is characterized as the unique non-split extension 
$0\rightarrow \Eins \rightarrow V\rightarrow \varepsilon\rightarrow 0$. 
\end{cor} 
\begin{proof} Let $\nn_0= (c_0, c_1, d_0, d_1)\subset R^{\psi}$. It follows from the proof
of Lemma \ref{last2} that there exists a surjection $\wM_{\Eins_T^{\vee}}/\nn_0\wM_{\Eins_T^{\vee}}\twoheadrightarrow \OO$ 
where $G$ acts trivially on $\OO$. It follows from Lemma \ref{thruOrd} applied with $\md=R^{\psi}/\nn_0\cong \OO$ that 
 $\Pi:=\Hom^{cont}_{\OO}( \wM_{\Eins_T^{\vee}}/\nn_0\wM_{\Eins_T^{\vee}}, L)$ is a parabolic induction of a unitary character, which reduces to the trivial character modulo $\varpi$.
Since $\Pi^G\neq 0$ we deduce that $\Pi\cong (\Indu{P}{G} {\Eins})_{cont}$ and thus $\VV(\Pi)\cong \varepsilon$.
Let $\nn'$ be the maximal ideal of $R^{\psi}[1/p]$ corresponding to $V$ and let $\nn'_0:=R^{\psi}\cap \nn'$.
Theorem \ref{varphisoNGII} and Corollary \ref{univdefII} imply that 
$V\cong \cV(\wP_{\pi_{\alpha}^{\vee}}/\nn_0' \wP_{\pi_{\alpha}^{\vee}})\otimes_{\OO} L$.  
It follows from Proposition \ref{equivofcatsII} that 
\begin{displaymath}
\begin{split}
\Hom_{\qcat(\OO)}( \wM_{\Eins_T^{\vee}}/\nn_0\wM_{\Eins_T^{\vee}}, &\wP_{\pi_{\alpha}}/\nn'_0 \wP_{\pi_{\alpha}^{\vee}}) \otimes_{\OO} L\cong \\
&\Hom_{\gal}( \cV(\wM_{\Eins_T^{\vee}}/\nn_0\wM_{\Eins_T^{\vee}}), \cV(\wP_{\pi_{\alpha}}/\nn_0' \wP_{\pi_{\alpha}^{\vee}}))\otimes_{\OO} L
\end{split}
\end{displaymath}
is non-zero thus $\nn_0=\nn_0'$.
\end{proof}   

Since 
\begin{displaymath}
\wE_{\BB}=\begin{pmatrix} R^{\psi}e_1 & R^{\psi} \varphi_{12} & R^{\psi}\varphi_{13}^0+ R^{\psi}\varphi_{13}^1\\ 
R^{\psi}\varphi_{21}^0+ R^{\psi}\varphi_{21}^1 & R^{\psi}e_2 & R^{\psi} \varphi_{23}^0+ R^{\psi}\varphi_{23}^1\\ 
R^{\psi}\varphi_{31} & R^{\psi}\varphi_{32}+ R^{\psi} \beta & R^{\psi}e_{3}\end{pmatrix}
\end{displaymath}  
the multiplication in $\wE_{\BB}$ is determined by \eqref{inter1}, \eqref{inter2}, \eqref{inter5} and \eqref{inter6}. 
One may check that the $R^{\psi}$-module structure of $\Hom_{\dualcat(\OO)}(\wP_{\Sp^{\vee}}, \wP_{\Eins_G^{\vee}})$
is completely determined by Lemmas \ref{last1}, \ref{last2} and Corollary \ref{last3}. We also point out that 
\begin{equation}\label{ringqcat}
\wE:=\End_{\dualcat(\OO)}(\wP_{\pi_{\alpha}^{\vee}}\oplus \wP_{\Sp^{\vee}})= \begin{pmatrix} 
R^{\psi}e_1 & R^{\psi}\varphi_{12}\\ R^{\psi} \varphi_{21}^0 +R^{\psi}\varphi_{21}^1 & R^{\psi}e_2\end{pmatrix} 
\end{equation}
and the multiplication is given by $\varphi_{12}\circ \varphi_{21}^k= c_k e_1$, $\varphi_{21}^k\circ \varphi_{12}= c_k e_2$
for $k=0,1$, where $c_0$ and $c_1$ are generators of $\mathfrak r$, the intersection of $R^{\psi}$ and the reducible locus in $R^{\psi}[1/p]$.

\begin{lem}\label{nrmod} Let $\wE$ be the ring in \eqref{ringqcat} and let $\nn$ be a maximal ideal of $R^{\psi}[1/p]$ with residue field $L$ 
containing $\rr$. Then $\wE\otimes_{R^{\psi}} R^{\psi}[1/p]/\nn$ has two non-isomorphic irreducible modules, both of them $1$-di\-men\-sio\-nal.
\end{lem}
\begin{proof} Let $\mathfrak b$ be the two sided ideal in $\wE\otimes_{R^{\psi}} R^{\psi}[1/p]/\nn$ generated by the images of 
$\varphi_{12}$, $\varphi_{21}^0$, $\varphi_{21}^1$. Since $\nn$ contains $\rr=(c_0, c_1)$ we have $\mathfrak b^2=0$ and the quotient by $\mathfrak b$ of  
$\wE\otimes_{R^{\psi}} R^{\psi}[1/p]/\nn$  is isomorphic to $L\times L$. This implies the assertion. 
\end{proof}

\begin{remar}\label{VseeQ} We note that the Galois side sees only the quotient category $\qcat(\OO)^{\BB}$, see Proposition 
\ref{equivofcatsII}, and this category is equivalent to the category of compact modules of the endomorphism ring of 
$\TT \wP_{\Sp^{\vee}} \oplus \TT \wP_{\pi_{\alpha}^{\vee}}$, which is isomorphic to the  ring in \eqref{ringqcat} by Corollary
\ref{qcatmod}. Moreover, it follows from Proposition \ref{equivofcatsII} that the ring is isomorphic to $\End_{\gal}(\cV(\wP_{\Sp^{\vee}}) \oplus \cV(\wP_{\pi_{\alpha}^{\vee}}))$.
\end{remar}

\begin{remar}\label{image_stein} We are going to describe $\cV(\wP_{\Sp^{\vee}})$ as a $\gal$-representation. Corollary \ref{univdefII} says that $\cV(\wP_{\pi_{\alpha}^{\vee}})$ is the universal deformation of $\rho$ with determinant $\zeta\varepsilon$. Hence,   $\cV(\wP_{\pi_{\alpha}^{\vee}})/\rr \cV(\wP_{\pi_{\alpha}^{\vee}})$ is the universal reducible 
deformation of $\rho$ with the determinant $\zeta \varepsilon$. Thus we have an exact sequence $0\rightarrow N_{\Eins}\rightarrow \cV(\wP_{\pi_{\alpha}^{\vee}})/\rr \cV(\wP_{\pi_{\alpha}^{\vee}})\rightarrow N_{\omega}\rightarrow 0$, where
$N_{\Eins}$ is the deformation of the trivial representation and $N_{\omega}$ is a deformation of $\omega$ to $R^{\psi}_{\rho}/\rr$. One may deduce from Theorem \ref{trbij6} and the proof of Proposition \ref{redloc} that these  deformations are universal.  We apply $\cV$ to \eqref{2B} 
to obtain an exact sequence $0\rightarrow \cV(\wP_{\Sp^{\vee}})\rightarrow \cV(\wP_{\pi_{\alpha}^{\vee}}) \rightarrow \cV(\wM_{\alpha^{\vee}})\rightarrow 0$.  Proposition \ref{imageofa} implies that $\rr$ acts trivially on 
$\cV(\wM_{\alpha^{\vee}})$, and since all the irreducible subquotients of $\cV(\wM_{\alpha^{\vee}})$ are isomorphic to $\omega$, the surjection $\cV(\wP_{\pi_{\alpha}^{\vee}})/\rr \cV(\wP_{\pi_{\alpha}^{\vee}})\twoheadrightarrow \cV(\wM_{\alpha^{\vee}})$ factors 
through the surjection $N_{\omega}\twoheadrightarrow \cV(\wM_{\alpha^{\vee}})$. But both are free $R^{\psi}_{\rho}/\rr$-modules of rank $1$. Hence, the surjection is an isomorphism. This implies that $\cV(\wP_{\Sp^{\vee}})$ is the kernel 
of the map from the universal deformation of $\rho$ with determinant $\zeta\varepsilon$ to $N_{\omega}$.
\end{remar}
 
Let $\wP_{\ast}$ be a direct summand of $\wP_{\BB}$, let $\wE_{\ast}:=\End_{\dualcat(\OO)}(\wP_{\ast})$. The rings 
$\wE_{\ast}$ and $\wE_{\BB}$ are finitely generated modules over a noetherian ring $R^{\psi}$, thus they are right and 
left noetherian. Every finitely generated module carries a canonical topology, with respect to which the action 
is continuous. Since the rings are noetherian the canonical topology is Hausdorff. Let $c$ be a non-zero element of $R^{\psi}$ 
and let $\Mod^{\mathrm{fg}}_{\wE_{\BB}[1/p]}[c^{-1}]$ denote the full subcategory 
of finitely generated $\wE_{\BB}[1/p]$-modules consisting of those modules on which $c$ acts invertibly. 
Define a functor 
$$\mathrm Q: \Mod^{\mathrm{fg}}_{\wE_{\BB}[1/p]}[c^{-1}]\rightarrow \Mod^{\mathrm{fg}}_{\wE_{\ast}[1/p]}[c^{-1}], \quad 
\md\mapsto \Hom_{\dualcat(\OO)}( \wP_{\ast}, \md^0 \otimes_{\wE_{\BB}} \wP_{\BB})_L,$$
 where we have chosen 
a finitely generated $\wE_{\BB}$-submodule $\md^0\subset \md$ such that $\md=\md^0[1/p]$ and equipped it with the canonical topology.
Since $\Hom_{\dualcat(\OO)}(\wP_{\ast}, \wP_{\BB})$ is a finitely generated $R^{\psi}$-module, $\mathrm Q(\md)$ is a finitely generated 
$\wE_{\ast}[1/p]$-module. The definition of $\mathrm Q$ does not depend on the choice of $\md^0$, since any two are commensurable.

\begin{lem}\label{Qequiv} If $\mathrm Q$ is faithful then it induces  an equivalence of categories. 
\end{lem}
\begin{proof} Since $\Hom_{\dualcat(\OO)}(\wP_{\BB}, \ast)$ induces an equivalence  between $\dualcat(\OO)$  and
the category of compact $\wE_{\BB}$-modules, Lemma \ref{headS0} implies that the natural map 
$$\Hom_{\dualcat(\OO)}(\wP_{\BB}, N)\wtimes_{\wE_{\BB}} \wP_{\BB}\rightarrow N$$
is an isomorphism. Thus the functor 
$\md^0\mapsto \md^0\wtimes_{\wE_{\BB}} \wP_{\BB}$ is exact and hence $\mathrm Q$ is exact. Define 
$$\mathrm R: \Mod^{\mathrm{fg}}_{\wE_{\ast}[1/p]}[c^{-1}]\rightarrow \Mod^{\mathrm{fg}}_{\wE_{\BB}[1/p]}[c^{-1}], \quad 
\md\mapsto \Hom_{\dualcat(\OO)}( \wP_{\BB}, \md^0 \wtimes_{\wE_{\ast}} \wP_{\ast})_L.$$
We claim that $\mathrm Q\circ \mathrm R$ is equivalent to the identity functor. The claim implies that $\mathrm Q$ is fully faithful 
and surjective, hence an equivalence of categories. We may choose $\mathrm R(\md)^0$
to be the maximal $\OO$-torsion free quotient of 
$\Hom_{\dualcat(\OO)}( \wP_{\BB}, \md^0 \wtimes_{\wE_{\ast}} \wP_{\ast})$. Then we have a surjection 
$$\md^0 \wtimes_{\wE_{\ast}} \wP_{\ast}\cong \Hom_{\dualcat(\OO)}( \wP_{\BB}, \md^0 \wtimes_{\wE_{\ast}} \wP_{\ast})\wtimes_{\wE_{\BB}}\wP_{\BB}
\twoheadrightarrow \mathrm R(\md)^0\wtimes_{\wE_{\BB}} \wP_{\BB}.$$ 
with the kernel killed by a power of $p$. Since $\Hom_{\dualcat(\OO)}(\wP_{\ast}, \md^0 \wtimes_{\wE_{\ast}} \wP_{\ast})\cong \md^0$, 
see Lemma \ref{headS0}, is $\OO$-torsion free, we get $\md^0\cong\Hom_{\dualcat(\OO)}(\wP_{\ast}, \mathrm R(\md)^0\wtimes_{\wE_{\BB}} \wP_{\BB})$. 
\end{proof}  
   
\begin{lem}\label{rightII} Let $\md$ be in $\Mod^{\mathrm{fg}}_{\wE_{\BB}[1/p]}[c^{-1}]$ and choose $\md^0\subset \md$ as above. Then the kernel 
of $c: \md^0\wtimes_{\wE_{\BB}}\wP_{\BB} \rightarrow \md^0\wtimes_{\wE_{\BB}}\wP_{\BB}$ is zero and the cokernel is killed by a power of $p$.
\end{lem}
\begin{proof} Let $K$ be the kernel and $C$ be the cokernel. Lemma \ref{headS0} gives an exact sequence 
$$0\rightarrow \Hom_{\dualcat(\OO)}(\wP_{\BB}, K)\rightarrow \md^0\overset{c}{\rightarrow}\md^0\rightarrow 
\Hom_{\dualcat(\OO)}(\wP_{\BB}, C)\rightarrow 0.$$
Since $\md^0$ is finitely generated and $c$ is invertible on $\md$ we deduce that there exist $p^n$ such that $\Hom_{\dualcat(\OO)}(\wP_{\BB}, p^n C)=0$
and $\Hom_{\dualcat(\OO)}(\wP_{\BB}, K)=0$.
Since $K$ and $p^nC$ are objects of $\dualcat(\OO)^{\BB}$ this implies that they are $0$.
\end{proof}

\begin{prop}\label{invertc} Let $c\in \rr$ be non-zero  and $\wP_{\ast}$ be either $\wP_{\pi_{\alpha}^{\vee}}$, 
$\wP_{\Sp^{\vee}}$ or $\wP_{\Eins_G^{\vee}}$ then $\mathrm Q$ induces an equivalence of categories between 
 $\Mod^{\mathrm{fg}}_{\wE_{\BB}[1/p]}[c^{-1}]$ and $\Mod^{\mathrm{fg}}_{R^{\psi}[1/p]}[c^{-1}]$.
\end{prop}
\begin{proof} Let $\md$ be in $\Mod^{\mathrm{fg}}_{\wE_{\BB}[1/p]}[c^{-1}]$ and let $N=\md^0\wtimes_{\wE_{\BB}} \wP_{\BB}$.
It follows from Lemma \ref{imageofr} that $c$ kills $\wM_{\Eins_{T}^{\vee}}$, $\wM_{\alpha^{\vee}}$. Since 
$c$ acts invertibly on $\md$ Lemma \ref{rightII} implies that $\Hom_{\dualcat(\OO)}(\wM_{\ast}, N)=0$ 
and $\Ext^1_{\dualcat(\OO)}(\wM_{\ast}, N)=0$ is killed by a power of $p$, 
where $\ast=\Eins_T^{\vee}$ or $\ast=\alpha^{\vee}$. Thus \eqref{1B} and \eqref{2B} imply that we have an isomorphism of $R^{\psi}[1/p]$-modules:
$$\Hom_{\dualcat(\OO)}(\wP_{\Eins_G^{\vee}}, N)_L\cong \Hom_{\dualcat(\OO)}(\wP_{\pi_{\alpha}^{\vee}}, N)_L\cong \Hom_{\dualcat(\OO)}(\wP_{\Sp^{\vee}}, N)_L.$$
If $\mathrm Q(\md)=0$ then $0=\Hom_{\dualcat(\OO)}(\wP_{\BB}, N)_L\cong \md$. Hence the functor $\md\mapsto \mathrm Q(\md)$ is faithful. 
The assertion follows from Lemma \ref{Qequiv}.
\end{proof}

\begin{prop}\label{invertd} Let $\nn$ be the maximal ideal of $R^{\psi}[1/p]$ corresponding to 
$0\rightarrow \Eins\rightarrow V\rightarrow \varepsilon\rightarrow 0$, let $c\in R^{\psi}\cap \nn$ 
be non-zero
 and let $\wP_{\ast}= \wP_{\Sp^{\vee}}\oplus \wP_{\pi_{\alpha}^{\vee}}$
then $\mathrm Q$ induces an equivalence of categories between 
 $\Mod^{\mathrm{fg}}_{\wE_{\BB}[1/p]}[c^{-1}]$ and $\Mod^{\mathrm{fg}}_{\wE[1/p]}[c^{-1}]$, where $\wE$ is the ring described in \eqref{ringqcat}.
\end{prop}
\begin{proof} We have an exact sequence $\wP_{\ast}\rightarrow \wP_{\Eins_G^{\vee}}\rightarrow \OO\rightarrow 0$, 
see \eqref{Z2}, and $c$ kills $\OO$, see Corollary \ref{last3}. The proof is then the same as the proof of Proposition \ref{invertc}.  
\end{proof}

Let $\nn$ be a maximal ideal of $R^{\psi}[1/p]$ with residue field $L$ and let $\nn_0:=\nn \cap R^{\psi}$. Suppose that $\nn_0$ 
contains $\rr$. Then the Galois representation corresponding to $\nn$ is reducible. Thus it follows from Theorem \ref{varphisoNGII} and 
Corollary \ref{univdefII}
that we have a non-split sequence 
$0\rightarrow \psi_1\rightarrow \cV(\wP_{\pi_{\alpha}^{\vee}}/\nn_0 \wP_{\pi_{\alpha}^{\vee}})\rightarrow \psi_2\rightarrow 0$, where 
$\psi_1, \psi_2: \gal\rightarrow \OO^{\times}$ are continuous characters such that 
$\psi_1$ is congruent to $\cV(\Sp^{\vee})=\Eins$ and $\psi_2$ is congruent $\cV(\pi_{\alpha}^{\vee})=\omega$ modulo $\varpi$.

\begin{prop}\label{gohome4} Let $\nn$, $\psi_1$ and $\psi_2$ be as above then we have  isomorphisms of Banach space representations of $G$:
$$ \Hom^{cont}_{\OO}(\wM_{\Eins_T^{\vee}}/\nn_0 \wM_{\Eins_T^{\vee}}, L)\cong (\Indu{P}{G}{\psi_1\otimes \psi_2 \varepsilon^{-1}})_{cont},$$
$$ \Hom^{cont}_{\OO}(\wM_{\alpha^{\vee}}/\nn_0 \wM_{\alpha^{\vee}}, L)\cong (\Indu{P}{G}{\psi_2\otimes \psi_1 \varepsilon^{-1}})_{cont}.$$
\end{prop}
\begin{proof} Lemma \ref{imageofr} identifies $\rr$ with $\wa_{11}$, and by the definition of $\wa_{11}$, we have that $\wP_{\pi_{\alpha}^{\vee}}/\wa_{11} \wP_{\pi^{\vee}_{\alpha}}$
is the quotient of $\wP_{\pi^{\vee}_{\alpha}}$ by the submodule generated by the images of all endomorphisms of $\wP_{\pi^{\vee}_{\alpha}}$, whose image 
lies in the first term $\wP_{\Sp^{\vee}}$ of \eqref{2B}. Now using the fact that $\Hom_{\dualcat(\OO)}(\wP_{\pi^{\vee}_{\alpha}}, \wM_{\Eins^{\vee}_{T, 0}})=0$, see Lemma 
\ref{allvanish}, we deduce that this submodule is precisely the image of the first arrow in \eqref{3B}. Hence, we obtain an exact sequence:
\begin{equation}\label{gohome}
0\rightarrow \wM_{\Eins_{T,0}^{\vee}}\rightarrow \wP_{\pi_{\alpha}^{\vee}}/\rr  \wP_{\pi_{\alpha}^{\vee}}\rightarrow \wM_{\alpha^{\vee}}\rightarrow 0.
\end{equation}
As $\wM_{\alpha^{\vee}}$ is $R^{\psi}/\rr$-flat, see Corollary \ref{headM}, and $\nn_0$ contains $\rr$ by applying 
$R^{\psi}/\nn_0\wtimes_{R^{\psi}/\rr}$
we obtain an exact sequence:
\begin{equation}\label{gohome1}
0\rightarrow \wM_{\Eins_{T,0}^{\vee}}/ \nn_0  \wM_{\Eins_{T,0}^{\vee}}
\rightarrow \wP_{\pi_{\alpha}^{\vee}}/\nn_0 \wP_{\pi_{\alpha}^{\vee}}\rightarrow \wM_{\alpha^{\vee}}/\nn_0\wM_{\alpha^{\vee}}\rightarrow 0.
\end{equation} 
Applying $R^{\psi}/\nn_0\wtimes_{R^{\psi}/\rr}$ to \eqref{0B} gives an exact sequence: 
\begin{equation}\label{gohome2}
\wM_{\Eins_{T,0}^{\vee}}/ \nn_0  \wM_{\Eins_{T,0}^{\vee}}\rightarrow \wM_{\Eins_T^{\vee}}/\nn_0  \wM_{\Eins_T^{\vee}}\rightarrow 
\OO\wtimes_{R^{\psi}} R^{\psi}/\nn_0\rightarrow 0.
\end{equation}
Lemma \ref{thruOrd} implies that $\cV(\wM_{\Eins_T^{\vee}}/\nn_0  \wM_{\Eins_T^{\vee}})\neq 0$. Since $\cV$ is exact and it 
kills the representations on which $G$ acts trivially we deduce that 
\begin{equation}\label{gohome3}
\cV(\wM_{\Eins_T^{\vee}}/\nn_0  \wM_{\Eins_T^{\vee}})\cong \psi_1, \quad \cV(\wM_{\alpha^{\vee}}/\nn_0  \wM_{\alpha^{\vee}})\cong \psi_2.
\end{equation}
Lemma \ref{thruOrd} says that $\Hom^{cont}_{\OO}(\wM_{\Eins_T^{\vee}}/\nn_0 \wM_{\Eins_T^{\vee}}, L)$ and    
$\Hom^{cont}_{\OO}(\wM_{\alpha^{\vee}}/\nn_0 \wM_{\alpha^{\vee}}, L)$ are parabolic inductions of unitary characters. 
As $\VV((\Indu{P}{G}{\chi_1\otimes \chi_2\varepsilon^{-1}})_{cont})\cong \chi_2$ and the central character is trivial we deduce the assertion.
\end{proof}

\subsection{Banach space representations}\label{Banachagain}

Let $\psi: T\rightarrow L^{\times}$ be a unitary character. It is shown in \cite[5.3.4]{emcoates} that if 
$\psi\neq \psi^s$ then $(\Indu{P}{G}{\psi})_{cont}$ is irreducible and otherwise $\psi$ factors through $\det$, and 
so extends to $\psi: G\rightarrow L^{\times}$ and we have a non-split exact sequence of admissible unitary $L$-Banach space representations
\begin{equation}\label{completeStein}
0\rightarrow \psi\rightarrow (\Indu{P}{G}{\psi})_{cont}\rightarrow \widehat{\Sp}\otimes \psi\rightarrow 0,
\end{equation}
where $\widehat{\Sp}$ is the universal unitary completion of the smooth Steinberg representation over $L$. Moreover, 
$\widehat{\Sp}$ is irreducible, see \cite[4.5.1]{breuilL}, \cite[5.1.8 (1)]{emcoates}.

Let $\zeta: Z\rightarrow L^{\times}$ be a continuous unitary character and let $\Pi$ be an admissible
 unitary $L$-Banach space representation of $G$ with a central character $\zeta$ and let $\Theta$ be 
an open bounded $G$-invariant lattice in $\Pi$. Let $\eta: \Qp^{\times}\rightarrow k^{\times}$ be a smooth character.

\begin{lem}\label{liftetatwist}  If $\Theta\otimes_{\OO} k$ 
contains $\eta\circ \det$, $\Sp\otimes\eta\circ \det$ or $(\Indu{P}{G}{\alpha})\otimes\eta\circ\det$ as a subquotient 
then there exist a unique continuous unitary character $\tilde{\eta}: \Qp^{\times}\rightarrow L^{\times}$ such that
$\zeta=\tilde{\eta}^2$ and $\tilde{\eta}\equiv \eta \pmod{\pL}$. 
\end{lem}
\begin{proof} Since $\Pi$ is unitary, the central character $\zeta$ is unitary and $Z$ acts on $\Theta\otimes_{\OO} k$ by the  character $\zeta$ modulo $\pL$. Since the central character of $\eta\circ \det$, $\Sp\otimes\eta\circ \det$ and 
 $(\Indu{P}{G}{\alpha})\otimes\eta\circ \det$ is $\eta^2$, we deduce that $\zeta \equiv \eta^2\pmod{\pL}$.
Let $[\eta]: \Qp^{\times}\rightarrow \OO^{\times}$ be the Teichm\"uller lift of $\eta$. Then 
$\zeta [\eta]^{-2}$ takes values in $1+\pL$. Since $p\neq 2$ we may take a square root by the usual power
series expansion. Let $\tilde{\eta}:= [\eta] \sqrt{\zeta [\eta]^{-2}}$. This proves existence. For the uniqueness 
we may assume that both $\eta$ and $\zeta$ are trivial, in which case the assertion follows since (as $p\neq 2$) the equation
$X^2-1$ has a unique solution in $L$, which is congruent to $1$ modulo $\pL$.
\end{proof}

\begin{prop}\label{boundcases} Suppose that $\Pi$ is absolutely irreducible  then:
\begin{itemize}
\item[(i)] if $\Theta\otimes_{\OO} k$ contains $\eta\circ \det$ and does not contain 
$(\Indu{P}{G}{\alpha})\otimes \eta\circ \det$ as subquotients and
\begin{itemize} 
\item[(a)] if $\Theta\otimes_{\OO} k$ contains  $\Sp\otimes \eta\circ \det$ as a subquotient
then $\Pi\cong (\Indu{P}{G}{\psi})_{cont}$ and $\overline{\Pi}\cong (\Eins\oplus\Sp) \otimes\eta \circ \det$;
\item[(b)] if $\Theta\otimes_{\OO} k$ does not contain $\Sp\otimes \eta\circ \det$ as a subquotient
then $\Pi\cong \tilde{\eta}\circ \det$ and $\overline{\Pi}=\eta\circ \det$;
\end{itemize}
\item[(ii)]  if $\Theta\otimes_{\OO} k$ does not contain $\eta\circ \det$ and contains $\Sp\otimes \eta\circ \det$ as subquotients
then $\Pi\cong \widehat{\Sp}\otimes \tilde{\eta}\circ \det$ and $\overline{\Pi}\cong \Sp\otimes \eta\circ \det$.
\item[(iii)] if  $\Theta\otimes_{\OO} k$  contains $(\Indu{P}{G}{\alpha})\otimes \eta\circ \det$ and does not contain 
$\Sp \otimes \eta\circ\det$ as subquotients then $\Pi\cong (\Indu{P}{G}{\psi})_{cont}$ and 
$\overline{\Pi}\cong (\Indu{P}{G}{\alpha})\otimes \eta\circ \det$.
\end{itemize}
\end{prop}
\begin{proof} After twisting by $\tilde{\eta}^{-1}\circ \det$, constructed in Lemma \ref{liftetatwist}, we may assume that 
$\eta$ and $\zeta$ are trivial. Let $\pi_{\alpha}=\Indu{P}{G}{\alpha}$ and let $\wP_{\Eins_G^{\vee}}$ , $\wP_{\Sp^{\vee}}$ and 
$\wP_{\pi_{\alpha}^{\vee}}$ be projective envelopes of $\Eins_G^{\vee}$, $\Sp^{\vee}$ and $\pi_{\alpha}^{\vee}$ in 
$\dualcat_{G/Z}(\OO)$. 
Let $\wP_{\Eins_T^{\vee}}$ be a projective envelope of the trivial representation of $T$ in $\dualcat_{T/Z}(\OO)$ and 
let $\wM_{\Eins_T^{\vee}}:=(\Indu{P}{G}{(\wP_{\Eins_T^{\vee}})^{\vee}})^{\vee}$. Recall that \eqref{1B} is an exact sequence:
\begin{equation}\label{liftJ1NG}
0\rightarrow \wP_{\pi_{\alpha}^{\vee}}\rightarrow \wP_{\Eins_G^{\vee}}\rightarrow \wM_{\Eins_T^{\vee}}\rightarrow 0.
\end{equation} 
Lemma  \ref{contextGL2} says that the Schikhof dual $\Theta^d$ is an object of $\dualcat(\OO)$. Suppose
that $\Theta\otimes_{\OO} k$ contains $\Eins$  and does not contain $\pi_{\alpha}$ as  subquotients. Then
Lemma \ref{pisub} implies that $\Hom_{\dualcat(\OO)}(\wP_{\pi_{\alpha}^{\vee}}, \Theta^d)=0$ and 
$\Hom_{\dualcat(\OO)}(\wP_{\Eins_G^{\vee}}, \Theta^d)\neq 0$. Using \eqref{liftJ1NG} we get 
$\Hom_{\dualcat(\OO)}(\wM_{\Eins_T^{\vee}}, \Theta^d)\neq 0$.
The assertion in  (i)  follows from Proposition \ref{QofM}.

Let $(\Indu{P}{G}{\Eins})_{cont}^0$ be a unit ball in $(\Indu{P}{G}{\Eins})_{cont}$ with respect to the supremum norm. Let $(\widehat{\Sp})^0$ be the 
image of $(\Indu{P}{G}{\Eins})_{cont}^0$ inside $\widehat{\Sp}$, then $(\widehat{\Sp})^0$  is an open bounded $G$-invariant lattice in 
$\widehat{\Sp}$. Since
$(\Indu{P}{G}{\Eins})_{cont}^0\otimes_{\OO} k\cong \Indu{P}{G}{\Eins}$ we deduce that $(\widehat{\Sp})^0\otimes_{\OO} k \cong \Sp$ and hence 
$((\widehat{\Sp})^0)^d\otimes_{\OO} k \cong \Sp^{\vee}$. Now using \eqref{resSpNG} and Corollary  
\ref{projaretfree3} we get an exact sequence 
$$ \wP_{\Eins_G^{\vee}}^{\oplus 2} \rightarrow \wP_{\Sp^{\vee}}\rightarrow ((\widehat{\Sp})^0)^d\rightarrow 0.$$
If $\Theta\otimes_{\OO} k$ contains $\Sp$ and does not contain $\Eins$ then $\Hom_{\dualcat(\OO)}(\wP_{\Sp^{\vee}}, \Theta^d)\neq 0$
and $\Hom_{\dualcat(\OO)}(\wP_{\Eins_G^{\vee}}, \Theta^d)=0$. Hence, $\Hom_{\dualcat(\OO)}(((\widehat{\Sp})^0)^d, \Theta^d)\neq 0$ and 
so dually $\Hom_{L[G]}^{cont}(\Pi, \widehat{\Sp})\neq 0$. As both representations are irreducible and admissible we deduce that 
$\Pi\cong \widehat{\Sp}$.

The proof of  (iii) is identical  to the proof of (i), using \eqref{2B} instead of \eqref{1B}, and $\wM_{\alpha^{\vee}}$ instead of $\wM_{\Eins_T^{\vee}}$.
 \end{proof}

\begin{thm}\label{mainNGII} Suppose that $\Pi$ is absolutely irreducible and $\Theta\otimes_{\OO} k$ contains 
$\eta\circ \det$, $\Sp\otimes \eta\circ \det$ or $(\Indu{P}{G}{\alpha})\otimes \eta\circ \det$ as a subquotient then 
$\overline{\Pi}$ is contained in $(\Eins \oplus \Sp \oplus \Indu{P}{G}{\alpha})\otimes \eta \circ \det$.
Moreover, if the inclusion is not an isomorphism then we are in one of the cases of Proposition \ref{boundcases}.
\end{thm} 
\begin{proof} By twisting we may assume that $\zeta$ and $\eta$ are both trivial. Let $\pi$ be either 
$\Eins$, $\Sp$ or $\pi_{\alpha}$ and $\wP_{\pi^{\vee}}$ a projective envelope of $\pi^{\vee}$ in $\dualcat(\OO)$. 
If $\Pi$ is not one of the representations described in Proposition \ref{boundcases} then 
$\overline{\Pi}$ contains $\Eins$, $\Sp$ and $\pi_{\alpha}$. Thus it follows from Lemma \ref{mult=rank0} that 
$\Hom_{\dualcat(\OO)}(\wP_{\pi^{\vee}}, \Theta^d)$ is non-zero. Since by Corollary \ref{rings}, $\End_{\dualcat(\OO)}(\wP_{\pi^{\vee}})\cong R^{\psi}$ 
is commutative and $\Pi$ is absolutely irreducible, we deduce from Theorem \ref{furtherDBan} 
that $\Hom_{\dualcat(\OO)}(\wP_{\pi^{\vee}}, \Theta^d)_L$ is an absolutely irreducible finite dimensional $R^{\psi}[1/p]$-module. 
Hence,  $\Hom_{\dualcat(\OO)}(\wP_{\pi^{\vee}}, \Theta^d)_L$ is one dimensional and Lemma \ref{mult=rank0} implies that 
$\pi$ occurs in $\overline{\Pi}$ with multiplicity $1$.
\end{proof}

Let $\BB=\{\eta\circ \det, \Sp\otimes \eta\circ \det, (\Indu{P}{G}{\alpha})\otimes \eta\circ \det\}$, 
$\pi_{\BB}:=(\Eins \oplus \Sp\oplus \pi_{\alpha})\otimes \eta\circ \det$, 
$\wP_{\BB}$ a projective envelope of $\pi_{\BB}^{\vee}$ in $\dualcat(\OO)$ and $\wE_{\BB}:=\End_{\dualcat(\OO)}(\wP_{\BB})$. 
The ring $\wE_{\BB}$ is a finitely generated module over its centre, and the centre is naturally isomorphic to 
$R^{\mathrm{ps}, \varepsilon\zeta}_{\chi}$, see Theorem \ref{ZNGII} and Remark \ref{pseudocentre}, where 
$R^{\mathrm{ps}, \varepsilon\zeta}_{\chi}$ is the universal deformation ring 
parameterizing $2$-di\-men\-sio\-nal pseudocharacters with determinant $\zeta\varepsilon$ lifting $\chi:=\eta + \omega \eta$. 
Let $\Ban^{\mathrm{adm}}_{G,\zeta}(L)^{\BB}$ be as in Proposition \ref{blockdecompB} and let $\Ban^{\mathrm{adm. fl}}_{G,\zeta}(L)^{\BB}$
be the full subcategory consisting of objects of finite length. Let $\Pi$ be in $\Ban^{\mathrm{adm. fl}}_{G,\zeta}(L)^{\BB}$, 
choose an open bounded $G$-invariant lattice $\Theta$ and let $\md(\Pi):=\Hom_{\dualcat(\OO)}(\wP_{\BB}, \Theta^d)\otimes_{\OO} L$. 
It follows from Proposition \ref{longproof} that $\md(\Pi)$ is a finite dimensional $L$-vector space with continuous $\wE_{\BB}$-action. 
Let $\nn$ be a maximal ideal in $R^{\mathrm{ps}, \varepsilon\zeta}_{\chi}[1/p]$ and
$\Ban^{\mathrm{adm. fl}}_{G, \zeta}(L)^{\BB}_{\nn}$ the full subcategory of $\Ban^{\mathrm{adm. fl}}_{G,\zeta}(L)^{\BB}$
consisting of those $\Pi$ such that $\md(\Pi)$ is killed by a power of $\nn$.

\begin{cor} We have an equivalence of categories 
$$\Ban^{\mathrm{adm. fl}}_{G,\zeta}(L)^{\BB}\cong 
\bigoplus_{\nn\in \MaxSpec R_{\chi}^{\mathrm{ps},\zeta \varepsilon}[1/p]}\Ban^{\mathrm{adm. fl}}_{G,\zeta}(L)^{\BB}_{\nn}.$$
The category $\Ban^{\mathrm{adm. fl}}_{G, \zeta}(L)^{\BB}_{\nn}$ is anti-equivalent to the category 
of modules of finite length of the  
$\nn$-adic completion of $\wE_{\BB}[1/p]$.
\end{cor} 
\begin{proof} Apply Theorem \ref{furtherDBan} with $\dualcat(\OO)=\dualcat(\OO)^{\BB}$.
\end{proof}

Let $\Pi$ in $\Ban^{\mathrm{adm. fl}}_{G, \zeta}(L)^{\BB}$ be absolutely irreducible, we say that $\Pi$ is \textit{ordinary} 
if it is isomorphic to  one of the representations in Proposition \ref{boundcases}, otherwise we say that $\Pi$ 
is \textit{non-ordinary}.

The ring $\wE_{\BB}$ is described explicitly in  \S \ref{Thecentre}. However, in many cases one can give 
a simpler description of the category $\Ban^{\mathrm{adm. fl}}_{G,\zeta}(L)^{\BB}_{\nn}$. Let $\nn$
 be a maximal ideal of $R_{\chi}^{\mathrm{ps},\zeta \varepsilon}[1/p]$ with residue field $L$, let $T_{\nn}: \gal\rightarrow L$ be 
the pseudocharacter corresponding to $\nn$ and let 
$\Irr(\nn)$ denote the set (of equivalence classes of) irreducible objects in  
$\Ban^{\mathrm{adm. fl}}_{G,\zeta}(L)^{\BB}_{\nn}$. 

\begin{prop}\label{NgIIred} 
\begin{itemize} 
\item[(i)] if $T_{\nn}= \tilde{\eta}+\tilde{\eta}\varepsilon$ then 
$$\Irr(\nn)=\{ \tilde{\eta}\circ \det, \widehat{\Sp}_L\otimes \tilde{\eta}\circ \det, 
(\Indu{P}{G} {\tilde{\eta}\varepsilon \otimes \tilde{\eta}\varepsilon^{-1}})_{cont}\}.$$
\item[(ii)] if $T_{\nn}= \psi_1 + \psi_2$ with $\psi_1, \psi_2: \gal\rightarrow L^{\times}$ continuous homomorphisms 
and $T_{\nn}\neq \tilde{\eta}+\tilde{\eta}\varepsilon$ then 
 $$\Irr(\nn)=\{ (\Indu{P}{G}{\psi_1\otimes \psi_2 \varepsilon^{-1}})_{cont}, (\Indu{P}{G}{\psi_2\otimes \psi_1 \varepsilon^{-1}})_{cont}\}$$
and $\Ban^{\mathrm{adm. fl}}_{G,\zeta}(L)^{\BB}_{\nn}$ is naturally equivalent to the category of modules of finite length 
of the $\nn$-adic completion of $\wE[1/p]$, see \eqref{ringqcat} for definition and description of $\wE$.
\item[(iii)] if $T_{\nn}$ is irreducible then $\Irr(\nn)=\{\Pi_{\nn}\}$ with $\Pi_{\nn}$ absolutely irreducible non-ordinary.
The category $\Ban^{\mathrm{adm. fl}}_{G,\zeta}(L)^{\BB}_{\nn}$ is equivalent to the category of modules of finite length 
of the $\nn$-adic completion of $R_{\chi}^{\mathrm{ps},\zeta \varepsilon}[1/p]$.
\end{itemize}
\end{prop}
\begin{proof} Suppose that $\nn$ contains the reducible locus in $R^{\mathrm{ps}, \varepsilon\zeta}_{\chi}[1/p]$. Since 
$\Eins\neq \omega$ and the residue field of $\nn$ is $L$, 
we get that $T_{\nn}=\psi_1 + \psi_2$ with $\psi_1, \psi_2: \gal\rightarrow L^{\times}$ continuous homomorphisms.
It follows from Proposition \ref{gohome4} that $\Irr(\nn)$ contains the semi-simplification of 
unitary principal series appearing in (ii). Recall that $(\Indu{P}{G}{\chi_1\otimes \chi_2})_{cont}$ is irreducible 
if and only if $\chi_1\neq \chi_2$. We get that $|\Irr(\nn)|\ge 3$ if $T_{\nn}= \tilde{\eta}+\tilde{\eta}\varepsilon$
and $|\Irr(\nn)|\ge 2$, otherwise. The representations in $\Ban^{\mathrm{adm. fl}}_{G,\zeta}(L)^{\BB}$,
on which $\SL_2(\Qp)$ acts trivially, form a thick subcategory. The quotient category $\QBan^{\mathrm{adm. fl}}_{G,\zeta}(L)^{\BB}$
is equivalent to the category of $\wE[1/p]$-modules of finite length, Theorem \ref{furtherDBan} and Lemma \ref{Z1}.
Since we have fixed a central character and $p>2$ any $\Pi$ in $\Ban^{\mathrm{adm. fl}}_{G,\zeta}(L)^{\BB}$ on which 
$\SL_2(\Qp)$ acts trivially is isomorphic to $\tilde{\eta}^{\oplus m}$. Hence, if $\nn$ does not kill $\md(\tilde{\eta})$ then 
$$\Ban^{\mathrm{adm. fl}}_{G,\zeta}(L)^{\BB}_{\nn}\cong \QBan^{\mathrm{adm. fl}}_{G,\zeta}(L)^{\BB}_{\nn}$$
and the last category is equivalent to the category of modules of finite length of the $\nn$-adic 
completion of $\wE[1/p]$ by Theorem \ref{furtherDBan}. This category has $2$-irreducible objects 
by Lemma \ref{nrmod}. If $\nn$  kills $\md(\tilde{\eta})$
then $ \QBan^{\mathrm{adm. fl}}_{G,\zeta}(L)^{\BB}_{\nn}$ has one irreducible object less than $\Ban^{\mathrm{adm. fl}}_{G,\zeta}(L)^{\BB}_{\nn}$.
Again Lemma \ref{nrmod} allows us to conclude.

Suppose that $\nn$ does not contain the reducible locus then
it follows from Proposition \ref{invertc} that $\Ban^{\mathrm{adm. fl}}_{G,\zeta}(L)^{\BB}_{\nn}$
is equivalent to the category of modules of finite length of the $\nn$-adic completion 
of $R_{\chi}^{\mathrm{ps},\zeta \varepsilon}[1/p]$. 
This category contains only one irreducible object and hence $\Ban^{\mathrm{adm. fl}}_{G,\zeta}(L)^{\BB}_{\nn}$
contains only one irreducible object  $\Pi$. Since all the ordinary representations have already appeared in 
(i) and (ii) and $|\Irr(\nn)|>1$ in those cases, we deduce that $\Pi$ cannot be ordinary.
\end{proof}

\section{\texorpdfstring{$p$-adic Langlands correspondence}{$p$-adic Langlands correspondence}}\label{padicLang}

Let $\Pi$ be a unitary irreducible admissible $L$-Banach space representation of $G$ with a central character. 
We say that $\Pi$ is \textit{ordinary} if $\Pi$ is either a unitary character $\Pi\cong \eta\circ \det$, 
a twist of the universal completion of the smooth Steinberg representation by a unitary character 
$\Pi\cong \widehat{\Sp}\otimes \eta\circ \det$ or $\Pi$ is a unitary parabolic induction of  a unitary character.
We assume throughout that $p\ge 5$.

\begin{thm}\label{redBan} Let $\Pi$ be a unitary admissible absolutely irreducible $L$-Ba\-na\-ch space representation of $G$ with
a central character and let $\Theta$ be an open bounded $G$-invariant lattice in $\Pi$. Then $\Theta\otimes_{\OO} k$ 
is of finite length. Moreover, one of the following holds: 
\begin{itemize}
\item[(i)] $\Theta\otimes_{\OO} k$ is absolutely irreducible supersingular;
\item[(iii)] $\Theta\otimes_{\OO} k$ is irreducible and 
\begin{equation}\label{bctol}
\Theta\otimes_{\OO} l\cong \Indu{P}{G}{\chi\otimes \chi^{\sigma} \omega^{-1}}\oplus \Indu{P}{G}{\chi^{\sigma}\otimes \chi \omega^{-1}},
\end{equation}
where $l$ is a quadratic extension of $k$, $\chi:\Qp^{\times}\rightarrow l^{\times}$ a smooth character and $\chi^{\sigma}$ is a conjugate of 
$\chi$ by the non-trivial element in $\Gal(l/k)$;
\item[(iii)] $(\Theta\otimes_{\OO} k)^{ss}\subseteq (\Indu{P}{G}{\chi_1\otimes \chi_2 \omega^{-1}})^{ss}\oplus  
(\Indu{P}{G}{\chi_2\otimes \chi_1 \omega^{-1}})^{ss}$ for some smooth characters $\chi_1, \chi_2: \Qp^{\times} \rightarrow k^{\times}$.
\end{itemize}
Further, the inclusion in (iii) is not an isomorphism if and only if  $\Pi$ is ordinary.
\end{thm}

\begin{proof} Let $\pi$ be an irreducible subquotient of $\Theta\otimes_{\OO} k$. Suppose that $\pi$ is absolutely irreducible. 
Then it follows from Theorems \ref{mainsuper}, \ref{mainGen}, \ref{mainNGI} and \ref{mainNGII} that either (i) or (iii) holds. 
Further, if the inclusion in (iii) is not an isomorphism then $\Pi$ is ordinary, see Corollaries \ref{cormainGen}, \ref{cormainNGI} and 
Theorem \ref{mainNGII}. If $\pi$ is not absolutely irreducible then it follows from Corollary \ref{irrsubQabs} that $\pi\otimes_k l$ is 
isomorphic to  $\Indu{P}{G}{\chi\otimes \chi^{\sigma} \omega^{-1}}\oplus \Indu{P}{G}{\chi^{\sigma}\otimes \chi \omega^{-1}}$. The previous argument
applied to $\Pi_{L'}$ where $L'$ is a quadratic  unramified extension of $L$ shows that 
$$(\Theta\otimes_{\OO} l)^{ss}\subseteq \Indu{P}{G}{\chi\otimes \chi^{\sigma} \omega^{-1}}\oplus  
\Indu{P}{G}{\chi^{\sigma}\otimes \chi \omega^{-1}}.$$
Since  $\Theta\otimes_{\OO} k$ contains $\pi$ we deduce that $\Theta\otimes_{\OO} k \cong \pi$. 
\end{proof}

We refer the reader to \S \ref{CMF} for the definition of the functors $\VV$ and $\cV$.
\begin{cor}\label{dimVPi} If $\Pi$ is a unitary admissible absolutely irreducible $L$-Ba\-na\-ch space representation of $G$ with
a central character then $\dim_L \VV(\Pi)\le 2$. Moreover, $\dim_L  \VV(\Pi)<2$ if and only if $\Pi$ is ordinary. 
\end{cor}
\begin{proof} Let $\Theta$ be an open bounded $G$-invariant lattice in $\Pi$. It follows from Theorem \ref{redBan} that 
$\dim_k \VV(\Theta\otimes_{\OO} k)\le 2$ and the equality is strict if and only if the inclusion in Theorem \ref{redBan} (iii) 
is not an isomorphism. Hence $\VV(\Theta)$ is a free $\OO$-module of rank at most $2$, see \cite[2.2.2]{kisin} or the proof Lemma \ref{flattofree}. 
Since $\VV(\Theta)$ is an $\OO$-lattice in $\VV(\Pi)$ we get the assertion.
\end{proof}

Let $\Pi$ be an absolutely irreducible non-ordinary unitary $L$-Banach space representation of $G$ 
with a central character $\zeta$. Then $\Pi$ is an object of $\Ban_{G,\zeta}^{\mathrm{adm}}(L)^{\BB}$ for some block $\BB$, 
Proposition \ref{blockdecompB}. Let $\overline{\Pi}$ be the semi-simp\-li\-fi\-ca\-tion of the reduction modulo $\varpi$ 
of an open bounded $G$-invariant lattice $\Theta$ in $\Pi$. Suppose that $\overline{\Pi}$ contains an absolutely irreducible representation;
this can be achieved by replacing $L$ with a quadratic  unramified extension. Then there are essentially four possibilities for $\BB$, 
described in Proposition \ref{blocksoverk}. Recall that $\Mod^{\mathrm{lfin}}_{G,\zeta}(\OO)^{\BB}$ is the full subcategory 
of $\Mod^{\mathrm{lfin}}_{G,\zeta}(\OO)$ consisting of representations with all the irreducible subquotients in $\BB$, and 
$\dualcat(\OO)^{\BB}$ is the full subcategory of $\dualcat(\OO)$ anti-equivalent to $\Mod^{\mathrm{lfin}}_{G,\zeta}(\OO)^{\BB}$
via Pontryagin duality. The centre of the category $\dualcat(\OO)^{\BB}$ is naturally isomorphic to 
$R^{\mathrm{ps}, \psi}_{\tr \VV(\overline{\Pi})}$, the deformation ring parameterizing $2$-dimensional pseudocharacters 
of $\gal$ with determinant $\psi=\varepsilon \zeta$ lifting $\tr \VV(\overline{\Pi})$, Corollaries \ref{Csuper}, \ref{Cgen}, 
\ref{CngI}, Theorem \ref{ZNGII}. Since $\Theta^d$ is an object of $\dualcat(\OO)^{\BB}$ and 
$\Pi\cong \Hom_{\OO}^{cont}(\Theta^d, L)$ we obtain a ring homomorphism:
$$x: R^{\mathrm{ps}, \psi}_{\tr \VV(\overline{\Pi})}[1/p]\rightarrow \End_G^{cont}(\Pi)\cong L,$$
where the last isomorphism follows from  Corollary \ref{absirre}. 

\begin{prop}\label{welldefined} The representation $\VV(\Pi)$ is absolutely irreducible with determinant $\varepsilon \zeta$. 
 Moreover,  $\tr \VV(\Pi)$ is equal to the pseudocharacter corresponding to $x\in \Spec  R^{\mathrm{ps}, \psi}_{\tr \VV(\overline{\Pi})}[1/p]$.
\end{prop}
\begin{proof} Let $T_x: \gal\rightarrow L$ be the pseudocharacter corresponding to $x$. There exists 
a unique semi-simple continuous representation $V_x$ of $\gal$, defined over a finite extension
of $L$, such that $\tr V_x= T_x$ and $\det V_x= \zeta \varepsilon$, \cite[Thm.1]{taylor}. The representation 
$V_x$ is absolutely irreducible, since otherwise Corollaries \ref{genred}, \ref{NgIred} and Proposition \ref{NgIIred}
would imply that $\Pi$ is ordinary. It follows from Corollaries 
\ref{superkill}, \ref{genkill}, \ref{NgIkill} and \ref{NgIIkill} that $\VV(\Pi)$ is killed by 
$g^2-T_x(g)g +\varepsilon\zeta(g)$ for all $g\in \gal$. Since $\VV(\Pi)$ is $2$-dimensional by Corollary \ref{dimVPi},   
the main result of \cite{boston} implies that $V_x\cong \VV(\Pi)$.
\end{proof}

\begin{thm}\label{Lbij}  Assume $p\ge 5$, the functor $\VV$ induces a bijection between isomorphism classes of 
\begin{itemize}
\item[(i)] absolutely irreducible admissible unitary non-ordinary $L$-Banach space representations of $G$ with the central character $\zeta$, and 
\item [(ii)] absolutely irreducible $2$-di\-men\-sio\-nal continuous $L$-representations of $\gal$ with determinant equal to $\zeta \varepsilon$,
\end{itemize} 
where $\varepsilon$ is the cyclotomic character, and we view $\zeta$ as  a character of $\gal$ via the class field theory. 
\end{thm}
\begin{proof} It follows from Proposition \ref{welldefined} that $\VV$ maps one set to the other. The surjectivity follows from 
\cite[2.3.8]{kisin}. We show injectivity: suppose that $\VV(\Pi_1)\cong \VV(\Pi_2)$. As
$\VV(\overline{\Pi})\cong \overline{\VV(\Pi)}$, 
Theorem \ref{redBan} implies that  $\Pi_1$ and $\Pi_2$ lie in 
$\Ban_{G,\zeta}^{\mathrm{adm}}(L)^{\BB}$ for the same block $\BB$. Let $x\in \Spec  R^{\mathrm{ps}, \psi}_{\tr \VV(\overline{\Pi})}[1/p]$
be the maximal ideal corresponding to $\tr \VV(\Pi_1)=\tr \VV(\Pi_2)$. Proposition \ref{welldefined} implies that 
$\Pi_1$ and $\Pi_2$ are killed by $x$ and hence are objects of  $\Ban_{G,\zeta}^{\mathrm{adm}}(L)^{\BB}_x$. Since $\VV(\Pi_1)\cong \VV(\Pi_2)$
is absolutely irreducible this category contains only one irreducible object, see Corollaries \ref{superBanfin}, \ref{genirr}, \ref{NgIirr}
and Proposition \ref{NgIIred}. Hence, $\Pi_1\cong \Pi_2$. 
\end{proof}

Let $V$ be an absolutely irreducible $2$-dimensional $L$-representation of $\gal$ with determinant $\psi:=\zeta\varepsilon$. Let $\overline{V}$ be the semi-simplification of a reduction modulo $\varpi$ of a $\gal$-stable $\OO$-lattice in $V$. 
We assume that $\overline{V}$ is either absolutely irreducible or a direct sum of two one dimensional representations. This can always be achieved by replacing $L$ by a quadratic unramified extension.
Let $R^{\psi}_V$ be the deformation ring representing the deformation problem of $V$ with determinant 
$\zeta \varepsilon$ to  local artinian $L$-algebras, and let $V^{\mathrm{u}}$ be the universal deformation of  $V$ with the determinant $\zeta\varepsilon$.

\begin{lem}\label{homs_iso}  Let $\md_1$, $\md_2$ be $R^{\psi}_V$-modules of finite length. Then the natural map $\Hom_{R^{\psi}_V} (\md_1, \md_2)\rightarrow \Hom_{\gal}(\md_1\otimes_{R^{\psi}_V} V^{\mathrm{u}}, \md_2\otimes_{R^{\psi}_V} V^{\mathrm{u}})$ is an isomorphism. 
\end{lem}
\begin{proof} The assertion is true if both modules are of length one, since then both groups are isomorphic to $\End_{\gal}(V)\cong L$. Moreover, 
$$\Ext^1_{R^{\psi}_V}(L, L)\cong \Hom(R^{\psi}_V, L[\epsilon]/(\epsilon^2))\cong \Ext^1_{\gal}(V, V).$$
 Given this we may finish the proof in the same way as in Proposition \ref{equivofcatsII}: we argue by induction on $\ell(\md_1)+\ell(\md_2)$ that the functor 
  $\md\mapsto \md\otimes_{R^{\psi}_V} V^{\mathrm{u}}$ induces an isomorphism between 
$\Hom$-groups  and an injection  on $\Ext^1$-groups. 
\end{proof}

\begin{cor}\label{get_m} Let $\md$ be a $R^{\psi}_V$-module of finite length. Then  $$\md\cong \Hom^{cont}_{\gal}(V^{\mathrm{u}}, \md\otimes_{ R^{\zeta\varepsilon}_V} V^{\mathrm{u}}).$$
\end{cor}
\begin{proof} Let $\mm$ be the maximal ideal of $R:=R^{\psi}_V$, and let $V(\md):=\md\otimes_R V^{\mathrm{u}}$. Then 
$\Hom^{cont}_{\gal}(V^{\mathrm{u}}, V(\md))\cong \underset{\longrightarrow}{\lim} 
\Hom_{\gal}(V^{\mathrm{u}}/\mm^n V^{\mathrm{u}}, V(\md))\cong \underset{\longrightarrow}{\lim} \Hom_{R}(R/\mm^n, \md)\cong \md$,
where the second isomorphism follows from Lemma \ref{homs_iso}.
\end{proof}

Let $\Pi$ be an absolutely irreducible admissible unitary $L$-Banach space representation of $G$ with central character $\zeta$, corresponding to $V$ by Theorem \ref{Lbij}, so that $\VV(\Pi)\cong V$. 
Let $\Ban_{G, \zeta}^{\mathrm{adm.fl}}(L)$ be the category of admissible unitary $L$-Banach space representations of $G$ of finite length with the central character $\zeta$ and let  
$\Ban_{G, \zeta}^{\mathrm{adm.fl}}(L)_{\Pi}$ be the full subcategory of  consisting of the representations with all irreducible subquotients isomorphic to $\Pi$.

\begin{thm}\label{holiday} Let $\mathrm{B}\in \Ban^{\mathrm{adm}}_{G, \zeta}(L)$ be of finite length with all irreducible subquotients isomorphic to $\Pi$ and let $\md(\mathrm{B})= \Hom_{\gal}(V^{\mathrm{u}}, \cV(\mathrm B))$ then 
$$\cV(\mathrm B)\cong \md(\mathrm{B})\otimes_{R^{\psi}_V} V^{\mathrm{u}}.$$ Moreover, the functor $\mathrm B\mapsto \md(\mathrm B)$ induces an anti-equivalence of categories between 
$\Ban_{G, \zeta}^{\mathrm{adm.fl}}(L)_{\Pi}$ and the category of $R^{\psi}_V$-modules of finite length.
\end{thm} 
\begin{proof} Let $\BB$ be the block corresponding to $\overline{V}$, so that if $\overline{V}$ is absolutely irreducible then $\BB=\{\pi\}$, where $\pi$ is a 
supersingular representation of $G$, with $\VV(\pi)\cong \overline{V}$, and if $\overline{V}\cong \chi_1\oplus \chi_2$, then $\BB$ consists of all the irreducible subquotients 
of the principal series representations $\Indu{P}{G}{\chi_1\otimes\chi_2\omega^{-1}}$ and $\Indu{P}{G}{\chi_2\otimes\chi_1\omega^{-1}}$, and let $\wZ$ be the centre of the category 
$\dualcat(\OO)^{\BB}$.  Let $\Theta$ be an open bounded $G$-invariant lattice in $\Pi$, and let $\overline{\Pi}$ denote the semi-simplification of $\Theta/\varpi\Theta$. 
The isomorphism $\VV(\Pi)\cong V$ implies that $\VV(\overline{\Pi})\cong \overline{V}$. This implies that  $\Pi$ is an object of  $\Ban_{G,\zeta}^{\mathrm{adm}}(L)^{\BB}$ and $\Theta^d$ is an object of $\dualcat(\OO)^{\BB}$.
The  action of $\wZ$ on $\Theta^d$ induces  a ring homomorphism $x:\wZ[1/p]\rightarrow \End_G(\Pi)\cong L$, and we let $\nn$ be the kernel of $x$.  Let $\Ban_{G,\zeta}^{\mathrm{adm.fl}}(L)^{\BB}_\nn$ be the 
full subcategory of $\Ban_{G,\zeta}^{\mathrm{adm}}(L)^{\BB}$ consisting of all Banach space representations of finite length, which are killed by some power of $\nn$.  We note that $\Pi$ is in $\Ban_{G,\zeta}^{\mathrm{adm}}(L)^{\BB}_\nn$ by construction of $\nn$. 
Moreover, it follows from Corollaries \ref{superBanfin},  \ref{genirr}, \ref{pups} and \ref{NgIirr} and  Proposition \ref{NgIIred} (iii) that $\Pi$ is the only irreducible object  in the category. Hence, 
$\Ban_{G,\zeta}^{\mathrm{adm.fl}}(L)^{\BB}_\nn= \Ban^{\mathrm{adm.fl}}_{G, \zeta}(L)_{\Pi}$.  The second part of the Corollaries referred to above says that   $\Ban_{G,\zeta}^{\mathrm{adm.fl}}(L)^{\BB}_\nn$ is anti-equivalent to the category 
of modules of finite length over the $\nn$-adic completion of $\wZ[1/p]$.  To prove the theorem we need to write out how this equivalence is realized.

If $\overline{V}$ is absolutely irreducible we let $\pi$ be supersingular representation of $G$, with $\VV(\pi)\cong \overline{V}$. If $\overline{V}\cong \chi_1\oplus \chi_2$ 
then we let $\pi=\Indu{P}{G}{\chi_1\otimes \chi_2 \omega^{-1}}$. Since $p\ge 5$ we may  assume without loss of generality that $\chi_1\neq \chi_2 \omega^{-1}$, so that $\pi$ is irreducible. Let $\wP$ be  a projective envelope 
of $\pi^{\vee}$ in $\dualcat(\OO)$ and let $\wE=\End_{\dualcat(\OO)}(\wP)$. The action of $\wZ$ on $\wP$ induces a homomorphism of rings $\wZ\rightarrow \wE$. 
If $\pi$ is supersingular, or $\pi$ is a principal series with $\chi_1\neq \chi_2$ then this map is an isomorphism, see Proposition \ref{superdone}, Corollary  \ref{Cgen} and Proposition \ref{rpsr},
Corollary \ref{rings} and Theorem \ref{ZNGII}. If $\pi\cong \Indu{P}{G}{\chi\otimes \chi \omega^{-1}}$ then $\BB=\{\pi\}$ and so $\wZ$ is the centre of $\wE$. 

The functor $\mathrm B\mapsto \md(\mathrm B):=\Hom_{\dualcat(\OO)}(\wP, \Theta^d)_L$, where $\Theta$ is any open bounded $G$-invariant lattice in $\mathrm B$, is faithfull 
when restricted to $\Ban^{\mathrm{adm.fl}}_{G, \zeta}(L)_{\Pi}$:  since $\pi$ appears as a subquotient of of $\overline{\Pi}$, Lemma \ref{pisub} implies that $\md(\Pi)\neq 0$, and the assertion follows from the exactness of $\md$, see Lemma \ref{PitomPI}. 
Since $\Ban^{\mathrm{adm.fl}}_{G, \zeta}(L)_{\Pi}=\Ban_{G,\zeta}^{\mathrm{adm}}(L)^{\BB}_\nn$ and $\md$ is faithfull, it follows from Theorem \ref{furtherDBan} that $\md$ induces an anti-equivalence 
of categories between $\Ban^{\mathrm{adm.fl}}_{G, \zeta}(L)_{\Pi}$ and the category of modules of finite length of the $\nn$-adic completion of $\wE[1/p]$. The inverse functor $\md\mapsto \Pi(\md)$ is defined in Definition 
\ref{defPimd}. So that for $\mathrm B\in \Ban^{\mathrm{adm.fl}}_{G, \zeta}(L)_{\Pi}$ and $\Theta$ an open bounded $G$-invariant lattice in $\mathrm B$, we have $\Theta^d\cong \Hom_{\dualcat(\OO)}(\wP, \Theta^d)\wtimes_{\wE} \wP$. 
Lemma \ref{VT2} implies that $\cV(\Theta^d)\cong  \Hom_{\dualcat(\OO)}(\wP, \Theta^d)\wtimes_{\wE} \cV(\wP)$. Since $\pi$ occurs in $\overline{\mathrm{B}}$ with finite multiplicity, Lemma \ref{mult=rank0} implies that 
$\Hom_{\dualcat(\OO)}(\wP, \Theta^d)$ is a free $\OO$-module of finite rank. In particular, it is finitely generated over $\wE$. Since $\wE$ is notherian, the module is finitely presented and hence we may replace 
$\wtimes$ with $\otimes$. Hence, $\cV(\mathrm B)\cong \md(\mathrm B)\otimes_{\wE} \cV(\wP)$. As  $\md(\mathrm B)$ is killed by a power of $\nn$ we may replace $\wE$ with the $\nn$-adic completion of $\wE[1/p]$ and $\cV(\wP)$ with the $\nn$-adic completion of $\cV(\wP)_L$.

To finish the proof we only have to relate the $\nn$-adic completion of $\wE[1/p]$ to $R^{\psi}_V$ and the $\nn$-adic completion of $\cV(\wP)_L$ to $V^{\mathrm u}$. Assume that $\pi\not \cong \Indu{P}{G}{\chi\otimes \chi\omega^{-1}}$ for any 
character $\chi$, so that $\wE$ is commutative. In this case we know that $\cV(\wP)$ is the universal deformation with determinant $\varepsilon \zeta$ of a $2$-dimensional representation $\rho$, and $\wE\cong R_{\rho}^{\psi}$ is the deformation 
ring representing this deformation problem, where $\rho\cong \VV(\pi)$ if $\pi$ is supersingular  and $\rho$ is a non-split extension of $\chi_2$ by $\chi_1$ if $\pi\cong \Indu{P}{G}{\chi_1\otimes \chi_2\omega^{-1}}$, see Proposition \ref{superdone}, Corollary \ref{wEcommGen},
Theorem \ref{varphisoNGII}, Corollary \ref{univdefII}. Since $V\cong \cV(\Pi)\cong L\otimes_{\wE, x} \cV(\wP)$, \cite[(2.3.5)]{kisin2} implies that the $\nn$-adic completion of $R^{\psi}_{\rho}[1/p]$ is isomorphic to 
$R^{\psi}_V$, and the $\nn$-adic completion of $\cV(\wP)_L$ is isomorphic to $V^{\mathrm{u}}$.  Hence, $\cV(\mathrm B)\cong \md(\mathrm B)\otimes_{R^{\psi}_V} V^{\mathrm u}$  for all $\mathrm B\in \Ban_{G, \zeta}^{\mathrm{adm.fl}}(L)_{\Pi}$, and 
Corollary \ref{get_m} implies that $\md(\mathrm B)\cong \Hom_{\gal}(V^{\mathrm u}, \cV(\mathrm B))$. 

We assume that $\pi\cong \Indu{P}{G}{\chi\otimes \chi\omega^{-1}}$, for some character $\chi$, so that $\overline{V}\cong \chi\oplus \chi$. Let  $R^{\mathrm{ps}, \psi}_{\tr \overline{V}}$ be 
 the deformation ring parameterizing $2$-dimensional pseudocharacters  of $\gal$ with determinant $\psi=\varepsilon \zeta$ lifting $\tr \overline{V}$, and let 
$T: \gal\rightarrow R^{\mathrm{ps}, \psi}_{\tr \overline{V}}$ be the universal pseudocharacter with determinant $\psi$ lifting $\tr \overline{V}$. In this case $\wE$ is isomorphic to 
the opposite ring of $R^{\mathrm{ps}, \psi}_{\tr \overline{V}}[[\gal]]/J$, where $J$ is a closed two-sided ideal generated by the elements $g^2- T(g)g + \psi(g)$ for all $g\in \gal$, and $\wZ\cong  R^{\mathrm{ps}, \psi}_{\tr \overline{V}}$, see \S \ref{CandBsp}.
Moreover, $\cV(\wP)$ is a free $\wE$-module of rank $1$, see \S  \ref{defNgI}. It follows from \eqref{good_action} that if $\md$ is a compact right $\wE$-module then 
$\cV(\md\wtimes_{\wE} \wP)\cong \md\wtimes_{\wE} \cV(\wP)\cong \md$, where the action of $\gal$ on $\md$ is induced by the natural homomorphism $\gal\rightarrow  R^{\mathrm{ps}, \psi}_{\tr \overline{V}}[[\gal]]/J$. In particular, 
$\cV(\wP)\cong R^{\mathrm{ps}, \psi}_{\tr \overline{V}}[[\gal]]/J$ with $\gal$ acting on the left.  Since the specialization of $T$ at $\nn$ is the trace of $V$, the $\nn$-adic completion of $R^{\mathrm{ps}, \psi}_{\tr \overline{V}}$ is isomorphic to  $R^{\mathrm{ps}, \psi}_{\tr V}$.

Let $\mathcal E$ be the $\nn$-adic completion of $\wE[1/p]$.
Corollary \ref{weiterso4a} implies that $\mathcal E$ is isomorphic to the ring of $2\times 2$ -matrices over $R^{\mathrm{ps}, \psi}_{\tr V}$. Let $e=\bigl (\begin{smallmatrix} 1 & 0 \\ 0 & 0\end{smallmatrix} \bigr)$ then 
$e\mathcal E$ is a free $R^{\mathrm{ps}, \psi}_{\tr V}$-module of rank $2$ with a continuous $\gal$-action. Since for every invertible $2\times 2$-matrix $A$  we have $A + (\det A) A^{-1}= (\tr A ) \Id$,  the trace of $\gal$-representation on $e\mathcal E$ is equal to $T$, and 
the trace of $\gal$-representation on $L\otimes_{R^{\mathrm{ps}, \psi}_{\tr V}} e\mathcal E$ is equal to $\tr V$.  Since $V$ is irreducible, this implies $e\mathcal E$ is a deformation of $V$ to  $R^{\mathrm{ps}, \psi}_{\tr V}$, which induces  a ring homomorphism $R^{\psi}_V \rightarrow R^{\psi}_{\tr V}$.  Moreover, the composition 
$R^{\mathrm{ps}, \psi}_{\tr V}\rightarrow R^{\psi}_V \rightarrow R^{\mathrm{ps}, \psi}_{\tr V}$, where the first arrow is induced by taking a trace of a deformation of $V$, is the identity map.  Since $V$ is absolutely irreducible the first arrow is an isomorphism by \cite{Nyssen}, and hence the second arrow 
is an isomorphism, which implies that $e\mathcal E\cong V^{\mathrm u}$. Since the same argument applies with $1-e$ instead of $e$, we deduce that the $\nn$-adic completion of $\cV(\wP)_L$ is isomorphic to $V^{\mathrm u} \oplus V^{\mathrm u}$ as a $\gal$-representation, and $\mathcal E$ is 
the ring of $2\times 2$-matrices over $R^{\psi}_V$.  Thus the rings $R^{\psi}_V$ and $\mathcal E$ are Morita equivalent, which implies that $\cV(\mathrm B)\cong \Hom_{\gal}(V^{\mathrm u}, \cV(\mathrm B))\otimes_{R^{\psi}_V }V^{\mathrm u}$ for all $\mathrm B\in \Ban_{G, \zeta}^{\mathrm{adm.fl}}(L)_{\Pi}$.
\end{proof}
\begin{remar} Since $\Ban_{G,\zeta}^{\mathrm{adm.fl}}(L)^{\BB}_\nn= \Ban^{\mathrm{adm.fl}}_{G, \zeta}(L)_{\Pi}$, 
 it follows from Proposition \ref{blockdecompB} and Theorem \ref{furtherDBan} that the category $\Ban^{\mathrm{adm.fl}}_{G, \zeta}(L)_{\Pi}$ is a direct summand of $\Ban^{\mathrm{adm.fl}}_{G, \zeta}(L)$. Concretely this means 
 that every admissible unitary $L$-Banach space representation $\mathrm B$, which is of finite length and has a central character $\zeta$ decomposes as $\mathrm B \cong \mathrm B_1 \oplus \mathrm B_2$, where all the irreducible subquotients 
 of $\mathrm B_1$ are isomorphic to $\Pi$ and none of the irreducible subquotients of $\mathrm B_2$ is isomorphic to $\Pi$.
 \end{remar}

\begin{remar} We note that one may also prove an analog of Theorem \ref{holiday}, when $V$ is reducible. Let $\psi_1, \psi_2: \Qp^{\times}\rightarrow L^{\times}$ be unitary characters satisfying 
$\psi_1\psi_2= \varepsilon \zeta$.  Let $\chi_1, \chi_2: \Qp^{\times}\rightarrow k^{\times}$ be their reduction modulo $\varpi$. Let $\BB$ be the block corresponding to $\chi_1\oplus \chi_2$, $\wZ$ the centre of $\dualcat(\OO)^{\BB}$ and
$\nn$ the maximal ideal of $\wZ[1/p]$ corresponding to the pseudocharacter $\psi_1+\psi_2$.  Then it follows from Corollaries \ref{genred}, \ref{NgIred} and Proposition  \ref{NgIIred} that the irreducible representations of 
$\Ban^{\mathrm{adm.fl}}_{G, \zeta}(L)^{\BB}_{\nn}$  are precisely the irreducible subquotients of $(\Indu{P}{G}{\psi_1\otimes \psi_2 \varepsilon^{-1}})_{cont}$, $(\Indu{P}{G}{\psi_2\otimes \psi_1 \varepsilon^{-1}})_{cont}$.
Since $\Ban^{\mathrm{adm.fl}}_{G, \zeta}(L)^{\BB}_{\nn}$ is closed under subquotients and extensions in $\Ban^{\mathrm{adm}}_{G, \zeta}(L)$,  it is uniquely determined by its irreducible objects.  One then can reinterpret 
the anti-equivalence of categories between  $\Ban^{\mathrm{adm.fl}}_{G, \zeta}(L)^{\BB}_{\nn}$ and the category of modules of finite length over certain $\nn$-adic completions, see Corollaries \ref{genBanach}, \ref{pups}
and Proposition \ref{NgIIred} (i), (ii) and Remark \ref{VseeQ} in terms of the Galois side.

For example, if $\psi_1\psi_2^{-1}\neq \varepsilon^{\pm 1}, \Eins$, so that both unitary principal series representations are irreducible and distinct, then Theorem \ref{holiday} holds if we replace $V^{\mathrm u}$ with 
$V^{\mathrm u}_1\oplus V^{\mathrm u}_2$, and $R^{\psi}_{V}$ with $\End^{cont}_{\gal}( V^{\mathrm u}_1\oplus V^{\mathrm u}_2)$, where $V^{\mathrm u}_1$ is the universal deformation of the non-split extension $\psi_1$ by $\psi_2$, 
and $V^{\mathrm u}_2$ is the universal deformation of the non-split extension $\psi_2$ by $\psi_1$ with determinant $\varepsilon \zeta$. Our assumptions imply that the extensions are unique up to isomorphism. If $\chi_1 \chi_2^{-1}\neq \Eins, \omega^{\pm 1}$ then the assertion follows from Proposition \ref{genrc},  
Corollary \ref{genBanach} and \cite[(2.3.5)]{kisin2}. If $\chi_1=\chi_2 \omega^{\pm 1}$ then one may show the assertion using  Remarks  \ref{VseeQ} and \ref{image_stein} instead. If $\chi_1=\chi_2$ then one has to 
do some work to show that the $\nn$-adic completion of $\cV(\wP)$ is isomorphic to $V^{\mathrm u}_1\oplus V^{\mathrm u}_2$.  
We leave the details to the interested reader.
\end{remar}

\section{Unitary completions}\label{U}
We determine all the absolutely irreducible admissible unitary completions of absolutely irreducible locally algebraic $L$-representations 
of $G$ with $p\ge 5$. Such representations are of the form $\pi\otimes_L W_{l,k}$, where $\pi$ is a smooth absolutely irreducible $L$-representation
of $G$, that is a stabilizer of $v$ is an open subgroup of $G$ for all $v\in \pi$, and $W_{l,k}= \det^l \otimes \Sym^{k-1} L^2$, see 
\cite{prasad}. The study of such completions was initiated by Breuil \cite{breuil2}, \cite{breuilL} and our results confirm his 
philosophy, see \cite[\S 1.3]{breuilL}. We deduce the main result of this section, Theorem \ref{mainU}, by combining 
Theorem \ref{Lbij} with some deep results of Colmez. 

\begin{lem} If $\pi=\eta\circ \det$ is a character then $\pi\otimes_L W_{l,k}$ admits a unitary completion if and only if 
$k=1$ and $\val(\eta(p))=-l$.
\end{lem}
\begin{proof} This is well known, see for example \cite[Lem. 7.3]{comp}.
\end{proof} 

\begin{lem}\label{critprince} Let $\chi_1, \chi_2: \Qp^{\times}\rightarrow L^{\times}$ be smooth characters. If the representation 
$(\Indu{P}{G}{\chi_1\otimes \chi_2|\centerdot|^{-1}})_{\mathrm{sm}}\otimes W_{l, k}$ admits a unitary completion then 
\begin{itemize}
\item[(i)] $-(k+l)\le \val(\chi_1(p)), \val(\chi_2(p))\le -l$ and 
\item[(ii)] $\val(\chi_1(p))+\val(\chi_2(p))=-(k+2l)$.
\end{itemize} 
\end{lem}
\begin{proof} See \cite[Lem 7.9]{comp}, \cite[Lem. 2.1]{emduke}.
\end{proof} 

\begin{thm} Suppose that $\pi\cong (\Indu{P}{G}{\chi_1\otimes \chi_2|\centerdot|^{-1}})_{\mathrm{sm}}$ satisfies the conditions
of Lemma \ref{critprince} then the universal unitary completion of $\pi\otimes W_{l,k}$ is an admissible absolutely irreducible
$L$-Banach space representation. Moreover, the universal completion is ordinary if and only if 
$\val(\chi_1(p))=-l$ or $\val(\chi_2(p))=-l$.
\end{thm}
\begin{proof} Since by assumption $\pi$ is irreducible, $\chi_1\chi_2^{-1}\neq |\centerdot|^{\pm 1}$ and so 
$$(\Indu{P}{G}{\chi_1\otimes \chi_2|\centerdot|^{-1}})_{\mathrm{sm}}\cong (\Indu{P}{G}{\chi_2\otimes \chi_1|\centerdot|^{-1}})_{\mathrm{sm}}.$$
We may assume that $\val(\chi_1(p))\le \val(\chi_2(p))$. Suppose that $\val(\chi_2(p))<-l$ then if $\chi_1\neq \chi_2$ 
the assertion is a deep result of Berger-Breuil \cite[5.3.4]{bergerbreuil}, if $\chi_1=\chi_2$ then the assertion follows from 
\cite{except}. If $\val(\chi_2(p))=-l$ then it follows from \cite[2.2.1]{breuilem} that the universal unitary completion 
is isomorphic to $(\Indu{P}{G}{\psi})_{cont}$, where $\psi(\left (\begin{smallmatrix} a & b \\0 & d \end{smallmatrix}\right ))= 
\chi_2(a) a^l \chi_1(d)|d|^{-1} d^{k+l-1}$. 
\end{proof} 

\begin{lem}\label{algcoinv} Let $\psi: P\rightarrow L^{\times}$ be a continuous character and let $P_0$ be a compact open subgroup of $P$. 
Then $\Hom_{P_0}(W_{l,k}, \psi)$ is at most $1$-dimensional and is non-zero if and only if   
$\psi(\left(\begin{smallmatrix} a & b \\0 & d \end{smallmatrix}\right))= a^l d^{k+l-1}$ for all 
$\left(\begin{smallmatrix} a & b \\0 & d \end{smallmatrix}\right)\in P_0.$
\end{lem}
\begin{proof} The restriction of $\psi$ to $U$ is trivial, since $U$ is contained in the derived subgroup of $P$.
We identify $W_{l,k}$ with the space of homogeneous polynomials in $x$ and $y$ of degree $k-1$ with $G$-action given by 
$\left(\begin{smallmatrix} a & b \\ c & d\end{smallmatrix}\right)\centerdot P(x,y)= (ad-bc)^l P(ax+cy, bx+dy)$.  
The space of $U\cap P_0$-co\-in\-va\-riants of $W_{l,k}$ is $1$-dimensional, spanned by the image of $y^{k-1}$. Since 
$ \left (\begin{smallmatrix} a & 0 \\0 & d \end{smallmatrix}\right) y^{k-1}= d^{k-1} (ad)^l y^{k-1}$ we obtain the assertion. 
\end{proof}

\begin{lem}\label{algordprince} Let $\psi: P\rightarrow L^{\times}$ be a continuous unitary character and 
let $\Pi:=(\Indu{P}{G}{\psi})_{cont}$. 
If $\Pi^{\mathrm{alg}}\neq 0$ then $\Pi^{\mathrm{alg}}\cong (\Indu{P}{G}{\chi_1\otimes \chi_2})_{\mathrm{sm}}\otimes W_{l,k}$
and $\psi(\left (\begin{smallmatrix} a & b \\0 & d \end{smallmatrix}\right ))= 
\chi_1(a) a^l \chi_2(d) d^{k+l-1}$, for some smooth characters $\chi_1, \chi_2: \Qp^{\times}\rightarrow L^{\times}$ and 
integers $k$, $l$ with $k\ge 0$.
\end{lem}
\begin{proof} Let $\tau$ be a smooth $L$-representation of $G$ and $W=W_{l,k}$ for some integers $k, l$, then 
$$ \Hom_G(\tau\otimes W, \Pi)\cong \Hom_G(\tau, \Hom_L(W, \Pi))\cong  \Hom_G(\tau, \Hom_L(W, \Pi)^{\mathrm{sm}}),$$
where $\Hom_L(W, \Pi)^{\mathrm{sm}}$ denotes smooth vectors for the action of $G$  on  
$\Hom_L(W, \Pi)$ by conjugation; explicitly it is the union of $\Hom_H(W_{l,k}, \Pi)$ for all compact open subgroups $H$ of $G$.
If  $\Hom_G(\tau\otimes W, \Pi)\neq 0$ there exists a compact open subgroup $H_0$ of $G$ such that $\Hom_{H_0}(W, \Pi)\neq 0$. 
Frobenius reciprocity and Lemma \ref{algcoinv} imply  that 
$\psi(\left (\begin{smallmatrix} a & b \\0 & d \end{smallmatrix}\right ))= a^l d^{k+l-1}$ for all 
$\left(\begin{smallmatrix} a & b \\0 & d \end{smallmatrix}\right)\in H_0\cap P.$ Hence, $\psi=\psi_{\mathrm{sm}} \psi_{\mathrm{alg}}$, 
where $\psi_{\mathrm{sm}}: P\rightarrow L^{\times}$ is a smooth character, trivial on $H_0\cap P$ and 
$\psi_{\mathrm{alg}} (\left (\begin{smallmatrix} a & b \\0 & d \end{smallmatrix}\right ))= a^l d^{k+l-1}$ for all
$\left (\begin{smallmatrix} a & b \\0 & d \end{smallmatrix}\right )\in P$. Lemma \ref{algcoinv} implies that if $(l', k')\neq (l,k)$ then 
$\Hom_{G}(\tau\otimes W_{l',k'}, \Pi)=0$ for all smooth representations $\tau$. It follows from \cite{prasad} that 
$\Pi^{\mathrm{alg}}\cong  \Hom_L(W, \Pi)^{\mathrm{sm}}\otimes W$. We identify $W$ with the space homogeneous polynomials 
in $x$ and $y$ of degree $k-1$. The map $f\otimes P\mapsto [g\mapsto f(g) P(c,d)]$, for all 
$g=\left (\begin{smallmatrix} a & b \\c & d \end{smallmatrix}\right)\in G$ induces an injection 
$(\Indu{P}{G}{\psi_{\mathrm{sm}}})_{\mathrm{sm}} \otimes W \hookrightarrow \Pi$, and hence an injection  
 $(\Indu{P}{G}{\psi_{\mathrm{sm}}})_{\mathrm{sm}}\hookrightarrow \Hom_L(W, \Pi)^{\mathrm{sm}}$. It follows from Lemma \ref{algcoinv} 
and Frobenius reciprocity that for all open subgroups $H\subseteq H_0$ the space of $H$-invariants in the source and the 
target have the same dimension equal to $| H\backslash G/P|$. Hence, the injection is an isomorphism. 
\end{proof}

\begin{lem}\label{algordsp} Let $\eta: \Qp^{\times}\rightarrow L^{\times}$ be a continuous unitary character. If 
$(\widehat{\Sp}\otimes \eta \circ \det )^{\mathrm{alg}}\neq 0$ then $\eta$ is locally algebraic and 
$(\widehat{\Sp}\otimes \eta \circ \det )^{\mathrm{alg}}\cong \Sp\otimes \eta\circ \det$.
\end{lem}
\begin{proof} Since the surjection $q: (\Indu{P}{G}{\eta\otimes\eta})_{cont}\twoheadrightarrow \widehat{\Sp}\otimes \eta \circ \det $
admits a $P$-equivariant splitting, \cite[VI.2.3]{colmez} implies that $q$ induces a surjection on 
locally algebraic vectors. The assertion follows from Lemma \ref{algordprince}.
\end{proof} 

\begin{thm}\label{mainU} Suppose that the central character of $\pi\otimes W_{l,k}$ is unitary and
either $\pi$ is special series and $k>1$ or $\pi$ is supercuspidal. Then $\pi\otimes W_{l, k}$ admits precisely 
$\mathbb{P}^1(L)$ non-isomorphic absolutely irreducible admissible unitary completions. 
\end{thm}
\begin{proof} Let $\Pi$ be an absolutely irreducible admissible unitary $L$-Banach space representation of $G$ containing 
 $\pi\otimes W_{l,k}$ as a $G$-invariant dense subspace. Since $\pi\otimes W_{l,k}$ is  dense in $\Pi$, the central character 
of $\Pi$ is equal to the central character of $\pi\otimes W_{l,k}$. It follows from Lemmas \ref{algordprince} and \ref{algordsp} that $\Pi$ 
is not ordinary. Hence, $V:=\VV(\Pi)$ is an absolutely irreducible $2$-dimensional $L$-representation of $\gal$ by Theorem \ref{Lbij}.
Since $\Pi$ contains a locally algebraic representation $\pi\otimes W_{l,k}$, $V$ is de Rham \cite[VI.6.13]{colmez}, with 
Hodge-Tate weights $a< b$, \cite[VI.5.1]{colmez}, where $b-a=k$ (the precise formula for $a$ and $b$ depends on the normalization of the 
correspondence). Since $V$ is de Rham, it is potentially semistable and to $V$ one may associate a $2$-dimensional Weil-Deligne representation 
$\WD(V)$, see for example \cite{ghm}. Colmez has shown  that $\Pi^{\mathrm{alg}}\cong \LL(\WD(V))\otimes W_{l,k}$, \cite[Thm.\ 0.21]{colmez}, 
where $\LL$ denotes the classical (modified) local Langlands correspondence \cite[\S VI.6.11]{colmez}. In the supercuspidal case the proof 
was conditional on the results of Emerton, which have now appeared in \cite[\S7.4]{emlg}.  Thus determining all the 
isomorphism classes of  the absolutely irreducible admissible unitary completions of $\pi\otimes W$ is equivalent to determining 
all the isomorphism classes of the absolutely irreducible $2$-dimensional potentially semistable $L$-representations $V$ of $\gal$ with Hodge-Tate weights 
$a<b$, such that $\Hom_G(\pi, \LL(\WD(V)))\neq 0$. If $\pi$ is special series then (after twisting by a smooth unitary character) 
it follows from \cite[VI.6.50]{colmez} that the set of such $V$ consists of a family of semi-stable non-crystalline representations 
indexed by the $L$-invariant  $\mathcal L\in L$ and one crystalline representation.  If $\pi$ is supercuspidal then 
the last condition is equivalent to $\LL(\WD(V))\cong \pi$ and the assertion follows from \cite{ghm}.
\end{proof}

\appendix
\section{Two dimensional pseudocharacters}\label{pseudocharacters}
We recall some standard facts about $2$-di\-men\-sio\-nal pseudocharacters. We refer the reader to \cite[\S 1]{BC} for more information. 
Let $\GG$ be a profinite group and $(A, \mm)$ a local artinian $\OO$-algebra. We assume that $p>2$. A $2$-di\-men\-sio\-nal 
$A$-valued pseudocharacter is a continuous function $T: \GG \rightarrow A$ satisfying: 1) $T(1)=2$; 2) 
$T(gh)=T(hg)$ for all $g,h \in \GG$;  3) the relation 
\begin{displaymath}
T(g)T(h)T(k)-T(g) T(hk)- T(h) T(gk)- T(k) T(gh)+T(ghk)+T(gk h)=0
\end{displaymath}
for all $g,h, k \in \GG$. One may show that if  $\rho: \GG \rightarrow \GL_2(A)$ is a continuous representation then $\tr \rho$ 
is a $2$-di\-men\-sio\-nal pseudocharacter. Given a $2$-di\-men\-sio\-nal pseudocharacter $T: \GG\rightarrow A$ one may show, \cite[Prop.1.29]{che}, that 
the function $D(g):=\frac{T(g)^2-T(g^2)}{2}$ defines a continuous group homomorphism $\GG \rightarrow A^{\times}$. It is shown 
in \cite[1.9, 1.29]{che} that $T\mapsto (T, D)$ induces a bijection between $2$-di\-men\-sio\-nal pseudocharacters and pairs of functions 
$(T, D)$, where $D:\GG \rightarrow A^{\times}$ is a continuous group homomorphism and $T:\GG\rightarrow A$ is a continuous function satisfying:
 $T(1)=2$, $T(gh)=T(hg)$, $D(g)T(g^{-1}h)-T(g)T(h)+T(gh)=0$ for all $g, h\in \GG$.  

Let $\rho:\GG\rightarrow \GL_2(k)$ be a continuous representation and let $D^{\mathrm{ps}}$ be the functor 
from local artinian augmented $\OO$-algebras with residue field $k$ to the category of sets, 
such that $D^{\mathrm{ps}}(A)$ is the set of all $2$-di\-men\-sio\-nal $A$-valued pseudocharacters $T$, such that 
$T\equiv \tr \rho \pmod{\mm_A}$. If for every open subgroup $\Hr$ of $\GG$, $\Hom^{cont}(\Hr, \Fp)$ is 
a finite di\-men\-sio\-nal $\Fp$-vector space then the functor $D^{\mathrm{ps}}$ is pro-represented by a complete 
local noetherian $\OO$-algebra. We note that this finiteness condition is satisfied if $\GG$ is the absolute  Galois group 
of a local field. We usually work with a variant: fix a continuous character $\psi: \GG\rightarrow \OO^{\times}$ lifting 
$\det \rho$ and let $D^{\mathrm{ps}, \psi}$ be a subfunctor of $D^{\mathrm{ps}}$ such that 
$T\in D^{\mathrm{ps}, \psi}(A)$ if and only if $\frac{T(g)^2-T(g^2)}{2}$ is equal to (the image of ) $\psi(g)$ for 
all $g\in \GG$. We will refer to $D^{\mathrm{ps}, \psi}$ as a deformation problem with a fixed determinant. 
One may show that if $D^{\mathrm{ps}}$ is pro-represented by $R$ then $D^{\mathrm{ps}, \psi}$ is pro-represented by a quotient of $R$.

\begin{lem}\label{pst1} Let $\mathrm G$ be a finite group, let $S= k[\mathrm G]/J$, where $J$ is the 
two sided ideal in $k[\mathrm G]$ generated by $g^2-2g +1$ for all $g\in \mathrm G$. Then the image 
of $\mathrm G$ in $S^{\times}$ is a $p$-group.
\end{lem}
\begin{proof} Suppose not then there exists a prime $l\neq p$ and $g\in \mathrm G$ such that 
the image of $g$ in $S^{\times}$ has order $l$. Since the greatest common divisor of 
$x^l-1$ and $(x-1)^2$ in $k[x]$ equal to $x-1$ we may find polynomials $a(x), b(x)\in k[x]$ such that 
$(x^l-1)a(x) +(x-1)^2b(x)= x-1$. Since the images of $g^l-1$ and $g^2-2 g+1$ are equal to $0$ in $S$, 
we deduce that the image of $g$ in $S$ is trivial.
\end{proof}

Let $\rho: \GG\rightarrow \GL_2(k)$ be a continuous representation, $\Kr$ be the kernel of $\rho$, $\Kr(p)$
the maximal pro-$p$ quotient of $\Kr$ and $\Hr$ the kernel of $\Kr\twoheadrightarrow \Kr(p)$. We note that 
$\Hr$ is a normal subgroup of $\GG$. 

Let $(A, \mm)$ be a local  artinian $\OO$-algebra with residue field $k$. Let $T: \GG \rightarrow A$ be a continuous $2$-di\-men\-sio\-nal 
pseudocharacter lifting $\tr \rho$. Since $A$ is finite and $T$ is continuous $\Ker T:=\{h \in \GG: T(gh)=T(g), \forall g\in \GG\}$
is an open subgroup of $\GG$.   

\begin{prop}\label{pst2} $\HH \subseteq \Ker T$.
\end{prop} 
\begin{proof} Choose an open normal subgroup $\Nr$ of $\GG$ contained in $\Kr\cap \Ker T$. Let $\mathrm G:=\GG/\Nr$ and let 
let $J$ be the two sided ideal in $A[\mathrm G]$ generated by elements $g^2-T(g)g +\frac{T(g)^2-T(g^2)}{2}$, for all 
$g\in \mathrm G$ and let $S:=A[\mathrm G]/J$. We claim that the image of $\Kr$ in $S^{\times}$ is a $p$-group. Since the 
kernel of $S^{\times}\rightarrow (S/\mm S)^{\times}$ is a $p$-group, it is enough to show that the image 
of $\Kr$ in $(S/\mm S)^{\times}$ is a $p$-group. Since $S/\mm S$ is a quotient of $k[\mathrm G]/(g^2-\tr \rho(g) g +\det \rho(g): g\in \mathrm G)$ 
the claim follows from Lemma \ref{pst1}. For each $g\in \GG$ we denote the image of $g$ in $S$ by $\bar{g}$. It follows from the claim that 
$\bar{h}=1$ for all $h\in \Hr$. We may extend $T: \GG\rightarrow A$ linearly to $T: A[\GG]\rightarrow A$, which factors through 
$T: A[\mathrm G]\rightarrow A$ as $\Nr \subseteq \Ker T$, and then factors through $T:S\rightarrow A$ and so we have 
$T(\bar{g})= T(g)$ for all $g\in \GG$. In particular, if $h\in \Hr$ then $T(gh)= T(\bar{g}\bar{h})=T(\bar{g})=T(g)$ for all $g\in \GG$.
\end{proof}

\begin{cor}\label{pst3} The inclusion $D^{\mathrm{ps}}_{\GG/\HH} \subseteq D^{\mathrm{ps}}_{\GG}$ is an isomorphism of functors. 
\end{cor} 
\begin{proof} It follows from Proposition \ref{pst2} that for all artinian local $\OO$-algebras $(A, \mm)$ with residue field $k$ we have 
$D^{\mathrm{ps}}_{\GG/\HH}(A)= D^{\mathrm{ps}}_{\GG}(A)$.
\end{proof}

Suppose that $D^{\mathrm{ps}}_{\GG/\HH}$ is pro-represented by a complete local noetherian $\OO$-algebra $(R, \mm)$, then 
$D^{\mathrm{ps}}_{\GG}$ is also pro-represented by $(R, \mm)$ by Corollary \ref{pst3}. Let $T: \GG \rightarrow \GG/\HH \rightarrow R$ 
be the universal pseudocharacter lifting $\tr \rho$. Let $J$ (resp. $J'$) be a closed two-sided ideal 
in $R[[\GG]]$ (resp. $R[[\GG/\HH]]$) generated by the elements $g^2-T(g) g +\frac{T(g)^2-T(g^2)}{2}$ for all 
$g\in \GG$ (resp. $g\in \GG/\HH$).

\begin{cor}\label{pst4} The natural map $R[[\GG]]/J \rightarrow R[[\GG/\HH]]/J'$ is an isomorphism.
\end{cor}
\begin{proof} If $\NN$ is an open normal subgroup of $\GG$ and $n\ge 1$ let $\mathfrak a(\Nr,n)$ 
be the kernel of $R[[\GG]]\twoheadrightarrow R/\mm^n[\GG/\Nr]$. The ideals $\mathfrak a(\Nr,n)$
for all open normal subgroups $\Nr$ and all $n\ge 1$ form a system of open neighbourhoods of $0$ in $R[[\GG]]$. 
It follows from the proof of Proposition \ref{pst2} that for each $n\ge 1$ we may choose
an open normal subgroup $\Nr_n$ of $\GG$ such that for all open normal subgroups 
$\Nr$ of $\GG$ contained in $\Nr_n$ the image of $\HH$ in $R[[\GG]]/(J +\mathfrak a(\NN, n))$ is trivial. 
Thus $R[[\GG]]/(J+ \mm^n R[[\GG]])\cong R[[\GG/\HH]]/(J'+\mm^nR[[\GG/\HH]])$ for all $n\ge 1$, which yields the claim.
\end{proof}

\section{Some deformation rings}\label{someDef}

Let $\omega: \gal\rightarrow \Fp^{\times}\hookrightarrow k^{\times}$ be the cyclotomic character 
modulo $p$. It follows from  local Tate duality 
and Euler characteristic that $\Ext^1_{\Fp[\gal]}(\omega, \Eins)$ is one dimensional. Let 
$0\rightarrow \Eins\rightarrow \rho\rightarrow \omega \rightarrow 0$ be  a non-split extension. This determines 
$\rho$ up to isomorphism. The purpose of this appendix is to describe explicitly various deformation rings of $\rho$, by 
spelling  out what a general result of B\"ockle in \cite{bockle}, says in this particular case. We then show using results of Bella\"{i}che 
\cite{joel} that the universal deformation ring of $\rho$ is isomorphic to the universal deformation ring of $\tr \rho$. In \S \ref{easier} 
we consider the easier, generic reducible case. We assume $p\ge 5$ until \S \ref{easier}, and $p\ge 3$ in \S \ref{easier}.

We may think of $\rho$ as a group homomorphism $\rho: \gal\rightarrow \GL_2(\Fp)\hookrightarrow \GL_2(k)$, 
$g\mapsto \left(\begin{smallmatrix} 1 &\kappa(g)\\ 0 & \omega(g)\end{smallmatrix}\right)$. Let $H$ be the image 
of $\gal$ in $\GL_2(k)$ and let $U$ be the $p$-Sylow subgroup of $H$. Since $\rho$ is non-split 
$U$ is non-trivial, hence $U\cong \Fp$, let $G$ be the subgroup of diagonal matrices in $H$, then $G\cong \Fp^{\times}$
and $H\cong U\rtimes G$. Let $L$ be the fixed field of $\Ker \rho$ and let $F= L^U$. Then 
$F$ is the fixed field of $\Ker \omega$ and hence is equal to $\Qp(\mu_p)$, where $\mu_p$ is the group 
of $p$-th roots of unity. We identify $\Gal(F/\Qp)$ with $G$. Let $\GG_F$ be the absolute Galois group of $F$. 

If $\rho_A: \gal\rightarrow \GL_2(A)$ is a deformation of $\rho$ to $(A,\mm)$ then $\rho_A(\GG_F)$ is 
contained in $\left(\begin{smallmatrix} 1+\mm & A\\ \mm & 1+\mm\end{smallmatrix}\right)$, and hence 
$\rho_A$ factors through $\Gal(F(p)/\Qp)$, where $F(p)$ is the compositum of all finite extensions 
of $p$ power order of $F$ in $\Qpbar$. Now $\Gal(F(p)/F)\cong \GG_F(p)$ the maximal pro-$p$ quotient of 
$\GG_F$. Since the order of $G=\Gal(F/\Qp)$ is prime to $p$, we may choose a splitting of exact sequence 
$1\rightarrow \GG_F(p)\rightarrow \Gal(F(p)/\Qp)\rightarrow G\rightarrow 1$, so that $\Gal(F(p)/\Qp)\cong \GG_F(p)\rtimes G$.

We will recall some  facts about Demu\v{s}kin groups,  see for example \cite[\S III.9]{nsw} for details.  
A finitely generated pro-$p$ group $P$ is a Demu\v{s}kin group, if $H^2(P, \Fp)$ is one dimensional and the cup product 
$H^1(P, \Fp)\times H^1(P, \Fp)\overset{\cup}{\rightarrow} H^2(P, \Fp)$ is a non-degenerate bilinear form. If $p>2$ a Demu\v{s}kin group 
$P$ is uniquely determined up to isomorphism by two parameters $n=n(P)$ the dimension of $H^1(P, \Fp)$ and $q=q(P)$ the number 
of elements in the torsion subgroup of $P^{ab}$, and is isomorphic to a pro-$p$ group generated by $n$ elements $x_1, \ldots, x_n$
and one  relation $x_1^q(x_1, x_2)(x_3, x_4)\ldots(x_{n-1}, x_n)$, where $(x,y)=x^{-1} y^{-1} xy$. 
We note that since $p>2$ the non-degeneracy of bilinear form implies  that $n$ is even and it follows from the presentation 
of $P$ that $P^{ab}\cong \Zp^{n-1}\oplus \ZZ/q\ZZ$. It is well known, see for example \cite[7.5.8]{nsw}, that 
if $F$ is  a finite extension of $\Qp$ containing $\mu_p$, then $\GG_F(p)$ is a Demu\v{s}kin group with $n=[F:\Qp]+2$
and $q$ equal to the number of $p$ power order roots of unity in $F$. In our situation $F=\Qp(\mu_p)$ and 
so $n=p+1$ and $q=p$. 

Following \cite{bockle} we are going to construct a universal deformation of $\rho$ using the presentation of $\GG_F(p)$. For 
a $p$-group $P$ we define a filtration $P_1=P$, $P_{i+1}=P^p_i (P_i,  P)$, where $(P_i, P)$  denotes a closed subgroup generated
by the commutators, and let $\gr_i P:= P_i/P_{i+1}$. We let $\FF$ be a free pro-$p$ group on $p+1$ generators, and we choose 
a surjection $\varphi: \FF\twoheadrightarrow \GG_F(p)$. Since $\GG_F(p)$ is a Demu\v{s}kin group there exists 
an element $r\in \FF$ such that $\Ker \varphi$ is the smallest normal closed subgroup of $\FF$ containing $r$.
Since the order of $G$ is prime to $p$, we may let $G$ act on $\FF$ so that $\varphi$ is $G$-equivariant, see 
Lemma 3.1 in \cite{bockle}. 
We denote by  $\tilde{\omega}: G\rightarrow \Zp^{\times}$ the Teichm\"uller lift of $\omega$.

\begin{lem}\label{generatorsA} We may choose generators $x_0, \ldots, x_p$ of $\FF$ so that
\begin{itemize}
\item[(i)] $g x_i g^{-1}= x_i^{\tilde{\omega}(g)^{i}}$, for $g\in G$  and $0\le i\le p$;
\item[(ii)] the image of $r$ in $\gr_2 \FF$ is equal to the image of 
$$r':= x_1^p (x_1, x_{p-1}) (x_2, x_{p-2}) \ldots (x_{\frac{p-1}{2}}, x_{\frac{p+1}{2}})(x_p, x_0).$$
\end{itemize}
\end{lem}
\begin{proof} The assertion  follows from \cite[Prop. 3]{lab}, where the cup product is described in terms of 
the image of $r$ in $\gr_2 \FF$. We know that 
$$\gr_1 \FF\cong \gr_1 \GG_F(p)\cong \Fp \oplus \mu_p \oplus \Fp[G]$$
as a representation of $G$, see Theorems 4.1 and 4.2 in \cite{bockle}. Moreover, the summand $\mu_p$ is the image 
of the torsion subgroup of $\GG_F(p)^{ab}$ under the natural map $\GG_F(p)^{ab}\twoheadrightarrow \gr_1 \GG_F(p)$. We 
fix $\xi_1\in \mu_p\subset \gr_1 \GG_F(p)$, which generates $\mu_p$ as $\Fp[G]$-module. Now 
$H^1(\GG_F(p), \Fp)\cong \Hom^{cont}(\GG_F(p), \Fp)\cong (\gr_1 \GG_F(p))^*$ as a $G$-representation,  Hence,  we may find an $\Fp$-basis $\chi_0, \ldots, \chi_p$ of $H^1(\GG_F(p), \Fp)$ such that $G$ acts on $\chi_i$ by $\omega^{-i}$, $\chi_1(\xi_1)\neq 0$, $\chi_p(\xi_1)=0$
and, since the cup product defines a non-degenerate bilinear pairing  and 
$G$ acts on $H^2(\GG_F, \Fp)$ by $\omega^{-1}$, we have $\chi_i\cup \chi_j=0$ unless $i+j\equiv 1\pmod {p-1}$.
Further, by replacing $\chi_i$ by a scalar multiple $\lambda \chi_i$, with $\lambda\in \Fp^{\times}$, we may 
achieve that $\bar{r}(\chi_p\cup \chi_0)=1$ and $\bar{r}(\chi_i \cup \chi_{p-i})=1$ for $1\le i\le (p-1)/2$, 
where $\bar{r}: H^2(\GG_F(p), \Fp)\overset{\cong}{\rightarrow} \Fp$ is the  isomorphism defined in \cite[Prop 3]{lab}.
Let $\xi_0, \ldots, \xi_p$ be an $\Fp$-basis of $\gr_1\GG_F(p)$ dual to $\chi_0, \ldots, \chi_p$. Then $G$ acts on $\xi_i$
by the character $\omega^i$. Since the order of $G$ is prime to $p$, we may find $x_i\in \FF$ satisfying (i) and mapping 
to  $\xi_i$ in $\gr_1\FF$. Since the images of $x_0, \ldots, x_p$ form a basis of $\gr_1 \FF$, they generate $\FF$. Part (ii) 
follows by construction from the Proposition 3 in \cite{lab}.
\end{proof} 

Let $R$ be the ring 
\begin{equation}\label{presentR}
 R:=\frac{\OO[[a_0, a_1, c_0, c_1, d_0, d_1 ]]}{(pc_0 + c_0 d_1+ c_1 d_0)}
\end{equation}
Let $P$ be a pro-$p$ subgroup of $\GL_2(R)$ generated by the matrices $m_i$ for  $0\le i\le p$, where $m_i= 1$ if $i\not\equiv 0, \pm 1\pmod{p-1}$, 
and
$$m_{p-2}=\begin{pmatrix} 1 & 1 \\ 0 & 1\end{pmatrix}, \quad m_{1+(p-1)j}= \begin{pmatrix}1 & 0\\ c_j & 1\end{pmatrix}, $$
$$m_{(p-1)j}=\begin{pmatrix} (1+a_j)^{\frac{1}{2}} (1+d_j)^{\frac{1}{2}} & 0\\ 0 & (1+a_j)^{\frac{1}{2}} (1+d_j)^{-\frac{1}{2}}\end{pmatrix}, $$
for $j=0$ and $j=1$. We embed $G\hookrightarrow \GL_2(R)$, $g\mapsto \bigl (\begin{smallmatrix} 1 & 0 \\ 0 & \tilde{\omega}(g)\end{smallmatrix}\bigr )$.
 One 
has $g m_i g^{-1}=m_i^{\tilde{\omega}(g)^i}$ for all $0\le i\le p$ and hence $x_i\mapsto m_i$ defines a $G$-equivariant homomorphism
$\alpha: \FF\rightarrow P$ and hence a group homomorphism $\alpha: \FF\rtimes G \rightarrow \GL_2(R)$.

\begin{prop}\label{repunA} There exists a continuous group homomorphism 
$$\varphi':\FF\rtimes G\twoheadrightarrow \Gal(F(p)/\Qp)$$
 such that
$\varphi'(g)\equiv \varphi(g)\pmod{\GG_F(p)_3}$, and a commutative diagram:
\begin{displaymath} 
\xymatrix@1{ \FF\rtimes G \ar[r]^-{\alpha}\ar@{->>}[dr]_-{\varphi'} & \GL_2(R)\\ & \Gal(F(p)/\Qp)\ar[u]_-{\tilde{\rho} }.}
\end{displaymath} 
\end{prop}
\begin{proof} Let us observe that for $j=0$ and $j=1$ the commutator 
$$(m_{1+(p-1)j}, m_{(p-1)(1-j)})= \begin{pmatrix} 1 & 0 \\ c_j d_{1-j} & 1 \end{pmatrix}$$
and $(m_i, m_{p-i})=1$ if $i\not \equiv 1, 0 \pmod{p-1}$ hence 
\begin{equation}\label{demsh}
 m_1^p (m_1, m_{p-1}) (m_2, m_{p-2}) \ldots (m_{\frac{p-1}{2}}, m_{\frac{p+1}{2}})(m_p, m_0)=1
\end{equation}
as $p c_0+ c_0 d_1+ c_1 d_0=0$ in $R$. Since $\alpha(x_i)= m_i$,  we get that $\alpha(r')=1$, where 
$r'$ is defined in  Lemma \ref{generatorsA}. Since $r\equiv r' \pmod{\FF_3}$ we deduce that 
$\alpha(r)\in \alpha(\FF_3)$ and the assertion follows from Proposition 3.8 in \cite{bockle}. Namely, it is
shown there that there exists an element $r_1\in \Ker \alpha \cap \FF^p(\FF, \FF)$, such that 
$r_1\equiv r\pmod{\FF_3}$, and $G$ acts on $r_1$ by a character. It follows Lemma \ref{generatorsA} (ii)
that the character is equal to $\tilde{\omega}$. Let $\mathcal R$ be the smallest closed normal subgroup of 
$\FF$ containing $r_1$ and set $D:=\FF/\mathcal R$. Since $\alpha(r_1)=1$ and $G$ acts on $r_1$ by a character, 
we deduce that $\alpha$ factors through $\alpha: D\rtimes G\rightarrow \GL_2(R)$.  

We claim  that $D\rtimes G\cong  \Gal(F(p)/\Qp)$. Since $r_1\equiv r \equiv r' \pmod{\FF_3}$, $D$ is a Demu\v{s}kin 
group with $n(D)=p+1$ and $q(D)=p$, see \cite[3.9.17]{nsw}. Hence, we know that $D\cong \Gal(F(p)/F)$. To see that we may 
choose this isomorphism $G$-equivariantly we observe that since $r\equiv r_1\pmod{\FF_3}$ Proposition 3 in \cite{lab}
implies that the diagram: 
\begin{displaymath}
\xymatrix@1{ H^1(D, \Fp)\times H^1(D, \Fp)\ar[r]^-{\cup} & H^2(D, \Fp)\ar[d]^-{\bar{r}_1}_-{\cong}\\
  H^1(\FF, \Fp)\times H^1(\FF, \Fp)\ar[r]\ar[u]_-{\cong}\ar[d]_-{\cong} & \Fp \\
H^1(\GG_F(p), \Fp)\times H^1(\GG_F(p), \Fp)\ar[r]^-{\cup}& H^2(\GG_F(p), \Fp)\ar[u]_-{\bar{r}}^-{\cong} }
\end{displaymath}
commutes and is $G$-equivariant. The claim follows from Theorem 3.4 in \cite{bockle}.
\end{proof}

\begin{thm}[\cite{bockle}]\label{bock} $R$ is the universal deformation ring of $\rho$ and the equivalence class 
of $\tilde{\rho}$, defined in Proposition \ref{repunA}, is the universal deformation.
\end{thm} 
\begin{proof} We note that since $\omega\neq \Eins$, $\End_{k[\gal]}(\rho)=k$ and hence the 
deformation functor $\Def_{\rho}$ is representable. Moreover, local Tate duality implies that
$$H^2(\gal, \Ad \rho)\cong H^0(\gal, \Hom(\Ad \rho, \omega))\cong \Hom_{\gal}(\rho, \rho\otimes \omega)$$
is $1$-di\-men\-sio\-nal and hence $H^1(\gal, \Ad \rho)$ is $6$-di\-men\-sio\-nal by local Euler-Poincar\'e characteristic. 
We have a natural transformation of functors $\eta: h_R\rightarrow \Def_{\rho}$, which  maps a homomorphism 
$\psi: R\rightarrow A$  to the equivalence class of the representation 
$\rho_A: \Gal(F(p)/\Qp)\overset{\tilde{\rho}}{\rightarrow} \GL_2(R)\overset{\psi}{\rightarrow} \GL_2(A)$. 
Moreover, one may check directly that  this induces an isomorphism $h_R(k[\epsilon])\cong \Def_{\rho}(k[\epsilon])$.
Hence, we obtain a surjection $R_{\rho}\twoheadrightarrow R$, where $R_{\rho}$ is the ring representing $\Def_{\rho}$.
It is shown in Theorem 6.2 of \cite{bockle} that this map is an isomorphism.
\end{proof}

\begin{cor}\label{reduciblerhoA} Let $x\in \Spec R[1/p]$ be a maximal ideal with residue field $E$. The corresponding  representation 
$\rho_x: \gal\rightarrow \GL_2(E)$ is reducible if and only if $c_0$ and  $c_1$ are $0$ in $E$.
\end{cor}
\begin{proof} Let $\mathfrak a$ be the ideal of $R$ generated by $c_0$ and $c_1$. It follows from the construction 
of the universal deformation $\tilde{\rho}$ that the image of  $\gal\overset{\tilde{\rho}}{\rightarrow} \GL_2(R)\rightarrow 
\GL_2(R/\mathfrak a)$ is contained in the  subgroup of upper triangular matrices. Hence, if the image of $c_0$ and $c_1$ 
in $E$ is zero then $\rho_x$ is reducible. Conversely, suppose that $\rho_x$ is reducible then for all 
$g,h\in \gal$ the matrix $\rho_x(g)\rho_x(h)-\rho_x(h)\rho_x(g)$ is nilpotent. In particular,
for $j=0$ and $j=1$
the matrix $\rho_x(\varphi'(x_{1+(p-1)j})\rho_x(\varphi'( x_{p-2}))- \rho_x(\varphi'( x_{p-2}))\rho_x(\varphi'(x_{1+(p-1)j}))$
is nilpotent. Since it is equal to 
$$\begin{pmatrix} 1 & 0 \\ \bar{c}_j & 1 \end{pmatrix} \begin{pmatrix}1 & 1\\ 0 & 1\end{pmatrix}- 
\begin{pmatrix}1 & 1\\ 0 & 1\end{pmatrix} \begin{pmatrix} 1 & 0 \\ \bar{c}_j & 1 \end{pmatrix}= \begin{pmatrix} -\bar{c}_j & 0 \\ 0 & \bar{c}_j
\end{pmatrix}$$
we deduce that $\bar{c}_j$ the image of $c_j$ in $E$ is zero.
\end{proof}   

Let $\psi: \gal\rightarrow \OO^{\times}$ be a continuous character, lifting $\omega$ and let $\Def^{\psi}_{\rho}$ be 
subfunctor of $\Def_{\rho}$ parameterizing the deformations with determinant equal to $\psi$.
\begin{cor}\label{RpsirhoA} The functor $\Def^{\psi}_{\rho}$ is  represented by 
$$R^{\psi}\cong \frac{\OO[[c_0, c_1, d_0, d_1]]}{(pc_0+ c_0d_1+ c_1d_0)}.$$
\end{cor}
\begin{proof} Let $\lambda_0, \lambda_1\in \varpi\OO$ such that 
$\psi(\varphi'(x_{j(p-1)}))=1+ \lambda_j$, for $j=0$ and $j=1$.
By construction we have $\det m_i=1$, if $i\not\equiv 0 \pmod{p-1}$, and $\det m_{j(p-1)}= 1+a_j$. We deduce that 
$\Def_{\rho}^{\psi}$ is represented by $R/(a_0-\lambda_0, a_1-\lambda_1)$, which implies the claim.
\end{proof}

\begin{cor}\label{A6} Let $\mathfrak r= R^{\psi}\cap  \bigcap_x \mm_x$ where the intersection is taken over all maximal ideals of $R^{\psi}[1/p]$ 
such that $\rho_x$ is reducible. Then $\mathfrak r= (c_0, c_1)$. In particular, 
\begin{itemize}
\item[(i)] $R^{\psi}/\mathfrak r\cong \OO[[d_0, d_1]]$;
\item[(ii)] let $\nn$ be the maximal ideal of $R^{\psi}_k/\rr_k$, then for all $i\ge 0$ there 
exists a surjection of $R^{\psi}$-modules: $\rr^i_k/\rr^{i+1}_k\twoheadrightarrow \nn^i$.
\end{itemize}
\end{cor}
\begin{proof} Corollary  \ref{reduciblerhoA} implies that $(c_0, c_1)$ is contained in $\mathfrak r$
and the image of $\mathfrak r$ in $R^{\psi}/(c_0, c_1)$ is equal to the intersection of all the maximal ideals of $R^{\psi}/(c_0, c_1)[1/p]$.
Since  $R^{\psi}/(c_0, c_1)\cong \OO[[d_0, d_1]]$  by Corollary \ref{RpsirhoA}, we deduce that $\mathfrak r=(c_0, c_1)$. 
Now $R^{\psi}_{k}\cong k[[c_0, c_1, d_0, d_1]]/(c_0d_1+d_0c_1)$. Let $S=k[[c_0, c_1, d_0, d_1]]$ and we denote by $\mathfrak b$ 
the ideal of $S$ generated by $c_0, c_1$. Then $\gr^{\bullet}_{\mathfrak b} S$ is isomorphic to a polynomial ring in 
two variables over $k[[d_0, d_1]]$. The element $t= c_0 d_1+ d_0 c_1$ is pure of grade $1$. Since $R^{\psi}_{k}\cong S/tS$, 
and $\mathfrak r_{k}$ is the image of $\mathfrak b$, we have 
an exact sequence 
$0\rightarrow \gr^{i-1}_{\mathfrak b} S \rightarrow \gr^{i}_{\mathfrak b} S\rightarrow \rr^i_k/\rr^{i+1}_k\rightarrow 0$ for all $i\ge 1$, 
where the first non-trivial arrow is given by multiplication by $t$. Now $\gr^{i}_{\mathfrak b} S$ is a free $k[[d_0, d_1]]$-module 
with monomials in $c_0$ and $c_1$ of homogeneous degree $i$ as a basis. Sending $c_0\mapsto d_0$, $c_1\mapsto -d_1$ induces a 
surjection of $k[[d_0, d_1]]$-modules $\gr^{i}_{\mathfrak b} S\twoheadrightarrow \nn^i$. Since this map kills $t$ the surjection 
factors through $\rr^i_k/\rr^{i+1}_k\twoheadrightarrow \nn^i$. 
\end{proof}

\begin{remar} We will deduce in the course of the proof of  Proposition \ref{filisthesame} that the map in (ii) is an isomorphism. 
\end{remar}

\begin{remar}\label{enoughL} We note that in the definition of $\rr$ it is enough to consider the ideals with residue field $L$, since 
it follows from Lemma \ref{fnotvanish} that such ideals are Zariski dense in $\OO[[d_0, d_1]][1/p]$. 
\end{remar}

\begin{lem}\label{moveL} Suppose that the representations $\rho_x$ and $\rho_y$ corresponding to maximal ideals $x$ and $y$ of
$R^{\psi}[1/p]$ with residue field $L$ are reducible and have a common subquotient then $x=y$.
\end{lem}
\begin{proof} Since the determinant is fixed we deduce that $\rho_x$ and $\rho_y$ have the same semisimplification
$\delta \oplus\delta^{-1}\psi$, where $\delta:\gal\rightarrow L^{\times}$ is a continuous character, lifting the 
trivial character $\Eins: \gal\rightarrow k^{\times}$.  If  $\rho_x$ is 
semisimple then the action of  $\gal$ on $\rho_x$ factors through $\gal^{ab}$, and hence the action of $\gal$ on any stable $\OO$-lattice
of $\rho_x$ factors through $\gal^{ab}$, and hence the same holds for the reduction of any stable $\OO$-lattice modulo $\varpi$. 
Since the action of $\gal$ on $\rho$ does not factor through $\gal^{ab}$ we deduce that both $\rho_x$ and $\rho_y$ are 
not semisimple. Since the reduction of $\delta^{2}\psi^{-1}$ modulo 
$\varpi$ is equal to $\omega^{-1}$ and 
$p\ge 5$, $\delta^{-2}\varepsilon^{-1}$ cannot be equal to the trivial or 
the cyclotomic character. This implies $\Ext^1_{\gal}(\varepsilon\delta^{-1}, \delta)$ is $1$-di\-men\-sio\-nal. Hence, 
$\rho_x \cong \rho_y$ and so $x=y$.
\end{proof}

\begin{cor}\label{A7} The intersection of all the maximal ideals of $R^{\psi}[1/p]$ such that $\rho_x$ is irreducible is zero.
\end{cor}
\begin{proof} Let $S=\OO[[c_0, c_1, d_0, d_1]]$ and $g=pc_0 + c_1d_0 +c_0d_1$ and $f\in S$, not divisible by $g$.
It is enough to construct  $\varphi: S\rightarrow \Qpbar$, such that $\varphi(f)\neq 0$, $\varphi(g)=0$ and $\varphi(c_0)\neq 0$, 
since the last condition implies that the representation associated to $\ker \varphi$ is irreducible via Corollary \ref{reduciblerhoA}.  
 
Substituting $c'_1:= c_1 -d_0$ we get $g= d_0^2+ c'_1 d_0 +c_0d_1 +pc_0$. Hence, we may write $f= q g + r$, 
where $r=d_0 f_1+ f_2$, with $f_1, f_2\in \OO[[c_0, c'_1, d_1]]$, see \cite[IV\S 9]{lang}. The polynomial $X^2+c'_1 X + c_0d_1 +pc_0$
is irreducible over $\OO[[c_0, c'_1, d_1]]$ and hence also over its quotient field. As $r\neq 0$ we deduce that 
$h:=f_2^2-c'_1 f_1 f_2 + (c_0d_1 +pc_0) f_1^2\neq 0$. We may choose $\varphi:\OO[[c_0, c'_1, d_1]]\rightarrow \Qpbar$ such that 
$\varphi(c_0)\neq 0$ and $\varphi(h)\neq 0$, see Lemma \ref{fnotvanish}.  We may extend it to $S$ so that $\varphi(g)=0$. If $\varphi(f)=0$
then $\varphi(f_1) \varphi(d_0) +\varphi(f_0)=0$, and since $\varphi(d_0)$ is a root of 
$X^2+\varphi(c'_1)X+ \varphi(c_0d_1 +pc_0)$, we would obtain that $\varphi(h)=0$.
\end{proof}

Let $D^{\mathrm{ps}}_{\tr \rho}$ be the deformation functor parameterizing all the $2$-di\-men\-sio\-nal pseudocharacters of $\gal$ 
lifting $\tr \rho$. We know that $D^{\mathrm{ps}}_{\tr \rho}$ is pro-represented by a complete local noetherian $\OO$-algebra
$(S, \mm_S)$. Trace induces a morphism of functors $\Def_{\rho} \rightarrow D^{\mathrm{ps}}_{\tr \rho}$ and hence 
homomorphism of local $\OO$-algebras $\theta: S\rightarrow R$. It follows from \cite[1.4.4]{kisinfm} that 
$\theta$ is surjective. We note that  this can also be deduced from Theorem \ref{bock}, since we have written down the universal 
deformation explicitly. We are going to show that $\theta$ is an isomorphism. 

\begin{lem}\label{trbij1} Trace induces a bijection $\Def_{\rho}(k[x]/(x^2))\overset{\cong}{\rightarrow} D^{\mathrm{ps}}_{\tr \rho}(k[x]/(x^2))$.
\end{lem}
\begin{proof} Since $\theta$ is surjective we already know that the map is an injection. Hence to show surjectivity 
it is enough to show that both spaces have the same dimension as $k$-vector spaces. It follows from \cite[Thm 2]{joel} that 
$\dim_k  D^{\mathrm{ps}}_{\tr \rho}(k[x]/(x^2))=6$, which is also the dimension of $\Def_{\rho}(k[x]/(x^2))$.
\end{proof}
Let $\mathrm{F}:=\OO[[a_0, a_1, c_0, c_1, d_0, d_1]]$ it follows from Lemma \ref{trbij1} that there exist  surjections 
$\mathrm{F}\overset{\kappa}{\twoheadrightarrow} S\overset{\theta}{\twoheadrightarrow} R$, which induce isomorphisms 
on the tangent spaces. We may assume that the composition $\kappa\circ \theta$ is the one used to present $R$ in \eqref{presentR}.
Given a local $\OO$-algebra $(\mathrm R, \mm)$ we denote $\overline{\mathrm R}:=\mathrm R/\varpi \mathrm R$, 
and let $\overline{\mathrm R}_n:= \mathrm R/(\mm^n+\varpi \mathrm R)$.

Let $\beta: \mathrm F\rightarrow k[x]/(x^3)$ be a homomorphism of $\OO$-algebras such that 
$a_0$, $a_1$, $c_0$, $d_1 \mapsto 0$, $c_1, d_0\mapsto x$.  
Let $\rho_{\beta}: \FF \rtimes G\rightarrow \GL_2(k[x]/(x^3))$ be a representation defined by the same formulas used to define 
$\alpha$ in Proposition \ref{repunA}. 

\begin{lem}\label{trbij2} Let $h\in \FF\rtimes G$ be such that $\tr \rho_{\beta}(hg)=\tr \rho_{\beta}(g)$ for all $g\in \FF\rtimes G$ then 
$\rho_{\beta}(h)=1$. 
\end{lem} 
\begin{proof} Since $\Eins\neq \omega$, $\tr \rho_{\beta}(g)$ determines the diagonal entries of $\rho_{\beta}(g)$ for all $g\in \FF\rtimes G$, 
see \eqref{formulaT1} below. In particular, for all $g\in \FF \rtimes G$ the diagonal 
entries of $\rho_{\beta}(gh)$ are equal to the diagonal entries of $\rho_{\beta}(g)$. Let $x_i \in \FF$ be the generators defined in 
Lemma \ref{generatorsA}. Applying the last observation  to $g=x_0$, $g=x_{p-2}$ 
we deduce that $\rho_{\beta}(h)$ is unipotent upper-triangular, and to $g=x_p$ we deduce that $\rho_{\beta}(h)=1$.
\end{proof}

\begin{lem}\label{trbij3} The surjection $\kappa: \overline{\mathrm F}_3\twoheadrightarrow \overline{S}_3$ is not an isomorphism.
\end{lem} 
\begin{proof} Suppose $\kappa$ is injective then  $\overline{\mathrm F}_3 \cong \overline{S}_3$. 
We may consider $\tr \rho$ as a pseudocharacter of $\FF\rtimes G$ and 
let $D'$ be the deformation functor parameterizing $2$-di\-men\-sio\-nal pseudocharacters of $\FF \rtimes G$ lifting $\tr \rho$.
Corollary \ref{pst3} says that every $2$-di\-men\-sio\-nal pseudocharacter 
of $\gal$ lifting $\tr \rho$ is a pseudocharacter of $\Gal(F(p)/\Qp)$ and thus using Proposition \ref{repunA} we may 
consider $D^{\mathrm{ps}}_{\tr \rho}$ as a subfunctor of $D'$. Using \cite[Thm.2]{joel} we deduce that 
$\dim_k D'(k[x]/(x^2))=6$. Thus if $\overline{S}_3\cong \overline{\mathrm F}_3$ then 
$D'(k[x]/(x^3))=D^{\mathrm{ps}}_{\tr \rho}(k[x]/(x^3))$. This would mean that every $2$-di\-men\-sio\-nal pseudocharacter 
of $\FF \rtimes G$ lifting $\tr \rho$ is automatically a pseudocharacter of $\Gal(F(p)/\Qp)$. This would mean that 
there exists $T\in D^{\mathrm{ps}}_{\tr \rho}(k[x]/(x^3))$, such that $T(\varphi'(g))= \tr \rho_{\beta}(g)$ for all 
$g\in \FF \rtimes G$. The equality would imply that for all $h\in \Ker \varphi'$ and $g\in \FF \rtimes G$ we have 
$\tr \rho_{\beta}(hg)= \tr \rho_{\beta}(g)$. Lemma \ref{trbij2} implies that $\Ker \varphi'$ is contained in 
$\Ker \rho_{\beta}$. However, as $\beta(pc_0+ c_1d_0+ c_1d_0)= x^2\neq 0$ we obtain a contradiction to the universality 
of the representation constructed in Proposition \ref{repunA}.  
\end{proof} 

\begin{lem}\label{trbij4} The map $\theta$ induces an isomorphism $\overline{S}_3\cong \overline{R}_3$.
\end{lem}
\begin{proof} Since $\overline{S}_3\rightarrow \overline{R}_3$ is surjective it is enough to show the equality of dimensions as $k$-vector spaces.
We have 
$\dim_k \overline{R}_3\le \dim_k \overline{S}_3 < \dim_k \overline{F}_3= \dim_k \overline{R}_3 +1$, 
where the strict inequality follows from Lemma \ref{trbij3} and the equality from \eqref{presentR}. 
\end{proof}

\begin{thm}\label{trbij5} The map $\theta: S\rightarrow R$ is an isomorphism.
\end{thm}
\begin{proof} Since $R$ is $\OO$-torsion free it is enough to show that $\theta: \overline{S} \rightarrow \overline{R}$ 
is an isomorphism. Let $f=c_0d_1+c_1d_0 \in \overline{\mathrm F}$, so that $\overline{R}= \overline{\mathrm F}/(f)$ 
and let $\mm$ be the maximal ideal of $\overline{\mathrm{F}}$.
It is enough to show that $\kappa(f)=0$. It follows from Lemma \ref{trbij4} that there exists $g\in \mm^3$, 
such that $\kappa(f)=\kappa(g)$. Thus $\theta(\kappa(g))=0$ an so there exists $h\in \overline{\mathrm F}$ such that 
$g=fh$. Now $h$ cannot be a unit as $g\in \mm^3$ and $f\not \in  \mm^3$. Hence 
$h\in \mm$ and so $1-h$ is a unit. Since $\kappa(f(1-h))= \kappa(f)-\kappa(g)=0$ and $1-h$ is a unit 
we deduce that $\kappa(f)=0$.
\end{proof}

\begin{cor}\label{trbij6} Let $\psi: \gal\rightarrow \OO^{\times}$ be a continuous character lifting $\det \rho$ 
and let $S^{\psi}$ and $R^{\psi}$ be the rings pro-representing functors $D^{\mathrm{ps}, \psi}_{\tr \rho}$ and 
$\Def^{\psi}_{\rho}$ then trace induces an isomorphism $S^{\psi}\cong R^{\psi}$.
\end{cor}

\subsection{Generic reducible case}\label{easier}

Let $\chi_1, \chi_2: \gal\rightarrow k^{\times}$ be continuous characters, such that $\chi_1\chi_2^{-1}\neq \Eins, \omega^{\pm 1}$. We assume that $p\ge 3$.
This assumption and a standard calculation with local duality and local Euler-Poincare characteristic imply that both subspaces 
$\Ext^1_{k[\gal]}(\chi_1, \chi_2)$  and $\Ext^1_{k[\gal]}(\chi_2, \chi_1)$ are $1$-di\-men\-sio\-nal. Let 
$$0\rightarrow \chi_1\rightarrow \rho_{12}\rightarrow \chi_2\rightarrow 0, \quad 0\rightarrow \chi_2\rightarrow \rho_{21}\rightarrow \chi_1\rightarrow 0$$  
be non-split extensions. From now on the indices $i,j$ will mean either $(i,j)=(1,2)$ or $(i,j)=(2,1)$. 
Since $\Ext^1_{k[\gal]}(\chi_j, \chi_i)$ is $1$-di\-men\-sio\-nal, such $\rho_{ij}$ exists and is unique up to isomorphism. 
Since $\chi_1\neq \chi_2 $ we have $\End_{\gal}(\rho_{ij})=k$ thus the universal deformation problem $D_{\rho_{ij}}$ for 
$\rho_{ij}$ is (pro-)representable by a ring $R_{ij}$. Our assumptions  $\chi_1\chi_2^{-1}\neq \Eins, \omega^{\pm 1}$ imply that $H^2(\gal, \Ad \rho_{ij})=0$ and 
$H^1(\gal,\Ad \rho_{ij})$ is $5$-di\-men\-sio\-nal. Hence, $R_{ij}$ is formally smooth of relative dimension $5$ over $\OO$.
Let $\tilde{\rho}_{ij}$ be the universal deformation of $\rho_{ij}$. 

 Let $G$ be the image of $\gal$ in $k^{\times}\times k^{\times}$ under the map $g\mapsto (\chi_1(g), \chi_2(g))$ and let 
 $P$ be the maximal pro-$p$ quotient of the kernel of this map. Since the order of $G$ is prime to $p$ after choosing some splitting we may assume that the action of $\gal$ on $\tilde{\rho}_{ij}$ factors through $P\rtimes G$. Let $\tilde{\chi}_1$ and $\tilde{\chi}_2$ be the Teichm\"{u}ller lifts of $\chi_1$ and $\chi_2$ respectively.

Let $\chi$ be the trace of $\rho_{ij}$, so that $\chi=\chi_1+ \chi_2$ and let $D^{\mathrm{ps}}_{\chi}$ be the functor parameterizing all the $2$-di\-men\-sio\-nal 
pseudo-characters lifting $\chi$. The functor is represented by a ring $R^{\mathrm{ps}}_{\chi}$. Trace induces a morphism of 
functors $D_{\rho_{ij}}\rightarrow D^{\mathrm{ps}}_{\chi}$ and hence a ring homomorphism $\theta: R^{\mathrm{ps}}_{\chi}\rightarrow R_{ij}$.

\begin{prop}\label{rpsr} Trace induces  an isomorphism $\theta: R^{\mathrm{ps}}_{\chi}\overset{\cong}{\rightarrow}R_{ij}$.
\end{prop}
\begin{proof} The map is surjective by  \cite[1.4.4]{kisinfm}. 
Now the tangent space of $R^{\mathrm{ps}}_{\chi}$ is at most $5$-di\-men\-sio\-nal by Theorem 2 in \cite{joel}, alternatively 
the Proposition  can be deduced from \cite{joel}  Theorem 4 and Remark 3. Since $R_{ij}$ is formally smooth of relative dimension $5$ over $\OO$, we deduce that the map is an isomorphism. 
\end{proof} 

\begin{cor}\label{traceisthesame} Let $T: \gal\rightarrow R^{\mathrm{ps}}_{\chi}$ be the universal $2$-di\-men\-sio\-nal pseudocharacter  lifting $\chi$ then $ \tr \tilde{\rho}_{12}(g)= T(g)= \tr \tilde{\rho}_{21}(g)$, for all $g\in \gal.$
\end{cor}

Let $\rr$ be the residubility ideal in $R^{\mathrm{ps}}_{\chi}$ in the sense of \cite[Def. 1.5.2]{BC}. It is uniquely determined by 
the following universal property: an ideal $J$ of $R^{\mathrm{ps}}_{\chi}$ contains $\rr$ if and only if 
$T \pmod {J} = T_1+ T_2$, where $T_1, T_2: \gal \rightarrow R^{\mathrm{ps}}_{\chi}/J$, are deformations of $\chi_1$ and $\chi_2$, respectively, to $R^{\mathrm{ps}}_{\chi}/J$, see \cite[Prop. 1.5.1]{BC}.  We note that since $\chi_1\neq \chi_2$,  $T_1$ and $T_2$ are determined uniquely by $T \pmod {J}$ by the formulas: 

\begin{equation}\label{formulaT1}
T_i(g) = \frac{1}{|G|} \sum_{h\in G} \tilde{\chi}_i^{-1}(h) T(h g) \pmod{J}.
\end{equation}

\begin{prop}\label{redloc} $R^{\mathrm{ps}}_{\chi}/\rr$ is formally smooth of relative dimension $4$ over $\OO$.
\end{prop} 
\begin{proof} Let $D_{\chi_1}$ and $D_{\chi_2}$ be the universal deformation problems for $\chi_1$ and $\chi_2$, respectively. 
A standard argument shows that they are represented by formally smooth $\OO$-algebras  $R_{\chi_1}$, $R_{\chi_2}$, which are of relative dimension $2$ over $\OO$. The functor $D_{\chi_1}\times D_{\chi_2}$ is represented by the ring $R_{\chi_1}\wtimes_{\OO} R_{\chi_2}$, which is formally smooth of relative dimension $4$ over $\OO$.  By the definition of the 
reducibility ideal, the map $a: D_{\chi_1}\times D_{\chi_2}\rightarrow D_{\chi}$, $(T_1, T_2)\mapsto T_1 + T_2$ factors through the 
map $\bar{a}:D_{\chi_1}\times D_{\chi_2}\rightarrow \Hom( R^{\mathrm{ps}}_{\chi}/\rr, \ast)$. It is trivially injective, when the 
functors are evaluated,  and it follows from \eqref{formulaT1}, that is also surjective. Hence, $\bar{a}$ is an isomorphism of functors and 
so $R^{\mathrm{ps}}_{\chi}/\rr\cong R_{\chi_1}\wtimes_{\OO} R_{\chi_2}.$
\end{proof}
 
\begin{cor} The reducibility ideal $\rr$ is a principal ideal.
\end{cor} 
\begin{proof} In fact we prove a stronger statement. It follows from Proposition \ref{rpsr} that $R^{\mathrm{ps}}_{\chi}$ is 
formally smooth of relative dimension $5$ over $\OO$. Thus we may deduce from Proposition 
\ref{redloc} that $\rr$ is generated by an element contained in the maximal ideal of $R^{\mathrm{ps}}_{\chi}$, but not contained in the square of the maximal ideal. Alternatively, one could use \cite[Prop.1.7.1]{BC}.
\end{proof}

Since the order of $G$ is prime to $p$, we may  choose a basis $\{v_1^{ij}, v_2^{ij}\}$ of $\tilde{\rho}_{ij}$, such that 
$G$ acts on  $v_1^{ij}$ by the character $\tilde{\chi}_1$ and on $v_2^{ij}$ by the character $\tilde{\chi}_2$. Fixing a basis allows us to think about $\tilde{\rho}_{ij}$ as a continuous group homomorphism $\tilde{\rho}_{ij}: \gal\rightarrow \GL_2(R_{ij})$, so that 
$\rho_{12}= \bigl ( \begin{smallmatrix} \chi_1 & \ast \\ 0 & \chi_2\end{smallmatrix}\bigr )$ and $\rho_{21}= \bigl ( \begin{smallmatrix} \chi_1 & 0\\ \ast & \chi_2\end{smallmatrix}\bigr )$. 

\begin{lem}\label{endoRij} $\End_{\gal}(\tilde{\rho}_{ij}) = R_{ij}$.
\end{lem} 
\begin{proof} Since the characters $\tilde{\chi}_1$ and  $\tilde{\chi}_2$ are distinct 
$\End_{\gal}(\tilde{\rho}_{ij})\subseteq \End_{G}(\tilde{\rho}_{ij})=\{ \bigl ( \begin{smallmatrix} \lambda & 0 \\ 0 & \mu\end{smallmatrix} \bigr ) : \lambda, \mu\in R_{ij} \}$. 
Since $\rho_{ij}$ is non-split, there exist $g\in \gal$ such that either the entry $(1,2)$ or the entry $(2,1)$  of $\tilde{\rho}_{ij}(g)$ 
is a unit in $R_{ij}$. The only elements of $\End_G(\tilde{\rho}_{ij})$ commuting with  $\tilde{\rho}_{ij}(g)$ are scalar matrices, which then commute with everything.
\end{proof}

\begin{defi} For $(i,j)=(1,2)$ and $(i, j)=(2,1)$, let  $\rr_{ij}$ be the ideal of $R_{ij}$ generated by the $(j,i)$-entry of 
$\tilde{\rho}_{ij}(g)$, for all $g\in \gal$. 
\end{defi}

\begin{prop}\label{rij} The isomorphism $\theta$ of Proposition \ref{rpsr} maps $\rr$ to $\rr_{ij}$.
\end{prop}
\begin{proof} Since $\rho_{12}= \bigl ( \begin{smallmatrix} \chi_1 & \ast \\ 0 & \chi_2\end{smallmatrix}\bigr )$ and $\rho_{21}= \bigl ( \begin{smallmatrix} \chi_1 & 0\\ \ast & \chi_2\end{smallmatrix}\bigr )$, the ideal $\rr_{ij}$ is contained in the maximal ideal 
of $R_{ij}$.  By construction of $\rr_{ij}$, the representation $\tilde{\rho}_{ij} \pmod{\rr_{ij}}$ is reducible. Hence, its trace is a 
direct sum of two characters, which are deformations of $\chi_1$ and $\chi_2$ to $R_{ij}/\rr_{ij}$. Thus, $\rr$ is contained in $\theta^{-1}(\rr_{ij})$, and so $\theta(\rr)\subset \rr_{ij}$.

Let $K$ be the quotient field of $R^{\mathrm{ps}}_{\chi}/\rr$. If for some $g\in \gal$ the $(j,i)$-entry of 
$\tilde{\rho}_{ij}(g) \pmod{ \theta(\rr)}$ is non-zero, then the representation $\tilde{\rho}_{ij}\otimes_{R_{ij}} K$ is absolutely irreducible. 
However, this is impossible as the trace of  $\tilde{\rho}_{ij}\otimes_{R_{ij}} K$ is a sum of two characters. This implies that 
$\rr_{ij} \subset \theta(\rr)$.
\end{proof} 

We fix a generator $c$ of the ideal $\rr$. It follows from Proposition \ref{rij} that $\theta(c)$ is a generator of $\rr_{12}$.
Let $\tilde{\rho}_{12}^c: \gal\rightarrow \GL_2(R_{12})$ be the representation defined by 
$$ \tilde{\rho}_{12}^c(g):= \bigl( \begin{smallmatrix} \theta(c) & 0 \\ 0 & 1 \end{smallmatrix}\bigr )\tilde{\rho}_{12}(g) \bigl( \begin{smallmatrix} \theta(c)^{-1} & 0 \\ 0 & 1 \end{smallmatrix}\bigr ).$$ 
A priori the image of $\tilde{\rho}_{12}^c$ lands in the $\GL_2$ of the quotient field of $R_{12}$, but since 
$R_{12}= \theta(c)^{-1} \rr_{12}$, the image is contained in $\GL_2(R_{12})$.

\begin{prop}\label{2112} The representation $\tilde{\rho}_{12}^c$ is a deformation of $\rho_{21}$ to $R_{12}$. The induced
map $\alpha: R_{21}\rightarrow R_{12}$ is an isomorphism, making the following diagram: 
\begin{displaymath} 
\xymatrix@1{ R_{21} \ar[r]^-{\alpha} & R_{12} \\ R_{\chi}^{\mathrm{ps}}\ar[u]^-{\theta}\ar[r]^-{=}&R_{\chi}^{\mathrm{ps}}\ar[u]^-{\theta} }
\end{displaymath} 
commute.
\end{prop}
\begin{proof} The reduction of $\tilde{\rho}_{12}^c$ modulo the maximal ideal of $R_{12}$ is of the form 
$\bigl( \begin{smallmatrix} \chi_1 & 0 \\ \ast & \chi_2\end{smallmatrix}\bigr )$. Since $R_{12}= \theta(c)^{-1} \rr_{12}$, there
exists $g\in P$, such that $\ast(g)\neq 0$. This implies that the extension is non-split. Since $\Ext^1_{\gal}(\chi_1, \chi_2)$ is 
one dimensional, we deduce that the reduction of $\tilde{\rho}_{12}^c$ modulo the maximal ideal of $R_{12}$ is isomorphic 
to $\rho_{21}$. Hence, $\tilde{\rho}_{12}^c$ is a deformation of $\rho_{21}$ to $R_{12}$. This induces 
the map $\alpha: R_{21}\rightarrow R_{12}$. Since $\tr \rho_{12}^c = \tr \rho_{12} = \tr \rho_{21}$ we obtain a commutative 
diagram as above. Since $\theta$ is an isomorphism, we deduce from the diagram that $\alpha$ is also an isomorphism.
\end{proof}

\begin{cor}\label{homrank1} $\Hom_{\gal}(\tilde{\rho}_{ij}, \tilde{\rho}_{ji})$ is a free $R^{\mathrm{ps}}_{\chi}$ module of rank $1$.
\end{cor}
\begin{proof} It follows from the Proposition \ref{2112} that $\tilde{\rho}^c_{12}\cong \tilde{\rho}_{21}\otimes_{R_{21}, \alpha} R_{12}$. Hence, it is enough to show that $\Hom_{\gal}(\tilde{\rho}_{12}, \tilde{\rho}_{12}^c)$ is a free $R_{12}$-module of rank $1$. If we think of $\tilde{\rho}_{12}$, $\tilde{\rho}_{12}^c$ as  representations of $\gal$ on $R_{12} v_1^{12} \oplus R_{12} 
v_2^{12}$, then inside the ring of $2\times 2$-matrices over the quotient field of $R_{ij}$, we have equalities of $R_{ij}$-modules:
\begin{equation}
\Hom_{\gal}(\tilde{\rho}_{12}, \tilde{\rho}_{12}^c)=\bigl( \begin{smallmatrix} \theta(c) & 0 \\ 0 & 1 \end{smallmatrix}\bigr ) \End_{\gal}(\tilde{\rho}_{12}),
\end{equation}
\begin{equation}
\Hom_{\gal}(\tilde{\rho}_{12}^c, \tilde{\rho}_{12})=\End_{\gal}(\tilde{\rho}_{12})\bigl(\begin{smallmatrix}1 & 0 \\0 &  \theta(c) \end{smallmatrix}\bigr ).
\end{equation}
 The assertion follows from Lemma \ref{endoRij}.
\end{proof}

\begin{prop}\label{genrc} The centre of the ring $\End_{\gal}(\tilde{\rho}_{ij}\oplus \tilde{\rho}_{ji})$ is isomorphic to $R^{\mathrm{ps}}_{\chi}$. 
Moreover, $\End_{\gal}(\tilde{\rho}_{ij}\oplus \tilde{\rho}_{ji})$ is a free $R^{\mathrm{ps}}_{\chi}$-module  of rank $4$.
\end{prop} 

\begin{proof} The ring $\End_{\gal}(\tilde{\rho}_{ij}\oplus \tilde{\rho}_{ji})$ is isomorphic to 
\begin{equation}
\begin{split}
 \begin{pmatrix} \End_{\gal}(\tilde{\rho}_{ij}) & \Hom_{\gal}(\tilde{\rho}_{ij}, 
\tilde{\rho}_{ji}) \\  \Hom_{\gal}(\tilde{\rho}_{ji}, \tilde{\rho}_{ij}) & \End_{\gal}(\tilde{\rho}_{ij})\end{pmatrix}\cong 
\begin{pmatrix} R^{\mathrm{ps}}_{\chi} \Eins_i & R^{\mathrm{ps}}_{\chi} \Phi_{ij}\\  R^{\mathrm{ps}}_{\chi} \Phi_{ji} & R^{\mathrm{ps}}_{\chi} \Eins_j
\end{pmatrix}
\end{split}
\end{equation}
where $\Phi_{ij}$ is described in Corollary  \ref{homrank1}.  It follows from Corollary \ref{homrank1} that 
$\Phi_{ij} \circ \Phi_{ji}= c \Eins_{i}$ and $\Phi_{ji} \circ \Phi_{ij}= c \Eins_{j}$. Since $R^{\mathrm{ps}}_{\chi}$ is 
an integral domain we deduce that $\alpha \Eins_i + \beta \Phi_{ji} + \gamma \Phi_{ij} + \delta  \Eins_j$ is central 
if and only if $\beta=\gamma=0$ and $\alpha=\delta$. 
\end{proof} 

\begin{cor}\label{redloc2} Let $c$ be a generator of  the reducibility ideal $\mathfrak r$ in $R^{\mathrm{ps}}_{\chi}$. Then 
 $\End_{\gal}(\tilde{\rho}_{ij}\oplus \tilde{\rho}_{ji})[1/c]$ is isomorphic to the ring of $2\times 2$ matrices over $R^{\mathrm{ps}}_{\chi}[1/c]$.
\end{cor}
\begin{proof} The isomorphism is induced by sending $\left(\begin{smallmatrix} 1 & 0 \\ 0 & 0\end{smallmatrix} \right)\mapsto \Eins_i$,
$\left(\begin{smallmatrix} 0 & 1 \\ 0 & 0\end{smallmatrix} \right)\mapsto \Phi_{ij}$, 
$\left(\begin{smallmatrix} 0 &  0\\ 1 & 0\end{smallmatrix} \right)\mapsto c^{-1}\Phi_{ji}$ and 
$\left(\begin{smallmatrix} 0 & 0 \\ 0 & 1\end{smallmatrix} \right)\mapsto \Eins_j$.
\end{proof} 

\begin{remar}\label{detcond} Let $\psi: \gal\rightarrow \OO^{\times}$ be a continuous character, congruent to $\chi_1\chi_2$ modulo $\varpi$.
The results of this section hold if instead of working with an unrestricted deformation problem, we consider only those deformations with determinate equal to $\psi$. The proofs carry over word for word, except that one has to subtract $2$ from every dimension, and in the proof of Proposition \ref{redloc} one obtains an isomorphism $R^{\mathrm{ps}, \psi}_{\chi}/\rr\cong R_{\chi_1}$, since the determinant condition imposes the relation $T_1 T_2= \psi$, and hence $T_2$ is uniquely determined by $T_1$.
\end{remar}

\end{document}